\documentclass{amsart}
\usepackage{hyperref}
\usepackage{enumerate} 
\usepackage{upref}
\usepackage{verbatim} 
\usepackage{color}
\usepackage{graphicx}
\usepackage{bbm}
\usepackage{xcolor}
\usepackage{psfrag}
\usepackage[textsize=tiny]{todonotes}

\newcommand{\sign}{\mathop{\rm sign}}
\newcommand*{\mailto}[1]{\href{mailto:#1}{\nolinkurl{#1}}}

\usepackage[latin1]{inputenc}
\usepackage{amssymb, amsmath, amsfonts, amsthm}
\usepackage[all]{xy}

\usepackage{marginnote}

\newcommand{\dott}{\, \cdot\,}

\renewcommand{\P}{\ensuremath{\mathcal{P}}}
\newcommand{\Q}{\ensuremath{\mathcal{Q}}}

\newcommand{\U}{\ensuremath{\mathcal{U}}}

\newcommand{\V}{\ensuremath{\mathcal{U}}} %newcommand{\V}{\ensuremath{\mathcal{V}}}

\newcommand{\X}{\ensuremath{\mathcal{X}}}
\newcommand{\Y}{\ensuremath{\mathcal{Y}}}
\newcommand{\Henergy}{\ensuremath{\mathcal{H}}}
\newcommand{\E}{\ensuremath{\mathcal{E}}}

\newcommand{\D}{\ensuremath{\mathcal{D}}}

\newcommand{\R}{\ensuremath{\mathcal{R}}}
\newcommand{\So}{\ensuremath{\mathcal{S}}}

\newcommand{\abs}[1]{\left\vert#1\right\vert}

\newcommand{\Real}{\mathbb R}

\newcommand{\norm}[1]{\left\Vert#1\right\Vert}

\newcommand{\muac}{\mu_{\text{\rm ac}}}
\newcommand{\e}{\ensuremath{\mathrm{e}}}

\newcommand{\nn}{\nonumber}

\newcommand{\sqP}[1]{\tilde\P_{#1}^{1/2}}
\newcommand{\sqDP}[1]{(\tilde\P_{#1}^{1/2})_\eta}
\newcommand{\sqtC}[1]{A_{#1}}
\newcommand{\A}{A}
\newcommand{\ma}{a}
\newcommand{\bigO}{\mathcal O}
\newtheorem{theorem}{Theorem}[section]

\newtheorem{lemma}[theorem]{Lemma}
\newtheorem{definition}[theorem]{Definition}
\newtheorem{proposition}[theorem]{Proposition}

\newtheorem{remark}[theorem]{Remark}

\numberwithin{equation}{section}

\allowdisplaybreaks

\begin{document}

\title[Metric for CH equation]{A Lipschitz metric for the Camassa--Holm equation}

\author[J. A. Carrillo]{Jos\'e Antonio Carrillo}
\address[J. A. Carrillo]{Department of Mathematics\\ Imperial College London \\ South Kensington Campus\\ London SW7 2AZ\\ UK}
\email{\mailto{carrillo@imperial.ac.uk}}
\urladdr{\url{http://wwwf.imperial.ac.uk/~jcarrill/index.html}}

\author[K. Grunert]{Katrin Grunert}
\address[K. Grunert]{Department of Mathematical Sciences\\ Norwegian University of Science and Technology\\ NO-7491 Trondheim\\ Norway}
\email{\mailto{katrin.grunert@ntnu.no}}
\urladdr{\url{https://www.ntnu.edu/employees/katrin.grunert}}

\author[H. Holden]{Helge Holden}
\address[H. Holden]{Department of Mathematical Sciences\\
  Norwegian University of Science and Technology\\
  NO-7491 Trondheim\\ Norway}
\email{\mailto{helge.holden@ntnu.no}}
\urladdr{\url{https://www.ntnu.edu/employees/holden}}

\thanks{Research supported by the grants {\it Waves and Nonlinear Phenomena (WaNP)} and {\it Wave Phenomena and Stability - a Shocking Combination (WaPheS)} from the Research Council of Norway.}  
\subjclass[2010]{Primary: 35Q53, 35B35; Secondary: 35B60}
\keywords{Camassa--Holm equation, Lipschitz metric, conservative solution}
\date{\today}
\thanks{JAC, KG and HH are very grateful to the Institut Mittag-Leffler, Stockholm,  for generous support. HH would like to thank the Isaac Newton Institute for Mathematical Sciences for support and hospitality during the program Nonlinear Water Waves when work on this paper was undertaken. This work was supported by EPSRC Grant Number EP/P031587/1.}

\begin{abstract}
We analyze stability of conservative solutions of the Cauchy problem on the line for the Camassa--Holm (CH) equation. Generically, the solutions of the CH 
equation develop singularities with steep gradients while preserving continuity of the solution itself. In order to obtain uniqueness, one is required to augment the equation 
itself by a measure that represents the associated energy, and the breakdown of the solution is associated with a complicated interplay where the measure becomes singular. 
The main result in this paper is the construction of a Lipschitz metric that compares two solutions of the CH equation with the respective initial data. The Lipschitz metric is based 
on the use of the Wasserstein metric.
\end{abstract}
\maketitle

\section{Introduction}

We study the Cauchy problem for weak conservative solutions of the Camassa--Holm (CH) equation, which reads
\begin{subequations}
\label{eq:CHbasic}
\begin{align}\label{CH:1}
u_t+uu_x+p_x&=0,\\ \label{CH:2}
\mu_t+(u\mu)_x&= (u^3-2pu)_x,
\end{align}
\end{subequations}
where $p(t,x)$ is given by 
\begin{equation}\label{P:rep}
p(t,x)=\frac14 \int_\Real e^{-\vert x-y\vert} u^2(t,y)dy+\frac14 \int_\Real e^{-\vert x-y\vert } d\mu(t,y),
\end{equation}
with initial data $(u,\mu)|_{t=0}=(u_0,\mu_0)$. In this case, the natural solution space consists of all pairs $(u,\mu)$ such that 
\begin{equation*}
u(t,\dott)\in H^1(\Real), \quad \mu(t,\dott)\in \mathcal{M}_+(\Real), \quad \text{and}\quad d\muac=(u^2+u_x^2)dx,
\end{equation*}
where $\mathcal{M}_+(\Real)$ denotes the set of all positive and finite Radon measures on $\Real$. Our main goal is to prove the existence of a metric $d$ such that 
\begin{equation*}
d((u_1(t),\mu_1(t)),(u_2(t),\mu_2(t)))\leq \alpha(t)d((u_{1,0},\mu_{1,0}),(u_{2,0}, \mu_{2,0})),
\end{equation*}
for two weak conservative solutions $(u_j(t),\mu_j(t))$ of \eqref{eq:CHbasic} with initial data $(u_{j,0},\mu_{j,0})$ $j=1,2$. Here $\alpha(t)$ depends on the total energy of the solutions and $\alpha(0)=1$.

The Camassa--Holm equation has been introduced in the seminal paper \cite{CH}, see also \cite{CHH}. Originally derived in the context of models for shallow water, the CH equation has
been intensively studied due to the intricate behavior of solutions of the Cauchy problem, which will be the focus of this paper.   We will not discuss further properties of the CH equation, e.g., the fact that the equation is completely integrable and allows for a geometric interpretation. Several extensions and generalizations exist, but we will focus on \eqref{eq:CHbasic}. 

The intriguing aspect of solutions to the Cauchy problem is the generic development of singularities in finite time, irrespective of the smoothness of the initial data. A solution may develop steep gradients, but in contrast to, say, hyperbolic conservation laws, the solution itself remains continuous. A finer analysis reveals that the energy density $u^2+u_x^2$ develops singularities, and at breakdown, which is often referred to as wave breaking, energy concentrates on sets of measure zero. Thus it becomes useful to introduce a measure, here denoted $\mu$, that encodes the energy, and away from breakdowns this measure should coincide with the energy density $u^2+u_x^2$. In technical terms we consider a nonnegative Radon measure $\mu$ with absolutely continuous part $d\muac=(u^2+u_x^2)dx$. This measure $\mu$ satisfies the equation \eqref{CH:2} (and as can easily be verified, $\mu=u^2+u_x^2$ will satisfy the same equation in the case of smooth solutions). An illustrating example of how intricate the structure of the points of wave breaking may be can be found in \cite{grunert}. The behavior in the proximity of the point of wave breaking, and, in particular, the prolongation of the solution past wave breaking, has been extensively studied.  See, e.g.,  
\cite{BC,BreCons:05a,ChenLiuZhang,ChenLiu,cons:98b,EscherLechtenfeldYin,FuQu,GHR2,GHR3,GHR5,GHR4,GHR,GuanKarlsenYin,GY, GuanYin2011, GuiLiu2010, GuiLiu2011, GuoZhou2010, HR,HRdiss,HRdissMP,GrasGrunHol} and references therein.  The key point here is that past wave breaking uniqueness fails, and there is a continuum of distinct solutions \cite{GHR}, with two extreme cases called dissipative and conservation solutions, respectively. To understand this conundrum it turns out to be advantageous to rewrite the equation in a different set of variables where the solution remains smooth. 

The explicit peakon-antipeakon solution \cite{GH}, which illustrates this problem, is given by  
\begin{equation} \label{eq:2peak}
u(t,x)
= \begin{cases} -\alpha(t)e^{x}, & \quad x\le-\gamma(t),\\
\beta(t)\sinh(x), & \quad -\gamma(t)<x<\gamma(t),\\
\alpha(t)e^{-x}, & \quad \gamma(t)\le x,
\end{cases}
\end{equation}
where 
\begin{equation*}
\alpha(t)= \frac{E}{2}\sinh(\frac{E}{2}t), \quad \beta(t)=E\frac{1}{\sinh(\frac{E}{2}t)}, \quad \gamma(t)=\ln(\cosh(\frac{E}{2}t)).
\end{equation*}
where $E=\norm{u(t)}_{H^1}$ for all $t\neq0$.
This function is a weak conservative solution which consists of a ``peak'' moving to the right and an ``anti-peak'' moving to the left, see Figure \ref{fig:1}. At $t=0$ the  ``peak''  and ``anti-peak'' collide, and the solution vanishes, yet the solution is highly non-trivial before and after the collision time. Clearly the trivial solution will coincide with this solution at  $t=0$, yet the trivial solution and  \eqref{eq:2peak} are very different at any other time. Thus it is not clear how to derive a metric comparing two solutions, that is stable under the time evolution. This is the task of the present work. 

To be more precise, we are here presenting a metric $d$ with the property that
\begin{equation*}
 d((u_1(t),\mu_1(t)), (u_2(t), \mu_2(t)))\leq  \alpha(t)d((u_{1,0},\mu_{1,0}), (u_{2,0}, \mu_{2,0})),
\end{equation*}
for two weak conservative solutions $(u_j(t),\mu_j(t))$, $j=1,2$, of \eqref{eq:CHbasic} with initial data $(u_{j,0},\mu_{j,0})$, $j=1,2$. Here  $\alpha(t)$ is a continuous function with $\alpha(0)=1$, which may depend on the total energies involved, but not on the particular solutions. We stress that no standard Sobolev norm nor Lebesgue space norm will work.  There exist alternative metrics for solutions of the CH equation, see \cite{BF}, \cite{GHR1}, and \cite{GHR2}. In \cite{BF} the periodic case is treated by approximating the solution by multi-peakons. The metric is defined by optimizing over a class of functions. The approach in   \cite{GHR1} and \cite{GHR2} depends on a reformulation of the CH equation in terms of Lagrangian variables. An intrinsic problem in this formulation  is that of {\it relabeling}, where there will be many different parametrizations in Lagrangian coordinates, corresponding to one and the same solution  $(u(t),\mu(t))$ in Eulerian variables. Thus one has to compute the distance between equivalence classes, which is not transparent. In the present approach, the key idea is to introduce a new set of variables, where one variable plays a role similar to a characteristic, while the remaining variables are linked to $(u,\mu)$ with the help of the ``characteristic". As we will outline next, there is no need to resort to equivalence classes or to optimize over classes of functions, and in spite of this proof being longer, we consider this approach to be more natural.  

Our approach is based on the fact that a natural metric for measuring distances between Radon measures (with the same total mass) is given through the Wasserstein (or Monge--Kantorovich) distance $d_W$, which in one dimension is defined with the help of pseudo inverses, see \cite{Vil03}. Given a nonnegative measure $\mu$ of finite mass $M>0$, we define the cumulative distribution function associated to $\mu$ as
$$
F(x)=\mu((-\infty,x))\,,
$$
which is a nondecreasing function from $\mathbb{R}$ onto $[0,M]$, left continuous and with limit from the right at any point $x\in\mathbb{R}$. The pseudo inverse associated to $\mu$ denoted by $\X$ is the function from  $[0,M]$ onto $\mathbb{R}$ given by 
$$
\mathcal{X}(\eta)=\sup\{x\mid F(x)<\eta\}.
$$
The pseudo-inverse of $F$ is a nondecreasing function from $[0,M]$ onto $\mathbb{R}$, left continuous and with limit from the right (caglad) at any point $x\in\mathbb{R}$. Notice the different convention adopted here with respect to the usual one in probability theory defining cumulative distribution functions continuous from the right and with limits from the left (cadlag) at every point. We prefer to have caglad instead of cadlag functions due to the use of the methods from \cite{HR} developed under the present convention. Wasserstein distances between nonnegative measures with the same mass can be defined via $L^p$-norms of the difference between their associated pseudo-inverses, see \cite{LT,CT,Vil03,CT2} and the references therein.

The approach of using Wasserstein distances to control the expansion of solutions of evolutionary PDEs leading to curves of probability measures goes back to the proofs of the mean-field limit of McKean--Vlasov and Vlasov equations in the late seventies and eighties of the last century. We refer to the classical references \cite{dobru,BH,Neun,Szni,Spohn} proving these large particle limits by means of the bounded Lipschitz distance and the coupling method.  See the recent results and surveys in \cite{Golse1,CCR,BCC,Golse2,CCHS}. The optimal transport viewpoint for one dimensional models was developed using pseudo inverse distributions for nonlinear aggregation and  diffusion equations in \cite{LT,CT,T2000,CT2,CDT} and the references therein, showing the contractivity of the Wasserstein distance in one dimension without the heavy machinery of optimal transport developed for general gradient flows in \cite{AGS}. More recently, these metrics have been used with success to show uniqueness past the blow-up time for multidimensional aggregation equations \cite{CDFLS2011} using gradient flow solutions. It is also interesting to point out that gradient flow solutions of the aggregation equation in one dimension with particular potentials are equivalent to entropy solutions of the Burgers equation as proven in \cite{BCDP}.
Another strategy using unbalanced optimal transport tools has been recently analyzed in \cite{GaVi} with the objective of understanding the relation between the incompressible Euler equations and the CH equations.

Finally, it is worth mentioning that there have been several works  \cite{gosse2006lagrangian,gosse2006identification,carrillo2009numerical, blanchet2008convergence,WW2010,during2010gradient,MO2014, junge2015fully,CRW2016} making use of this change of variables to produce numerical schemes capable of going over blow-up of solutions to nonlinear aggregations and being able to capture the blow-up of solutions of aggregation-diffusion models in one dimension such as toy versions of the Keller--Segel model for 
chemotaxis. It is a nice avenue of research to use this approach to produce numerical schemes for conservative solutions of the CH equation, see \cite{ChLiPe} for related particle methods.

In the present work, we will adapt this strategy of defining suitable distances between measures to the present problem of finding good metrics for solutions of the CH equation.
Let $(u(t,\dott),\mu(t,\dott))$ be  a weak conservative solution to the CH equation with total energy $\mu(t,\Real)=C>0$ (which for simplicity here is assumed to be smooth). 
Let 
\begin{equation*}
F(t,x)=\mu(t,(-\infty,x))=\int_{-\infty}^x d\mu(t)
\end{equation*}
due to the smoothness, and introduce the basic quantity
\begin{equation*}
G(t,x)=\int_{-\infty}^x (2p-u^2)(t,y)dy +F(t,x)= 2p_{x}(t,x)+2F(t,x).
\end{equation*}
The key function here is the (spatial) inverse of  the strictly increasing function $G$ for fixed time $t$. To that end we define
\begin{equation*} 
\Y(t,\eta)=\sup\{x\mid G(t,x)<\eta\}.
\end{equation*}
Formally we have that $G(t,\Y(t,\eta))=\eta$ for all $\eta\in(0,2C)$ and $\Y(t,G(t,x))=x$ for all $x\in \Real$. Here it is important to note that the domain of $\Y$ depends on the total energy $C$. Next, we want to determine the time evolution of $\Y$. Direct formal calculations yield that 
\begin{subequations}\label{CH:4_intro}
\begin{align}
\Y_{t}(t,G(t,x))+\Y_{\eta}(t,G(t,x))G_{t}(t,x)&=0,\\
\Y_{\eta}(t,G(t,x))G_{x}(t,x)& =1.
\end{align}
\end{subequations}
Thus we need to compute the time evolution of $G(t,x)$ before being able to compute the time evolution of $\Y(t,\eta)$. To that end we find, after some computations, that 
\begin{equation}\label{CH:5_intro}
G_{t}(t,x)+uG_{x}(t,x)=\frac23 u^3(t,x)+S(t,x),
\end{equation}
where 
\begin{equation}
S(t,x)= \int_\Real e^{-\vert x-y\vert }\left(\frac23 u^3-u_{x}p_{x}-2pu\right) (t,y) dy.
\end{equation}
Introducing $\eta=G(t,x)$,  $\So(t,\eta)=S(t,\Y(t,\eta))$, and
\begin{equation}\label{eq:Udef_intro}
\U(t,\eta)=u(t,\Y(t,\eta)),
\end{equation}
we find by combining \eqref{CH:4_intro} and \eqref{CH:5_intro}, that
\begin{equation*}
\Y_{t}(t,\eta)+(\frac23 \U^3+\So)(t,\eta)\Y_{\eta}(t,\eta)= \U(t,\eta), 
\end{equation*}
where we used that $\Y(t,G(t,x))=x$ for all $x\in\Real$. As far as the time evolution of 
$\U(t,\eta)$ is concerned, we find 
\begin{align*}
\U_{t}(t,\eta)& = -\Q(t,\eta)-(\frac23 \U^3+\So)\U_{\eta}(t,\eta),
\end{align*}
where we introduced $\Q(t,\eta)=p_{x}(t,\Y(t,\eta))$.
Thus, formally we end up with the system
\begin{align*}
\Y_{t}(t,\eta)+(\frac23 \U^3+\So)\Y_{\eta}(t,\eta)& =\U(t,\eta),\\
\U_{t}(t,\eta)+(\frac23 \U^3+\So)\U_{\eta}(t,\eta)& = -\Q(t,\eta).
\end{align*}
However, this system is not closed, and we need to introduce the function 
\begin{equation}\label{eq:Pdef_intro}
\P(t,\eta)=p(t,\Y(t,\eta)),
\end{equation}
and determine its time evolution.   We find, after some computations, that
\begin{equation*}
\P_{t}(t,\eta)+(\frac23 \U^3+\So)\P_{\eta}(t,\eta) = \Q\U(t,\eta)+\R(t,\eta),
\end{equation*}
where
\begin{align*}
\R(t,\eta)& =  \frac14 \int_0^{2C}\sign(\eta-\theta)e^{-\vert \Y(t,\eta)-\Y(t,\theta)\vert} (\frac23 \U^3\Y_{\eta}+\U)(t,\theta)d\theta \\
& \quad -\frac12 \int_0^{2C} e^{-\vert \Y(t,\eta)-\Y(t,\theta)\vert}\U\Q\Y_{\eta}(t,\theta) d\theta.
\end{align*}
To summarize, we have established the following system of differential equations
\begin{subequations}  \label{eq:evolution_intro}
\begin{align}
\Y_{t}(t,\eta)+(\frac23 \U^3+\So)\Y_{\eta}(t,\eta)& =\U(t,\eta),\\
\U_{t}(t,\eta)+(\frac23 \U^3+\So)\U_{\eta}(t,\eta)& = -\Q(t,\eta),\\
\P_{t}(t,\eta)+(\frac23 \U^3+\So)\P_{\eta}(t,\eta)& = \Q\U(t,\eta)+\R(t,\eta),
\end{align}
\end{subequations}
where 
\begin{subequations}
\begin{align}
\Q(t,\eta)& = -\frac14 \int_0^{2C} \sign(\eta-\theta)e^{-\vert \Y(t,\eta)-\Y(t,\theta)\vert } (2(\U^2-\P)\Y_{\eta}(t,\theta)+1) d\theta,\\
\So(t,\eta)&= \int_{0}^{2C} e^{-\vert \Y(t,\eta)-\Y(t,\theta)\vert} (\frac23 \U^3\Y_{\eta}-\U_{\eta}\Q-2\P\U\Y_{\eta})(t,\theta)d\theta,\\
\R(t,\eta)& =  \frac14 \int_0^{2C}\sign(\eta-\theta)e^{-\vert \Y(t,\eta)-\Y(t,\theta)\vert} (\frac23 \U^3\Y_{\eta}+\U)(t,\theta)d\theta \notag\\
& \quad -\frac12 \int_0^{2C} e^{-\vert \Y(t,\eta)-\Y(t,\theta)\vert}\U\Q\Y_{\eta}(t,\theta) d\theta.  \label{eq:Reta_intro}
\end{align}
\end{subequations}
Derived under assumptions of smoothness of the functions involved,  the same system is valid also
in the general case of weak conservation solutions. However, that requires considerable analysis, and Section \ref{sec:general} is devoted to that. The next step is to estimate the time evolution of these quantities $(\Y,\U,\P)$. It turns out that the natural functional space is the space of square integrable functions for the unknowns $(\Y,\U,\P^{1/2})$. For this reason we prefer to work with $\P^{1/2}$ rather than $\P$. Section \ref{sec:general} focuses on the first qualitative properties of the time evolution of the solutions of \eqref{eq:evolution_intro} for weak conservative solutions of the CH equation \eqref{eq:CHbasic} as well as the propagation in time of the $L^2$ norm of the unknowns.

The main aim of our work is to identify the right distance between two general conservative solutions of the CH equation \eqref{eq:CHbasic}, or equivalently, between two general $L^2$ solutions $(\Y,\U,\P^{1/2})$ of the system \eqref{eq:evolution_intro} with possibly different energies. In order to compare solutions with different energies we need to rescale the solutions of \eqref{eq:evolution_intro} in such a way that they are defined on the same interval. Since the natural functional spaces for our unknowns $(\Y,\U,\P^{1/2})$ was identified as the $L^2$-functional space, it seems natural to do a scaling conserving the $L^2$-norms of the unknowns $(\Y,\U,\P^{1/2})$, but leading to the domain being independent of the total energy $C$.

Let us define the scaled unknowns $(\tilde \Y,\tilde \U,\tilde \P^{1/2})$ associated to a conservative solution $(u(t),\mu(t))$ with energy $C=\mu(t,\Real)$ of the CH equation \eqref{eq:CHbasic} as
$\tilde \Y(t,\eta)= \sqrt{2C}\,\Y(t, 2C\eta)$, $\tilde \U(t,\eta)=\sqrt{2C}\, \U(t,2C\eta)$, and 
$\tilde\P^{1/2}(t,\eta)=\sqrt{2C}\,\P^{1/2}(t, 2C\eta)$,
where $(\Y,\U,\P^{1/2})$ is the solution of \eqref{eq:evolution_intro}. This scaling allows also for the zero solution to  \eqref{eq:CHbasic} to be included in our considerations, as outlined in Section~\ref{sec:difen}. A similar system to \eqref{eq:evolution_intro} can be written for $(\tilde \Y,\tilde \U,\tilde \P^{1/2})$, but this is postponed to Section \ref{sec:difen}. With this new set of unknowns in place, we can now define a metric to compare two general conservative solutions $(u_i,\mu_i)$, $i=1,2$, of \eqref{eq:CHbasic} with total energy $C_i=\mu_i(\Real)$.  We define it as
\begin{align*}
d((u_1,\mu_1), (u_2, \mu_2)) = &\,\norm{\tilde \Y_1-\tilde \Y_2}_{L^2([0,1])}+\norm{\tilde \U_1-\tilde \U_2}_{L^2([0,1])}\\
& + \norm{\tilde \P^{1/2}_1-\tilde \P^{1/2}_2}_{L^2([0,1])}+\vert \sqrt{2C_1}-\sqrt{2C_2}\vert.
\end{align*}
Our main result reads as follows: 

\begin{theorem}
Consider initial data $u_{i,0}\in H^1(\Real)$, $\mu_{i,0}\in \mathcal{M}_+(\Real)$ such that $d(\muac)_{i,0}=(u_i^2+u_{i,x}^2)dx$ and $C_i=\mu_i(\Real)$, and let 
$(u_i,\mu_i)$ for $i=1,2$ denote the corresponding weak conservative solutions of the Camassa--Holm equation \eqref{eq:CHbasic}.  Then we have that
\begin{equation*}
 d((u_1(t),\mu_1(t)), (u_2(t), \mu_2(t)))\leq e^{\bigO(1)t}d((u_{1,0},\mu_{1,0}), (u_{2,0}, \mu_{2,0})),
\end{equation*}
where $\bigO(1)$ denotes a constant depending only on $\max_j(C_j)$ remaining bounded as $\max_j(C_j)\to 0$.
\end{theorem}

The main core of this work lies in estimating the Lipschitz property of the right-hand side of the equivalent system to \eqref{eq:evolution_intro} in the $L^2$-sense for the unknowns $(\tilde \Y,\tilde \U,\tilde \P^{1/2})$. This is much easier in case we compare to the zero solution as it coincides with the propagation of the $L^2$ norms of the unknowns. Due to the intricate nonlinearities of the right-hand sides of  \eqref{eq:evolution_intro}, this leads in the general case to long detailed technical estimates that are displayed in full in Subsections \ref{subsec:lipy}, \ref{subsec:lipu}, and \ref{subsec:lipp} in Section \ref{sec:Lip}.  In the case of peakon-antipeakon solutions, as the solution \eqref{eq:2peak} is denoted, all quantities described in this paper can be computed explicitly. The details are to be found in Appendix \ref{app:peak}.

A notational comment is in order, we decided to denote by $\bigO(1)$ constants depending on $\max_j(C_j)$ that may change from line to line along the proofs, but remain bounded as $\max_j(C_j)\to 0$. Explicit tracking of the constants could be possible but it is highly cumbersome and avoided for the sake of the reader.

%%%%%%%%%%%%%%%%%%%%%%%
%%%%%%%%%%%%%%%%%%%%%%%
%%%%%%%%%%%%%%%%%%%%%%%

\section{Formal Ideas: Transformations with smoothness}

Let us start by explaining all the mathematical details for the transformation in the case of smooth solutions as outlined in the introduction. Let $(u(t,\dott),\mu(t,\dott))$ be a weak conservative solution to the Camassa--Holm equation with total energy $\mu(t,\Real)=C>0$. We assume that $F(t,x)$, given by
\begin{equation}\label{deff}
F(t,x)=\int_{-\infty}^x d\mu(t), 
\end{equation}
is increasing and smooth, and, in particular, that $\mu=\mu_{\rm ac}=(u^2+u_x^2)dx$ for all $t$. Introduce the function 
\begin{align*}
G(t,x)&=\int_{-\infty}^x (2p-u^2)(t,y)dy +F(t,x)\\ \nn
& = 2p_x(t,x)+2F(t,x),
\end{align*}
where we used integration by parts and \eqref{P:rep}.
First of all note that the function $G(t,x)$ satisfies
\begin{equation*}
\lim_{x\to -\infty} G(x)=0 \quad \text{ and } \quad \lim_{x\to \infty} G(x)=2C,
\end{equation*} 
since $|p_x(t,x)|\leq p(t,x)$ and $p$ is an $H^1$ function on the line due to \eqref{P:rep}.
Moreover, the function $(2p-u^2)(t,x)\geq 0$ for all $(t,x)\in \Real^2$ as the following computation shows,
\begin{align} \nn
(2p-u^2)(t,x)& = \frac12 \int_\Real e^{-\vert x-y\vert } (u^2+F_x)(t,y)dy -u^2(t,x)\\ \nn
& = \frac12 \int_{-\infty}^x e^{y-x} u^2(t,y)dy + \frac12 \int_x^\infty e^{x-y} u^2(t,y)dy\\ \nn
& \qquad  +\frac12 \int_\Real e^{-\vert x-y\vert } F_x(t,y)dy -u^2(t,x)\\ \nn
& = \frac12 u^2(t,x)-\frac12 \int_{-\infty}^x e^{y-x} 2uu_x(t,y)dy\\ \nn
& \qquad+\frac12 u^2(t,x)+\frac12 \int_x^\infty e^{x-y}2uu_x(t,y)dy\\ \nn
& \qquad + \frac12 \int_\Real e^{-\vert x-y\vert } F_x(t,y)dy -u^2(t,x)\\ \nn
& \geq \frac12 \int_\Real e^{-\vert x-y\vert } \big(F_x(t,y)-2\vert uu_x\vert (t,y)\big) dy \\
& = \frac12 \int_\Real e^{-\vert x-y\vert } (\vert u(t,y)\vert -\vert u_x(t,y)\vert )^2 dy\geq 0.  \label{eq:2Ppluss}
\end{align}
Thus the function $G(t,x)$ is non-decreasing and, in our case, since the function $F(t,x)$ is smooth, also $G(t,x)$ is smooth.

\begin{remark}
The estimate \eqref{eq:2Ppluss}, that is, $2p-u^2\ge0$, remains valid also in the case where the functions are non-smooth.
\end{remark}

Last, but not least, we want to make sure that $G(t,x)$ is strictly increasing, so that its pseudo-inverse will have no jumps. $F(t,x)$ is constant if and only if both $d\mu$, $u$, and $u_x$ are equal to zero. Therefore assume that there exists (for fixed $t$) some interval $[b,c]$ such that $d\mu(t,x)=u(t,x)=u_x(t,x)=0$ for all $x\in [b,c]$. Then the only term that can save us is $p(t,x)$, which in general satisfies $p(t,x)\geq 0$ for all $(t,x)\in \Real^2$. However, whenever $\mu(t,\Real)\not =0$, one has by its definition in \eqref{P:rep} that $p(t,x)>0$ and the claim follows.

Thus the function $G(t,x)$ is strictly increasing and continuous, and we can consider its pseudo-inverse $\Y\colon[0,2C]\to\Real$, which in this case coincides with its inverse and which is given by 
\begin{equation} \label{eq:Ydef}
\Y(t,\eta)=\sup\{x\mid G(t,x)<\eta\}.
\end{equation}
Since $G(t,x)$ is strictly increasing and continuous, we have that $G(t,\Y(t,\eta))=\eta$ for all $\eta\in(0,2C)$ and $\Y(t,G(t,x))=x$ for all $x\in \Real$. By the smoothness assumption on $F$, direct calculations yield that 
\begin{subequations}\label{CH:4}
\begin{align}
\Y_t(t,G(t,x))+\Y_\eta(t,G(t,x))G_t(t,x)&=0,\\
\Y_\eta(t,G(t,x))G_x(t,x)& =1.
\end{align}
\end{subequations}
Thus we need to compute the time evolution of $G(t,x)$ before being able to compute the time evolution of $\Y(t,\eta)$. The following calculations are only valid in the case of smooth solutions, but we will show in the next section how to overcome this issue for weak conservative solutions. Since $e^{-\vert x-y\vert}/2$ is the integral kernel of $(-\partial_x^2+1)^{-1}$, we observe from \eqref{P:rep} that $p$ is the solution to 
\begin{equation*}
p-p_{xx}=\frac12 u^2+\frac12 \mu
\end{equation*}
and hence 
\begin{align*}
p_t-p_{txx}& = uu_t+\frac12 F_{xt}\\ \nn
& = -u^2u_x-up_x-\frac12 (uF_x)_x+\frac12 (u^3)_x-(pu)_x\\ \nn
& = \frac16 (u^3)_x-(pu)_x-\frac12(uF_x)_x-up_x,
\end{align*}
where we used the abbreviation $\mu=F_x$.
Thus we end up with
\begin{align*}
p_t(t,x) &= -\frac12 \int_{\Real} \sign(x-y)e^{-\vert x-y\vert } \left(\frac16 u^3-pu-\frac12 uF_x\right)(t,y)dy\\ \nn
& \quad -\frac12 \int_\Real e^{-\vert x-y\vert} up_x(t,y)dy.
\end{align*}
Similar calculations yield that 
\begin{equation*}
p_{xt}(t,x) = -\frac16 u^3+pu+\frac12 uF_x+\frac12 \int_\Real e^{-\vert x-y\vert } \left(\frac23 u^3-u_xp_x-2pu\right)(t,y)dy.
\end{equation*}
Thus we get for the time evolution of $G(t,x)$ that 
\begin{equation}\label{CH:5}
G_t(t,x)+uG_x(t,x)=\frac23 u^3(t,x)+S(t,x),
\end{equation}
where 
\begin{equation*}
S(t,x)= \int_\Real e^{-\vert x-y\vert }\left(\frac23 u^3-u_xp_x-2pu\right) (t,y) dy.
\end{equation*}
Combining \eqref{CH:4} and \eqref{CH:5}, we end up with
\begin{equation*}
\Y_t(t,G(t,x))+(\frac23 u^3(t,x)+S(t,x))\Y_\eta(t,G(t,x))= u(t,x).
\end{equation*}
Introducing $\eta=G(t,x)$, we deduce
\begin{equation*}
\Y_t(t,\eta)+(\frac23 u^3+S)(t,\Y(t,\eta))\Y_\eta(t,\eta)= u(t,\Y(t,\eta)), 
\end{equation*}
where we used that $\Y(t,G(t,x))=x$ for all $x\in\Real$. As far as the time evolution of 
\begin{equation}\label{eq:Udef}
\U(t,\eta)=u(t,\Y(t,\eta))
\end{equation}
is concerned, we have 
\begin{align*}
\U_t(t,\eta)&= u_t(t,\Y(t,\eta))+u_x(t,\Y(t,\eta))\Y_t(t,\eta)\\
& = u_t(t,\Y(t,\eta))+ uu_x(t,\Y(t,\eta))-(\frac23 u^3+S)u_x(t,\Y(t,\eta))\Y_n(t,\eta)\\
& =-p_x(t,\Y(t,\eta))-(\frac23 u^3+S)u_x(t,\Y(t,\eta))\Y_n(t,\eta)\\
& = -\Q(t,\eta)-(\frac23 \U^3+\So)\U_\eta(t,\eta),
\end{align*}
where we introduced $\Q(t,\eta)=p_x(t,\Y(t,\eta))$ and $\So(t,\eta)=S(t,\Y(t,\eta))$.
Thus, formally we end up with the system
\begin{subequations}\label{hejkatrin}
\begin{align}
\Y_t(t,\eta)+(\frac23 \U^3+\So)\Y_\eta(t,\eta)& =\U(t,\eta),\\
\U_t(t,\eta)+(\frac23 \U^3+\So)\U_\eta(t,\eta)& = -\Q(t,\eta),
\end{align}
\end{subequations}
where $\Q(t,\eta)$ and $\So(t,\eta)$ can be written as
\begin{align*}
\Q(t,\eta)&=-\frac14 \int_{\Real} \sign(\Y(t,\eta)-y)e^{-\vert \Y(t,\eta)-y\vert} (u^2(t,y)+F_x(t,y))dy\\
& = -\frac14\int_\Real \sign(\Y(t,\eta)-y)e^{-\vert \Y(t,\eta)-y\vert}  (2(u^2(t,y)-p(t,y))+G_x(t,y))dy\\
& = -\frac14 \int_0^{2C} \sign(\eta-\theta)e^{-\vert \Y(t,\eta)-\Y(t,\theta)\vert } (2(\U^2-\P)\Y_\eta(t,\theta)+1) d\theta, 
\end{align*}
and
\begin{align*}
\So(t,\eta)&= \int_\Real e^{-\vert \Y(t,\eta)-y\vert } (\frac23 u^3-u_xp_x-2pu)(t,y)dy\\
& = \int_{0}^{2C} e^{-\vert \Y(t,\eta)-\Y(t,\theta)\vert} (\frac23 \U^3\Y_\eta-\U_\eta\Q-2\P\U\Y_\eta)(t,\theta)d\theta,
\end{align*}
with
\begin{equation}\label{eq:Pdef}
\P(t,\eta)=p(t,\Y(t,\eta)).
\end{equation}
It is then natural, in order to close the system \eqref{hejkatrin}, that besides the quantities $u$ and $\mu$, also $p$ in the new variables must be considered. One main reason being that these three quantities turn up in the definition of $G$. We already computed before that 
\begin{align*}
p_t(t,x)&=-\frac12 \int_\Real \sign(x-y)e^{-\vert x-y\vert} (\frac16 u^3-pu-\frac12 uF_x)(t,y)dy\\ \nn
& \quad -\frac12 \int_\Real e^{-\vert x-y\vert} up_x(t,y)dy\\ \nn
&= \frac12 \int_\Real \sign(x-y) e^{-\vert x-y\vert}(\frac13 u^3+\frac 12uG_x)(t,y)dy\\ \nn
& \quad -\frac12 \int_\Real e^{-\vert x-y\vert} up_x(t,y)dy,
\end{align*}
where we used that 
\begin{align*}
G_x(t,x)=2p(t,x)-u^2(t,x)+F_x(t,x).
\end{align*}
Thus direct computations yield the following additional equation
\begin{align*}
\P_t(t,\eta)+(\frac23 \U^3+\So)\P_\eta(t,\eta)& = \Q\U(t,\eta)+\R(t,\eta),
\end{align*}
where 
\begin{align*}
\R(t,\eta)& = \frac14 \int_\Real \sign(\Y(t,\eta)-y)e^{-\vert \Y(t,\eta)-y\vert}(\frac23 u^3+uG_x)(t,y)dy \\
& \quad -\frac12 \int_\Real e^{-\vert \Y(t,\eta)-y\vert} up_x(t,y)dy\\
& = \frac14 \int_0^{2C}\sign(\eta-\theta)e^{-\vert \Y(t,\eta)-\Y(t,\theta)\vert} (\frac23 \U^3\Y_\eta+\U)(t,\theta)d\theta\\
& \quad -\frac12 \int_0^{2C} e^{-\vert \Y(t,\eta)-\Y(t,\theta)\vert}\U\Q\Y_\eta(t,\theta) d\theta.
\end{align*}
We summarize the  result in the following proposition.
%-------------------
\begin{proposition}\label{prop:smooth}
Let $(u,\mu)$ denote a smooth solution of \eqref{eq:CHbasic}. Define $\Y$ by \eqref{eq:Ydef}, $\U$ by \eqref{eq:Udef}, and $\P$ by \eqref{eq:Pdef}. Then the following system of differential equations holds
\begin{subequations}  \label{eq:evolutionPROP}
\begin{align}
\Y_t(t,\eta)+(\frac23 \U^3+\So)\Y_\eta(t,\eta)& =\U(t,\eta),\\
\U_t(t,\eta)+(\frac23 \U^3+\So)\U_\eta(t,\eta)& = -\Q(t,\eta),\\
\P_t(t,\eta)+(\frac23 \U^3+\So)\P_\eta(t,\eta)& = \Q\U(t,\eta)+\R(t,\eta),
\end{align}
\end{subequations}
where 
\begin{subequations}
\begin{align}
\Q(t,\eta)& = -\frac14 \int_0^{2C} \sign(\eta-\theta)e^{-\vert \Y(t,\eta)-\Y(t,\theta)\vert } (2(\U^2-\P)\Y_\eta(t,\theta)+1) d\theta,\\
\So(t,\eta)&= \int_{0}^{2C} e^{-\vert \Y(t,\eta)-\Y(t,\theta)\vert} (\frac23 \U^3\Y_\eta-\U_\eta\Q-2\P\U\Y_\eta)(t,\theta)d\theta,\\
\R(t,\eta)& =  \frac14 \int_0^{2C}\sign(\eta-\theta)e^{-\vert \Y(t,\eta)-\Y(t,\theta)\vert} (\frac23 \U^3\Y_\eta+\U)(t,\theta)d\theta \notag\\
& \quad -\frac12 \int_0^{2C} e^{-\vert \Y(t,\eta)-\Y(t,\theta)\vert}\U\Q\Y_\eta(t,\theta) d\theta.  
\end{align}
\end{subequations}
\end{proposition}
%--------------------------

\vskip12pt

In the next section we will derive this system of equations also in the general case without assuming smoothness of the quantities involved, see  \eqref{eq:evolutionTHEOREM}.

Let us finish this section by checking some properties of the system \eqref{eq:evolutionPROP}, which will also hold in the case of weak conservative solutions as we will see in the next section. 

The quantity 
$\frac23 \U^3+\So$ is the velocity field of the three equations in \eqref{eq:evolutionPROP}. Instead of applying a characteristic method to estimate the solutions to this system, we will perform integration by parts by which the $\eta$-derivative of this quantity will naturally appear. 

\begin{lemma}
Given $(u,\mu)$ a smooth solution of \eqref{eq:CHbasic}, then the solution to \eqref{eq:evolutionPROP} satisfies 
$$
\left|(\frac23 \U^3+\So)_\eta\right| \leq \bigO(1).
$$
\end{lemma}
\begin{proof}
Since $G(t,\Y(t,\eta))=2p_x(t,\Y(t,\eta))+2F(t, \Y(t,\eta))=\eta$, we have due to the smoothness that
\begin{equation*}
G_x(t,\Y(t,\eta))\Y_\eta(t,\eta)=(2p-u^2+F_x)(t,\Y(t,\eta))\Y_\eta(t,\eta)=1.
\end{equation*}
Since $(2p-u^2)(t,x)\geq 0$ due to \eqref{eq:2Ppluss} and $(F_x-u^2)(t,x)\geq 0$ due to \eqref{deff}, we have that 
\begin{align*}
2\P\Y_\eta(t,\eta)&\leq (2p-u^2+F_x)(t,\Y)\Y_\eta(t,\eta) = 1, \\
\U^2\Y_\eta(t,\eta)&\leq 2\P\Y_\eta(t,\eta)\le 1, \\
2\vert \U\U_\eta(t,\eta)\vert&=2 \vert u u_x(t,\Y) \Y_\eta(t,\eta)\vert\le (u^2+u_x^2)(t,\Y) \Y_\eta(t,\eta)\\
&=  F_x(t,\Y)\Y_\eta(t,\eta)
\leq (2p-u^2+F_x)(t,\Y)\Y_\eta(t,\eta)= 1.
\end{align*}
We conclude that
\begin{equation}  \label{eq:key_est}
\P\Y_\eta(t,\eta)\leq \frac12 \, , \qquad \vert \U\U_\eta(t,\eta)\vert \leq \frac12,\quad \text{ and }\quad \U^2\Y_\eta(t,\eta)\leq 1.
\end{equation}
From the fact that the energy is conserved, it follows that $u(t)\in H^1(\Real)$ for all $t\geq0$ and 
$$
\| u(t,\cdot) \|_{L^\infty(\Real)}\leq \sqrt{C}.
$$
Therefore, the term $\U^2 \U_\eta$ is bounded by $\bigO(1)$.

Furthermore, $\So_\eta(t,\eta)$ is bounded. In particular, one can establish that 
\begin{equation*}
\So_\eta(t,\eta)\leq \bigO(1) \P\Y_\eta(t,\eta),
\end{equation*}
which is going to play a key role. Indeed, by definition one has 
\begin{equation*}
S(t,x)= \int_\Real e^{-\vert x-y\vert } (\frac23 u^3-u_xp_x-2pu)(t,y)dy,
\end{equation*}
and hence 
\begin{equation*}
S_x(t,x)=-\int_\Real \sign(x-y)e^{-\vert x-y\vert} (\frac23 u^3-u_xp_x-2pu)(t,y)dy.
\end{equation*}
Our aim is to show that 
\begin{equation*}
S_x(t,x)\leq \bigO(1) p(t,x).
\end{equation*}
First of all note that we have 
\begin{align*}
\left| \int_\Real \sign(x-y)e^{-\vert x-y\vert }\frac23 u^3(t,y)dy\right| 
& \leq \frac83 \norm{u(t,\dott)}_{L^\infty(\Real)}\frac14 \int_\Real e^{-\vert x-y\vert } u^2(t,y)dy\\
& \leq \bigO(1) p(t,x).
\end{align*}
Moreover, 
\begin{align*}
\left| \int_\Real \sign(x-y) e^{-\vert x-y\vert} (u_xp_x+2pu)(t,y)dy\right|
& \leq  \frac12 \int_\Real e^{-\vert x-y\vert } (2u^2+u_x^2+3p^2)(t,y)dy \\
& \leq 2 p(t,x)+ \frac32 \int_\Real e^{-\vert x-y\vert }  p^2(t,y)dy,
\end{align*}
since $\vert p_x(t,x)\vert \leq p(t,x)$.
Thus it remains to show that the last term can be bounded by a multiple of $p(t,x)$. Our reasoning will be based on integration by parts and the fact that 
\begin{equation*}
p(t,x)-p_{xx}(t,x)=\frac12 u^2(t,x)+\frac12 F_x(t,x).
\end{equation*}
Indeed, first direct computations yield
\begin{align*}
\int_\Real e^{-\vert x-y\vert } p^2(t,y)dy &= \int_\Real e^{-\vert x-y\vert } p(\frac12 u^2+\frac12 F_x)(t,y)dy + \int_\Real e^{-\vert x-y\vert } p p_{xx}(t,y)dy\\
& = I_1(t,x)+I_2(t,x).
\end{align*}
Since $p(t,x)\leq \frac12 \int_\Real F_x(t,y)dy\leq \frac12 C$, we have that 
\begin{equation*}
I_1(t,x)\leq \norm{p(t,\dott)}_{L^\infty(\Real)}\frac12\int_\Real e^{-\vert x-y\vert } (u^2+F_x)(t,y)dy\leq 2\norm{p(t,x)}_{L^\infty(\Real)}p(t,x).
\end{equation*}
As far as $I_2$ is concerned, we have 
\begin{align*}
\int_{-\infty}^x e^{y-x}pp_{xx}(t,y)dy& =pp_x(t,x)-\int_{-\infty}^x e^{y-x}(pp_x+p_x^2)(t,y)dy\\
& = pp_x(t,x)-\frac12 p^2(t,x)+\int_{-\infty}^x e^{y-x}(\frac12 p^2-p_x^2)(t,y)dy,
\end{align*}
and 
\begin{align*}
\int_{x}^\infty e^{x-y} pp_{xx}(t,y)dy & =-pp_x(t,x)-\int_x^\infty e^{x-y} (p_x^2-pp_x)(t,y)dy\\
& =-pp_x(t,x)-\frac12 p^2(t,x)+\int_x^\infty e^{x-y} (\frac12 p^2-p_x^2)(t,y)dy.
\end{align*}
Thus
\begin{align*}
I_2(t,x)&=-p^2(t,x)+\int_\Real e^{-\vert x-y\vert} (\frac12 p^2-p_x^2)(t,y) dy
\end{align*}
and subsequently
\begin{align*}
\int_\Real e^{-\vert x-y\vert} p^2(t,y)dy&\leq 2\norm{p(t,\dott)}_{L^\infty(\Real)}p(t,x)-p^2(t,x)+\int_\Real e^{-\vert x-y\vert} (\frac12 p^2-p_x^2)(t,y)dy\\
& \leq 2\norm{p(t,\dott)}_{L^\infty(\Real)}p(t,x) +\int_\Real e^{-\vert x-y\vert} (\frac12 p^2-p_x^2)(t,y)dy.
\end{align*}
Reshuffling the terms, we end up with 
\begin{align}\nn
\frac12 \int_\Real e^{-\vert x-y\vert } p^2(t,y)dy
& \leq \frac12 \int_\Real e^{-\vert x-y\vert } (p^2+p_x^2)(t,y)dy\\ \label{rel:p2p}
 &\leq 2\norm{p(t,\dott)}_{L^\infty(\Real)}p(t,x)\leq Cp(t,x),
\end{align}
showing the desired estimate.
\end{proof}

Next, we show that all properties seen in this section for smooth solutions remain true for weak conservative solutions to \eqref{eq:CHbasic}.

%%%%%%%%%%%%%%%%%%%%%%%
%%%%%%%%%%%%%%%%%%%%%%%
%%%%%%%%%%%%%%%%%%%%%%%

\section{Rigorous Transformation: weak conservative solutions} \label{sec:general}
To accommodate for the wave breaking of the solutions, it has turned out to be advantageous to rewrite the Camassa--Holm equation from the original Eulerian variables into Lagrangian variables, see \cite{BC, HR}. We will show that the system of equations obtained in Proposition \ref{prop:smooth} holds for the weak conservative solutions introduced in \cite{HR}. With this aim in mind let us start by summarizing their approach, which uses the different adopted convention followed in this work for cumulative distribution functions to be continuous from the left and with limit from the right (caglad) at all points $x\in\mathbb{R}$. Therefore, both $F(t,\cdot)$ and $G(t,\cdot)$ are nondecreasing and caglad functions.

Given some initial data $(u_0,\mu_0)$, the corresponding initial data in Lagrangian coordinates is then given by 
\begin{subequations}
\begin{align}
y(0,\xi)&=\sup\{ x\mid x+F_0(x)<\xi\},\\
H(0,\xi)&= \xi-y(0,\xi),\\
U(0,\xi)&= u(0,y(0,\xi)),
\end{align}
\end{subequations}
and $(y(t,\dott), U(t,\dott), H(t,\dott))$ are the solutions of 
\begin{subequations}\label{sys:Lagrange}
\begin{align}
y_t(t,\xi)&=U(t,\xi),\\
U_t(t,\xi)&= -Q(t,\xi),\\
H_t(t,\xi)&= (U^3-2PU)(t,\xi),
\end{align}
\end{subequations}
where 
\begin{subequations}
\begin{align}
P(t,\xi)&=\frac14 \int_\Real e^{-\vert y(t,\xi)-y(t,\sigma)\vert} (U^2y_\xi+H_\xi)(t,\sigma)d\sigma, \label{eq:Pest}\\
Q(t,\xi)&=-\frac14 \int_\Real \sign(\xi-\sigma) e^{-\vert y(t,\xi)-y(t,\sigma)\vert}(U^2y_\xi+H_\xi)(t,\sigma)d\sigma.
\end{align}
\end{subequations}
Moreover, the relation between $P$ and $Q$ is given by
\begin{subequations}  \label{eq:PQ_deriv}
\begin{align}
P_\xi(t,\xi)&= Q(t,\xi)y_\xi(t,\xi),\label{eq:PQ_deriv1} \\
Q_\xi(t,\xi)&= (P-\frac12 U^2)y_\xi(t,\xi)-\frac12 H_\xi(t,\xi).\label{auxq}
\end{align}
\end{subequations}
Introduce the function 
\begin{align*}
I(t,\xi)& =\int_{-\infty}^\xi (2P-U^2)y_\xi(t,\sigma) d\sigma =\int_{-\infty}^\xi (2Q_\xi+H_\xi)(t,\sigma)d\sigma\\ \nn
& = 2Q(t,\xi)+H(t,\xi),
\end{align*}
where we used that 
\begin{equation*}
\lim_{\xi\to -\infty} Q(t,\xi)=0=\lim_{\xi\to -\infty} H(t,\xi),
\end{equation*}
which follows from the definition of $H(t,\xi)$.
The relation between $H$  and $F$ is given by  
\begin{equation*}
 F(t, y(t,\xi))\leq H(t,\xi)\leq F(t,y(t,\xi)+).
\end{equation*}
Here we have introduced the common notation
\begin{equation} \label{eq:notation}
 \Phi(x\pm)=\lim_{\epsilon\downarrow 0} \Phi(x\pm\epsilon).
\end{equation}
Notice that $I(t,\xi)+ H(t,\xi)$ is the Lagrangian counterpart to the function $G(t,x)$. To convince oneself that this is really the case, one should take a quick look back first. The function $G(t,x)$ was defined as 
\begin{equation*}
G(t,x)=\int_{-\infty}^x (2p-u^2)(t,y) dy+ F(t,x).
\end{equation*}
Thus, whenever $F(t,x)$ has a jump of height $\alpha$ at a point $\bar x$, i.e., $\mu(t,\{\bar x\})=\alpha$, then also $G(t,x)$ has a jump of height $\alpha$ at $\bar x$, since the function $2p-u^2$ is continuous. Furthermore, the point $\bar x$ in Eulerian coordinates is mapped to some maximal interval $[\xi_l, \xi_r]$ in Lagrangian coordinates, on which $y_\xi(t,\xi)=0$ and $H_\xi(t,\xi)=1$ for a specific choice of a relabeling function. 
In fact, $H_\xi(t,\xi)>0$ for all $\xi \in [\xi_l, \xi_r]$ as proven in \cite[Theorem 4.2]{HR} and \cite[Definition 2.6]{HR}. 
Thus a close look at $Q(t,\xi)$ reveals that $Q(t,\xi)=Q(t,\xi_l)-\frac12 \int_{\xi_l}^\xi H_\xi(t,\sigma) d\sigma$ for all $\xi\in [\xi_l, \xi_r]$ and hence 
\begin{align*}
I(t,\xi)& = 2Q(t,\xi)+H(t,\xi)\\
& =2Q(t,\xi_l)-\int_{\xi_l}^\xi H_\xi(t,\sigma)d\sigma+H(t,\xi_l)+\int_{\xi_l}^\xi H_\xi(t,\sigma)d\sigma\\
& = 2Q(t,\xi_l)+H(t,\xi_l)=I(t,\xi_l).
\end{align*}
To put it short we have that $I(t,\xi)=I(t, \xi_l)$ for all $\xi\in [\xi_l, \xi_r]$. This allows us now to follow a similar approach as for the Hunter--Saxton (HS) equation in \cite{BHR,CGH}. Therefore introduce 
\begin{equation} \label{eq:J}
J(t,\xi)= I(t,\xi)+H(t,\xi)= \int_{-\infty}^\xi (2P-U^2)y_\xi(t,\sigma)d\sigma+H(t,\xi),
\end{equation}
and observe that for all solutions except the zero solution $J(t,\xi)$ is strictly increasing and continuous. In more detail, one has for all solutions except the zero solution that $P(t,\xi)\not=0$ for all $\xi \in \Real$, since $y_\xi+H_\xi>0$ almost everywhere due to \cite[Definition 2.6]{HR}. Moreover, if $H_\xi(t,\bar \xi)=0$ for some $\bar \xi$ one has that $U(t,\bar \xi)=0$ and $y_\xi(t,\bar \xi)\not =0$, since $y_\xi H_\xi (t,\bar\xi)=U^2y_\xi(t,\bar\xi)+U_\xi^2(t,\bar\xi)$ almost everywhere due to \cite[Definition 2.6]{HR}, and hence the $\xi$-derivative of $J(t,\xi)$ is strictly positive at the point $\bar\xi$.

\begin{lemma}\label{newcoor}
Given $\Y(t,\eta)= \sup \{ x\mid G(t,x)<\eta\}$, then
\begin{equation}\label{eq:Yy}
\Y(t,\eta)= y(t, l(t,\eta)),
\end{equation}
where we have introduced $l(t,\dott)\colon [0,2C]\to \Real$ by 
\begin{equation} \label{def:l}
l(t,\eta)=\sup\{\xi\mid J(t,\xi)<\eta\}.
\end{equation}
\end{lemma}

\begin{proof}
For each time $t$ we have 
\begin{equation*}
y(t,\xi)=\sup \{ x\mid x+F(t,x)< y(t,\xi)+H(t,\xi)\},
\end{equation*}
which implies that 
\begin{equation*}
y(t,\xi)+F(t,y(t,\xi))\leq y(t,\xi)+H(t,\xi)\leq y(t,\xi)+F(t,y(t,\xi)+).
\end{equation*}
Moreover, one has that $G(t,x)-F(t,x)$ is continuous, and, in particular,
\begin{align*}
G(t,y(t,\xi))-F(t,y(t,\xi))& =\int_{-\infty}^{y(t,\xi)} (2p-u^2)(t,y)dy\\ \nn
& = \int_{-\infty}^\xi (2P-U^2)y_\xi(t,\sigma)d\sigma=I(t,\xi).
\end{align*}
Thus one has 
\begin{equation*}
y(t,\xi)+G(t,y(t,\xi))\leq y(t,\xi)+J(t,\xi)\leq y(t,\xi)+ G(t,y(t,\xi)+).
\end{equation*}
Subtracting $y(t,\xi)$ in the above inequality, we end up with 
\begin{equation*}
G(t, y(t,\xi))\leq J(t,\xi)\leq G(t, y(t,\xi)+) \quad \text{ for all } \xi \in \Real.
\end{equation*}
 Comparing the last equation and \eqref{def:l}, we have 
 \begin{equation*}
 G(t, y(t, l(t,\eta)))\leq J(t, l(t,\eta))=\eta\leq G(t, y(t, l(t,\eta))+).
 \end{equation*}
 Since $y(t,\dott)$ is surjective and non-decreasing, we end up with 
 \begin{equation}\label{eq:consistency}
 \Y(t,\eta)=\sup\{x\mid G(t,x)<\eta\}=y(t,l(t,\eta)),
 \end{equation}
thereby proving \eqref{eq:Yy}.
\end{proof}

In the next step we want to establish rigorously the corresponding system of differential equations. Hence we first have to establish that the function $l(t,\eta)$ is differentiable, both with respect to time and space.

\subsection{The differentiability of $Q(t,\xi)$}\label{diffq}
The differentiability of $Q(t,\xi)$ with respect to $\xi$ has been proven in \cite{HR} and  $Q_\xi(t,\xi)$ is given by \eqref{auxq}.
Thus it is left to establish the differentiability with respect to time of $Q(t,\xi)$. To be more precise, we are going to establish the Lipschitz continuity of $Q(t,\xi)$ with respect to time. 

Let us recall that since $U(t,\dott)\in H^1(\Real)$,  one has in particular that $U(t,\dott)\in L^\infty(\Real)$. Moreover, direct calculations yield 
\begin{align}\label{LI1}
U^2(t,\xi)&= U^2(t,\xi)-U^2(t,-\infty)=\int_{-\infty}^\xi 2UU_\xi(t,\sigma)d\sigma\\ \nn
& \leq \int_{-\infty}^\xi H_\xi(t,\sigma)d\sigma\leq H(t,\infty)=C,
\end{align}
and we end up with 
\begin{equation*}
\norm{U(t,\dott)}_{L^\infty(\Real)}\leq \sqrt{C}\quad \text{ for all } t\in\Real.
\end{equation*}

Here we used \textit{Xavier's relation} from \cite{HR} which asserts that 
\begin{equation}\label{eq:xavier}
U^2y_\xi^2(t,\xi)+U_\xi^2(t,\xi)=y_\xi H_\xi(t,\xi),
\end{equation} 
and hence 
\begin{subequations}\label{eq:stand:est}
\begin{align}
\vert U_\xi(t,\xi)\vert &\leq \sqrt{y_\xi H_\xi}(t,\xi)\leq \frac12 (y_\xi+H_\xi)(t,\xi),\label{stand:est2} \\
\vert Uy_\xi(t,\xi)\vert &\leq \sqrt{y_\xi H_\xi}(t,\xi)\leq \frac12 (y_\xi+H_\xi)(t,\xi), \label{stand:est3} \\
\vert UU_\xi(t,\xi) \vert &\le \frac12 H_\xi(t,\xi).\label{stand:est4}
\end{align}
\end{subequations}

Similar considerations apply for $P(t,\xi)$. Indeed one has
\begin{align}\label{LI2}
0\leq P(t,\xi)&=\frac14 \int_\Real e^{-\vert y(t,\sigma)-y(t,\xi)\vert} (U^2y_\xi+H_\xi)(t,\sigma)d\sigma\\ \nn
& \leq \frac14 \int_\Real 2H_\xi(t,\sigma)d\sigma=\frac12 C.
\end{align}
Since $\vert Q(t,\xi)\vert \leq P(t,\xi)$, we end up with 
\begin{equation}\label{LI23}
\norm{Q(t,\dott)}_{L^\infty(\Real)}\leq \frac12 C, \quad \norm{P(t,\dott)}_{L^\infty(\Real)} \leq \frac12 C \quad \text{ for all } t\in \Real.
\end{equation}

Direct calculations, using \eqref{eq:xavier} and \eqref{eq:stand:est}, yield
\begin{align*}
-\bigO(1)(y_\xi(t,\xi)+H_\xi(t,\xi))&\leq (y_\xi(t,\xi)+H_\xi(t,\xi))_t\\
& = U_\xi(t,\xi)+3U^2U_\xi(t,\xi)-2QUy_\xi(t,\xi)-2PU_\xi(t,\xi)\\ 
&\leq \bigO(1)(y_\xi(t,\xi)+H_\xi(t,\xi)),
\end{align*}
since $U(t,\xi)$, $P(t,\xi)$, and $Q(t,\xi)$ are uniformly bounded with respect to both space and time due to \eqref{LI1}, \eqref{LI2}, and \eqref{LI23}.
Thus, we have for $s<t$ that
\begin{equation*}
(y_\xi(s,\xi)+H_\xi(s,\xi))e^{-\bigO(1)(t-s)}\leq y_\xi(t,\xi)+H_\xi(t,\xi)\leq (y_\xi(s,\xi)+H_\xi(s,\xi))e^{\bigO(1)(t-s)}
\end{equation*}
or equivalently 
\begin{align}\notag
(y_\xi(t,\xi)+H_\xi(t,\xi))e^{-\bigO(1)(t-s)}&\leq (y_\xi(s,\xi)+H_\xi(s,\xi))\\
&\leq (y_\xi(t,\xi)+H_\xi(t,\xi))e^{\bigO(1)(t-s)}.\label{stand:est}
\end{align} 

Recall that 
\begin{align*}
Q(t,\xi)&=-\frac14 \int_{-\infty}^\xi e^{y(t,\sigma)-y(t,\xi)} (U^2y_\xi+H_\xi)(t,\sigma) d\sigma\\
& \quad + \frac14 \int_\xi^\infty e^{y(t,\xi)-y(t,\sigma)} (U^2y_\xi+H_\xi)(t,\sigma)d\sigma\\
& =: \frac14Q_1(t,\xi)+\frac14Q_2(t,\xi).
\end{align*}
We are only establishing the Lipschitz continuity with respect to time for $Q_1(t,\xi)$, since the argument for $Q_2(t,\xi)$ follows the same lines. Let $t<\hat t$.  Then one has 
\begin{align*}\nn
\vert Q_1(t,\xi)-Q_1(\hat t,\xi)\vert &\leq \int_{-\infty}^\xi \vert e^{y(t,\sigma)-y(t,\xi)}-e^{y(\hat t, \sigma)-y(\hat t, \xi)}\vert   (U^2y_\xi+H_\xi)(t,\sigma)d\sigma\\ \nn
& \quad + \int_{-\infty}^\xi e^{y(\hat t,\sigma)-y(\hat t, \xi)} \vert (U^2y_\xi+H_\xi)(t,\sigma)-(U^2y_\xi+H_\xi)(\hat t, \sigma)\vert d\sigma\\ 
& =: I_1+I_2.
\end{align*}

As far as $I_1$ is concerned, we observe, using Lemma\ref{lemma:enkel} (i), that
\begin{align*}
\vert e^{y(t,\sigma)-y(t,\xi)}-e^{y(\hat t, \sigma)-y(\hat t, \xi)}\vert & \leq \vert y(t,\sigma)-y(\hat t, \sigma)\vert +\vert y(t, \xi)-y(\hat t, \xi)\vert \\
& \leq \int_t^{\hat t} \big(\vert U(s,\sigma)\vert +\vert U(s,\xi)\vert\big) ds \leq \bigO(1) \vert t-\hat t\vert,
\end{align*}
where we used that both $e^{y(\hat t,\sigma)-y(\hat t,\xi)}$ and $e^{y(t,\sigma)-y(t,\xi)}$ are bounded from above by one, and that $U$ can be uniformly bounded both with respect to space and time due to \eqref{LI1}.
Thus we end up with 
\begin{align*}
I_1& = \int_{-\infty}^\xi \vert e^{y(t,\sigma)-y(t,\xi)}-e^{y(\hat t, \sigma)-y(\hat t, \xi)}\vert   (U^2y_\xi+H_\xi)(t,\sigma)d\sigma\\ \nn
& \leq \bigO(1)\vert t-\hat t\vert \int_{-\infty}^\xi 2H_\xi(t,\sigma)d\sigma \leq \bigO(1) \vert t-\hat t\vert,
\end{align*}
where we used that $U^2y_\xi\leq H_\xi$. 

Thus it remains to establish a similar estimate for $I_2$. The idea is to use a similar strategy combined with \eqref{stand:est2} and \eqref{stand:est}. We have 
\begin{align*}
I_2& =  \int_{-\infty}^\xi e^{y(\hat t,\sigma)-y(\hat t, \xi)} \vert (U^2y_\xi+H_\xi)(t,\sigma)-(U^2y_\xi+H_\xi)(\hat t, \sigma)\vert d\sigma\\
& \leq \int_{-\infty}^\xi e^{y(\hat t, \sigma)-y(\hat t, \xi)}\int_t^{\hat t } \vert 4U^2U_\xi-4QUy_\xi-2PU_\xi\vert (s, \sigma) ds d\sigma\\
& \leq \int_{-\infty}^\xi e^{y(\hat t, \sigma)-y(\hat t, \xi)} \bigO(1)\int_{t}^{\hat t} (y_\xi+H_\xi) (s, \sigma)ds d\sigma\\
& \leq \bigO(1) \int_{-\infty} ^\xi e^{y(\hat t, \sigma)-y(\hat t, \xi)} \int_t^{\hat t } e^{\bigO(1)(\hat t-s)} (y_\xi+H_\xi) (\hat t, \sigma) ds d\sigma\\
& \leq \bigO(1)\int_{-\infty}^\xi e^{y(\hat t, \sigma)-y(\hat t, \xi)}(y_\xi+H_\xi)(\hat t, \sigma) \int_t^{\hat t} e^{\bigO(1)(\hat t-s)} ds d\sigma\\
&\leq \bigO(1) e^{\bigO(1) (\hat t-t)}\vert\hat t-t\vert \leq \bigO(1)\vert\hat t-t\vert,
\end{align*}
under the assumption that $\vert \hat t-t\vert\leq 1$ (or in general bounded). Thus we established that 
\begin{equation*}
\vert Q_1(t,\xi)-Q_1(\hat t, \xi)\vert \leq \bigO(1)\vert \hat t-t\vert
\end{equation*}
for all pairs $t$, $\hat t$, such that $\vert \hat t-t\vert \leq 1$. Furthermore, one has 
\begin{equation*}
\vert Q(t,\xi)-Q(\hat t,\xi)\vert \leq \bigO(1)\vert \hat t-t\vert 
\end{equation*}
for all pairs $t$, $\hat t$, such that $\vert \hat t-t\vert \leq 1$. Thus for fixed $\xi$, the function $Q$ is locally Lipschitz with respect to time, and hence differentiable almost everywhere by Rademacher's theorem.

A close look at the above argument reveals that we cannot only differentiate $Q(t,\xi)$ with respect to time, but also that we can apply the dominated convergence theorem, which yields
\begin{align}\label{tder:Q}
Q_t(t,\xi)&=-\frac14 \int_\Real \frac{d}{dt} \left(\sign(\xi-\sigma)e^{-\vert y(t,\xi)-y(t,\sigma)\vert} (U^2y_\xi+H_\xi)(t,\sigma) \right)d\sigma\\ \nn
&=-\frac23 U^3(t,\xi)+2PU(t,\xi)\\ \nn
& \quad +\frac12 \int_\Real e^{-\vert y(t,\xi)-y(t,\sigma)\vert} (\frac23 U^3y_\xi-2PUy_\xi-QU_\xi)(t,\sigma) d\sigma
\end{align}
after some integrations by parts. 

Next we want to show that $Q_t(t,\xi)$ can be uniformly bounded by a constant $\bigO(1)$. 
Observe that  \eqref{LI1} and \eqref{LI2} imply
\begin{align}\label{LI3}
 \left| \int_\Real \right. & \left. e^{-\vert y(t,\sigma)-y(t,\xi)\vert} (\frac23 U^3y_\xi-2PUy_\xi-QUy_\xi)(t,\sigma) d\sigma\right| \\ \nn
 & \leq \bigO(1) \left(\int_\Real e^{-\vert y(t,\sigma)-y(t,\xi)\vert } y_\xi(t,\sigma)d\sigma\right) \leq \bigO(1).  
 \end{align}
Thus recalling \eqref{tder:Q} and combining \eqref{LI1}, \eqref{LI2}, and \eqref{LI3}, we finally end up with
\begin{equation}\label{tder:est:Q}
\norm{Q_t(t,\dott)}_{L^\infty(\Real)}\leq \bigO(1) \quad \text{ for all } t\in \Real.
\end{equation}
This completes the argument regarding the differentiability of $Q$. Notice that a direct application of the dominated convergence theorem using formula \eqref{tder:Q} shows that $Q_t$ is a continuous function in time, under the assumption that $P_t$ exists and satisfies an estimate similar to \eqref{tder:est:Q}.

Next we focus on the differentiability with respect to time of $P$. 
An analogous argument to the one for $Q$ leads to the differentiability in time of $P$. Let us show that the derivative of $P$ with respect to time exists by applying the dominated convergence theorem to 
\begin{equation}\label{quot}
\lim_{s\to t} \frac{P(t,\xi)-P(s,\xi)}{t-s},
\end{equation}
where $\xi$ is chosen such that $y$ is differentiable with respect to time.
Since 
\begin{equation*}
\frac{d}{dt} \big(e^{-\vert y(t,\xi)-y(t,\sigma)\vert } (U^2y_\xi+H_\xi)(t,\sigma)\big)
\end{equation*}
exists almost everywhere, it is left to show that we can find a function in $L^1(\Real)$, which 
bounds the integrand of \eqref{quot}  uniformly in $s$.
Therefore observe that we can write, using \eqref{sys:Lagrange},  for $s<t$, that
\begin{align*}
P(t,\xi)-P(s,\xi)& = \frac14 \int_{-\infty}^\xi \int_s^t e^{y(l,\sigma)-y(l,\xi)} (U(l,\sigma)-U(l,\xi)) (U^2y_\xi+H_\xi)(l,\sigma)dl d\sigma\\
& \quad + \frac14 \int_{\xi}^{\infty} \int_s^t e^{y(l,\xi)-y(l,\sigma)} (U(l,\xi)-U(l,\sigma)) (U^2y_\xi+H_\xi)(l,\sigma)dl d\sigma\\
& \quad +\frac14 \int_\Real \int_s^t e^{-\vert y(l,\sigma)-y(l,\xi)\vert} (4U^2U_\xi-4UQy_\xi-2PU_\xi)(l,\sigma)dl d\sigma\\
& = I_1+I_2+I_3.
\end{align*}
As far as the first term is concerned, observe that  \eqref{stand:est} implies that 
\begin{align*}
& \left| \int_s^t e^{y(l,\sigma)-y(l,\xi)} (U(l,\sigma)-U(l,\xi))(U^2y_\xi+H_\xi)(l,\sigma)dl \right| 
\\
& \qquad\qquad \qquad\qquad   \leq \bigO(1) \int_s^t e^{y(l,\sigma)-y(l,\xi)}|(U^2y_\xi+H_\xi)(l,\sigma)|dl  \\
& \qquad\qquad \qquad \qquad  \leq \bigO(1)e^{\bigO(1)\vert t-s\vert } e^{y(t,\sigma)-y(t,\xi)} (y_\xi+H_\xi)(t,\sigma)\vert t-s\vert \\
& \qquad\qquad \qquad \qquad  \leq \bigO(1)  e^{y(t,\sigma)-y(t,\xi)} (y_\xi+H_\xi)(t,\sigma)\vert t-s\vert,  
\end{align*} 
if we assume that $\vert t-s\vert \leq 1$ or any other fixed constant. Moreover, the function in the last line belongs to $L^1(\Real)$.
The remaining two terms can be bounded by a function, which is of the same form and belongs to $L^1(\Real)$ uniformly in $s$. 
Thus all assumptions of the dominated convergence theorem are fulfilled, and we end up with 
\begin{align}\label{tder:P}
P_t(t,\xi)& = UQ(t,\xi)+\frac14 \int_\Real \sign(\xi-\sigma) e^{-\vert y(t,\xi)-y(t,\sigma)\vert } U(U^2y_\xi+H_\xi)(t,\sigma)d\sigma\\ \nn
& \quad + \frac14 \int_\Real e^{-\vert y(t,\xi)-y(t,\sigma)\vert} (4U^2U_\xi-4UQy_\xi-2PU_\xi)(t,\sigma) d\sigma\\ \nn
& = UQ(t,\xi)+\frac12\int_\Real \sign(\xi-\sigma) e^{-\vert y(t,\xi)-y(t,\sigma)\vert}U(\frac13 U^2y_\xi+Q_\xi+H_\xi)(t,\sigma)d\sigma\\ \nn
& \quad -\frac12 \int_\Real e^{-\vert y(t,\xi)-y(t,\sigma)\vert}UQy_\xi(t,\sigma) d\sigma,
\end{align} 
where we used integration by parts together with \eqref{auxq} in the last step.
Moreover, we have that 
\begin{align*}\nn
\vert P_t(t,\xi)\vert& \le \vert UQ(t,\xi)\vert+\frac14 \int_\Real e^{-\vert y(t,\xi)-y(t,\sigma)\vert } \vert U\vert (U^2y_\xi+H_\xi)(t,\sigma)d\sigma\\  \nn
& \quad + \frac14 \int_\Real e^{-\vert y(t,\xi)-y(t,\sigma)\vert} (4U^2\vert U_\xi\vert+4\vert UQ\vert y_\xi+2P\vert U_\xi\vert)(t,\sigma) d\sigma\\
&\leq  \norm{U(t,\dott)}_{L^\infty} P(t,\xi)+\norm{U(t,\dott)}_{L^\infty} \frac14\int_\Real e^{-\vert y(t,\xi)-y(t,\sigma)\vert} (U^2y_\xi+H_\xi)(t,\sigma)d\sigma\\\ \nn
& \quad + \frac14 \int_\Real e^{-\vert y(t,\xi)-y(t,\sigma)\vert } (2U^4y_\xi+3H_\xi+ 2U^2y_\xi+3P^2y_\xi)(t,\sigma) d\sigma\\ \nn
& \leq  \bigO(1)P(t,\xi)+\frac34 \int_\Real e^{-\vert y(t,\xi)-y(t,\sigma)\vert } P^2y_\xi(t,\sigma)d\sigma\\ \nn
& \quad + \frac14 \int_\Real e^{-\vert y(t,\xi)-y(t,\sigma)\vert } (2\norm{U(t,\dott)}_{L^\infty}^2U^2y_\xi+3H_\xi+ 2U^2y_\xi)(t,\sigma) d\sigma\\ \nn
& \leq \bigO(1)P(t,\xi) +\frac34 \int_\Real e^{-\vert y(t,\xi)-y(t,\sigma)\vert } P^2y_\xi(t,\sigma)d\sigma\\ \nn
& \leq \bigO(1) \Big(P(t,\xi) + \int_\Real e^{-\vert y(t,\xi)-y(t,\sigma)\vert } P^2y_\xi(t,\sigma)d\sigma\Big)
\end{align*}
due to \eqref{LI1}. It remains to show that
\begin{equation*}
\int_\Real e^{-\vert y(t,\xi)-y(t,\sigma)\vert } P^2y_\xi(t,\sigma)d\eta\leq \bigO(1) P(t,\xi).
\end{equation*}
The proof follows the same lines as the one of \eqref{rel:p2p} in Eulerian coordinates. 

%%%%%%%%%%%%%%%%%%%%%%%%%%%%%%%%%%%%%%

\subsection{The differentiability of $l(t,\eta)$}\label{diffl}
Finally, we can start thinking about the differentiability with respect to space and time of $l(t,\eta)$. Therefore, recall that the function $l(t,\eta)$ is defined as 
\begin{equation*}
l(t,\eta)=\sup\{\xi\in\Real \mid J(t,\xi)<\eta\}.
\end{equation*}

{\bf Differentiability with respect to $\eta$:}  For every $t\in\Real$, the function $J(t,\dott)=2Q(t,\dott)+2H(t,\dott)$ is strictly increasing and continuous, and hence differentiable almost everywhere with respect to $\xi$. Its inverse $J^{-1}(t,\dott) \colon [0,2C]\to \Real$ is therefore also strictly increasing and continuous, and hence differentiable almost everywhere. 
Since we have by definition that $J(t,l(t,\eta))=\eta$ for all $\eta$, it follows immediately that $J(t,l(t,\dott))$ is Lipschitz continuous with Lipschitz constant one and hence differentiable almost everywhere with respect to $\eta$. Since 
\begin{equation*}
J^{-1}(t, J(t,l(t,\eta)))=l(t,\eta) \quad \text{ for all } \eta,
\end{equation*}
we can finally conclude that $l(t,\eta)$ is differentiable almost everywhere with respect to $\eta$. In particular, one has,   
cf.~\eqref{eq:J}, that 
\begin{equation}\label{vik:estprep}
J_\xi(t,l(t,\eta))l_\eta(t,\eta)=2\P\Y_\eta(t,\eta)-\U^2\Y_\eta(t,\eta)+\Henergy_\eta(t,\eta)=1,
\end{equation}
with 
$$\Henergy(t,\eta)=H(t,l(t,\eta)),\quad  \P(t,\eta)=P(t,l(t,\eta)),\quad \text{ and }\quad \U(t,\eta)=U(t,l(t,\eta)).$$ These identities follow from equality \eqref{eq:consistency} since $\P(t,\eta)=p(t,\Y(t,\eta))$ 
and $P(t,\xi)=p(t,y(t,\xi))$, and analogously $\U(t,\eta)=u(t,\Y(t,\eta))$ 
and $U(t,\xi)=u(t,y(t,\xi))$.

Following the same argument as in the smooth case, see \eqref{eq:key_est}, we end up with 
\begin{equation}\label{vik:est}
0\leq 2\P\Y_\eta(t,\eta)\leq 1,\quad 2\vert \U\U_\eta(t,\eta)\vert \leq 1, \quad \text{ and }\quad 0\leq \U^2\Y_\eta (t,\eta)\leq 1.
\end{equation}

{\bf Differentiability with respect to $t$:}  As far as the differentiability of $l(t,\eta)$ with respect to time is concerned, the argument is a bit more involved. First of all note that for any time $t$ we have $J(t,l(t,\eta))=\eta$ and hence 
\begin{equation*}
J(t,l(t,\eta))=J(\tilde t, l(\tilde t,\eta))\quad \text{ for all } \eta \text{ and } \tilde t.
\end{equation*}
In particular, we can  conclude that 
\begin{align*}
J(t,l(t,\eta))-J(t, l(\tilde t, \eta))&=J(\tilde t, l(\tilde t, \eta))-J(t, l(\tilde t, \eta))\\ \nn
& =\int_{t}^{\tilde t} J_t(s, l(\tilde t, \eta))ds\\ \nn
& = 2\int_t^{\tilde t} \big(Q_t(s, l(\tilde t, \eta))+H_t(s,l(\tilde t, \eta))\big) ds.
\end{align*}
We already showed that $Q_t(t,\xi)$ exists and is given by \eqref{tder:Q}. However, a closer look reveals that $Q_t(t,\xi)$ is uniformly bounded both with respect to space and time and $Q_t(t,\xi)$ is continuous with respect to space for fixed $t$. The same holds for 
\begin{equation*}
H_t(t,\xi)=U^3(t,\xi)-2PU(t,\xi).
\end{equation*}
Thus we have that 
\begin{equation}\label{want:lim}
\lim_{\tilde t\to t} \frac{J(t,l(t,\eta))-J(t,l(\tilde t, \eta))}{t-\tilde t}=\lim_{\tilde t \to t} 2\int_{t}^{\tilde t} \frac{Q_t(s, l(\tilde t, \eta))+H_t(s, l(\tilde t, \eta))}{t-\tilde t}ds
\end{equation}
from a formal point of view. Thus it is left to establish that the limit exists and is finite. 

Since $J_t(t,\xi)=2Q_t(t,\xi) + 2H_t(t,\xi)$ and both $Q_t(t,\xi)$ and $H_t(t,\xi)$ can be bounded uniformly both in space and time by a constant only dependent on the total energy $C$, it follows that there exists a constant $\bigO(1)$ such that 
\begin{equation*}
\vert J_t(t,\xi)\vert \leq \bigO(1) \quad \text{ for all } t \text{ and } \xi.
\end{equation*}
Thus it follows that for all $s\in [t-\vert t-\tilde t\vert, t+ \vert t-\tilde t \vert]$, we have 
\begin{equation*}
J(s, \xi)\in [J(t,\xi)-\bigO(1)\vert \tilde t-t\vert , J(t,\xi)+\bigO(1)\vert \tilde t-t\vert ].
\end{equation*}
Accordingly we can conclude, since $J(t, l(t,\eta))=\eta$ for all $\eta$, that 
\begin{equation*}
J(s, l(t,\eta))\in [\eta-\bigO(1)\vert \tilde t-t\vert , \eta+ \bigO(1) \vert \tilde t-t\vert],
\end{equation*}
and hence 
\begin{equation}\label{L.time}
l(t,\eta)=l(s, \eta(s)) \quad \text{ for some } \eta(s)\in[\eta-\bigO(1)\vert \tilde t-t\vert , \eta+\bigO(1)\vert \tilde t-t\vert ],
\end{equation}
and, in particular, 
\begin{equation*}
\vert \eta-\eta(s)\vert \leq \bigO(1)\vert \tilde t-t\vert \text{ for all } s\in [t-\vert t-\tilde t\vert , t+\vert t-\tilde t\vert ].
\end{equation*}
Similar considerations yield that for all $s\in [\tilde t -\vert t-\tilde t\vert , \tilde t+\vert t-\tilde t \vert]$, we can find $\tilde \eta(s)$ such that 
\begin{equation*}
l(\tilde t, \eta)=l(s, \tilde \eta(s)) \quad \text{ and }\quad \vert \eta-\tilde \eta(s)\vert \leq \bigO(1)\vert t-\tilde t\vert.
\end{equation*}

Let us return to the integral we are actually interested in. Namely
\begin{align*}
\int_{t}^{\tilde t} J_t(s, l(\tilde t, \eta))ds
&=2\int_t^{\tilde t}  \big(Q_t(s, l(\tilde t, \eta))+H_t(s, l(\tilde t, \eta))\big)ds  \\
 &\qquad = 2\int_t^{\tilde t} \big(Q_t(s , l(t,\eta))+H_t(s, l(t,\eta))\big)ds\\
& \qquad\quad +2 \int_t^{\tilde t}\Big( (Q_t(s, l(\tilde t, \eta))+H_t(s, l(\tilde t, \eta)))\\
& \qquad\qquad\qquad -(Q_t(s, l(t, \eta))+H_t(s, l(t,\eta))) \Big)ds\\
&\qquad  = \tilde I_1+\tilde I_2.
\end{align*}
What we are hoping for, is that the second term $\tilde I_2$ is of order $o(\vert t-\tilde t\vert)$, and hence plays no role when computing \eqref{want:lim}. Therefore observe that 
\begin{align*}
H_{t\xi}(t,\xi)&= 3U^2U_\xi(t,\xi)-2QUy_\xi(t,\xi)-2PU_\xi(t,\xi),\\[2mm]
Q_{t\xi}(t,\xi)&= -2U^2U_\xi(t,\xi)+2QUy_\xi(t,\xi)+2PU_\xi(t,\xi)\\
& \quad -\frac12 \int_\Real \sign(\xi-\sigma) e^{-\vert y(t,\xi)-y(t,\sigma)\vert} 
(\frac23 U^3y_\xi-2PUy_\xi-QU_\xi)(t,\sigma)d\sigma y_\xi(t,\xi).
\end{align*}
Notice that the second derivative $Q_{t\xi}$ is well-defined by a dominated convergence argument as before for $Q_t$. Recalling \eqref{L.time}, we have 
\begin{align*}
\tilde I_2=&\int_{t}^{\tilde t} J_t(s, l(\tilde t, \eta))-J_t(s, l(t, \eta))ds\\
 &=2\int_t^{\tilde t} \Big((Q_t(s, l(\tilde t, \eta))+H_t(s, l(\tilde t, \eta)))-(Q_t(s, l(t,\eta))+H_t(s, l(t,\eta)))\Big)ds\\
& = 2\int_t^{\tilde t} \int^{l(\tilde t, \eta)}_{l(t, \eta)} (Q_{t,\xi}+H_{t,\xi})(s, z) dz ds\\
& = 2\int_t^{\tilde t} \int_{l(t, \eta)}^{l(\tilde t, \eta)} \Big( U^2U_\xi(s,z)\\
& \quad -\frac12 \int_\Real \sign(z-\sigma) e^{-\vert y(s, z)-y(s,\sigma)\vert} (\frac23 U^3y_\xi-2PUy_\xi-QU_\xi)(s,\sigma) d\sigma y_\xi(s,z)\Big) dz ds \\
& = 2\int_t^{\tilde t} \int_{l(s, \eta(s))}^{l(s, \tilde \eta(s))} \Big( U^2U_\xi(s,z)\\
& \quad -\frac12 \int_\Real \sign(z-\sigma) e^{-\vert y(s, z)-y(s,\sigma)\vert} (\frac23 U^3y_\xi-2PUy_\xi-QU_\xi)(s,\sigma) d\sigma y_\xi(s,z)\Big) dz ds \\
& = 2\int_t^{\tilde t} \int_{\eta(s)}^{\tilde \eta(s)} \U^2\U_\eta(s,m)dm ds -\int_t^{\tilde t} \int_{\eta(s)}^{\tilde\eta(s)} \int_0^{2C} \sign(m-n)e^{-\vert \Y(s,m)-\Y(s,n)\vert}\\
&\qquad\qquad\qquad\qquad\qquad\qquad\qquad\times(\frac23 \U^3\Y_\eta-2\P\U\Y_\eta-\Q\U_\eta)(s,n)dn\Y_\eta(s,m)dmds
\end{align*}
where $\eta$, $\eta(s)$, and $\tilde \eta(s)$ satisfy
\begin{equation*}
l(t,\eta)=l(s,\eta(s)) \quad \text{ and } \quad l(\tilde t, \eta)=l(s, \tilde \eta(s)),
\end{equation*}
and 
\begin{equation}\label{vik:est2}
\vert \eta-\eta(s)\vert\leq \bigO(1)\vert t-\tilde t\vert \quad \text{ and }\quad \vert \eta-\tilde \eta(s)\vert \leq \bigO(1) \vert t-\tilde t\vert.
\end{equation}
As far as the first term is concerned, we can apply \eqref{LI1}, \eqref{vik:est} and \eqref{vik:est2} as follows
\begin{multline}\label{vik:est4}
\left| 2\int_t^{\tilde t} \int_{\eta(s)}^{\tilde \eta(s)} \U^2\U_\eta(s,m)dmds\right| \\
 \leq \int_t^{\tilde t} \norm{\U(s,\dott)}_{L^\infty([0,2C])} \vert \eta(s)-\tilde\eta(s)\vert ds \leq \bigO(1) \vert t-\tilde t\vert ^2.
\end{multline}
As far as the second and last term is concerned, we want to apply the Cauchy--Schwarz inequality to the integral term at first. Indeed,
\begin{align*}
&\left| \int_{\eta(s)}^{\tilde\eta(s)} \int_0^{2C} \sign(m-n)e^{-\vert \Y(s,m)-\Y(s,n)\vert}\right.\\
&\qquad \qquad\qquad\qquad \left.\times(\frac23 \U^3\Y_\eta-2\P\U\Y_\eta-\Q\U_\eta)(s,n)dn\Y_\eta(s,m) dm\right| \\
& \leq \left|  \int_{\eta(s)}^{\tilde\eta(s)} \int_0^{2C} e^{-\vert \Y(s,m)-\Y(s,n)\vert} 
\Big(\frac23 \U^2\sqrt{\Y_\eta} (\U^2\Y_\eta)^{1/2}\right.\\
&\qquad \qquad\qquad  \qquad\qquad\left.+2\P\sqrt{\Y_\eta}(\U^2\Y_\eta)^{1/2}+\P\sqrt{\Y_\eta}\sqrt{\Henergy_\eta}\Big)(s,n)dn\, \Y_\eta(s,m)dm\right| \\
& \leq \left| \int_{\eta(s)}^{\tilde\eta(s)} \int_0^{2C}e^{-\vert \Y(s,m)-\Y(s,n)\vert} \big(\frac23 \U^2\sqrt\Y_\eta+3\P\sqrt\Y_\eta\big)
\sqrt{\Henergy_\eta}(s,n)dn\, \Y_\eta(s,m)dm\right| \\
& \leq \int_{\eta(s)}^{\tilde\eta(s)} \left(\int_0^{2C} e^{-\vert \Y(s,m)-\Y(s,n)\vert} \big(\frac23 \U^2+3\P\big)^2\Y_\eta(s,n)dn\right)^{1/2}\\
& \qquad \qquad \times\left( \int_0^{2C} e^{-\vert \Y(s,m)-\Y(s,n)\vert } \Henergy_\eta(s, n) dn \right)^{1/2} \Y_\eta(s,m)dm \\
& \leq \int_{\eta(s)}^{\tilde \eta(s)} \left(\int_0^{2C} e^{-\vert \Y(s,m)-\Y(s,n)\vert } \big(\frac89 \U^4+ 18\P^2\big)\Y_\eta(s,n)dn\right)^{1/2}\\
& \qquad \qquad \times\left( \int_0^{2C} e^{-\vert \Y(s,m)-\Y(s,n)\vert } \Henergy_\eta(s, n) dn \right)^{1/2} \Y_\eta(s,m)dm \\
& \leq \bigO(1) \left(\int_{\eta(s)}^{\tilde\eta(s)} \int_0^{2C} e^{-\vert \Y(s,m)-\Y(s,n)\vert } (\U^2 +\P) \Y_\eta(s,n) dn \Y_\eta(s,m)dm\right)^{1/2} \\
&\qquad \qquad \times \left(\int_{\eta(s)}^{\tilde \eta(s)} \int_0^{2C} e^{-\vert \Y(s,m)-\Y(s,n)\vert } \Henergy_\eta(s,n)dn\Y_\eta(s,m) dm \right)^{1/2}\\
& \leq \bigO(1) \left( \int_{\eta(s)}^{\tilde \eta(s)} \int_0^{2C} e^{-\vert \Y(s,m)-\Y(s,n)\vert } (\U^2+\P)\Y_\eta(s,n)dn\Y_\eta(s,m)dm \right)^{1/2} \\
& \qquad \qquad \times \left(\int_{\eta(s)}^{\tilde \eta(s)} 4\P \Y_\eta(s,m) dm \right)^{1/2}\\
&  \leq \bigO(1) \left( \int_{\eta(s)}^{\tilde\eta(s)} \int_0^{2C}e^{-\vert \Y(s,m)-\Y(s,n)\vert} (\U^2+\P)\Y_\eta(s,n)dn\Y_\eta(s,m)dm\right)^{1/2}\\
& \qquad \times \sqrt{\vert \eta(s)-\tilde\eta(s)\vert}\\
& \leq  \bigO(1) \left( \vert t-\tilde t \vert\int_{\eta(s)}^{\tilde\eta(s)} \int_0^{2C}e^{-\vert \Y(s,m)-\Y(s,n)\vert} (\U^2+\P)\Y_\eta(s,n)dn\Y_\eta(s,m)dm \right)^{1/2}.
\end{align*}
Thus it is left to show that 
\begin{equation*}
 \left| \int_{\eta(s)}^{\tilde\eta(s)} \int_0^{2C}e^{-\vert \Y(s,m)-\Y(s,n)\vert} (\U^2+\P)\Y_\eta(s,n)dn\,\Y_\eta(s,m)dm\right|
 \end{equation*} 
 is bounded uniformly with respect to both space and time.
 Therefore observe that, since all the involved terms are positive, we have 
 \begin{align*}
  \left| \int_{\eta(s)}^{\tilde\eta(s)} \int_0^{2C} \right. & \left. e^{-\vert \Y(s,m)-\Y(s,n)\vert} (\U^2+\P)\Y_\eta(s,n)dn\,\Y_\eta(s,m)dm \right| \\
  & \leq  \int_0^{2C} \int_0^{2C}e^{-\vert \Y(s,m)-\Y(s,n)\vert} (\U^2+\P)\Y_\eta(s,n)dn\,\Y_\eta(s,m)dm\\
  & = \int_0^{2C} \int_0^{2C} e^{-\vert \Y(s,m)-\Y(s,n)\vert} \Y_\eta(s,m)dm\, (\U^2+\P)\Y_\eta(s,n)dn\\
  & \leq 2 \int_0^{2C}  (\U^2+\P)\Y_\eta(s,n)dn \leq 6C,
  \end{align*}
  where  we used \eqref{vik:est} in the last step. Thus we obtain 
  \begin{align*}
& \Big\vert \int_{\eta(s)}^{\tilde\eta(s)} \int_0^{2C} \sign(m-n)e^{-\vert \Y(s,m)-\Y(s,n)\vert}\\
&\qquad \qquad\qquad\qquad\times (\frac23 \U^3\Y_\eta-2\P\U\Y_\eta-\Q\U_\eta)(s,n)dn\,\Y_\eta(s,m) dm\Big\vert \\
 & \leq  \bigO(1) \left( \vert t-\tilde t \vert \int_{\eta(s)}^{\tilde \eta(s)} \int_0^{2C}e^{-\vert \Y(s,m)-\Y(s,n)\vert} (\U^2+\P)\Y_\eta(s,n)dn\,\Y_\eta(s,m)dm \right)^{1/2} \\
& \leq \bigO(1) \vert t-\tilde t \vert^{1/2},
\end{align*}
and 
\begin{align}\nn
  \Big\vert \int_t^{\tilde t}  \int_{\eta(s)}^{\tilde\eta(s)}& \int_0^{2C} \sign(m-n)e^{-\vert \Y(s,m)-\Y(s,n)\vert}\\ \nn
  &\qquad\qquad\qquad \times\big(\frac23 \U^3\Y_\eta-2\P\U\Y_\eta-\Q\U_\eta\big)(s,n)dn\,\Y_\eta(s,m)dmds\Big\vert \\ 
&  \leq \bigO(1) \vert t-\tilde t\vert^{3/2}.\label{vik:est3}
  \end{align}
Combining \eqref{vik:est4} and \eqref{vik:est3}, we end up with 
\begin{align*}\nn
\vert \tilde I_2\vert &=\vert 2\int_t^{\tilde t} (Q_t(s, l(\tilde t, \eta))+H_t(s, l(\tilde t, \eta)))-(Q_t(s, l(t,\eta))+H_t(s, l(t,\eta)))ds\vert \\ 
& \leq \bigO(1)(\vert t-\tilde t\vert ^{3/2} +\vert t-\tilde t\vert^2).
\end{align*}

As far as $\tilde I_1$ is concerned, we would like to apply the mean-value theorem. 
Note therefore that we showed before that the function $(Q_t+H_t)(t,\xi)$ is continuous with respect to time, and hence the 
function $(Q_t+H_t)(s,l(t,\eta))$, considered  as a function of $s$, is continuous with respect to $s$. Thus we end up with 
\begin{equation*}
\tilde I_1= 2\int_t^{\tilde t} (Q_t+H_t)(s, l(t,\eta))ds= 2(\tilde t-t) (Q_t+H_t)(\tilde s, l(t,\eta))
\end{equation*}
for some $\tilde s$ between $t$ and $\tilde t$.

Last, but not least, we therefore have
\begin{align*}
2\int_t^{\tilde t} (Q_t+H_t)(s, l(\tilde t,\eta))ds=&\, \tilde I_1+\tilde I_2\\
= &\, 2(\tilde t-t) (Q_t+H_t)(\tilde s, l(t,\eta))+ \bigO(1)\vert \tilde t-t\vert ^{3/2} \\
&+ \bigO(1)\vert \tilde t-t\vert^2,
\end{align*}
for some $\tilde s$ between $t$ and $\tilde t$, which implies, cf.~\eqref{want:lim}, that
\begin{align*}
\lim_{\tilde t\to t} \frac{J(t,l(t,\eta))-J(t,l(\tilde t, \eta))}{t-\tilde t} & =\lim_{\tilde t\to t} 2\int_t^{\tilde t} \frac{(Q_t+H_t)(s, l(\tilde t, \eta))}{t-\tilde t} ds \\ \nn
&=2(Q_t+H_t)(t,l(t,\eta)).
\end{align*}

Recalling that $J(t,l(t,\eta))=\eta$ for all $t$ and $\eta$, we have that 
\begin{equation*}
J^{-1}(t, J(t,l(t,\eta)))=l(t,\eta),
\end{equation*}
if we, as before, denote the inverse to $J(t,\dott)$ by $J^{-1}(t,\dott)$. 
Thus we have  
\begin{align*}
l_t(t,\eta)& =\lim_{\tilde t\to t} \frac{l(t,\eta)-l(\tilde t, \eta)}{t-\tilde t}\\ \nn
&= \lim_{\tilde t\to t} \frac{J^{-1}(t, J(t,l(t,\eta)))-J^{-1}(t, J(t, l(\tilde t, \eta)))}{t-\tilde t}\\ \nn
& = \lim_{\tilde t\to t} \frac{J^{-1}(t, J(t, l(t,\eta)))-J^{-1}(t, J(t, l(\tilde t, \eta))}{J(t,l(t,\eta))-J(t, l(\tilde t, \eta))}\lim_{\tilde t \to t}\frac{J(t,l(t,\eta))-J(t,l(\tilde t, \eta))}{t-\tilde t}.
\end{align*}
We have established that the last limit on the right-hand side exists and the existence of the first one follows from the fact that $J^{-1}(t, \dott)$ is differentiable almost everywhere.

\begin{remark}
The above argument relies heavily on the fact that the function $l(t,\cdot)\colon[0,2C]\to \Real$ is continuous and strictly increasing, the reason being that $J(t,\dott)$ is continuous and strictly increasing. This is in contrast to the HS equation in \cite{CGH}, where the function $H(t,\dott)$ is increasing, but not necessarily strictly increasing and its inverse may have jumps. In the HS equations, if $H(t,\dott)$ is constant on some interval $I$, then $H(s, \dott)$ will be constant on $I$ for any $s$, since $H$ is independent of time. Furthermore, this means that the jumps in the new coordinates might change in height with respect to time, but never change their position. This is the essential difference to the CH equation where the jumps would turn up and disappear again immediately, which motivates the choice of $G(t,x)$ and $J(t,\xi)$. 
\end{remark}

%%%%%%%%%%%%%%%%%%%%%%%%

\subsection{New system: the right quantities.}
Since we have shown that $l(t,\eta)$ is differentiable almost everywhere  with respect to both space and time, we can now establish rigorously the system of differential equations in our new coordinates. Recall that $J(t,l(t,\eta))=\eta$ and hence direct calculations yield
\begin{align*}
J_t(t,l(t,\eta))+J_\xi(t,l(t,\eta))l_t(t,\eta)&=0,\\
J_\xi(t, l(t,\eta))l_\eta(t,\eta)&=1.
\end{align*}
According to the time evolution in Lagrangian coordinates we thus end up with 
\begin{align*}
l_\eta(t,\eta)& = \frac{1}{J_\xi(t,l(t,\eta))},\\
l_t(t,\eta)&=-\frac{J_t(t,l(t,\eta))}{J_\xi(t, l(t, \eta))}= -J_t(t,l(t,\eta))l_\eta(t,\eta).
\end{align*}
As far as the new coordinates are concerned, we recall that
\begin{align*}
\Y(t,\eta)=y(t,l(t,\eta)), \qquad \U(t, \eta)=U(t,l(t,\eta)), \qquad \mbox{and} \qquad \P(t,\eta)=P(t,l(t,\eta)),
\end{align*}
and direct calculations using the differentiabilities proved for $Q$, $P$, and $l$ in Subsections \ref{diffq} and \ref{diffl} yield the following differential equations for the unknowns $(\Y,\U,\P)$.
 
%-------------------
\begin{theorem}
Let $(u,\mu)$ denote a weak conservative solution of \eqref{eq:CHbasic}. Define $\Y$ by \eqref{eq:Ydef}, $\U$ by \eqref{eq:Udef}, and $\P$ by \eqref{eq:Pdef}. Then the following system of differential equations holds
\begin{subequations}  \label{eq:evolutionTHEOREM}
\begin{align}
\Y_t(t,\eta)+(\frac23 \U^3+\So)\Y_\eta(t,\eta)& =\U(t,\eta),\\
\U_t(t,\eta)+(\frac23 \U^3+\So)\U_\eta(t,\eta)& = -\Q(t,\eta),\\
\P_t(t,\eta)+(\frac23 \U^3+\So)\P_\eta(t,\eta)& = \Q\U(t,\eta)+\R(t,\eta),
\end{align}
\end{subequations}
where 
\begin{subequations}\label{eq:funcTHEOREM}
\begin{align}
\Q(t,\eta)& = -\frac14 \int_0^{2C} \sign(\eta-\theta)e^{-\vert \Y(t,\eta)-\Y(t,\theta)\vert } (2(\U^2-\P)\Y_\eta(t,\theta)+1) d\theta,\\
\So(t,\eta)&= \int_{0}^{2C} e^{-\vert \Y(t,\eta)-\Y(t,\theta)\vert} (\frac23 \U^3\Y_\eta-\U_\eta\Q-2\P\U\Y_\eta)(t,\theta)d\theta,\\
\R(t,\eta)& =  \frac14 \int_0^{2C}\sign(\eta-\theta)e^{-\vert \Y(t,\eta)-\Y(t,\theta)\vert} (\frac23 \U^3\Y_\eta+\U)(t,\theta)d\theta \notag\\
& \quad -\frac12 \int_0^{2C} e^{-\vert \Y(t,\eta)-\Y(t,\theta)\vert}\U\Q\Y_\eta(t,\theta) d\theta.  
\end{align}
\end{subequations}
\end{theorem}

\begin{remark}
This system is identical to the system derived in the smooth case, cf.~\eqref{eq:evolutionPROP} as expected.
We also remind the reader that $\Q(t,\eta)=Q(t,l(t,\eta))$, $\So(t,\eta)=J_t(t,l(t,\eta))-\frac23 \U^3(t,\eta)$, and $\R(t,\eta)=P_t(t,l(t,\eta))-\Q\U(t,\eta)$.
\end{remark}

%%%%%%%%%%%%%%%%%%%%%

\subsection{Functional setting: Consistency of the new coordinates}

The next main question is which functional space we should work in such that the right-hand side of the system \eqref{eq:evolutionTHEOREM} can be regarded as a Lipschitz function of the chosen unknowns. For instance, it is difficult or rather impossible to establish the Lipschitz continuity of $\Q$ with respect to $\P$, $\U$, and $\Y$. However, we will see that one can establish that $\Q$ is Lipschitz continuous with respect to $\P^{1/2}$, $\U$, and $\Y$. At first sight this seems to be surprising, but on the other hand we established in Eulerian coordinates that $2p(t,x)\geq u^2(t,x)$, which rewrites in our new coordinates as 
\begin{equation}\label{con:PU}
2\P(t,\eta)\geq \U^2(t,\eta).
\end{equation}
Thus it seems somehow natural that $\P(t,\eta)^{1/2}$ and $\vert \U(t,\eta)\vert$ will behave similarly. 

The aim of this subsection is to derive the system of differential equations for $\Y$, $\U$, and $\P^{1/2}$ and subsequently to establish that all terms turning up in this system are well-defined by assuming that each of the new variables $(\Y,\U,\P^{1/2})$ is in $L^2([0,2C])$.

Replacing the equation for $\P$ with the corresponding one for $\P^{1/2}$, we find that 
the system \eqref{eq:evolutionTHEOREM} reads
\begin{subequations} \label{sys:perfect}
\begin{align}
\Y_t+(\frac23 \U^3+\So)\Y_\eta&=\U,\\
\U_t+(\frac23 \U^3+\So)\U_\eta& = -\Q,\\ \label{evol:sqrtP}
(\P^{1/2})_t+(\frac23 \U^3+\So)(\P^{1/2})_\eta& = \frac{\Q\U}{2\P^{1/2}}+\frac{\R}{2\P^{1/2}},
\end{align}
\end{subequations}
where $\Q$, $\So$, and $\R$ are given by \eqref{eq:funcTHEOREM}.

Assume for the moment  that $\P^{1/2}$, $\U$, and $\Y$ belong to $L^2([0,2C])$. Then we want to show that all terms appearing in the above system also belong to $L^2([0,2C])$. Therefore it is important to keep in mind, in addition to \eqref{con:PU}, that 
\begin{equation}\label{eq:PUYH}
2\P\Y_\eta(t,\eta)-\U^2\Y_\eta(t,\eta)+ \Henergy_\eta(t,\eta)=1,  \end{equation}
and, in particular, cf.~\eqref{eq:stand:est}, \eqref{vik:estprep}, and \eqref{vik:est}, 
\begin{subequations} \label{eq:basic}
\begin{align}\label{eq:basic0}
2\vert \U\U_\eta(t,\eta)\vert& \leq \Henergy_\eta(t,\eta)\leq 1,\\ \label{eq:basic1} 
 \U^2\Y_\eta(t,\eta)& \leq \Henergy_\eta(t,\eta)\leq 1,\\ \label{eq:basic2}  2\vert \Q\vert \Y_\eta(t,\eta)&\leq 2\P\Y_\eta(t,\eta)\leq 1.
\end{align}
\end{subequations}
As an immediate consequence we obtain that 
\begin{equation*}
\U^3\Y_\eta(t,\eta) \quad \text{ and }\quad \U^3\U_\eta(t,\eta)
\end{equation*}
are uniformly bounded. The last term of this form can be estimated as follows 
\begin{equation*}
\vert \U^3(\P^{1/2})_\eta(t,\eta)\vert =\left\vert \frac{\U^3\Q\Y_\eta}{2\P^{1/2}}(t,\eta)\right\vert \leq \left\vert \frac{\U^3}{2\P^{1/2}}(t,\eta)\right\vert\leq \U^2(t,\eta),
\end{equation*}
and is therefore uniformly bounded. Here we used \eqref{con:PU}.

As far as $\So$ is concerned, we want to establish that 
\begin{equation*}
\vert \So(t,\eta)\vert\leq \bigO(1)\P(t,\eta) ,
\end{equation*}
implying that
\begin{align*}
\vert \So\Y_\eta(t,\eta)\vert &\leq \bigO(1)\P\Y_\eta(t,\eta)\leq \bigO(1),\\
\vert \So\U_\eta(t,\eta)\vert &\leq \So^2\Y_\eta(t,\eta)+\Henergy_\eta(t,\eta) \leq \bigO(1),\\ 
\vert \So(\P^{1/2})_\eta(t,\eta)\vert & \leq \bigO(1) \P^{3/2}\Y_\eta(t,\eta)\leq \bigO(1).
\end{align*}
Indeed,  by the definition of $\So$ and since $|\Q|\leq \P$ and (cf.~\eqref{stand:est2})
\begin{equation*}
\U_\eta^2(t,\eta)\le \Henergy_\eta\Y_\eta(t,\eta),
\end{equation*}
we have 
\begin{align*}
\vert \So(t,\eta)\vert &= \left| \int_0^{2C} e^{-\vert \Y(t,\eta)-\Y(t,\theta)\vert} (\frac23 \U^3\Y_\eta-\U_\eta\Q-2\P\U\Y_\eta)(t,\theta)d\theta\right| \\ 
& \leq \norm{\U(t,\dott)}_{L^\infty([0,2C])}\int_0^{2C} e^{-\vert \Y(t,\eta)-\Y(t,\theta)\vert} \U^2\Y_\eta(t,\theta)d\theta\\
& \quad +2 \left( \int_0^{2C} e^{-\vert \Y(t,\eta)-\Y(t,\theta)\vert} \P^2\Y_\eta(t,\theta)d\theta\right)^{1/2}\\
&\qquad\qquad\qquad  \times\left(\int_0^{2C} e^{-\vert \Y(t,\eta)-\Y(t,\theta)\vert } \U^2\Y_\eta(t,\theta)d\theta\right)^{1/2}\\
& \quad +\left(\int_0^{2C} e^{-\vert \Y(t,\eta)-\Y(t,\theta)\vert} \P^2\Y_\eta(t,\theta)d\theta \right)^{1/2} \\
&\qquad\qquad\qquad  \times\left( \int_0^{2C} e^{-\vert \Y(t,\eta)-\Y(t,\theta)\vert} \Henergy_\eta(t,\theta) d\theta\right)^{1/2}\\  
& \leq  4\sqrt{C}\,\P(t,\eta)+6\left(\int_0^{2C} e^{-\vert \Y(t,\eta)-\Y(t,\theta)\vert} \P^2\Y_\eta (t,\theta) d\theta\right)^{1/2} \P^{1/2}(t,\eta).
\end{align*} 
Thus it is left to show that 
\begin{equation}\label{eq:109}
\int_0^{2C} e^{-\vert \Y(t,\eta)-\Y(t,\theta)\vert} \P^2\Y_\eta(t,\theta)d\theta \leq \bigO(1)\P(t,\eta).
\end{equation}
As several times before, the main tool will be integration by parts together with 
\begin{equation*}
\P(t,\eta)=\frac14 \int_0^{2C} e^{-\vert \Y(t,\eta)-\Y(t,\theta)\vert} (2(\U^2-\P)\Y_\eta(t,\theta)+1)d\theta.
\end{equation*}
Direct computations yield
\begin{align}
\int_0^\eta e^{\Y(t,\theta)-\Y(t,\eta)} \P^2\Y_\eta(t,\theta)d\theta
& = \P^2(t,\eta)-\int_0^{\eta} e^{\Y(t,\theta)-\Y(t,\eta)} 2\P\Q\Y_\eta(t,\theta)d\theta \notag\\
& = \P^2(t,\eta)-2\P\Q(t,\eta) \notag\\ 
& \quad +\int_0^\eta e^{\Y(t,\theta)-\Y(t,\eta)} (2\Q^2\Y_\eta+2\P\Q_\eta)(t,\theta)d\theta\notag\\
& = \P^2(t,\eta)-2\P\Q(t,\eta)\notag\\ 
& \quad +\int_0^\eta e^{\Y(t,\theta)-\Y(t,\eta)} 2(\P^2+\Q^2)\Y_\eta(t,\theta)d\theta\notag\\
& \quad +\int_0^\eta e^{\Y(t,\theta)-\Y(t,\eta)} (2(\P-\U^2)\Y_\eta(t,\theta)-1)\P(t,\theta)d\theta.\label{eq:Q_d}
\end{align}
Here we have used \eqref{auxq} and \eqref{eq:PUYH}.
Thus, rearranging the terms, we end up with 
\begin{align}\nn
\int_0^\eta e^{\Y(t,\theta)-\Y(t,\eta)} \P^2\Y_\eta(t,\theta)d\theta& \leq \int_0^\eta e^{ \Y(t,\theta)-\Y(t,\eta)} (\P^2+2\Q^2)\Y_\eta(t,\theta)d\theta\\ \nn
& = 2\P\Q(t,\eta)-\P^2(t,\eta)\\ \nn
&\quad + \int_0^\eta e^{\Y(t,\theta)-\Y(t,\eta)} (2(\U^2-\P)\Y_\eta(t,\theta)+1) \P(t,\theta) d\theta\\ \label{rel:P2PP}
& \leq 2\P^2(t,\eta)+4\norm{\P(t,\dott)}_{L^\infty([0,2C])} \P(t,\eta)\\ \nn
&\leq 6 \norm{\P(t,\dott)}_{L^\infty([0,2C])} \P(t,\eta).
\end{align}
Similar computations yield that 
\begin{align*} \nn
\int_\eta^{2C} e^{\Y(t,\eta)-\Y(t,\theta)} \P^2\Y_\eta(t,\theta) d\theta
& \leq \int_\eta^{2C} e^{\Y(t,\eta)-\Y(t,\theta)} (\P^2+2\Q^2)\Y_\eta(t,\theta) d\theta\\ \nn
& = -\P^2(t,\eta)-2\P\Q(t,\eta)\\ \nn
& \quad +\int_\eta^{2C} e^{\Y(t,\eta)-\Y(t,\theta)} \\ \nn
&\qquad\qquad\times(2(\U^2-\P)\Y_\eta(t,\theta)+1)\P(t,\theta)d\theta\\ & \leq 2\P^2(t,\eta)+4\norm{\P(t,\dott)}_{L^\infty([0,2C])} \P(t,\eta) \\ \nn
&\leq 6\norm{\P(t,\dott)}_{L^\infty([0,2C])} \P(t,\eta).
\end{align*}
This proves \eqref{eq:109} due to \eqref{LI2}. 

There are two terms left to investigate. Namely, the term $\Q\U/\P^{1/2}$, and we easily find that
\begin{equation*}
\left\vert \frac{\Q\U}{\sqrt{2}\P^{1/2}}(t,\eta)\right\vert \leq \P(t,\eta),
\end{equation*}
and hence it is uniformly bounded. The last term $\R/\P^{1/2}$ is a bit more involved. We have 
\begin{align}
\vert \R(t,\eta)\vert& \leq \left| \frac14 \int_0^{2C} \sign(\eta-\theta)e^{-\vert \Y(t,\eta)-\Y(t,\theta)\vert} (\frac23 \U^3\Y_\eta+\U)(t,\theta)d\theta\right| \notag\\
& \qquad + \left| \frac12 \int_0^{2C} e^{-\vert \Y(t,\eta)-\Y(t,\theta)\vert } \U\Q\Y_\eta(t,\theta)d\theta\right| \notag\\ 
& \leq \frac14 \int_0^{2C} e^{-\vert \Y(t,\eta)-\Y(t,\theta)\vert} (2(\U^2-\P)\Y_\eta(t,\theta)+1)\vert \U\vert (t,\theta)d\theta\notag\\
& \quad + \frac14 \int_0^{2C} e^{-\vert \Y(t,\eta)-\Y(t,\theta)\vert }2\P\vert \U\vert\Y_\eta (t,\theta)d\theta\notag\\
& \quad + \frac12 \int_0^{2C} e^{-\vert \Y(t,\eta)-\Y(t,\theta)\vert } \vert \U\Q\vert\Y_\eta(t,\theta) d\theta\notag\\
& \leq \norm{\U(t,\dott)}_{L^\infty([0,2C])} \P(t,\eta) +\frac12 \int_0^{2C} e^{-\vert \Y(t,\eta)-\Y(t,\theta)\vert } (\P^2+\U^2)\Y_\eta(t,\theta)d\theta\notag\\
& \leq \bigO(1)\P(t,\eta).  \label{eq:R}
\end{align}
Thus
\begin{equation*}
\left| \frac{\R}{\P^{1/2}(t,\eta)}\right| \leq \bigO(1)\P^{1/2}(t,\eta),
\end{equation*} 
and hence belongs to $L^2([0,2C])$.

Later on, we will need in some of our estimates, cf.~\eqref{eq:Pest}, that
\begin{equation*}%\label{eq:Pestcal}
\P(t,\eta)=\frac14 \int_0^{2C} e^{-\vert \Y(t,\eta)-\Y(t,\theta)\vert} (\U^2\Y_\eta+\Henergy_\eta)(t,\theta)d\theta,
\end{equation*}
which implies the estimate
\begin{equation}\label{eq:PestcalA}
 \int_0^\eta e^{-(\Y(t,\eta)-\Y(t,\theta))} (\U^2\Y_\eta+\Henergy_\eta)(t,\theta)d\theta\le 4 \P(t,\eta).
\end{equation}

%%%%%%%%%%%%%%%%%%%%%%%

\subsection{The choice of the distance}\label{dist1}
In order to motivate the choice of the distance between two solutions with the same energy and to outline the strategy for proving the Lipschitz continuity, let us come back to the system \eqref{sys:perfect}. We observe that the unknowns $(\Y,\U,\P^{1/2})$ are advected by the same velocity field in the $\eta$ variable. Moreover, since all terms make sense under the assumption that the unknowns belong to $L^2([0,2C])$, as checked in the previous subsection, we consider a toy problem of the form 
$$
f_t+{\mathcal A}(f)f_\eta={\mathcal B}(f)
$$
where $f$ might be a vector valued function in $L^2([0,2C])$. This toy problem resembles our situation with the complicated operator dependencies on the unknowns. Within this functional framework, it is natural to look at the evolution of the $L^2$ norm of the unknown $f$. 
Given two solutions $f_1$ and $f_2$ to $f_t+{\mathcal A}(f)f_\eta={\mathcal B}(f)$, direct computations yield
\begin{align*}
\frac{d}{dt} \norm{f_1-f_2}_{L^2([0,2C])}^2& =2\int_0^{2C} (f_1-f_2)(f_{1,t}-f_{2,t})(t, \theta) d\theta\\
& = -2\int_0^{2C} ({\mathcal A}(f_1)f_{1,\eta}-{\mathcal A}(f_2)f_{2,\eta})(f_1-f_2)(t,\theta)d\theta \\ 
& \quad + 2 \int_0^{2C} ({\mathcal B}(f_1)-{\mathcal B}(f_2))(f_1-f_2)(t,\theta)d\theta\\ 
& = -2\int_0^{2C} {\mathcal A}(f_1)(f_{1,\eta}-f_{2,\eta})(f_1-f_2)(t,\theta)d\theta\\
& \quad -2\int_0^{2C} f_{2, \eta} ({\mathcal A}(f_1)-{\mathcal A}(f_2))(f_1-f_2)(t,\theta)d\theta\\
& \quad + 2 \int_0^{2C} ({\mathcal B}(f_1)-{\mathcal B}(f_2))(f_1-f_2)(t,\theta)d\theta\\ 
& = {\mathcal A}(f_1)(f_1-f_2)^2(t,0)-{\mathcal A}(f_1)(f_1-f_2)^2(t,2C)\\
& \quad + \int_0^{2C} {\mathcal A}'(f_1)f_{1,\eta} (f_1-f_2)^2(t,\theta)d\theta\\
& \quad -2\int_0^{2C} f_{2, \eta} ({\mathcal A}(f_1)-{\mathcal A}(f_2))(f_1-f_2)(t,\theta)d\theta\\
& \quad + 2 \int_0^{2C} ({\mathcal B}(f_1)-{\mathcal B}(f_2))(f_1-f_2)(t,\theta)d\theta,
\end{align*} 
where the first two terms in the last line are going to vanish if we impose the correct boundary conditions or rather if we have a good behavior of the solutions at both boundaries. Under this assumption, we can then use norm estimates if we know that ${\mathcal A}'(f_1) f_{1,\eta}$ and $f_{2, \eta}$ are uniformly bounded by a constant $\bigO(1)$ and that ${\mathcal A}$ and ${\mathcal B}$ are Lipschitz continuous with Lipschitz constant $\bigO(1)$ with respect to $f$ to conclude that 
\begin{equation*}
\frac{d}{dt}\norm{f_1-f_2}_{L^2([0,2C])}^2 \leq \bigO(1) \norm{f_1-f_2}_{L^2([0,2C])}^2,
\end{equation*}
and Gronwall's lemma then implies 
\begin{equation*}
\norm{(f_1-f_2)(t)}_{L^2([0,2C])} \leq \norm{(f_1-f_2)(0)}_{L^2([0,2C])} e^{\bigO(1)t}.
\end{equation*}

Hence the strategy consists in proving first a propagation in time for the $L^2$-norm of $f$ in order to get bounds on the unknowns and check the validity of the approach. Then the second, and the most important and technical point, is to establish the Lipschitz estimates. We will come back to this point in Section \ref{sec:difen}, where we will also fix the following problem in the definition of our metric: the domain of definition of our unknowns depends on the total energy. This is unsatisfactory, if we want to compare solutions with different total energies. 

%%%%%%%%%%%%%%%%%%%%%%%%

\subsection{Propagation of $L^2$ bounds: moment conditions}\label{moments}

Given $(u,\mu)$ a weak conservative solution of \eqref{eq:CHbasic} and $f=(\Y,\U,\P^{1/2})$, defined by \eqref{eq:Ydef}, \eqref{eq:Udef}, and \eqref{eq:Pdef}, the solution to the system \eqref{sys:perfect}. We want to show that the $L^2$-norm of $f$ is propagated in time. Since both $\U$ and $\P$ are bounded functions due to \eqref{LI1} and \eqref{LI2}, we are reduced to show the propagation of the $L^2$-norm of $\Y$. First of all note that we have 
\begin{equation*}
\Y(t, G(t,x))= x \quad \text{ for all } x\in \Real,
\end{equation*}
where $G(t,x)$ is given by 
\begin{equation*}
G(t,x)=\int_{-\infty}^x (2p-u^2)(t,y)dy +F(t,x).
\end{equation*}
Notice that the $x$-distributional derivative of $G$ is a measure that we denoted by $\nu(t,\dott)\in {\mathcal M}_+(\Real)$ given by
$$
d\nu(t,x)=(2p-u^2)(t,x) + d\mu(t,x) .
$$
Since $\Y(t,\eta)$ is the pseudo-inverse to $G(t,x)$, then $\Y$ pushes forward the uniform distribution on $[0,2C]$ to $\nu$, see \cite{Vil03}, and then
\begin{equation*}
\int_\Real x^2 d\nu = \int_0^{2C} \Y^2(t,\eta) d\eta.
\end{equation*}
Thus $\Y\in L^2([0,2C])$ is equivalent to $\nu(t,\dott)$ having finite second moment.

\begin{proposition}
Let $(u,\mu)$ denote a weak conservative solution of \eqref{eq:CHbasic}, and by $f=(\Y,\U,\P^{1/2})$ the solution to the system \eqref{sys:perfect}, then 
\begin{align*}
 \int_0^{2C} \Y^2(t,\eta) d\eta \leq e^{\bigO(1)t}\left(1+\int_0^{2C} \Y^2(0,\eta) d\eta\right).
\end{align*}
\end{proposition}

\begin{proof}
We show the propagation of the second moment in time using the Lagrangian coordinates. Note first that
\begin{align*}
\int_\Real x^2d\mu(t,x) &= \int_\Real y^2H_\xi(t,\xi) d\xi,\\
\int_\Real x^2 u^2(t,x)dx& =\int_\Real y^2U^2y_\xi(t,\xi)d\xi,\\
\int_\Real x^2p(t,x)dx &= \int_\Real y^2Py_\xi(t,\xi)d\xi.
\end{align*}
We start by computing the derivatives of $y^2H_\xi$, $y^2U^2y_\xi$, and $y^2Py_\xi$ with respect to time. Direct computations using the formulas in subsection \ref{diffq} then yield
\begin{align*}
\vert (y^2H_\xi)_t\vert & = \vert 2yUH_\xi+y^2(3U^2U_\xi-2PU_\xi-2QUy_\xi)\vert \\
& \leq y^2H_\xi +U^2H_\xi+3y^2\vert U\vert H_\xi + 2y^2P^2y_\xi+2y^2H_\xi\\
& \leq \bigO(1) (y^2H_\xi  +y^2Py_\xi +H_\xi),
\end{align*}
and 
\begin{align*}
\vert (y^2Py_\xi)_t\vert &= \vert 2yUPy_\xi+ y^2P_ty_\xi+ y^2PU_\xi\vert \\
& \leq y^2P^2y_\xi +U^2y_\xi+ y^2\vert P_t\vert y_\xi + y^2P(y_\xi+H_\xi)\\
& \leq \bigO(1) (y^2 H_\xi + y^2Py_\xi+ H_\xi),
\end{align*}
where we used \eqref{tder:P}. 
Thus 
\begin{equation*}
\frac{d}{dt}\left(\norm{y^2Py_\xi}_{L^1} + \norm{y^2H_\xi}_{L^1}\right)\leq \bigO(1) \left(\norm{y^2Py_\xi}_{L^1}+\norm{y^2H_\xi}_{L^1}+ C\right)
\end{equation*}
and 
\begin{equation*}
\int_\Real (y^2Py_\xi+y^2H_\xi)(t,\xi)d\xi \leq e^{\bigO(1)t} \left(\int_\Real (y^2Py_\xi+y^2H_\xi)(0,\xi)d\xi+ C\right).
\end{equation*}
Since $y^2U^2y_\xi(t,\xi)\leq y^2H_\xi(t,\xi)$ it follows that also 
\begin{equation*}
\int_\Real y^2U^2y_\xi(t,\xi) d\xi
\end{equation*}
remains finite for all times, which proves the desired estimate since 
$$
\int_\Real x^2d\nu(t,x)=\int_\Real x^2(2p-u^2)(t,x)dx + \int_\Real x^2 d\mu(t,x).
$$
\end{proof}

The propagation of moments not only implies the feasibility of the strategy illustrated in subsection \ref{dist1} but also gives us a control on the ``boundary terms''. In fact, since $(u^2+u_x^2)(t,x)\leq d\mu(t,x)$ and $(2p-u^2)(t,x)\geq 0$, the previous result implies that
\begin{equation}\label{cond:mom}
\int_\Real x^2p(t,x) dx, \quad \int_\Real x^2u^2(t,x) dx, \int_\Real x^2u_x^2(t,x) dx, \text{ and }\int_\Real x^2d\mu(t,x)
\end{equation}
are all finite. We show next that \eqref{cond:mom} implies that 
\begin{equation}\label{decay:impl}
xu(t,x)\to 0 \quad \text{ and } x^2p(t,x)\to 0\quad \text{ as } x\to \pm\infty.
\end{equation}
Indeed, \eqref{cond:mom} and $u\in H^1(\Real)$ imply that 
\begin{equation*}
xu(t,x) \in L^2(\Real) \quad \text{ and } \quad (xu(t,x))_x=u(t,x)+xu_x(t,x)\in L^2(\Real).
\end{equation*}
Hence $xu(t,x)$ belongs to $H^1(\Real)$ for any fixed $t$ and in particular $xu(t,x)\to 0$ as $x\to \pm \infty$ for any fixed $t$.

The argument for $xp(t,x)$ is a bit more involved, but follows the same lines. To be more precise, 
$p(t,x)\in L^2(\Real)$, $\vert p_x(t,x)\vert \leq p(t,x)$, and 
\begin{equation*}
\int_\Real x^2p^2(t,x)dx \leq \bigO(1) \int_\Real x^2p(t,x)dx <\infty,
\end{equation*}
imply that 
\begin{equation*}
xp(t,x)\in L^2(\Real) \quad \text{ and } \quad (xp(t,x))_x=p(t,x)+ xp_x(t,x)\in L^2(\Real).
\end{equation*}
Hence $xp(t,x)$ belongs to $H^1(\Real)$ for any fixed $t$ and in particular $xp(t,x)\to 0$ as $x\to \pm\infty$. 

Finally, it is left to show that $x\sqrt{p(t,x)}\to 0$ as $x\to \pm\infty$. 
Therefore note that $p(t,x)\in L^1(\Real)\cap L^2(\Real)$ and
\begin{equation*}
\vert (x\sqrt{p(t,x)})_x\vert =\vert  \sqrt{p(t,x)}+ x\frac{p_x(t,x)}{2\sqrt{p(t,x)}}\vert \leq \sqrt{p(t,x)}+ \vert x\vert \sqrt{p(t,x)},
\end{equation*}
which together with \eqref{cond:mom} implies that 
\begin{equation*}
x\sqrt{p(t,x)} \in L^2(\Real) \quad \text{ and }\quad (x\sqrt{p(t,x)})_x\in L^2(\Real).
\end{equation*}
Hence $x\sqrt{p(t,x)}$ belongs to $H^1(\Real)$ for any fixed $t$ and in particular $x^2p(t,x)\to 0$ as $x\to \pm\infty$. These properties will allow us carry out several integrations by parts when deriving the anticipated Lipschitz estimate.

%%%%%%%%%%%%%%
%%%%%%%%%%%%%%
%%%%%%%%%%%%%%

\section{The right metric: solutions with different energies}\label{sec:difen}

For the remaining part of this paper, we will consider two distinct solutions $(u_j,\mu_j)$ for $j=1,2$ of the Camassa--Holm equation \eqref{eq:CHbasic} and estimate their difference using a carefully selected norm based on our new variables $(\Y_j, \U_j, \P_j^{1/2})$ for $j=1,2$ that yields a Lipschitz metric. The main idea remains to a large extent the one from the Hunter--Saxton equation in \cite{CGH}, where a Lipschitz metric for measuring the distance between solutions with non-zero energy has been constructed.

We want to define a Lipschitz metric, which can measure the distance between solutions with different total energy based on our new coordinates. At first sight this seems to be impossible for two reasons: \\
\noindent 1.: The support of the new variables depends on the total energy, since $$\Y, \U,\P^{1/2}:[0,2C]\to\Real.$$
\noindent 2.: In case of the zero-solution, i.e., $(u, \mu)=(0,0)$, the function $\Y$ is not defined and the same applies to the new variables $(\Y, \U, \P^{1/2})$.

The key idea is to use a rescaling. More precisely, given $(\Y, \U, \P^{1/2})$, where $\Y:[0,2C]\to \Real$ with $C\not=0$. Then we can introduce
\begin{align}
\tilde \Y(t,\eta)&= A\Y(t, A^2\eta), \quad \tilde \U(t,\eta)=A \U(t,A^2\eta), \notag\\
 \tilde\P^{1/2}(t,\eta)&=A\P^{1/2}(t, A^2\eta), \quad \tilde \Henergy(t,\eta)= A^3\Henergy(t, A^2\eta),\label{newchange}
\end{align}
with 
\begin{equation}\label{eq:defA}
A=\sqrt{2C}
\end{equation}
for notational purposes.  

This rescaling has three main properties that motivate our choice:\\
\noindent 1.: The support of the variables $(\tilde \Y, \tilde \U, \tilde \P^{1/2})$ is independent of the energy, which allows to compare arbitrary solutions with non-zero energy. \\
\noindent 2.: In the last section, we discussed that the natural solution space for the unknowns $(\Y, \U, \P^{1/2})$ is $L^2([0,2C])$. Since our rescaling preserves the $L^2$-norm, the natural solution space for the unknowns $(\tilde \Y, \tilde \U, \tilde \P^{1/2})$ is $L^2([0,1])$. \\
\noindent 3.: The most important property is that this rescaling allows us to set $$(\tilde \Y, \tilde \U, \tilde \P^{1/2})=(0,0,0) \quad \text{ for } \quad  (u,\mu)=(0,0).$$ 
A close look at the definition of $\tilde \Y(t,\eta)$ reveals that $\tilde \Y(t,\eta)$ is the pseudo inverse to $\tilde G(t,x)$, which is given by 
\begin{equation*}
\tilde G(t,x)=\frac{1}{\A^2}G(t, \frac{x}{\A}) \quad \text{ for } \A\not=0.
\end{equation*}
If one compares $\tilde G(t,x)$ with $G(t,x)$ for small $A$, then the graph of $G(t,x)$ is squeezed in the $x$ direction, but stretched in the $y$ direction, cf. Figure~\ref{resc}. 
\begin{figure}\centering
\includegraphics[width=4cm]{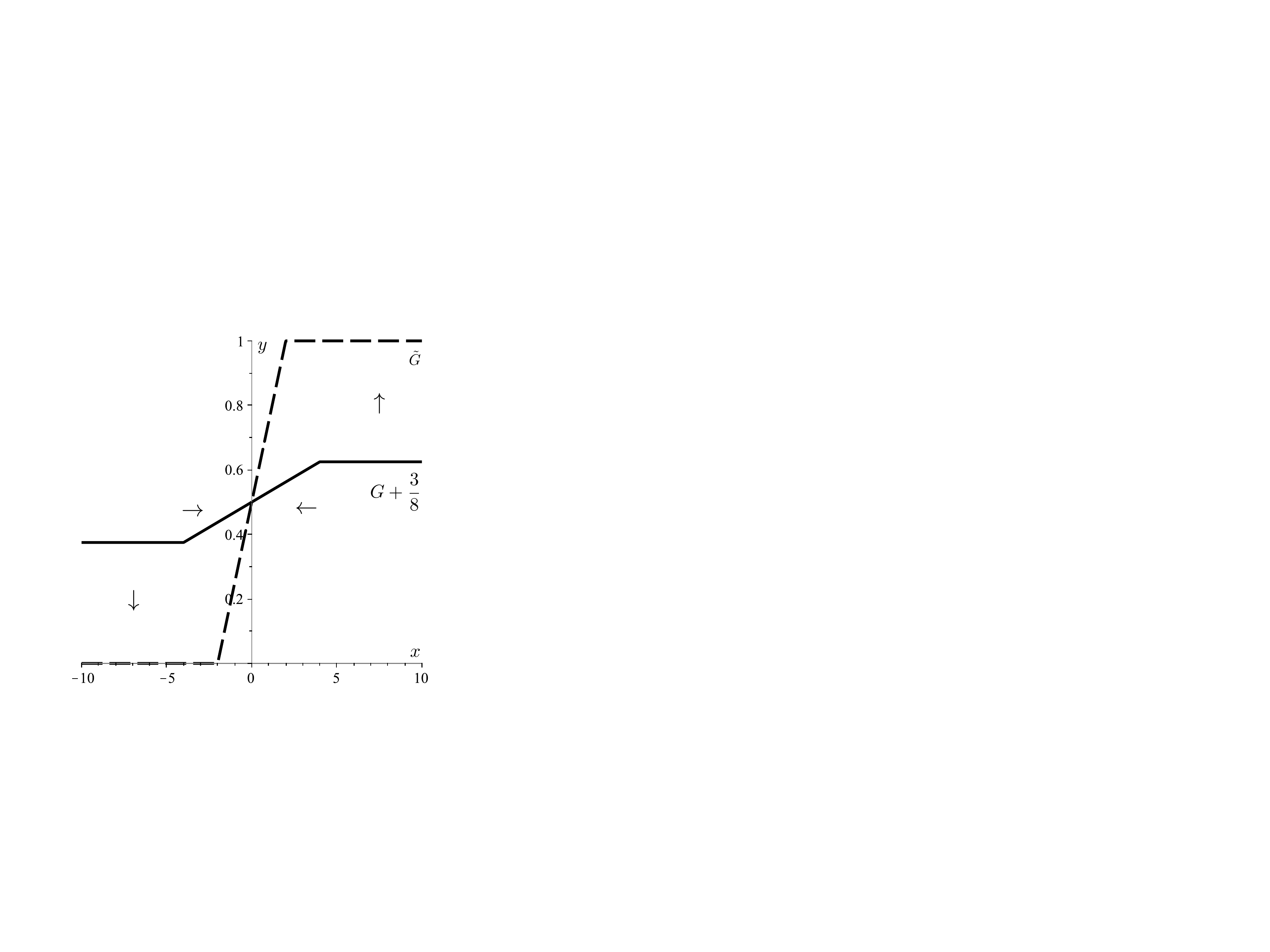}
\caption{An example of how the rescaling changes the graph of a function $G:\Real\to \frac14$. The picture shows the function $G+\frac{3}{8}$ (solid) and its rescaled version $\tilde G$ (dash).}\label{resc}
\end{figure}
Taking a sequence of functions $G_n(t,x)$ with $A_n$ tending to zero, then the corresponding sequence of functions $\tilde G_n(t,x)$ tends to the Heaviside function, which has as pseudo inverse $\tilde \Y(t,\eta)=0$, cf. Figure~\ref{resc2}.
\begin{figure}\centering
\includegraphics[width=13cm]{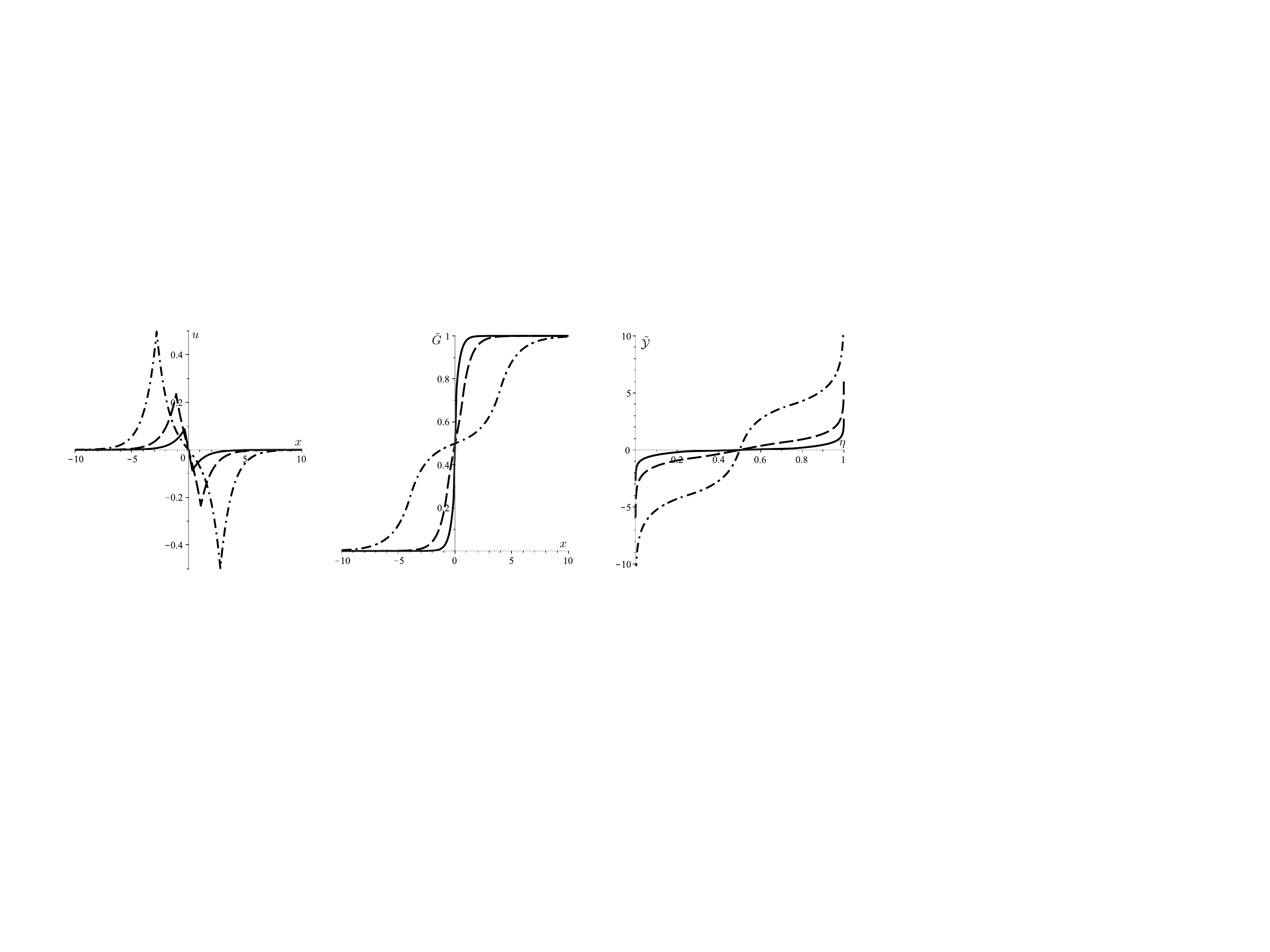}
\caption{Three different peakon-antipeakon solutions $u$ (left) with total energy $1$ (dashdot), $0.5$ (dash) and $0.25$ (solid), respectively and the corresponding functions $\tilde G$ (middle)  and $\tilde \Y$ (right).  }\label{resc2}
\end{figure}

Direct calculation yield the following theorem.
\begin{theorem}
Let $(u,\mu)$ denote a weak conservative solution of \eqref{eq:CHbasic}. Define $\tilde \Y$, $\tilde \U$ and $\tilde \P^{1/2}$ by \eqref{newchange}. Then the following system of differential equations holds
\begin{subequations}\label{finalsetequations}
\begin{align}
\tilde \Y_t+(\frac2{3}\frac{1}{A^5} \tilde \U^3+\frac{1}{A^6}\tilde \So)\tilde \Y_\eta& =\tilde \U, \label{eq:ynew}\\
\tilde \U_t+(\frac23 \frac{1}{A^5}\tilde \U^3+\frac{1}{A^6}\tilde \So)\tilde \U_\eta&= -\frac{1}{A^2}\tilde \Q,\\
(\tilde\P^{1/2})_t+ (\frac23\frac{1}{A^5} \tilde \U^3+\frac{1}{A^6}\tilde \So)(\tilde\P^{1/2})_\eta& =\frac{1}{2A^2}\frac{\tilde \Q\tilde\U}{\tilde\P^{1/2}}+\frac{1}{2A^3}\frac{\tilde \R}{\tilde\P^{1/2}},
\end{align}
\end{subequations}
where
\begin{subequations}
\begin{align}
\tilde \Q(t,\eta)&= A^3\Q(t,A^2\eta) \nn\\
& = -\frac14 \int_0^1 \sign(\eta-\theta)e^{-\frac{1}{A }\vert \tilde \Y(t,\eta)-\tilde \Y(t,\theta)\vert } (2(\tilde \U^2-\tilde \P)\tilde \Y_\eta(t,\theta)+A^5)d\theta,\\
\tilde \So(t,\eta)&=A^4\So(t,A^2\eta) \nn\\
&  = \int_0^1 e^{-\frac{1}{A}\vert \tilde \Y(t,\eta)-\tilde \Y(t,\theta)\vert} (\frac23 \tilde\U^3\tilde\Y_\eta-\tilde\U_\eta\tilde \Q-2\tilde\P\tilde \U\tilde\Y_\eta)(t,\theta)d\theta,\\
\tilde \R(t,\eta)&=A^5\R(t,A^2\eta) \nn\\
= & \,\frac14 \int_0^1 \sign(\eta-\theta)e^{-\frac{1}{A}\vert \tilde\Y(t,\eta)-\tilde\Y(t,\theta)\vert} (\frac23 A \tilde \U^3\tilde \Y_\eta+A^6\tilde\U)(t,\theta) \nn \\
& -\frac12 \int_0^1 e^{-\frac{1}{A}\vert \tilde \Y(t,\eta)-\tilde \Y(t,\theta)\vert} \tilde \U\tilde\Q\tilde \Y_\eta(t,\theta) d\theta.
\end{align}
\end{subequations}
\end{theorem}

Now, we can define the metric between two general conservative solutions to \eqref{eq:CHbasic}. To simplify the notation we will from now on assume that all norms are on $L^2([0,1])$, unless otherwise indicated, and write
\begin{equation*}
\norm{\Psi}=\norm{\Psi}_{L^2([0,1])},
\end{equation*}
for any function $\Psi$. We will not always explicitly indicate the time-dependence in all quantities and write $\norm{\Psi}$ rather than $\norm{\Psi(t)}$. 

%---------------------------
\begin{definition}\label{def:metric}
Let $(u_i,\mu_i)$ for $i=1,2$ denote two conservative solutions of \eqref{eq:CHbasic}.  We define the distance between them as
\begin{align*}
d((u_1,\mu_1), (u_2, \mu_2)) = &\, \norm{\tilde \Y_1 -\tilde \Y_2}+\norm{\tilde \U_1-\tilde\U_2}\\
& + \norm{\tilde \P^{1/2}_1 - \tilde \P^{1/2}_2} +\vert A_1-A_2 \vert,
\end{align*}
where $A_i=\sqrt{2C_i}$ and $C_i=\mu_i(\Real)$, $i=1,2$, and $(\tilde \Y_i,\tilde \U_i,\tilde \P_i^{1/2})$ are given by \eqref{newchange}.
\end{definition}
%---------------------------

\begin{remark}
Notice that in the case of one of the two solutions being the trivial solution, the distance reduces to the $L^2$ norm of the solution in the right functional space, i.e., 
\begin{equation*}
d((u,\mu), (0,0))=\norm{\tilde\Y}+\norm{\tilde\U}+\norm{\tilde\P^{1/2}}+\sqrt{2C},
\end{equation*}
where $C=\mu(\Real)$, that we already estimated in Subsection \ref{moments}. It remains to show that the right-hand side of the system \eqref{finalsetequations} is Lipschitz continuous with respect to the new unknowns $(\tilde \Y,\tilde \U,\tilde \P^{1/2})$. 
\end{remark}

The main result of this work can  finally be stated. The proof consists in showing the corresponding Lipschitz estimates in each of the components of the distance in Definition \ref{def:metric}. This is done in the next section in Lemmas \ref{lemlipy}, \ref{lemlipu} and \ref{lemma:sqP}. Collecting these results leads to our main theorem due to Gronwall's lemma.

\begin{theorem}
Consider initial data $u_{i,0}\in H^1(\Real)$, $\mu_{i,0}\in \mathcal{M}_+(\Real)$ such that $d(\muac)_{i,0}=(u_i^2+u_{i,x}^2)dx$ and $C_i=\mu_i(\Real)$, and let 
$(u_i,\mu_i)$ for $i=1,2$ denote the corresponding conservative solutions of the Camassa--Holm equation \eqref{eq:CHbasic}.  Then we have that
\begin{equation*}
 d((u_1(t),\mu_1(t)), (u_2(t), \mu_2(t)))\leq e^{\bigO(1)t}d((u_{1,0},\mu_{1,0}), (u_{2,0}, \mu_{2,0})),
\end{equation*}
where $\bigO(1)$ denotes a constant depending only on $\max_j(C_j)$ remaining bounded as $\max_j(C_j)\to 0$.
\end{theorem}

Now, let us start to do some preparatory work for the next section by collecting the main estimates we need in the new variables. Introduce
\begin{subequations} \label{eq:DogEdef}
\begin{align}
\tilde\D(t,\eta)&= \int_0^\eta e^{\frac{1}{A}(\tilde \Y(t,\theta)-\tilde\Y(t,\eta))} ((\tilde\U^2-\tilde\P)\tilde\Y_{\eta}(t,\theta)+ \frac12A^5)d\theta\\ \nn
& =\frac12 \int_0^\eta e^{\frac{1}{A}(\tilde\Y(t,\theta)-\tilde\Y(t,\eta))} (\tilde\U^2\tilde\Y_{\eta}+\tilde\Henergy_{\eta})(t,\theta) d\theta,   \label{eq:Ddef}\\[2mm]
\tilde \E(t,\eta)&= \int_{\eta}^1 e^{\frac{1}{A}(\tilde \Y(t,\eta)-\tilde\Y(t,\theta))} ((\tilde \U^2-\tilde\P)\tilde \Y_{\eta} (t,\theta)+\frac{1}{2}A^5)d\theta\\ \nn
& =\frac12 \int_\eta^1 e^{\frac{1}{A}(\tilde\Y(t,\eta)-\tilde\Y(t,\theta))}(\tilde\U^2\tilde\Y_{\eta}+\tilde\Henergy_{\eta})(t,\theta) d\theta, \label{eq:Edef}
\end{align}
\end{subequations}
so that we can write 
\begin{equation} \label{eq:DogE}
\tilde \P(t,\eta)= \frac1{2A}(\tilde \D(t,\eta)+\tilde \E(t,\eta)) \qquad \text{ and } \qquad \tilde \Q(t,\eta)=\frac12(-\tilde \D(t,\eta)+\tilde \E(t,\eta))
\end{equation}
and 
\begin{equation} \label{eq:DogP}
\tilde\Q(t,\eta)=-\tilde \D(t,\eta)+A\tilde\P(t,\eta).
\end{equation}
Note that both $\tilde \D(t,\eta) $ and $\tilde \E(t,\eta)$ have some very nice properties. Namely, 
\begin{equation}\label{eq:DPEP}
0\leq \tilde \D(t,\eta) \leq 2A\tilde \P(t,\eta) \qquad \text{and} \qquad 0\leq \tilde \E(t,\eta) \leq 2A\tilde\P(t,\eta),
\end{equation}
and 
\begin{align}
\left| \frac{d}{d\eta} \tilde \D(t,\eta)\right| &=\left| (\tilde\U^2-\tilde\P)\tilde \Y_{\eta} (t,\eta)+\frac12 A^5-\frac{1}{A}\tilde \D\tilde \Y_{\eta}(t,\eta)\right|  \leq \bigO(1)A^5, \label{eq:D_deriv}\\
\left| \frac{d}{d\eta} \tilde \E(t,\eta)\right| &= \left| -(\tilde \U^2-\tilde \P)\tilde \Y_{\eta}(t,\eta)-\frac12 A^5+ \frac1A\tilde \E\tilde\Y_{\eta}(t,\eta)\right| \leq \bigO(1)A^5.\label{eq:E_deriv}
\end{align}
The scaling leads to some changes in the standard estimates. Some of the problematic terms depending on 
$\frac{1}{A}$ in the system \eqref{finalsetequations} can be controlled since 
\begin{equation}\label{evik:est}
2\tilde\P\tilde \Y_\eta(t,\eta)\leq A^5,\quad 2\vert \tilde \U\tilde \U_\eta(t,\eta)\vert \leq A^4, \quad \tilde\U^2\tilde\Y_\eta(t,\eta)\leq A^5, \, \text{and} \, \tilde\Henergy_\eta(t,\eta)\le A^5,
\end{equation}
from \eqref{eq:basic} and \eqref{newchange}, as well as
\begin{equation}\label{evik:estA}
A^2\tilde\U_\eta^2(t,\eta)\le \tilde\Henergy_\eta\tilde \Y_\eta(t,\eta)\leq A^5\tilde \Y_\eta(t,\eta).
\end{equation}
In addition, the equation \eqref{eq:PUYH} becomes
\begin{equation}\label{eq:PUYH_scale}
2\tilde\P\tilde\Y_\eta(t,\eta)-\tilde\U^2\tilde\Y_\eta(t,\eta)+ \tilde\Henergy_\eta(t,\eta)=A^5.
\end{equation}
Moreover, we can show that
\begin{align}\label{evik:estB}
4\norm{\tilde {\P}(t,\dott)}_{L^\infty}&=4\norm{A^2\P(t,\dott)}_{L^\infty}\leq A^4 \\ 
2\norm{\tilde\U^2(t,\dott)}_{L^\infty}&=2\norm{A^2\U^2(t,\dott)}_{L^\infty}\leq A^4. \nn
\end{align}

We now collect all the estimates we will need on our solutions in the new variables with the explicit dependence on $A$. This will be repeatedly used in the next section.
 \begin{subequations} \label{eq:all_estimates}
 \begin{alignat}{2}
0&\le4\tilde \P\le A^4, \quad & & \text{[from \eqref{evik:estB} and \eqref{LI2}]}\label{eq:all_estimatesA} \\
\sqrt{2}\abs{\tilde \U} &\le A^2,  \quad & & \text{[from  \eqref{evik:estB} and \eqref{LI1}]} \label{eq:all_estimatesB}\\
\abs{\tilde\Q}&\le A\tilde\P,   \quad & & \text{[from common sense]} \label{eq:all_estimatesC} \\
 \tilde \U^2 &\le 2\tilde \P,   \quad & & \text{[from  \eqref{con:PU}]} \label{eq:all_estimatesD} \\
2\tilde \P \tilde \Y_\eta&\le A^5,   \quad & & \text{[from   \eqref{evik:est}]} \label{eq:all_estimatesE} \\
2 \abs{\tilde \U \tilde \U_\eta}&\le A^4,  \quad & & \text{[from  \eqref{evik:est}]} \label{eq:all_estimatesF}\\
0\leq\tilde \U^2 \tilde\Y_\eta&\le A^5,  \quad & & \text{[from   \eqref{evik:est}]} \label{eq:all_estimatesG} \\
 \tilde \U^2_\eta &\le A^3 \tilde \Y_\eta,  \quad & & \text{[from \eqref{evik:estA} and \eqref{stand:est2}]} \label{eq:all_estimatesH}\\
\sqrt{2}\abs{\tilde\U}\tilde\P^{1/2} \tilde \Y_\eta &\le A^5,  \quad & & \text{[from \eqref{eq:all_estimatesE} and \eqref{eq:all_estimatesG}]} \label{eq:all_estimatesI}\\
\abs{\tilde \R}&\le \bigO(1)A^3 \tilde \P,  \quad & & \text{[from  \eqref{eq:R}]}\label{eq:all_estimatesJ} \\
0&\le\tilde \Henergy_\eta\le A^5,  \quad & & \text{[from  \eqref{evik:est}]} \label{eq:all_estimatesK} \\
0&\le  \tilde \Y_\eta,   \quad & & \text{[from  the definition]}  \label{eq:all_estimatesL} \\
A^2\tilde \U_\eta^2&\le \tilde \Henergy_\eta \tilde\Y_\eta,  \quad & & \text{[from  \eqref{evik:estA} and \eqref{stand:est2}]} \label{eq:all_estimatesM}\\
0&\leq \tilde \D \leq 2A\tilde \P,  \quad & & \text{[from \eqref{eq:DPEP}]} \label{eq:all_estimatesN}\\
0&\leq \tilde \E \leq 2A\tilde\P,   \quad & & \text{[from  \eqref{eq:DPEP}]} \label{eq:all_estimatesO}\\
2\sqrt{2}\tilde \P \tilde \U_\eta&\le A^6,  \quad & & 
\text{[from \eqref{eq:all_estimatesA}, \eqref{eq:all_estimatesE}, and \eqref{eq:all_estimatesH}]} \label{eq:all_estimatesP} \\
2\tilde \P \tilde \U_\eta^2&\le A^8,  \quad & & \text{[from \eqref{eq:all_estimatesE} and \eqref{eq:all_estimatesH}]} \label{eq:all_estimatesQ} \\
4\tilde \P \tilde \U_\eta^2&\le A^7\tilde\Y_\eta,  \quad & & \text{[from \eqref{eq:all_estimatesA} and \eqref{eq:all_estimatesH}]}. \label{eq:all_estimatesR}
\end{alignat}
\end{subequations}

In addition, we have have several integrals that appear frequently:
\begin{subequations} \label{eq:all_Pestimates}
\begin{align}
\int_0^\eta e^{-\frac{1}{A}(\tilde\Y(t,\eta)-\tilde\Y(t,\theta))}\tilde \U^2(t,\theta) d\theta & \le 6\tilde \P(t,\eta), \quad & & \text{[from  \eqref{est:L3b}]}, \label{eq:all_PestimatesA}\\
\int_0^\eta e^{-\frac{1}{A}(\tilde\Y(t,\eta)-\tilde\Y(t,\theta))}\tilde \P^2\tilde \Y_\eta(t,\theta) d\theta & \le \frac32 A^5\tilde \P(t,\eta), \quad & & \text{[from  \eqref{rel:P2PP} and \eqref{eq:all_estimatesA}]}, \label{eq:all_PestimatesB}\\
\int_0^\eta e^{-\frac{1}{A}(\tilde\Y(t,\eta)-\tilde\Y(t,\theta))}\tilde \U^2\tilde \Y_\eta(t,\theta) d\theta & \le 4A\tilde \P(t,\eta), \quad & & \text{[from  \eqref{eq:PestcalA}]},\label{eq:all_PestimatesC} \\
\int_0^\eta e^{-\frac{1}{A}(\tilde\Y(t,\eta)-\tilde\Y(t,\theta))}\tilde \Henergy_\eta(t,\theta) d\theta & \le 4A\tilde \P(t,\eta), \quad & & \text{[from  \eqref{eq:PestcalA}]},  \label{eq:all_PestimatesD}  \\
\int_0^\eta e^{-\frac3{2A} (\tilde\Y(t,\eta)-\tilde\Y(t,\theta))}\tilde\P\tilde\Y_{\eta} (t,\theta) d\theta&\leq 2A\tilde\P(t,\eta), \quad & & \text{[from  Lemma \ref{lemma:E}]}, 
\label{eq:343}  \\
\int_0^\eta e^{-\frac{3}{2A}(\tilde \Y(t,\eta)-\tilde \Y(t,\theta))} \tilde \Henergy_{\eta}(t,\theta) d\theta 
& \le 4A \tilde\P(t,\eta), \quad & & \text{[from  Lemma \ref{lemma:E}]},  \label{eq:Henergy32} \\
\int_0^\eta e^{-\frac{1}{A}(\tilde\Y(t,\eta)-\tilde\Y(t,\theta))} \tilde\U^2(t,\theta)d \theta & \leq 6 \tilde\P(t,\eta), \quad & & \text{[from  Lemma \ref{lemma:E}]},  \label{est:L3b} \\
\int_0^\eta e^{-\frac{5}{4A}(\tilde \Y(t,\eta)-\tilde \Y(t,\theta))} \tilde \P\tilde \Y_{\eta}(t,\theta) d\theta&\leq 4 A\tilde\P(t,\eta), \quad & & \text{[from  Lemma \ref{lemma:E}]}
   \label{eq:all_PestimatesE} \\
\int_0^\eta e^{-\frac3{2A} (\tilde\Y(t,\eta)-\tilde\Y(t,\theta))}\tilde \P(t,\theta) d\theta
& \leq 7\tilde\P(t,\eta), \quad & & \text{[from  Lemma \ref{lemma:E}]},  \label{eq:32P} \\
\int_0^\eta e^{-\frac1{2A}(\tilde\Y(t,\eta)-\tilde\Y(t,\theta))} \tilde\P^2\tilde\Y_{\eta} (t,\theta) d\theta&\leq A\bigO(1)\P^{1/2}(t,\eta), 
\quad & & \text{[from  Lemma \ref{lemma:E}]} \label{eq:p2y}  \\
\int_0^\eta e^{-\frac{1}{A}(\tilde \Y(t,\eta)-\tilde \Y(t,\theta))} \tilde \P^{1+\beta}\tilde \Y_{\eta}(t,\theta) d\theta
&\leq 3\frac{1+\beta}{\beta}\frac{A^{1+4\beta}}{4^{\beta}}\tilde \P(t,\eta), \, \beta>0,\, & & \text{[from  Lemma \ref{lemma:E}]}.  \label{eq:general}
\end{align}
\end{subequations}
As for derivatives of the quantity $\P^{1/2}$ we have
\begin{subequations}\label{eq:all_Pderiv_estimates}
\begin{align}
\abs{(\tilde \P^{1/2})_\eta}&\le \frac1{2A} \tilde \P^{1/2}\tilde  \Y_\eta, \label{eq:all_Pderiv_estimatesA}\\
\abs{\tilde \U(\tilde \P^{1/2})_\eta}&\le \frac38 A^4, \label{eq:all_Pderiv_estimatesB}\\
\big((\tilde \P^{1/2})_\eta\big)^2&\le \frac18 A^3 \tilde \Y_\eta, \label{eq:all_Pderiv_estimatesC}\\
\big((\tilde \P^{1/2})_\eta\big)^2&\le \frac1{16} A^2\tilde  \Y_\eta^2.\label{eq:all_Pderiv_estimatesD}
\end{align}
\end{subequations}

%%%%%%%%%%%%%%
%%%%%%%%%%%%%%
%%%%%%%%%%%%%%

\section{Lipschitz estimates in the new metric} \label{sec:Lip}

We will repeatedly use the following elementary identity and estimate.
\begin{lemma}  \label{lemma:LUR}
Let $a_j,b_j\in\Real$ for $j=1,2$. Then we have
\begin{equation*}
a_2b_2-a_1b_1=(b_2-b_1)\big(a_2\mathbbm{1}_{b_1<b_2}+a_1\mathbbm{1}_{b_1\ge b_2} \big)+\min(b_1,b_2)(a_2-a_1).
\end{equation*}
Here 
\[
\mathbbm{1}_\mathcal{K}=\begin{cases}
1,& \text{$\mathcal{K}$ is true}, \\
0, & \text{$\mathcal{K}$ is false.}
\end{cases}
\]
We also use $\mathbbm{1}_\mathcal{K}$ to denote the characteristic (indicator) function of a set $\mathcal{K}$.
Furthermore, we have the following estimate:
\begin{equation*}
\abs{\min(a_1,b_1)-\min(a_2,b_2)}\le \max\big(\abs{a_1-a_2},\abs{b_1-b_2}\big).
\end{equation*}
\end{lemma}
\begin{lemma}
Let $f,g\colon\Real\to\Real$ be two Lipschitz continuous functions with Lipschitz constant $c_f$ and $c_g$, respectively. Then the function
$h\colon\Real\to\Real$ defined by $h=\min(f,g)$ is  Lipschitz continuous function with Lipschitz constant bounded by $\max(c_f,c_g)$.
\end{lemma}
\begin{proof}
Use the previous lemma.
\end{proof}
Introduce for any function $\Phi$ the functions  
\begin{equation}\label{eq:pm0}
\Phi^-=\min(0, \Phi), \quad \Phi^+=\max(0, \Phi).
\end{equation}
We then have
\begin{equation}\label{eq:pm1}
\Phi^-\le0\le  \Phi^+, \quad \Phi^-\Phi^+=0, \quad \Phi=\Phi^+ +\Phi^-, \quad \abs{\Phi}=\Phi^+ -\Phi^-. 
\end{equation}
For two functions $\Phi,\Psi$ we have
\begin{equation}\label{eq:pm2}
\vert \Phi^\pm - \Psi^\pm\vert  \le\vert \Phi- \Psi\vert.
\end{equation}

Frequently we will have to estimate quantities like 
\[
 \int_0^{1} e^{-\frac{1}{A}\vert \tilde\Y(t,\eta)-\tilde\Y(t,\theta)\vert } (\cdots)(t,\theta) d\theta.
\]
We will rewrite it as follows
\begin{align} \nn
 \int_0^{1}& e^{-\frac{1}{A}\vert \tilde\Y(t,\eta)-\tilde\Y(t,\theta)\vert } (\cdots)(t,\theta) d\theta\\  \nn
 &=
\big( \int_0^{\eta}+\int^1_{\eta}\big) e^{-\frac{1}{A}\vert \tilde\Y(t,\eta)-\tilde\Y(t,\theta)\vert } (\cdots)(t,\theta) d\theta\\
&=  \int_0^{\eta} e^{-\frac{1}{A}(\tilde\Y(t,\eta)-\tilde\Y(t,\theta))} (\cdots)(t,\theta) d\theta
+ \int^1_{\eta} e^{\frac{1}{A}(\tilde\Y(t,\eta)-\tilde\Y(t,\theta)) } (\cdots)(t,\theta) d\theta. \label{eq:triks}
\end{align}
Both integrals can be estimated in the same manner, using that the argument in the exponential is negative in both cases.  Eliminating the absolute value will allow us to perform integration by parts.

\subsection{Lipschitz estimates for $\tilde\Y$}\label{subsec:lipy}
From the system of differential equations, we have
\begin{equation}\label{eq:Ylip}
\tilde \Y_{i,t}+(\frac23 \frac{1}{\sqtC{i}^5} \tilde \U_i^3+\frac{1}{\sqtC{i}^6}\tilde \So_i)\tilde \Y_{i,\eta}=\tilde \U_i,
\end{equation}
where 
\begin{align*}
\tilde \P_i(t,\eta)&=\frac{1}{4\sqtC{i}} \int_0^{1} e^{-\frac{1}{\sqtC{i}}\vert \tilde \Y_i(t,\eta)-\tilde\Y_i(t,\theta)\vert} (2(\tilde \U_i^2-\tilde \P_i)\tilde \Y_{i,\eta}(t,\theta)+\sqtC{i}^5)d\theta,\\
\tilde \Q_i(t,\eta)&= -\frac14 \int_0^{1} \sign(\eta-\theta) e^{-\frac{1}{\sqtC{i}}\vert \tilde\Y_i(t,\eta)-\tilde\Y_i(t,\theta)\vert} (2(\tilde\U_i^2-\tilde\P_i)\tilde\Y_{i,\eta}(t,\theta)+\sqtC{i}^5) d\theta,\\
\tilde\So_i(t,\eta)& = \int_0^{1} e^{-\frac{1}{\sqtC{i}}\vert \tilde\Y_i(t,\eta)-\tilde\Y_i(t,\theta)\vert } (\frac23 \tilde\U_i^3\tilde\Y_{i,\eta}-\tilde\Q_i\tilde\U_{i,\eta}-2\tilde\P_i\tilde\U_i\tilde\Y_{i,\eta})(t,\theta)d\theta.
\end{align*}
Thus we have 
\begin{align} \nn
\frac{d}{dt} \int_0^{1} (\tilde\Y_1-\tilde\Y_2)^2(t,\eta)d\eta &= 2\int_0^{1} (\tilde\Y_1-\tilde \Y_2)(\tilde\Y_{1,t}-\tilde\Y_{2,t})(t,\eta) d\eta\\ \nn
& = 2\int_0^{1} (\tilde\Y_1-\tilde \Y_2)(\tilde\U_1-\tilde\U_2)(t,\eta)d\eta\\ \nn
& \quad + \frac4{3} \int_0^{1} (\tilde\Y_1-\tilde\Y_2)(\frac{1}{\sqtC{2}^5}\tilde\U_2^3\tilde\Y_{2,\eta}-\frac{1}{\sqtC{1}^5}\tilde\U_1^3\tilde\Y_{1,\eta})(t,\eta)d\eta\\ \nn
& \quad + 2 \int_0^{1} (\tilde\Y_1-\tilde\Y_2)(\frac{1}{\sqtC{2}^6}\tilde\So_2\tilde\Y_{2,\eta}-\frac{1}{\sqtC{1}^6}\tilde\So_1\tilde\Y_{1,\eta})(t,\eta)d\eta\\
& = I_1+I_2+I_3. \label{eq:I1-3}
\end{align}
The strategy is to use integration by parts for the last two integrals $I_2$ and $I_3$, while we want to use straight forward estimates for $I_1$, which will finally yield that 
\begin{align*}
\frac{d}{dt}&\norm{\tilde\Y_1-\tilde\Y_2}^2\\ &\leq \bigO(1)\Big(\norm{\tilde\Y_1-\tilde\Y_2}^2+\norm{\tilde\U_1-\tilde\U_2}^2+ 
\norm{\sqP{1}-\sqP{2}}^2
+\vert \sqtC{1}-\sqtC{2}\vert ^2\Big),
\end{align*}
where $\bigO(1)$ denotes some constant which depends on $A=\max_j(A_j)$ and which remains bounded as $A\to0$ .

{\it The first integral $I_1$:} Note that 
\begin{align*}
\vert I_1\vert=\vert 2\int_0^{1} (\tilde\Y_1-\tilde\Y_2)(\tilde\U_1-\tilde\U_2)(t,\eta) d\eta\vert & \leq \int_0^{1} ((\tilde\Y_1-\tilde\Y_2)^2+(\tilde\U_1-\tilde\U_2)^2)(t,\eta) d\eta\\
& = \norm{\tilde\Y_1-\tilde\Y_2}^2+\norm{\tilde\U_1-\tilde\U_2}^2.
\end{align*}

{\it The second integral $I_2$:}
Note that 
\begin{align*} 
\frac34\vert I_2\vert&=\vert \int_0^{1} (\tilde\Y_1-\tilde\Y_2) (\frac{1}{\sqtC{2}^5}\tilde\U_2^3 \tilde\Y_{2,\eta}-\frac{1}{\sqtC{1}^5}\tilde\U_1^3\tilde\Y_{1,\eta})(t,\eta)d\eta\vert \\ 
& \leq \frac{1}{(\max_j(\sqtC{j}))^5}\vert\int_0^1 (\tilde\Y_1-\tilde\Y_2) (\tilde\U_2^3\tilde\Y_{2,\eta}-\tilde\U_1^3\tilde\Y_{1,\eta})(t,\eta) d\eta\vert \\
& \quad +  \frac{\vert \sqtC{1}^5-\sqtC{2}^5\vert}{\sqtC{1}^5\sqtC{2}^5} \vert \int_0^1 (\tilde\Y_1-\tilde\Y_2)(\tilde\U_1^3\tilde\Y_{1,\eta}\mathbbm{1}_{\sqtC{1}\leq \sqtC{2}}+ \tilde\U_2^3\tilde\Y_{2,\eta} \mathbbm{1}_{\sqtC{2}< \sqtC{1}})(t,\eta)d\eta\vert \\
& \leq \frac{1}{\A^5}\vert \int_0^{1} (\tilde\Y_1- \tilde\Y_2)(\tilde\U_2-\tilde\U_1) \tilde\U_2^2\tilde\Y_{2,\eta}(t,\eta) d\eta\vert\\ 
& \quad +\frac{1}{\A^5}\vert \int_0^{1} (\tilde\Y_1- \tilde\Y_2)[\tilde\U_1\tilde\Y_{1,\eta} \mathbbm{1}_{\tilde\U_2^2\leq \tilde\U_1^2}+\tilde\U_1 \tilde\Y_{2,\eta}\mathbbm{1}_{\tilde \U_1^2<\tilde\U_2^2}]\\
&\qquad\qquad\qquad\qquad\qquad\qquad\qquad\qquad \times(\tilde\U_2+\tilde\U_1)(\tilde\U_2-\tilde\U_1)(t,\eta) d\eta\vert\\ 
& \quad + \frac1{\A^5}\vert \int_0^{1} (\tilde\Y_1-\tilde\Y_2)\tilde\U_1\min_j(\tilde\U_j^2)( \tilde\Y_{2,\eta}-\tilde\Y_{1,\eta})(t,\eta)d \eta\vert\\
&\quad +  \frac{\vert \sqtC{1}^5-\sqtC{2}^5\vert}{\sqtC{1}^5\sqtC{2}^5} \vert \int_0^1 (\tilde\Y_1-\tilde\Y_2)(\tilde\U_1^3\tilde\Y_{1,\eta}\mathbbm{1}_{\sqtC{1}\leq \sqtC{2}}+ \tilde\U_2^3\tilde\Y_{2,\eta} \mathbbm{1}_{\sqtC{2}< \sqtC{1}})(t,\eta)d\eta\vert \\
& = J_1+J_2+J_3+J_4.
\end{align*}
We will write
\begin{equation}\label{eq:Amaxmin}
A=\max_j(A_j), \quad a=\min_j(A_j).
\end{equation}
We commence with $J_1$: Since $\tilde \U_i^2\tilde \Y_{i,\eta}(t,\eta)\leq \sqtC{i}^5\leq \A^5$ for all $t$ and $\eta$, we have 
\begin{align*}
J_1&=\frac{1}{A^5}\vert \int_0^{1} (\tilde\Y_1-\tilde\Y_2)(\tilde\U_2-\tilde\U_1)\tilde\U_2^2\tilde\Y_{2,\eta}(t,\eta) d\eta\vert \\
&\leq  \norm{\tilde\Y_1-\tilde \Y_2}^2+\norm{\tilde\U_1-\tilde\U_2}^2.
\end{align*}

Next term is $J_2$: Using once more that $\tilde \U_i^2\tilde\Y_{i,\eta}(t,\eta)\leq A^5$ for all $t$ and $\eta$, we get 
\begin{align*} \nn
J_2&=\frac{1}{\A^5}\vert \int_0^{1} (\tilde\Y_1-\tilde\Y_2) [\tilde\U_1\tilde\Y_{1,\eta} \mathbbm{1}_{\tilde\U_2^2\leq \tilde\U_1^2}+\tilde\U_1\tilde\Y_{2,\eta}\mathbbm{1}_{ \tilde\U_1^2<\tilde\U_2^2}]\\ \nn
&\qquad\qquad\qquad\qquad\qquad\qquad\qquad\qquad \times(\tilde\U_2+\tilde\U_1)(\tilde\U_2-\tilde\U_1)(t,\eta) d\eta\vert\\ \nn
& \leq \frac{1}{\A^5}\int_0^{1} \vert \tilde\Y_1-\tilde\Y_2\vert \vert \tilde\U_1- \tilde\U_2\vert (2\tilde\U_1^2\tilde\Y_{1,\eta}\mathbbm{1}_{\tilde\U_2^2\leq \tilde\U_1^2}+2\tilde\U_2^2 \tilde\Y_{2,\eta}\mathbbm{1}_{\tilde\U_1^2<\tilde\U_2^2})(t,\eta)d\eta\\
& \leq 2(\norm{\tilde\Y_1-\tilde\Y_2}^2+\norm{\tilde\U_1-\tilde\U_2}^2).
\end{align*}

Next, it is $J_3$: Here integration by parts plays an essential role. Thus we have to determine first whether or not the function $\eta\mapsto\tilde\U_1\min_j(\tilde \U_j^2)(t,\eta)$ is differentiable, and, if so, if its derivative is bounded. Recall from Lemma \ref{lemma:1} (ii) that the function $\eta\mapsto \min_j(\tilde \U_j^2)(t,\eta)$ is Lipschitz continuous with Lipschitz constant at most $A^4$. Thus  
\begin{equation}  \label{eq:U2lip}
\vert \frac{d}{d\eta}\left(\tilde\U_1(t,\eta)\min_j(\tilde\U_j^2(t,\eta))\right)\vert \leq \frac32 A^4 \norm{\tilde\U_1}_{L^\infty}\leq \bigO(1)A^6,
\end{equation}
and integration by parts together with \eqref{decay:impl} yields
\begin{align*}
J_3&=\frac{1}{\A^5}\vert \int_0^{1} (\tilde\Y_1-\tilde\Y_2) \tilde\U_1\min_j(\tilde\U_j^2)(\tilde\Y_{2,\eta}-\tilde\Y_{1,\eta})(t,\eta)d \eta\vert \\
& = \big\vert-\frac1{2\A^5} (\tilde\Y_1-\tilde\Y_2)^2\tilde\U_1\min_j(\tilde\U_j^2)(t,\eta)\big|_{\eta=0}^{1}\\
& \qquad + \frac1{2A^5} \int_0^{1} (\tilde\Y_1-\tilde\Y_2)^2 \frac{d}{d\eta} (\tilde\U_1\min_j(\tilde\U_j^2)) (t,\eta) d\eta\big\vert \\
& = \frac1{2A^5} \vert \int_0^{1} (\tilde\Y_1-\tilde\Y_2)^2 \frac{d}{d\eta} (\tilde\U_1\min_j(\tilde\U_j^2)) (t,\eta) d\eta\vert\\
& \leq \frac{3}{4\A}\norm{\tilde\U_1(t,\dott)}_{L^\infty} \int_0^{1} (\tilde\Y_1-\tilde\Y_2)^2(t,\eta)d\eta\\
& \leq \bigO(1) \norm{\tilde\Y_1-\tilde\Y_2}^2.
\end{align*}

Finally, we consider $J_4$: Direct calculations yield 
\begin{align*}
J_4&\leq \frac{\vert \sqtC{1}^5-\sqtC{2}^5\vert}{\sqtC{1}^5\sqtC{2}^5} \Big( \mathbbm{1}_{\sqtC{1}\leq \sqtC{2}}\vert \int_0^1 (\tilde\Y_1-\tilde\Y_2)\tilde\U_1^3\tilde\Y_{1,\eta}(t,\eta) d\eta\vert \\
&\qquad\qquad\qquad\qquad+\mathbbm{1}_{\sqtC{2}< \sqtC{1}}\vert \int_0^1(\tilde\Y_1-\tilde\Y_2)\tilde \U_2^3\tilde\Y_{2,\eta}(t,\eta) d\eta\vert \Big)\\
& \leq \frac{\vert \sqtC{1}^5-\sqtC{2}^5\vert}{\sqtC{1}^5\sqtC{2}^5} \left(\mathbbm{1}_{\sqtC{1}\leq \sqtC{2}}\sqtC{1}^7+\mathbbm{1}_{\sqtC{2}< \sqtC{1}} \sqtC{2}^7\right) \int_0^1\vert \tilde\Y_1-\tilde\Y_2\vert (t,\eta)d\eta\\
& \leq \frac{\vert \sqtC{1}^5-\sqtC{2}^5\vert}{\A^3}\norm{\tilde\Y_1-\tilde\Y_2}\\
& \leq 5\A  \vert \sqtC{1}-\sqtC{2}\vert \norm{\tilde\Y_1-\tilde\Y_2}\\
& \leq \bigO(1)\Big(\norm{\tilde\Y_1-\tilde\Y_2}^2+\vert \sqtC{1}-\sqtC{2}\vert ^2\Big),
\end{align*}
where we again used that $\tilde\U_i^2\tilde\Y_{i,\eta}\leq \sqtC{i}^5$ and $\norm{\tilde\U_i}_{L^\infty}\leq \sqtC{i}^2$.

{\it The third integral $I_3$:} We will consider several smaller parts of $I_3$ by inserting the definition of $\tilde\So_i$,  and combine them in the end.  We write
\begin{equation}  \label{eq:I3_delt}
I_3=\frac43I_{31}-4I_{32}-2I_{33},
\end{equation}
where
\begin{align*}
I_{31}&= \int_0^{1} (\tilde\Y_1-\tilde\Y_2)(t,\eta)\\ 
& \qquad \times\Big(\frac{1}{\sqtC{2}^6}\tilde\Y_{2,\eta}(t,\eta)\int_0^{1} e^{-\frac{1}{\sqtC{2}}\vert \tilde\Y_2(t,\eta)-\tilde\Y_2(t,\theta)\vert}\tilde\U_2^3\tilde\Y_{2,\eta}(t,\theta)d\theta\\
&\qquad\qquad\qquad\qquad\qquad-\frac{1}{\sqtC{1}^6}\tilde\Y_{1,\eta}(t,\eta)\int_0^{1} e^{-\frac{1}{\sqtC{1}}\vert \tilde\Y_1(t,\eta)-\tilde\Y_1(t,\theta)\vert} \tilde\U_1^3\tilde\Y_{1,\eta}(t,\theta) d\theta\Big)d\eta,\\[2mm]
I_{32}&=\int_0^{1} (\tilde\Y_1-\tilde\Y_2)(t,\eta)\\ 
& \qquad\quad \times\Big(\frac{1}{\sqtC{2}^6}\tilde\Y_{2,\eta}(t,\eta)\int_0^{1} e^{-\frac{1}{\sqtC{2}}\vert \tilde\Y_2(t,\eta)-\tilde\Y_2(t,\theta)\vert}\tilde\P_2\tilde\U_2\tilde\Y_{2,\eta}(t,\theta)d\theta\\
&\qquad\qquad\qquad\qquad-\frac{1}{\sqtC{1}^6}\tilde\Y_{1,\eta}(t,\eta)\int_0^{1} e^{-\frac{1}{\sqtC{1}}\vert \tilde\Y_1(t,\eta)-\tilde\Y_1(t,\theta)\vert} \tilde\P_1\tilde\U_1\tilde\Y_{1,\eta}(t,\theta) d\theta\Big)d\eta,   \\[2mm]
I_{33}&=\int_0^{1} (\tilde\Y_1-\tilde\Y_2)(t,\eta)\\ 
& \qquad \times\Big(\frac{1}{\sqtC{2}^6}\tilde\Y_{2,\eta}(t,\eta)\int_0^{1} e^{-\frac{1}{\sqtC{2}}\vert \tilde\Y_2(t,\eta)-\tilde\Y_2(t,\theta)\vert}\tilde\Q_2\tilde\U_{2,\eta}(t,\theta)d\theta\\
&\qquad\qquad\qquad\qquad\qquad-\frac{1}{\sqtC{1}^6}\tilde\Y_{1,\eta}(t,\eta)\int_0^{1} e^{-\frac{1}{\sqtC{1}}\vert \tilde\Y_1(t,\eta)-\tilde\Y_1(t,\theta)\vert} \tilde\Q_1\tilde\U_{1,\eta}(t,\theta) d\theta\Big)d\eta.
\end{align*}

Recall the definitions \eqref{eq:pm0}. Then we have 
\begin{align*}
I_{31}&= \int_0^{1} (\tilde\Y_1-\tilde\Y_2)(t,\eta)\\ 
& \qquad \times\Big(\frac{1}{\sqtC{2}^6}\tilde\Y_{2,\eta}(t,\eta)\int_0^{1} e^{-\frac{1}{\sqtC{2}}\vert \tilde\Y_2(t,\eta)-\tilde\Y_2(t,\theta)\vert}\tilde\U_2^3\tilde\Y_{2,\eta}(t,\theta)d\theta\\
&\qquad\qquad\qquad-\frac{1}{\sqtC{1}^6}\tilde\Y_{1,\eta}(t,\eta)\int_0^{1} e^{-\frac{1}{\sqtC{1}}\vert \tilde\Y_1(t,\eta)-\tilde\Y_1(t,\theta)\vert} \tilde\U_1^3\tilde\Y_{1,\eta}(t,\theta) d\theta\Big)d\eta\\
& \quad =  \int_0^{1} (\tilde\Y_1-\tilde\Y_2)(t,\eta)\\
& \qquad \times\Big(\frac{1}{\sqtC{2}^6}\tilde\Y_{2,\eta}(t,\eta)\int_0^{1} e^{-\frac{1}{\sqtC{2}}\vert \tilde\Y_2(t,\eta)-\tilde\Y_2(t,\theta)\vert}(\tilde\V_2^-)^3\tilde\Y_{2,\eta}(t,\theta)d\theta\\
&\qquad\qquad\qquad\qquad-\frac{1}{\sqtC{1}^6}\tilde\Y_{1,\eta}(t,\eta)\int_0^{1} e^{-\frac{1}{\sqtC{1}}\vert \tilde\Y_1(t,\eta)-\tilde\Y_1(t,\theta)\vert} (\tilde\V_1^-)^3\tilde\Y_{1,\eta}(t,\theta) d\theta\Big)d\eta\\
& \qquad +  \int_0^{1} (\tilde\Y_1-\tilde\Y_2)(t,\eta)\\ 
& \qquad \times\Big(\frac{1}{\sqtC{2}^6}\tilde\Y_{2,\eta}(t,\eta)\int_0^{1} e^{-\frac{1}{\sqtC{2}}\vert \tilde\Y_2(t,\eta)-\tilde\Y_2(t,\theta)\vert}(\tilde\V_2^+)^3\tilde\Y_{2,\eta}(t,\theta)d\theta\\
&\qquad\qquad\qquad\qquad-\frac{1}{\sqtC{1}^6}\tilde\Y_{1,\eta}(t,\eta)\int_0^{1} e^{-\frac{1}{\sqtC{1}}\vert \tilde\Y_1(t,\eta)-\tilde\Y_1(t,\theta)\vert} (\tilde\V_1^+)^3\tilde\Y_{1,\eta}(t,\theta) d\theta\Big)d\eta \\
&= \int_0^{1} (\tilde\Y_1-\tilde\Y_2)(\tilde K^-+\tilde K^+) (t,\eta)d\eta.
\end{align*}
Since both integrals have the same structure, it suffices to consider the second integral. Furthermore, applying the device described in \eqref{eq:triks}, it suffices to study the terms, denoted  $K^\pm$, with the upper limit of the inner integral replaced by $\eta$. 

Therefore observe that we can write
\begin{align} \nn
& K^+(t,\eta) \\ \nn
&\quad= \frac{1}{\sqtC{2}^6}\tilde\Y_{2,\eta}(t,\eta)  \int_0^{\eta} e^{-\frac{1}{\sqtC{2}}(\tilde\Y_2(t,\eta)-\tilde\Y_2(t,\theta))} (\tilde\V_2^+)^3\tilde\Y_{2,\eta}(t,\theta) d\theta\\ \nn
&\quad\qquad\qquad-\frac{1}{\sqtC{1}^6}\tilde\Y_{1,\eta} (t,\eta)\int_0^{\eta} e^{-\frac{1}{\sqtC{1}}(\tilde\Y_1(t,\eta)-\tilde\Y_1(t,\theta))} (\tilde \V^+_1)^3\tilde\Y_{1,\eta}(t,\theta) d\theta\\ \nn
&\quad = \frac{1}{(\max_j(\sqtC{j}))^6}\Big(\tilde \Y_{2,\eta}(t,\eta)\int_0^\eta e^{-\frac{1}{\sqtC{2}}(\tilde\Y_2(t,\eta)-\tilde\Y_2(t,\theta))}(\tilde\V_2^+)^3\tilde\Y_{2,\eta}(t,\theta) d\theta\\ \nn
&\quad \qquad \qquad \qquad \qquad\qquad -\tilde\Y_{1,\eta}(t,\eta)\int_0^\eta e^{-\frac{1}{\sqtC{1}}(\tilde\Y_1(t,\eta)-\tilde\Y_1(t,\theta))}(\tilde\V_1^+)^3\tilde\Y_{1,\eta}(t,\theta)d\theta\Big)\\ \nn
& \quad\quad + \frac{\sqtC{1}^6-\sqtC{2}^6}{\sqtC{1}^6\sqtC{2}^6}\mathbbm{1}_{\sqtC{1}\leq \sqtC{2}} \tilde \Y_{1,\eta}(t,\eta) \int_0^\eta e^{-\frac{1}{\sqtC{1}}(\tilde\Y_1(t,\eta)-\tilde\Y_1(t,\theta))} (\tilde\V_1^+)^3\tilde \Y_{1,\eta}(t,\theta) d\theta\\ \nn
&\quad \quad + \frac{\sqtC{1}^6-\sqtC{2}^6}{\sqtC{1}^6\sqtC{2}^6} \mathbbm{1}_{\sqtC{2}< \sqtC{1}} \tilde \Y_{2,\eta}(t,\eta)\int_0^\eta e^{-\frac{1}{\sqtC{2}}(\tilde\Y_2(t,\eta)-\tilde\Y_2(t,\theta))}(\tilde\V_2^+)^3\tilde\Y_{2,\eta}(t,\theta)d\theta\\ \nn
&\quad = \frac{1}{\A^6}\tilde\Y_{2,\eta}(t,\eta)\\ \nn
&\qquad\times \left(\int_0^{\eta} e^{-\frac{1}{\sqtC{2}}(\tilde\Y_2(t,\eta)-\tilde\Y_2(t,\theta))} ((\tilde\V_2^+)^3-(\tilde\V_1^+)^3) \tilde\Y_{2,\eta}(t,\theta) \mathbbm{1}_{\tilde\V_1^+\leq \tilde\V_2^+}(t,\theta)d\theta\right)\\ \nn
&\quad \quad +\frac{1}{\A^6}\tilde\Y_{1,\eta}(t,\eta)\\ \nn
&\qquad\times\left(\int_0^\eta e^{-\frac{1}{\sqtC{1}}(\tilde\Y_1(t,\eta)-\tilde\Y_1(t,\theta))}((\tilde\V_2^+)^3- (\tilde\V_1^+)^3)\tilde\Y_{1,\eta}(t,\theta) \mathbbm{1}_{\tilde\V_2^+<\tilde\V_1^+}(t,\theta)d\theta\right)\\ \nn
&\quad \quad + \frac{1}{\A^6}\tilde\Y_{2,\eta}(t,\eta)\left(\int_0^{\eta} e^{-\frac{1}{\sqtC{2}}(\tilde\Y_2(t,\eta)-\tilde\Y_2(t,\theta))}\min_j(\tilde\V_j^+)^3\tilde\Y_{2,\eta}(t,\theta)d\theta\right) \\ \nn
& \quad\quad-\frac{1}{\A^6}\tilde\Y_{1,\eta}(t,\eta) \left(\int_0^{\eta} e^{-\frac{1}{\sqtC{1}}(\tilde\Y_1(t,\eta)-\tilde\Y_1(t,\theta))} \min_j(\tilde\V_j^+)^3\tilde\Y_{1,\eta} (t,\theta) d\theta\right)\\ \nn
&\quad \quad + \frac{\sqtC{1}^6-\sqtC{2}^6}{\sqtC{1}^6\sqtC{2}^6}\mathbbm{1}_{\sqtC{1}\leq \sqtC{2}} \tilde \Y_{1,\eta}(t,\eta) \int_0^\eta e^{-\frac{1}{\sqtC{1}}(\tilde\Y_1(t,\eta)-\tilde\Y_1(t,\theta))} (\tilde\V_1^+)^3\tilde \Y_{1,\eta}(t,\theta) d\theta\\ \nn
&\quad \quad + \frac{\sqtC{1}^6-\sqtC{2}^6}{\sqtC{1}^6\sqtC{2}^6} \mathbbm{1}_{\sqtC{2}< \sqtC{1}} \tilde \Y_{2,\eta}(t,\eta)\int_0^\eta e^{-\frac{1}{\sqtC{2}}(\tilde\Y_2(t,\eta)-\tilde\Y_2(t,\theta))}(\tilde\V_2^+)^3\tilde\Y_{2,\eta}(t,\theta)d\theta\\ \nn
& \quad= \frac{1}{\A^6}\tilde\Y_{2,\eta}(t,\eta) \Big(\int_0^{\eta} e^{-\frac{1}{\sqtC{2}}(\tilde\Y_2(t,\eta)-\tilde\Y_2(t,\theta))}\\ \nn
&\qquad\qquad\qquad\qquad\qquad\times ((\tilde\V_2^+)^3-(\tilde\V_1^+)^3) \tilde\Y_{2,\eta}(t,\theta) \mathbbm{1}_{\tilde\V_1^+\leq \tilde\V_2^+}(t,\theta)d\theta\Big)\\ \nn
&\quad \quad +\frac{1}{\A^6}\tilde\Y_{1,\eta}(t,\eta)\Big(\int_0^\eta e^{-\frac{1}{\sqtC{1}}(\tilde\Y_1(t,\eta)-\tilde\Y_1(t,\theta))}\\ \nn
&\qquad\qquad\qquad\qquad\qquad\times ((\tilde\V_2^+)^3-(\tilde\V_1^+)^3)\tilde\Y_{1,\eta}(t,\theta) \mathbbm{1}_{\tilde\V_2^+<\tilde\V_1^+}(t,\theta)d\theta\Big)\\ \nn
& \quad\quad + \frac{1}{\A^6} \mathbbm{1}_{\sqtC{1}\leq \sqtC{2}}\tilde \Y_{2,\eta}(t,\eta)\\ \nn
&\qquad\quad\times \left(\int_0^\eta (e^{-\frac{1}{\sqtC{2}}(\tilde \Y_2(t,\eta)-\tilde\Y_2(t,\theta))}-e^{-\frac{1}{\sqtC{1}}(\tilde\Y_2(t,\eta)-\tilde\Y_2(t,\theta))} ) \min_j(\tilde\V_j^+)^3\tilde \Y_{2,\eta}(t,\theta) d\theta\right)\\ \nn
&\quad \quad +\frac{1}{\A^6} \mathbbm{1}_{\sqtC{2}<\sqtC{1}}\tilde \Y_{1,\eta}(t,\eta)\\ \nn
&\qquad\quad\times\left(\int_0^\eta (e^{-\frac{1}{\sqtC{2}}(\tilde\Y_1(t,\eta)-\tilde \Y_1(t,\theta))}-e^{-\frac{1}{\sqtC{1}}(\tilde \Y_1(t,\eta)-\tilde \Y_1(t,\theta))})\min_j(\tilde \V_j^+)^3\tilde \Y_{1,\eta}(t,\theta) d\theta\right)\\ \nn
& \quad\quad +\frac{1}{\A^6} \tilde\Y_{2,\eta}(t,\eta) \Big(\int_0^\eta (e^{-\frac{1}{\ma}(\tilde\Y_2(t,\eta)-\tilde\Y_2(t,\theta))}-e^{ -\frac{1}{\ma}(\tilde\Y_1(t,\eta)-\tilde\Y_1(t,\theta))}) \\ \nn
&\quad\qquad\qquad\qquad\qquad\qquad\qquad\times\min_j(\tilde\V_j^+)^3\tilde\Y_{2,\eta}(t,\theta)\mathbbm{1}_{B(\eta)}(t,\theta) d\theta\Big)\\  \nn
&\quad \quad -\frac{1}{\A^6}\tilde\Y_{1,\eta}(t,\eta) \Big(\int_0^\eta (e^{-\frac{1}{\ma}(\tilde\Y_1(t,\eta)-\tilde\Y_1(t,\theta))}-e^{-\frac{1}{\ma}(\tilde\Y_2(t,\eta)-\tilde\Y_2(t,\theta))})\\ \nn
&\quad\qquad\qquad\qquad\qquad\qquad\qquad\times\min_j(\tilde\V_j^+)^3 \tilde\Y_{1,\eta}(t,\theta) \mathbbm{1}_{B^c(\eta)} (t,\theta) d\theta\Big)\\ \nn
&\quad \quad + \frac{1}{\A^6}\tilde\Y_{2,\eta}(t,\eta) \Big(\int_0^\eta \min_j(e^{-\frac{1}{\ma}(\tilde\Y_j(t,\eta)-\tilde\Y_j(t,\theta))})\\ \nn
&\quad\qquad\qquad\qquad\qquad\qquad\qquad\times\min_j(\tilde\V_j^+)^3\tilde\Y_{2,\eta}(t,\theta) d\theta\Big)\\ \nn
&\quad \quad - \frac{1}{\A^6} \tilde\Y_{1,\eta}(t,\eta) \Big(\int_0^\eta \min_j(e^{-\frac{1}{\ma}(\tilde\Y_j(t,\eta)-\tilde\Y_j(t,\theta))})\\ \nn
&\quad\qquad\qquad\qquad\qquad\qquad\qquad\times\min_j(\tilde\V_j^+)^3\tilde\Y_{1,\eta}(t,\theta) d\theta\Big)\\ \nn
&\quad \quad + \frac{\sqtC{1}^6-\sqtC{2}^6}{\sqtC{1}^6\sqtC{2}^6}\mathbbm{1}_{\sqtC{1}\leq \sqtC{2}} \tilde \Y_{1,\eta}(t,\eta) \int_0^\eta e^{-\frac{1}{\sqtC{1}}(\tilde\Y_1(t,\eta)-\tilde\Y_1(t,\theta))} (\tilde\V_1^+)^3\tilde \Y_{1,\eta}(t,\theta) d\theta\\ \nn
&\quad \quad + \frac{\sqtC{1}^6-\sqtC{2}^6}{\sqtC{1}^6\sqtC{2}^6} \mathbbm{1}_{\sqtC{2}< \sqtC{1}} \tilde \Y_{2,\eta}(t,\eta)\int_0^\eta e^{-\frac{1}{\sqtC{2}}(\tilde\Y_2(t,\eta)-\tilde\Y_2(t,\theta))}(\tilde\V_2^+)^3\tilde\Y_{2,\eta}(t,\theta)d\theta\\ 
&\quad = (J_1+J_2+J_3+J_4+J_5+J_6+J_7+J_8+J_9+J_{10})(t,\eta).  \label{eq:alleJ}
\end{align}
Here 
\begin{equation}\label{Def:Bn}
B(\eta)=\{(t,\theta)\mid e^{\tilde\Y_1(t,\theta)-\tilde \Y_1(t,\eta)}\leq e^{\tilde\Y_2(t,\theta)-\tilde\Y_2(t,\eta)} \}
\end{equation}
which means especially that $B(\eta)$ depends heavily on $\eta$!  Observe that the set $B(\eta)$ is scale invariant in the sense that $(\tilde\Y_1,\tilde\Y_2)$ and $(c\tilde\Y_1,c\tilde\Y_2)$ define the same set $B(\eta)$ for any constant $c$.

As far as the first two terms $J_1$ and $J_2$ are concerned, they have the same structure, and hence we only consider the term  $J_1$, i.e., 
\begin{align*}
\int_0^{1} J_1(&\tilde\Y_1-\tilde\Y_2)(t,\eta)d\eta \\
&=\frac{1}{\A^6}\int_0^{1} (\tilde\Y_1-\tilde\Y_2) \tilde\Y_{2,\eta}(t,\eta) \\
&\quad\times\left(\int_0^{\eta} e^{-\frac{1}{\sqtC{2}}(\tilde\Y_2(t,\eta)-\tilde\Y_2(t,\theta))} ((\tilde\V_2^+)^3- (\tilde\V_1^+)^3) \tilde\Y_{2,\eta}(t,\theta) \mathbbm{1}_{\tilde\V_1^+\leq \tilde\V_2^+}(t,\theta)d\theta\right) d\eta.
\end{align*}
The main ingredients are the following observations
\begin{equation*}
\vert (\tilde\V_1^+)^3(t,\eta)-(\tilde\V_2^+)^3(t,\eta)\vert\mathbbm{1}_{\tilde \V_1^+\leq \tilde\V_2^+} \leq 3 \tilde\U_2^2(t,\eta) \vert \tilde\U_1(t,\eta)-\tilde \U_2(t,\eta)\vert,
\end{equation*} 
and \eqref{eq:all_PestimatesC}.

 Thus we have 
 \begin{align*}
\vert &\int_0^{1} J_1(\tilde\Y_1-\tilde\Y_2)(t,\eta)d\eta\vert\\
&=\frac{1}{\A^6}\Big\vert \int_0^{1} (\tilde\Y_1-\tilde \Y_2) \tilde\Y_{2,\eta}(t,\eta)\\
&\qquad \times \left(\int_0^{\eta} e^{-\frac{1}{\sqtC{2}}(\tilde\Y_2(t,\eta)-\tilde\Y_2(t,\theta))} ((\tilde\U_2^+)^3-(\tilde\U_1^+)^3) \tilde\Y_{2,\eta}(t,\theta) \mathbbm{1}_{\tilde\V_1^+\leq \tilde\V_2^+}(t,\theta)d\theta\right)d\eta\Big\vert \\
& \leq \norm{\tilde\Y_1-\tilde\Y_2}^2 +\frac{1}{\A^{12}} \int_0^{1} \tilde\Y_{2,\eta}^2(t,\eta)\\
&\qquad \times \left(\int_0^\eta e^{-\frac{1}{\sqtC{2}}(\tilde\Y_2(t,\eta)-\tilde\Y_2(t,\theta))} ((\tilde\V_2^+)^3-(\tilde \V_1^+)^3) \tilde\Y_{2,\eta}(t,\theta) \mathbbm{1}_{\tilde\V_1^+\leq \tilde\V_2^+}(t,\theta)d\theta\right)^2d\eta\\
& \leq \norm{\tilde\Y_1-\tilde\Y_2}^2\\
& \quad +\frac{9}{\A^{12}} \int_0^{1} \tilde\Y_{2,\eta}^2(t,\eta) \left(\int_0^\eta e^{-\frac{1}{\sqtC{2}}(\tilde\Y_2(t,\eta)-\tilde\Y_2(t,\theta))} \tilde\U_2^2 \tilde\Y_{2,\eta}(t,\theta) d\theta\right)\\
& \qquad \qquad \times \left( \int_0^\eta e^{-\frac{1}{\sqtC{2}}(\tilde\Y_2(t,\eta)-\tilde\Y_2(t,\theta))} \tilde\U_2^2\tilde\Y_{2,\eta} (\tilde\U_1-\tilde\U_2)^2(t,\theta) d\theta\right) d\eta\\ 
& \leq \norm{\tilde\Y_1-\tilde\Y_2}^2\\
& \quad +\frac{36}{\A^{11}}\int_0^{1} \tilde\P_2\tilde\Y_{2,\eta}^2(t,\eta) \left( \int_0^\eta e^{-\frac{1}{\sqtC{2}}(\tilde\Y_2(t,\eta)-\tilde\Y_2(t,\theta))} \tilde\U_2^2\tilde\Y_{2,\eta} (\tilde\U_1-\tilde\U_2)^2(t,\theta) d\theta\right) d\eta\\
& \leq \norm{\tilde\Y_1-\tilde\Y_2}^2 \\
& \quad + \frac{18}{\A^6}\int_0^{1} \tilde\Y_{2,\eta}(t,\eta) e^{-\frac{1}{\sqtC{2}}\tilde\Y_2(t,\eta)} \int_0^\eta e^{\frac{1}{\sqtC{2}}\tilde\Y_2(t,\theta)}\tilde\U_2^2\tilde\Y_{2,\eta}(\tilde\U_1-\tilde\U_2)^2(t,\theta) d\theta d\eta\\
& = \norm{\tilde\Y_1-\tilde\Y_2}^2\\
& \quad -\frac{18 \sqtC{2}}{\A^6}\int_0^\eta e^{-\frac{1}{\sqtC{2}}(\tilde\Y_2(t,\eta)-\tilde\Y_2(t,\theta))}\tilde\U_2^2\tilde\Y_{2,\eta}(\tilde\U_1-\tilde\U_2)^2(t,\theta) d\theta\big|_{\eta=0}^{1}\\
& \quad + \frac{18\sqtC{2}}{\A^6}\int_0^{1} \tilde\V_2^2\tilde\Y_{2,\eta} (\tilde\U_1-\tilde \U_2)^2(t,\eta) d\eta\\
& \leq \norm{\tilde\Y_1-\tilde\Y_2}^2+18\norm{\tilde\U_1-\tilde\U_2}^2,
\end{align*}
where we in the last step used \eqref{eq:all_estimatesB} and \eqref{eq:all_PestimatesC}, which imply that 
\begin{equation}\label{eq:lhs}
\int_0^\eta e^{-\frac{1}{\sqtC{2}}(\tilde\Y_2(t,\eta)-\tilde\Y_2(t,\theta))}\tilde\U_2^2\tilde\Y_{2,\eta}(\tilde\U_1-\tilde\U_2)^2(t,\theta) d\theta\leq 8\A^5 \tilde\P_2(t,\eta),
\end{equation} 
and hence the left-hand side of \eqref{eq:lhs} tends to $0$ as $\eta$ to $0$ or $1$ according to \eqref{decay:impl}. Thus the second term above vanishes.

As far as the third and the fourth term $J_3$ and $J_4$ are concerned, they again have the same structure, and hence we only consider the integral corresponding to $J_3$, i.e.,
\begin{align*}
\int_0^1& J_3(\tilde \Y_1-\tilde \Y_2)(t,\eta) d\eta \\
&= \frac{1}{\A^6} \mathbbm{1}_{\sqtC{1}\leq \sqtC{2}}\int_0^1 (\tilde \Y_1-\tilde \Y_2)\tilde \Y_{2,\eta}(t,\eta) \\
& \quad \times \Big(\int_0^\eta (e^{-\frac{1}{\sqtC{2}}(\tilde \Y_2(t,\eta)-\tilde \Y_2(t,\theta))}-e^{-\frac{1}{\sqtC{1}}(\tilde \Y_2(t,\eta)-\tilde \Y_2(t,\theta))})\min_j(\tilde \V_j^+)^3 \tilde \Y_{2,\eta}(t,\theta)d\theta\Big) d\eta.
\end{align*}

Recall Lemma \ref{lemma:enkel} (ii). Direct computations yield
\begin{align*}
& \vert \int_0^1 J_3(\tilde \Y_1-\tilde \Y_2)(t,\eta) d\eta\vert \\
& \quad =\frac{1}{\A^6} \mathbbm{1}_{\sqtC{1}\leq \sqtC{2}}\vert \int_0^1 (\tilde \Y_1-\tilde \Y_2) 
\tilde \Y_{2,\eta}(t,\eta)\\
& \qquad \times \Big(\int_0^\eta (e^{-\frac{1}{\sqtC{2}}(\tilde \Y_2(t,\eta)-\tilde \Y_2(t,\theta))}-e^{-\frac{1}{\sqtC{1}}(\tilde \Y_2(t,\eta)-\tilde \Y_2(t,\theta))}\min_j(\tilde \V_j^+)^3\tilde \Y_{2,\eta}(t,\theta) d\theta\Big)d\eta\vert\\
& \quad \leq \frac{4\vert \sqtC{1}-\sqtC{2}\vert }{\ma \A^6 e} \mathbbm{1}_{\sqtC{1}\leq \sqtC{2}} \int_0^1 \vert \tilde\Y_1-\tilde \Y_2\vert \tilde \Y_{2,\eta}(t,\eta)\\
& \qquad \times\Big(\int_0^\eta e^{-\frac{3}{4\sqtC{2}} (\tilde \Y_2(t,\eta)-\tilde \Y_2(t,\theta))}\min_j(\tilde \V_j^+)^3 \tilde \Y_{2,\eta}(t,\theta) d\theta\Big) d\eta \\
&\quad  \leq \norm{\tilde \Y_1-\tilde \Y_2}^2+\frac{16\vert \sqtC{1}-\sqtC{2}\vert ^2}{\ma^2\A^{12}e^2}\mathbbm{1}_{\sqtC{1}\leq\sqtC{2}}\\ 
& \qquad\qquad\qquad \times\int_0^1 \tilde \Y_{2,\eta}^2(t,\eta)\Big(\int_0^\eta e^{-\frac{3}{4\sqtC{2}}(\tilde \Y_2(t,\eta)-\tilde \Y_2(t,\theta))}\vert \tilde \U_1\vert \tilde \U_2^2\tilde \Y_{2,\eta}(t,\theta)d\theta\Big)^2d\eta \\
& \quad \leq \norm{\tilde\Y_1-\tilde\Y_2}^2\\
& \qquad +\frac{16\ma^4\vert \sqtC{1}-\sqtC{2}\vert ^2}{\ma^2\A^{12}e^2}\int_0^1 \tilde \Y_{2,\eta}^2 (t,\eta)\Big(\int_0^\eta e^{-\frac{1}{\sqtC{2}} (\tilde \Y_2(t,\eta)-\tilde \Y_2(t,\theta))}\tilde \U_2^2\tilde \Y_{2,\eta}(t,\theta) d\theta\Big)\\
& \qquad \qquad \qquad \qquad \qquad \times  \Big(\int_0^\eta e^{-\frac{1}{2\sqtC{2}}(\tilde \Y_2(t,\eta)-\tilde \Y_2(t,\theta))}\tilde \U_2^2\tilde \Y_{2,\eta}(t,\theta) d\theta\Big) d\eta\\
& \quad \leq \norm{\tilde\Y_1-\tilde\Y_2}^2\\
& \qquad  +\frac{64\vert \sqtC{1}-\sqtC{2}\vert ^2}{\A^9e^2}\int_0^1 \tilde \P_2\tilde \Y_{2,\eta}^2(t,\eta) \Big(\int_0^\eta e^{-\frac{1}{2\sqtC{2}}(\tilde \Y_2(t,\eta)-\tilde\Y_2(t,\theta))} \tilde \U_2^2\tilde \Y_{2,\eta} (t,\theta) d\theta\Big) d\eta \\
& \quad \leq \norm{\tilde\Y_1-\tilde\Y_2}^2\\
& \qquad + \frac{32\vert \sqtC{1}-\sqtC{2}\vert ^2}{\A^4 e^2} \int_0^1 \tilde \Y_{2,\eta}(t,\eta) e^{-\frac{1}{2\sqtC{2}}\tilde \Y_2(t,\eta)}\int_0^\eta e^{\frac{1}{2\sqtC{2}}\tilde \Y_2(t,\theta)} \tilde \U_2^2\tilde \Y_{2,\eta} (t,\theta) d\theta d\eta \\
& \quad \leq \norm{\tilde \Y_1-\tilde \Y_2}^2 + \frac{64\sqtC{2}\vert \sqtC{1}-\sqtC{2}\vert ^2}{\A^4 e^2}\\
& \qquad\qquad \times \Big( -\int_0^ \eta e^{-\frac{1}{2\sqtC{2}}(\tilde \Y_2(t,\eta)-\tilde \Y_2(t,\theta))} \tilde \U_2^2\tilde \Y_{2,\eta} (t,\theta)d\theta |_{\eta=0}^1+ \int_0^1 \tilde \U_2^2 \tilde \Y_{2,\eta} (t,\eta) d\eta\Big) \\
& \leq \norm{\tilde \Y_1-\tilde \Y_2}^2+ \frac{128\A^2}{e^2} \vert \sqtC{1}-\sqtC{2}\vert ^2\\
& \leq \bigO(1)(\norm{\tilde \Y_1-\tilde \Y_2}^2+\vert \sqtC{1}-\sqtC{2}\vert ^2).
\end{align*}

As far as the  terms $J_5$ and $J_6$ are concerned, they again have the same structure, and hence we only consider the integral corresponding to $J_5$, i.e.,
\begin{align*}
\int_0^{1} J_5(\tilde\Y_1-\tilde\Y_2)(t,\eta)d\eta&
=\frac{1}{\A^6}\int_0^{1} (\tilde\Y_1-\tilde\Y_2) \tilde\Y_{2,\eta}(t,\eta)\\
&\qquad\times \Big(\int_0^\eta (e^{-\frac{1}{\ma}(\tilde\Y_2(t,\eta)-\tilde\Y_2(t,\theta))}-e^{-\frac{1}{\ma}(\tilde\Y_1(t,\eta)-\tilde\Y_1(t,\theta))}) \\
&\qquad\qquad\qquad\times\min_j(\tilde\V_j^+)^3\tilde\Y_{2,\eta}(t,\theta)\mathbbm{1}_{B(\eta)}(t,\theta) d\theta\Big)d\eta.
\end{align*}
The main ingredient is the following estimate
\begin{multline}\label{Diff:Exp}
\vert e^{-\frac{1}{\ma}(\tilde\Y_2(t,\eta)-\tilde\Y_2(t,\theta))}-e^{-\frac{1}{\ma}(\tilde\Y_1(t,\eta)-\tilde\Y_1(t,\theta))}\vert\\
 \leq \frac{1}{\ma}e^{-\frac{1}{\ma}(\tilde\Y_2(t,\eta)-\tilde\Y_2(t,\theta))} (\vert \tilde\Y_2(t,\theta)-\tilde\Y_1(t,\theta)\vert +\vert \tilde\Y_2(t,\eta)-\tilde\Y_1(t,\eta)\vert ) \\ \text{ for all } (t,\theta)\in B(\eta),
\end{multline}
which follows from Lemma \ref{lemma:enkel} (i).

Direct computations yield
\begin{align*}
&\vert\int_0^{1} J_5(\tilde\Y_1-\tilde\Y_2)(t,\eta)d\eta  \vert\\
&\quad=\frac{1}{A^6}\vert \int_0^{1}  (\tilde\Y_1-\tilde\Y_2) \tilde\Y_{2,\eta}(t,\eta) \Big(\int_0^\eta (e^{-\frac{1}{\ma}(\tilde\Y_2(t,\eta)-\tilde\Y_2(t,\theta))}-e^{-\frac{1}{\ma}(\tilde\Y_1(t,\eta)-\tilde\Y_1(t,\theta))}) \\
&\qquad\qquad\qquad\qquad\qquad\qquad\qquad\qquad\times\min_j(\tilde\V_j^+)^3\tilde\Y_{2,\eta}(t,\theta)\mathbbm{1}_{B(\eta)}(t,\theta) d\theta\Big)d\eta\vert \\
&\quad \leq \frac{1}{\ma\A^6}\int_0^{1} \vert \tilde\Y_1-\tilde \Y_2 \vert \tilde\Y_{2,\eta}(t,\eta) \\
& \qquad \qquad \times \Big(\int_0^\eta (\vert \tilde\Y_2(t,\eta)-\tilde \Y_1(t,\eta)\vert +\vert \tilde\Y_2(t,\theta)-\tilde\Y_1(t,\theta) \vert) \\
&\qquad\qquad\qquad\qquad\qquad\qquad\qquad\times e^{-\frac{1}{a}(\tilde\Y_2(t,\eta)-\tilde\Y_2(t,\theta))} \min_j(\tilde\V_j^+)^3 \tilde\Y_{2,\eta}(t,\theta) d\theta\Big) d\eta\\
&\quad \leq \frac{1}{\ma\A^6}\int_0^{1} (\tilde\Y_1-\tilde\Y_2)^2 \tilde\Y_{2,\eta}(t,\eta)\\
&\qquad\qquad\qquad\qquad\times\left(\int_0^{\eta} e^{-\frac{1}{\ma}(\tilde\Y_2(t,\eta)-\tilde\Y_2(t,\theta))} \min_j(\tilde\V_j^+)^3\tilde\Y_{2,\eta}(t,\theta) d\theta\right) d\eta\\
& \qquad +\norm{\tilde\Y_1-\tilde\Y_2}^2 + \frac{1}{\ma^2\A^{12}}\int_0^{1} \tilde\Y_{2,\eta}^2(t,\eta)\\
&\qquad\quad\times\left( \int_0^\eta \vert \tilde\Y_1(t,\theta)-\tilde\Y_2(t,\theta)\vert e^{-\frac{1}{\ma}(\tilde\Y_2(t,\eta)-\tilde\Y_2(t,\theta))} \min_j(\tilde\V_j^+)^3 \tilde\Y_{2,\eta}(t,\theta) d\theta \right)^2 d\eta\\
& \quad\leq \frac{4}{\A^4}\int_0^{1} (\tilde\Y_1-\tilde\Y_2)^2\tilde\P_2\tilde\Y_{2,\eta}(t,\eta) d\eta + \norm{\tilde\Y_1-\tilde\Y_2}^2\\
& \qquad + \frac{1}{\ma^2A^{12}}\int_0^{1} \tilde\Y_{2,\eta}^2(t,\eta) \left(\int_0^\eta e^{-\frac{1}{\ma}(\tilde\Y_2(t,\eta)-\tilde\Y_2(t,\theta))} \tilde\U_2^2\tilde\Y_{2,\eta}(t,\theta) d\theta\right)\\ 
& \qquad \qquad\qquad  \times \left(\int_0^{\eta} (\tilde\Y_1-\tilde\Y_2)^2(t,\theta) e^{-\frac{1}{\ma}(\tilde\Y_2(t,\eta)-\tilde\Y_2(t,\theta))}\min_j(\tilde\U_j^4) \tilde\Y_{2,\eta}(t,\theta) d\theta \right)d\eta\\
&\quad \leq (2\A+1)\norm{\tilde\Y_1-\tilde\Y_2}^2 +\frac{4}{\ma^2\A^{11}} \int_0^{1} \tilde\P_2\tilde\Y_{2,\eta}^2(t,\eta) \\
&\qquad\qquad\quad\times\left(\int_0^{\eta} (\tilde\Y_1-\tilde \Y_2)^2(t,\theta) e^{-\frac{1}{\ma}(\tilde\Y_2(t,\eta)-\tilde\Y_2(t,\theta))}\min_j(\tilde\U_j^4) \tilde\Y_{2,\eta}(t,\theta) d\theta \right)d\eta\\
&\quad \leq \bigO(1)\norm{\tilde\Y_1-\tilde \Y_2}^2 +\frac{2}{\ma^2A^6} \int_0^{1} \tilde\Y_{2,\eta}(t,\eta)\\
&\qquad\qquad\quad\times \left(\int_0^{\eta} (\tilde\Y_1-\tilde \Y_2)^2(t,\theta) e^{-\frac{1}{\ma}(\tilde\Y_2(t,\eta)-\tilde\Y_2(t,\theta))}\min_j(\tilde\U_j^4) \tilde\Y_{2,\eta}(t,\theta) d\theta \right)d\eta\\
& \quad= \bigO(1) \norm{\tilde\Y_1-\tilde\Y_2}^2 +\frac{2}{\ma^2\A^6}\int_0^{1} \tilde\Y_{2,\eta}(t,\eta) e^{-\frac{1}{\ma}\tilde\Y_2(t,\eta)} \\
&\qquad\qquad\quad\times\int_0^\eta (\tilde\Y_1-\tilde \Y_2)^2(t,\theta) e^{\frac{1}{\ma}\tilde\Y_2(t,\theta)} \min_j(\tilde\U_j^4)\tilde\Y_{2,\eta}(t, \theta) d\theta d\eta\\
& \quad= \bigO(1) \norm{\tilde\Y_1-\tilde\Y_2}^2 \\
& \qquad - \frac{2}{\ma\A^6}\int_0^\eta e^{-\frac{1}{\ma}(\tilde\Y_2(t,\eta)-\tilde\Y_2(t,\theta))} (\tilde\Y_1-\tilde \Y_2)^2  \min_j(\tilde\U_j^4)\tilde\Y_{2,\eta}(t,\theta) d\theta \Big| _{\eta=0}^{1}\\
& \qquad +\frac{2}{\ma\A^6}\int_0^{1} (\tilde\Y_1-\tilde \Y_2)^2\min_j(\tilde\U_j^4)\tilde\Y_{2,\eta}(t,\eta) d\eta\\\
&\quad \leq \bigO(1) \norm{\hat\Y_1-\hat\Y_2}^2
\end{align*}
where $\bigO(1)$ denotes some constant which only depends on $A$ and which remains bounded as $A\to 0$. We used, cf.~\eqref{eq:all_PestimatesC}, that 
\begin{align*}
\int_0^\eta e^{-\frac{1}{\ma}(\tilde\Y_2(t,\eta)-\tilde\Y_2(t,\theta))} \tilde\U_2^2\tilde\Y_{2,\eta}(t,\theta) d\theta
&\le \int_0^\eta e^{-\frac{1}{A_2}(\tilde\Y_2(t,\eta)-\tilde\Y_2(t,\theta))} \tilde\U_2^2\tilde\Y_{2,\eta}(t,\theta) d\theta\\
&\le 4A_2\tilde\P_2(t,\eta). 
\end{align*}
In the last step we used \eqref{eq:all_estimatesB}, \eqref{eq:all_estimatesG}, and finally that 
\begin{multline*}
\vert \int_0^\eta e^{-\frac{1}{\ma}(\tilde\Y_2(t,\eta)-\tilde\Y_2(t,\theta))}  (\tilde\Y_1- \tilde\Y_2)^2 \min_j(\tilde\U_j^4) \tilde\Y_{2,\eta}(t,\theta) d\theta\vert 
\\ \leq (\norm{ \tilde\Y_1\tilde\U_1}_{L^\infty}^2+ \norm{\tilde\Y_2\tilde\U_2}_{L^\infty}^2) 8\A \tilde\P_2(t,\eta).  
\end{multline*}
Since $\tilde\P_i(t,\eta)$ tends to $0$ as $\eta$ tends to $0$ and $1$, the term on left-hand side tends to zero  as $\eta$ tends to $0$ and $1$, respectively, cf.~\eqref{decay:impl}. 

Consider next the  terms $J_7$ and $J_8$, i.e., 
\begin{align*}
(J_7+J_8)(t,\eta)
&= \frac{1}{\A^6}\tilde\Y_{2,\eta}(t,\eta) \\
&\qquad\times \left(\int_0^\eta\min_j(e^{-\frac{1}{\ma}(\tilde \Y_j(t,\eta)-\tilde\Y_j(t,\theta))})\min_j(\tilde\V_j^+)^3\tilde\Y_{2,\eta}(t,\theta) d\theta\right)\\
& \quad -\frac1{\A^6} \tilde\Y_{1,\eta}(t,\eta)\\
&\qquad\times \left(\int_0^\eta \min_j(e^{-\frac{1}{\ma}(\tilde \Y_j(t,\eta)-\tilde\Y_j(t,\theta))})\min_j(\tilde\V_j^+)^3\tilde \Y_{1,\eta}(t,\theta) d\theta\right).
\end{align*}
Here we have to be a bit more careful. Introducing 
\begin{align}
E&=\Big\{ (t,\eta)\mid   \int_0^\eta \Omega(t,\eta,\theta)\tilde\Y_{2,\eta}(t,\theta) d\theta 
 \leq \int_0^\eta \Omega(t,\eta,\theta)\tilde\Y_{1,\eta}(t,\theta) d\theta\Big\}, \label{Def:EE} \\
 \intertext{with}
\Omega(t,\eta,\theta)&= \min_j(e^{-\frac{1}{\ma}(\tilde \Y_j(t,\eta)-\tilde\Y_j(t,\theta))})\min_j(\tilde\V_j^+)^3(t,\theta), \nn
\end{align}
 we can write
\begin{align*}
(J_7&+J_8)(t,\eta)\\
&= \frac{1}{\A^6}\tilde\Y_{2,\eta}(t,\eta) \\
&\qquad\times \left(\int_0^\eta \min_j(e^{-\frac{1}{\ma}(\tilde \Y_j(t,\eta)-\tilde\Y_j(t,\theta))})\min_j(\tilde\V_j^+)^3\tilde\Y_{2,\eta}(t,\theta) d\theta\right)\\
& \quad - \frac{1}{A^6}\tilde\Y_{1,\eta}(t,\eta)\\
&\qquad\times \left(\int_0^\eta \min_j(e^{-\frac{1}{\ma}(\tilde \Y_j(t,\eta)-\tilde\Y_j(t,\theta))})\min_j(\tilde\V_j^+)^3\tilde \Y_{1,\eta}(t,\theta) d\theta\right)\\
& =  \frac{1}{A^6}(\tilde\Y_{2,\eta}-\tilde\Y_{1,\eta})(t,\eta)\\
&\qquad\times \min_k\Big[ \int_0^\eta \min_j(e^{-\frac{1}{\ma}(\tilde \Y_j(t,\eta)-\tilde\Y_j(t,\theta))})\min_j(\tilde\V_j^+)^3\tilde\Y_{k,\eta}(t,\theta) d\theta\Big]\\
& \quad + \frac{1}{\A^6}\hat\Y_{1,\eta}(t,\eta) \int_0^\eta \min_j(e^{-\frac{1}{\ma}(\tilde \Y_j(t,\eta)-\tilde\Y_j(t,\theta))})\\
&\qquad\qquad\qquad\qquad\qquad\qquad\times \min_j(\tilde\V_j^+)^3(\tilde\Y_{2,\eta}-\tilde \Y_{1,\eta})(t, \theta) d\theta\mathbbm{1}_E(t,\eta)\\
& \quad +\frac{1}{\A^6}\hat\Y_{2,\eta}(t,\eta) \int_0^\eta \min_j(e^{-\frac{1}{\ma}(\tilde \Y_j(t,\eta)-\tilde\Y_j(t,\theta))})\\
&\qquad\qquad\qquad\qquad\qquad\qquad\times \min_j(\tilde\V_j^+)^3(\tilde\Y_{2,\eta}-\tilde \Y_{1,\eta})(t,\theta) d\theta \mathbbm{1}_{E^c}(t,\eta)\\
& =( L_1+L_2+L_3)(t,\eta).
\end{align*}

As far as the first term $L_1$ is concerned, the corresponding integral can be estimated as follows  (we use \eqref{decay:impl})
\begin{align*}
&\vert \int_0^{1} L_1(\tilde\Y_1-\tilde \Y_2)(t,\eta)d\eta\vert \notag \\
&\quad=\frac{1}{\A^6}\Big\vert \int_0^{1}   (\tilde\Y_1-\tilde \Y_2)(\tilde\Y_{2,\eta}-\tilde\Y_{1,\eta})(t,\eta) \notag \\
&\qquad\qquad\times\min_k\Big[ \int_0^\eta \min_j(e^{-\frac{1}{\ma}(\tilde \Y_j(t,\eta)-\tilde\Y_j(t,\theta))})\min_j(\tilde\V_j^+)^3\tilde\Y_{k,\eta}(t,\theta) d\theta\Big]d\eta\Big\vert \notag\\
&\quad = \Big\vert -\frac1{2\A^6} (\tilde\Y_1-\tilde\Y_2)^2(t,\eta)\notag\\
&\qquad\qquad\times \min_k\Big[ \int_0^\eta \min_j(e^{-\frac{1}{\ma}(\tilde \Y_j(t,\eta)-\tilde\Y_j(t,\theta))})\min_j(\tilde\V_j^+)^3\tilde\Y_{k,\eta}(t,\theta) d\theta\Big] 
\Big|_{\eta=0}^{1}\notag\\ 
&\qquad +\frac1{2\A^6} \int_0^{1} (\tilde\Y_1-\tilde\Y_2)^2(t,\eta)\notag\\
&\qquad\quad\times\frac{d}{d\eta} \min_k\Big[ \int_0^\eta \min_j(e^{-\frac{1}{\ma}(\tilde \Y_j(t,\eta)-\tilde\Y_j(t,\theta))})\min_j(\tilde\V_j^+)^3\tilde\Y_{k,\eta}(t,\theta) d\theta\Big]d\eta  \Big\vert\notag \\
&\quad \leq \bigO(1) \norm{\tilde\Y_1-\tilde\Y_2}^2,
\end{align*}
where $\bigO(1)$ denotes some constant only depending on $\A$, which remains bounded as $\A\to 0$, since the derivative 
\begin{equation*} 
\frac{d}{d\eta} \min_k \Big[ \int_0^\eta \min_j(e^{-\frac{1}{\ma}(\tilde \Y_j(t,\eta)-\tilde\Y_j(t,\theta))})\min_j(\tilde\V_j^+)^3\tilde\Y_{k,\eta}(t,\theta) d\theta\Big] 
\end{equation*}
exists and is uniformly bounded, see Lemma \ref{lemma:7}.

\medskip
As far as the last term $L_3$ (a similar argument works for $L_2$) is concerned, the corresponding integral can be estimated as follows,
\begin{align*}
&\int_0^{1} L_3(\tilde\Y_1-\tilde\Y_2)(t,\eta)d\eta\\
&= \frac1{\A^6}\int_0^{1} (\tilde\Y_1-\tilde \Y_2) \tilde\Y_{2,\eta}(t,\eta)\int_0^\eta\min_j(e^{-\frac{1}{\ma}(\tilde \Y_j(t,\eta)-\tilde\Y_j(t,\theta))})\\
&\qquad\qquad\qquad\qquad\qquad\qquad\qquad\times\min_j(\tilde\V_j^+)^3(\tilde\Y_{2,\eta}-\tilde\Y_{1,\eta})(t,\theta) d\theta \mathbbm{1}_{E^c}(t,\eta) d\eta\\
& = \frac1{\A^6}\int_0^{1} (\tilde\Y_1-\tilde\Y_2)\tilde\Y_{2,\eta}(t,\eta) \mathbbm{1}_{E^c}(t,\eta)\\ 
& \qquad  \times \Big[ (\tilde\Y_2-\tilde\Y_1)(t,\theta)\min_j(e^{-\frac{1}{\ma}(\tilde \Y_j(t,\eta)-\tilde\Y_j(t,\theta))}) \min_j(\tilde\V_j^+)^3(t,\theta)\Big|_{\theta=0}^{\eta}\\
& \qquad   - \int_0^\eta (\tilde\Y_2-\tilde\Y_1)(t,\theta) \\
&\qquad\qquad\times[(\frac{d}{d\theta} \min_j(e^{-\frac{1}{\ma}(\tilde \Y_j(t,\eta)-\tilde\Y_j(t,\theta))}))\min_j(\tilde\V_j^+)^3(t,\theta)\\
&\qquad  \qquad \qquad  + \min_j(e^{-\frac{1}{\ma}(\tilde \Y_j(t,\eta)-\tilde\Y_j(t,\theta))})(\frac{d}{d\theta} \min_j(\tilde\V_j^+)^3)(t,\theta) ] d\theta\Big]d\eta\\
&=-\frac{1}{\A^6}\int_0^{1} (\tilde\Y_1-\tilde\Y_2)^2 \min_j(\tilde\V_j^+)^3\tilde\Y_{2,\eta}(t,\eta) \mathbbm{1}_{E^c}(t,\eta) d\eta\\
& \quad +\frac1{\A^6}\int_0^{1} (\tilde\Y_1-\tilde\Y_2) \tilde\Y_{2,\eta}(t,\eta) \mathbbm{1}_{E^c}(t,\eta)\\
& \qquad \quad \times\int_0^\eta (\tilde\Y_1-\tilde\Y_2)(t,\theta)\\
&\qquad\qquad\quad\times\Big[(\frac{d}{d\theta} \min_j(e^{-\frac{1}{\ma}(\tilde \Y_j(t,\eta)-\tilde\Y_j(t,\theta))}))\min_j(\tilde\V_j^+)^3(t,\theta)\\
&\qquad \qquad\qquad  + \min_j(e^{-\frac{1}{\ma}(\tilde \Y_j(t,\eta)-\tilde\Y_j(t,\theta))})(\frac{d}{d\theta} \min_j(\tilde\V_j^+)^3)(t,\theta) ] d\theta\Big]d\eta\\
&=L_{31}+L_{32}.
\end{align*}

As far as the first term $L_{31}$ is concerned, we have, since 
\begin{equation*}
\min_j(\tilde\V_j^+)^3\tilde\Y_{i,\eta}(t,\eta) \leq \A^5\min_j(\tilde\V_j^+)\leq \frac{1}{\sqrt{2}}\A^7,
\end{equation*} 
that 
\begin{align*}
\vert L_{31}\vert&\leq\big\vert \frac{1}{\A^6}\int_0^{1} (\tilde\Y_1-\tilde\Y_2)^2 \min_j(\tilde\V_j^+)^3\tilde\Y_{2,\eta}(t,\eta) \mathbbm{1}_{E^c}(t,\eta) d\eta\big\vert \\
&\leq \frac{\A}{\sqrt{2}} \norm{\tilde\Y_1-\tilde\Y_2}^2=\bigO(1) \norm{\tilde\Y_1-\tilde\Y_2}^2,
\end{align*}
where $\bigO(1)$ denotes a constant, which only depends on $\A$ and which remains bounded as $\A\to 0$.

The second term $L_{32}$, on the other hand, is a bit more demanding. We start by considering the first part of $L_{32}$.  From Lemma \ref{lemma:1}
we have that
\begin{equation*}
\vert \frac{d}{d\theta}  \min_j(e^{-\frac{1}{\ma}(\tilde \Y_j(t,\eta)-\tilde\Y_j(t,\theta))})\vert
 \leq
 \frac{1}{\ma}\min_j(e^{-\frac{1}{\ma}(\tilde \Y_j(t,\eta)-\tilde\Y_j(t,\theta))}) \max_j(\tilde\Y_{j,\eta}(t,\theta)).
\end{equation*}
This implies that 
\begin{align*}
\frac{1}{\A^6}& \vert \int_0^{1} (\tilde\Y_1-\tilde\Y_2)\tilde\Y_{2,\eta}(t,\eta)\mathbbm{1}_{E^c}(t,\eta) \int_0^\eta (\tilde\Y_1-\tilde\Y_2)(t,\theta) \\
&\qquad\qquad\qquad\qquad\times(\frac{d}{d\theta}\min_j(e^{-\frac{1}{\ma}(\tilde \Y_j(t,\eta)-\tilde\Y_j(t,\theta))})\min_j (\tilde\V_j^+)^3(t,\theta) d\theta d\eta\vert \\
& \leq \norm{\tilde\Y_1-\tilde\Y_2} ^2 + \frac{1}{A^{12}}\int_0^1 \tilde\Y_{2,\eta}^2 (t,\eta)   \Big( \int_0^\eta (\tilde\Y_1-\tilde\Y_2) (t,\theta) \\
&\qquad\qquad\qquad\qquad\times(\frac{d}{d\theta} \min_j(e^{-\frac{1}{\ma}(\tilde \Y_j(t,\eta)-\tilde\Y_j(t,\theta))})\min_j (\tilde\V_j^+)^3(t,\theta) d\theta\Big) ^2 d\eta\\
& \leq \norm{\tilde\Y_1-\tilde \Y_2}^2 + \frac{1}{\ma^2\A^{12}} \int_0^1 \tilde \Y_{2,\eta}^2(t,\eta)\\
& \qquad\times
 \Big(\int_0^\eta (\tilde \Y_1-\tilde \Y_2)(t,\theta) \min_j(e^{-\frac{1}{\ma}(\tilde \Y_j(t,\eta)-\tilde \Y_j(t,\theta))})\\
 &\qquad\qquad\qquad\qquad\qquad\qquad\qquad\qquad\times\min_j(\tilde \V_j^+)^3 \max_j (\tilde \Y_{j,\eta})(t,\theta)d\theta\Big)^2 d\eta\\
& \leq \norm{\tilde \Y_1-\tilde \Y_2}^2+ \frac1{\ma^2\A^2}\int_0^1 \tilde \Y_{2,\eta}^2 (t,\eta)\\
& \qquad  \times \Big( \int_0^\eta (\tilde \Y_1-\tilde \Y_2)(t,\theta) \min_j(e^{-\frac{1}{\ma} (\tilde \Y_j(t,\eta)-\tilde \Y_j(t,\theta))})\min_j(\tilde \V_j^+)(t,\theta)d\theta\Big)^2d\eta\\
& \leq \norm{\tilde \Y_1-\tilde \Y_2}^2  +\frac{\ma^2}{2\A^2} \int_0^1\tilde \Y_{2,\eta}^2(t,\eta) \Big( \int_0^\eta (\tilde \Y_1-\tilde \Y_2)(t,\theta) \\
& \qquad \qquad\qquad\qquad\qquad\qquad\qquad\times\min_j(e^{-\frac{1}{\ma} (\tilde \Y_j(t,\eta)-\tilde \Y_j(t,\theta))})d\theta\Big)^2d\eta\\
& \leq \norm{\tilde \Y_1-\tilde \Y_2}^2\\
& \quad + \frac12\int_0^1 \tilde \Y_{2,\eta}^2(t,\eta) \Big( \int_0^\eta e^{-\frac{1}{2\sqtC{2}} (\tilde \Y_2(t,\eta)-\tilde \Y_2(t,\theta))}(\tilde \Y_1-\tilde \Y_2)^2(t,\theta) d\theta\Big)\\
& \qquad \qquad\qquad \qquad\qquad \times \Big(\int_0^\eta e^{-\frac{3}{2\sqtC{2}}(\tilde \Y_2(t,\eta)-\tilde \Y_2(t,\theta))} d\theta\Big) d\eta\\
& \leq \norm{\tilde \Y_1-\tilde \Y_2}^2\\
& \qquad + \frac{1}{2\A^5} \int_0^1 \tilde \Y_{2,\eta}^2(t,\eta)\Big( \int_0^\eta e^{-\frac{1}{2\sqtC{2}}(\tilde \Y_2(t,\eta)-\tilde \Y_2(t,\theta))} (\tilde\Y_1-\tilde\Y_2)^2(t,\theta)d\theta\Big)\\
& \qquad \qquad \qquad \times \Big(\int_0^\eta e^{-\frac{3}{2\sqtC{2}}(\tilde \Y_2(t,\eta)-\tilde \Y_2(t,\theta))} (2\tilde \P_2\tilde \Y_{2,\eta}+\tilde \Henergy_{2,\eta})(t,\theta) d\theta\Big) d\eta\\
& \leq \norm{\tilde \Y_1-\tilde \Y_2}^2\\
& \quad + \frac{1}{2\A^5} \int_0^1 8\A \tilde \P_2\tilde \Y_{2,\eta}^2(t,\eta)\Big(\int_0^\eta e^{-\frac{1}{2\sqtC{2}}(\tilde \Y_2(t,\eta)-\tilde\Y_2(t,\theta))}(\tilde \Y_1-\tilde\Y_2)^2(t,\theta)d\theta\Big) d\eta\\
& \leq \norm{\tilde \Y_1-\tilde \Y_2}^2\\
& \quad + 2\A\int_0^1 \tilde \Y_{2,\eta}(t,\eta) e^{-\frac{1}{2\sqtC{2}} \tilde \Y_2(t,\eta)}\int_0^\eta e^{\frac{1}{2\sqtC{2}}\tilde \Y_2(t,\theta) } (\tilde\Y_1-\tilde\Y_2)^2(t,\theta)d\theta\Big) d\eta\\
&  \leq \norm{\tilde\Y_1-\tilde\Y_2}^2 +4\A\sqtC{2}\Big(-\int_0^\eta e^{-\frac{1}{2\sqtC{2}}(\tilde \Y_2(t,\eta)-\tilde \Y_2(t,\theta))}(\tilde \Y_1-\tilde\Y_2)^2(t,\theta)d\theta|_{\eta=0}^1 \\
&\qquad\qquad\qquad\qquad\qquad\qquad\qquad\qquad+\int_0^1 (\tilde \Y_1-\tilde \Y_2)^2(t,\eta) d\eta\Big)\\
& \leq \bigO(1) \norm{\tilde \Y_1-\tilde \Y_2}^2.
\end{align*}
Here we used \eqref{eq:PUYH_scale} and \eqref{eq:343}.
Furthermore, we use \eqref{eq:Henergy32}.

Next, we turn to the second half of $L_{32}$. Recall first \eqref{est:L3b}.

From Lemma \ref{lemma:1} we have that 
\begin{equation*}
\big\vert \frac{d}{d\theta} \min_j(\tilde\V_j^+)^3(t,\theta)\big\vert \leq 2\A^4\min_j(\tilde\V_j^+)(t,\theta).
\end{equation*}
Thus we can conclude as before
\begin{align*}
& \frac{1}{\A^6}\Big\vert \int_0^{1} (\tilde\Y_1-\tilde\Y_2)\tilde\Y_{2,\eta}(t,\eta)\mathbbm{1}_{E^c}(t,\eta)
\int_0^\eta (\tilde\Y_1-\tilde\Y_2)(t,\theta) \min_j(e^{-\frac{1}{\ma}(\tilde \Y_j(t,\eta)-\tilde\Y_j(t,\theta))})\\
&\qquad\qquad\qquad\qquad\qquad\qquad\qquad\qquad\qquad\qquad
\times\left(\frac{d}{d\theta}\min_j(\tilde\V_j^+)^3\right)(t,\theta) d\theta\, d\eta\Big\vert\\
& \leq \norm{\tilde\Y_1-\tilde\Y_2}^2 + \frac{1}{\A^{12}}\int_0^{1} \tilde\Y_{2,\eta}^2 (t,\eta) 
\Big(\int_0^\eta (\tilde\Y_1-\tilde\Y_2) (t,\theta) \min_j(e^{-\frac{1}{\ma}(\tilde \Y_j(t,\eta)-\tilde\Y_j(t,\theta))})\\
&\qquad\qquad\qquad\qquad\qquad\qquad\qquad\qquad\qquad\qquad\times
\left(\frac{d}{d\theta}\min_j(\tilde\V_j^+)^3\right)(t,\theta) d\theta\Big)^2 d\eta\\
&  \leq \norm{\tilde\Y_1-\tilde\Y_2}^2
+  \frac{4}{\A^4}\int_0^{1} \tilde\Y_{2,\eta}^2 (t,\eta)  \Big( \int_0^\eta \vert \tilde\Y_1-\tilde \Y_2\vert  (t,\theta)  e^{-\frac{1}{\ma}(\tilde\Y_2(t,\eta)-\tilde\Y_2(t,\theta))}\min_j(\tilde\V_j^+)(t,\theta) d\theta\Big)^2 d\eta\\
& \leq \norm{\tilde\Y_1-\tilde\Y_2}^2 \\ 
& \quad + \frac{4}{\A^4}\int_0^{1} \tilde\Y_{2,\eta}^2 (t,\eta) \Big(\int_0^\eta (\tilde\Y_1-\tilde \Y_2)^2 (t,\theta) e^{-\frac{1}{\ma}(\tilde\Y_2(t,\eta)-\tilde\Y_2(t,\theta))}
d\theta\Big) \\
&\qquad\qquad\qquad\qquad\qquad\qquad\qquad\qquad\times
\Big(\int_0^{\eta} e^{-\frac{1}{\sqtC{2}}(\tilde\Y_2(t,\eta)-\tilde\Y_2(t,\theta)) } \tilde\U_2^2(t,\theta) d\theta\Big) d\eta\\
&\leq \norm{\tilde\Y_1-\tilde \Y_2}^2 \\
& \quad+\frac{24}{\A^4}\int_0^{1}  \tilde\P_2 \tilde\Y_{2,\eta}^2(t,\eta)\left(\int_0^\eta (\tilde\Y_1-\tilde\Y_2)^2(t,\theta) e^{-\frac{1}{\ma}(\tilde\Y_2(t,\eta)-\tilde\Y_2(t,\theta))} d\theta\right) d\eta \\
& \leq \norm{\tilde\Y_1-\tilde \Y_2}^2\\ 
& \quad + 12\A \int_0^{1} \tilde\Y_{2,\eta}(t,\eta)e^{-\frac{1}{\ma}\tilde\Y_2(t,\eta)}\left( \int_0^\eta (\tilde\Y_1-\tilde\Y_2)^2(t,\theta)e^{\frac{1}{\ma}\tilde\Y_2(t,\theta)}d\theta\right) d\eta\\
& =  \norm{\tilde\Y_1-\tilde\Y_2}^2 + 12\ma\A \Big(-\int_0^\eta e^{-\frac{1}{\ma}(\tilde\Y_2(t,\eta)-\tilde\Y_2(t,\theta))}(\tilde\Y_1-\tilde\Y_2)^2(t,\theta)d\theta\big|_{\eta=0}^{1}\\
& \qquad\qquad\qquad\qquad\qquad\qquad\qquad\qquad+ \int_0^{1} (\tilde\Y_1-\tilde \Y_2)^2(t,\eta)d\eta\Big)\\
& \leq \bigO(1) \norm{\tilde\Y_1-\tilde\Y_2}^2.
\end{align*}
We conclude that
\begin{equation*}
\abs{L_{32}}\leq \bigO(1) \norm{\tilde\Y_1-\tilde\Y_2}^2.
\end{equation*}

Finally, we have a look at $J_9$ (the argument for $J_{10}$ follows the same lines). We have 
\begin{align*}
& \big\vert \int_0^1 J_9(\tilde\Y_1-\tilde\Y_2)(t,\eta) d\eta\big\vert \\
&\quad  = \big\vert \frac{\sqtC{1}^6-\sqtC{2}^6}{\sqtC{1}^6\sqtC{2}^6}\mathbbm{1}_{\sqtC{1}\leq \sqtC{2}}\int_0^1(\tilde \Y_1-\tilde \Y_2)\tilde \Y_{1,\eta} (t,\eta) \\
&\qquad\qquad\qquad\quad\times\left(\int_0^\eta e^{-\frac{1}{\sqtC{1}}(\tilde \Y_1(t,\eta)-\tilde \Y_1(t,\theta))} (\tilde \V_1^+)^3\tilde \Y_{1,\eta} (t,\theta) d\theta\right) d\eta\big\vert \\
& \quad\leq\frac{\sqtC{2}^6-\sqtC{1}^6}{\sqrt{2}\sqtC{1}^6\sqtC{2}^6}\sqtC{1}^2\mathbbm{1}_{\sqtC{1}\leq \sqtC{2}} \\
&\qquad\qquad\quad\times\int_0^1 \vert \tilde\Y_1-\tilde \Y_2\vert \tilde \Y_{1,\eta} (t, \eta) \left(\int_0^\eta e^{-\frac{1}{\sqtC{1}}(\tilde \Y_1(t,\eta)-\tilde \Y_1(t,\theta))} \tilde \U_1^2\tilde\Y_{1,\eta} (t,\theta) d\theta\right) d\eta\\
& \quad\leq 4\frac{\sqtC{2}^6-\sqtC{1}^6}{\sqrt{2}\sqtC{1}^6\sqtC{2}^6}\sqtC{1}^3\mathbbm{1}_{\sqtC{1}\leq \sqtC{2}}\int_0^1\vert \tilde \Y_1-\tilde \Y_2\vert \tilde \P_1\tilde \Y_{1,\eta} (t,\eta) d\eta\\
& \quad\leq 2\frac{\sqtC{2}^6-\sqtC{1}^6}{\sqrt{2}\sqtC{1}^6\sqtC{2}^6}\sqtC{1}^8\mathbbm{1}_{\sqtC{1}\leq \sqtC{2}}\int_0^1\vert \tilde \Y_1-\tilde \Y_2\vert  (t,\eta) d\eta\\
& \quad\leq 6\sqrt{2}\A(\sqtC{2}-\sqtC{1}) \norm{\tilde \Y_1-\tilde\Y_2}\\
& \quad \leq 6\A \Big(\norm{\tilde \Y_1-\tilde\Y_2}^2+\vert \sqtC{1}-\sqtC{2}\vert ^2\Big).
\end{align*}

We now turn our attention to $I_{32}$, i.e.,
\begin{align*}
I_{32}& =\int_0^1 (\tilde\Y_1-\tilde \Y_2)(t,\eta)\Big(\frac{1}{\A_2^6}\tilde \Y_{2,\eta}(t,\eta) \int_0^1 e^{-\frac{1}{\sqtC{2}}\vert \tilde \Y_2(t,\eta)-\tilde \Y_2(t,\theta)\vert } 
\tilde \P_2\tilde \U_2\tilde \Y_{2,\eta} (t,\theta)d\theta\\ 
& \qquad \qquad\qquad \qquad -\frac{1}{\sqtC{1}^6} \tilde \Y_{1,\eta} (t,\eta) \int_0^1 e^{-\frac{1}{\sqtC{1}}\vert \tilde \Y_1(t,\eta)-\tilde \Y_1(t,\theta)\vert } \tilde \P_1\tilde\U_1\tilde \Y_{1,\eta} (t,\theta) d\theta\Big)d\eta.
\end{align*}
As before, we are only going to establish the estimates for one part of it, since the other parts can be treated similarly.
Let
\begin{align}\nn
\tilde I_{32}&= \int_0^1 (\tilde \Y_1-\tilde \Y_2)(t,\eta)\Big(\frac{1}{\sqtC{2}^6}\tilde \Y_{2,\eta} (t,\eta) \int_0^\eta e^{-\frac{1}{\sqtC{2}}(\tilde \Y_2(t,\eta)-\tilde \Y_2(t,\theta))}\tilde \P_2\tilde \V_2^+\tilde \Y_{2,\eta} (t,\theta) d\theta\\ \nn
& \qquad \qquad \qquad \qquad -\frac{1}{\sqtC{1}^6} \tilde \Y_{1,\eta} (t,\eta) \int_0^\eta e^{-\frac{1}{\sqtC{1}}(\tilde\Y_1(t,\eta)-\tilde \Y_1(t,\theta))} \tilde \P_1\tilde \V_1^+\tilde \Y_{1,\eta} (t,\theta) d\theta\Big) d\eta\\ \nn
& = \frac{1}{\A^6} \int_0^1 (\tilde \Y_1-\tilde \Y_2)(t,\eta) \Big(\tilde \Y_{2,\eta} (t,\eta)\int_0^\eta e^{-\frac{1}{\sqtC{2}}(\tilde \Y_2(t,\eta)-\tilde \Y_2(t,\theta))}\tilde\P_2\tilde \V_2^+\tilde \Y_{2,\eta} (t,\theta) d\theta\\ \nn
& \qquad \qquad \qquad \qquad - \tilde \Y_{1,\eta} (t,\eta) \int_0^\eta e^{-\frac{1}{\sqtC{1}}(\tilde \Y_1(t,\eta)-\tilde \Y_1(t,\theta))} \tilde \P_1\tilde \V_1^+\tilde \Y_{1,\eta} (t,\theta) d\theta\Big) d\eta\\ \nn
& \quad +\Big(\frac{1}{\sqtC{2}^6}-\frac{1}{\sqtC{1}^6} \Big) \mathbbm{1}_{\sqtC{1}\leq \sqtC{2}} \int_0^1 (\tilde\Y_1-\tilde \Y_2)\tilde \Y_{1,\eta} (t,\eta)\\ \nn
&\qquad\qquad\qquad\qquad \times\left( \int_0^\eta e^{-\frac{1}{\sqtC{1}}(\tilde \Y_1(t,\eta)-\tilde \Y_1(t,\theta))}\tilde \P_1\tilde \V_1^+\tilde \Y_{1,\eta} (t,\theta)d\theta \right)d\eta\\ \nn
& \quad +\Big(\frac{1}{\sqtC{2}^6}-\frac{1}{\sqtC{1}^6}\Big) \mathbbm{1}_{\sqtC{2}< \sqtC{1}} \int_0^1 (\tilde\Y_1-\tilde \Y_2) \tilde \Y_{2,\eta} (t,\eta)\\ \nn
&\qquad\qquad\qquad\qquad\qquad\qquad \times \left(\int_0^\eta e^{-\frac{1}{\sqtC{2}}(\tilde\Y_2(t,\eta)-\tilde \Y_2(t,\theta))}\tilde \P_2\tilde \V_2^+\tilde \Y_{2,\eta} (t,\theta) d\theta\right) d\eta\\ 
& = \tilde K_1+\tilde K_2+\tilde K_3.  \label{eq:alleK}
\end{align}

Direct calculations yield for $\tilde K_2$ (and similar for $\tilde K_3$) that 
\begin{align*}
\vert \tilde K_2\vert & \leq \frac{\sqtC{2}^6-\sqtC{1}^6}{\sqtC{1}^6\sqtC{2}^6}\mathbbm{1}_{\sqtC{1}\le \sqtC{2}} \\
&\qquad\times \int_0^1 \vert \tilde\Y_1-\tilde \Y_2\vert \tilde \Y_{1,\eta} (t,\eta) \left(\int_0^\eta e^{-\frac{1}{\sqtC{1}}(\tilde \Y_1(t,\eta)-\tilde \Y_1(t,\theta))} \tilde \P_1\tilde \V_1^+\tilde \Y_{1,\eta}(t,\theta)d\theta \right)d\eta\\
& \leq \frac{\sqtC{2}^6-\sqtC{1}^6}{\sqtC{1}^6\sqtC{2}^6} \norm{\tilde \Y_1-\tilde \Y_2}\\
&\qquad\qquad\times\left( \int_0^1 \tilde \Y_{1,\eta}^2(t,\eta) \Big(\int_0^\eta e^{-\frac{1}{\sqtC{1}}(\tilde \Y_1(t,\eta)-\tilde \Y_1(t,\theta))} \tilde \P_1\tilde \V_1^+\tilde \Y_{1,\eta} (t,\theta) d\theta\Big)^2d\eta\right)^{1/2}\\
& \leq \frac{\sqtC{2}^6-\sqtC{1}^6}{\sqtC{1}^6\sqtC{2}^6} \norm{\tilde \Y_1-\tilde\Y_2} \Big(\int_0^1 \tilde \Y_{1,\eta}^2(t,\eta) \int_0^\eta e^{-\frac{1}{\sqtC{1}}(\tilde\Y_1(t,\eta)-\tilde \Y_1(t,\theta))}\tilde \P_1^2\tilde \Y_{1,\eta} (t,\theta) d\theta\\
& \qquad \qquad\qquad\qquad \qquad\qquad  \times \int_0^\eta e^{-\frac{1}{\sqtC{1}}(\tilde \Y_1(t,\eta)-\tilde \Y_1(t,\theta))} \tilde \U_1^2\tilde \Y_{1,\eta} (t,\theta) d\theta d\eta\Big)^{1/2}\\
& \leq \sqrt{6} \frac{\sqtC{2}^6-\sqtC{1}^6}{\sqtC{1}^3\sqtC{2}^6} \norm{\tilde\Y_1-\tilde \Y_2}\Big( \int_0^1 \tilde \P_1^2\tilde \Y_{1,\eta}^2 (t,\eta) d\eta\Big)^{1/2}\\
& \leq \frac{\sqrt{6}}{2} \frac{\sqtC{2}^6-\sqtC{1}^6} {\A^4} \norm{\tilde \Y_1-\tilde \Y_2}\\
& \leq \bigO(1) (\norm{\tilde\Y_1-\tilde\Y_2}^2+\vert \sqtC{2}-\sqtC{1}\vert ^2),
\end{align*}
where we used \eqref{eq:all_PestimatesB}, \eqref{eq:all_PestimatesC}, and $\sqtC{1}\leq \sqtC{2}$.

On the other hand, the term $\tilde K_1$ needs to be rewritten a bit more. Namely, 
\begin{align*}
\tilde K_1& = \frac{1}{\A^6} \int_0^1 (\tilde \Y_1-\tilde \Y_2)(t,\eta) \Big(\tilde \Y_{2,\eta} (t,\eta)\int_0^\eta e^{-\frac{1}{\sqtC{2}}(\tilde \Y_2(t,\eta)-\tilde \Y_2(t,\theta))}\tilde\P_2\tilde \V_2^+\tilde \Y_{2,\eta} (t,\theta) d\theta\\ 
& \qquad \qquad \qquad \qquad\qquad - \tilde \Y_{1,\eta} (t,\eta) \int_0^\eta e^{-\frac{1}{\sqtC{1}}(\tilde \Y_1(t,\eta)-\tilde \Y_1(t,\theta))} \tilde \P_1\tilde \V_1^+\tilde \Y_{1,\eta} (t,\theta) d\theta\Big) d\eta\\ \nn
& = \frac{1}{\A^6} \int_0^1 (\tilde \Y_1-\tilde\Y_2)\tilde\Y_{2,\eta} (t,\eta)\\
&\qquad\qquad\times\Big(\int_0^\eta e^{-\frac{1}{\sqtC{2}}(\tilde \Y_2(t,\eta)-\tilde\Y_2(t,\theta))}(\tilde\P_2-\tilde\P_1)\tilde\V_2^+\tilde\Y_{2,\eta}(t,\theta)\mathbbm{1}_{\tilde \P_1\leq \tilde\P_2}(t,\theta) d\theta\Big)d\eta\\
& \quad +\frac1{\A^6}\int_0^1 (\tilde \Y_1-\tilde\Y_2)\tilde\Y_{1,\eta} (t,\eta)\\
&\qquad\qquad\times\Big(\int_0^\eta e^{-\frac{1}{\sqtC{1}}(\tilde\Y_1(t,\eta)-\tilde\Y_1(t,\theta))}(\tilde\P_2-\tilde\P_1)\tilde \V_1^+\tilde\Y_{1,\eta}(t,\theta) \mathbbm{1}_{\tilde \P_2< \tilde\P_1}(t,\theta)d\theta\Big) d\eta\\
& \quad +\frac{1}{\A^6} \int_0^1 (\tilde\Y_1-\tilde\Y_2)\tilde\Y_{2,\eta} (t,\eta)\\
&\quad\times\Big( \int_0^\eta e^{-\frac{1}{\sqtC{2}}(\tilde\Y_2(t,\eta)-\tilde\Y_2(t,\theta))}\min_j(\tilde\P_j)(\tilde\V_2^+-\tilde\V_1^+)\tilde \Y_{2,\eta} (t,\theta)\mathbbm{1}_{\tilde\V_1^+\leq \tilde\V_2^+}(t,\theta)d\theta\Big)d\eta\\
& \quad +\frac{1}{\A^6} \int_0^1 (\tilde\Y_1-\tilde\Y_2)\tilde\Y_{1,\eta} (t,\eta)\\
&\quad\times\Big(\int_0^\eta e^{-\frac{1}{\sqtC{1}}(\tilde \Y_1(t,\eta)-\tilde \Y_1(t,\theta))} \min_j(\tilde\P_j)(\tilde \V_2^+-\tilde \V_1^+)\tilde \Y_{1,\eta} (t,\theta)\mathbbm{1}_{\tilde\V_2^+<\tilde\V_1^+}(t,\theta)d\theta\Big) d\eta\\
& \quad + \frac{1}{\A^6} \int_0^1 (\tilde \Y_1-\tilde\Y_2)(t,\eta)\\
&\qquad\qquad\times \Big(\tilde \Y_{2,\eta} (t,\eta) \int_0^\eta e^{-\frac{1}{\sqtC{2}}(\tilde \Y_2(t,\eta)-\tilde \Y_2(t,\theta))} \min_j(\tilde\P_j)\min_j(\tilde\V_j^+) \tilde \Y_{2,\eta} (t,\theta) d\theta\\
& \qquad \qquad  -\tilde \Y_{1,\eta} (t,\eta) \int_0^\eta e^{-\frac{1}{\sqtC{1}}(\tilde \Y_1(t,\eta)-\tilde \Y_1(t,\theta))} \min_j(\tilde\P_j)\min_j(\tilde\V_j^+)\tilde\Y_{1,\eta} (t,\theta)d\theta\Big) d\eta\\
& = \tilde J_1+\tilde J_2+\tilde J_3+\tilde J_4+\tilde J_5.
\end{align*}

We start by having a close look at $\tilde J_1$ ($\tilde J_2$ can be handled similarly). One has
\begin{align*}
\vert \tilde J_1\vert&  \leq \frac1{\A^6}\vert \int_0^1 (\tilde \Y_1-\tilde \Y_2)\tilde\Y_{2,\eta} (t,\eta)\\
&\qquad\times \Big(\int_0^\eta e^{-\frac{1}{\sqtC{2}}(\tilde\Y_2(t,\eta)-\tilde\Y_2(t,\theta))}(\tilde\P_2-\tilde\P_1)\tilde\V_2^+\tilde\Y_{2,\eta} (t,\theta) \mathbbm{1}_{\tilde \P_1\leq \tilde\P_2} (t,\theta) d\theta\big)d\eta\vert \\
& \leq \norm{\tilde \Y_1-\tilde\Y_2}^2
 + \frac1{\A^{12}} \int_0^1 \tilde \Y_{2,\eta}^2(t,\eta)\Big( \int_0^\eta e^{-\frac{1}{\sqtC{2}}(\tilde \Y_2(t,\eta)-\tilde\Y_2(t,\theta))}\\
&\qquad\qquad\times(\sqP{2}-\sqP{1})(\sqP{2}+\sqP{1})\tilde\V_2^+\tilde\Y_{2,\eta}(t,\theta) \mathbbm{1}_{\tilde \P_1\leq \tilde\P_2}(t,\theta) d\theta\Big)^2d\eta\\
& \leq \norm{\tilde\Y_1-\tilde\Y_2}^2 + \frac{1}{\A^{12}} \int_0^1 \tilde \Y_{2,\eta}^2(t,\eta)\\
&\qquad\qquad \times\Big( 2\int_0^\eta e^{-\frac{1}{\sqtC{2}}(\tilde \Y_2(t,\eta)-\tilde\Y_2(t,\theta))}(\sqP{2}-\sqP{1})\sqP{2}\tilde\V_2^+\tilde\Y_{2,\eta}(t,\theta)  d\theta\Big)^2d\eta\\
& \leq \norm{\tilde \Y_1-\tilde\Y_2}^2
+\frac{4}{\A^{12}} \int_0^1 \tilde\Y_{2,\eta}^2(t,\eta) \Big( \int_0^\eta e^{-\frac{1}{\sqtC{2}}(\tilde\Y_2(t,\eta)-\tilde\Y_2(t,\theta))}\tilde \U_2^2\tilde\Y_{2,\eta}(t,\theta)d\theta\Big)\\
& \qquad \qquad \qquad    \times\Big(\int_0^\eta e^{-\frac{1}{\sqtC{2}}(\tilde \Y_2(t,\eta)-\tilde\Y_2(t,\theta))} (\sqP{2}-\sqP{1})^2\tilde\P_2\tilde\Y_{2,\eta} (t,\theta)d\theta\Big)d\eta\\
& \leq \norm{\tilde\Y_1-\tilde\Y_2}^2 +\frac{16}{\A^{11}} \int_0^1 \tilde\P_2\tilde\Y_{2,\eta}^2 (t,\eta)\\
& \qquad\qquad\qquad\times\Big(\int_0^\eta e^{-\frac{1}{\sqtC{2}}(\tilde\Y_2(t,\eta)-\tilde\Y_2(t,\theta))}(\sqP{2}-\sqP{1})^2\tilde \P_2\tilde\Y_{2,\eta} (t,\theta) d\theta\Big)d\eta\\
& \leq \norm{\tilde\Y_1-\tilde\Y_2}^2 +\frac{8}{\A^6} \int_0^1 \tilde\Y_{2,\eta} (t,\eta)\\
& \qquad\times\Big(\int_0^\eta e^{-\frac{1}{\sqtC{2}}(\tilde\Y_2(t,\eta)-\tilde\Y_2(t,\theta))}(\sqP{2}-\sqP{1})^2\tilde \P_2\tilde\Y_{2,\eta} (t,\theta) d\theta\Big)d\eta\\
& \leq \norm{\tilde\Y_1-\tilde\Y_2}^2 + \frac{8}{\A^6} \int_0^1 \tilde\Y_{2,\eta}(t,\eta) e^{-\frac{1}{\sqtC{2}}\tilde\Y_2(t,\eta)}\\
& \qquad\times\Big( \int_0^\eta e^{\frac{1}{\sqtC{2}}\tilde\Y_2(t,\theta)} (\sqP{2}-\sqP{1})^2 \tilde \P_2\tilde\Y_{2,\eta} (t,\theta) d\theta\Big) d\eta\\
& \leq \norm{\tilde\Y_1-\tilde\Y_2}^2\\
& \quad + \frac{8\sqtC{2}}{\A^6} \Big(-\int_0^\eta e^{-\frac{1}{\sqtC{2}}(\tilde\Y_2(t,\eta)-\tilde \Y_2(t,\theta))}(\sqP{2}-\sqP{1})^2\tilde\P_2\tilde\Y_{2,\eta} (t,\theta)d\theta |_{\eta=0}^1\\
& \qquad \qquad \qquad \qquad \qquad \qquad +\int_0^1 (\sqP{2}-\sqP{1})^2\tilde\P_2\tilde\Y_{2,\eta}(t,\eta) d\eta\Big)\\
&\leq \bigO(1) (\norm{\tilde\Y_1-\tilde\Y_2}^2+\norm{\sqP{1}-\sqP{2}}^2),
\end{align*}
where we used that 
\begin{align*}
0&\leq \int_0^\eta e^{-\frac{1}{\sqtC{2}}(\tilde\Y_2(t,\eta)-\tilde\Y_2(t,\theta))} (\sqP{2}-\sqP{1})^2\tilde\P_2\tilde\Y_{2,\eta} (t,\eta)d\eta\\
& \leq \A^4\int_0^\eta e^{-\frac{1}{\sqtC{2}}(\tilde\Y_2(t,\eta)-\tilde\Y_2(t,\theta))}\tilde \P_2\tilde\Y_{2,\eta} (t,\theta) d\theta\\
& \leq \A^4\Big(\int_0^\eta e^{-\frac{1}{\sqtC{2}}(\tilde \Y_2(t,\eta)-\tilde\Y_2(t,\theta))} \tilde\Y_{2,\eta}(t,\theta) d\theta\Big)^{1/2} \\
&\qquad\qquad\times \Big(\int_0^\eta e^{-\frac{1}{\sqtC{2}}(\tilde\Y_2(t,\eta)-\tilde\Y_2(t,\theta))} \tilde \P_2^2\tilde\Y_{2,\eta} (t,\theta) d\theta\Big)^{1/2}\\
& \leq \A^5\bigO(1) \sqP{2}(t,\eta),
\end{align*}
where the very last term tends to $0$ as $\eta\to 0,1$. In the last step we used \eqref{eq:all_PestimatesB}.

Next, we investigate $\tilde J_3$ ($\tilde J_4$ can be handled in a similar way). The argument is a bit more involved and hence we start with some preliminary estimates. One has
\begin{align*}
\int_0^\eta e^{-\frac1{2\sqtC{2}} (\tilde\Y_2(t,\eta)-\tilde\Y_2(t,\theta))}& (\tilde\V_2^+-\tilde\V_1^+)^2\tilde\P_2\tilde\Y_{2,\eta}(t,\theta) d\theta\\
& \leq 2\A^4\int_0^\eta e^{-\frac1{2\sqtC{2}} (\tilde \Y_2(t,\eta)-\tilde\Y_2(t,\theta))} \tilde\P_2\tilde\Y_{2,\eta}(t,\theta) d\theta\\
& \leq 2\A^4\Big(\int_0^\eta e^{-\frac1{2\sqtC{2}} (\tilde\Y_2(t,\eta)-\tilde\Y_2(t,\theta))}\tilde\Y_{2,\eta}(t,\theta) d\theta\Big)^{1/2}\\
&\qquad\qquad\times\Big(\int_0^\eta e^{-\frac1{2\sqtC{2}} (\tilde \Y_2(t,\eta)-\tilde\Y_2(t,\theta))}\tilde\P_2^2\tilde\Y_{2,\eta}(t,\theta) d\theta\Big)^{1/2}\\
& \leq A^5\bigO(1) \tilde\P_2^{1/4}(t,\eta),
\end{align*}
where the term on the right-hand side tends to $0$ for $\eta\to 0,1$.
In the last step we used \eqref{eq:p2y}, cf.~\eqref{eq:Q_d} (and the derivate $\tilde\Q_{2,\eta}$ can be found from \eqref{eq:DogP} and \eqref{eq:D_deriv}).
With these estimates in mind, we end up with, recalling \eqref{eq:343},  
\begin{align*}
\vert \tilde J_3\vert & = \frac{1}{\A^6}\vert \int_0^1 (\tilde\Y_1-\tilde\Y_2)\tilde\Y_{2,\eta}(t,\eta)\\
&\quad\times\Big(\int_0^\eta e^{-\frac{1}{\sqtC{2}}(\tilde\Y_2(t,\eta)-\tilde\Y_2(t,\theta))}\min_j(\tilde\P_j)(\tilde\V_2^+-\tilde\V_1^+)\tilde\Y_{2,\eta} (t,\theta)\mathbbm{1}_{\tilde\V_1^+\leq \tilde\V_2^+} (t,\theta) d\theta\Big)d\eta\vert \\
& \leq \norm{\tilde\Y_1-\tilde\Y_2}^2\\
& \quad +\frac{1}{\A^{12}} \int_0^1\tilde \Y_{2,\eta}^2(t,\eta)\\
&\quad\times \Big(\int_0^\eta e^{-\frac{1}{\sqtC{2}}(\tilde\Y_2(t,\eta)-\tilde\Y_2(t,\theta))}\min_j(\tilde\P_j)(\tilde\V_2^+-\tilde\V_1^+)\tilde\Y_{2,\eta}(t,\theta) \mathbbm{1}_{\tilde\V_1^+\leq \tilde\V_2^+}(t,\theta) d\theta\Big)^2 d\eta\\
& \leq \norm{\tilde\Y_1-\tilde\Y_2}^2\\
& \quad +\frac{1}{\A^{12}} \int_0^1 \tilde\Y_{2,\eta}^2(t,\eta)\Big(\int_0^\eta e^{-\frac3{2\sqtC{2}} (\tilde\Y_2(t,\eta)-\tilde\Y_2(t,\theta))}\tilde\P_2\tilde\Y_{2,\eta}(t,\theta) d\theta\Big)\\
& \qquad \qquad \qquad  \times\Big(\int_0^\eta e^{-\frac1{2\sqtC{2}} (\tilde\Y_2(t,\eta)-\tilde\Y_2(t,\theta))}(\tilde\V_2^+-\tilde\V_1^+)^2\tilde\P_2\tilde\Y_{2,\eta}(t,\theta)d\theta \Big)d\eta\\
&\leq\norm{\tilde \Y_1-\tilde\Y_2}^2 + \frac{2}{\A^{11}} \int_0^1 \tilde\P_2\tilde\Y_{2,\eta}^2 (t,\eta)\\
& \qquad\times \Big(\int_0^\eta e^{-\frac1{2\sqtC{2}} (\tilde\Y_2(t,\eta)-\tilde\Y_2(t,\theta))}(\tilde\V_2^+-\tilde \V_1^+)^2\tilde\P_2\tilde\Y_{2,\eta} (t,\theta) d\theta \Big)d\eta\\
&\leq \norm{\tilde\Y_1-\tilde\Y_2}^2\\
& \quad + \frac{1}{\A^6} \int_0^1 \tilde\Y_{2,\eta} (t,\eta) e^{-\frac1{2\sqtC{2}} \tilde\Y_2(t,\eta)}\Big(\int_0^\eta e^{\frac1{2\sqtC{2}} \tilde\Y_2(t,\theta)} (\tilde\V_2^+-\tilde\V_1^+)^2\tilde \P_2\tilde\Y_{2,\eta} (t,\theta)d\theta\Big) d\eta\\ 
& = \norm{\tilde\Y_1-\tilde\Y_2}^2\\
& \quad+ \frac{1}{\A^6} \Big(-2\sqtC{2}\int_0^\eta e^{-\frac1{2\sqtC{2}} (\tilde\Y_2(t,\eta)-\tilde\Y_2(t,\theta))}(\tilde\V_2^+-\tilde\V_1^+)^2\tilde\P_2\tilde\Y_{2,\eta} (t,\theta) d\theta\vert_{\eta=0}^1\\
& \qquad \qquad \qquad \qquad +2\sqtC{2}\int_0^1 (\tilde\V_2^+-\tilde\V_1^+)^2 \tilde \P_2\tilde\Y_{2,\eta} (t,\eta) d\eta\Big)\\
& \leq \bigO(1)\big(\norm{\tilde \Y_1-\tilde\Y_2}^2 +\norm{\tilde \U_2-\tilde\U_1}^2\big).
\end{align*}

As far as $\tilde J_5$ is concerned, we have to rewrite it a bit more. Namely,
\begin{align*}
\tilde J_5&= \frac1{\A^6} \int_0^1 (\tilde \Y_1-\tilde\Y_2)(t,\eta)\\
&\qquad\times\Big(\tilde\Y_{2,\eta} (t,\eta) \int_0^\eta e^{-\frac{1}{\sqtC{2}}(\tilde\Y_2(t,\eta)-\tilde\Y_2(t,\theta))}\min_j(\tilde\P_j)\min_j(\tilde\V_j^+)\tilde \Y_{2,\eta} (t,\theta) d\theta\\ 
& \qquad \quad  -\tilde\Y_{1,\eta} (t,\eta) \int_0^\eta e^{-\frac{1}{\sqtC{1}}(\tilde\Y_1(t,\eta)-\tilde\Y_1(t,\theta))}\min_j(\tilde\P_j)\min_j(\tilde\V_j^+)\tilde\Y_{1,\eta} (t,\theta) d\theta\Big) d\eta\\
& = \frac{1}{\A^6} \mathbbm{1}_{\sqtC{1}\leq \sqtC{2}}\int_0^1 (\tilde \Y_1-\tilde \Y_2)\tilde \Y_{2,\eta}(t,\eta)\\
&\qquad\qquad\qquad\times\Big(\int_0^\eta (e^{-\frac{1}{\sqtC{2}}(\tilde \Y_2(t,\eta)-\tilde \Y_2(t,\theta))}-e^{-\frac{1}{\sqtC{1}}(\tilde \Y_2(t,\eta)-\tilde \Y_2(t,\theta))})\\
& \qquad\qquad\qquad\qquad\qquad \times \min_j(\tilde \P_j)\min_j(\tilde \V_j^+)\tilde \Y_{2,\eta}(t,\theta) d\theta\Big) d\eta\\
& \quad +\frac{1}{\A^6} \mathbbm{1}_{\sqtC{2}<\sqtC{1}} \int_0^1 (\tilde \Y_1-\tilde \Y_2)\tilde \Y_{1,\eta}(t,\eta)\\
&\qquad\qquad\qquad\times \Big( \int_0^\eta (e^{-\frac{1}{\sqtC{2}}(\tilde \Y_1(t,\eta)-\tilde \Y_1(t,\theta))}-e^{-\frac{1}{\sqtC{1}}(\tilde \Y_1(t,\eta)-\tilde \Y_1(t,\theta))})\\
& \qquad\qquad\qquad \qquad\qquad\times\min_j(\tilde \P_j)\min_j(\tilde \V_j^+)\tilde \Y_{1,\eta}(t,\theta) d\theta\Big) d\eta\\
& \quad +\frac1{\A^6} \int_0^1  (\tilde\Y_1-\tilde\Y_2)\tilde \Y_{2,\eta}(t,\eta) \\
& \qquad \times\Big(\int_0^\eta ( e^{-\frac{1}{\ma}(\tilde \Y_2(t,\eta)-\tilde\Y_2(t,\theta))}-e^{-\frac{1}{\ma}(\tilde \Y_1(t,\eta)-\tilde \Y_1(t,\theta))})\min_j(\tilde\P_j)\\
&\qquad\qquad\qquad\qquad\qquad\qquad\qquad\qquad\times\min_j(\tilde\V_j^+)\tilde \Y_{2,\eta} (t,\theta)\mathbbm{1}_{B(\eta)}(t,\theta) d\theta\Big) d\eta\\
& \quad +\frac1{\A^6} \int_0^1 (\tilde \Y_1-\tilde \Y_2)\tilde \Y_{1,\eta} (t,\eta)\\
& \qquad \times \Big(\int_0^\eta (e^{-\frac{1}{\ma}(\tilde \Y_2(t,\eta) -\tilde \Y_2(t,\theta))}-e^{-\frac{1}{\ma}(\tilde \Y_1(t,\eta)-\tilde \Y_1(t,\theta))})\min_j(\tilde\P_j)\\
&\qquad\qquad\qquad\qquad\qquad\qquad\qquad\qquad\times
\min_j(\tilde \V_j^+)\tilde \Y_{1,\eta} (t,\theta) \mathbbm{1}_{B^c(\eta)}(t,\theta) d\theta\Big) d\eta\\
& \quad +\frac1{\A^6} \int_0^1 (\tilde \Y_1-\tilde \Y_2)(t,\eta) \\
& \qquad \times \Big(\tilde \Y_{2,\eta} (t,\eta) \int_0^\eta \min_j(e^{-\frac{1}{\ma}(\tilde \Y_j(t,\eta)-\tilde\Y_j(t,\theta))})\min_j(\tilde\P_j)\\
&\qquad\qquad\qquad\qquad\qquad\qquad\qquad\qquad\qquad\qquad\times\min_j(\tilde \V_j^+)\tilde \Y_{2,\eta} (t,\theta)d\theta\\
& \qquad \qquad -\tilde \Y_{1,\eta} (t,\eta)\int_0^\eta \min_j(e^{-\frac{1}{\ma}(\tilde \Y_j(t,\eta)-\tilde\Y_j(t,\theta))})\min_j(\tilde\P_j)\\
&\qquad\qquad\qquad\qquad\qquad\qquad\qquad\qquad\qquad\qquad\times \min_j(\tilde \V_j^+)\tilde \Y_{1,\eta} (t,\theta) d\theta\Big) d\eta\\
& = \tilde L_1+\tilde L_2+\tilde L_3+\tilde L_4+ \tilde L_5,
\end{align*}
where $B(\eta)$ is given by \eqref{Def:Bn}.

Both $\tilde L_1$ and $\tilde L_2$ can be handled in much the same way, and  therefore we only consider $\tilde L_1$. 
One has 
\begin{align*}
\vert \tilde L_1\vert & =  \frac{1}{\A^6} \mathbbm{1}_{\sqtC{1}\leq \sqtC{2}}\vert \int_0^1 (\tilde \Y_1-\tilde \Y_2)\tilde \Y_{2,\eta}(t,\eta)\\
&\qquad\qquad\qquad\times\Big(\int_0^\eta (e^{-\frac{1}{\sqtC{2}}(\tilde \Y_2(t,\eta)-\tilde \Y_2(t,\theta))}-e^{-\frac{1}{\sqtC{1}}(\tilde \Y_2(t,\eta)-\tilde \Y_2(t,\theta))})\\
& \qquad\qquad\qquad\qquad\qquad\qquad\qquad\qquad \times \min_j(\tilde \P_j)\min_j(\tilde \V_j^+)\tilde \Y_{2,\eta}(t,\theta) d\theta\Big) d\eta\vert\\ 
& \leq \frac{4}{\ma\A^6 e}\mathbbm{1}_{\sqtC{1}\leq \sqtC{2}}\int_0^1 \vert \tilde \Y_1-\tilde \Y_2\vert \tilde \Y_{2,\eta}(t,\eta)\\
& \qquad \times \Big( \int_0^\eta e^{-\frac{3}{4\sqtC{2}}(\tilde \Y_2(t,\eta)-\tilde \Y_2(t,\theta))}\min_j(\tilde \P_j)\min(\tilde \V_j^+)\tilde \Y_{2,\eta}(t,\theta) d\theta\Big) d\eta\vert \sqtC{1}-\sqtC{2}\vert\\
& \leq \norm{\tilde \Y_1-\tilde \Y_2}^2  + \frac{16}{\ma^2\A^{12}e^2}\mathbbm{1}_{\sqtC{1}\leq \sqtC{2}}\int_0^1 \tilde \Y_{2,\eta}^2(t,\eta) \\
&\qquad\times\Big(\int_0^\eta e^{-\frac{3}{4\sqtC{2}}(\tilde \Y_2(t,\eta)-\tilde \Y_2(t,\theta))}\vert \tilde \U_1\vert \tilde \P_2\tilde \Y_{2,\eta}(t,\theta) d\theta\Big)^2 d\eta \vert \sqtC{1}-\sqtC{2}\vert^2\\
& \leq \norm{\tilde \Y_1-\tilde \Y_2}^2 + \frac{8}{\A^{10}e^2}\int_0^1 \tilde \Y_{2,\eta}^2 (t,\eta)\\
& \qquad \quad\times\Big(\int_0^\eta e^{-\frac{3}{4\sqtC{2}}(\tilde \Y_2(t,\eta)-\tilde \Y_2(t,\theta))} \tilde \P_2\tilde \Y_{2,\eta}(t,\theta)d\theta\Big)^2 d\eta \vert \sqtC{1}-\sqtC{2}\vert ^2\\
&  \leq \norm{\tilde \Y_1-\tilde \Y_2}^2 + \frac{8}{\A^{10}e^2}\int_0^1 \tilde \Y_{2,\eta}^2(t,\eta)\Big( \int_0^\eta e^{-\frac{5}{4\sqtC{2}}(\tilde \Y_2(t,\eta)-\tilde \Y_2(t,\theta))}\tilde \P_2\tilde \Y_{2,\eta}(t,\theta) d\theta\Big)\\
& \qquad \qquad\qquad \qquad \times \Big(\int_0^\eta e^{-\frac{1}{4\sqtC{2}}(\tilde \Y_2(t,\eta)-\tilde \Y_2(t,\theta))}\tilde \P_2\tilde \Y_{2,\eta}(t,\theta) d\theta\Big) d\eta\vert \sqtC{1}-\sqtC{2}\vert ^2\\
& \leq \norm{\tilde \Y_1-\tilde \Y_2}^2  + \frac{32}{\A^9 e^2} \int_0^1 \tilde \P_2\tilde \Y_{2,\eta}^2(t,\eta) \\
& \qquad \qquad\qquad\qquad\times\Big(\int_0^\eta e^{-\frac{1}{4\sqtC{2}} (\tilde \Y_2(t,\eta)-\tilde \Y_2(t,\theta))} \tilde \P_2\tilde \Y_{2,\eta} (t,\theta) d\theta\Big) d\eta \vert \sqtC{1}-\sqtC{2}\vert ^2\\
&  \leq \norm{\tilde \Y_1-\tilde \Y_2}^2  + \frac{16}{\A^4e^2} \int_0^{1} \tilde \Y_{2,\eta}(t,\eta) e^{-\frac1{4\sqtC{2}}\tilde \Y_2(t,\eta)}\\
& \qquad \qquad\qquad\qquad\times\int_0^\eta e^{\frac{1}{4\sqtC{2}}\tilde \Y_2(t,\theta)}\tilde \P_2\tilde \Y_{2,\eta} (t,\theta) d\theta d\eta \vert \sqtC{1}-\sqtC{2}\vert ^2\\
&  \leq \norm{\tilde \Y_1-\tilde \Y_2}^2 +\frac{64}{\A^3 e^2}\Big(-\int_0^\eta e^{-\frac{1}{4\sqtC{2}}(\tilde \Y_2(t,\eta)-\tilde \Y_2(t,\theta))}\tilde \P_2\tilde \Y_{2,\eta} (t,\theta) d\theta\big\vert _{\eta=0}^1 \\
& \qquad \qquad \qquad \qquad \qquad \qquad \quad+\int_0^1 \tilde \P_2\tilde \Y_{2,\eta} (t,\eta) d\eta\Big) \vert \sqtC{1}-\sqtC{2}\vert ^2\\
& \leq \bigO(1)(\norm{\tilde \Y_1-\tilde \Y_2}^2 +\vert \sqtC{1}-\sqtC{2}\vert^2),
\end{align*}
where we used  \eqref{eq:all_PestimatesE}.

Next we turn our attention to $\tilde L_3$ and $\tilde L_4$, which can be handled in much the same way, and therefore we only consider $\tilde L_3$.
One has 
\begin{align*}
\vert \tilde L_3\vert & = \frac1{\A^6} \vert \int_0^1 (\tilde \Y_1-\tilde\Y_2)\tilde \Y_{2,\eta} (t,\eta)\\ 
& \qquad  \times \Big(\int_0^\eta (e^{-\frac{1}{\ma}(\tilde \Y_2(t,\eta)-\tilde \Y_2(t,\theta))}-e^{-\frac{1}{\ma}(\tilde \Y_1(t,\eta)-\tilde \Y_1(t,\theta))}) \min_j(\tilde \P_j)\\
&\qquad\qquad\qquad\qquad\qquad\qquad\qquad\qquad\times \min_j(\tilde \V_j^+)\tilde \Y_{2,\eta} (t,\theta) \mathbbm{1}_{B(\eta)}(t,\theta) d\theta\Big) d\eta\vert \\
& \leq \frac{1}{\ma\A^6} \int_0^1 \vert \tilde \Y_1-\tilde \Y_2\vert \tilde \Y_{2,\eta} (t,\eta) \\
& \quad \times \Big(\int_0^\eta (\vert \tilde \Y_1(t,\eta)-\tilde \Y_2(t,\eta) \vert + \vert \tilde \Y_1(t,\theta)-\tilde \Y_2(t,\theta)\vert) \\
&\qquad\qquad\qquad\qquad\qquad\times e^{-\frac{1}{\ma}(\tilde \Y_2(t,\eta)-\tilde \Y_2(t,\theta))}\min_j(\tilde\P_j)\min_j(\tilde \V_j^+)\tilde \Y_{2,\eta} (t,\theta) d\theta\Big)d\eta\\
& = \frac1{\ma\A^6} \int_0^1 (\tilde \Y_1-\tilde \Y_2)^2\tilde \Y_{2,\eta} (t,\eta)\\
&\qquad\times\Big(\int_0^\eta e^{-\frac{1}{\ma}(\tilde \Y_2(t,\eta)-\tilde \Y_2(t,\theta))} \min_j(\tilde \P_j)\min_j(\tilde \V_j^+)\tilde \Y_{2,\eta} (t,\theta) d\theta\Big) d\eta\\
& \quad + \frac{1}{\ma\A^6} \int_0^1 \vert \tilde \Y_1-\tilde \Y_2\vert \tilde \Y_{2,\eta} (t,\eta) \Big( \int_0^\eta \vert \tilde \Y_1(t,\theta)-\tilde \Y_2(t,\theta)\vert \\
&\qquad\qquad\qquad\qquad\times e^{-\frac{1}{\ma}(\tilde \Y_2(t,\eta)-\tilde \Y_2(t,\theta))} \min_j(\tilde \P_j)\min_j(\tilde \V_j^+)\tilde \Y_{2,\eta} (t,\theta) d\theta\Big)d\eta\\
& \leq \frac{1}{\sqrt{2}\A^6}\int_0^1 (\tilde \Y_1-\tilde \Y_2)^2\tilde \Y_{2,\eta}(t,\eta)\\
&\qquad\qquad\times\Big(\int_0^\eta e^{-\frac{1}{\sqtC{2}}(\tilde \Y_2(t,\eta)-\tilde \Y_2(t,\theta))} \tilde \P_2^{3/4}\tilde \U_2^+\tilde \Y_{2,\eta}(t,\theta) d\theta\Big) d\eta\\
& \quad + \frac{1}{\sqrt{2}\A^6}\int_0^1 \vert \tilde \Y_1-\tilde \Y_2\vert \tilde \Y_{2,\eta}(t,\eta)\\
& \qquad \quad \times \Big( \int_0^\eta \vert \tilde \Y_1(t,\theta)-\tilde \Y_2(t,\theta)\vert e^{-\frac{1}{\sqtC{2}}(\tilde \Y_2(t,\eta)-\tilde \Y_2(t,\theta))} \tilde \P_2^{3/4}\tilde \U_2^+\tilde \Y_{2,\eta}(t,\theta) d\theta\Big) d\eta\\
& \leq \frac{1}{\sqrt{2}\A^6}\int_0^1 (\tilde \Y_1-\tilde \Y_2)^2 \tilde \Y_{2,\eta} (t,\eta)\Big(\int_0^\eta e^{-\frac{1}{\sqtC{2}}(\tilde \Y_2(t,\eta)-\tilde \Y_2(t,\theta))} \tilde \P_2^{3/2}\tilde \Y_{2,\eta} (t,\theta) d\theta\Big)^{1/2} \\
&\qquad\qquad\qquad\qquad\qquad\times\Big(\int_0^\eta e^{-\frac{1}{\sqtC{2}}(\tilde \Y_2(t,\eta)-\tilde \Y_2(t,\theta))} \tilde \U_2^2\tilde \Y_{2,\eta} (t,\theta) d\theta\Big)^{1/2} d\eta\\
&\quad + \norm{\tilde \Y_1-\tilde \Y_2}^2 +\frac1{2\A^{12}} \int_0^1 \tilde \Y_{2,\eta}^2 (t,\eta) \\
&\qquad\qquad\qquad\times\Big(\int_0^\eta e^{-\frac{1}{\sqtC{2}}(\tilde\Y_2(t,\eta)-\tilde \Y_2(t,\theta))} \vert \tilde\Y_1-\tilde\Y_2\vert \tilde \P_2^{3/4}\tilde \V_2^+\tilde \Y_{2,\eta} (t,\theta) d\theta\Big)^2d\eta\\
& \leq \frac{3}{\A^4} \int_0^1 (\tilde\Y_1-\tilde \Y_2)^2\tilde \P_2\tilde \Y_{2,\eta} (t,\eta) d\eta + \norm{\tilde \Y_1-\tilde \Y_2}^2\\
& \quad + \frac{1}{2\A^{12}} \int_0^1 \tilde \Y_{2,\eta}^2(t,\eta)\Big(\int_0^\eta e^{-\frac{1}{\sqtC{2}}(\tilde \Y_2(t,\eta)-\tilde \Y_2(t,\theta))} \tilde \U_2^2\tilde \Y_{2,\eta} (t,\theta) d\theta\Big)\\ 
& \qquad \qquad \qquad \qquad \times\Big( \int_0^\eta  e^{-\frac{1}{\sqtC{2}}(\tilde \Y_2(t,\eta)-\tilde \Y_2(t,\theta))}
(\tilde \Y_1-\tilde \Y_2)^2 \tilde \P_2^{3/2}\tilde \Y_{2,\eta} (t,\theta) d\theta\Big)d\eta\\
& \leq \bigO(1) \norm{\tilde \Y_1-\tilde \Y_2}^2 \\
& \quad +\frac{1}{\A^6} \int_0^1 \tilde \Y_{2,\eta} (t,\eta) e^{-\frac{1}{\sqtC{2}}\tilde \Y_2(t,\eta)} \Big( \int_0^\eta e^{\frac{1}{\sqtC{2}}\tilde\Y_2(t,\theta)}(\tilde \Y_1-\tilde \Y_2)^2  \tilde \P_2^{3/2}\tilde \Y_{2,\eta} (t,\theta) d\theta\Big)\\
& \leq \bigO(1)\norm{\tilde \Y_1-\tilde\Y_2}^2\\
& \quad + \frac{1}{\A^5} \Big(-\int_0^\eta e^{-\frac{1}{\sqtC{2}}(\tilde \Y_2(t,\eta)-\tilde \Y_2(t,\theta))} (\tilde \Y_1-\tilde\Y_2)^2\tilde \P_2^{3/2}\tilde \Y_{2,\eta} (t,\theta) d\theta\Big\vert_{\theta=0}^1\\
& \qquad \qquad \qquad \qquad +\int_0^1 (\tilde \Y_1-\tilde \Y_2) ^2 \tilde \P_2^{3/2}\tilde \Y_{2,\eta} (t,\eta) d\eta\Big)\\
& \leq \bigO(1) \norm{\tilde \Y_1-\tilde \Y_2}^2,
\end{align*}
where we used  \eqref{eq:general} for $\beta=\frac12$.

Finally, we can turn our attention to $\tilde L_5$, which we need to split into several parts. We write
\begin{align*}
\tilde L_5& = \frac{1}{\A^6} \int_0^1 (\tilde \Y_1-\tilde \Y_2)(t,\eta)\\
& \qquad \qquad  \times\Big( \tilde \Y_{2,\eta} (t,\eta)\int_0^\eta \min_j(e^{-\frac{1}{\ma}(\tilde \Y_j(t,\eta)-\tilde\Y_j(t,\theta))})\min_j(\tilde\P_j)\\ 
& \qquad \qquad \qquad \qquad \qquad \qquad\qquad \qquad\times\min_j(\tilde \V_j^+) \tilde \Y_{2,\eta} (t,\theta) d\theta\\
& \qquad \qquad \qquad -\tilde \Y_{1,\eta} (t,\eta)\int_0^\eta \min_j(e^{-\frac{1}{\ma}(\tilde \Y_j(t,\eta)-\tilde\Y_j(t,\theta))})\min_j(\tilde\P_j)\\ 
& \qquad \qquad \qquad \qquad\qquad \qquad\qquad \qquad \times\min_j(\tilde\V_j^+)\tilde \Y_{1,\eta} (t,\theta) d\theta\Big)d\eta\\
& = \frac{1}{\A^6} \int_0^1 (\tilde \Y_1-\tilde \Y_2)(\tilde \Y_{2,\eta}-\tilde \Y_{1,\eta} )(t,\eta)\\
& \qquad \qquad \times \min_k\Big( \int_0^\eta \min_j(e^{-\frac{1}{\ma}(\tilde \Y_j(t,\eta)-\tilde\Y_j(t,\theta))}) \min_j(\tilde\P_j)\\ 
& \qquad \qquad \qquad \qquad \qquad \qquad\qquad \qquad\qquad \qquad\times\min_j(\tilde\V_j^+)\tilde \Y_{k,\eta} (t,\theta) d\theta\Big)d\eta \\
& \quad + \frac{1}{\A^6} \int_0^1 (\tilde \Y_1-\tilde \Y_2) \tilde \Y_{2,\eta} (t,\eta) \mathbbm{1}_{\tilde E^c}(t,\eta) \\
& \qquad \qquad \times\Big(\int_0^\eta \min_j(e^{-\frac{1}{\ma}(\tilde \Y_j(t,\eta)-\tilde\Y_j(t,\theta))})\min_j(\tilde\P_j)\\ 
& \qquad \qquad \qquad \qquad \qquad \qquad\qquad \qquad\times\ \min_j(\tilde\V_j^+)(\tilde \Y_{2,\eta}-\tilde \Y_{1,\eta})(t,\theta) d\theta\Big) d\eta\\
& \quad +\frac{1}{\A^6} \int_0^1 (\tilde\Y_1-\tilde \Y_2)\tilde \Y_{1,\eta} (t,\eta) \mathbbm{1}_{\tilde E} (t,\eta) \\
& \qquad \qquad \times\Big(\int_0^\eta \min_j(e^{-\frac{1}{\ma}(\tilde \Y_j(t,\eta)-\tilde\Y_j(t,\theta))})\min_j(\tilde\P_j) \\ 
& \qquad \qquad \qquad \qquad \qquad \qquad\qquad \qquad\times\min_j(\tilde\V_j^+)(\tilde \Y_{2,\eta}-\tilde \Y_{1,\eta})(t,\theta) d\theta\Big) d\eta\\
& = \tilde L_{51}+\tilde L_{52}+\tilde L_{53},
\end{align*}
where $\tilde E$ is given by 
\begin{align}
\tilde E=\Big\{ &(t, \eta) \mid  \int_0^\eta \min_j(e^{-\frac{1}{\ma}(\tilde \Y_j(t,\eta)-\tilde\Y_j(t,\theta))})\min_j(\tilde\P_j)\min_j(\tilde\V_j^+)\tilde\Y_{2,\eta}(t,\theta) d\theta \notag\\
& \qquad \quad\leq \int_0^\eta \min_j(e^{-\frac{1}{\ma}(\tilde \Y_j(t,\eta)-\tilde\Y_j(t,\theta))})\min_j(\tilde\P_j)\min_j(\tilde\V_j^+)\tilde\Y_{1,\eta}(t,\theta) d\theta \Big\}. \label{eq:tildeE}
\end{align}

As far as the first term $\tilde L_{51}$ is concerned, it can be estimated as follows (we use \eqref{decay:impl})
\begin{align}
 \vert \tilde L_{51}\vert 
& = \frac{1}{\A^6} \vert \int_0^1 (\tilde \Y_1-\tilde \Y_2)(\tilde \Y_{2,\eta}-\tilde \Y_{1,\eta})(t,\eta)\notag\\ \nn
& \qquad \qquad \qquad \times \min_k\Big( \int_0^\eta \min_j(e^{-\frac{1}{\ma}(\tilde \Y_j(t,\eta)-\tilde\Y_j(t,\theta))})
\min_j(\tilde\P_j)\\ 
& \qquad \qquad \qquad \qquad \qquad \qquad\qquad\qquad \qquad\times\min_j(\tilde\V_j^+)\tilde \Y_{k,\eta}(t,\theta) d\theta\Big) d\eta\vert \notag \\ \nn
& =\frac{1}{2\A^6} \Big\vert -(\tilde \Y_1-\tilde \Y_2)^2(t,\eta)\\
& \qquad \qquad \qquad \times \min_k\Big( \int_0^\eta \min_j(e^{-\frac{1}{\ma}(\tilde \Y_j(t,\eta)-\tilde\Y_j(t,\theta))})\min_j(\tilde\P_j)\notag\\ 
& \qquad \qquad \qquad \qquad \qquad \qquad\qquad \qquad\qquad\times\min_j(\tilde\V_j^+)\tilde \Y_{k,\eta}(t,\theta) d\theta\Big) \Big\vert _{\eta=0}^1 \notag\\
& \quad +\int_0^1 (\tilde \Y_1-\tilde \Y_2)^2(t,\eta)\notag \\ \nn
& \qquad \qquad  \times\frac{d}{d\eta} \min_k\Big( \int_0^\eta \min_j(e^{-\frac{1}{\ma}(\tilde \Y_j(t,\eta)-\tilde\Y_j(t,\theta))})\min_j(\tilde\P_j)\\ 
& \qquad \qquad \qquad \qquad \qquad \qquad\qquad\qquad \qquad\times\min_j(\tilde\V_j^+)\tilde \Y_{k,\eta}(t,\theta) d\theta\Big) d\eta\Big\vert\notag \\
& \quad\leq \bigO(1) \norm{\tilde \Y_1-\tilde \Y_2}^2, \label{eq:tildeL31}
\end{align}
where $\bigO(1)$ denotes some constant only depending on $\A$, which remains boundes as $\A\to 0$, provided we can show that the derivative  in the latter integral exists and is uniformly bounded.  In fact,  from Lemma \ref{lemma:2} we have that
\begin{align*}
&\frac{d}{d\eta} \min_k\Big( \int_0^\eta \min_j(e^{-\frac{1}{\ma}(\tilde \Y_j(t,\eta)-\tilde\Y_j(t,\theta))})\min_j(\tilde\P_j)
\min_j(\tilde\V_j^+)\tilde \Y_{k,\eta}(t,\theta) d\theta\Big)
\end{align*} 
exists and is uniformly bounded. 

As far as the term $\tilde L_{52}$ (a similar argument works for $\tilde L_{53}$) is concerned, the integral can be estimated as follows,
\begin{align*}
\tilde L_{52}&= \frac{1}{\A^6} \int_0^1 (\tilde \Y_1-\tilde \Y_2)\tilde \Y_{2,\eta}(t,\eta) \mathbbm{1}_{\tilde E^c} (t,\eta)\\
& \qquad \times \Big(\int_0^\eta \min_j(e^{-\frac{1}{\ma}(\tilde \Y_j(t,\eta)-\tilde\Y_j(t,\theta))})\min_j(\tilde\P_j)\\
&\qquad\qquad\qquad\qquad\qquad\qquad\qquad\qquad\times\min_j(\tilde\V_j^+) (\tilde \Y_{2,\eta}-\tilde \Y_{1,\eta}) (t,\theta) d\theta\Big)d\eta\\
& =\frac{1}{\A^6} \int_0^1 (\tilde \Y_1-\tilde \Y_2) \tilde \Y_{2,\eta}(t,\eta) \mathbbm{1}_{\tilde E^c}(t,\eta)\\
& \qquad \times\Big[(\tilde \Y_2-\tilde \Y_1)(t,\theta) \min_j(e^{-\frac{1}{\ma}(\tilde \Y_j(t,\eta)-\tilde\Y_j(t,\theta))})\min_j(\tilde\P_j)\\
&\qquad\qquad\qquad\qquad\qquad\qquad\qquad\qquad\qquad\qquad\times \min_j(\tilde\V_j^+) (t,\theta) \Big\vert_{\theta=0}^\eta\\
& \qquad \quad-\int_0^\eta (\tilde \Y_2-\tilde \Y_1)(t,\theta)\\
& \qquad\qquad \times \big[(\frac{d}{d\theta}\min_j(e^{-\frac{1}{\ma}(\tilde \Y_j(t,\eta)-\tilde\Y_j(t,\theta))})) \min_j(\tilde\P_j)
\min_j(\tilde\V_j^+)(t,\theta)\\
& \qquad \qquad\quad + \min_j(e^{-\frac{1}{\ma}(\tilde \Y_j(t,\eta)-\tilde\Y_j(t,\theta))})(\frac{d}{d\theta} \min_j(\tilde\P_j) 
\min_j(\tilde\V_j^+))(t,\theta)d\theta\big]d\eta\Big]\\
& = -\frac{1}{\A^6}\int_0^1 (\tilde\Y_1-\tilde \Y_2)^2 \min_j(\tilde\P_j)  \min_j(\tilde\V_j^+)\tilde \Y_{2,\eta} (t,\eta) \mathbbm{1}_{\tilde E^c}(t,\eta) d\eta\\
& \quad +\frac{1}{\A^6}\int_0^1 (\tilde \Y_1-\tilde \Y_2)\tilde \Y_{2,\eta} (t,\eta) \mathbbm{1}_{\tilde E^c}(t,\eta)\\
& \qquad \times\int_0^\eta (\tilde \Y_1-\tilde \Y_2)(t,\theta)\\
& \qquad\qquad \times \Big[(\frac{d}{d\theta}\min_j(e^{-\frac{1}{\ma}(\tilde \Y_j(t,\eta)-\tilde\Y_j(t,\theta))})) 
\min_j(\tilde\P_j)\min_j(\tilde\V_j^+)(t,\theta)\\
& \qquad  + \min_j(e^{-\frac{1}{\ma}(\tilde \Y_j(t,\eta)-\tilde\Y_j(t,\theta))})
\big(\frac{d}{d\theta} \min_j(\tilde\P_j)\min_j(\tilde\V_j^+)\big)(t,\theta)d\theta\Big]d\eta\\
& = M_1+M_2.
\end{align*}

As far as the first term $M_1$ is concerned, we have since 
\begin{equation*}
2\min_j(\tilde\P_j) \min_j(\tilde\V_j^+)\tilde \Y_{i,\eta}(t,\eta)\leq \A^5\min_j(\tilde\V_j^+)(t,\eta)\leq \frac{\A^7}{\sqrt{2}},
\end{equation*}
that
\begin{align*} 
\vert M_1\vert & \leq \frac{1}{\A^6} \int_0^1 (\tilde \Y_1-\tilde \Y_2)^2\min_j(\tilde\P_j) \min_j(\tilde\V_j^+)\tilde \Y_{2,\eta} (t,\eta) d\eta\\
& \leq \frac{\A}{2\sqrt{2}}\norm{\tilde \Y_1-\tilde \Y_2}^2= \bigO(1) \norm{\tilde \Y_1-\tilde \Y_2}^2,
\end{align*}
where $\bigO(1)$ denotes some constant, which only depends on $\A$ and which remains bounded as $\A\to 0$. 

The second term $M_2$, on the other hand, is a bit more demanding.
(i): First of all recall \eqref{est:L3a_lemma}, i.e., 
\begin{align*}
\vert \frac{d}{d\theta} \min_j(e^{-\frac{1}{\ma}(\tilde \Y_j(t,\eta)-\tilde\Y_j(t,\theta))})\vert
& \leq \frac{1}{\ma}\min_j(e^{-\frac{1}{\ma}(\tilde \Y_j(t,\eta)-\tilde\Y_j(t,\theta))})\max_j(\tilde \Y_{j,\eta})(t,\theta),
\end{align*}
which implies that
\begin{align*}
\frac{1}{\A^6}& \vert \int_0^1 (\tilde \Y_1-\tilde \Y_2)\tilde \Y_{2,\eta} (t,\eta) \mathbbm{1}_{\tilde E^c}(t,\eta)
 \int_0^\eta (\tilde \Y_1-\tilde \Y_2)(t,\theta)\\
& \qquad \qquad \times \big(\frac{d}{d\theta} \min_j(e^{-\frac{1}{\ma}(\tilde \Y_j(t,\eta)-\tilde\Y_j(t,\theta))})\big)\min_j(\tilde\P_j) 
 \min_j(\tilde\V_j^+)(t,\theta) d\theta d\eta\\
& \leq \norm{\tilde \Y_1-\tilde \Y_2}^2 +\frac{1}{A^{12}} \int_0^1 \tilde \Y_{2,\eta}^2(t,\eta) \mathbbm{1}_{\tilde E^c}(t,\eta)
 \Big(\int_0^\eta (\tilde \Y_1-\tilde \Y_2)(t,\theta)\\
& \quad   \times \big(\frac{d}{d\theta} \min_j(e^{-\frac{1}{\ma}(\tilde \Y_j(t,\eta)-\tilde\Y_j(t,\theta))})\big)
 \min_j(\tilde \P_j) \min_j(\tilde \V_j^+) (t,\theta) d\theta\Big)^2 d\eta\\
& \leq \norm{\tilde \Y_1-\tilde \Y_2}^2
 +\frac{1}{4\ma^2\A^2}\int_0^1 \tilde \Y_{2,\eta}^2 (t,\eta) \\
&  \qquad \times \Big( \int_0^\eta \vert \tilde \Y_1-\tilde \Y_2\vert(t,\theta) e^{-\frac{1}{\ma}(\tilde \Y_2(t,\eta)-\tilde \Y_2(t,\theta))}\min_j(\tilde\V_j^+)(t,\theta) d\theta\Big)^2 d\eta\\
& \leq \norm{\tilde \Y_1-\tilde \Y_2}^2
 + \frac{\ma^2}{8\A^2\sqtC{2}^5}\int_0^1 \tilde \Y_{2,\eta}^2 (t,\eta)\\
& \quad \times \Big( \int_0^\eta \vert \tilde \Y_1-\tilde \Y_2\vert (t,\theta) e^{-\frac{1}{\sqtC{2}}(\tilde \Y_2(t,\eta)-\tilde \Y_2(t,\theta))} \big(2\tilde \P_2\tilde \Y_{2,\eta} +\tilde \Henergy_{2,\eta}\big)^{1/2}(t,\theta) d\theta\Big)^2 d\eta\\
& \leq \norm{\tilde\Y_1-\tilde \Y_2}^2\\ 
& \quad + \frac{1}{8\sqtC{2}^5}\int_0^1 \tilde \Y_{2,\eta}^2 (t, \eta)\Big( \int_0^\eta (\tilde \Y_1-\tilde \Y_2)^2(t,\theta) e^{-\frac{1}{2\sqtC{2}}(\tilde \Y_2(t,\eta)-\tilde \Y_2(t,\theta))}d\theta\Big)\\
& \quad \qquad \qquad \qquad \times \Big( \int_0^\eta e^{-\frac{3}{2\sqtC{2}}(\tilde \Y_2(t,\eta)-\tilde \Y_2(t,\theta))}(2\tilde \P_2\tilde \Y_{2,\eta}+\tilde \Henergy_{2,\eta})(t,\theta) d\theta\Big) d\eta\\
& \leq \norm{\tilde \Y_1-\tilde \Y_2}^2\\
&\quad+ \frac{1}{\sqtC{2}^4}\int_0^1 \tilde \P_2\tilde \Y_{2,\eta}^2(t,\eta)\Big( \int_0^\eta (\tilde \Y_1-\tilde \Y_2)^2(t,\theta)e^{-\frac{1}{2\sqtC{2}}(\tilde \Y_2(t,\eta)-\tilde\Y_2(t,\theta))}d\theta \Big)d\eta\\
& \leq \norm{\tilde \Y_1-\tilde \Y_2}^2 \\
&\quad+\frac{\A}{2}\int_0^1 \tilde \Y_{2,\eta} (t,\eta) e^{-\frac{1}{2\sqtC{2}}\tilde \Y_2(t,\eta)}\Big( \int_0^\eta (\tilde \Y_1-\tilde \Y_2)^2(t,\theta) e^{\frac{1}{2\sqtC{2}}\tilde \Y_2(t,\theta)}d\theta\Big) d\eta\\
& = \norm{\tilde \Y_1-\tilde \Y_2}^2-\A\sqtC{2}\int_0^\eta e^{-\frac{1}{2\sqtC{2}}(\tilde \Y_2(t,\eta)-\tilde \Y_2(t,\theta))} (\tilde \Y_1-\tilde \Y_2)^2(t,\theta) d\theta\Big\vert_{\eta=0}^1\\
& \quad \qquad \qquad \qquad +\A\sqtC{2} \int_0^1 (\tilde \Y_1-\tilde \Y_2)^2(t,\eta) d\eta\\
& \leq \bigO(1) \norm{\tilde \Y_1-\tilde \Y_2}^2,
\end{align*}
where we used \eqref{eq:343} and \eqref{eq:Henergy32}.

(ii): First of all we have to establish that $\min_j(\tilde\P_j)\min_j(\tilde \V_j^+)(t,\theta)$ is Lipschitz continuous with a uniformly bounded Lipschitz constant.  More precisely, in Lemma \ref{lemma:3} we show that
\begin{equation*}
\vert \frac{d}{d\theta} (\min_j(\tilde \P_j)\min_j(\tilde\V_j^+))(t,\theta)\vert\leq 2\A^4 (\min_j(\tilde \P_j)^{1/2}+\vert \tilde\U_2\vert)(t,\theta).
\end{equation*}
We are now ready to establish a Lipschitz estimate for the second part of $M_2$. Indeed, 
\begin{align*}
\frac{1}{\A^6} & \vert \int_0^1 (\tilde \Y_1-\tilde\Y_2)\tilde\Y_{2,\eta}(t,\eta)\mathbbm{1}_{\tilde E^c} (t,\eta)\\
& \quad \times \int_0^\eta (\tilde \Y_1-\tilde\Y_2)(t,\theta)\min_j(e^{-\frac{1}{\ma}(\tilde \Y_j(t,\eta)-\tilde\Y_j(t,\theta))})\\
& \qquad \qquad \qquad \qquad \qquad \qquad \qquad \times 
\left(\frac{d}{d\theta} \min_j(\tilde \P_j)\min_j(\tilde\V_j^+)\right)(t,\theta) d\theta d\eta\vert\\
& \quad \leq \norm{\tilde \Y_1-\tilde\Y_2}^2
 +\frac{1}{A^{12}}\int_0^1 \tilde \Y_{2,\eta}^2(t,\eta)\mathbbm{1}_{\tilde E^c} (t,\eta)\\
& \qquad \qquad \times\Big( \int_0^\eta (\tilde\Y_1-\tilde\Y_2) (t,\theta)  \min_j(e^{-\frac{1}{\ma}(\tilde \Y_j(t,\eta)-\tilde\Y_j(t,\theta))})\\
& \qquad \qquad \qquad \qquad \qquad \qquad \qquad \times\left(\frac{d}{d\theta} \min_j(\tilde \P_j) \min_j(\tilde\V_j^+)\right) (t,\theta) d\theta\Big)^2d\eta\\
&\quad \leq \norm{\tilde \Y_1-\tilde\Y_2}^2 + \frac{4}{\A^4} \int_0^1 \tilde\Y_{2,\eta}^2(t,\eta)\\
& \qquad \qquad \times\Big(\int_0^\eta \min_j(e^{-\frac{1}{\ma}(\tilde \Y_j(t,\eta)-\tilde\Y_j(t,\theta))})\\
& \qquad \qquad \qquad \qquad \qquad \qquad \qquad\times\vert \tilde \Y_1-\tilde\Y_2\vert (\min_j(\tilde\P_j)^{1/2} +\vert \tilde \U_2\vert )(t,\theta) d\theta\Big)^2d\eta\\
& \quad\leq \norm{\tilde\Y_1-\tilde\Y_2}^2 +\frac{4(1+\sqrt{2})^2}{\A^4} \int_0^1\tilde \Y_{2,\eta}^2 (t,\eta)\\
&\qquad\qquad\qquad\qquad\times\Big( \int_0^\eta  e^{-\frac{1}{\sqtC{2}}(\tilde \Y_2(t,\eta)-\tilde \Y_2(t,\theta))}\vert \tilde \Y_1-\tilde \Y_2\vert \tilde \P_2^{1/2}(t,\theta) d\theta\Big)^2 d\eta\\
& \quad \leq\norm{\tilde\Y_1-\tilde\Y_2}^2\\
& \qquad +\frac{36}{\A^4}\int_0^1 \tilde \Y_{2,\eta}^2(t,\eta) \Big(\int_0^\eta e^{-\frac3{2\sqtC{2}}(\tilde \Y_2(t,\eta)-\tilde \Y_2(t,\theta))} \tilde \P_2(t,\theta) d\theta\Big)\\
&\qquad\qquad\qquad\qquad\qquad\times\Big( \int_0^\eta e^{-\frac1{2\sqtC{2}}(\tilde \Y_2(t,\eta)-\tilde\Y_2(t,\theta))}(\tilde \Y_1-\tilde\Y_2)^2(t,\theta) d\theta\Big)d\eta\\
&\quad \leq \norm{\tilde \Y_1-\tilde\Y_2}^2\\
& \qquad +\frac{252}{\A^4} \int_0^1 \tilde \P_2\tilde\Y_{2,\eta}^2(t,\eta)\Big(\int_0^\eta e^{-\frac1{2\sqtC{2}}(\tilde\Y_2(t,\eta)-\tilde\Y_2(t,\theta))}(\tilde\Y_1-\tilde\Y_2)^2(t,\theta) d\theta\Big) d\eta\\
& \quad\leq \norm{\tilde \Y_1-\tilde\Y_2}^2\\
& \qquad +126\A \int_0^1 \tilde\Y_{2,\eta}(t,\eta) e^{-\frac1{2\sqtC{2}}\tilde\Y_2(t,\eta)}\Big(\int_0^\eta e^{\frac1{2\sqtC{2}} \tilde \Y_2(t,\theta)}(\tilde \Y_1-\tilde \Y_2)^2(t,\theta) d\theta\Big) d\eta\\
& \quad\leq \norm{\tilde \Y_1-\tilde\Y_2}^2 \\
& \qquad + 126\A\Big( -2\sqtC{2}\int_0^\eta e^{-\frac1{2\sqtC{2}}(\tilde \Y_2(t,\eta)-\tilde\Y_2(t,\theta))}(\tilde \Y_1-\tilde\Y_2)^2(t,\theta)d\theta\Big\vert_{\eta=0}^1\\
&\qquad\qquad\qquad\qquad\qquad\qquad\qquad\qquad\qquad\qquad+\int_0^1 2\sqtC{2}(\tilde \Y_1-\tilde \Y_2)^2(t,\eta) d\eta\Big)\\
& \quad\leq \bigO(1) \norm{\tilde \Y_1-\tilde \Y_2}^2,
\end{align*}
where we used \eqref{eq:32P}.
Thus we find that
\begin{equation*}
\abs{\tilde L_{52}}+\abs{\tilde L_{53}} \le \bigO(1) \norm{\tilde \Y_1-\tilde \Y_2}^2.
\end{equation*}

We now turn our attention to $I_{33}$, i.e., 
\begin{align*}
I_{33}& = \int_0^1 (\tilde \Y_1-\tilde \Y_2)(t,\eta)\Big( \frac{1}{\sqtC{2}^6}\tilde\Y_{2,\eta}(t,\eta)\int_0^1 e^{-\frac{1}{\sqtC{2}}\vert \tilde \Y_2(t,\eta)-\tilde \Y_2(t,\theta)\vert}\tilde \Q_2\tilde \U_{2,\eta}(t,\theta) d\theta\\
& \qquad \qquad \qquad \qquad\qquad - \frac{1}{\sqtC{1}^6}\tilde\Y_{1,\eta} (t,\eta)\int_0^1 e^{-\frac{1}{\sqtC{1}}\vert \tilde \Y_1(t,\eta)-\tilde\Y_1(t,\theta)\vert} \tilde \Q_1\tilde \U_{1,\eta} (t,\theta) d\theta\Big) d\eta.
\end{align*}

As before, we apply \eqref{eq:triks}. Thus it suffices to study the following term.  
\begin{align*}
\bar I_{33}& = \int_0^1 (\tilde \Y_1-\tilde\Y_2)(t,\eta)\Big(\frac{1}{\sqtC{2}^6} \tilde \Y_{2,\eta}(t,\eta)\int_0^\eta e^{-\frac{1}{\sqtC{2}}(\tilde \Y_2(t,\eta)-\tilde\Y_2(t,\theta))} \tilde \Q_2\tilde\U_{2,\eta} (t,\theta) d\theta\\
& \qquad \qquad \qquad \qquad  -\frac{1}{\sqtC{1}^6} \tilde \Y_{1,\eta}(t,\eta)\int_0^\eta e^{-\frac{1}{\sqtC{1}}(\tilde\Y_1(t,\eta)-\tilde\Y_1(t,\theta))}\tilde \Q_1\tilde \U_{1,\eta} (t,\theta) d\theta\Big) d\eta\\
& = \frac{1}{\A^6} \int_0^1 (\tilde\Y_1-\tilde\Y_2)(t,\eta)\Big(\tilde \Y_{2,\eta}(t,\eta)\int_0^\eta e^{-\frac{1}{\sqtC{2}}(\tilde\Y_2(t,\eta)-\tilde\Y_2(t,\theta))}\tilde \Q_2\tilde\U_{2,\eta} (t,\theta) d\theta\\
&\qquad \qquad \qquad \qquad \qquad -\tilde \Y_{1,\eta} (t,\eta)\int_0^\eta e^{-\frac{1}{\sqtC{1}}(\tilde\Y_1(t,\eta)-\tilde\Y_1(t,\theta))} \tilde \Q_1\tilde \U_{1,\eta} (t,\theta) d\theta\Big) d\eta\\
& \quad +\left(\frac{1}{\sqtC{2}^6}-\frac{1}{\sqtC{1}^6}\right) \mathbbm{1}_{\sqtC{1}\leq \sqtC{2}} \int_0^1 (\tilde \Y_1-\tilde\Y_2)\tilde \Y_{1,\eta}(t,\eta)\\
&\qquad\qquad\qquad\qquad\qquad\qquad\times \Big(\int_0^\eta e^{-\frac{1}{\sqtC{1}}(\tilde\Y_1(t,\eta)-\tilde\Y_1(t,\theta))}\tilde \Q_1\tilde\U_{1,\eta} (t,\theta) d\theta\Big) d\eta\\
& \quad +\left(\frac{1}{\sqtC{2}^6}-\frac{1}{\sqtC{1}^6}\right) \mathbbm{1}_{\sqtC{2}< \sqtC{1}} \int_0^1 (\tilde\Y_1-\tilde\Y_2)\tilde \Y_{2,\eta} (t,\eta)\\
&\qquad\qquad\qquad\qquad\qquad\quad\times \Big(\int_0^\eta e^{-\frac{1}{\sqtC{2}}(\tilde\Y_2(t,\eta)-\tilde\Y_2(t,\theta))} \tilde \Q_2\tilde\U_{2,\eta} (t,\theta) d\theta\Big) d\eta\\
&= \bar K_1+\bar K_2+\bar K_3.
\end{align*}

Direct calculations yield for $\bar K_2$ (and similar for $\bar K_3$) that 
\begin{align*}
\vert \bar K_2\vert &\leq \frac{\sqtC{2}^6-\sqtC{1}^6}{\sqtC{1}^5\sqtC{2}^6}\int_0^1 \vert \tilde\Y_1-\tilde\Y_2\vert \tilde \Y_{1,\eta}(t,\eta)\Big( \int_0^\eta e^{-\frac{1}{\sqtC{1}}(\tilde \Y_1(t,\eta)-\tilde\Y_1(t,\theta))} \tilde \P_1\vert \tilde\U_{1,\eta} \vert (t,\theta) d\theta\Big) d\eta\\
& \leq \frac{\sqtC{2}^6-\sqtC{1}^6}{\sqtC{1}^5\sqtC{2}^6}\norm{\tilde\Y_1-\tilde\Y_2}\\
&\qquad\qquad\times\Big(\int_0^1 \tilde\Y_{1,\eta}^2(t,\eta)\Big(\int_0^\eta e^{-\frac{1}{\sqtC{1}}(\tilde\Y_1(t,\eta)-\tilde\Y_1(t,\theta))}\tilde \P_1\vert \tilde \U_{1,\eta}\vert (t,\theta) d\theta\Big)^2 d\eta\Big)^{1/2}\\
& \leq \frac{\sqtC{2}^6-\sqtC{1}^6}{\sqtC{1}^6\sqtC{2}^6} \norm{\tilde\Y_1-\tilde\Y_2}\Big(\int_0^1 \tilde \Y_{1,\eta}^2(t,\eta) \Big(\int_0^\eta e^{-\frac{1}{\sqtC{1}}(\tilde \Y_1(t,\eta)-\tilde \Y_1(t,\theta))} \tilde \P_1^2\tilde\Y_{1,\eta} (t,\theta) d\theta\Big)\\
&\quad \qquad \qquad\qquad\qquad\qquad \qquad \times\Big( \int_0^\eta e^{-\frac{1}{\sqtC{1}}(\tilde \Y_1(t,\eta)-\tilde\Y_1(t,\theta))}\tilde \Henergy_{1,\eta} (t,\theta)d\theta\Big) d\eta\Big)^{1/2}\\
& \leq \sqrt{6}\frac{\sqtC{2}^6-\sqtC{1}^6}{\sqtC{1}^3\sqtC{2}^6} \norm{\tilde\Y_1-\tilde\Y_2} \Big(\int_0^1 \tilde \P_1^2\tilde\Y_{1,\eta}^2(t,\eta) d\eta\Big)^{1/2}\\
& \leq \frac{\sqrt{6}}{2} \frac{\sqtC{2}^6-\sqtC{1}^6}{\A^4} \norm{\tilde \Y_1-\tilde \Y_2}\\
& \leq \bigO(1)\Big(\norm{\tilde\Y_1-\tilde\Y_2}^2+ \vert \sqtC{1}-\sqtC{2}\vert^2\Big).
\end{align*}
Here we used \eqref{eq:all_estimatesM}, \eqref{eq:all_PestimatesB}, and \eqref{eq:all_PestimatesD}.

$\bar K_1$, on the other hand, needs to be rewritten a bit more.   
Recall \eqref{eq:DogP}. Then we can write
\begin{align} \label{eq:barK1_1}
\bar K_1&=\frac{1}{\A^6}\int_0^1 (\tilde\Y_1-\tilde\Y_2)(t,\eta)\Big(\tilde\Y_{2,\eta}(t,\eta)\int_0^\eta e^{-\frac{1}{\sqtC{2}}(\tilde\Y_2(t,\eta)-\tilde\Y_2(t,\theta))}\tilde\Q_2\tilde\U_{2,\eta}(t,\theta) d\theta\\
& \qquad \qquad \qquad \qquad \qquad -\tilde\Y_{1,\eta}(t,\eta)\int_0^\eta e^{-\frac{1}{\sqtC{1}}(\tilde\Y_1(t,\eta)-\tilde\Y_1(t,\theta))} \tilde\Q_1\tilde\U_{1,\eta}(t,\theta)d\theta\Big)d\eta\notag\\
&= \frac{1}{\A^6}\int_0^1 (\tilde\Y_1-\tilde\Y_2)(t,\eta)\Big(\sqtC{2}\tilde\Y_{2,\eta}(t,\eta)\int_0^\eta e^{-\frac{1}{\sqtC{2}}(\tilde\Y_2(t,\eta)-\tilde\Y_2(t,\theta))}\tilde \P_2\tilde\U_{2,\eta}(t,\theta)d\theta\notag\\
& \qquad \qquad \qquad \qquad\qquad -\sqtC{1}\tilde \Y_{1,\eta}(t,\eta) \int_0^\eta e^{-\frac{1}{\sqtC{1}}(\tilde \Y_1(t,\eta)-\tilde\Y_1(t,\theta))}\tilde \P_1\tilde\U_{1,\eta}(t,\theta)d\theta\Big)d\eta\notag\\
& \quad +\frac{1}{\A^6} \int_0^1(\tilde\Y_1-\tilde\Y_2)(t,\eta)\Big(\tilde\Y_{1,\eta}(t,\eta)\int_0^\eta e^{-\frac{1}{\sqtC{1}}(\tilde\Y_1(t,\eta)-\tilde\Y_1(t,\theta))}\tilde \D_1\tilde\U_{1,\eta}(t,\theta) d\theta\notag\\
&\qquad \qquad \qquad \qquad\qquad -\tilde \Y_{2,\eta}(t,\eta)\int_0^\eta e^{-\frac{1}{\sqtC{2}}(\tilde\Y_2(t,\eta)-\tilde\Y_2(t,\theta))}\tilde\D_2\tilde\U_{2,\eta}(t,\theta)d\theta\Big) d\eta.\notag
\end{align}
In the next step we are applying integration by parts, so that we get rid of $\tilde\U_{j,\eta}(t,\theta)$ in the integrands and hence can therefore have a splitting into positive and negative parts.
Therefore note that for $i=1$, $2$, we have
\begin{align*}
\int_0^\eta e^{-\frac{1}{\sqtC{i}}(\tilde\Y_i(t,\eta)-\tilde\Y_i(t,\theta))}&\tilde \P_i\tilde\U_{i,\eta}(t,\theta)d\theta\\
& = e^{-\frac{1}{\sqtC{i}}(\tilde \Y_i(t,\eta)-\tilde\Y_i(t,\theta))} \tilde \P_i\tilde\U_i(t,\theta)\Big\vert_{\theta=0}^{\eta} \\
& \quad-\int_0^\eta e^{-\frac{1}{\sqtC{i}}(\tilde\Y_i(t,\eta)-\tilde\Y_i(t,\theta))}(\frac{1}{\sqtC{i}}\tilde\P_i+\frac{1}{\sqtC{i}^2}\tilde\Q_i)\tilde\U_i\tilde\Y_{i,\eta}(t,\theta)d\theta\\
& = \tilde \P_i\tilde\U_i(t,\eta)\\
& \quad -\int_0^\eta e^{-\frac{1}{\sqtC{i}}(\tilde\Y_i(t,\eta)-\tilde\Y_i(t,\theta))}(\frac{2}{\sqtC{i}}\tilde\P_i-\frac{1}{\sqtC{i}^2}\tilde\D_i)\tilde\U_i\tilde\Y_{i,\eta}(t,\theta)d\theta\\
& =\tilde\P_i\tilde\U_i(t,\eta)\\
& \quad -\frac{2}{\sqtC{i}}\int_0^\eta e^{-\frac{1}{\sqtC{i}}(\tilde\Y_i(t,\eta)-\tilde\Y_i(t,\theta))}\tilde\P_i\tilde\U_i\tilde\Y_{i,\eta}(t,\theta)d\theta\\
& \quad + \frac{1}{\sqtC{i}^2}\int_0^\eta e^{-\frac{1}{\sqtC{i}}(\tilde \Y_i(t,\eta) -\tilde\Y_i(t,\theta))} \tilde \D_i\tilde\U_i\tilde\Y_{i,\eta} (t,\theta)d\theta,
\end{align*}
and 
\begin{align*}
\int_0^\eta &e^{-\frac{1}{\sqtC{i}}(\tilde\Y_i(t,\eta)-\tilde\Y_i(t,\theta))} \tilde \D_i\tilde \U_{i,\eta}(t,\theta)d\theta\\ 
& = \int_0^\eta e^{-\frac{1}{\sqtC{i}}(\tilde \Y_i(t,\eta)-\tilde\Y_i(t,\theta))} \\
&\quad\times\Big(\int_0^\theta e^{-\frac{1}{\sqtC{i}}(\tilde\Y_i(t,\theta)-\tilde\Y_i(t,l))}((\tilde\U_i^2-\tilde\P_i)\tilde\Y_{i,\eta}(t,l)+\frac12 \sqtC{i}^5) dl\Big)\tilde\U_{i,\eta}(t,\theta) d\theta\\
& = \int_0^\eta \Big(\int_0^\theta e^{-\frac{1}{\sqtC{i}}(\tilde\Y_i(t,\eta)-\tilde\Y_i(t,l))} ((\tilde\U_i^2-\tilde\P_i)\tilde\Y_{i,\eta}(t,l)+\frac12 \sqtC{i}^5)dl\Big) \tilde \U_{i,\eta}(t,\theta) d\theta\\
& = \Big(\int_0^\theta e^{-\frac{1}{\sqtC{i}}(\tilde\Y_i(t,\eta)-\tilde\Y_i(t,l))} ((\tilde\U_i^2-\tilde\P_i)\tilde\Y_{i,\eta}(t,l)+\frac12 \sqtC{i}^5)dl\Big) \tilde \U_i(t,\theta) \Big\vert_{\theta=0}^{\eta}\\
&\quad -\int_0^\eta e^{-\frac{1}{\sqtC{i}}(\tilde \Y_i(t,\eta)-\tilde\Y_i(t,\theta))}((\tilde\U_i^2-\tilde\P_i)\tilde\Y_{i,\eta}(t,\theta)+\frac12 \sqtC{i}^5)\tilde \U_i(t,\theta)d\theta\\
& = \tilde \D_i\tilde\U_i(t,\eta)\\
&\quad -\int_0^\eta e^{-\frac{1}{\sqtC{i}}(\tilde \Y_i(t,\eta)-\tilde\Y_i(t,\theta))}((\tilde\U_i^2-\tilde\P_i)\tilde\Y_{i,\eta}(t,\theta)+\frac12\sqtC{i}^5)\tilde \U_i(t,\theta)d\theta.
\end{align*}
Here we used that for $\theta\leq \eta$,
\begin{align*}
0&\leq \int_0^\theta  e^{-\frac{1}{\sqtC{i}}(\tilde\Y_i(t,\eta)-\tilde\Y_i(t,l))}((\tilde\U_i^2-\tilde\P_i)\tilde\Y_{i,\eta}(t,l)+\frac12 \sqtC{i}^5)dl\\
& \leq \int_0^\eta e^{-\frac{1}{\sqtC{i}}(\tilde\Y_i(t,\eta)-\tilde\Y_i(t,l))}((\tilde\U_i^2-\tilde\P_i)\tilde\Y_{i,\eta}(t,l)+\frac12 \sqtC{i}^5)dl\leq \tilde\D_i(t,\eta)\leq 2\sqtC{i}\tilde\P_i(t,\eta).
\end{align*}
We finally end up with 
\begin{align}\label{eq:barK1_2}
\bar K_1&= \frac{1}{\A^6} \int_0^1(\tilde\Y_1-\tilde\Y_2)(\sqtC{2}\tilde\P_2\tilde\U_2\tilde\Y_{2,\eta}-\sqtC{1}\tilde\P_1\tilde\U_1\tilde\Y_{1,\eta})(t,\eta)d\eta\\
& \quad +\frac{1}{\A^6} \int_0^1 (\tilde\Y_1-\tilde\Y_2)(\tilde \D_1\tilde\U_1\tilde\Y_{1,\eta}-\tilde\D_2\tilde\U_2\tilde\Y_{2,\eta})(t,\eta)d\eta\notag\\  \nn
& \quad +\frac{1}{\A^6} \int_0^1 (\tilde \Y_1-\tilde\Y_2)(t,\eta)\\ \nn
&\qquad\qquad \qquad\times\Big(\tilde\Y_{2,\eta} (t,\eta) \int_0^\eta e^{-\frac{1}{\sqtC{2}}(\tilde \Y_2(t,\eta)-\tilde\Y_2(t,\theta))}\frac{1}{\sqtC{2}}\tilde\D_2\tilde\U_2\tilde\Y_{2,\eta} (t,\theta) d\theta\notag\\
& \qquad \qquad \qquad - \tilde\Y_{1,\eta} (t,\eta) \int_0^\eta e^{-\frac{1}{\sqtC{1}}(\tilde\Y_1(t,\eta)-\tilde\Y_1(t,\theta))}\frac{1}{\sqtC{1}}\tilde\D_1\tilde\U_1\tilde\Y_{1,\eta}(t,\theta)d\theta\Big)d\eta\notag\\
& \quad +\frac{3}{\A^6} \int_0^1(\tilde\Y_1-\tilde\Y_2)(t,\eta) \Big(\tilde\Y_{1,\eta} (t,\eta)\int_0^\eta e^{-\frac{1}{\sqtC{1}}(\tilde\Y_1(t,\eta)-\tilde\Y_1(t,\theta))}\tilde \P_1\tilde\U_1\tilde\Y_{1,\eta}(t,\theta) d\theta\notag\\
& \qquad \qquad \qquad \qquad \qquad -\tilde\Y_{2,\eta}(t,\eta)\int_0^\eta e^{-\frac{1}{\sqtC{2}}(\tilde \Y_2(t,\eta)-\tilde\Y_2(t,\theta))}\tilde\P_2\tilde\U_2\tilde\Y_{2,\eta}(t,\theta) d\theta\Big) d\eta\notag\\
& \quad +\frac{1}{\A^6} \int_0^1(\tilde\Y_1-\tilde\Y_2)(t,\eta) \Big(\tilde\Y_{2,\eta} (t,\eta) \int_0^\eta e^{-\frac{1}{\sqtC{2}}(\tilde\Y_2(t,\eta)-\tilde\Y_2(t,\theta))}\tilde\U_2^3\tilde\Y_{2,\eta}(t,\theta)d\theta\notag\\
& \qquad \qquad \qquad \qquad \qquad -\tilde\Y_{1,\eta}(t,\eta)\int_0^\eta e^{-\frac{1}{\sqtC{1}}(\tilde \Y_1(t,\eta)-\tilde\Y_1(t,\theta))} \tilde \U_1^3\tilde\Y_{1,\eta}(t,\theta)d\theta\Big) d\eta\notag\\
& \quad +\frac{1}{2\A^6} \int_0^1(\tilde\Y_1-\tilde\Y_2)(t,\eta) \Big(\tilde\Y_{2,\eta}(t,\eta)\int_0^\eta e^{-\frac{1}{\sqtC{2}}(\tilde\Y_2(t,\eta)-\tilde\Y_2(t,\theta))}\sqtC{2}^5\tilde\U_{2}(t,\theta)d\theta\notag\\
& \qquad \qquad \qquad \qquad \qquad -\tilde\Y_{1,\eta}(t,\eta) \int_0^\eta e^{-\frac{1}{\sqtC{1}}(\tilde \Y_1(t,\eta)-\tilde\Y_1(t,\theta))}\sqtC{1}^5\tilde\U_1(t,\theta)d\theta\Big) d\eta\notag\\
& = \bar K_{11}+\bar K_{12}+\bar K_{13}+\bar K_{14}+\bar K_{15}+\bar K_{16}.\notag
\end{align}

We start by considering $\bar K_{11}$, which can be further split into 
\begin{align*}
\bar K_{11}& = \frac{1}{\A^6} \int_0^1(\tilde\Y_1-\tilde\Y_2)(\sqtC{2}\tilde\P_2\tilde\U_2\tilde\Y_{2,\eta}-\sqtC{1}\tilde\P_1\tilde\U_1\tilde\Y_{1,\eta})(t,\eta)d\eta\\
& = \frac{1}{\A^6} \mathbbm{1}_{\sqtC{1}\leq \sqtC{2}}(\sqtC{2}-\sqtC{1})\int_0^1 (\tilde \Y_1-\tilde \Y_2)\tilde \P_2\tilde\U_2\tilde \Y_{2,\eta}(t,\eta) d\eta\\
& \quad + \frac{1}{\A^6} \mathbbm{1}_{\sqtC{2}<\sqtC{1}}(\sqtC{2}-\sqtC{1})\int_0^1 (\tilde \Y_1-\tilde \Y_2) \tilde \P_1\tilde \U_1\tilde \Y_{1,\eta}(t,\eta) d\eta\\
& \quad +\frac{\ma}{\A^6} \int_0^1 (\tilde \Y_1-\tilde\Y_2)(\tilde\P_2-\tilde\P_1)\tilde\U_2\tilde\Y_{2,\eta}\mathbbm{1}_{\tilde\P_1\leq\tilde\P_2}(t,\eta)d\eta\\
& \quad +\frac{\ma}{\A^6}\int_0^1 (\tilde\Y_1-\tilde\Y_2)(\tilde\P_2-\tilde\P_1)\tilde\U_1\tilde\Y_{1,\eta}\mathbbm{1}_{\tilde\P_2<\tilde\P_1}(t,\eta)d\eta\\
& \quad +\frac{\ma}{\A^6}\int_0^1 (\tilde\Y_1-\tilde\Y_2)(\tilde\Y_{2,\eta}-\tilde\Y_{1,\eta})\min_j(\tilde\P_j)\tilde\U_{2}(t,\eta)d\eta\\
& \quad +\frac{\ma}{\A^6}\int_0^1(\tilde\Y_1-\tilde\Y_2)(\tilde\U_2-\tilde\U_1)\min_j(\tilde\P_j)\tilde\Y_{1,\eta}(t,\eta)d\eta\\
&= \bar B_{11}+\bar B_{12}+\bar B_{13}+\bar B_{14}+\bar B_{15}+\bar B_{16}.
\end{align*}

For $\bar B_{11}$ we have (and similarly for $B_{12}$) that
\begin{align*}
\vert \bar B_{11}\vert & = \frac{1}{\A^6} \mathbbm{1}_{\sqtC{1}\leq \sqtC{2}} (\sqtC{2}-\sqtC{1}) \vert \int_0^1 (\tilde \Y_1-\tilde \Y_2)\tilde \P_2\tilde \U_2\tilde \Y_{2,\eta}(t,\eta) d\eta\vert \\
& \leq \frac{\A}{2\sqrt{2}}\vert \sqtC{1}-\sqtC{2}\vert \int_0^1 \vert \tilde \Y_1-\tilde \Y_2\vert (t,\eta) d\eta\\
& \leq \bigO(1)(\norm{\tilde \Y_1-\tilde \Y_2}^2 +\vert \sqtC{1}-\sqtC{2}\vert).
\end{align*}

Recalling \eqref{eq:all_estimatesI}, we have for $\bar B_{13}$ (and similarly for $\bar B_{14}$) that 
\begin{align*}
\vert \bar B_{13}\vert &= \frac{\ma}{\A^6} \vert \int_0^1 (\tilde \Y_1-\tilde\Y_2)(\tilde\P_2-\tilde\P_1)\tilde\U_2\tilde\Y_{2,\eta}\mathbbm{1}_{\tilde\P_1\leq \tilde\P_2}(t,\eta) d\eta\vert\\
& \leq \frac{2}{\A^5} \int_0^1\vert \tilde\Y_1-\tilde\Y_2\vert\, \vert \sqP{2}-\sqP{1}\vert \sqP{2}\vert\tilde\U_2\vert\tilde\Y_{2,\eta}(t,\eta) d\eta\\
& \leq \sqrt{2} \int_0^1 \vert \tilde\Y_1-\tilde\Y_2\vert\,\vert \sqP{2}-\sqP{1}\vert (t,\eta)d\eta\\
& \leq \bigO(1)\Big(\norm{\tilde\Y_1-\tilde\Y_2}^2+\norm{\sqP{2}-\sqP{1}}^2\Big).
\end{align*}

As far as $\bar B_{15}$ is concerned, we want to use integration by parts. Therefore it is important to recall \eqref{decay:impl}, and 
Lemma \ref{lemma:3}(ii), which imply that 
\begin{align}\label{eq:barB13}
\vert \bar B_{15}\vert & =\frac{\ma}{\A^6}\vert \int_0^1 (\tilde\Y_1-\tilde\Y_2)(\tilde\Y_{1,\eta}-\tilde\Y_{2,\eta})\min_j(\tilde\P_j)\tilde\U_{2}(t,\eta)d\eta\vert\\
& = \frac{\ma}{2\A^6}\vert (\tilde\Y_1-\tilde\Y_2)^2 \min_j(\tilde\P_j)\tilde\U_2(t,\eta) \Big\vert_{\eta=0}^1\notag\\
& \qquad -\int_0^1 (\tilde\Y_1-\tilde\Y_2)^2 \frac{d}{d\eta}(\min_j(\tilde\P_j)\tilde\U_2)(t,\eta) d\eta\vert\notag\\
& = \frac{\ma}{2\A^6}\vert \int_0^1 (\tilde\Y_1-\tilde\Y_2)^2 \frac{d}{d\eta}(\min_j(\tilde\P_j)\tilde\U_2)(t,\eta) d\eta\vert\notag\\
& \leq \frac{\A}{2\sqrt{2}} \norm{\tilde\Y_1-\tilde\Y_2}^2.\notag
\end{align} 

$\bar B_{16}$ is straightforward. Indeed, one has using \eqref{eq:all_estimatesE} that 
\begin{align*}
\vert\bar B_{16}\vert&=\frac{\ma}{\A^6}\vert \int_0^1 (\tilde\Y_1-\tilde\Y_2)(\tilde\U_1-\tilde\U_2)\min_j(\tilde\P_j)\tilde\Y_{1,\eta}(t,\eta)d\eta\vert\\
& \leq \frac{1}{2}\int_0^1\vert \tilde\Y_1-\tilde\Y_2\vert \vert \tilde\U_1-\tilde\U_2\vert (t,\eta)d\eta\\
& \leq \norm{\tilde\Y_1-\tilde\Y_2}^2+\norm{\tilde\U_1-\tilde\U_2}.
\end{align*}

We continue with $\bar K_{12}$. We can split $\bar K_{12}$ as follows
\begin{align*}
\bar K_{12}&=  \frac1{\A^6} \int_0^1 (\tilde \Y_1-\tilde\Y_2)(\tilde \D_1\tilde\U_1\tilde\Y_{1,\eta}-\tilde\D_2\tilde\U_2\tilde\Y_{2,\eta})(t,\eta) d\eta\\
&=\frac{1}{\A^6}\int_0^1 (\tilde \Y_1-\tilde\Y_2) (\tilde \D_1-\tilde\D_2) \tilde \U_1\tilde \Y_{1,\eta} \mathbbm{1}_{\tilde \D_2\leq \tilde\D_1}(t,\eta)d\eta\\
& \quad +\frac{1}{\A^6} \int_0^1 (\tilde\Y_1-\tilde\Y_2) (\tilde \D_1-\tilde\D_2) \tilde \U_2\tilde\Y_{2,\eta} \mathbbm{1}_{\tilde\D_1<\tilde\D_2}(t,\eta) d\eta\\
& \quad +\frac{1}{\A^6} \int_0^1 (\tilde \Y_1-\tilde\Y_2)(\tilde\U_1-\tilde\U_2) \min_j(\tilde\D_j) \tilde\Y_{1,\eta}(t,\eta) d\eta\\
& \quad + \frac{1}{\A^6} \int_0^1 (\tilde\Y_1-\tilde\Y_2) (\tilde \Y_{1,\eta}-\tilde\Y_{2,\eta}) \min_j(\tilde\D_j) \tilde \U_2(t,\eta) d\eta\\
& = \bar B_{21}+\bar B_{22}+\bar B_{23}+\bar B_{24}.
\end{align*}

As a first step we need to establish an estimate for $(\tilde\D_1-\tilde\D_2)(t,\eta)$. Note that $\vert \tilde\U_i\tilde\Y_{i,\eta}(t,\eta)\vert \leq \A^{5/2}\sqrt{\tilde\Y_{i,\eta}}(t,\eta)$, and hence it is in general unbounded. Thus our estimate for $(\tilde\D_1-\tilde\D_2)(t,\eta)$ must take care of this problem.

Applying Lemma \ref{lemma:D} finally yields
\begin{align*}
\vert \bar B_{21}\vert  & \leq \frac{1}{\A^6} \int_0^1\vert\tilde\Y_1-\tilde\Y_2\vert\, \vert\tilde\D_1-\tilde\D_2\vert\, \vert \tilde \U_1\vert \tilde\Y_{1,\eta}(t,\eta) \mathbbm{1}_{\tilde\D_2\leq \tilde\D_1}(t,\eta)d\eta\\
& \leq \frac2{\A^{9/2}} \int_0^1 (\tilde\Y_1-\tilde\Y_2)^2 \tilde\D_1^{1/2}\vert \tilde\U_1\vert \tilde\Y_{1,\eta} (t,\eta)d\eta\\
&\quad +\frac1{\A^6}\int_0^1 (\tilde\Y_1-\tilde\Y_2)^2(\tilde \U_1^2+\tilde\P_1)\vert \tilde\U_1\vert \tilde\Y_{1,\eta} (t,\eta) d\eta\\
& \quad+ \frac{2\sqrt{2}}{\A^{9/2}} \norm{\tilde \Y_1-\tilde\Y_2}\int_0^1 \vert \tilde\Y_1-\tilde\Y_2\vert \tilde\D_1^{1/2}\vert \tilde\U_1\vert \tilde\Y_{1,\eta} (t,\eta) d\eta\\
& \quad+ \frac4{\A^3} \int_0^1\vert \tilde \Y_1-\tilde\Y_2\vert(t,\eta) \\
&\qquad\qquad\times\Big(\int_0^\eta e^{-\frac{1}{\A}(\tilde \Y_1(t,\eta)-\tilde\Y_1(t,\theta))}(\tilde \U_1-\tilde\U_2)^2(t,\theta) d\theta\Big)^{1/2}\vert \tilde \U_1\vert\tilde\Y_{1,\eta}(t,\eta)d\eta\\
& \quad+\frac{2\sqrt{2}}{\A^3} \int_0^1\vert \tilde \Y_1-\tilde\Y_2\vert(t,\eta) \\
&\qquad\qquad\times\Big(\int_0^\eta e^{-\frac{1}{\A}(\tilde \Y_1(t,\eta)-\tilde\Y_1(t,\theta))}(\sqP{1}-\sqP{2})^2(t,\theta) d\theta\Big)^{1/2}\vert \tilde \U_1\vert\tilde\Y_{1,\eta}(t,\eta)d\eta\\
& \quad+\frac{3}{\sqrt{2}\A^2} \int_0^1\vert \tilde \Y_1-\tilde\Y_2\vert (t,\eta)\\
&\qquad\qquad\qquad\times\Big(\int_0^\eta e^{-\frac{1}{\ma}(\tilde \Y_1(t,\eta)-\tilde\Y_1(t,\theta))}(\tilde \Y_1-\tilde\Y_2)^2(t,\theta) d\theta\Big)^{1/2}\vert \tilde \U_1\vert\tilde\Y_{1,\eta}(t,\eta)d\eta\\
& \quad +\frac{12\sqrt{2}}{\sqrt{3}e\A^2}\vert \sqtC{1}-\sqtC{2}\vert \int_0^1 \vert \tilde \Y_1-\tilde \Y_2\vert (t,\eta)\\&\qquad\qquad\qquad\times\Big(\int_0^\eta e^{-\frac{3}{4\A}(\tilde \Y_1(t,\eta)-\tilde \Y_1(t,\theta))} d\theta\Big)^{1/2} \vert \tilde \U_1\vert \tilde \Y_{1,\eta}(t,\eta) d\eta\\
& \quad + \frac{3}{2\A^2} \int_0^1 \vert \tilde\Y_1-\tilde \Y_2\vert (t,\eta)\\
&\qquad\qquad\times\Big(\int_0^\eta e^{-\frac{1}{\ma}(\tilde \Y_1(t,\eta)-\tilde \Y_1(t,\theta))}\vert \tilde \Y_1-\tilde \Y_2\vert (t,\theta) d\theta\Big) \vert \tilde \U_1\vert \tilde \Y_{1,\eta} (t,\eta) d\eta\\
&\quad+\frac{6}{\A^2} \vert \sqtC{1}-\sqtC{2}\vert \int_0^1\vert \tilde \Y_1-\tilde\Y_2\vert(t,\eta)\\
&\qquad\qquad\times \Big(\int_0^\eta e^{-\frac{3}{4\A}(\tilde \Y_1(t,\eta)-\tilde\Y_1(t,\theta))} d\theta\Big)\vert \tilde \U_1\vert\tilde\Y_{1,\eta}(t,\eta)d\eta\\
& \leq \bigO(1) \norm{\tilde\Y_1-\tilde\Y_2}^2
 + \frac{16}{\A}\int_0^1 \tilde \Y_{1,\eta}(t,\eta) e^{-\frac{1}{\A}\tilde\Y_1(t,\eta)} \\
&\quad \times\Big(\int_0^\eta e^{\frac{1}{\A}\tilde \Y_1(t,\theta)} ((\tilde\U_1-\tilde\U_2)^2+ (\sqP{1}-\sqP{2})^2+\A(\tilde \Y_1-\tilde\Y_2)^2)(t,\theta) d\theta\Big) d\eta\\
& \quad + 12\A^{1/2}\vert \sqtC{1}-\sqtC{2}\vert \norm{\tilde \Y_1-\tilde \Y_2}\\
&\qquad\qquad\times\Big(\int_0^1 \tilde \Y_{1,\eta}(t,\eta) e^{-\frac{3}{4\A}\tilde \Y_1(t,\eta)}\Big(\int_0^\eta e^{\frac{3}{4\A}\tilde \Y_1(t,\theta)}d\theta\Big) d\eta\Big)^{1/2}\\
&\quad+ \frac{3\A^{1/2}}{2}\Big(\int_0^1 \tilde\Y_{1,\eta}(t,\eta) e^{-\frac{2}{\ma}\tilde\Y_1(t,\eta)}\Big( \int_0^\eta e^{\frac{2}{\ma}\tilde\Y_1(t,\theta)}d\theta\Big) d\eta\Big)^{1/2} \norm{\tilde \Y_1-\tilde\Y_2}^2\\
&\quad+ 6\A^{1/2} \vert \sqtC{1}-\sqtC{2}\vert \norm{\tilde \Y_1-\tilde\Y_2} \\
&\qquad\qquad\qquad\qquad\times \Big(\int_0^1 \tilde \Y_{1,\eta}(t,\eta)e^{-\frac{3}{4\A}\tilde\Y_1(t,\eta)}\Big(\int_0^\eta e^{\frac{3}{4A}\tilde\Y_1(t,\theta)}d\theta\Big)d\eta\Big)^{1/2}\\
& \leq \bigO(1)(\norm{\tilde \Y_1-\tilde\Y_2}^2+\norm{\tilde\U_1-\tilde\U_2}^2+\norm{\sqP{1}-\sqP{2}}^2+\vert \sqtC{1}-\sqtC{2}\vert^2).
\end{align*}
Following the same lines, one obtains
\begin{equation*}
\vert \bar B_{22}\vert\leq \bigO(1)\Big(\norm{\tilde \Y_1-\tilde\Y_2}^2+\norm{\tilde\U_1-\tilde\U_2}^2+\norm{\sqP{1}-\sqP{2}}^2+\vert \sqtC{1}-\sqtC{2}\vert^2\Big).
\end{equation*}

The estimate for $\bar B_{23}$ is straightforward. Namely,
\begin{align*}
\vert \bar B_{23}\vert &\leq\frac{2}{\A^5} \int_0^1 \vert \tilde\Y_1-\tilde\Y_2\vert \vert \tilde\U_1-\tilde\U_2\vert \tilde \P_1\tilde\Y_{1,\eta} (t,\eta) d\eta\\
& \leq \norm{\tilde \Y_1-\tilde\Y_2}^2+\norm{\tilde \U_1-\tilde\U_2}^2.
\end{align*}

As far as $\bar B_{24}$ is concerned, recall from Lemma \ref{lemma:5} (ii) that
\begin{align*}
\vert \frac{d}{d\eta} \big(\min_j(\tilde \D_j)\tilde \U_2(t,\eta)\big) \vert & \leq\bigO(1)\A^7,
\end{align*}
which yields, together with \eqref{decay:impl}, that
\begin{align*}
\vert \bar B_{24}\vert &=  \frac{1}{\A^6}\vert \int_0^1 (\tilde \Y_1-\tilde\Y_2)(\tilde\Y_{1,\eta}-\tilde\Y_{2,\eta})\min_j(\tilde\D_j)\tilde\U_2(t,\eta) d\eta\vert\\
& =\frac1{2\A^6} \vert (\tilde \Y_1-\tilde\Y_2)^2 \min_j(\tilde\D_j)\tilde \U_2(t,\eta)\Big\vert_{\eta=0}^1 \\
&\qquad\qquad\qquad\qquad-\int_0^1 (\tilde\Y_1-\tilde\Y_2)^2(t,\eta) \frac{d}{d\eta} (\min_j(\tilde\D_j)\tilde\U_2)(t,\eta) d\eta\vert \\
& = \frac{1}{2\A^6}\vert \int_0^1 (\tilde \Y_1-\tilde\Y_2)^2(t,\eta) \frac{d}{d\eta} (\min_j(\tilde \D_j)\tilde\U_2)(t,\eta) d\eta\vert \\
& \leq \bigO(1) \norm{\tilde \Y_1-\tilde\Y_2}^2.
\end{align*}

Next, we have a look at $\bar K_{13}$, which can be rewritten as follows
\begin{align*}
\bar K_{13}&= \frac{1}{\A^6}\int_0^1 (\tilde\Y_1-\tilde\Y_2)(t,\eta)\\
&\qquad\times\Big(\tilde\Y_{2,\eta}(t,\eta) \int_0^\eta e^{-\frac{1}{\sqtC{2}}(\tilde\Y_2(t,\eta)-\tilde\Y_2(t,\theta))}\frac{1}{\sqtC{2}}\tilde\D_2\tilde\U_2\tilde\Y_{2,\eta}(t,\theta) d\theta\\
& \qquad \qquad\qquad  -\tilde\Y_{1,\eta}(t,\eta) \int_0^\eta e^{-\frac{1}{\sqtC{1}}(\tilde\Y_1(t,\eta)-\tilde\Y_1(t,\theta))}\frac{1}{\sqtC{1}}\tilde\D_1\tilde\U_1\tilde\Y_{1,\eta}(t,\theta) d\theta\Big)d\eta\\
& = \frac{1}{\A^6} \int_0^1 (\tilde\Y_1-\tilde\Y_2)(t,\eta)\\
&\qquad\times\Big(\tilde\Y_{2,\eta}(t,\eta)\int_0^\eta e^{-\frac{1}{\sqtC{2}}(\tilde\Y_2(t,\eta)-\tilde\Y_2(t,\theta))}\frac{1}{\sqtC{2}}\tilde\D_2\tilde\V_2^+\tilde\Y_{2,\eta}(t,\theta) d\theta\\
& \qquad\qquad \qquad  -\tilde\Y_{1,\eta}(t,\eta)\int_0^\eta e^{-\frac{1}{\sqtC{1}}(\tilde\Y_1(t,\eta)-\tilde\Y_1(t,\theta))}\frac{1}{\sqtC{1}}\tilde\D_1\tilde\V_1^+\tilde\Y_{1,\eta}(t,\theta) d\theta\Big) d\eta\\
& +\frac{1}{\A^6} \int_0^1 (\tilde\Y_1-\tilde\Y_2)(t,\eta) \\
&\qquad\times\Big(\tilde\Y_{2,\eta}(t,\eta)\int_0^\eta e^{-\frac{1}{\sqtC{2}}(\tilde\Y_2(t,\eta)-\tilde\Y_2(t,\theta))}\frac{1}{\sqtC{2}}\tilde\D_2\tilde\V_2^-\tilde\Y_{2,\eta}(t,\theta)d\theta\\
& \qquad \qquad\qquad  -\tilde\Y_{1,\eta}(t,\eta)\int_0^\eta e^{-\frac{1}{\sqtC{1}}(\tilde\Y_1(t,\eta)-\tilde\Y_1(t,\theta))}\frac{1}{\sqtC{1}}\tilde\D_1\tilde\V_1^-\tilde\Y_{1,\eta}(t,\theta) d\theta\Big) d\eta\\
& = \bar K_{13}^++\bar K_{13}^-.
\end{align*}
Note that both $\bar K_{13}^+$ and $\bar K_{13}^-$ have the same structure and hence we are only going to present the details for $\bar K_{13}^+$, which needs to be rewritten a bit more. Namely, 
\begin{align*}
\bar K_{13}^+& = \frac{1}{\A^6}\int_0^1 (\tilde\Y_1-\tilde\Y_2)(t,\eta)\Big( \tilde\Y_{2,\eta}(t,\eta)\int_0^\eta e^{-\frac{1}{\sqtC{2}}(\tilde\Y_2(t,\eta)-\tilde\Y_2(t,\theta))}\frac{1}{\sqtC{2}}\tilde \D_2\tilde\V_2^+\tilde\Y_{2,\eta}(t,\theta) d\theta\\
& \qquad \qquad \qquad \qquad -\tilde\Y_{1,\eta}(t,\eta)\int_0^\eta e^{-\frac{1}{\sqtC{1}}(\tilde\Y_1(t,\eta)-\tilde\Y_1(t,\theta))}\frac{1}{\sqtC{1}}\tilde\D_1\tilde\V_1^+\tilde\Y_{1,\eta}(t,\theta) d\theta\Big) d\eta\\
& = \mathbbm{1}_{\sqtC{1}\leq \sqtC{2}}\frac{1}{\A^6}\Big(\frac{1}{\sqtC{2}}-\frac{1}{\sqtC{1}}\Big) \int_0^1 (\tilde \Y_1-\tilde \Y_2)\tilde \Y_{1,\eta}(t,\eta)\\
&\qquad\qquad\qquad\qquad\qquad\times \Big(\int_0^\eta e^{-\frac{1}{\sqtC{1}}(\tilde \Y_1(t,\eta)-\tilde \Y_1(t,\theta))}\tilde \D_1\tilde \V_1^+\tilde \Y_{1,\eta} (t,\theta) d\theta\Big) d\eta\\
& \quad + \mathbbm{1}_{\sqtC{2}<\sqtC{1}}\frac{1}{\A^6}\Big(\frac{1}{\sqtC{2}}-\frac{1}{\sqtC{1}}\Big)\int_0^1 (\tilde \Y_1-\tilde \Y_2)\tilde \Y_{2,\eta} (t,\eta)\\
&\qquad\qquad\qquad\qquad\qquad\times \Big(\int_0^\eta e^{-\frac{1}{\sqtC{2}}(\tilde \Y_2(t,\eta)-\tilde \Y_2(t,\theta))}\tilde \D_2\tilde \V_2^+\tilde \Y_{2,\eta}(t,\theta)d\theta\Big) d\eta\\
& \quad+  \frac{1}{\A^7} \int_0^1 (\tilde\Y_1-\tilde\Y_2)\tilde\Y_{2,\eta}(t,\eta)\\
&\qquad\qquad\qquad\times\Big(\int_0^\eta e^{-\frac{1}{\sqtC{2}}(\tilde\Y_2(t,\eta)-\tilde\Y_2(t,\theta))}(\tilde\D_2-\tilde\D_1)\tilde\V_2^+\tilde\Y_{2,\eta}\mathbbm{1}_{\tilde \D_1\leq \tilde\D_2}(t,\theta) d\theta\Big) d\eta\\
&\quad +\frac{1}{\A^7} \int_0^1 (\tilde\Y_1-\tilde\Y_2)\tilde\Y_{1,\eta}(t,\eta)\\
&\qquad\qquad\qquad\times\Big(\int_0^\eta e^{-\frac{1}{\sqtC{1}}(\tilde\Y_1(t,\eta)-\tilde\Y_1(t,\theta))}(\tilde\D_2-\tilde\D_1)\tilde\V_1^+\tilde\Y_{1,\eta}\mathbbm{1}_{\tilde \D_2<\tilde\D_1}(t,\theta)d\theta\Big) d\eta\\
&\quad +\frac{1}{\A^7}\int_0^1 (\tilde\Y_1-\tilde\Y_2)\tilde\Y_{2,\eta}(t,\eta)\\
&\qquad\quad\times\Big(\int_0^\eta e^{-\frac{1}{\sqtC{2}}(\tilde\Y_2(t,\eta)-\tilde\Y_2(t,\theta))}\min_j(\tilde\D_j)(\tilde\V_2^+-\tilde\V_1^+)\tilde\Y_{2,\eta}\mathbbm{1}_{\tilde\V_1^+\leq \tilde\V_2^+}(t,\theta) d\theta\Big) d\eta\\
&\quad+ \frac{1}{\A^7} \int_0^1 (\tilde\Y_1-\tilde\Y_2)\tilde\Y_{1,\eta}(t,\eta)\\
&\qquad\quad\times\Big(\int_0^\eta e^{-\frac{1}{\sqtC{1}}(\tilde\Y_1(t,\eta)-\tilde\Y_1(t,\theta))}\min_j(\tilde\D_j)(\tilde\V_2^+-\tilde\V_1^+)\tilde\Y_{1,\eta}\mathbbm{1}_{\tilde\V_2^+<\tilde\V_1^+}(t,\theta) d\theta\Big) d\eta\\
&\quad +\frac{1}{\A^7} \int_0^1 (\tilde\Y_1-\tilde\Y_2)\tilde\Y_{2,\eta}(t,\eta) \\
& \quad\qquad\qquad\qquad\quad \times\Big(\int_0^\eta (e^{-\frac{1}{\sqtC{2}}(\tilde\Y_2(t,\eta)-\tilde\Y_2(t,\theta))}-e^{-\frac{1}{\sqtC{2}}(\tilde\Y_1(t,\eta)-\tilde\Y_1(t,\theta))})\\
&\qquad\qquad\qquad\qquad\qquad\qquad\qquad\times\min_j(\tilde\D_j)\min_j(\tilde\V_j^+)\tilde\Y_{2,\eta}\mathbbm{1}_{B(\eta)}(t,\theta)d\theta\Big) d\eta\\
&\quad +\frac{1}{\A^7} \int_0^1 (\tilde\Y_1-\tilde\Y_2)\tilde\Y_{1,\eta}(t,\eta)\\
&\qquad\qquad\qquad \qquad\quad \times\Big(\int_0^\eta (e^{-\frac{1}{\sqtC{1}}(\tilde\Y_2(t,\eta)-\tilde\Y_2(t,\theta))}-e^{-\frac{1}{\sqtC{1}}(\tilde\Y_1(t,\eta)-\tilde\Y_1(t,\theta))})\\
&\qquad\qquad\qquad\qquad\qquad\qquad\qquad\times\min_j(\tilde\D_j)\min_j(\tilde\V_j^+)\tilde\Y_{1,\eta}\mathbbm{1}_{B(\eta)^c}(t,\theta)d\theta\Big) d\eta\\
& \quad+\mathbbm{1}_{\sqtC{1}\leq \sqtC{2}}\frac{1}{\A^7} \int_0^1 (\tilde \Y_1-\tilde \Y_2)\tilde \Y_{2,\eta} (t,\eta)\\
& \qquad \qquad \qquad \quad  \times \Big(\int_0^\eta (\min_j(e^{-\frac{1}{\sqtC{2}}(\tilde \Y_j(t,\eta)-\tilde \Y_j(t,\theta))})-\min_j(e^{-\frac{1}{\sqtC{1}}(\tilde \Y_j(t,\eta)-\tilde \Y_j(t,\theta))}))\\
& \qquad \qquad \qquad \qquad \qquad \qquad \quad \times \min_j(\tilde \D_j)\min_j(\tilde \V_j^+)\tilde \Y_{2,\eta} (t,\theta)d\theta\Big) d\eta\\
& \quad+\mathbbm{1}_{\sqtC{2}<\sqtC{1}}\frac{1}{\A^7} \int_0^1 (\tilde \Y_1-\tilde \Y_2) \tilde \Y_{1,\eta} (t,\eta)\\
& \qquad \qquad \qquad \quad  \times \Big(\int_0^\eta (\min_j(e^{-\frac{1}{\sqtC{2}}(\tilde \Y_j(t,\eta)-\tilde \Y_j(t,\theta))})-\min_j(e^{-\frac{1}{\sqtC{1}}(\tilde \Y_j(t,\eta)-\tilde \Y_j(t,\theta))}))\\
& \qquad \qquad \qquad \qquad \qquad \qquad  \times \min_j(\tilde \D_j)\min_j(\tilde \V_j^+)\tilde \Y_{1,\eta}(t,\theta) d\theta\Big) d\eta\\
&\quad+\frac{1}{A^7}\int_0^1 (\tilde\Y_1-\tilde\Y_2)(\tilde\Y_{2,\eta}-\tilde\Y_{1,\eta})(t,\eta)\\
& \qquad \times \min_k\Big(\int_0^\eta \min_j(e^{-\frac{1}{\ma}(\tilde \Y_j(t,\eta)-\tilde\Y_j(t,\theta))}) \min_j(\tilde\D_j)\min_j(\tilde\V_j^+)\tilde\Y_{k,\eta}(t,\theta) d\theta\Big) d\eta\\
& \quad+ \frac{1}{\A^7}\int_0^1 (\tilde\Y_1-\tilde\Y_2)\tilde\Y_{2,\eta}\mathbbm{1}_{D^c}(t,\eta)\\
& \quad  \times \Big(\int_0^\eta \min_j(e^{-\frac{1}{\ma}(\tilde \Y_j(t,\eta)-\tilde\Y_j(t,\theta))})\\
&\qquad\qquad\qquad\qquad\qquad\qquad\times\min_j(\tilde\D_j)\min_j(\tilde\V_j^+)(\tilde\Y_{2,\eta}-\tilde\Y_{1,\eta})(t,\theta) d\theta\Big) d\eta\\
& \quad+ \frac{1}{\A^7} \int_0^1 (\tilde\Y_1-\tilde\Y_2)\tilde\Y_{1,\eta} \mathbbm{1}_D(t,\eta)\\
& \quad \times \Big(\int_0^\eta \min_j(e^{-\frac{1}{\ma}(\tilde \Y_j(t,\eta)-\tilde\Y_j(t,\theta))})\\
&\qquad\qquad\qquad\qquad\qquad\qquad\times\min_j(\tilde\D_j)\min_j(\tilde\V_j^+)(\tilde\Y_{2,\eta}-\tilde\Y_{1,\eta})(t,\theta) d\theta\Big) d\eta\\
& = \hat B_{31}+\hat B_{32} +\bar B_{31}+\bar B_{32}+\bar B_{33} +\bar B_{34}\\
&\quad+\bar B_{35}+\bar B_{36}+\hat B_{35}+\hat B_{36}+\bar B_{37}+\bar B_{38}+\bar B_{39},
\end{align*}
where  
$B(\eta)$ by \eqref{Def:Bn}  
and 
\begin{equation}\label{eq:set_e}
D=\Big\{ (t,\eta)\mid  \int_0^\eta \varUpsilon(t,\eta,\theta)\tilde\Y_{2,\eta}(t,\theta) d\theta
 \leq \int_0^\eta \varUpsilon(t,\eta,\theta)\tilde\Y_{1,\eta}(t,\theta) d\theta\Big\},
\end{equation} 
where
\begin{equation*}
\varUpsilon(t,\eta,\theta)= \min_j(e^{-\frac{1}{\ma}(\tilde \Y_j(t,\eta)-\tilde\Y_j(t,\theta))}) 
\min_j(\tilde\D_j)\min_j(\tilde\V_j^+).
\end{equation*}
We start by estimating $\hat B_{31}$ (a similar argument works for $\hat B_{32}$.
Direct calculations yield
\begin{align*}
\vert \hat B_{31}\vert &= \mathbbm{1}_{\sqtC{1}\leq \sqtC{2}}\frac{1}{\ma\A^7}\vert \sqtC{2}-\sqtC{1}\vert \int_0^1 \vert \tilde \Y_1-\tilde \Y_2\vert \tilde \Y_{1,\eta}(t,\eta)\\
&\qquad\qquad\qquad\times\Big(\int_0^\eta e^{-\frac{1}{\sqtC{1}}(\tilde \Y_1(t,\eta)-\tilde \Y_1(t,\theta))} \tilde \D_1\tilde \V_1^+\tilde \Y_{1,\eta}(t,\theta) d\theta\Big) d\eta\\
& \leq \frac{2}{\A^7} \vert \sqtC{1}-\sqtC{2}\vert \int_0^1 \vert \tilde \Y_1-\tilde \Y_2\vert \tilde \Y_{1,\eta}(t,\eta) \\
&\qquad\qquad\qquad\times\Big(\int_0^\eta e^{-\frac{1}{\sqtC{1}}(\tilde \Y_1(t,\eta)-\tilde \Y_1(t,\theta))} \tilde \P_1 \tilde \V_1^+\tilde \Y_{1,\eta}(t,\theta) d\theta\Big) d\eta\\
& \leq \frac{2}{\A^7} \vert \sqtC{1}-\sqtC{2}\vert \int_0^1 \vert \tilde \Y_1-\tilde \Y_2\vert \tilde \Y_{1,\eta}(t,\eta) \\
&\qquad\qquad\qquad\times\Big(\int_0^\eta e^{-\frac{1}{\sqtC{1}}(\tilde \Y_1(t,\eta)-\tilde \Y_1(t,\theta))}\tilde \P_1^2\tilde \Y_{1,\eta}(t,\theta) d\theta\Big)^{1/2}\\
& \qquad \qquad \qquad \qquad \qquad  \times \Big(\int_0^\eta e^{-\frac{1}{\sqtC{1}}(\tilde \Y_1(t,\eta)-\tilde \Y_1(t,\theta))}\tilde \U_1^2\tilde \Y_{1,\eta}(t,\theta) d\theta\Big)^{1/2} d\eta\\
& \leq \frac{2\sqrt{6}}{\A^4} \vert \sqtC{1}-\sqtC{2}\vert \int_0^1 \vert \tilde \Y_1-\tilde \Y_2\vert \tilde \P_1\tilde \Y_{1,\eta} (t,\eta) d\eta\\
& \leq \sqrt{6}\A \vert \sqtC{1}-\sqtC{2}\vert \int_0^1 \vert \tilde \Y_1-\tilde \Y_2\vert (t,\eta) d\eta\\
& \leq \bigO(1) \big( \norm{\tilde \Y_1-\tilde \Y_2}^2+\vert \sqtC{1}-\sqtC{2}\vert^2\big).
\end{align*}

For $\bar B_{31}$ (a similar argument works for $\bar B_{32}$) we would like to apply some of the estimates established when investigating $\bar K_1$. Thus we split $\bar B_{31}$ into even smaller parts, i.e., 
\begin{align*} \notag
\vert \bar B_{31}\vert &\leq \frac{1}{\A^7}\int_0^1 \vert \tilde\Y_1-\tilde\Y_2\vert \tilde\Y_{2,\eta}(t,\eta)\\ \nn
&\qquad\times\Big(\int_0^\eta e^{-\frac{1}{\sqtC{2}}(\tilde\Y_2(t,\eta)-\tilde\Y_2(t,\theta))}(\bar d_{11}+\bar d_{12})\tilde\V_2^+\tilde\Y_{2,\eta}\mathbbm{1}_{\tilde \D_1\leq \tilde\D_2}(t,\theta) d\theta\Big) d\eta\\\nn
& \quad +\frac{1}{\A^7} \int_0^1 \vert \tilde\Y_1-\tilde\Y_2\vert \tilde\Y_{2,\eta}(t,\eta)\\ \nn
&\qquad\times\Big(\int_0^\eta e^{-\frac{1}{\sqtC{2}}(\tilde\Y_2(t,\eta)-\tilde\Y_2(t,\theta))}\bar T_1\tilde\V_2^+\tilde\Y_{2,\eta}\mathbbm{1}_{\tilde \D_1\leq \tilde\D_2}(t,\theta) d\theta\Big)d\eta\notag\\\nn
& \quad +\frac{1}{\A^7} \int_0^1 \vert \tilde\Y_1-\tilde\Y_2\vert \tilde\Y_{2,\eta}(t,\eta)\\ \nn
&\qquad\times \Big(\int_0^\eta e^{-\frac{1}{\sqtC{2}}(\tilde\Y_2(t,\eta)-\tilde\Y_2(t,\theta))}\bar T_2\tilde\V_2^+\tilde\Y_{2,\eta}\mathbbm{1}_{\tilde \D_1\leq \tilde\D_2}(t,\theta)d\theta\Big) d\eta\notag\\\nn
& \quad +\frac{1}{\A^7} \int_0^1 \vert \tilde\Y_1-\tilde\Y_2\vert \tilde\Y_{2,\eta}(t,\eta)\\ \nn
&\qquad\times\Big(\int_0^\eta e^{-\frac{1}{\sqtC{2}}(\tilde\Y_2(t,\eta)-\tilde\Y_2(t,\theta))}\bar T_3\tilde\V_2^+\tilde\Y_{2,\eta}\mathbbm{1}_{\tilde \D_1\leq \tilde\D_2}(t,\theta) d\theta \Big) d\eta\notag\\
& = B_{311}+B_{312}+B_{313}+B_{314}.%\label{eq:barB31}
\end{align*}

By \eqref{est:d11} and \eqref{est:d12}, we have that
\begin{equation*}
(\bar d_{11}+\bar d_{12})(t,\theta) \leq 2\A^{3/2}\tilde\D_2^{1/2}\vert \tilde\Y_1-\tilde\Y_2\vert (t,\theta)+ 2\sqrt{2}\A^{3/2} \tilde\D_2^{1/2}(t,\theta) \norm{\tilde\Y_1-\tilde\Y_2},
\end{equation*}
and hence 
\begin{align*}
B_{311}& \leq \frac{2}{A^{11/2}} \int_0^1 \vert \tilde\Y_1-\tilde\Y_2\vert \tilde\Y_{2,\eta}(t,\eta) \\
&\qquad\qquad\qquad\times\Big(\int_0^\eta e^{-\frac{1}{\sqtC{2}}(\tilde\Y_2(t,\eta)-\tilde\Y_2(t,\theta))}\vert \tilde\Y_1-\tilde\Y_2\vert \tilde\D_2^{1/2}\tilde\V_2^+\tilde\Y_{2,\eta}(t,\theta) d\theta\Big) d\eta\\
& \quad + \frac{2\sqrt{2}}{\A^{11/2}}\int_0^1 \vert \tilde\Y_1-\tilde\Y_2\vert \tilde\Y_{2,\eta}(t,\eta)\\
&\qquad\qquad\qquad\times \Big(\int_0^\eta e^{-\frac{1}{\sqtC{2}}(\tilde\Y_2(t,\eta)-\tilde\Y_2(t,\theta))}\tilde\D_2^{1/2} \tilde\V_2^+\tilde\Y_{2,\eta}(t,\theta) d\theta\Big) d\eta \norm{\tilde\Y_1-\tilde\Y_2}\\
& \leq \norm{\tilde\Y_1-\tilde\Y_2}^2 \\
& \quad + \frac{4}{\A^{11}} \int_0^1 \tilde\Y_{2,\eta}^2(t,\eta)\\
&\qquad\qquad\qquad\times\Big(\int_0^\eta e^{-\frac{1}{\sqtC{2}}(\tilde \Y_2(t,\eta)-\tilde\Y_2(t,\theta))}\vert\tilde\Y_1-\tilde\Y_2\vert \tilde\D_2^{1/2}\tilde\V_2^+\tilde\Y_{2,\eta}(t,\theta) d\theta\Big)^2d\eta\\
& \quad +\frac{2\sqrt{2}}{A^{11/2}} \norm{\tilde\Y_1-\tilde\Y_2}^2\\
&\qquad\quad\times\Big(\int_0^1 \tilde\Y_{2,\eta}^2(t,\eta)\Big(\int_0^\eta e^{-\frac{1}{\sqtC{2}}(\tilde\Y_2(t,\eta)-\tilde\Y_2(t,\theta))}\tilde\D_2^{1/2} \tilde\V_2^+\tilde\Y_{2,\eta}(t,\theta)d\theta\Big)^2d\eta\Big)^{1/2}\\
& \leq \norm{\tilde\Y_1-\tilde\Y_2}^2\\
& \quad +\frac{4}{\A^{11}} \int_0^1 \tilde\Y_{2,\eta}^2(t,\eta)\Big(\int_0^\eta e^{-\frac{1}{\sqtC{2}}(\tilde\Y_2(t,\eta)-\tilde\Y_2(t,\theta))}\tilde\U_2^2\tilde\Y_{2,\eta}(t,\theta) d\theta\Big)\\
& \qquad \qquad \qquad \qquad \times \Big(\int_0^\eta e^{-\frac{1}{\sqtC{2}}(\tilde\Y_2(t,\eta)-\tilde\Y_2(t,\theta))} (\tilde\Y_1-\tilde\Y_2)^2 \tilde\D_2 \tilde\Y_{2,\eta}(t,\theta) d\theta\Big) d\eta\\
& \quad +\frac{2\sqrt{2}}{\A^{11/2}} \norm{\tilde\Y_1-\tilde\Y_2}^2 \Big(\int_0^1 \tilde\Y_{2,\eta}^2(t,\eta) \Big( \int_0^\eta e^{-\frac{1}{\sqtC{2}}(\tilde\Y_2(t,\eta)-\tilde\Y_2(t,\theta))}\tilde\U_2^2\tilde\Y_{2,\eta}(t,\theta) d\theta\Big)\\
& \qquad \qquad \qquad \qquad \qquad\qquad\times \Big(\int_0^\eta e^{-\frac{1}{\sqtC{2}}(\tilde\Y_2(t,\eta)-\tilde\Y_2(t,\theta))}\tilde\D_2\tilde\Y_{2,\eta}(t,\theta)d\theta\Big) d\eta\Big)^{1/2}\\
& \leq \norm{\tilde\Y_1-\tilde\Y_2}^2\\
& \quad + \frac{8}{\A^{5}} \int_0^1 \tilde\Y_{2,\eta}(t,\eta) e^{-\frac{1}{\sqtC{2}}\tilde\Y_2(t,\eta)}\Big(\int_0^\eta e^{\frac{1}{\sqtC{2}}\tilde\Y_2(t,\theta)}(\tilde\Y_1-\tilde\Y_2)^2\tilde\D_2\tilde\Y_{2,\eta}(t,\theta) d\theta\Big) d\eta\\
& \quad +\frac{4}{\A^{5/2}} \norm{\tilde\Y_1-\tilde\Y_2}^2 \\
&\qquad\qquad\times\Big(\int_0^1 \tilde\Y_{2,\eta}(t,\eta)\Big(\int_0^\eta e^{-\frac{1}{\sqtC{2}}(\tilde\Y_2(t,\eta)-\tilde\Y_2(t,\theta))}\tilde\D_2\tilde\Y_{2,\eta}(t,\theta) d\theta\Big) d\eta\Big)^{1/2}\\
& \leq \norm{\tilde\Y_1-\tilde\Y_2}^2 \\
& \quad +\frac{8\sqtC{2}}{\A^5} \Big[-\int_0^\eta e^{-\frac{1}{\sqtC{2}}(\tilde\Y_2(t,\eta)-\tilde\Y_2(t,\theta))} (\tilde\Y_1-\tilde\Y_2)^2 \tilde\D_2\tilde\Y_{2,\eta}(t,\theta) d\theta\Big\vert_{\eta=0}^1 \\
&\qquad\qquad\qquad\qquad\qquad\qquad\qquad\qquad+\int_0^1 (\tilde\Y_1-\tilde\Y_2)^2 \tilde\D_2\tilde\Y_{2,\eta}(t,\eta) d\eta\Big]\\
& \quad +4\A^{1/2} \norm{\tilde\Y_1-\tilde\Y_2}^2 \Big(\int_0^1 \tilde\Y_{2,\eta}(t,\eta)e^{-\frac{1}{\sqtC{2}}\tilde\Y_2(t,\eta)}\Big(\int_0^\eta e^{\frac{1}{\sqtC{2}}\tilde\Y_2(t,\theta)}d\theta\Big) d\eta\Big)^{1/2}\\
& \leq \bigO(1) \norm{\tilde\Y_1-\tilde\Y_2}^2.
\end{align*}

Recalling the estimate for $\bar T_1$, we have that
\begin{align*}
B_{312}&\leq \frac{1}{\A^7} \int_0^1 \vert \tilde\Y_1-\tilde\Y_2\vert\tilde\Y_{2,\eta}(t,\eta)\\
&\qquad\qquad\qquad\times\Big(\int_0^\eta e^{-\frac{1}{\sqtC{2}}(\tilde\Y_2(t,\eta)-\tilde\Y_2(t,\theta))}\vert \tilde\Y_1-\tilde\Y_2\vert \vert \tilde\U_2\vert ^3\tilde\Y_{2,\eta}(t,\theta) d\theta\Big) d\eta\\
& \quad + \frac{4}{\A^4}\int_0^1 \vert\tilde\Y_1-\tilde\Y_2\vert \tilde\Y_{2,\eta}(t,\eta)
\Big(\int_0^\eta e^{-\frac{1}{\sqtC{2}}(\tilde\Y_2(t,\eta)-\tilde\Y_2(t,\theta))}\\
& \qquad \qquad   \times\Big(\int_0^\theta e^{-\frac{1}{\A}(\tilde\Y_2(t,\theta)-\tilde\Y_2(t,l))} (\tilde\U_1-\tilde\U_2)^2(t,l) dl\Big)^{1/2} \tilde\V_2^+\tilde\Y_{2,\eta}(t,\theta) d\theta\Big) d\eta\\
& \quad +\frac{\sqrt{2}}{\A^3} \int_0^1\vert \tilde \Y_1-\tilde \Y_2\vert \tilde \Y_{2,\eta}(t,\eta)
 \Big(\int_0^\eta e^{-\frac{1}{\sqtC{2}}(\tilde \Y_2(t,\eta)-\tilde \Y_2(t,\theta))}\\
& \qquad \qquad \times \Big(\int_0^\theta e^{-\frac{1}{\ma}(\tilde \Y_2(t,\theta)-\tilde \Y_2(t,l))}(\tilde \Y_1-\tilde \Y_2)^2(t,l) dl\Big)^{1/2}\tilde \V_2^+\tilde \Y_{2,\eta}(t,\theta) d\theta\Big) d\eta\\
 & \quad + \frac{1}{\A^3}\int_0^1 \vert \tilde\Y_1-\tilde\Y_2\vert \tilde\Y_{2,\eta}(t,\eta)
 \Big(\int_0^\eta e^{-\frac{1}{\sqtC{2}}(\tilde\Y_2(t,\eta)-\tilde\Y_2(t,\theta))}\\
& \qquad \qquad  \qquad  \times\Big(\int_0^\theta e^{-\frac{1}{\ma}(\tilde\Y_2(t,\theta)-\tilde\Y_2(t,l))} \vert \tilde\Y_1-\tilde\Y_2\vert (t,l)dl\Big) \tilde\V_2^+\tilde\Y_{2,\eta}(t,\theta)d\theta\Big)d\eta\\
& \quad + \frac{8\sqrt{2}}{\sqrt{3}e\A^3} \int_0^1 \vert \tilde \Y_1-\tilde \Y_2\vert \tilde \Y_{2,\eta} (t,\eta)
\Big(\int_0^\eta e^{-\frac{1}{\sqtC{2}}(\tilde \Y_2(t,\eta)-\tilde \Y_2(t,\theta))}\\
& \qquad \qquad \times \Big(\int_0^\theta  e^{-\frac{3}{4\A}(\tilde \Y_2(t,\theta) -\tilde \Y_2(t,l))}dl\Big)^{1/2} \tilde \V_2^+\tilde \Y_{2,\eta}(t,\theta)d\theta\Big) d\eta\vert \sqtC{1}-\sqtC{2}\vert \\
& \leq 5\norm{\tilde\Y_1-\tilde\Y_2}^2
 + \frac{1}{\A^{14}}\int_0^1 \tilde\Y_{2,\eta}^2(t,\eta)\\
&\qquad\times\Big(\int_0^\eta e^{-\frac{1}{\sqtC{2}}(\tilde\Y_2(t,\eta)-\tilde\Y_2(t,\theta))}\vert \tilde\Y_1-\tilde\Y_2\vert\, \vert \tilde\U_2\vert^3\tilde\Y_{2,\eta}(t,\theta) d\theta\Big)^2d\eta\\
& \quad +\frac{16}{\A^8} \int_0^1 \tilde\Y_{2,\eta}^2(t,\eta)\Big(\int_0^\eta e^{-\frac{1}{\sqtC{2}}(\tilde\Y_2(t,\eta)-\tilde\Y_2(t,\theta))}\\
&\qquad\qquad\times \Big(\int_0^\theta e^{-\frac{1}{\A}(\tilde\Y_2(t,\theta)-\tilde\Y_2(t,l))}(\tilde\U_1-\tilde\U_2)^2 (t,l)dl \Big)^{1/2} \tilde\V_2^+\tilde\Y_{2,\eta}(t,\theta)d\theta\Big)^2 d\eta\\
& \quad + \frac{2}{\A^6}\int_0^1 \tilde \Y_{2,\eta}^2(t,\eta)\Big(\int_0^\eta e^{-\frac{1}{\sqtC{2}}(\tilde \Y_2(t,\eta)-\tilde \Y_2(t,\theta))}\\
& \qquad \qquad \times \Big(\int_0^\theta e^{-\frac{1}{\ma}(\tilde \Y_2(t,\theta)-\tilde \Y_2(t,l))}(\tilde \Y_1-\tilde\Y_2)^2(t,l) dl\Big)^{1/2}\tilde \V_2^+\tilde \Y_{2,\eta}(t,\theta) d\theta\Big)^2 d\eta\\
& \quad + \frac{1}{\A^6} \int_0^1 \tilde\Y_{2,\eta}^2(t,\eta) \Big(\int_0^\eta e^{-\frac{1}{\sqtC{2}}(\tilde\Y_2(t,\eta)-\tilde\Y_2(t,\theta))}\\
&\qquad\qquad\times\Big(\int_0^\theta e^{-\frac{1}{\ma}(\tilde\Y_2(t,\theta)-\tilde\Y_2(t,l))}\vert \tilde \Y_1-\tilde\Y_2\vert (t,l)dl\Big)\tilde\U_2^+\tilde\Y_{2,\eta}(t,\theta)d\theta\Big)^2d\eta\\
& \quad + \frac{128}{3e^2\A^6} \int_0^1 \tilde \Y_{2,\eta}^2(t,\eta) \Big(\int_0^\eta e^{-\frac{1}{\sqtC{2}}(\tilde \Y_2(t,\eta)-\tilde \Y_2(t,\theta))}\\
& \qquad \qquad  \times \Big(\int_0^\theta e^{-\frac{3}{4\A}(\tilde \Y_2(t,\theta)-\tilde \Y_2(t,l))}dl\Big)^{1/2}\tilde \V_2^+\tilde \Y_{2,\eta}(t,\theta)d\theta\Big)^2 d\eta\vert \sqtC{1}-\sqtC{2}\vert^2\\
& \leq 5\norm{\tilde\Y_1-\tilde\Y_2}^2
 +\frac{1}{\A^{14}}\int_0^1\tilde\Y_{2,\eta}^2(t,\eta)\Big(\int_0^\eta  e^{-\frac{1}{\sqtC{2}}(\tilde\Y_2(t,\eta)-\tilde\Y_2(t,\theta))}\tilde\U_2^4\tilde\Y_{2,\eta}(t,\theta)d\theta\Big)\\
& \qquad \qquad \qquad  \times\Big(\int_0^\eta e^{-\frac{1}{\sqtC{2}}(\tilde\Y_2(t,\eta)-\tilde\Y_2(t,\theta))}(\tilde\Y_1-\tilde\Y_2)^2\tilde\U_2^2\tilde\Y_{2,\eta}(t,\theta)d\theta\Big) d\eta\\
& \quad +\frac{16}{\A^8} \int_0^1 \tilde\Y_{2,\eta}^2(t,\eta)\Big(\int_0^\eta e^{-\frac{1}{\sqtC{2}}(\tilde\Y_2(t,\eta)-\tilde\Y_2(t,\theta))}\tilde\U_2^2\tilde\Y_{2,\eta}(t,\theta)d\theta\Big)\\
& \qquad \qquad  \quad \times \Big(\int_0^\eta e^{-\frac{1}{\sqtC{2}}(\tilde\Y_2(t,\eta)-\tilde\Y_2(t,\theta))}\\
&\qquad\qquad\quad\times \Big(\int_0^\theta e^{-\frac{1}{\A}(\tilde\Y_2(t,\theta)-\tilde\Y_2(t,l))} (\tilde\U_1-\tilde\U_2)^2(t,l)dl\Big) \tilde\Y_{2,\eta}(t,\theta) d\theta\Big) d\eta\\
& \quad +\frac{2}{\A^6} \int_0^1 \tilde\Y_{2,\eta}^2(t,\eta)\Big(\int_0^\eta e^{-\frac{1}{\sqtC{2}}(\tilde\Y_2(t,\eta)-\tilde\Y_2(t,\theta))}\tilde\U_2^2\tilde\Y_{2,\eta}(t,\theta)d\theta\Big)\\
& \qquad \qquad  \quad \times \Big(\int_0^\eta e^{-\frac{1}{\sqtC{2}}(\tilde\Y_2(t,\eta)-\tilde\Y_2(t,\theta))}\\
&\qquad\qquad\quad\times \Big(\int_0^\theta e^{-\frac{1}{\ma}(\tilde\Y_2(t,\theta)-\tilde\Y_2(t,l))} (\tilde\Y_1-\tilde\Y_2)^2(t,l)dl\Big) \tilde\Y_{2,\eta}(t,\theta) d\theta\Big) d\eta\\
& \quad +\frac{1}{\A^6} \int_0^1 \tilde\Y_{2,\eta}^2(t,\eta) \Big(\int_0^\eta e^{-\frac{1}{\sqtC{2}}(\tilde\Y_2(t,\eta)-\tilde\Y_2(t,\theta))}\tilde\U_2^2\tilde\Y_{2,\eta}(t,\theta) d\theta\Big)\\
& \qquad \qquad  \quad\times \Big(\int_0^\eta e^{-\frac{1}{\sqtC{2}}(\tilde\Y_2(t,\eta)-\tilde\Y_2(t,\theta))}\\
&\qquad\qquad\quad\times\Big(\int_0^\theta e^{-\frac{1}{\ma}(\tilde\Y_2(t,\theta)-\tilde\Y_2(t,l))} \vert \tilde\Y_1-\tilde\Y_2\vert (t,l) dl\Big)^2 \tilde\Y_{2,\eta}(t,\theta)d\theta\Big) d\eta\\
& \quad + \frac{128}{3e^2\A^6}\int_0^1 \tilde \Y_{2,\eta}^2(t,\eta)\Big(\int_0^\eta e^{-\frac{1}{\sqtC{2}}(\tilde \Y_2(t,\eta)-\tilde \Y_2(t,\theta))}\tilde \U_2^2\tilde \Y_{2,\eta}(t,\theta)d\theta\Big)\\
& \qquad \qquad \quad  \times \Big(\int_0^\eta e^{-\frac{1}{\sqtC{2}}(\tilde \Y_2(t,\eta)-\tilde \Y_2(t,\theta))}\\
& \qquad \qquad \quad  \times \Big(\int_0^\theta e^{-\frac{3}{4\A}(\tilde \Y_2(t,\theta)-\tilde \Y_2(t,l)}dl\Big)\tilde \Y_{2,\eta} (t,\theta) d\theta\Big) d\eta\vert \sqtC{1}-\sqtC{2}\vert^2\\
& \leq 5\norm{\tilde\Y_1-\tilde\Y_2}^2\\
& \quad + \frac{2}{\A^4}\int_0^1 \tilde\P_2\tilde\Y_{2,\eta}^2(t,\eta)\Big(\int_0^\eta e^{-\frac{1}{\sqtC{2}}(\tilde\Y_2(t,\eta)-\tilde\Y_2(t,\theta))}(\tilde\Y_1-\tilde\Y_2)^2(t,\theta) d\theta\Big) d\eta\\
& \quad + \frac{64}{\A^7} \int_0^1 \tilde\P_2\tilde\Y_{2,\eta}^2(t,\eta)\Big( \int_0^\eta e^{-\frac1{2\A}(\tilde\Y_2(t,\eta)-\tilde\Y_2(t,\theta))}\\
&\qquad\qquad\quad\times\Big(\int_0^\theta e^{-\frac{1}{\A}(\tilde\Y_2(t,\theta)-\tilde\Y_2(t,l))} (\tilde\U_1-\tilde\U_2)^2(t,l)dl\Big) \tilde\Y_{2,\eta}(t,\theta) d\theta\Big) d\eta\\
& \quad + \frac{8}{\A^5} \int_0^1 \tilde\P_2\tilde\Y_{2,\eta}^2(t,\eta)\Big( \int_0^\eta e^{-\frac1{2\sqtC{2}}(\tilde\Y_2(t,\eta)-\tilde\Y_2(t,\theta))}\\
&\qquad\qquad\quad\times\Big(\int_0^\theta e^{-\frac{1}{\sqtC{2}}(\tilde\Y_2(t,\theta)-\tilde\Y_2(t,l))} (\tilde\Y_1-\tilde\Y_2)^2(t,l)dl\Big) \tilde\Y_{2,\eta}(t,\theta) d\theta\Big) d\eta\\
& \quad + \frac{4}{\A^5}\int_0^1 \tilde\P_2\tilde\Y_{2,\eta}^2(t,\eta)\Big(\int_0^\eta e^{-\frac1{2\sqtC{2}}(\tilde\Y_2(t,\eta)-\tilde\Y_2(t,\theta))}\Big(\int_0^\theta e^{-\frac{1}{\sqtC{2}}(\tilde\Y_2(t,\theta)-\tilde\Y_2(t,l))}dl\Big)\\
& \qquad \qquad \quad  \times\Big(\int_0^\theta e^{-\frac{1}{\sqtC{2}}(\tilde\Y_2(t,\theta)-\tilde\Y_2(t,l))} (\tilde\Y_1-\tilde\Y_2)^2(t,l)dl\Big) \tilde\Y_{2,\eta}(t,\theta) d\theta\Big)d\eta\\
& \quad + \frac{512}{3e^2\A^5}\int_0^1 \tilde \P_2\tilde \Y_{2,\eta}^2(t,\eta)  \Big(\int_0^\eta e^{-\frac{1}{2\A}(\tilde \Y_2(t,\eta)-\tilde \Y_2(t,\theta))}\\
&\qquad\qquad\quad\times\Big(\int_0^\theta e^{-\frac{3}{4\A}(\tilde \Y_2(t,\theta)-\tilde \Y_2(t,l))}dl\Big) \tilde \Y_{2,\eta}(t,\theta) d\theta\Big) d\eta\vert \sqtC{1}-\sqtC{2}\vert^2\\
& \leq 5\norm{\tilde\Y_1-\tilde\Y_2}^2\\
& \quad +\A \int_0^1 \tilde\Y_{2,\eta}(t,\eta)e^{-\frac{1}{\sqtC{2}}\tilde\Y_2(t,\eta)}\Big(\int_0^\eta e^{\frac{1}{\sqtC{2}}\tilde\Y_2(t,\theta)}(\tilde\Y_1-\tilde\Y_2)^2(t,\theta) d\theta\Big) d\eta\\
& \quad + \frac{32}{\A^2} \int_0^1 \tilde\Y_{2,\eta}(t,\eta)\Big(\int_0^\eta \tilde\Y_{2,\eta}(t,\theta)e^{-\frac1{2\A} \tilde\Y_2(t,\theta) }\\
&\qquad\qquad\quad\times\Big( \int_0^\theta e^{-(\frac1{2\A}\tilde\Y_2(t,\eta)- \frac{1}{\A}\tilde\Y_2(t,l))}(\tilde\U_1-\tilde\U_2)^2(t,l) dl\Big) d\theta\Big) d\eta\\
& \quad + 4 \int_0^1 \tilde\Y_{2,\eta}(t,\eta)\Big(\int_0^\eta \tilde\Y_{2,\eta}(t,\theta)e^{-\frac1{2\sqtC{2}} \tilde\Y_2(t,\theta) }\\
&\qquad\qquad\quad\times\Big( \int_0^\theta e^{-(\frac1{2\sqtC{2}}\tilde\Y_2(t,\eta)- \frac{1}{\sqtC{2}}\tilde\Y_2(t,l))}(\tilde\Y_1-\tilde\Y_2)^2(t,l) dl\Big) d\theta\Big) d\eta\\
& \quad +2 \int_0^1\tilde\Y_{2,\eta}(t,\eta) \Big(\int_0^\eta \tilde\Y_{2,\eta}(t,\theta)e^{-\frac1{2\sqtC{2}} \tilde\Y_2(t,\theta)}\\
&\qquad\quad\times\Big( \int_0^\theta e^{-(\frac1{2\sqtC{2}}\tilde\Y_2(t,\eta)-\frac{1}{\sqtC{2}}\tilde\Y_2(t,l))}dl\Big) d\theta\Big) d\eta\norm{\tilde\Y_1-\tilde\Y_2}^2\\
& \quad + \frac{256}{3e^2}\int_0^1 \tilde \Y_{2,\eta}(t,\eta)\Big(\int_0^\eta \tilde \Y_{2,\eta}(t,\theta) e^{-\frac{1}{4\A}\tilde \Y_2(t,\theta)}\\
&\qquad\quad\times\Big( \int_0^\theta e^{-(\frac{1}{2\A}\tilde \Y_2(t,\eta)-\frac{3}{4\A}\tilde \Y_2(t,l))}dl \Big) d\theta\Big) d\eta \vert \sqtC{1}-\sqtC{2}\vert^2\\
& \leq 5\norm{\tilde\Y_1-\tilde\Y_2}^2 +\A \Big[-\sqtC{2}\int_0^\eta e^{-\frac{1}{\sqtC{2}}(\tilde\Y_2(t,\eta)-\tilde\Y_2(t,\theta))}(\tilde\Y_1-\tilde\Y_2)^2(t,\theta) d\theta\Big\vert_{\eta=0}^1\\
& \qquad\qquad\qquad\qquad\qquad +\int_0^1 (\tilde\Y_1-\tilde\Y_2)^2(t,\eta) d\eta\Big]\\
& \quad +\frac{32}{\A} \int_0^1 \tilde\Y_{2,\eta}(t,\eta)\Big[-2 \int_0^\theta e^{\frac{1}{\A}\tilde\Y_2(t,l)-\frac1{2\A}(\tilde\Y_2(t,\eta)+\tilde\Y_2(t,\theta))} (\tilde\U_1-\tilde\U_2)^2(t,l) dl\Big\vert_{\theta=0}^\eta\\
& \qquad \qquad \qquad \qquad \qquad +2 \int_0^\eta e^{\frac1{2\A}(\tilde\Y_2(t,\theta)-\tilde\Y_2(t,\eta))}(\tilde\U_1-\tilde\U_2)^2(t,\theta) d\theta\Big]d\eta\\
& \quad +4\sqtC{2} \int_0^1 \tilde\Y_{2,\eta}(t,\eta)\\
&\qquad\qquad\times\Big[-2 \int_0^\theta e^{\frac{1}{\sqtC{2}}\tilde\Y_2(t,l)-\frac1{2\sqtC{2}}(\tilde\Y_2(t,\eta)+\tilde\Y_2(t,\theta))} (\tilde\Y_1-\tilde\Y_2)^2(t,l) dl\Big\vert_{\theta=0}^\eta\\
& \qquad \qquad \qquad \qquad \qquad +2 \int_0^\eta e^{\frac1{2\sqtC{2}}(\tilde\Y_2(t,\theta)-\tilde\Y_2(t,\eta))}(\tilde\Y_1-\tilde\Y_2)^2(t,\theta) d\theta\Big]d\eta\\
& \quad + 2\sqtC{2}\int_0^1 \tilde\Y_{2,\eta}(t,\eta) \Big[ -2\int_0^\theta e^{\frac{1}{\sqtC{2}}\tilde\Y_2(t,l)-\frac1{2\sqtC{2}}(\tilde\Y_2(t,\theta)+\tilde\Y_2(t,\eta))}dl\Big\vert_{\theta=0}^\eta \\
& \qquad \qquad \qquad \qquad \qquad\qquad +2 \int_0^\eta e^{\frac1{2\sqtC{2}}(\tilde\Y_2(t,\theta)-\tilde\Y_2(t,\eta))}d\theta\Big]d\eta\norm{\tilde\Y_1-\tilde\Y_2}^2\\
& \quad +\frac{256\A}{3e^2} \int_0^1 \tilde \Y_{2,\eta}(t,\eta)\Big[-4\int_0^\theta e^{\frac{3}{4A} \tilde \Y_2(t,l)-(\frac1{4\A}\tilde \Y_2(t,\theta)+\frac{1}{2\A}\tilde \Y_2(t,\eta))}\Big\vert_{\theta=0}^\eta\\
& \qquad \qquad \qquad \qquad \qquad \qquad + 4\int_0^\eta e^{\frac{1}{2\A} (\tilde \Y_2(t,\theta)-\tilde \Y_2(t,\eta))}d\theta\Big]d\eta \vert \sqtC{1}-\sqtC{2}\vert^2\\
& \leq \bigO(1)\norm{\tilde\Y_1-\tilde\Y_2}^2\\
& \quad +\frac{64}{\A}\int_0^1 \tilde\Y_{2,\eta}(t,\eta) e^{-\frac1{2\A} \tilde\Y_2(t,\eta)}\Big(\int_0^\eta e^{\frac1{2\A}\tilde\Y_2(t,\theta)}(\tilde\U_1-\tilde\U_2)^2(t,\theta) d\theta\Big) d\eta\\
& \quad +8\A\int_0^1 \tilde\Y_{2,\eta}(t,\eta) e^{-\frac1{2\sqtC{2}} \tilde\Y_2(t,\eta)}\Big(\int_0^\eta e^{\frac1{2\sqtC{2}}\tilde\Y_2(t,\theta)}(\tilde\Y_1-\tilde\Y_2)^2(t,\theta) d\theta\Big) d\eta\\
&\quad +4\A \int_0^1\tilde\Y_{2,\eta}(t,\eta)e^{-\frac1{2\sqtC{2}} \tilde\Y_2(t,\eta)}\Big(\int_0^\eta e^{\frac1{2\sqtC{2}} \tilde\Y_2(t,\theta)}d\theta\Big) d\eta\norm{\tilde\Y_1-\tilde\Y_2}^2\\
& \quad + \frac{1024\A}{3e^2}\int_0^1 \tilde \Y_{2,\eta}(t,\eta)e^{-\frac{1}{2\A}\tilde \Y_2(t,\eta)}\Big(\int_0^\eta e^{\frac{1}{2\A}\tilde \Y_2(t,\theta)}d\theta\Big) d\eta \vert \sqtC{1}-\sqtC{2}\vert^2\\
& \leq \bigO(1)\norm{\tilde\Y_1-\tilde\Y_2}^2\\
& \quad +64\Big[ -2 \int_0^\eta e^{\frac1{2\A}(\tilde\Y_2(t,\theta)-\tilde\Y_2(t,\eta))} (\tilde\U_1-\tilde\U_2)^2(t,\theta)d\theta\Big\vert_{\eta=0}^1\\
&\qquad\qquad\qquad\qquad\qquad\qquad\qquad\qquad+2\int_0^1 (\tilde\U_1-\tilde\U_2)^2(t,\eta)d\eta\Big]\\
& \quad +8\A\sqtC{2}\Big[ -2 \int_0^\eta e^{\frac1{2\sqtC{2}}(\tilde\Y_2(t,\theta)-\tilde\Y_2(t,\eta))} (\tilde\Y_1-\tilde\Y_2)^2(t,\theta)d\theta\Big\vert_{\eta=0}^1\\
&\qquad\qquad\qquad\qquad\qquad\qquad\qquad\qquad+2\int_0^1 (\tilde\Y_1-\tilde\Y_2)^2(t,\eta)d\eta\Big]\\
& \quad + 4\A\sqtC{2}\Big[ -2 \int_0^\eta e^{\frac1{2\sqtC{2}}(\tilde\Y_2(t,\theta)-\tilde\Y_2(t,\eta))}d\theta\Big\vert_{\eta=0}^1 + 2\int_0^1 d\eta\Big] \norm{\tilde\Y_1-\tilde\Y_2}^2\\
& \quad + \frac{1024\A^2}{3e^2}\Big[ -2 \int_0^\eta e^{\frac1{2\A}(\tilde\Y_2(t,\theta)-\tilde\Y_2(t,\eta))}d\theta\Big\vert_{\eta=0}^1 + 2\int_0^1 d\eta\Big] \vert\sqtC{1}-\sqtC{2}\vert^2\\
& \leq \bigO(1) (\norm{\tilde\Y_1-\tilde\Y_2}^2+\norm{\tilde\U_1-\tilde\U_2}^2+\vert \sqtC{1}-\sqtC{2}\vert^2).
\end{align*}

Similar calculations yield for $B_{313}$, 
\begin{align*} \nn
B_{313}&\leq \frac{1}{\A^7} \int_0^1 \vert \tilde\Y_1-\tilde\Y_2\vert\tilde\Y_{2,\eta}(t,\eta)\\
&\qquad \quad\times\Big(\int_0^\eta e^{-\frac{1}{\sqtC{2}}(\tilde\Y_2(t,\eta)-\tilde\Y_2(t,\theta))}
\vert \tilde\Y_1-\tilde\Y_2\vert \tilde \P_2\vert \tilde\U_2\vert \tilde\Y_{2,\eta}(t,\theta) d\theta\Big) d\eta \nn \\
& \quad + \frac{2\sqrt{2}}{\A^4}\int_0^1 \vert\tilde\Y_1-\tilde\Y_2\vert \tilde\Y_{2,\eta}(t,\eta)
\Big(\int_0^\eta e^{-\frac{1}{\sqtC{2}}(\tilde\Y_2(t,\eta)-\tilde\Y_2(t,\theta))} \nn\\
& \qquad \quad   \times\Big(\int_0^\theta e^{-\frac{1}{\A}(\tilde\Y_2(t,\theta)-\tilde\Y_2(t,l))} (\sqP{1}-\sqP{2})^2(t,l) dl\Big)^{1/2} \tilde\V_2^+\tilde\Y_{2,\eta}(t,\theta) d\theta\Big) d\eta \nn\\
& \quad +\frac{1}{\sqrt{2}\A^3} \int_0^1\vert \tilde \Y_1-\tilde \Y_2\vert \tilde \Y_{2,\eta}(t,\eta)
 \Big(\int_0^\eta e^{-\frac{1}{\sqtC{2}}(\tilde \Y_2(t,\eta)-\tilde \Y_2(t,\theta))}\nn \\
& \qquad \quad \times \Big(\int_0^\theta e^{-\frac{1}{\ma}(\tilde \Y_2(t,\theta)-\tilde \Y_2(t,l))}(\tilde \Y_1-\tilde \Y_2)^2(t,l) dl\Big)^{1/2}\tilde \V_2^+\tilde \Y_{2,\eta}(t,\theta) d\theta\Big) d\eta\nn \\
 & \quad + \frac{1}{2\A^3}\int_0^1 \vert \tilde\Y_1-\tilde\Y_2\vert \tilde\Y_{2,\eta}(t,\eta) 
 \Big(\int_0^\eta e^{-\frac{1}{\sqtC{2}}(\tilde\Y_2(t,\eta)-\tilde\Y_2(t,\theta))}\nn \\
& \qquad \quad  \times\Big(\int_0^\theta e^{-\frac{1}{\ma}(\tilde\Y_2(t,\theta)-\tilde\Y_2(t,l))} \vert \tilde\Y_1-\tilde\Y_2\vert (t,l)dl\Big) \tilde\V_2^+\tilde\Y_{2,\eta}(t,\theta)d\theta\Big)d\eta\nn \\
& \quad + \frac{4\sqrt{2}}{\sqrt{3}e\A^3} \int_0^1 \vert \tilde \Y_1-\tilde \Y_2\vert \tilde \Y_{2,\eta} (t,\eta)
\Big(\int_0^\eta e^{-\frac{1}{\sqtC{2}}(\tilde \Y_2(t,\eta)-\tilde \Y_2(t,\theta))}\nn \\
& \qquad \quad \times \Big(\int_0^\theta  e^{-\frac{3}{4\A}(\tilde \Y_2(t,\theta) -\tilde \Y_2(t,l))}dl\Big)^{1/2} \tilde \V_2^+\tilde \Y_{2,\eta}(t,\theta)d\theta\Big) d\eta\vert \sqtC{1}-\sqtC{2}\vert\nn \\ \nn
& \leq 5\norm{\tilde\Y_1-\tilde\Y_2}^2 + \frac{1}{\A^{14}}\int_0^1 \tilde\Y_{2,\eta}^2(t,\eta)\\
&\qquad\quad\times\Big(\int_0^\eta e^{-\frac{1}{\sqtC{2}}(\tilde\Y_2(t,\eta)-\tilde\Y_2(t,\theta))}\vert \tilde\Y_1-\tilde\Y_2\vert\, \tilde\P_2\vert \tilde\U_2\vert\tilde\Y_{2,\eta}(t,\theta) d\theta\Big)^2d\eta \nn\\
& \quad +\frac{8}{\A^8} \int_0^1 \tilde\Y_{2,\eta}^2(t,\eta)\Big(\int_0^\eta e^{-\frac{1}{\sqtC{2}}(\tilde\Y_2(t,\eta)-\tilde\Y_2(t,\theta))}\nn \\
&\qquad\times \Big(\int_0^\theta e^{-\frac{1}{\A}(\tilde\Y_2(t,\theta)-\tilde\Y_2(t,l))}(\sqP{1}-\sqP{2})^2 (t,l)dl \Big)^{1/2} \tilde\V_2^+\tilde\Y_{2,\eta}(t,\theta)d\theta\Big)^2 d\eta\nn \\
& \quad + \frac{1}{2\A^6}\int_0^1 \tilde \Y_{2,\eta}^2(t,\eta)\Big(\int_0^\eta e^{-\frac{1}{\sqtC{2}}(\tilde \Y_2(t,\eta)-\tilde \Y_2(t,\theta))}\nn \\
& \qquad \quad\times \Big(\int_0^\theta e^{-\frac{1}{\ma}(\tilde \Y_2(t,\theta)-\tilde \Y_2(t,l))}(\tilde \Y_1-\tilde\Y_2)^2(t,l) dl\Big)^{1/2}\tilde \V_2^+\tilde \Y_{2,\eta}(t,\theta) d\theta\Big)^2 d\eta\nn \\
& \quad + \frac{1}{4\A^6} \int_0^1 \tilde\Y_{2,\eta}^2(t,\eta) \Big(\int_0^\eta e^{-\frac{1}{\sqtC{2}}(\tilde\Y_2(t,\eta)-\tilde\Y_2(t,\theta))}\nn \\
&\qquad\quad\times\Big(\int_0^\theta e^{-\frac{1}{\ma}(\tilde\Y_2(t,\theta)-\tilde\Y_2(t,l))}\vert \tilde \Y_1-\tilde\Y_2\vert (t,l)dl\Big)\tilde\U_2^+\tilde\Y_{2,\eta}(t,\theta)d\theta\Big)^2d\eta\nn \\
& \quad + \frac{32}{3e^2\A^6} \int_0^1 \tilde \Y_{2,\eta}^2(t,\eta) \Big(\int_0^\eta e^{-\frac{1}{\sqtC{2}}(\tilde \Y_2(t,\eta)-\tilde \Y_2(t,\theta))}\nn \\
& \qquad \quad  \times \Big(\int_0^\theta e^{-\frac{3}{4\A}(\tilde \Y_2(t,\theta)-\tilde \Y_2(t,l))}dl\Big)^{1/2}\tilde \V_2^+\tilde \Y_{2,\eta}(t,\theta)d\theta\Big)^2 d\eta\vert \sqtC{1}-\sqtC{2}\vert^2\nn \\
& \leq 5\norm{\tilde\Y_1-\tilde\Y_2}^2\nn \\
& \quad +\frac{1}{\A^{14}}\int_0^1\tilde\Y_{2,\eta}^2(t,\eta)\Big(\int_0^\eta  e^{-\frac{1}{\sqtC{2}}(\tilde\Y_2(t,\eta)-\tilde\Y_2(t,\theta))}\tilde\U_2^2\tilde\Y_{2,\eta}(t,\theta)d\theta\Big)\nn \\
& \qquad \quad  \times\Big(\int_0^\eta e^{-\frac{1}{\sqtC{2}}(\tilde\Y_2(t,\eta)-\tilde\Y_2(t,\theta))}(\tilde\Y_1-\tilde\Y_2)^2\tilde\P_2^2\tilde\Y_{2,\eta}(t,\theta)d\theta\Big) d\eta\nn \\
& \quad +\frac{8}{\A^8} \int_0^1 \tilde\Y_{2,\eta}^2(t,\eta)\Big(\int_0^\eta e^{-\frac{1}{\sqtC{2}}(\tilde\Y_2(t,\eta)-\tilde\Y_2(t,\theta))}\tilde\U_2^2\tilde\Y_{2,\eta}(t,\theta)d\theta\Big)\nn \\
& \qquad  \quad \times \Big(\int_0^\eta e^{-\frac{1}{\sqtC{2}}(\tilde\Y_2(t,\eta)-\tilde\Y_2(t,\theta))}\nn \\
&\qquad\quad\times \Big(\int_0^\theta e^{-\frac{1}{\A}(\tilde\Y_2(t,\theta)-\tilde\Y_2(t,l))} (\sqP{1}-\sqP{2})^2(t,l)dl\Big) \tilde\Y_{2,\eta}(t,\theta) d\theta\Big) d\eta\nn \\
& \quad +\frac{1}{2\A^6} \int_0^1 \tilde\Y_{2,\eta}^2(t,\eta)\Big(\int_0^\eta e^{-\frac{1}{\sqtC{2}}(\tilde\Y_2(t,\eta)-\tilde\Y_2(t,\theta))}\tilde\U_2^2\tilde\Y_{2,\eta}(t,\theta)d\theta\Big)\nn \\
& \qquad  \quad \times \Big(\int_0^\eta e^{-\frac{1}{\sqtC{2}}(\tilde\Y_2(t,\eta)-\tilde\Y_2(t,\theta))}\nn \\
&\qquad\quad\times \Big(\int_0^\theta e^{-\frac{1}{\ma}(\tilde\Y_2(t,\theta)-\tilde\Y_2(t,l))} (\tilde\Y_1-\tilde\Y_2)^2(t,l)dl\Big) \tilde\Y_{2,\eta}(t,\theta) d\theta\Big) d\eta\nn \\
& \quad +\frac{1}{4\A^6} \int_0^1 \tilde\Y_{2,\eta}^2(t,\eta) \Big(\int_0^\eta e^{-\frac{1}{\sqtC{2}}(\tilde\Y_2(t,\eta)-\tilde\Y_2(t,\theta))}\tilde\U_2^2\tilde\Y_{2,\eta}(t,\theta) d\theta\Big)\nn \\
& \qquad \quad  \times \Big(\int_0^\eta e^{-\frac{1}{\sqtC{2}}(\tilde\Y_2(t,\eta)-\tilde\Y_2(t,\theta))}\nn \\
&\qquad\quad\times\Big(\int_0^\theta e^{-\frac{1}{\ma}(\tilde\Y_2(t,\theta)-\tilde\Y_2(t,l))} \vert \tilde\Y_1-\tilde\Y_2\vert (t,l) dl\Big)^2 \tilde\Y_{2,\eta}(t,\theta)d\theta\Big) d\eta\nn \\
& \quad + \frac{32}{3e^2\A^6}\int_0^1 \tilde \Y_{2,\eta}^2(t,\eta)\Big(\int_0^\eta e^{-\frac{1}{\sqtC{2}}(\tilde \Y_2(t,\eta)-\tilde \Y_2(t,\theta))}\tilde \U_2^2\tilde \Y_{2,\eta}(t,\theta)d\theta\Big)\nn \\
& \qquad \quad  \times \Big(\int_0^\eta e^{-\frac{1}{\sqtC{2}}(\tilde \Y_2(t,\eta)-\tilde \Y_2(t,\theta))}\nn \\
& \qquad \quad  \times \Big(\int_0^\theta e^{-\frac{3}{4\A}(\tilde \Y_2(t,\theta)-\tilde \Y_2(t,l)}dl\Big)\tilde \Y_{2,\eta} (t,\theta) d\theta\Big) d\eta\vert \sqtC{1}-\sqtC{2}\vert^2\nn \\
& \leq 5\norm{\tilde\Y_1-\tilde\Y_2}^2\nn \\
& \quad + \frac{1}{2\A^4}\int_0^1 \tilde\P_2\tilde\Y_{2,\eta}^2(t,\eta)\Big(\int_0^\eta e^{-\frac{1}{\sqtC{2}}(\tilde\Y_2(t,\eta)-\tilde\Y_2(t,\theta))}(\tilde\Y_1-\tilde\Y_2)^2(t,\theta) d\theta\Big) d\eta\nn \\
& \quad + \frac{32}{\A^7} \int_0^1 \tilde\P_2\tilde\Y_{2,\eta}^2(t,\eta)\Big( \int_0^\eta e^{-\frac1{2\A}(\tilde\Y_2(t,\eta)-\tilde\Y_2(t,\theta))}\nn \\
&\qquad\quad\times\Big(\int_0^\theta e^{-\frac{1}{\A}(\tilde\Y_2(t,\theta)-\tilde\Y_2(t,l))} (\sqP{1}-\sqP{2})^2(t,l)dl\Big) \tilde\Y_{2,\eta}(t,\theta) d\theta\Big) d\eta\nn \\
& \quad + \frac{2}{\A^5} \int_0^1 \tilde\P_2\tilde\Y_{2,\eta}^2(t,\eta)\Big( \int_0^\eta e^{-\frac1{2\sqtC{2}}(\tilde\Y_2(t,\eta)-\tilde\Y_2(t,\theta))}\nn \\
&\qquad\quad\times\Big(\int_0^\theta e^{-\frac{1}{\sqtC{2}}(\tilde\Y_2(t,\theta)-\tilde\Y_2(t,l))} (\tilde\Y_1-\tilde\Y_2)^2(t,l)dl\Big) \tilde\Y_{2,\eta}(t,\theta) d\theta\Big) d\eta\nn \\
& \quad + \frac{1}{\A^5}\int_0^1 \tilde\P_2\tilde\Y_{2,\eta}^2(t,\eta)\Big(\int_0^\eta e^{-\frac1{2\sqtC{2}}(\tilde\Y_2(t,\eta)-\tilde\Y_2(t,\theta))}\Big(\int_0^\theta e^{-\frac{1}{\sqtC{2}}(\tilde\Y_2(t,\theta)-\tilde\Y_2(t,l))}dl\Big)\nn \\
& \qquad \quad  \times\Big(\int_0^\theta e^{-\frac{1}{\sqtC{2}}(\tilde\Y_2(t,\theta)-\tilde\Y_2(t,l))} (\tilde\Y_1-\tilde\Y_2)^2(t,l)dl\Big) \tilde\Y_{2,\eta}(t,\theta) d\theta\Big)d\eta\nn \\
& \quad + \frac{128}{3e^2\A^5}\int_0^1 \tilde \P_2\tilde \Y_{2,\eta}^2(t,\eta)\Big(\int_0^\eta e^{-\frac{1}{2\A}(\tilde \Y_2(t,\eta)-\tilde \Y_2(t,\theta))} \nn \\
&\qquad\quad \times\Big(\int_0^\theta e^{-\frac{3}{4\A}(\tilde \Y_2(t,\theta)-\tilde \Y_2(t,l))}dl\Big) \tilde \Y_{2,\eta}(t,\theta) d\theta\Big) d\eta\vert \sqtC{1}-\sqtC{2}\vert^2\nn \\
& \leq 5\norm{\tilde\Y_1-\tilde\Y_2}^2\nn \\
& \quad +\frac{\A}{4} \int_0^1 \tilde\Y_{2,\eta}(t,\eta)e^{-\frac{1}{\sqtC{2}}\tilde\Y_2(t,\eta)}\Big(\int_0^\eta e^{\frac{1}{\sqtC{2}}\tilde\Y_2(t,\theta)}(\tilde\Y_1-\tilde\Y_2)^2(t,\theta) d\theta\Big) d\eta\nn \\
& \quad + \frac{16}{\A^2} \int_0^1 \tilde\Y_{2,\eta}(t,\eta)\Big(\int_0^\eta \tilde\Y_{2,\eta}(t,\theta)e^{-\frac1{2\A} \tilde\Y_2(t,\theta) }\nn \\
&\qquad\quad\times\Big( \int_0^\theta e^{-(\frac1{2\A}\tilde\Y_2(t,\eta)- \frac{1}{\A}\tilde\Y_2(t,l))}(\sqP{1}-\sqP{2})^2(t,l) dl\Big) d\theta\Big) d\eta\nn \\
& \quad +  \int_0^1 \tilde\Y_{2,\eta}(t,\eta)\Big(\int_0^\eta \tilde\Y_{2,\eta}(t,\theta)e^{-\frac1{2\sqtC{2}} \tilde\Y_2(t,\theta) }\nn \\
&\qquad\quad\times\Big( \int_0^\theta e^{-(\frac1{2\sqtC{2}}\tilde\Y_2(t,\eta)- \frac{1}{\sqtC{2}}\tilde\Y_2(t,l))}(\tilde\Y_1-\tilde\Y_2)^2(t,l) dl\Big) d\theta\Big) d\eta\nn \\
& \quad +\frac{1}{2} \int_0^1\tilde\Y_{2,\eta}(t,\eta)\Big(\int_0^\eta \tilde\Y_{2,\eta}(t,\theta)e^{-\frac1{2\sqtC{2}} \tilde\Y_2(t,\theta)}\nn \\
&\qquad\quad\times \Big( \int_0^\theta e^{-(\frac1{2\sqtC{2}}\tilde\Y_2(t,\eta)-\frac{1}{\sqtC{2}}\tilde\Y_2(t,l))}dl\Big) d\theta\Big) d\eta\norm{\tilde\Y_1-\tilde\Y_2}^2\nn \\
& \quad + \frac{64}{3e^2}\int_0^1 \tilde \Y_{2,\eta}(t,\eta)\Big(\int_0^\eta \tilde \Y_{2,\eta}(t,\theta) e^{-\frac{1}{4\A}\tilde \Y_2(t,\theta)} \nn \\
&\qquad\quad\times\Big( \int_0^\theta e^{-(\frac{1}{2\A}\tilde \Y_2(t,\eta)-\frac{3}{4\A}\tilde \Y_2(t,l))}dl \Big) d\theta\Big) d\eta \vert \sqtC{1}-\sqtC{2}\vert^2\nn \\
& \leq 5\norm{\tilde\Y_1-\tilde\Y_2}^2 +\frac{\A}{4} \Big[-\sqtC{2}\int_0^\eta e^{-\frac{1}{\sqtC{2}}(\tilde\Y_2(t,\eta)-\tilde\Y_2(t,\theta))}(\tilde\Y_1-\tilde\Y_2)^2(t,\theta) d\theta\Big\vert_{\eta=0}^1\nn \\
&\qquad \qquad \qquad\qquad\qquad\qquad \quad +\int_0^1 (\tilde\Y_1-\tilde\Y_2)^2(t,\eta) d\eta\Big]\nn \\
& \quad +\frac{16}{\A} \int_0^1 \tilde\Y_{2,\eta}(t,\eta)\nn \\
&\qquad\times\Big[-2 \int_0^\theta e^{\frac{1}{\A}\tilde\Y_2(t,l)-\frac1{2\A}(\tilde\Y_2(t,\eta)+\tilde\Y_2(t,\theta))} (\sqP{1}-\sqP{2})^2(t,l) dl\Big\vert_{\theta=0}^\eta\nn \\
& \qquad \quad  +2 \int_0^\eta e^{\frac1{2\A}(\tilde\Y_2(t,\theta)-\tilde\Y_2(t,\eta))}(\sqP{1}-\sqP{2})^2(t,\theta) d\theta\Big]d\eta\nn \\
& \quad +\sqtC{2} \int_0^1 \tilde\Y_{2,\eta}(t,\eta)\nn \\
&\qquad\times\Big[-2 \int_0^\theta e^{\frac{1}{\sqtC{2}}\tilde\Y_2(t,l)-\frac1{2\sqtC{2}}(\tilde\Y_2(t,\eta)+\tilde\Y_2(t,\theta))} (\tilde\Y_1-\tilde\Y_2)^2(t,l) dl\Big\vert_{\theta=0}^\eta\nn \\
& \qquad \quad +2 \int_0^\eta e^{\frac1{2\sqtC{2}}(\tilde\Y_2(t,\theta)-\tilde\Y_2(t,\eta))}(\tilde\Y_1-\tilde\Y_2)^2(t,\theta) d\theta\Big]d\eta\nn \\
& \quad + \frac{\sqtC{2}}{2}\int_0^1 \tilde\Y_{2,\eta}(t,\eta) \Big[ -2\int_0^\theta e^{\frac{1}{\sqtC{2}}\tilde\Y_2(t,l)-\frac1{2\sqtC{2}}(\tilde\Y_2(t,\theta)+\tilde\Y_2(t,\eta))}dl\Big\vert_{\theta=0}^\eta\nn  \\
& \qquad \quad  +2 \int_0^\eta e^{\frac1{2\sqtC{2}}(\tilde\Y_2(t,\theta)-\tilde\Y_2(t,\eta))}d\theta\Big]d\eta\norm{\tilde\Y_1-\tilde\Y_2}^2\nn \\
& \quad +\frac{64\A}{3e^2} \int_0^1 \tilde \Y_{2,\eta}(t,\eta)\Big[-4\int_0^\theta e^{\frac{3}{4A} \tilde \Y_2(t,l)-(\frac1{4\A}\tilde \Y_2(t,\theta)+\frac{1}{2\A}\tilde \Y_2(t,\eta))}\Big\vert_{\theta=0}^\eta\nn \\
& \qquad \quad  + 4\int_0^\eta e^{\frac{1}{2\A} (\tilde \Y_2(t,\theta)-\tilde \Y_2(t,\eta))}d\theta\Big]d\eta \vert \sqtC{1}-\sqtC{2}\vert^2\nn \\
& \leq \bigO(1)\norm{\tilde\Y_1-\tilde\Y_2}^2\nn \\
& \quad +\frac{32}{\A}\int_0^1 \tilde\Y_{2,\eta}(t,\eta) e^{-\frac1{2\A} \tilde\Y_2(t,\eta)}\Big(\int_0^\eta e^{\frac1{2\A}\tilde\Y_2(t,\theta)}(\sqP{1}-\sqP{2})^2(t,\theta) d\theta\Big) d\eta\nn \\
& \quad +2\A\int_0^1 \tilde\Y_{2,\eta}(t,\eta) e^{-\frac1{2\sqtC{2}} \tilde\Y_2(t,\eta)}\Big(\int_0^\eta e^{\frac1{2\sqtC{2}}\tilde\Y_2(t,\theta)}(\tilde\Y_1-\tilde\Y_2)^2(t,\theta) d\theta\Big) d\eta\nn \\
&\quad +\A \int_0^1\tilde\Y_{2,\eta}(t,\eta)e^{-\frac1{2\sqtC{2}} \tilde\Y_2(t,\eta)}\Big(\int_0^\eta e^{\frac1{2\sqtC{2}} \tilde\Y_2(t,\theta)}d\theta\Big) d\eta\norm{\tilde\Y_1-\tilde\Y_2}^2\nn \\
& \quad + \frac{256\A}{3e^2}\int_0^1 \tilde \Y_{2,\eta}(t,\eta)e^{-\frac{1}{2\A}\tilde \Y_2(t,\eta)}\Big(\int_0^\eta e^{\frac{1}{2\A}\tilde \Y_2(t,\theta)}d\theta\Big) d\eta \vert \sqtC{1}-\sqtC{2}\vert^2\nn \\
& \leq \bigO(1)\norm{\tilde\Y_1-\tilde\Y_2}^2\nn \\
& \quad +32\Big[ -2 \int_0^\eta e^{\frac1{2\A}(\tilde\Y_2(t,\theta)-\tilde\Y_2(t,\eta))} (\sqP{1}-\sqP{2})^2(t,\theta)d\theta\Big\vert_{\eta=0}^1\nn \\
&\qquad\quad+2\int_0^1 (\sqP{1}-\sqP{2})^2(t,\eta)d\eta\Big]\nn \\
& \quad +2\A\sqtC{2}\Big[ -2 \int_0^\eta e^{\frac1{2\sqtC{2}}(\tilde\Y_2(t,\theta)-\tilde\Y_2(t,\eta))} (\tilde\Y_1-\tilde\Y_2)^2(t,\theta)d\theta\Big\vert_{\eta=0}^1\nn \\
&\qquad\quad+2\int_0^1 (\tilde\Y_1-\tilde\Y_2)^2(t,\eta)d\eta\Big]\nn \\
& \quad + \A\sqtC{2}\Big[ -2 \int_0^\eta e^{\frac1{2\sqtC{2}}(\tilde\Y_2(t,\theta)-\tilde\Y_2(t,\eta))}d\theta\Big\vert_{\eta=0}^1 + 2\int_0^1 d\eta\Big] \norm{\tilde\Y_1-\tilde\Y_2}^2\nn \\
& \quad + \frac{256\A^2}{3e^2}\Big[ -2 \int_0^\eta e^{\frac1{2\A}(\tilde\Y_2(t,\theta)-\tilde\Y_2(t,\eta))}d\theta\Big\vert_{\eta=0}^1 + 2\int_0^1 d\eta\Big] \vert\sqtC{1}-\sqtC{2}\vert^2\nn \\ %\label{eq:B313}
& \leq \bigO(1) (\norm{\tilde\Y_1-\tilde\Y_2}^2+\norm{\sqP{1}-\sqP{2}}^2+\vert \sqtC{1}-\sqtC{2}\vert^2).
\end{align*}

Last, but not least, direct calculations for $B_{314}$ yield
\begin{align*}
B_{314}&\leq \frac6{\A^3}\vert \sqtC{1}-\sqtC{2}\vert \int_0^1 \vert \tilde\Y_1-\tilde\Y_2\vert \tilde\Y_{2,\eta}(t,\eta) \\
& \quad  \times\Big(\int_0^\eta e^{-\frac{1}{\sqtC{2}}(\tilde\Y_2(t,\eta)-\tilde\Y_2(t,\theta))}\Big(\int_0^\theta e^{-\frac{3}{4\A}(\tilde\Y_2(t,\theta)-\tilde\Y_2(t,l))}dl\Big) \tilde\V_2^+\tilde\Y_{2,\eta}(t,\theta)d\theta\Big)d\eta\\
& \leq \frac6{\A^3}\vert \sqtC{1}-\sqtC{2}\vert \norm{\tilde\Y_1-\tilde\Y_2}\\
& \qquad  \times\Big(\int_0^1 \tilde\Y_{2,\eta}^2(t,\eta)\Big(\int_0^\eta e^{-\frac{1}{\sqtC{2}}(\tilde\Y_2(t,\eta)-\tilde\Y_2(t,\theta))}\\
&\qquad\qquad\qquad\times\Big(\int_0^\theta e^{-\frac{3}{4\A}(\tilde\Y_2(t,\theta)-\tilde\Y_2(t,l))}dl\Big) \tilde\V_2^+\tilde\Y_{2,\eta}(t,\theta)d\theta\Big)^2 d\eta\Big)^{1/2}\\
& \leq \frac{6}{\A^3}\vert \sqtC{1}-\sqtC{2}\vert \norm{\tilde\Y_1-\tilde\Y_2}\\
& \quad  \times \Big(\int_0^1 \tilde\Y_{2,\eta}^2(t,\eta) \Big(\int_0^\eta e^{-\frac{1}{\sqtC{2}}(\tilde\Y_2(t,\eta)-\tilde\Y_2(t,\theta))} \tilde\U_2^2\tilde\Y_{2,\eta}(t,\theta) d\theta\Big)\\
& \quad   \times\Big(\int_0^\eta e^{-\frac{1}{\sqtC{2}}(\tilde\Y_2(t,\eta)-\tilde\Y_2(t,\theta))} \Big(\int_0^\theta e^{-\frac{3}{4\A}(\tilde\Y_2(t,\theta)-\tilde\Y_2(t,l))} dl\Big)^2 \tilde\Y_{2,\eta}(t,\theta)d\theta\Big) d\eta\Big)^{1/2}\\
& \leq \frac{12}{\A^{5/2}} \vert \sqtC{1}-\sqtC{2}\vert \norm{\tilde\Y_1-\tilde\Y_2}\\
& \qquad  \times \Big(\int_0^1 \tilde\P_2\tilde\Y_{2,\eta}^2 (t,\eta)\Big(\int_0^\eta e^{-\frac1{2\A}(\tilde\Y_2(t,\eta)-\tilde\Y_2(t,\theta))}\\
&\qquad\qquad\qquad\qquad\times\Big(\int_0^\theta e^{-\frac{3}{4\A}(\tilde\Y_2(t,\theta)-\tilde\Y_2(t,l))}dl\Big) \tilde\Y_{2,\eta}(t,\theta) d\theta\Big) d\eta\Big)^{1/2}\\
& \leq 6\sqrt{2}\vert \sqtC{1}-\sqtC{2}\vert \norm{\tilde\Y_1-\tilde\Y_2}\\
& \qquad \times \Big(\int_0^1 \tilde\Y_{2,\eta}(t,\eta) \Big(\int_0^\eta \tilde\Y_{2,\eta}(t,\theta) e^{-\frac1{4\A} \tilde\Y_2(t,\theta)}\\
&\qquad\qquad\qquad\qquad\times\Big(\int_0^\theta e^{\frac{3}{4\A}\tilde\Y_2(t,l)-\frac1{2\A} \tilde\Y_2(t,\eta)}dl\Big) d\theta\Big) d\eta\Big)^{1/2}\\
& \leq \bigO(1)\vert \sqtC{1}-\sqtC{2}\vert \norm{\tilde\Y_1-\tilde\Y_2}.
\end{align*}
Collecting the estimates we have obtained, we conclude
\begin{equation*}%\label{eq:barB31_final}
\vert \bar B_{31}\vert \leq \bigO(1)\big(\norm{\tilde\U_1-\tilde\U_2}^2+ \norm{\tilde\Y_1-\tilde\Y_2}^2+\norm{\sqP{1}-\sqP{2}}^2+\abs{\sqtC{1}-\sqtC{2}}^2\big).
\end{equation*}

Recalling \eqref{eq:343} direct computation yield for $\bar B_{33}$ (and in much the same way for $\bar B_{34}$) that 
\begin{align*}
\vert \bar B_{33}\vert &\leq \frac{1}{\A^7}\int_0^1 \vert \tilde\Y_1-\tilde\Y_2\vert \tilde\Y_{2,\eta}(t,\eta)\notag\\
&\qquad\qquad\times\Big(\int_0^\eta e^{-\frac{1}{\sqtC{2}}(\tilde\Y_2(t,\eta)-\tilde\Y_2(t,\theta))}\min_j(\tilde\D_j)\vert \tilde\U_1-\tilde\U_2\vert \tilde\Y_{2,\eta}(t,\theta) d\theta\Big)d\eta\notag\\
& \leq \norm{\tilde\Y_1-\tilde\Y_2}^2 + \frac{1}{\A^{14}} \int_0^1\tilde\Y_{2,\eta}^2(t,\eta)\nn \\ 
&\qquad\qquad\times\Big(\int_0^\eta e^{-\frac{1}{\sqtC{2}}(\tilde\Y_2(t,\eta)-\tilde\Y_2(t,\theta))} \min_j(\tilde\D_j)\tilde\Y_{2,\eta} \vert \tilde\U_1-\tilde\U_2\vert (t,\theta) d\theta\Big)^2d\eta\notag\\
& \leq \norm{\tilde\Y_1-\tilde\Y_2}^2\notag\\
& \quad +\frac{4}{\A^{12}}\int_0^1 \tilde \Y_{2,\eta}^2(t,\eta)\Big(\int_0^\eta e^{-\frac3{2\sqtC{2}}(\tilde\Y_2(t,\eta)-\tilde\Y_2(t,\theta))}\tilde\P_2\tilde\Y_{2,\eta}(t,\theta) d\theta\Big)\notag\\
& \quad \qquad \qquad \qquad \qquad \times\Big(\int_0^\eta e^{-\frac1{2\sqtC{2}}(\tilde\Y_2(t,\eta)-\tilde\Y_2(t,\theta))}\tilde\P_2\tilde\Y_{2,\eta}(\tilde\U_1-\tilde\U_2)^2(t,\theta)d\theta\Big) d\eta\notag\\
& \leq \norm{\tilde\Y_1-\tilde\Y_2}^2\notag\\
& \quad + \frac{4}{\A^{6}} \int_0^1 \tilde\P_2\tilde\Y_{2,\eta}^2(t,\eta)\Big(\int_0^\eta e^{-\frac1{2\sqtC{2}} (\tilde\Y_2(t,\eta)-\tilde\Y_2(t,\theta))}(\tilde\U_1-\tilde\U_2)^2(t,\theta) d\theta\Big) d\eta\notag\\
& \leq \norm{\tilde\Y_1-\tilde\Y_2}^2\notag\\
& \quad + \frac{2}{\A}\int_0^1 \tilde\Y_{2,\eta}(t,\eta)e^{-\frac1{2\sqtC{2}} \tilde\Y_2(t,\eta)}\Big(\int_0^\eta e^{\frac1{2\sqtC{2}}\tilde\Y_2(t,\theta)}(\tilde\U_1-\tilde\U_2)^2(t,\theta) d\theta\Big) d\eta\notag\\
& \leq \norm{\tilde\Y_1-\tilde\Y_2}^2\notag\\
& \quad + \frac{2\sqtC{2}}{\A} \Big[ -2 \int_0^\eta e^{-\frac1{2\sqtC{2}}(\tilde\Y_2(t,\eta)-\tilde\Y_2(t,\theta))} (\tilde\U_1-\tilde\U_2)^2(t,\theta)d\theta\Big\vert_{\eta=0}^1\notag\\
&\qquad\qquad\qquad\qquad\qquad\qquad\qquad+2 \int_0^1 (\tilde \U_1-\tilde\U_2)^2(t,\eta) d\eta\Big]\notag\\
& \leq \bigO(1)\Big(\norm{\tilde\Y_1-\tilde\Y_2}^2+\norm{\tilde\U_1-\tilde\U_2}^2\Big).%\label{eq:barB33}
\end{align*}

Direct calculations yield for $\bar B_{35}$ (and in much the same way for $\bar B_{36}$) that 
\begin{align}
\vert \bar B_{35}\vert & \leq \frac{1}{\A^7} \int_0^1 \vert \tilde\Y_1-\tilde\Y_2\vert \tilde\Y_{2,\eta}(t,\eta) \notag\\
& \qquad  \times\Big(\int_0^\eta (e^{-\frac{1}{\sqtC{2}}(\tilde\Y_2(t,\eta)-\tilde\Y_2(t,\theta))}-e^{-\frac{1}{\sqtC{2}}(\tilde\Y_1(t,\eta)-\tilde\Y_1(t,\theta))})\notag\\
&\qquad\qquad\qquad\qquad\qquad\times\min_j(\tilde\D_j)\min_j(\tilde\V_j^+)\tilde\Y_{2,\eta} \mathbbm{1}_{B(\eta)}(t,\theta) d\theta\Big)d\eta\notag\\
& \leq \frac{1}{\A^7\sqtC{2}} \int_0^1 \vert \tilde\Y_1-\tilde\Y_2\vert \tilde\Y_{2,\eta}(t,\eta) \Big(\int_0^\eta \big(\vert \tilde\Y_1-\tilde\Y_2\vert (t,\eta)+\vert \tilde\Y_1-\tilde\Y_2\vert (t,\theta)\big)\notag\\
&\qquad\qquad\qquad\qquad\qquad\times e^{-\frac{1}{\sqtC{2}}(\tilde\Y_2(t,\eta)-\tilde\Y_2(t,\theta))}\tilde\D_2\tilde\V_2^+\tilde\Y_{2,\eta}(t,\theta) d\theta\Big) d\eta\notag\\
& \leq \frac1{\A^7\sqtC{2}} \int_0^1 (\tilde\Y_1-\tilde\Y_2)^2 \tilde\Y_{2,\eta} (t,\eta)\nn \\
&\qquad\qquad\qquad\qquad\qquad\times\Big(\int_0^\eta e^{-\frac{1}{\sqtC{2}}(\tilde\Y_2(t,\eta)-\tilde\Y_2(t,\theta))}\tilde\D_2\tilde\U_2^+\tilde\Y_{2,\eta}(t,\theta) d\theta\Big) d\eta\notag\\
& \quad +\frac1{\A^7\sqtC{2}}\int_0^1 \vert \tilde\Y_1-\tilde\Y_2\vert \tilde\Y_{2,\eta}(t,\eta)\notag\\
&\qquad\qquad\qquad\times\Big(\int_0^\eta e^{-\frac{1}{\sqtC{2}}(\tilde\Y_2(t,\eta)-\tilde\Y_2(t,\theta))} \tilde\D_2\tilde\V_2^+\tilde\Y_{2,\eta}\vert \tilde\Y_1-\tilde\Y_2\vert (t,\theta)d\theta\Big) d\eta\notag\\
& \leq \frac{2}{\A^7} \int_0^1 (\tilde\Y_1-\tilde\Y_2)^2 \tilde\Y_{2,\eta}(t,\eta)\Big(\int_0^\eta e^{-\frac{1}{\sqtC{2}}(\tilde\Y_2(t,\eta)-\tilde\Y_2(t,\theta))} \tilde\P_2^2\tilde\Y_{2,\eta}(t,\theta) d\theta\Big)^{1/2} \notag\\
& \quad \qquad \qquad \qquad  \times \Big(\int_0^\eta e^{-\frac{1}{\sqtC{2}}(\tilde\Y_2(t,\eta)-\tilde\Y_2(t,\theta))}\tilde\U_2^2\tilde\Y_{2,\eta}(t,\theta)d\theta\Big)^{1/2} d\eta\notag\\
& \quad +\norm{\tilde\Y_1-\tilde\Y_2}^2\notag\\
& \quad + \frac{4}{\A^{14}} \int_0^1 \tilde\Y_{2,\eta}^2(t,\eta) \Big(\int_0^\eta e^{-\frac{1}{\sqtC{2}}(\tilde\Y_2(t,\eta)-\tilde\Y_2(t,\theta))}\tilde\U_2^2\tilde\Y_{2,\eta}(t,\theta) d\theta\Big)\notag\\
& \qquad \qquad \qquad   \times \Big(\int_0^\eta e^{-\frac{1}{\sqtC{2}}(\tilde\Y_2(t,\eta)-\tilde\Y_2(t,\theta))} \tilde\P_2^2\tilde\Y_{2,\eta} (\tilde\Y_1-\tilde\Y_2)^2(t,\theta) d\theta\Big) d\eta\notag\\
& \leq \norm{\tilde\Y_1-\tilde\Y_2}^2 +\frac{2\sqrt{6}}{\A^{4}} \int_0^1(\tilde\Y_1-\tilde\Y_2)^2 \tilde\P_2\tilde\Y_{2,\eta}(t,\eta)d\eta\notag\\
& \quad + \frac {16}{\A^{13}}\int_0^1 \tilde\P_2\tilde\Y_{2,\eta}^2(t,\eta)\nn \\
&\qquad\qquad\qquad\times\Big(\int_0^\eta e^{-\frac{1}{\sqtC{2}}(\tilde\Y_2(t,\eta)-\tilde\Y_2(t,\theta))}\tilde\P_2^2\tilde\Y_{2,\eta}(\tilde\Y_1-\tilde\Y_2)^2(t,\theta) d\theta\Big) d\eta\notag\\
& \leq (1+\sqrt{6}\A) \norm{\tilde\Y_1-\tilde\Y_2}^2\notag\\
& \quad + \A\int_0^1 \tilde\Y_{2,\eta}(t,\eta) e^{-\frac{1}{\sqtC{2}}\tilde\Y_2(t,\eta)}\Big(\int_0^\eta e^{\frac{1}{\sqtC{2}}\tilde\Y_2(t,\theta)} (\tilde\Y_1-\tilde\Y_2)^2(t,\theta) d\theta\Big)d \eta\notag\\
& = (1+\sqrt{6}\A)\norm{\tilde\Y_1-\tilde\Y_2}^2\notag\\
& \quad +\A\sqtC{2} \Big[-\int_0^\eta e^{-\frac{1}{\sqtC{2}}(\tilde\Y_2(t,\eta)-\tilde\Y_2(t,\theta))}(\tilde\Y_1-\tilde\Y_2)^2(t,\theta) d\theta\Big\vert_{\eta=0}^1\notag\\
&\qquad\qquad\qquad\qquad\qquad\qquad\qquad\qquad\qquad+\int_0^1 (\tilde\Y_1-\tilde\Y_2)^2(t,\eta) d\eta\Big]\notag\\
& \leq \bigO(1) \norm{\tilde\Y_1-\tilde\Y_2}^2.  \label{eq:barB35}
\end{align}

Direct calculations yield for $\hat B_{35}$ (and much in the same way for $\hat B_{36}$) that 
\begin{align*} \nn
\vert \hat B_{35}\vert & \leq \mathbbm{1}_{\sqtC{1}\leq\sqtC{2}}\frac{4\vert \sqtC{1}-\sqtC{2}\vert}{\ma\A^7 e} \int_0^1 \vert \tilde \Y_1-\tilde \Y_2\vert \tilde \Y_{2,\eta}(t,\eta)\\ \nn
& \qquad \qquad \times \Big(\int_0^\eta \min_j(e^{-\frac{3}{4\A}(\tilde \Y_j(t,\eta)-\tilde \Y_j(t,\theta))})\min_j(\tilde \D_j)\min_j(\tilde \V_j^+)\tilde \Y_{2,\eta}(t,\theta)d\theta\Big) d\eta\\ \nn
& \leq \frac{8}{\ma\A^6e} \int_0^1 \vert \tilde \Y_1-\tilde \Y_2\vert \tilde \Y_{2,\eta}(t,\eta)\\ \nn
& \qquad \qquad \times \Big(\int_0^\eta e^{-\frac{1}{\sqtC{2}}(\tilde \Y_2(t,\eta)-\tilde \Y_2(t,\theta))} \tilde \U_2^2\tilde \Y_{2,\eta}(t,\theta) d\theta\Big)^{1/2}\\ \nn
& \qquad \qquad \times \Big(\int_0^\eta e^{-\frac{1}{2\sqtC{2}}(\tilde \Y_2(t,\eta)-\tilde \Y_2(t,\theta))}\tilde \P_1\tilde \P_2\tilde \Y_{2,\eta} (t,\theta) d\theta \Big)^{1/2}d\eta \vert \sqtC{1}-\sqtC{2}\vert\\ \nn
& \leq \frac{16}{a\A^{11/2} e} \int_0^1 \vert \tilde \Y_1-\tilde \Y_2\vert \sqP{2}\tilde \Y_{2,\eta}(t,\eta)\\ \nn
& \qquad \qquad \times \Big( \int_0^\eta e^{-\frac{1}{2\sqtC{2}}(\tilde\Y_2(t,\eta)-\tilde \Y_2(t,\theta))}\tilde\P_1\tilde \P_2\tilde \Y_{2,\eta}(t,\theta)d\theta\Big)^{1/2}d\eta \vert \sqtC{1}-\sqtC{2}\vert\\ \nn
& \leq \norm{\tilde \Y_1-\tilde \Y_2}^2 \\ \nn
& \qquad + \frac{256}{\ma^2 \A^{11}e^2}\int_0^1 \tilde \P_2\tilde \Y_{2,\eta}^2(t,\eta)\\ \nn
& \qquad \qquad \times \Big(\int_0^\eta e^{-\frac{1}{2\sqtC{2}}(\tilde \Y_2(t,\eta)-\tilde \Y_2(t,\theta))}\tilde \P_1\tilde\P_2\tilde \Y_{2,\eta}(t,\theta) d\theta\Big) d\eta \vert \sqtC{1}-\sqtC{2}\vert^2\\ \nn
& \leq \norm{\tilde \Y_1-\tilde \Y_2}^2\\ \nn
& \qquad + \frac{16\A}{e^2}\int_0^1 \tilde \Y_{2,\eta}(t,\eta) e^{-\frac{1}{2\sqtC{2}}\tilde \Y_2(t,\eta)}\Big(\int_0^\eta e^{\frac{1}{2\sqtC{2}}\tilde \Y_2(t,\theta)}d\theta\Big) d\eta \vert \sqtC{1}-\sqtC{2}\vert^2\\ \
& \leq \bigO(1)(\norm{\tilde \Y_1-\tilde \Y_2}^2+\vert \sqtC{1}-\sqtC{2}\vert^2).   
\end{align*}

As far as the term $\bar B_{37}$ is concerned it can be estimated as follows (we use \eqref{decay:impl})
\begin{align}
\vert \bar B_{37}\vert &= \frac{1}{\A^7} \vert \int_0^1 (\tilde\Y_1-\tilde\Y_2)(\tilde\Y_{1,\eta}-\tilde\Y_{2,\eta})(t,\eta)\notag\\
& \qquad  \times\min_k\Big(\int_0^\eta \min_j(e^{-\frac{1}{\ma}(\tilde \Y_j(t,\eta)-\tilde\Y_j(t,\theta))})\notag\\
&\qquad\qquad\qquad\qquad\qquad\qquad\qquad\times\min_j(\tilde\D_j)\min_j(\tilde\V_j^+)\tilde\Y_{k,\eta}(t,\theta)d\theta\Big) d\eta\vert \notag\\
& = \frac{1}{2\A^7} \vert (\tilde\Y_1-\tilde\Y_2)^2(t,\eta)\min_k\Big(\int_0^\eta \min_j(e^{-\frac{1}{\ma}(\tilde \Y_j(t,\eta)-\tilde\Y_j(t,\theta))})\notag\\
&\qquad\qquad\qquad\qquad\qquad\qquad\qquad\times\min_j(\tilde\D_j)\min_j(\tilde\V_j^+)\tilde\Y_{k,\eta}(t,\theta)d\theta\Big) \Big\vert_{\eta=0}^1\notag\\
& \quad -\int_0^1 (\tilde\Y_1-\tilde\Y_2)^2(t,\eta)\frac{d}{d\eta}\min_k\Big(\int_0^\eta \min_j(e^{-\frac{1}{\ma}(\tilde \Y_j(t,\eta)-\tilde\Y_j(t,\theta))})\notag\\
&\qquad\qquad\qquad\qquad\qquad\qquad\qquad\times\min_j(\tilde\D_j)\min_j(\tilde\V_j^+)\tilde\Y_{k,\eta}(t,\theta)d\theta\Big)d\eta \vert\notag \\
& =\frac1{2\A^7} \vert \int_0^1 (\tilde\Y_1-\tilde\Y_2)^2(t,\eta)
\frac{d}{d\eta}\min_k\Big(\int_0^\eta \min_j(e^{-\frac{1}{\ma}(\tilde \Y_j(t,\eta)-\tilde\Y_j(t,\theta))})\notag\\
&\qquad\qquad\qquad\qquad\qquad\qquad\qquad\times\min_j(\tilde\D_j)\min_j(\tilde\V_j^+)\tilde\Y_{k,\eta}(t,\theta)d\theta\Big)d\eta \vert \notag\\
&\leq \bigO(1)\norm{\tilde\Y_1-\tilde\Y_2}^2,  \label{eq:barB37}
\end{align}
where $\bigO(1)$ denotes some constant only depending on $\A$, which remains bounded as $\A\to 0$, provided we can show that the derivative  in the latter integral exists and is uniformly bounded, see Lemma \ref{lemma:4}. 

As far as the term $\bar B_{38}$ (a similar argument works for $\bar B_{39}$) is concerned, the integral can be estimated as follows,
\begin{align*}
\bar B_{38}&= \frac{1}{\A^7} \int_0^1 (\tilde \Y_1-\tilde \Y_2)\tilde \Y_{2,\eta} \mathbbm{1}_{D^c} (t,\eta)
\Big(\int_0^\eta \min_j(e^{-\frac{1}{\ma}(\tilde \Y_j(t,\eta)-\tilde\Y_j(t,\theta))})\\
&\qquad\qquad\qquad\qquad\times\min_j(\tilde\D_j)\min_j(\tilde\V_j^+) (\tilde \Y_{2,\eta}-\tilde \Y_{1,\eta}) (t,\theta) d\theta\Big)d\eta\\
& =\frac{1}{\A^7} \int_0^1 (\tilde \Y_1-\tilde \Y_2) \tilde \Y_{2,\eta} \mathbbm{1}_{D^c}(t,\eta)
\Big[(\tilde \Y_2-\tilde \Y_1)(t,\theta) \min_j(e^{-\frac{1}{\ma}(\tilde \Y_j(t,\eta)-\tilde\Y_j(t,\theta))})\\
&\qquad\qquad\qquad\qquad\qquad\qquad\qquad\times\min_j(\tilde\D_j) \min_j(\tilde\V_j^+) (t,\theta) \Big\vert_{\theta=0}^\eta\\
& \qquad -\int_0^\eta (\tilde \Y_2-\tilde \Y_1)(t,\theta)
\big[(\frac{d}{d\theta}\min_j(e^{-\frac{1}{\ma}(\tilde \Y_j(t,\eta)-\tilde\Y_j(t,\theta))})) \\
&\qquad\qquad\qquad\qquad\qquad\qquad\qquad\times\min_j(\tilde \D_j)\min_j(\tilde\V_j^+)(t,\theta)\\
& \qquad  + \min_j(e^{-\frac{1}{\ma}(\tilde \Y_j(t,\eta)-\tilde\Y_j(t,\theta))})\big(\frac{d}{d\theta} \min_j(\tilde\D_j)\min_j(\tilde\V_j^+)\big)(t,\theta)d\theta\big]d\eta\Big]\\
& = -\frac{1}{\A^7}\int_0^1 (\tilde\Y_1-\tilde \Y_2)^2 \min_j(\tilde\D_j) \min_j(\tilde\V_j^+)\tilde \Y_{2,\eta} \mathbbm{1}_{D^c}(t,\eta) d\eta\\
& \quad +\frac{1}{\A^7}\int_0^1 (\tilde \Y_1-\tilde \Y_2)\tilde \Y_{2,\eta} (t,\eta) \mathbbm{1}_{D^c}(t,\eta)\\
& \qquad \times\int_0^\eta (\tilde \Y_1-\tilde \Y_2)(t,\theta)
\Big[\big(\frac{d}{d\theta}\min_j(e^{-\frac{1}{\ma}(\tilde \Y_j(t,\eta)-\tilde\Y_j(t,\theta))})\big) \\
&\qquad\qquad\qquad\qquad\qquad\qquad\qquad\times\min_j(\tilde \D_j)\min_j(\tilde\V_j^+)(t,\theta)\\
& \qquad  + \min_j(e^{-\frac{1}{\ma}(\tilde \Y_j(t,\eta)-\tilde\Y_j(t,\theta))})\big(\frac{d}{d\theta} \min_j(\tilde\D_j)\min_j(\tilde\V_j^+)\big)(t,\theta)d\theta\Big]d\eta\\
& = \bar M_1+\bar M_2.
\end{align*}

As far as the first term $\bar M_1$ is concerned, we have since 
\begin{align*}
\min_j(\tilde \D_j) \min_j(\tilde\V_j^+)\tilde \Y_{i,\eta}(t,\eta)
& \leq2 \min_j(\sqtC{j}\tilde\P_j) \min_j(\tilde\V_j^+)\tilde \Y_{i,\eta}(t,\eta)\\
&\leq \A^6\min_j(\tilde\V_j^+)(t,\eta)\leq \frac{\A^8}{\sqrt{2}},
\end{align*}
that
\begin{align*} 
\vert \bar M_1\vert & \leq \frac{1}{\A^7} \int_0^1 (\tilde \Y_1-\tilde \Y_2)^2\min_j(\tilde\D_j) \min_j(\tilde\V_j^+)\tilde \Y_{2,\eta} (t,\eta) d\eta\\
& \leq \frac{\A}{\sqrt{2}}\norm{\tilde \Y_1-\tilde \Y_2}^2= \bigO(1) \norm{\tilde \Y_1-\tilde \Y_2}^2,
\end{align*}
where $\bigO(1)$ denotes some constant, which only depends on $\A$ and which remains bounded as $\A\to 0$. 

The second term $\bar M_2$, on the other hand, is a bit more demanding.

(i): First of all recall \eqref{est:L3a_lemma}, i.e., 
\begin{equation*}
\vert \frac{d}{d\theta} \min_j (e^{-\frac{1}{\ma}(\tilde \Y_j(t,\eta)-\tilde \Y_j(t,\theta))})\vert \leq \frac{1}{\ma}\min_j(e^{-\frac{1}{\ma}(\tilde \Y_j(t,\eta)-\tilde\Y_j(t,\theta))})\max_j(\tilde \Y_{j,\eta})(t,\theta),
\end{equation*}
which implies that
\begin{align*}
\frac{1}{\A^7}& \vert \int_0^1 (\tilde \Y_1-\tilde \Y_2)\tilde \Y_{2,\eta} (t,\eta) \mathbbm{1}_{D^c}(t,\eta) \int_0^\eta (\tilde \Y_1-\tilde \Y_2)(t,\theta)\\
&  \qquad \times (\frac{d}{d\theta} \min_j(e^{-\frac{1}{\ma}(\tilde \Y_j(t,\eta)-\tilde\Y_j(t,\theta))}))
\min_j(\tilde\D_j) \min_j(\tilde\V_j^+)(t,\theta) d\theta d\eta\vert\\
& \leq \norm{\tilde \Y_1-\tilde \Y_2}^2 +\frac{1}{\A^{14}} \int_0^1 \tilde \Y_{2,\eta}^2(t,\eta) \mathbbm{1}_{D^c}(t,\eta)\\
& \quad \times \Big(\int_0^\eta (\tilde \Y_1-\tilde \Y_2)(t,\theta)
 (\frac{d}{d\theta} \min_j(e^{-\frac{1}{\ma}(\tilde \Y_j(t,\eta)-\tilde\Y_j(t,\theta))}))\\
&\qquad\qquad\qquad\qquad\qquad\qquad\times\min_j(\tilde\D_j) \min_j(\tilde \V_j^+) (t,\theta) d\theta\Big)^2 d\eta\\
& \leq \norm{\tilde \Y_1-\tilde \Y_2}^2 +\frac{1}{\ma^2\A^2}\int_0^1 \tilde\Y_{2,\eta}^2 (t,\eta) \\
&  \quad \qquad \times \Big( \int_0^\eta \vert \tilde \Y_1-\tilde \Y_2\vert(t,\theta) e^{-\frac{1}{\sqtC{2}}(\tilde \Y_2(t,\eta)-\tilde \Y_2(t,\theta))}\min_j(\tilde\V_j^+)(t,\theta) d\theta\Big)^2 d\eta\\
& \leq \norm{\tilde\Y_1-\tilde \Y_2}^2\\
& \quad +\frac{\ma^2}{2\A^2}\int_0^1 \tilde \Y_{2,\eta}^2(t,\eta)\Big(\int_0^\eta \vert \tilde \Y_1-\tilde \Y_2\vert (t,\theta)e^{-\frac{1}{\sqtC{2}}(\tilde \Y_2(t,\eta)-\tilde \Y_2(t,\theta))}d\theta\Big)^2 d\eta\\
& \leq \norm{\tilde\Y_1-\tilde \Y_2}^2\\ 
& \quad + \frac12\int_0^1 \tilde \Y_{2,\eta}^2 (t, \eta)\Big( \int_0^\eta (\tilde \Y_1-\tilde \Y_2)^2(t,\theta) e^{-\frac{1}{2\sqtC{2}}(\tilde \Y_2(t,\eta)-\tilde \Y_2(t,\theta))}d\theta\Big)\\
& \quad \qquad \qquad \qquad \times \Big( \int_0^\eta e^{-\frac{3}{2\sqtC{2}}(\tilde \Y_2(t,\eta)-\tilde \Y_2(t,\theta))} d\theta\Big) d\eta\\
& \leq \norm{\tilde \Y_1-\tilde \Y_2}^2\\
& \qquad + \frac{1}{\sqtC{2}^5}\int_0^1\tilde \Y_{2,\eta}^2(t,\eta)\Big(\int_0^\eta (\tilde \Y_1-\tilde \Y_2)^2(t,\theta) e^{-\frac{1}{2\sqtC{2}}(\tilde \Y_2(t,\eta)-\tilde \Y_2(t,\theta))}d\theta\Big)\\
& \qquad \qquad \qquad \qquad \times \Big(\int_0^\eta e^{-\frac{3}{2\sqtC{2}}(\tilde \Y_2(t,\eta)-\tilde \Y_2(t,\theta))}(\tilde \P_2\tilde \Y_{2,\eta}+\tilde \Henergy_{2,\eta})(t,\theta)d\theta\Big) d\eta\\
& \leq \norm{\tilde \Y_1-\tilde \Y_2}^2\\
&\quad+ \frac{6}{\sqtC{2}^4}\int_0^1 \tilde \P_2\tilde \Y_{2,\eta}^2(t,\eta)\Big( \int_0^\eta (\tilde \Y_1-\tilde \Y_2)^2(t,\theta)e^{-\frac{1}{2\sqtC{2}}(\tilde \Y_2(t,\eta)-\tilde\Y_2(t,\theta))}d\theta \Big)d\eta\\
& \leq \norm{\tilde \Y_1-\tilde \Y_2}^2 \\
&\quad+3\A\int_0^1 \tilde \Y_{2,\eta} (t,\eta) e^{-\frac{1}{2\sqtC{2}}\tilde \Y_2(t,\eta)}\Big( \int_0^\eta (\tilde \Y_1-\tilde \Y_2)^2(t,\theta) e^{\frac{1}{2\sqtC{2}}\tilde \Y_2(t,\theta)}d\theta\Big) d\eta\\
& = \norm{\tilde \Y_1-\tilde \Y_2}^2-6\A\sqtC{2}\int_0^\eta e^{-\frac{1}{2\sqtC{2}}(\tilde \Y_2(t,\eta)-\tilde \Y_2(t,\theta))} (\tilde \Y_1-\tilde \Y_2)^2(t,\theta) d\theta\Big\vert_{\eta=0}^1\\
& \quad \qquad \qquad \qquad +6\A\sqtC{2} \int_0^1 (\tilde \Y_1-\tilde \Y_2)^2(t,\eta) d\eta\\
& \leq \bigO(1)\norm{\tilde \Y_1-\tilde \Y_2}^2,
\end{align*}
where we used \eqref{eq:PUYH_scale}, \eqref{eq:343}, and \eqref{eq:Henergy32}. Again, $\bigO(1)$ denotes some constant only depending on $\A$, which remains bounded as $\A\to 0$.

(ii): First of all we have to establish that $\min_j(\tilde\D_j)\min_j(\tilde \V_j^+)(t,\theta)$ is Lipschitz continuous with a uniformly bounded Lipschitz constant.   More precisely, in Lemma \ref{lemma:5} we show that
\begin{equation*}
\vert \frac{d}{d\theta} (\min_j(\tilde\D_j)\min_j(\tilde\V_j^+))(t,\theta)\vert\leq \bigO(1)\sqrt{\A}\A^4 (\min_j(\tilde\D_j)^{1/2}+\vert \tilde\U_2\vert)(t,\theta).
\end{equation*}

We are now ready to establish a Lipschitz estimate for the second part of $\bar M_2$. Indeed, 
\begin{align*}
\frac{1}{\A^7} & \vert \int_0^1 (\tilde \Y_1-\tilde\Y_2)\tilde\Y_{2,\eta}\mathbbm{1}_{D^c} (t,\eta)
 \int_0^\eta (\tilde \Y_1-\tilde\Y_2)(t,\theta) \min_j(e^{-\frac{1}{\ma}(\tilde \Y_j(t,\eta)-\tilde\Y_j(t,\theta))})\\
&\qquad\qquad\qquad\qquad\qquad\qquad\times\left(\frac{d}{d\theta} \min_j(\tilde\D_j)\min_j(\tilde\V_j^+)\right)(t,\theta) d\theta d\eta\vert \\
&  \leq \norm{\tilde \Y_1-\tilde\Y_2}^2 +\frac{1}{\A^{14}}\int_0^1 \tilde \Y_{2,\eta}^2\mathbbm{1}_{D^c} (t,\eta)\\
& \quad \qquad \times\Big( \int_0^\eta (\tilde\Y_1-\tilde\Y_2) (t,\theta)  \min_j(e^{-\frac{1}{\ma}(\tilde \Y_j(t,\eta)-\tilde\Y_j(t,\theta))})\\
&\qquad\qquad\qquad\qquad\qquad\qquad\times\left(\frac{d}{d\theta} \min_j(\tilde\D_j) \min_j(\tilde\V_j^+)\right) (t,\theta) d\theta\Big)^2d\eta\\
& \leq \norm{\tilde \Y_1-\tilde\Y_2}^2\\
& \quad + \bigO(1)\frac{1}{\A^5} \int_0^1 \tilde\Y_{2,\eta}^2(t,\eta)\Big(\int_0^\eta \min_j(e^{-\frac{1}{\ma}(\tilde \Y_j(t,\eta)-\tilde\Y_j(t,\theta))})\vert \tilde \Y_1-\tilde\Y_2\vert \\
& \qquad \qquad\qquad \qquad\qquad \qquad\qquad \qquad \times\big(\min_j(\tilde\D_j)^{1/2} +\vert \tilde \U_2\vert \big)(t,\theta) d\theta\Big)^2d\eta\\
& \leq \norm{\tilde\Y_1-\tilde\Y_2}^2\\
& \quad +\bigO(1)\frac{1}{\A^5}\int_0^1\tilde \Y_{2,\eta}^2 (t,\eta)\Big( \int_0^\eta  e^{-\frac{1}{\sqtC{2}}(\tilde \Y_2(t,\eta)-\tilde \Y_2(t,\theta))}\vert \tilde \Y_1-\tilde \Y_2\vert \tilde \P_2^{1/2}(t,\theta) d\theta\Big)^2 d\eta\\
&  \leq\norm{\tilde\Y_1-\tilde\Y_2}^2\\
& \quad +\bigO(1)\frac{1}{\A^5}\int_0^1 \tilde \Y_{2,\eta}^2(t,\eta) \Big(\int_0^\eta e^{-\frac3{2\sqtC{2}}(\tilde \Y_2(t,\eta)-\tilde \Y_2(t,\theta))} \tilde \P_2(t,\theta) d\theta\Big)\\
&\qquad\qquad\qquad\qquad\qquad\qquad\times\Big( \int_0^\eta e^{-\frac1{2\sqtC{2}}(\tilde \Y_2(t,\eta)-\tilde\Y_2(t,\theta))}(\tilde \Y_1-\tilde\Y_2)^2(t,\theta) d\theta\Big)d\eta\\
& \leq \norm{\tilde \Y_1-\tilde\Y_2}^2\\
& \quad +\bigO(1)\frac{1}{\A^5} \int_0^1 \tilde \P_2\tilde\Y_{2,\eta}^2(t,\eta)\Big(\int_0^\eta e^{-\frac1{2\sqtC{2}}(\tilde\Y_2(t,\eta)-\tilde\Y_2(t,\theta))}(\tilde\Y_1-\tilde\Y_2)^2(t,\theta) d\theta\Big) d\eta\\
& \leq \norm{\tilde \Y_1-\tilde\Y_2}^2\\
& \quad +\bigO(1)\int_0^1 \tilde\Y_{2,\eta}(t,\eta) e^{-\frac1{2\sqtC{2}}\tilde\Y_2(t,\eta)}\Big(\int_0^\eta e^{\frac1{2\sqtC{2}} \tilde \Y_2(t,\theta)}(\tilde \Y_1-\tilde \Y_2)^2(t,\theta) d\theta\Big) d\eta\\
& \leq \norm{\tilde \Y_1-\tilde\Y_2}^2 \\
& \quad + \bigO(1)\sqtC{2}\Big( -2\int_0^\eta e^{-\frac1{2\sqtC{2}}(\tilde \Y_2(t,\eta)-\tilde\Y_2(t,\theta))}(\tilde \Y_1-\tilde\Y_2)^2(t,\theta)d\theta\Big\vert_{\eta=0}^1\\
&\qquad\qquad\qquad\qquad\qquad\qquad\qquad\qquad\qquad+\int_0^1 2(\tilde \Y_1-\tilde \Y_2)^2(t,\eta) d\eta\Big)\\
& \leq\bigO(1) \norm{\tilde \Y_1-\tilde \Y_2}^2,
\end{align*}
where we used \eqref{eq:32P}.
Thus we conclude that
\begin{equation*}
\abs{\bar B_{38}}+ \abs{\bar B_{39}} \le \bigO(1) \norm{\tilde \Y_1-\tilde \Y_2}^2.
\end{equation*}

Next we have a look at $\bar K_{14}$, which can be rewritten as follows
\begin{align*}
\bar K_{14}&= \frac{3}{\A^6} \int_0^1 (\tilde\Y_1-\tilde\Y_2)(t,\eta) \Big(\tilde\Y_{1,\eta}(t,\eta)\int_0^\eta e^{-\frac{1}{\sqtC{1}}(\tilde\Y_1(t,\eta)-\tilde\Y_1(t,\theta))}\tilde\P_1\tilde\U_1\tilde\Y_{1,\eta}(t,\theta) d\theta\\
& \qquad \qquad \qquad \qquad  -\tilde\Y_{2,\eta}(t,\eta) \int_0^\eta e^{-\frac{1}{\sqtC{2}}(\tilde\Y_2(t,\eta)-\tilde\Y_2(t,\theta))}\tilde\P_2\tilde\U_2\tilde\Y_{2,\eta}(t,\theta) d\theta\Big) d\eta\\
& = \frac{3}{\A^6} \int_0^1 (\tilde\Y_1-\tilde\Y_2)(t,\eta)\Big(\tilde\Y_{1,\eta}(t,\eta) \int_0^\eta e^{-\frac{1}{\sqtC{1}}(\tilde\Y_1(t,\eta)-\tilde\Y_1(t,\theta))}\tilde\P_1\tilde\V_1^+\tilde\Y_{1,\eta}(t,\theta) d\theta\\
& \qquad \qquad \qquad \qquad  -\tilde\Y_{2,\eta}(t,\eta) \int_0^\eta e^{-\frac{1}{\sqtC{2}}(\tilde\Y_2(t,\eta)-\tilde\Y_2(t,\theta))}\tilde\P_2\tilde\V_2^+\tilde\Y_{2,\eta}(t,\theta) d\theta\Big) d\eta\\
& +\frac{3}{\A^6} \int_0^1(\tilde\Y_1-\tilde\Y_2)(t,\eta)\Big(\tilde\Y_{1,\eta}(t,\eta)\int_0^\eta e^{-\frac{1}{\sqtC{1}}(\tilde\Y_1(t,\eta)-\tilde\Y_1(t,\theta))}\tilde\P_1\tilde\V_1^-\tilde\Y_{1,\eta}(t,\theta) d\theta\\
& \qquad \qquad \qquad \qquad  -\tilde\Y_{2,\eta}(t,\eta)\int_0^\eta e^{-\frac{1}{\sqtC{2}}(\tilde\Y_2(t,\eta)-\tilde\Y_2(t,\theta))}\tilde\P_2\tilde\V_2^-\tilde\Y_{2,\eta}(t,\theta)d\theta\Big) d\eta\\
& = \bar K_{14}^++\bar K_{14}^-.
\end{align*}
Note that both $\bar K_{14}^+$ and $\bar K_{14}^-$ have the same structure. Moreover, having a closer look at $\bar K_{14}^+$ one has 
\begin{equation*}
\bar K_{14}^+=- 3\tilde K_1,
\end{equation*}
where $\tilde K_1$ is defined in \eqref{eq:alleK}. Thus we can immediately conclude that 
\begin{equation*}
\vert \bar K_{14}^+\vert \leq \bigO(1) \left(\norm{\tilde\Y_1-\tilde\Y_2}^2+\norm{\tilde\U_1-\tilde\U_2}^2+\norm{\sqP{1}-\sqP{2}}^2+\vert \sqtC{1}-\sqtC{2}\vert\right).
\end{equation*}

Next, we have a look at $\bar K_{15}$, which can be rewritten as follows
\begin{align*}
\bar K_{15}&= \frac{1}{\A^6} \int_0^1 (\tilde\Y_1-\tilde\Y_2)(t,\eta)\Big(\tilde\Y_{2,\eta}(t,\eta)\int_0^\eta e^{-\frac{1}{\sqtC{2}}(\tilde\Y_2(t,\eta)-\tilde\Y_2(t,\theta))}\tilde\U_2^3\tilde\Y_{2,\eta}(t,\theta) d\theta\\
& \qquad \qquad \qquad \qquad -\tilde\Y_{1,\eta}(t,\eta)\int_0^\eta e^{-\frac{1}{\sqtC{1}}(\tilde\Y_1(t,\eta)-\tilde\Y_1(t,\theta))}\tilde\U_1^3\tilde\Y_{1,\eta}(t,\theta)d\theta\Big) d\eta\\
& = \frac{1}{\A^6} \int_0^1 (\tilde\Y_1-\tilde\Y_2)(t,\eta)\Big(\tilde\Y_{2,\eta}(t,\eta) \int_0^\eta e^{-\frac{1}{\sqtC{2}}(\tilde\Y_2(t,\eta)-\tilde\Y_2(t,\theta))}(\tilde\V_2^+)^3\tilde\Y_{2,\eta}(t,\theta) d\theta\\
& \qquad \qquad \qquad \qquad  -\tilde\Y_{1,\eta}(t,\eta)\int_0^\eta e^{-\frac{1}{\sqtC{1}}(\tilde\Y_1(t,\eta)-\tilde\Y_1(t,\theta))}(\tilde\V_1^+)^3\tilde\Y_{1,\eta}(t,\theta) d\theta\Big) d\eta\\
& + \frac{1}{\A^6} \int_0^1 (\tilde\Y_1-\tilde\Y_2)(t,\eta)\Big(\tilde\Y_{2,\eta}(t,\eta) \int_0^\eta e^{-\frac{1}{\sqtC{2}}(\tilde\Y_2(t,\eta)-\tilde\Y_2(t,\theta))}(\tilde\V_2^-)^3\tilde\Y_{2,\eta}(t,\theta) d\theta\\
& \qquad \qquad \qquad \qquad  -\tilde\Y_{1,\eta}(t,\eta)\int_0^\eta e^{-\frac{1}{\sqtC{1}}(\tilde\Y_1(t,\eta)-\tilde\Y_1(t,\theta))}(\tilde\V_1^-)^3\tilde\Y_{1,\eta}(t,\theta) d\theta\Big) d\eta\\
&= \bar K_{15}^++\bar K_{15}^-.
\end{align*}
Note that both $\bar K_{15}^+$ and $\bar K_{15}^-$ have the same structure. Having a close look at $\bar K_{15}^+$ one has 
\begin{equation*}
\bar K_{15}^+= \int_0^1 (\tilde\Y_1-\tilde\Y_2)(J_1+J_2+J_3+J_4+J_5+J_6+J_7+J_8)(t,\eta) d\eta,
\end{equation*}
where $J_1,\dots,J_8$ are defined in \eqref{eq:alleJ}.  Thus we can conclude immediately that  
\begin{align*}
\vert \bar K_{15}^+\vert &\leq \bigO(1) \Big(\norm{\tilde\Y_1-\tilde\Y_2}^2+\norm{\tilde\U_1-\tilde\U_2}^2\\
&\qquad\qquad\qquad+\norm{\sqP{1}-\sqP{2}}^2+\vert \sqtC{1}-\sqtC{2}\vert^2\Big).
\end{align*}

Finally, we have a look at $\bar K_{16}$, which can be rewritten as follows
\begin{align*}
\bar K_{16}&= \frac{1}{2\A^6}\int_0^1 (\tilde \Y_1-\tilde\Y_2)(t,\eta)\Big(\tilde \Y_{2,\eta}(t,\eta)\int_0^\eta e^{-\frac{1}{\sqtC{2}}(\tilde \Y_2(t,\eta)-\tilde\Y_2(t,\theta))}\sqtC{2}^5\tilde\U_2(t,\theta)d\theta\\
& \qquad \qquad \qquad \qquad \qquad -\tilde\Y_{1,\eta}(t,\eta)\int_0^\eta e^{-\frac{1}{\sqtC{1}}(\tilde \Y_1(t,\eta)-\tilde\Y_1(t,\theta))} \sqtC{1}^5\tilde\U_1(t,\theta)d\theta\Big)d\eta\\
& = \frac{\sqtC{2}^5-\sqtC{1}^5}{2\A^6} \mathbbm{1}_{\sqtC{1}\leq \sqtC{2}} \int_0^1(\tilde\Y_1-\tilde\Y_2)\tilde\Y_{2,\eta}(t,\eta)\\
&\qquad\qquad\qquad\qquad\qquad\qquad\qquad\times\Big(\int_0^\eta e^{-\frac{1}{\sqtC{2}}(\tilde \Y_2(t,\eta)-\tilde\Y_2(t,\theta))} \tilde\U_2(t,\theta)d\theta\Big) d\eta\\
& \quad +\frac{\sqtC{2}^5-\sqtC{1}^5}{2\A^6} \mathbbm{1}_{\sqtC{2}< \sqtC{1}} \int_0^1 (\tilde\Y_1-\tilde\Y_2)\tilde\Y_{1,\eta}(t,\eta)\\
&\qquad\qquad\qquad\qquad\qquad\qquad\qquad\times\Big(\int_0^\eta e^{-\frac{1}{\sqtC{1}}(\tilde\Y_1(t,\eta)-\tilde\Y_1(t,\theta))}\tilde\U_1(t,\theta)d\theta\Big)d\eta\\
& \quad +\frac{\ma^5}{2\A^6}\int_0^1(\tilde\Y_1-\tilde\Y_2)\tilde\Y_{2,\eta}(t,\eta)\\
&\qquad\times\Big(\int_0^\eta (e^{-\frac{1}{\sqtC{2}}(\tilde \Y_2(t,\eta)-\tilde\Y_2(t,\theta))}-e^{-\frac{1}{\sqtC{2}}(\tilde\Y_1(t,\eta)-\tilde\Y_1(t,\theta))})\tilde\U_2\mathbbm{1}_{B(\eta)}(t,\theta)d\theta\Big)d\eta\\
& \quad +\frac{\ma^5}{2\A^6}\int_0^1 (\tilde\Y_1-\tilde\Y_2)\tilde\Y_{1,\eta}(t,\eta)\\
&\qquad\times\Big(\int_0^\eta (e^{-\frac{1}{\sqtC{1}}(\tilde \Y_2(t,\eta)-\tilde\Y_2(t,\theta))}-e^{-\frac{1}{\sqtC{1}}(\tilde\Y_1(t,\eta)-\tilde\Y_1(t,\theta))})\tilde\U_1\mathbbm{1}_{B(\eta)^c}(t,\theta)d\theta\Big) d\eta\\
& \quad + \mathbbm{1}_{\sqtC{1}\leq \sqtC{2}}\frac{\ma^5}{2\A^6}\int_0^1 (\tilde \Y_1-\tilde \Y_2)\tilde \Y_{2,\eta}(t,\eta) \Big(\int_0^\eta \big(\min_j(e^{-\frac{1}{\sqtC{2}}(\tilde \Y_j(t,\eta)-\tilde \Y_j(t,\theta))})\\
& \qquad\qquad\qquad\qquad\qquad\qquad  -\min_j(e^{-\frac{1}{\sqtC{1}}(\tilde \Y_j(t,\eta)-\tilde \Y_j(t,\theta))})\big)\tilde \U_2(t,\theta) d\theta\Big) d\eta\\
& \quad +\mathbbm{1}_{\sqtC{2}<\sqtC{1}}\frac{\ma^5}{2\A^6} \int_0^1 (\tilde \Y_1-\tilde \Y_2)\tilde \Y_{1,\eta}(t,\eta)\Big(\int_0^{\eta } \big(\min_j(e^{-\frac{1}{\sqtC{2}}(\tilde \Y_j(t,\eta)-\tilde \Y_j(t,\theta))})\\
& \qquad\qquad\qquad\qquad\qquad\qquad-\min_j( e^{-\frac{1}{\sqtC{1}}(\tilde \Y_j(t,\eta)-\tilde \Y_j(t,\theta))})\big)\tilde \U_1(t,\theta) d\theta\Big) d\eta\\
& \quad + \frac{\ma^5}{2\A^6}\int_0^1(\tilde\Y_1-\tilde\Y_2)\tilde\Y_{2,\eta}(t,\eta) \\
& \qquad \times\Big(\int_0^\eta \min_j(e^{-\frac{1}{\ma}(\tilde \Y_j(t,\eta)-\tilde\Y_j(t,\theta))})(\tilde\V_2^+-\tilde\V_1^+)\mathbbm{1}_{\tilde\V_1^+\leq \tilde\V_2^+}(t,\theta) d\theta\Big) d\eta\\
& \quad + \frac{\ma^5}{2\A^6}\int_0^1(\tilde\Y_1-\tilde\Y_2)\tilde\Y_{2,\eta}(t,\eta) \\
& \qquad  \times\Big(\int_0^\eta \min_j(e^{-\frac{1}{\ma}(\tilde \Y_j(t,\eta)-\tilde\Y_j(t,\theta))})(\tilde\V_2^--\tilde\V_1^-)\mathbbm{1}_{\tilde\V_2^-\leq \tilde\V_1^-}(t,\theta) d\theta\Big) d\eta\\
& \quad + \frac{\ma^5}{2\A^6}\int_0^1(\tilde\Y_1-\tilde\Y_2)\tilde\Y_{1,\eta}(t,\eta) \\
& \qquad \times\Big(\int_0^\eta \min_j(e^{-\frac{1}{\ma}(\tilde \Y_j(t,\eta)-\tilde\Y_j(t,\theta))})
(\tilde\V_2^+-\tilde\V_1^+)\mathbbm{1}_{\tilde\V_2^+\leq \tilde\V_1^+}(t,\theta) d\theta\Big) d\eta\\
& \quad + \frac{\ma^5}{2\A^6}\int_0^1(\tilde\Y_1-\tilde\Y_2)\tilde\Y_{1,\eta}(t,\eta) \\
& \qquad  \times\Big(\int_0^\eta \min_j(e^{-\frac{1}{\ma}(\tilde \Y_j(t,\eta)-\tilde\Y_j(t,\theta))})
(\tilde\V_2^--\tilde\V_1^-)\mathbbm{1}_{\tilde\V_1^-\leq \tilde\V_2^-}(t,\theta) d\theta\Big) d\eta\\
& \quad+ \frac{\ma^5}{2\A^6}\int_0^1(\tilde \Y_1-\tilde\Y_2)(\tilde\Y_{2,\eta}-\tilde\Y_{1,\eta})(t,\eta)\\
&\qquad  \times\Big(\int_0^\eta \min_j(e^{-\frac{1}{\ma}(\tilde \Y_j(t,\eta)-\tilde\Y_j(t,\theta))})
\min_j(\tilde \V_j^+)(t,\theta)d\theta\Big)d\eta\\
& \quad+ \frac{\ma^5}{2\A^6}\int_0^1(\tilde \Y_1-\tilde\Y_2)(\tilde\Y_{2,\eta}-\tilde\Y_{1,\eta})(t,\eta)\\
& \qquad  \times \Big(\int_0^\eta \min_j(e^{-\frac{1}{\ma}(\tilde \Y_j(t,\eta)-\tilde\Y_j(t,\theta))})\max_j(\tilde \V_j^-)(t,\theta)d\theta\Big)d\eta \\
& =\bar B_{61}+\bar B_{62}+\bar B_{63}+\bar B_{64}+ \bar B_{65}+\bar B_{66}\\
&\quad+\bar B_{67}^++\bar B_{67}^-+\bar B_{68}^++\bar B_{68}^-+\bar B_{69}^++\bar B_{69}^-.
\end{align*}

The key observation, which rescues the whole paper, is again
\begin{equation*}
\sqtC{i}^5\leq 2(\tilde\P_i\tilde\Y_{i,\eta}(t,\eta)+\tilde\Henergy_{i,\eta})(t,\eta),  %\label{eq:key}
\end{equation*}
which yields for $\bar B_{61}$ (and similar for $\bar B_{62}$) that
\begin{align*}
\vert \bar B_{61}\vert &\leq \frac{\sqtC{2}^5-\sqtC{1}^5}{2\A^{11}} \int_0^1 \vert\tilde\Y_1-\tilde\Y_2\vert \tilde\Y_{2,\eta}(t,\eta)\\
&\qquad\qquad\qquad\times \Big(\int_0^\eta  e^{-\frac{1}{\sqtC{2}}(\tilde\Y_2(t,\eta)-\tilde\Y_2(t,\theta))}\sqtC{2}^{5}\vert\tilde\U_2\vert(t,\theta)d\theta \Big) d\eta\\
& \leq \frac{5}{2\A^{7}} \vert \sqtC{2}-\sqtC{1}\vert \norm{\tilde\Y_1-\tilde\Y_2}\\
&\qquad\qquad\times\Big(\int_0^1 \tilde \Y_{2,\eta}^2(t,\eta)\Big(\int_0^\eta e^{-\frac{1}{\sqtC{2}}(\tilde\Y_2(t,\eta)-\tilde\Y_2(t,\theta))}\sqtC{2}^{5}\vert \tilde\U_2\vert(t,\theta)\Big)^2d\eta\Big)^{1/2}\\
& \leq \frac{10}{\A^{7}} \vert \sqtC{2}-\sqtC{1}\vert \norm{\tilde\Y_1-\tilde\Y_2}\\
&\qquad\times\Big(\int_0^1\tilde\Y_{2,\eta}^2(t,\eta)\Big[\Big(\int_0^\eta e^{-\frac{1}{\sqtC{2}}(\tilde\Y_2(t,\eta)-\tilde\Y_2(t,\theta))} \tilde \P_2\vert\tilde\U_2\vert \tilde\Y_{2,\eta}(t,\theta)d\theta\Big)^2\\
&\qquad\qquad\qquad\qquad+\Big(\int_0^\eta e^{-\frac{1}{\sqtC{2}}(\tilde\Y_2(t,\eta)-\tilde\Y_2(t,\theta))} \vert \tilde\U_2\vert \tilde\Henergy_{2,\eta}(t,\theta) d\theta\Big)^2\Big]d\eta\Big)^{1/2}\\
&\leq \frac{10}{\A^{7}}\vert \sqtC{2}-\sqtC{1}\vert \norm{\tilde\Y_1-\tilde\Y_2}\Big(\int_0^1 \tilde\Y_{2,\eta}^2(t,\eta)\\
& \quad  \times \Big[\Big(\int_0^\eta e^{-\frac{1}{\sqtC{2}}(\tilde\Y_2(t,\eta)-\tilde\Y_2(t,\theta))}\tilde\U_2^2\tilde\Y_{2,\eta}(t,\theta)d\theta\Big)\\
&\qquad\qquad\times\Big(\int_0^\eta e^{-\frac{1}{\sqtC{2}}(\tilde\Y_2(t,\eta)-\tilde\Y_2(t,\theta))}\tilde\P_2^2\tilde\Y_{2,\eta}(t,\theta)d\theta\Big)\\
& \qquad \qquad+\norm{\tilde\U_2(t,\dott)}_{L^\infty}^2\Big(\int_0^\eta e^{-\frac{1}{\sqtC{2}}(\tilde\Y_2(t,\eta)-\tilde\Y_2(t,\theta))}\tilde\Henergy_{2,\eta}(t,\theta)d\theta\Big)^2\Big]d\eta\Big)^{1/2}\\
& \leq \frac{10}{\A^7} \vert \sqtC{2}-\sqtC{1}\vert\norm{\tilde\Y_1-\tilde\Y_2}\\
&\qquad\qquad\qquad\times\Big(\int_0^1 (6A^6+16\A^2\norm{\tilde\U_2(t,\dott)}_{L^\infty}^2)
\tilde\P_2^2\tilde\Y_{2,\eta}^2(t,\eta)d\eta\Big)^{1/2}\\
&\leq \bigO(1)(\norm{\tilde\Y_1-\tilde\Y_2}^2+\vert \sqtC{2}-\sqtC{1}\vert^2).
\end{align*}

Next we have a closer look at $\bar B_{63}$ (and similar for $\bar B_{64}$). Recalling the definition of $B(\eta)$ \eqref{Def:Bn}, we have
\begin{align*}
\vert \bar B_{63}\vert &\leq \frac{\ma^5}{2\A^6\sqtC{2}}\int_0^1 \vert \tilde\Y_1-\tilde\Y_2\vert \tilde\Y_{2,\eta}(t,\eta)\Big(\int_0^\eta \big(\vert \tilde\Y_1-\tilde\Y_2\vert(t,\eta)+\vert \tilde\Y_1-\tilde\Y_2\vert (t,\theta)\big)\nn\\
&\qquad\qquad\qquad\qquad\qquad\qquad \times e^{-\frac{1}{\sqtC{2}}(\tilde\Y_2(t,\eta)-\tilde\Y_2(t,\theta))}\vert \tilde\U_2\vert(t,\theta)d\theta\Big) d\eta\notag\\
& \leq \frac{\ma^5}{2\A^6\sqtC{2}}\int_0^1(\tilde\Y_1-\tilde\Y_2)^2\tilde\Y_{2,\eta}(t,\eta) \Big(\int_0^\eta e^{-\frac{1}{\sqtC{2}}(\tilde\Y_2(t,\eta)-\tilde\Y_2(t,\theta))}\vert \tilde\U_2\vert (t,\theta) d\theta\Big)d\eta\notag\\
&\quad +\frac{\ma^5}{2\A^6\sqtC{2}} \int_0^1\vert \tilde\Y_1-\tilde\Y_2\vert \tilde\Y_{2,\eta}(t,\eta)\notag\\
&\qquad\qquad\qquad\qquad\times\Big(\int_0^\eta \vert \tilde\Y_1-\tilde\Y_2\vert(t,\theta)e^{-\frac{1}{\sqtC{2}}(\tilde\Y_2(t,\eta)-\tilde\Y_2(t,\theta))}\vert \tilde\U_2\vert (t,\theta)d\theta\Big)d\eta\notag\\
& \leq \frac{\ma^5}{\A^6\sqtC{2}^6} \int_0^1 (\tilde\Y_1-\tilde\Y_2)^2\tilde\Y_{2,\eta}(t,\eta)\notag \\
&\qquad\qquad\qquad\times\Big(\int_0^\eta e^{-\frac{1}{\sqtC{2}}(\tilde\Y_2(t,\eta)-\tilde\Y_2(t,\theta))}(\tilde\P_2\tilde\Y_{2,\eta}+\tilde\Henergy_{2,\eta})\vert\tilde\U_2\vert(t,\theta)d\theta\Big) d\eta\notag\\
& \quad +\frac{\ma^5}{\A^6\sqtC{2}^6}\int_0^1 \vert \tilde\Y_1-\tilde\Y_2\vert \tilde\Y_{2,\eta}(t,\eta)\notag \\
&\qquad\times\Big(\int_0^\eta \vert \tilde\Y_1-\tilde\Y_2\vert (t,\theta) e^{-\frac{1}{\sqtC{2}}(\tilde\Y_2(t,\eta)-\tilde\Y_2(t,\theta))} (\tilde\P_2\tilde\Y_{2,\eta}+\tilde\Henergy_{2,\eta})\vert \tilde\U_2\vert (t,\theta)d\theta\Big)d\eta\notag\\
& \leq \frac{1}{\A^6\sqtC{2}} \int_0^1 (\tilde\Y_1-\tilde\Y_2)^2\tilde\Y_{2,\eta}(t,\eta)\Big(\int_0^\eta e^{-\frac{1}{\sqtC{2}}(\tilde\Y_2(t,\eta)-\tilde\Y_2(t,\theta))}\notag\\
&\qquad \qquad\qquad  \times  \big(\frac{1}{\sqtC{2}}\tilde\P_2^2\tilde\Y_{2,\eta}+\sqtC{2}\tilde\U_2^2\tilde\Y_{2,\eta}+\norm{\tilde\U_2(t,\dott)}_{L^\infty}\tilde\Henergy_{2,\eta}\big)(t,\theta) d\theta\Big) d\eta\notag\\
& \quad +\norm{\tilde\Y_1-\tilde\Y_2}^2+\frac{1}{\A^{12}\sqtC{2}^2}\int_0^1\tilde\Y_{2,\eta}^2(t,\eta)\Big(\int_0^\eta \vert \tilde\Y_1-\tilde\Y_2\vert (t,\theta) e^{-\frac{1}{\sqtC{2}}(\tilde\Y_2(t,\eta)-\tilde\Y_2(t,\theta))} \notag\\
&\qquad\qquad\qquad\qquad\qquad\qquad\times\big(\tilde\P_2\tilde\Y_{2,\eta}+\tilde\Henergy_{2,\eta}\big)\vert \tilde\U_2\vert (t,\theta)d\theta\Big)^2d\eta\notag\\
& \leq \norm{\tilde\Y_1-\tilde\Y_2}^2+ \bigO(1)\frac{1}{\A^5}\int_0^1(\tilde\Y_1-\tilde\Y_2)^2\tilde\P_2\tilde\Y_{2,\eta}(t,\eta) d\eta+\frac{2}{\A^{12}\sqtC{2}^2} \int_0^1\tilde\Y_{2,\eta}^2(t,\eta)\notag\\
& \qquad\qquad\quad \times\Big[\Big(\int_0^\eta \vert \tilde\Y_1-\tilde\Y_2\vert (t,\theta) e^{-\frac{1}{\sqtC{2}}(\tilde\Y_2(t,\eta)-\tilde\Y_2(t,\theta))}\tilde\P_2\vert \tilde\U_2\vert\tilde\Y_{2,\eta}(t,\theta)d\theta\Big)^2\notag\\
& \qquad \qquad \qquad+\Big(\int_0^\eta \vert \tilde\Y_1-\tilde\Y_2\vert (t,\theta) e^{-\frac{1}{\sqtC{2}}(\tilde\Y_2(t,\eta)-\tilde\Y_2(t,\theta))}\vert \tilde\U_2\vert \tilde\Henergy_{2,\eta}(t,\theta)d\theta\Big)^2\Big]d\eta\notag\\
& \leq \bigO(1)\norm{\tilde\Y_1-\tilde\Y_2}^2 +\frac{2}{\A^{12}\sqtC{2}^2}\int_0^1 \tilde\Y_{2,\eta}^2(t,\eta)\notag\\
& \qquad\qquad\qquad\times\Big[\Big(\int_0^\eta (\tilde\Y_1-\tilde\Y_2)^2(t,\theta) e^{-\frac{1}{\sqtC{2}}(\tilde\Y_2(t,\eta)-\tilde\Y_2(t,\theta))}\tilde\P_2^2\tilde\Y_{2,\eta}(t,\theta)d\theta\Big)\notag\\
& \qquad \qquad \qquad \qquad \qquad\qquad\quad \times\Big(\int_0^\eta e^{-\frac{1}{\sqtC{2}}(\tilde\Y_2(t,\eta)-\tilde\Y_2(t,\theta))}\tilde\U_2^2\tilde\Y_{2,\eta}(t,\theta) d\theta\Big)\notag\\
& \qquad \qquad \qquad \quad  + \Big(\int_0^\eta (\tilde\Y_1-\tilde\Y_2)^2(t,\theta) e^{-\frac{1}{\sqtC{2}}(\tilde\Y_2(t,\eta)-\tilde\Y_2(t,\theta))}\tilde \U_2^2\tilde\Henergy_{2,\eta}(t,\theta)d\theta\Big)\notag\\
& \qquad \qquad \qquad \qquad \qquad\quad \times\Big(\int_0^\eta e^{-\frac{1}{\sqtC{2}}(\tilde\Y_2(t,\eta)-\tilde\Y_2(t,\theta))}\tilde\Henergy_{2,\eta}(t,\theta)d\theta\Big)\Big]d\eta\notag\\
& \leq \bigO(1)\norm{\tilde\Y_1-\tilde\Y_2}^2\notag \\
& \quad +\frac{8}{\A^{12}\sqtC{2}} \int_0^1 \tilde\P_2\tilde\Y_{2,\eta}^2(t,\eta)\notag\\
&\qquad\times\Big(\int_0^\eta (\tilde\Y_1-\tilde\Y_2)^2(t,\theta) e^{-\frac{1}{\sqtC{2}}(\tilde\Y_2(t,\eta)-\tilde\Y_2(t,\theta))} (\tilde\P_2^2\tilde\Y_{2,\eta}+\tilde\U_2^2\tilde\Henergy_{2,\eta})(t,\theta) d\theta\Big) d\eta\notag\\
& \leq \bigO(1)\norm{\tilde\Y_1-\tilde\Y_2}^2
 +\frac{4}{\A^8} \int_0^1 \tilde\Y_{2,\eta}(t,\eta) e^{-\frac{1}{\sqtC{2}}\tilde\Y_2(t,\eta)}\notag\\
&\qquad\qquad\qquad\times\Big(\int_0^\eta (\tilde\Y_1-\tilde\Y_2)^2e^{\frac{1}{\sqtC{2}}\tilde\Y_2(t,\theta)}(\tilde\P_2^2\tilde\Y_{2,\eta}+\tilde\U_2^2\tilde\Henergy_{2,\eta})(t,\theta) d\theta\Big)d\eta\notag\\
& =\bigO(1) \norm{\tilde\Y_1-\tilde\Y_2}^2 + \frac{4\sqtC{2}}{\A^8}\Big( -\int_0^\eta (\tilde\Y_1-\tilde\Y_2)^2(t,\theta) e^{-\frac{1}{\sqtC{2}}(\tilde\Y_2(t,\eta)-\tilde\Y_2(t,\theta))}\nn \\
&\qquad\qquad\qquad\qquad\qquad\qquad\qquad\qquad\qquad\times\big(\tilde\P_2^2\tilde\Y_{2,\eta}+\tilde\U_2^2\tilde\Henergy_{2,\eta}\big)(t,\theta) d\theta\Big\vert_{\eta=0}^1\notag \\
& \quad \qquad \qquad\qquad\qquad+ \int_0^1 (\tilde \Y_1-\tilde\Y_2)^2 (\tilde\P_2^2\tilde\Y_{2,\eta}+\tilde\U_2^2\tilde\Henergy_{2,\eta})(t,\eta) d\eta\Big)\notag\\
& \leq \bigO(1) \norm{\tilde\Y_1-\tilde\Y_2}^2. %\label{eq:barB63}
\end{align*}

Next, we have a look at $\bar B_{65}$  (a similar argument works for $\bar B_{66}$). 
Direct calculations yield
\begin{align*}
\vert \bar B_{65}\vert &\leq \frac{\ma^5}{2\A^6}\int_0^1 \vert \tilde \Y_1-\tilde \Y_2\vert \tilde \Y_{2,\eta}(t,\eta)\\&\qquad\qquad\times\Big(\frac{4}{\ma e}\int_0^\eta  \min_j(e^{-\frac{3}{4\sqtC{2}}(\tilde \Y_j(t,\eta)-\tilde \Y_j(t,\theta))})\vert \tilde \U_2\vert (t,\theta) d\theta \Big)d\eta\vert \sqtC{1}-\sqtC{2}\vert\\
& \leq \frac{2}{\A^2 e}\int_0^1 \vert \tilde \Y_1-\tilde \Y_2\vert \tilde \Y_{2,\eta}(t,\eta)\\
&\qquad\qquad\times\Big(\int_0^\eta e^{-\frac{3}{4\sqtC{2}}(\tilde \Y_2(t,\eta)-\tilde \Y_2(t,\theta))}\vert \tilde \U_2\vert (t,\theta) d\theta\Big) d\eta\vert \sqtC{1}-\sqtC{2}\vert\\
& \leq \frac{4}{\A^2\sqtC{2}^5 e} \int_0^1 \vert \tilde \Y_1-\tilde \Y_2\vert \tilde \Y_{2,\eta} (t,\eta)\\
&\qquad\times \Big(\int_0^\eta e^{-\frac{3}{4\sqtC{2}}(\tilde \Y_2(t,\eta)-\tilde \Y_2(t,\theta))} (\tilde \P_2\tilde \Y_{2,\eta}+\tilde \Henergy_{2,\eta})\vert \tilde \U_2\vert (t,\theta) d\theta\Big) d\eta \vert \sqtC{1}-\sqtC{2}\vert\\
& \leq \vert \sqtC{1}-\sqtC{2}\vert ^2 + \frac{32}{\A^4\sqtC{2}^{10}e^2} \norm{\tilde \Y_1-\tilde \Y_2}^2\int_0^1 \tilde \Y_{2,\eta}^2(t,\eta)\\
&\qquad\qquad\qquad \qquad \quad\times\Big[\Big(\int_0^{\eta} e^{-\frac{1}{2\sqtC{2}}(\tilde \Y_2(t,\eta)-\tilde \Y_2(t,\theta))}\tilde \P_2^2\tilde \Y_{2,\eta}(t,\theta) d\theta\Big) \\
& \qquad \qquad \qquad \qquad \qquad \qquad  \times \Big(\int_0^\eta e^{-\frac{1}{\sqtC{2}}(\tilde \Y_2(t,\eta)-\tilde \Y_2(t,\theta))} \tilde \U_2^2\tilde \Y_{2,\eta}(t,\theta) d\theta\Big) \\
& \qquad \qquad \qquad \qquad  \qquad \quad +\Big(\int_0^\eta e^{-\frac{1}{2\sqtC{2}}(\tilde \Y_2(t,\eta)-\tilde \Y_2(t,\theta))}\tilde \U_2^2\tilde \Henergy_{2,\eta}(t,\theta) d\theta\Big) \\
& \qquad \qquad \qquad \qquad  \qquad \qquad  \times \Big(\int_0^\eta e^{-\frac{1}{\sqtC{2}}(\tilde \Y_2(t,\eta)-\tilde \Y_2(t,\theta))}\tilde \Henergy_{2,\eta}(t,\theta) d\theta \Big) \Big] d\eta\\
&\leq \vert \sqtC{1}-\sqtC{2}\vert^2 + \frac{128}{\A^4\sqtC{2}^9 e^2}\norm{\tilde \Y_1-\tilde \Y_2}^2 \int_0^1 \tilde \P_2\tilde \Y_{2,\eta}^2(t,\eta)\\
& \qquad \qquad \qquad  \times \Big(\int_0^\eta e^{-\frac{1}{2\sqtC{2}}(\tilde \Y_2(t,\eta)-\tilde \Y_2(t,\theta))}(\tilde \P_2^2\tilde \Y_{2,\eta} +\tilde \U_2^2\tilde \Henergy_{2,\eta})(t,\theta)d\theta\Big)d\eta\\
& \leq \vert \sqtC{1}-\sqtC{2}\vert ^2 +\frac{64}{\A^4\sqtC{2}^4 e^2} \norm{\tilde \Y_1-\tilde \Y_2}^2 \\
&\quad\times\int_0^1 \tilde \Y_{2,\eta} (t,\eta) e^{-\frac{1}{2\sqtC{2}}\tilde \Y_2(t,\eta)}\Big(\int_0^\eta e^{\frac{1}{2\sqtC{2}}\tilde \Y_2(t,\theta)}(\tilde \P_2^2\tilde \Y_{2,\eta}+\tilde \U_2^2\tilde \Henergy_{2,\eta})(t,\theta) d\theta\Big)d\eta\\
& \leq \vert \sqtC{1}-\sqtC{2}\vert^2 + \frac{128}{\A^4\sqtC{2}^3 e^2}\norm{\tilde \Y_1-\tilde \Y_2}^2\\
&\qquad\qquad\times \Big[-\int_0^\eta e^{-\frac{1}{2\sqtC{2}}(\tilde \Y_2(t,\eta)-\tilde \Y_2(t,\theta))}(\tilde \P_2^2\tilde \Y_{2,\eta}+\tilde \U_2^2\tilde \Henergy_{2,\eta})(t,\theta)d\theta\Big\vert_{\eta=0}^1\\
& \qquad \qquad \qquad \qquad \qquad \qquad \qquad  + \int_0^1 (\tilde \P_2^2\tilde \Y_{2,\eta}+\tilde \U_2^2\tilde \Henergy_{2,\eta})(t,\eta) d\eta\Big]\\
& \leq \bigO(1)(\norm{\tilde \Y_1-\tilde \Y_2}^2 +\vert  \sqtC{1}-\sqtC{2}\vert^2)
\end{align*}

Next, we have a look at $\bar B_{67}^+$ (a similar argument works for $\bar B_{67}^-$ and $\bar B_{68}^{\pm}$). Direct calculations yield
\begin{align}
\vert \bar B_{67}^+\vert &\leq \frac{\ma^5}{2\A^6} \int_0^1 \vert \tilde\Y_1-\tilde\Y_2\vert \tilde \Y_{2,\eta}(t,\eta)\Big(\int_0^\eta e^{-\frac{1}{\sqtC{2}}(\tilde\Y_2(t,\eta)-\tilde\Y_2(t,\theta))} \vert \tilde\V_2^+-\tilde\V_1^+\vert (t,\theta) d\theta\Big) d\eta\notag\\
& \leq \norm{\tilde\Y_1-\tilde\Y_2}^2\notag\\
& \quad + \frac{\sqtC{2}^5}{4\A^7}\int_0^1 \tilde\Y_{2,\eta}^2(t,\eta)\Big(\int_0^\eta e^{-\frac{1}{\sqtC{2}}(\tilde\Y_2(t,\eta)-\tilde\Y_2(t,\theta))}\vert \tilde\U_2-\tilde\U_1\vert (t,\theta) d\theta\Big)^2d\eta\notag\\
& \leq \norm{\tilde\Y_1-\tilde\Y_2}^2 +\frac{1}{2\A^7}\int_0^1 \tilde\Y_{2,\eta}^2(t,\eta)\notag \\
& \qquad\qquad\times\Big(\int_0^\eta e^{-\frac{1}{\sqtC{2}}(\tilde\Y_2(t,\eta)-\tilde\Y_2(t,\theta))}(\tilde\P_2\tilde\Y_{2,\eta}+\tilde\Henergy_{2,\eta})
\vert \tilde\U_2-\tilde\U_1\vert (t,\theta) d\theta\Big)^2d\eta\notag\\
& \leq \norm{\tilde\Y_1-\tilde\Y_2}^2\notag\\
& \quad +\frac{1}{\A^7} \int_0^1\tilde\Y_{2,\eta}^2(t,\eta)\Big[\Big(\int_0^\eta e^{-\frac{1}{\sqtC{2}}(\tilde\Y_2(t,\eta)-\tilde\Y_2(t,\theta))}\tilde\P_2\tilde\Y_{2,\eta}\vert \tilde\U_2-\tilde\U_1\vert (t,\theta) d\theta\Big)^2\notag\\
& \qquad \qquad \qquad\qquad \qquad +\Big(\int_0^\eta e^{-\frac{1}{\sqtC{2}}(\tilde\Y_2(t,\eta)-\tilde\Y_2(t,\theta))}\tilde\Henergy_{2,\eta}\vert \tilde\U_2-\tilde\U_1\vert(t,\theta)d\theta\Big)^2\Big]d\eta\notag\\
&\leq \norm{\tilde\Y_1-\tilde\Y_2}^2
 +\frac{1}{\A^7}\int_0^1 \tilde\Y_{2,\eta}^2(t,\eta) \Big(\int_0^\eta e^{-\frac3{2\sqtC{2}} (\tilde\Y_2(t,\eta)-\tilde\Y_2(t,\theta))}\tilde\P_2\tilde\Y_{2,\eta}(t,\theta) d\theta\Big)\notag\\
& \qquad \qquad \qquad \qquad \times \Big(\int_0^\eta e^{-\frac1{2\sqtC{2}}(\tilde\Y_2(t,\eta)-\tilde\Y_2(t,\theta))}(\tilde\U_2-\tilde\U_1)^2\tilde\P_2\tilde\Y_{2,\eta}(t,\theta) d\theta\Big)d\eta\notag\\
& \quad +\frac{1}{\A^7}\int_0^1 \tilde \Y_{2,\eta}^2(t,\eta) \Big(\int_0^\eta e^{-\frac{1}{\sqtC{2}}(\tilde\Y_2(t,\eta)-\tilde\Y_2(t,\theta))}\tilde\Henergy_{2,\eta}(t,\theta) d\theta\Big)\notag\\
& \qquad \qquad \qquad \qquad \times\Big(\int_0^\eta e^{-\frac{1}{\sqtC{2}}(\tilde\Y_2(t,\eta)-\tilde\Y_2(t,\theta))}(\tilde\U_1-\tilde\U_2)^2\tilde\Henergy_{2,\eta}(t,\theta)d\theta\Big)d\eta\notag\\
& \leq \norm{\tilde\Y_1-\tilde\Y_2}^2\notag\\
& \quad +\frac{2}{\A^6} \int_0^1 \tilde\Y_{2,\eta}(t,\eta) e^{-\frac1{2\sqtC{2}}\tilde\Y_2(t,\eta)}\Big(\int_0^\eta e^{\frac1{2\sqtC{2}}\tilde\Y_2(t,\theta)} (\tilde\U_1-\tilde\U_2)^2\tilde\P_2\tilde\Y_{2\eta}(t,\theta) d\theta\Big)d\eta\notag\\
& \quad +\frac{4}{\A^6} \int_0^1 \tilde\Y_{2,\eta}(t,\eta)e^{-\frac{1}{\sqtC{2}}\tilde\Y_2(t,\eta)}\Big(\int_0^\eta e^{\frac{1}{\sqtC{2}}\tilde\Y_2(t,\theta)}(\tilde\U_1-\tilde\U_2)^2\tilde \Henergy_{2,\eta}(t,\theta) d\theta\Big) d\eta\notag\\
& \leq \norm{\tilde\Y_1-\tilde\Y_2}^2\notag\\
& \quad + \frac{2\sqtC{2}}{\A^6} \Big(- 2\int_0^\eta e^{-\frac1{2\sqtC{2}} (\tilde\Y_2(t,\eta)-\tilde\Y_2(t,\theta))}(\tilde\U_1-\tilde\U_2)^2\tilde\P_2\tilde\Y_{2,\eta}(t,\theta)d\theta\Big\vert_{\eta=0}^1\notag\\
& \qquad \qquad \qquad \qquad +2 \int_0^1 (\tilde\U_1-\tilde\U_2)^2\tilde\P_2\tilde\Y_{2,\eta}(t,\eta) d\eta\Big)\notag\\
& \quad + \frac{4\sqtC{2}}{\A^6}\Big(-\int_0^\eta e^{-\frac{1}{\sqtC{2}}(\tilde\Y_2(t,\eta)-\tilde\Y_2(t,\theta))}(\tilde\U_1-\tilde\U_2)^2\tilde\Henergy_{2,\eta}(t,\theta) d\theta\Big\vert_{\eta=0}^1\notag\\
& \qquad \qquad \qquad \qquad + \int_0^1 (\tilde\U_1-\tilde\U_2)^2\tilde\Henergy_{2,\eta}(t,\eta) d\eta\Big)\notag\\
& \leq \bigO(1) \Big(\norm{\tilde\Y_1-\tilde\Y_2}^2+\norm{\tilde\U_1-\tilde\U_2}^2\Big). \label{eq:barB65}
\end{align}

Finally, we consider $\bar B_{69}^+$ (a similar argument works for $\bar B_{69}^-$). Here integration by parts will play the main role. Indeed, we have 
\begin{align}
\vert \bar B_{69}^+\vert &=\frac{\ma^5}{2\A^6}\vert \int_0^1 (\tilde \Y_1-\tilde\Y_2)(\tilde \Y_{1,\eta}-\tilde\Y_{2,\eta})(t,\eta)\notag\\
& \qquad \qquad \times\Big(\int_0^\eta \min_j(e^{-\frac{1}{\ma}(\tilde \Y_j(t,\eta)-\tilde\Y_j(t,\theta))})\min_j(\tilde\V_j^+)(t,\theta) d\theta\Big) d\eta\vert\notag \\
& = \frac{\ma^5}{4\A^6} \vert (\tilde\Y_1-\tilde\Y_2)^2(t,\eta)\notag\\
&\qquad\times\Big(\int_0^\eta \min_j(e^{-\frac{1}{\ma}(\tilde \Y_j(t,\eta)-\tilde\Y_j(t,\theta))})\min_j(\tilde\V_j^+)(t,\theta) d\theta\Big)\Big\vert_{\eta=0}^1\notag \\
& \qquad   -\int_0^1 (\tilde\Y_1-\tilde\Y_2)^2(t,\eta)\notag \\
& \qquad \quad  \times \frac{d}{d\eta}\Big(\int_0^\eta \min_j(e^{-\frac{1}{\ma}(\tilde \Y_j(t,\eta)-\tilde\Y_j(t,\theta))})\min_j(\tilde\V_j^+)(t,\theta) d\theta\Big) d\eta\vert \notag\\
& \leq \frac{\ma^5}{4\A^6}\vert \int_0^1 (\tilde\Y_1-\tilde\Y_2)^2(t,\eta) \notag\\
& \qquad  \times \frac{d}{d\eta}\Big(\int_0^\eta \min_j(e^{-\frac{1}{\ma}(\tilde \Y_j(t,\eta)-\tilde\Y_j(t,\theta))})\min_j(\tilde\V_j^+)(t,\theta) d\theta \Big)d\eta\vert \notag\\
& \leq \bigO(1) \norm{\tilde\Y_1-\tilde\Y_2} \label{eq:barB67}
\end{align}
where $\bigO(1)$ denotes some constant depending on $\A$, which remains bounded as $\A\to 0$, provided that we can show that both
\begin{equation*}
\ma^5\int_0^\eta \min_j(e^{-\frac{1}{\ma}(\tilde \Y_j(t,\eta)-\tilde\Y_j(t,\theta))})\min_j(\tilde\V_j^+)(t,\theta) d\theta \leq \bigO(1)\A^2 \min_j(\tilde\P_j)(t,\eta)
\end{equation*}
and  the derivative
\begin{equation}\label{eq:eU_pluss}
\ma\frac{d}{d\eta}\Big(\int_0^\eta \min_j(e^{-\frac{1}{\ma}(\tilde \Y_j(t,\eta)-\tilde\Y_j(t,\theta))})\min_j(\tilde\V_j^+)(t,\theta) d\theta \Big)
\end{equation}
exist and are uniformly bounded.

Direct computations yield
\begin{align*}
\ma^5& \int_0^\eta \min_j(e^{-\frac{1}{\ma}(\tilde \Y_j(t,\eta)-\tilde\Y_j(t,\theta))})\min_j(\tilde\V_j^+)(t,\theta) d\theta\\
& \leq \int_0^\eta e^{-\frac{1}{\sqtC{i}}(\tilde \Y_i(t,\eta)-\tilde\Y_i(t,\theta))}\sqtC{i}^5\vert \tilde\U_i\vert (t,\theta) d\theta\\
& \leq 2\int_0^\eta e^{-\frac{1}{\sqtC{i}}(\tilde \Y_i(t,\eta)-\tilde\Y_i(t,\theta))}(\tilde\P_i\tilde\Y_{i,\eta}+\tilde\Henergy_{i,\eta})\vert \tilde\U_i\vert (t,\theta) d\theta\\
& \leq2 \int_0^\eta e^{-\frac{1}{\sqtC{i}}(\tilde \Y_i(t,\eta)-\tilde\Y_i(t,\theta))} (\frac{1}{\sqtC{i}}\tilde \P_i^2\tilde\Y_{i,\eta}+\sqtC{i}\tilde\U_i^2\tilde\Y_{i,\eta}+\norm{\tilde\U_i(t,\dott)}_{L^\infty}\tilde\Henergy_{i,\eta})(t,\theta) d\theta\\
& \leq\bigO(1)\sqtC{i}^2  \tilde \P_i(t,\eta).
\end{align*}
The result for \eqref{eq:eU_pluss} is contained in Lemma \ref{lemma:6}.

%-------------------------
\begin{lemma}\label{lemlipy}
Let $\tilde\Y_i$ denote two solutions of \eqref{eq:Ylip}. Then we have
\begin{align} \nn
\frac{d}{dt}\norm{\tilde\Y_1-\tilde\Y_2}^2 &\leq \bigO(1)\Big(\norm{\tilde\Y_1-\tilde\Y_2}^2+\norm{\tilde\U_1-\tilde\U_2}^2\\
&\qquad\qquad\qquad+ \norm{\sqP{1}-\sqP{2}}^2+\vert \sqtC{1}-\sqtC{2}\vert ^2\Big),
\end{align}
where $\bigO(1)$ denotes some constant which depends on $\A=\max_j(\sqtC{j})$, which remains bounded as $\A\to 0$.
\end{lemma}
%---------------------

\subsection{Lipschitz estimates for $\tilde \U$}\label{subsec:lipu}
From the system of differential equations, we have 
\begin{equation}\label{eq:Ulip}
\tilde \U_{i,t}+(\frac23 \frac{1}{\sqtC{i}^5}\tilde \U_i^3+\frac{1}{\sqtC{i}^6}\tilde \So_i)\tilde \U_{i,\eta}=-\frac{1}{\sqtC{i}^2}\tilde \Q_i,
\end{equation}
where
\begin{align*}
\tilde \Q_i(t,\eta)&= -\frac14 \int_0^{1} \sign(\eta-\theta) e^{-\frac{1}{\sqtC{i}}\vert \tilde \Y_i(t,\eta)-\tilde \Y_i(t,\theta)\vert} \big(2(\tilde \U_i^2-\tilde\P_i)\tilde\Y_{i,\eta}(t,\theta)+\sqtC{i}^5\big) d\theta,\\
\tilde\So_i(t,\eta)& = \int_0^{1} e^{-\frac{1}{\sqtC{i}}\vert \tilde\Y_i(t,\eta)-\tilde\Y_i(t,\theta)\vert } (\frac23 \tilde \U_i^3\tilde \Y_{i,\eta}-\tilde \Q_i\tilde \U_{i,\eta}-2\tilde \P_i\tilde \U_i\tilde \Y_{i,\eta})(t,\theta)d\theta.
\end{align*}

Thus we have 
\begin{align*} \nn
\frac{d}{dt}\int_0^{1} (\tilde \U_1-\tilde \U_2)^2(t,\eta)d\eta& = 2\int_0^{1}(\tilde \U_1- \tilde \U_2)(\tilde \U_{1,t}-\tilde\U_{2,t})(t,\eta)d\eta\\ \nn
& = 2\int_0^{1} (\tilde\U_1-\tilde\U_2)(\frac{1}{\sqtC{2}^2}\tilde\Q_2-\frac{1}{\sqtC{1}^2}\tilde\Q_1)(t,\eta)d\eta\\ \nn
& \quad + \frac43\int_0^{1} (\tilde\U_1- \tilde\U_2)(\frac{1}{\sqtC{2}^5}\tilde\U_2^3 \tilde\U_{2,\eta}-\frac{1}{\sqtC{1}^5}\tilde \U_1^3\tilde \U_{1,\eta})(t,\eta)d\eta\\ \nn 
& \quad + 2\int_0^{1}(\tilde\U_1-\tilde\U_2)(\frac{1}{\sqtC{2}^6}\tilde\So_2 \tilde\U_{2,\eta}-\frac{1}{\sqtC{1}^6}\tilde\So_1\tilde\U_{1,\eta})(t,\eta)d\eta\\
& = I_1+I_2+I_3.
\end{align*}
The strategy is to use integration by parts for the last two integrals $I_2$ and $I_3$, while we want to use straight forward estimates for $I_1$, which will finally yield that 
\begin{align*}
\frac{d}{dt} &\norm{\tilde\U_1-\tilde\U_2}^2\\&\leq \bigO(1)\Big(\norm{\tilde\Y_1-\tilde\Y_2}^2+\norm{\tilde\U_1-\tilde\U_2}^2+\norm{\sqP{1}-\sqP{2}}^2+ \vert \sqtC{1}-\sqtC{2}\vert ^2\Big),
\end{align*}
where $\bigO(1)$ denotes some constant which only depends on $\A=\max_j(\sqtC{j})$ and which remains bounded as $\A\to 0$.

{\it The first integral $I_1$:}
Note that we can split $I_1$ as follows
\begin{align*}
I_1&=\mathbbm{1}_{\sqtC{1}\leq\sqtC{2}}2\frac{\sqtC{1}^2-\sqtC{2}^2}{\sqtC{1}^2\sqtC{2}^2}\int_0^1(\tilde \U_1-\tilde\U_2)\tilde \Q_1(t,\eta) d\eta\\
&\quad + \mathbbm{1}_{\sqtC{2}<\sqtC{1}}2\frac{\sqtC{1}^2-\sqtC{2}^2}{\sqtC{1}^2\sqtC{2}^2}\int_0^1 (\tilde \U_1-\tilde \U_2)\tilde \Q_2(t,\eta) d\eta\\
& \quad + 2\frac{1}{\A^2}\int_0^1 (\tilde \U_1-\tilde \U_2)(\tilde \Q_2-\tilde \Q_1)(t,\eta) d\eta\\
& = I_{11}+I_{12}+I_{13}.
\end{align*}
As far as $I_{11}$ is concerned (a similar argument works for $I_{12}$), we have 
\begin{align*}
\vert I_{11}\vert & \leq \mathbbm{1}_{\sqtC{1}\leq \sqtC{2}}4\frac{\vert \sqtC{1}-\sqtC{2}\vert}{\sqtC{1}\sqtC{2}}\int_0^1 \vert \tilde \U_1-\tilde \U_2\vert \tilde \P_1(t,\eta) d\eta\\
& \leq \A^2\vert \sqtC{1}-\sqtC{2}\vert \int_0^1\vert \tilde\U_1-\tilde \U_2\vert (t,\eta) d\eta\\
& \leq \bigO(1) (\norm{\tilde \U_1-\tilde \U_2}^2+\vert \sqtC{1}-\sqtC{2}\vert^2).
\end{align*}
Note that 
\begin{align*}
\vert I_{13}\vert&=\vert 2\frac{1}{\A^2}\int_{0}^{1} (\tilde\U_1-\tilde\U_2)(\tilde\Q_1-\tilde\Q_2)(t,\eta)d\eta\vert \\
&\leq \int_0^{1} ((\tilde\U_1-\tilde\U_2)^2+\frac{1}{\A^4}(\tilde\Q_1-\tilde\Q_2)^2)(t,\eta)d\eta,
\end{align*}
and hence it suffices to show that 
\begin{align} \notag
\norm{\tilde\Q_1-\tilde\Q_2}   & \leq \bigO(1)\A^2 
\Big(\norm{\tilde\Y_1-\tilde\Y_2}+\norm{\tilde\U_1-\tilde\U_2}\\
&\qquad \qquad\qquad+\norm{\sqP{1}-\sqP{2}}+\vert \sqtC{1}-\sqtC{2}\vert \Big),  \label{LipQ}
\end{align}
which is equivalent to 
\begin{align*}
\norm{\tilde\Q_1- \tilde\Q_2}^2  &  \leq \bigO(1)\A^4
\Big(\norm{\tilde\Y_1- \tilde\Y_2}^2+\norm{\tilde\U_1-\tilde\U_2}^2\\
&\qquad \qquad\qquad +\norm{\sqP{1}-\sqP{2}}^2+\vert \sqtC{1}-\sqtC{2}\vert ^2 \Big).
\end{align*}

To begin with, we observe that we can write
\begin{align*}\nn
(\tilde \Q_1-\tilde\Q_2)(t,\eta)&= (\sqtC{1}\tilde\P_1-\sqtC{2}\tilde\P_2)(t,\eta)\\ \nn
& \quad + (\tilde\D_2-\tilde\D_1)(t,\eta)\\ \nn
& = (\sqtC{1}-\sqtC{2})\tilde \P_1\\
& \quad + \sqtC{2}\big(\sqP{1}+\sqP{2}\big)\big(\sqP{1}-\sqP{2}\big)(t,\eta)\\ \nn
& \quad + (\tilde\D_2-\tilde\D_1)(t,\eta)\\ \nn
& = K_1(t,\eta)+K_2(t,\eta)+K_3(t,\eta).
\end{align*}

As far as $K_1(t,\eta)$ is concerned, we have 
\begin{equation*}
\vert \sqtC{1}-\sqtC{2}\vert \norm{\tilde \P_1}_{L^\infty}\leq \frac{\A^4}{4}\vert \sqtC{1}-\sqtC{2}\vert.
\end{equation*}

As far as $K_2(t,\eta)$ is concerned, we have 
\begin{equation*}
\norm{\tilde\P_1-\tilde\P_2}\leq  \A^2\norm{\sqP{1}-\sqP{2}},
\end{equation*}
since $\norm{\sqP{i}}_{L^\infty}$ can be bounded by a constant, which only depends on $\A$.

As far as $K_3(t,\eta)$ is concerned, Lemma~\ref{lemma:D} implies immediately that 
\begin{align*}
\norm{\tilde\D_1-\tilde\D_2}& \leq \bigO(1)\A^2 
\Big(\norm{\tilde\Y_1-\tilde\Y_2}+\norm{\tilde\U_1-\tilde\U_2}\\
&\qquad \qquad\qquad+\norm{\sqP{1}-\sqP{2}}+\vert \sqtC{1}-\sqtC{2}\vert \Big).
\end{align*}
This finishes the proof of \eqref{LipQ}.

\vspace{0.3cm}
{\it The second integral $I_2$:}
Note that we can write 
\begin{align*}
\frac34 I_2 &=\int_0^{1} (\tilde\U_1-\tilde\U_2) (\frac1{\sqtC{2}^5}\tilde\U_2^3\tilde\U_{2,\eta}-\frac{1}{\sqtC{1}^5}\tilde\U_1^3\tilde\U_{1,\eta})(t,\eta)d\eta\\
& = \frac1{\A^5}\int_0^1(\tilde\U_1-\tilde\U_2) (\tilde\U_2^3\tilde\U_{2,\eta}-\tilde\U_1^3\tilde\U_{1,\eta})(t,\eta) d\eta\\ 
& \quad +  \frac{\sqtC{1}^5-\sqtC{2}^5}{\sqtC{1}^5\sqtC{2}^5} \int_0^1 (\tilde\U_1-\tilde\U_2) (\tilde\U_1^3\tilde\U_{1,\eta}\mathbbm{1}_{\sqtC{1}\leq \sqtC{2}}+\tilde\U_2^3\tilde\U_{2,\eta}\mathbbm{1}_{\sqtC{2}< \sqtC{1}})(t,\eta) d\eta\\
& = \frac1{\A^5} \int_0^1 (\tilde\U_1-\tilde\U_2)(\tilde\U_2-\tilde\U_1)\tilde\U_2^2\tilde\U_{2,\eta}(t,\eta) d\eta\\
& \quad + \frac1{\A^5} \int_0^1 (\tilde\U_1-\tilde\U_2)(\tilde\U_2^2\tilde\U_{2,\eta}-\tilde\U_1^2\tilde\U_{1,\eta})\tilde\U_1(t,\eta)d\eta\\
& \quad + \frac{\sqtC{1}^5-\sqtC{2}^5}{\sqtC{1}^5\sqtC{2}^5} \int_0^1 (\tilde\U_1-\tilde\U_2) (\tilde\U_1^3\tilde\U_{1,\eta}\mathbbm{1}_{\sqtC{1}\leq \sqtC{2}}+\tilde\U_2^3\tilde\U_{2,\eta}\mathbbm{1}_{\sqtC{2}< \sqtC{1}})(t,\eta) d\eta\\
& =  \frac{1}{\A^5}\int_0^{1} (\tilde\U_1-\tilde\U_2)(\tilde\U_2-\tilde\U_1)\tilde\U_2^2 \tilde\U_{2,\eta}(t,\eta) d\eta\\
& \quad +\frac1{\A^5}\int_0^{1} (\tilde\U_1-\tilde\U_2)(\tilde\U_2^2-\tilde\U_1^2)\tilde\U_{2,\eta} \tilde\U_1(t,\eta) \mathbbm{1}_{\tilde\U_1^2\leq \tilde\U_2^2}(t,\eta)d\eta\\
& \quad +\frac{1}{\A^5} \int_0^{1} (\tilde\U_1-\tilde\U_2)(\tilde\U_2^2-\tilde\U_1^2)\tilde\U_1\tilde\U_{1,\eta}(t,\eta) \mathbbm{1}_{\tilde \U_2^2<\tilde\U_1^2}(t,\eta)d\eta\\
& \quad +\frac{1}{\A^5} \int_0^{1} (\tilde\U_1-\tilde\U_2)(\tilde\U_{2,\eta}-\tilde\U_{1,\eta})\tilde\U_1 \min_j(\tilde\U_j^2)(t,\eta) d\eta\\
& \quad + \frac{\sqtC{1}^5-\sqtC{2}^5}{\sqtC{1}^5\sqtC{2}^5} \int_0^1 (\tilde\U_1-\tilde\U_2) (\tilde\U_1^3\tilde\U_{1,\eta}\mathbbm{1}_{\sqtC{1}\leq \sqtC{2}}+\tilde\U_2^3\tilde\U_{2,\eta}\mathbbm{1}_{\sqtC{2}< \sqtC{1}})(t,\eta) d\eta\\
& =  \frac{1}{\A^5}\int_0^{1} (\tilde\U_1-\tilde\U_2)( \tilde\U_2-\tilde\U_1) \tilde\U_2^2\tilde\U_{2,\eta}(t,\eta) d\eta\\
& \quad +\frac{1}{\A^5}\int_0^{1} (\tilde\U_1-\tilde \U_2)(\tilde\U_2^2-\tilde\U_1^2)\tilde\U_{2,\eta} \tilde\U_1(t,\eta) \mathbbm{1}_{ \tilde\U_1^2\leq  \tilde\U_2^2}(t,\eta)d\eta\\
& \quad +\frac1{\A^5} \int_0^{1} (\tilde\U_1-\tilde\U_2)(\tilde\U_2^2-\tilde\U_1^2)\tilde\U_1\tilde\U_{1,\eta}(t,\eta) \mathbbm{1}_{\tilde\U_2^2<\tilde\U_1^2}(t,\eta)d\eta\\
& \quad + \frac1{2\A^5} \int_0^{1} (\tilde\U_1-\tilde \U_2)^2 \frac{d}{d\eta} (\tilde\U_1\min_j(\tilde\U_j^2))(t,\eta) d\eta\\
& \quad +  \frac{\sqtC{1}^5-\sqtC{2}^5}{\sqtC{1}^5\sqtC{2}^5} \int_0^1 (\tilde\U_1-\tilde\U_2) (\tilde\U_1^3\tilde\U_{1,\eta}\mathbbm{1}_{\sqtC{1}\leq \sqtC{2}}+\tilde\U_2^3\tilde\U_{2,\eta}\mathbbm{1}_{\sqtC{2}< \sqtC{1}})(t,\eta) d\eta,
\end{align*}
where we used integration by parts in the last step together with $\tilde\U_i(t,\eta)\to 0$ as $\eta\to 0$, $1$. 
As far as the derivative in the last integral is concerned, observe that 
\begin{equation} 
\vert \tilde\U_{1,\eta}\min_j(\tilde\U_j^2)(t,\eta)\vert \leq \vert\tilde \U_1^2\tilde\U_{1,\eta}(t,\eta)\vert \leq \frac{\A^4}{2}\norm{\tilde\U_1}_{L^\infty}\leq \frac{\A^6}{2\sqrt{2}},
\end{equation}
since $2\vert\tilde \U_i\tilde\U_{i,\eta}(t,\eta)\vert \leq \sqtC{i}^4\leq \A^4$ for all $t$ and $\eta$. Furthermore, we established before that the function $\theta\mapsto \min_j(\tilde\U_j^2(t,\eta))$ is Lipschitz continuous with Lipschitz constant at most $\A^4$, see \eqref{eq:U2lip}. Thus
\begin{equation} \label{eq:470}
\vert \frac{d}{d\eta} \tilde\U_1\min_j(\tilde\U_j^2)(t,\eta)\vert \leq \frac32 \A^4\norm{\tilde\U_1}_{L^\infty} \leq \frac{3}{2\sqrt{2}}\A^6,
\end{equation}
and 
\begin{equation*}
\vert \frac{1}{\sqtC{1}^5\sqtC{2}^5} \tilde\U_i^3\tilde\U_{i,n} \mathbbm{1}_{\sqtC{i}=\ma}\vert \leq \frac{1}{\sqtC{1}^5\sqtC{2}^5} \frac{1}{4}\sqtC{i}^8 \mathbbm{1}_{\sqtC{i}=\ma} \leq \frac{1}{\A^2}.
\end{equation*}
Finally, we get that 
\begin{align*}
\frac34\vert I_2\vert& =\vert \int_0^{1} (\tilde\U_1-\tilde\U_2) (\frac{1}{\sqtC{2}^5}\tilde\U_2^3\tilde\U_{2,\eta}-\frac{1}{\sqtC{1}^5}\tilde\U_1^3\tilde\U_{1,\eta})(t,\eta)d\eta\vert \\
 & \leq \bigO(1)\norm{\tilde\U_1-\tilde\U_2}^2+ \frac{2}{\A^2} \vert \sqtC{1}^5-\sqtC{2}^5\vert \int_0^1\vert \tilde\U_1-\tilde\U_2\vert (t,\eta) d\eta\\
 & \leq \bigO(1)\big( \norm{\tilde\U_1-\tilde\U_2}^2+\vert \sqtC{1}-\sqtC{2}\vert ^2\big).
\end{align*}

\vspace{0.3cm}
{\it The third integral $I_3$:} We will split $I_3$ in several terms that we treat separately, and combine them in the end.

Recall that we introduced the functions $\tilde\U_i^-=\min(0, \tilde\U_i)$ and $\tilde\U_i^+=\max(0,\tilde\U_i)$ with properties \eqref{eq:pm1} and \eqref{eq:pm2}.  
We write
\begin{equation*}
I_3=\frac43I_{31}-4I_{32}-2I_{33},
\end{equation*}
where 
\begin{subequations}  \label{eq:I_for_U}
\begin{align}
I_{31}&= \int_0^{1} (\tilde\U_1-\tilde\U_2)(t, \eta)\Big(\frac{1}{\sqtC{2}^6}\tilde\U_{2,\eta}(t, \eta)\int_0^{1}
 e^{-\frac{1}{\sqtC{2}}\vert \tilde\Y_2(t,\eta)-\tilde\Y_2(t,\theta)\vert }\tilde \U_2^3\tilde\Y_{2,\eta}(t,\theta)d\theta \notag\\
&\qquad\qquad\qquad -\frac{1}{\sqtC{1}^6}\tilde\U_{1,\eta}(t, \eta)\int_0^{1} e^{-\frac{1}{\sqtC{1}}\vert \tilde\Y_1(t,\eta)-\tilde\Y_1(t,\theta)\vert } \tilde\U_1^3\tilde\Y_{1,\eta}(t,\theta)d\theta\Big)d\eta, \label{eq:I31_for_U}
\\[2mm]
I_{32}&=\int_0^{1} (\tilde\U_1- \tilde\U_2)(t, \eta)\Big(\frac{1}{\sqtC{2}^6}\tilde\U_{2,\eta}(t, \eta)\int_0^{1}
 e^{-\frac{1}{\sqtC{2}}\vert \tilde\Y_2(t,\eta)-\tilde\Y_2(t,\theta)\vert }\tilde\P_2\tilde\U_2\tilde\Y_{2,\eta}(t,\theta)d\theta\notag\\
&\qquad\qquad\qquad- \frac{1}{\sqtC{1}^6}\tilde\U_{1,\eta}(t, \eta)\int_0^{1} e^{-\frac{1}{\sqtC{1}}\vert \tilde\Y_1(t,\eta)-\tilde\Y_1(t,\theta)\vert } 
\tilde\P_1\tilde\U_1\tilde\Y_{1,\eta}(t,\theta)d\theta\Big)d\eta,\label{eq:I32_for_U} \\[2mm]
I_{33}&=\int_0^{1} (\tilde\U_1-\tilde\U_2)(t, \eta)\Big(\frac{1}{\sqtC{2}^6}\tilde\U_{2,\eta}(t, \eta)\int_0^{1}
 e^{-\frac{1}{\sqtC{2}}\vert \tilde\Y_2(t,\eta)-\tilde\Y_2(t,\theta)\vert }\tilde\Q_2\tilde\U_{2,\eta}(t,\theta)d\theta\notag\\
&\qquad\qquad\qquad- \frac{1}{\sqtC{1}^6}\tilde\U_{1,\eta}(t, \eta)\int_0^{1} e^{-\frac{1}{\sqtC{1}}\vert \tilde\Y_1(t,\eta)-\tilde\Y_1(t,\theta)\vert } 
\tilde\Q_1\tilde\U_{1,\eta}(t,\theta)d\theta\Big)d\eta. \label{eq:I33_for_U}
\end{align}
\end{subequations}
Thus we have 
\begin{align*}
I_{31}&= \int_0^{1} (\tilde\U_1-\tilde\U_2)(t, \eta)\\
& \qquad  \times\Big(\frac{1}{\sqtC{2}^6}\tilde\U_{2,\eta}(t,\eta) \int_0^{1} e^{-\frac{1}{\sqtC{2}}\vert \tilde\Y_2(t,\eta)-\tilde\Y_2(t,\theta)\vert }\tilde \U_2^3\tilde\Y_{2,\eta}(t,\theta) d\theta\\
& \qquad \qquad\qquad \qquad -\frac{1}{\sqtC{1}^6}\tilde\U_{1,\eta}(t,\eta)\int_0^{1} e^{-\frac{1}{\sqtC{1}}\vert \tilde\Y_1(t,\eta)-\tilde\Y_1(t,\theta)\vert }\tilde\U_1^3\tilde\Y_{1,\eta}(t,\theta) d\theta\Big) d\eta \\
& = \int_0^{1}(\tilde\U_1-\tilde\U_2)(t,\eta)\\
& \qquad \times \Big(\frac{1}{\sqtC{2}^6}\tilde\U_{2,\eta}(t,\eta) \int_0^{1} e^{-\frac{1}{\sqtC{2}}\vert \tilde\Y_2(t,\eta)-\tilde\Y_2(t,\theta)\vert } (\tilde\V_2^-)^3\tilde\Y_{2,\eta}(t,\theta) d\theta\\
& \qquad \qquad\qquad \qquad -\frac{1}{\sqtC{1}^6}\tilde\U_{1,\eta}(t,\eta)\int_0^{1} e^{-\frac{1}{\sqtC{1}}\vert \tilde\Y_1(t,\eta)-\tilde\Y_1(t,\theta)\vert }(\tilde\V_1^-)^3\tilde\Y_{1,\eta}(t,\theta) d\theta\Big) d\eta\\
& \quad + \int_0^{1}(\tilde\U_1-\tilde\U_2)(t,\eta)\\
& \qquad  \times \Big(\frac{1}{\sqtC{2}^6}\tilde\U_{2,\eta}(t,\eta) \int_0^{1} e^{-\frac{1}{\sqtC{2}}\vert \tilde\Y_2(t,\eta)-\tilde\Y_2(t,\theta)\vert } (\tilde\V_2^+)^3\tilde\Y_{2,\eta}(t,\theta) d\theta\\
& \qquad \qquad\qquad \qquad-\frac{1}{\sqtC{1}^6} \tilde\U_{1,\eta}(t,\eta)\int_0^{1} e^{-\frac{1}{\sqtC{1}}\vert \tilde\Y_1(t,\eta)-\tilde\Y_1(t,\theta)\vert }(\tilde\V_1^+)^3\tilde\Y_{1,\eta}(t,\theta) d\theta\Big) d\eta.
\end{align*}
Since both inner integrals have the same structure, it suffices to consider the second integral.  Recall the rewrite \eqref{eq:triks}.
Thus we need to estimate the following term:
\begin{align}\nn
& \frac{1}{\sqtC{2}^6}  \tilde\U_{2,\eta}(t,\eta) \int_0^{\eta} e^{-\frac{1}{\sqtC{2}}( \tilde\Y_2(t,\eta)-\tilde\Y_2(t,\theta))} (\tilde\V_2^+)^3\tilde\Y_{2,\eta}(t,\theta) d\theta\\ \nn
&\qquad\qquad\qquad\qquad\qquad-\frac{1}{\sqtC{1}^6}\tilde\U_{1,\eta}(t,\eta)\int_0^{\eta} e^{-\frac{1}{\sqtC{1}}(\tilde\Y_1(t,\eta)-\tilde\Y_1(t,\theta))}(\tilde\V_1^+)^3\tilde \Y_{1,\eta}(t,\theta) d\theta\\ \nn
 & = \frac{1}{\A^6}\Big(\tilde\U_{2,\eta}(t,\eta) \int_0^{\eta} e^{-\frac{1}{\sqtC{2}}(\tilde \Y_2(t,\eta)-\tilde\Y_2(t,\theta)) } (\tilde\V_2^+)^3 \tilde\Y_{2,\eta}(t,\theta) d\theta\\ \nn
 &\qquad\qquad\qquad\qquad\qquad- \tilde\U_{1,\eta}(t,\eta)\int_0^{\eta} e^{-\frac{1}{\sqtC{1}}(\tilde \Y_1(t,\eta)-\tilde\Y_1(t,\theta))}(\tilde\V_1^+)^3\tilde\Y_{1,\eta}(t,\theta) d\theta\Big)\\  \nn
 & \quad + \mathbbm{1}_{\sqtC{1}\leq \sqtC{2}} (\frac{1}{\sqtC{2}^6}-\frac{1}{\sqtC{1}^6}) \left(\tilde\U_{1,\eta}(t,\eta)\int_0^\eta e^{-\frac{1}{\sqtC{1}}(\tilde\Y_1(t,\eta)-\tilde\Y_1(t,\theta))} (\tilde\V_1^+)^3\tilde\Y_{1,\eta}(t,\theta)d\theta\right)\\  \nn
 & \quad + \mathbbm{1}_{\sqtC{2}< \sqtC{1}} (\frac{1}{\sqtC{2}^6}-\frac1{\sqtC{1}^6})\left( \tilde\U_{2,\eta}(t,\eta) \int_0^\eta e^{-\frac{1}{\sqtC{2}}(\tilde\Y_2(t,\eta)-\tilde\Y_2(t,\theta))} (\tilde\V_2^+)^3 \tilde\Y_{2,\eta} (t,\theta) d\theta\right)\\  \nn
 &= \frac{1}{\A^6} \tilde\U_{2,\eta}(t,\eta)\int_0^\eta e^{-\frac{1}{\sqtC{2}}(\tilde \Y_2(t,\eta)-\tilde\Y_2(t,\theta))}((\tilde\V_2^+)^3-(\tilde\V_1^+)^3)\tilde\Y_{2,\eta}(t,\theta)\mathbbm{1}_{\tilde\V_1^+\leq \tilde\V_2^+}(t,\theta)d\theta\\ \nn
 & \quad +\frac1{\A^6}\tilde\U_{1,\eta}(t,\eta)\int_0^\eta e^{-\frac{1}{\sqtC{1}}(\tilde \Y_1(t,\eta)-\tilde \Y_1(t,\theta))}((\tilde \V_2^+)^3-(\tilde\V_1^+)^3)\tilde \Y_{1,\eta}(t,\theta) \mathbbm{1}_{\tilde\V_2^+<\tilde\V_1^+}(t,\theta) d\theta\\ \nn
 & \quad +\frac{1}{\A^6} \tilde \U_{2,\eta}(t,\eta)\int_0^\eta e^{-\frac{1}{\sqtC{2}}(\tilde\Y_2(t,\eta) -\tilde \Y_2(t,\theta))}\min_j(\tilde\V_j^+)^3 \tilde \Y_{2,\eta} (t,\theta) d\theta\\ \nn
 & \quad - \frac{1}{\A^6} \tilde \U_{1,\eta} (t,\eta) \int_0^\eta e^{-\frac{1}{\sqtC{1}}(\tilde \Y_1(t,\eta)-\tilde \Y_1(t,\theta))} \min_j(\tilde\V_j^+)^3\tilde\Y_{1,\eta}(t,\theta) d\theta\\ \nn
  & \quad + \mathbbm{1}_{\sqtC{1}\leq \sqtC{2}} (\frac{1}{\sqtC{2}^6}-\frac{1}{\sqtC{1}^6}) \left(\tilde\U_{1,\eta}(t,\eta)\int_0^\eta e^{-\frac{1}{\sqtC{1}}(\tilde\Y_1(t,\eta)-\tilde\Y_1(t,\theta))} (\tilde\V_1^+)^3\tilde\Y_{1,\eta}(t,\theta)d\theta\right)\\ \nn
 & \quad + \mathbbm{1}_{\sqtC{2}< \sqtC{1}} (\frac{1}{\sqtC{2}^6}-\frac1{\sqtC{1}^6})\left( \tilde\U_{2,\eta}(t,\eta) \int_0^\eta e^{-\frac{1}{\sqtC{2}}(\tilde\Y_2(t,\eta)-\tilde\Y_2(t,\theta))} (\tilde\V_2^+)^3 \tilde\Y_{2,\eta} (t,\theta) d\theta\right)\\ \nn
 & = \frac{1}{\A^6} \tilde\U_{2,\eta}(t,\eta)\int_0^\eta e^{-\frac{1}{\sqtC{2}}(\tilde \Y_2(t,\eta)-\tilde\Y_2(t,\theta))}((\tilde\V_2^+)^3-(\tilde\V_1^+)^3)\tilde\Y_{2,\eta}(t,\theta)\mathbbm{1}_{\tilde\V_1^+\leq \tilde\V_2^+}(t,\theta)d\theta\\ \nn
 & \quad +\frac1{\A^6}\tilde\U_{1,\eta}(t,\eta)\int_0^\eta e^{-\frac{1}{\sqtC{1}}(\tilde \Y_1(t,\eta)-\tilde \Y_1(t,\theta))}((\tilde \V_2^+)^3-(\tilde\V_1^+)^3)\tilde \Y_{1,\eta}(t,\theta) \mathbbm{1}_{\tilde\V_2^+<\tilde\V_1^+}(t,\theta) d\theta\\ \nn
 & \quad  +\frac{1}{\A^6}\mathbbm{1}_{\sqtC{1}\leq \sqtC{2}} \tilde \U_{2,\eta}(t,\eta)\int_0^\eta (e^{-\frac{1}{\sqtC{2}}(\tilde \Y_2(t,\eta)-\tilde \Y_2(t,\theta))}-e^{-\frac{1}{\sqtC{1}}(\tilde \Y_2(t,\eta)-\tilde \Y_2(t,\theta))})\nn \\
&\qquad\qquad\qquad\qquad\qquad\qquad\qquad\qquad\qquad\qquad\times\min_j(\tilde \V_j^+)^3\tilde \Y_{2,\eta}(t,\theta) d\theta \nn \\ \nn
 & \quad +\frac{1}{\A^6} \mathbbm{1}_{\sqtC{2}<\sqtC{1}}\tilde \U_{1,\eta}(t,\eta)\int_0^\eta (e^{-\frac{1}{\sqtC{2}}(\tilde \Y_1(t,\eta)-\tilde \Y_1(t,\theta))}-e^{-\frac{1}{\sqtC{1}}(\tilde \Y_1(t,\eta)-\tilde \Y_1(t,\theta))})\nn \\
&\qquad\qquad\qquad\qquad\qquad\qquad\qquad\qquad\qquad\qquad\times\min_j(\tilde \V_j^+)\tilde \Y_{1,\eta}(t,\theta)d\theta \nn \\ \nn
 & \quad +\frac{1}{\A^6} \tilde \U_{2,\eta}(t,\eta) \int_0^\eta (e^{-\frac{1}{\ma}(\tilde \Y_2(t,\eta)-\tilde \Y_2(t,\theta))}-e^{-\frac{1}{\ma}(\tilde \Y_1(t,\eta)-\tilde \Y_1(t,\theta))})\\ \nn
 &\qquad\qquad\qquad\qquad\qquad\qquad\times\min_j(\tilde\V_j^+)^3 \tilde\Y_{2,\eta} (t,\theta)\mathbbm{1}_{B(\eta)}(t,\theta) d\theta\\ \nn
 & \quad -\frac{1}{\A^6} \tilde \U_{1,\eta} (t,\eta)\int_0^\eta (e^{-\frac{1}{\ma}(\tilde \Y_1(t,\eta)-\tilde \Y_1(t,\theta))}-e^{-\frac{1}{\ma}(\tilde \Y_2(t,\eta)-\tilde \Y_2(t,\theta))}\\ \nn
 &\qquad\qquad\qquad\qquad\qquad\qquad\times\min_j(\tilde\V_j^+)^3\tilde \Y_{1,\eta} (t,\theta) \mathbbm{1}_{B^c(\eta)}(t,\theta) d\theta)\\ \nn
 &\quad + \frac{1}{\A^6} \tilde \U_{2,\eta} (t,\eta)\int_0^\eta \min_j(e^{-\frac{1}{\ma}(\tilde \Y_j(t,\eta)-\tilde\Y_j(t,\theta))})\min_j(\tilde\V_j^+)^3\tilde \Y_{2,\eta} (t,\theta) d\theta\\ \nn
 & \quad -\frac{1}{\A^6} \tilde\U_{1,\eta} (t,\eta)\int_0^\eta \min_j(e^{-\frac{1}{\ma}(\tilde \Y_j(t,\eta)-\tilde\Y_j(t,\theta))})\min_j(\tilde \V_j^+)^3\tilde \Y_{1,\eta} (t,\theta) d\theta\\ \nn
 & \quad + \mathbbm{1}_{\sqtC{1}\leq \sqtC{2}} (\frac{1}{\sqtC{2}^6}-\frac{1}{\sqtC{1}^6}) \left(\tilde\U_{1,\eta}(t,\eta)\int_0^\eta e^{-\frac{1}{\sqtC{1}}(\tilde\Y_1(t,\eta)-\tilde\Y_1(t,\theta))} (\tilde\V_1^+)^3\tilde\Y_{1,\eta}(t,\theta)d\theta\right)\\ \nn
 & \quad + \mathbbm{1}_{\sqtC{2}< \sqtC{1}} (\frac{1}{\sqtC{2}^6}-\frac1{\sqtC{1}^6})\left( \tilde\U_{2,\eta}(t,\eta) \int_0^\eta e^{-\frac{1}{\sqtC{2}}(\tilde\Y_2(t,\eta)-\tilde\Y_2(t,\theta))} (\tilde\V_2^+)^3 \hat\Y_{2,\eta} (t,\theta) d\theta\right)\\ \label{eq:splitU}
 & = (J_1+J_2+J_3+J_4+J_5+J_6+ J_7+J_8+J_9+J_{10})(t,\eta),
 \end{align}
 where $B(\eta)$ is given by \eqref{Def:Bn}.

As far as the integral which contains $J_1$ (a similar argument works for $J_2$) is concerned we have 
\begin{align} \nn
& \vert \int_0^1J_1(\tilde \U_1-\tilde\U_2)(t,\eta)d\eta\vert \\ \nn
& \leq \norm{\tilde \U_1-\tilde\U_2}^2 + \int_0^1 J_1^2(t,\eta) d\eta\\ \nn
& \leq \norm{\tilde \U_1-\tilde \U_2} _{L^2}^2  + \frac1{A^{12}} \int_0^1 \tilde \U_{2,\eta}^2(t,\eta) \nn \\
&  \qquad \quad\times \left(\int_0^\eta e^{-\frac{1}{\sqtC{2}}(\tilde \Y_2(t,\eta)-\tilde \Y_2(t,\theta))}((\tilde \V_2^+)^3-(\tilde \V_1^+)^3)\tilde \Y_{2,\eta} (t,\theta) \mathbbm{1}_{\tilde\V_1^+\leq \tilde\V_2^+} (t,\theta) d\theta\right)^2 d\eta \nn \\ \nn
& \leq \norm{\tilde \U_1-\tilde \U_2}^2 \\ \nn
& \quad +\frac1{\A^9} \int_0^1 \tilde \Y_{2,\eta} (t,\eta)\left(\int_0^\eta e^{-\frac{1}{\sqtC{2}}(\tilde \Y_2(t,\eta)-\tilde \Y_2(t,\theta))}3(\tilde \V_2^+)^2 \vert \tilde \V_2^+-\tilde \V_1^+\vert \tilde \Y_{2,\eta} (t,\theta) d\theta\right)^2 d\eta\\ \nn
& \leq \norm{\tilde \U_1-\tilde \U_2}^2 \\ \nn
& \quad + \frac{9}{\A^4}\int_0^1 \tilde \Y_{2,\eta} (t,\eta) \Big(\int_0^\eta e^{-\frac{1}{\sqtC{2}}(\tilde \Y_2(t,\eta)-\tilde \Y_2(t,\theta))} \tilde\U_2^2\tilde \Y_{2,\eta} (t,\theta)d\theta\Big) d\eta \norm{\tilde\U_1-\tilde \U_2}^2\\ \nn
& \leq \big(1+ \frac{36}{\A^3}\int_0^1 \tilde\P_2\tilde \Y_{2,\eta} (t,\eta) d\eta\big)\norm{\tilde \U_1-\tilde \U_2}^2\\ 
& \leq \bigO(1)\norm{\tilde \U_1-\tilde \U_2}^2, \label{eq:J1J2}
\end{align}
where we used \eqref{eq:all_estimatesE}, \eqref{eq:all_estimatesG}, and \eqref{eq:all_estimatesH}. 
 
 As far as the third and the fourth term $J_3$ and $J_4$ are concerned, they again have the same structure, and hence we only consider the integral corresponding to $J_3$. 
\begin{align*}
& \vert \int_0^1 J_3(\tilde \U_1-\tilde \U_2)(t,\eta) d\eta\vert \\
& \quad = \mathbbm{1}_{\sqtC{1}\leq \sqtC{2}}\frac{1}{\A^6} \vert \int_0^1 (\tilde \U_1-\tilde \U_2)\tilde \U_{2,\eta}(t,\eta)\\ &\qquad\times\int_0^\eta (e^{-\frac{1}{\sqtC{2}}(\tilde \Y_2(t,\eta)-\tilde \Y_2(t,\theta))}-e^{-\frac{1}{\sqtC{1}}(\tilde \Y_2(t,\eta)-\tilde \Y_2(t,\theta))})\min_j(\tilde \V_j^+)^3\tilde \Y_{2,\eta}(t,\theta) d\theta d\eta\vert\\
& \quad \leq \frac{2\sqrt{2}\ma}{\A^6 e}  \int_0^1 (\tilde \U_1-\tilde \U_2)\vert \tilde \U_{2,\eta}\vert(t,\eta)\\
&\qquad\qquad\qquad\times \int_0^\eta e^{-\frac{3}{4\sqtC{2}}(\tilde \Y_2(t,\eta)-\tilde \Y_2(t,\theta))}\vert  \tilde \U_2^2\tilde \Y_{2,\eta}(t,\theta) d\theta d\eta\vert \sqtC{1}-\sqtC{2}\vert\\
& \quad \leq \norm{\tilde \U_1-\tilde \U_2}^2  + \frac{8}{\A^{10}e^2} \int_0^1 \tilde \U_{2,\eta}^2(t,\eta)\Big(\int_0^\eta e^{-\frac{1}{\sqtC{2}}(\tilde \Y_2(t,\eta)-\tilde \Y_2(t,\theta))}\tilde \U_2^2\tilde \Y_{2,\eta}(t,\theta) d\theta\Big)\\
& \qquad \qquad \qquad \qquad \qquad\times \Big(\int_0^\eta e^{-\frac{1}{2\sqtC{2}}(\tilde \Y_2(t,\eta)-\tilde \Y_2(t,\theta))} \tilde \U_2^2\tilde \Y_{2,\eta}(t,\theta) d\theta\Big) d\eta\vert \sqtC{1}-\sqtC{2}\vert^2\\
& \quad \leq \norm{\tilde \U_1-\tilde \U_2}^2 
+ \frac{32}{\A e^2} \int_0^1 \tilde \P_2\tilde \Y_{2,\eta}(t,\eta) d\eta \vert \sqtC{1}-\sqtC{2}\vert^2\\
& \leq \bigO(1)(\norm{\tilde \U_1-\tilde \U_2}^2+\vert \sqtC{1}-\sqtC{2}\vert^2).
\end{align*} 
 
 As far as the third and the fourth term $J_5$ and $J_6$ are concerned, they again have the same structure, and hence we only consider the integral corresponding to $J_5$. Thus, using \eqref{Diff:Exp},
\begin{align} \nn
&\vert\int_0^{1} J_5(\tilde\U_1-\tilde\U_2)(t,\eta)d\eta  \vert\\ \nn
&\quad=\frac{1}{\A^6}\vert \int_0^{1}  (\tilde\U_1-\tilde\U_2) \tilde\U_{2,\eta}(t,\eta) \Big(\int_0^\eta (e^{-\frac{1}{\ma}(\tilde\Y_2(t,\eta)-\tilde\Y_2(t,\theta))}-e^{-\frac{1}{\ma}(\tilde\Y_1(t,\eta)-\tilde\Y_1(t,\theta))}) \\ \nn
&\qquad\qquad\qquad\qquad\qquad\qquad\qquad\qquad\times\min_j(\tilde\V_j^+)^3\tilde\Y_{2,\eta}(t,\theta)\mathbbm{1}_{B(\eta)}(t,\theta) d\theta\Big)d\eta\vert \\ \nn
&\quad \leq \frac{1}{\ma\A^6}\int_0^{1} \vert \tilde\U_1-\tilde \U_2 \vert \vert\tilde\U_{2,\eta}\vert(t,\eta) \\ \nn
& \qquad \qquad \times \Big(\int_0^\eta (\vert \tilde\Y_2(t,\eta)-\tilde \Y_1(t,\eta)\vert +\vert \tilde\Y_2(t,\theta)-\tilde\Y_1(t,\theta) \vert) \\ \nn
&\qquad\qquad\qquad\qquad\qquad\qquad\qquad\times e^{-\frac{1}{\ma}(\tilde\Y_2(t,\eta)-\tilde\Y_2(t,\theta))} \min_j(\tilde\V_j^+)^3 \tilde\Y_{2,\eta}(t,\theta) d\theta\Big) d\eta\\ \nn
&\quad \leq \frac{1}{\ma\A^6}\int_0^{1} \vert \tilde\U_1-\tilde\U_2\vert \vert \tilde \Y_1-\tilde\Y_2\vert  \vert \tilde\U_{2,\eta}\vert (t,\eta)\\ \nn
&\qquad\qquad\qquad\qquad\qquad\qquad\times\int_0^{\eta} e^{-\frac{1}{\ma}(\tilde\Y_2(t,\eta)-\tilde\Y_2(t,\theta))} \min_j(\tilde\V_j^+)^3\tilde\Y_{2,\eta}(t,\theta) d\theta d\eta\\ \nn
& \qquad +\norm{\tilde\U_1-\tilde\U_2}^2 + \frac{1}{\ma^2\A^{12}}\int_0^{1} \tilde\U_{2,\eta}^2(t,\eta)\\ \nn
&\qquad\qquad\times\left( \int_0^\eta \vert \tilde\Y_1(t,\theta)-\tilde\Y_2(t,\theta)\vert e^{-\frac{1}{\ma}(\tilde\Y_2(t,\eta)-\tilde\Y_2(t,\theta))} \min_j(\tilde\V_j^+)^3 \tilde\Y_{2,\eta}(t,\theta) d\theta \right)^2 d\eta\\ \nn
& \quad\leq \frac{4\ma}{\A^5}\int_0^{1} \vert \tilde\U_1-\tilde\U_2\vert \vert \tilde\Y_1-\tilde \Y_2\vert \tilde\P_2\vert \tilde\U_{2,\eta}\vert (t,\eta) d\eta + \norm{\tilde\U_1-\tilde\U_2}^2\\ \nn
& \qquad + \frac{1}{\ma^2\A^9}\int_0^{1} \tilde\Y_{2,\eta}(t,\eta) \left(\int_0^\eta e^{-\frac{1}{\sqtC{2}}(\tilde\Y_2(t,\eta)-\tilde\Y_2(t,\theta))} \tilde\U_2^2\tilde\Y_{2,\eta}(t,\theta) d\theta\right)\\  \nn
& \qquad \qquad\qquad  \times \left(\int_0^{\eta} (\tilde\Y_1-\tilde\Y_2)^2(t,\theta) e^{-\frac{1}{\ma}(\tilde\Y_2(t,\eta)-\tilde\Y_2(t,\theta))}\min_j(\tilde\U_j^4) \tilde\Y_{2,\eta}(t,\theta) d\theta \right)d\eta\\ \nn
&\quad \leq (\sqrt{2}\A^2+1)(\norm{\tilde\Y_1-\tilde\Y_2}^2+\norm{\tilde\U_1-\tilde\U_2}^2)
 +\frac{2}{\A^6} \int_0^{1} \tilde\P_2\tilde\Y_{2,\eta}(t,\eta) \\ \nn
&\qquad\qquad\quad\times\left(\int_0^{\eta} (\tilde\Y_1-\tilde \Y_2)^2(t,\theta) e^{-\frac{1}{\ma}(\tilde\Y_2(t,\eta)-\tilde\Y_2(t,\theta))}\tilde\U_2^2 \tilde\Y_{2,\eta}(t,\theta) d\theta \right)d\eta\\ 
& \quad\leq \bigO(1) (\norm{\tilde \Y_1-\tilde \Y_2}^2 +\norm{\tilde\U_1-\tilde\U_2}^2).\label{eq:J3}
\end{align}

Next are the terms $J_7$ and $J_8$. Therefore recall \eqref{Def:EE}, which implies
\begin{align*}
J_7+J_8&= \frac{1}{\A^6}\tilde\U_{2,\eta}(t,\eta) \\
&\quad\times \left(\int_0^\eta \min_j(e^{-\frac{1}{\ma}(\tilde \Y_j(t,\eta)-\tilde\Y_j(t,\theta))})\min_j(\tilde\V_j^+)^3\tilde\Y_{2,\eta}(t,\theta) d\theta\right)\\
& \quad - \frac{1}{\A^6}\tilde\U_{1,\eta}(t,\eta)\\
&\quad\times \left(\int_0^\eta \min_j(e^{-\frac{1}{\ma}(\tilde \Y_j(t,\eta)-\tilde\Y_j(t,\theta))})\min_j(\tilde\V_j^+)^3\tilde \Y_{1,\eta}(t,\theta) d\theta\right)\\
& =  \frac{1}{\A^6}(\tilde\U_{2,\eta}-\tilde\U_{1,\eta})(t,\eta)\\
&\quad\times \min_k\Big[ \int_0^\eta \min_j(e^{-\frac{1}{\ma}(\tilde \Y_j(t,\eta)-\tilde\Y_j(t,\theta))})\min_j(\tilde\V_j^+)^3\tilde\Y_{k,\eta}(t,\theta) d\theta\Big]\\
& \quad + \frac{1}{\A^6}\tilde\U_{1,\eta}(t,\eta) \int_0^\eta \min_j(e^{-\frac{1}{\ma}(\tilde \Y_j(t,\eta)-\tilde\Y_j(t,\theta))})\\
&\quad\qquad\qquad\qquad\qquad\qquad\times \min_j(\tilde\V_j^+)^3(\tilde\Y_{2,\eta}-\tilde \Y_{1,\eta})(t, \theta) d\theta\mathbbm{1}_E(t,\eta)\\
& \quad +\frac{1}{\A^6}\tilde\U_{2,\eta}(t,\eta) \int_0^\eta \min_j(e^{-\frac{1}{\ma}(\tilde \Y_j(t,\eta)-\tilde\Y_j(t,\theta))})\\
&\quad\qquad\qquad\qquad\qquad\qquad\times \min_j(\tilde\V_j^+)^3(\tilde\Y_{2,\eta}-\tilde \Y_{1,\eta})(t,\theta) d\theta \mathbbm{1}_{E^c}(t,\eta)\\
& = L_1+L_2+L_3.
\end{align*}

As far as the first term $L_1$ is concerned, the corresponding integral can be estimates as follows, using \eqref{eq:MinMax},
\begin{align} \nn
&\vert \int_0^{1} L_1(\tilde\U_1-\tilde \U_2)(t,\eta)d\eta\vert \\  \nn
&\quad=\frac{1}{\A^6}\Big\vert \int_0^{1}   (\tilde\U_1-\tilde \U_2)(\tilde\U_{2,\eta}-\tilde\U_{1,\eta})(t,\eta) \\  \nn
&\qquad\qquad\times\min_k\Big[ \int_0^\eta \min_j(e^{-\frac{1}{\ma}(\tilde \Y_j(t,\eta)-\tilde\Y_j(t,\theta))})\min_j(\tilde\V_j^+)^3\tilde\Y_{k,\eta}(t,\theta) d\theta\Big]d\eta\Big\vert \\  \nn
&\quad = \Big\vert -\frac1{2\A^6} (\tilde\U_1-\tilde\U_2)^2(t,\eta)\\ \nn
&\qquad\qquad\times \min_k\Big[ \int_0^\eta \min_j(e^{-\frac{1}{\ma}(\tilde \Y_j(t,\eta)-\tilde\Y_j(t,\theta))})\min_j(\tilde\V_j^+)^3\tilde\Y_{k,\eta}(t,\theta) d\theta\Big]
\Big|_{\eta=0}^{1}\\   \nn
&\qquad +\frac1{2\A^6} \int_0^{1} (\tilde\U_1-\tilde\U_2)^2(t,\eta)\\  \nn
&\qquad\quad\times\frac{d}{d\eta} \min_k\Big[ \int_0^\eta \min_j(e^{-\frac{1}{\ma}(\tilde \Y_j(t,\eta)-\tilde\Y_j(t,\theta))})\min_j(\tilde\V_j^+)^3\tilde\Y_{k,\eta}(t,\theta) d\theta\Big]d\eta  \Big\vert \\
&\quad \leq \bigO(1) \norm{\tilde\U_1-\tilde\U_2}^2, \label{eq:L1}
\end{align}
where $\bigO(1)$ denotes some constant only depending on $\A$, which remains bounded as $\A\to 0$.

As far as the last term $L_3$ (a similar argument works for $L_2$) is concerned, the corresponding integral can be estimated as follows
\begin{align} \nn
&\vert \int_0^{1} L_3(\tilde\U_1-\tilde\U_2)(t,\eta)d\eta\vert \\ \nn
&= \frac1{\A^6}\vert \int_0^{1} (\tilde\U_1-\tilde \U_2) \tilde\U_{2,\eta}(t,\eta)\int_0^\eta\min_j(e^{-\frac{1}{\ma}(\tilde \Y_j(t,\eta)-\tilde\Y_j(t,\theta))})\\ \nn
&\qquad\qquad\qquad\qquad\qquad\qquad\qquad\times\min_j(\tilde\V_j^+)^3(\tilde\Y_{2,\eta}-\tilde\Y_{1,\eta})(t,\theta) d\theta \mathbbm{1}_{E^c}(t,\eta) d\eta\vert \\ \nn
& = \frac1{\A^6}\vert \int_0^{1} (\tilde\U_1-\tilde\U_2)\tilde\U_{2,\eta}(t,\eta) \mathbbm{1}_{E^c}(t,\eta)\\  \nn
& \qquad  \times \Big[ (\tilde\Y_2-\tilde\Y_1)(t,\theta)\min_j(e^{-\frac{1}{\ma}(\tilde \Y_j(t,\eta)-\tilde\Y_j(t,\theta))}) \min_j(\tilde\V_j^+)^3(t,\theta)\Big|_{\theta=0}^{\eta}\\ \nn
& \qquad   - \int_0^\eta (\tilde\Y_2-\tilde\Y_1)(t,\theta) 
\big[(\frac{d}{d\theta} \min_j(e^{-\frac{1}{\ma}(\tilde \Y_j(t,\eta)-\tilde\Y_j(t,\theta))}))\min_j(\tilde\V_j^+)^3(t,\theta)\\ \nn
&\qquad  \qquad\qquad  + \min_j(e^{-\frac{1}{\ma}(\tilde \Y_j(t,\eta)-\tilde\Y_j(t,\theta))})
(\frac{d}{d\theta} \min_j(\tilde\V_j^+)^3)(t,\theta) \big] d\theta\Big]\mathbbm{1}_{E^c}(t,\eta)d\eta\vert \\ \nn
&=\frac{1}{\A^6}\vert -\int_0^{1} (\tilde\U_1-\tilde\U_2)(\tilde \Y_1-\tilde \Y_2)\min_j(\tilde\V_j^+)^3\tilde\U_{2,\eta}(t,\eta) \mathbbm{1}_{E^c}(t,\eta) d\eta\\ \nn
& \quad +\int_0^{1} (\tilde\U_1-\tilde\U_2) \tilde\U_{2,\eta}(t,\eta) \mathbbm{1}_{E^c}(t,\eta)\\ \nn
& \qquad \quad \times\int_0^\eta (\tilde\Y_1-\tilde\Y_2)(t,\theta)
\Big[(\frac{d}{d\theta} \min_j(e^{-\frac{1}{\ma}(\tilde \Y_j(t,\eta)-\tilde\Y_j(t,\theta))}))\min_j(\tilde\V_j^+)^3(t,\theta)\\ \nn
&\qquad \qquad\qquad  + \min_j(e^{-\frac{1}{\ma}(\tilde \Y_j(t,\eta)-\tilde\Y_j(t,\theta))})(\frac{d}{d\theta} \min_j(\tilde\V_j^+)^3)(t,\theta)  \Big]d\theta d\eta\vert \\ \nn
&\leq \bigO(1) (\norm{\tilde \Y_1-\tilde \Y_2}^2+\norm{\tilde\U_1-\tilde \U_2}^2)
 +\frac{1}{\A^{12}} \int_0^1 \tilde\U_{2,\eta}^2 (t,\eta) \Big(\int_0^\eta \vert\tilde \Y_1-\tilde\Y_2\vert (t,\theta)\\ \nn
& \qquad \qquad  \times \big[\frac{1}{\ma}\min_j(e^{-\frac{1}{\ma}(\tilde \Y_j(t,\eta)-\tilde\Y_j(t,\theta))})\min_j(\tilde \V_j^+)^3\max_j(\tilde\Y_{j,\eta}) (t,\theta)\\ \nn
& \qquad \qquad \qquad \qquad + 2\A^4\min_j(e^{-\frac{1}{\ma}(\tilde \Y_j(t,\eta)-\tilde\Y_j(t,\theta))})\min_j(\tilde \V_j^+)(t,\theta) \big]d\theta\Big)^2d\eta\\ \nn
& \leq  \bigO(1) (\norm{\tilde \Y_1-\tilde \Y_2}^2+\norm{\tilde\U_1-\tilde \U_2}^2)
 +\frac{54}{\A} \int_0^1 \tilde \P_2\tilde \Y_{2,\eta}(t,\eta) d\eta\norm{\tilde \Y_1-\tilde \Y_2}^2\\ 
& \leq \bigO(1) (\norm{\tilde \Y_1-\tilde \Y_2}^2+\norm{\tilde\U_1-\tilde \U_2}^2), \label{eq:L3}
\end{align}
where we used \eqref{est:L3a_lemma}, \eqref{eq:all_PestimatesA}, and \eqref{est:L3c_lemma}.

Finally, we have a look at the integral, which contains $J_9$ (a similar argument works for $J_{10}$), where we can 
assume $\sqtC{1}\leq \sqtC{2}$.  Thus  
\begin{align}  \nn
& \vert \int_0^1 J_9(\tilde\U_1-\tilde\U_2)(t,\eta) d\eta\vert \\ \nn
& \quad = \left(\frac{1}{\sqtC{1}^6}-\frac{1}{\sqtC{2}^6}\right) \vert \int_0^1 (\tilde\U_1-\tilde\U_2) \tilde\U_{1,\eta}(t,\eta) \nn\\ \nn
&\qquad\qquad\qquad\qquad\qquad\times\Big(\int_0^\eta e^{-\frac{1}{\sqtC{1}}(\tilde\Y_1(t,\eta)-\tilde\Y_1(t,\theta))} (\tilde\V_1^+)^3\tilde \Y_{1,\eta}(t,\theta) d\theta\Big) d\eta \vert  \\  \nn
& \quad= \frac{\sqtC{2}^6-\sqtC{1}^6} {\sqtC{1}^6\sqtC{2}^6} \vert \int_0^1 (\tilde\U_1-\tilde\U_2)\tilde\U_{1,\eta}(t,\eta) 
\nn\\ \nn
&\qquad\qquad\qquad\qquad\qquad\times\Big(\int_0^\eta e^{-\frac{1}{\sqtC{1}}(\tilde\Y_1(t,\eta)-\tilde\Y_1(t,\theta))} (\tilde\V_1^+)^3\tilde\Y_{1,\eta}(t,\theta) d\theta\Big) d\eta \vert\\ \nn
&\quad \leq 3\frac{\sqtC{2}-\sqtC{1}}{\sqtC{1}^2\sqtC{2}} \vert \int_0^1(\tilde\U_1-\tilde\U_2) \tilde\U_{1,\eta}(t,\eta)\Big( \int_0^\eta e^{-\frac{1}{\sqtC{1}}(\tilde\Y_1(t,\eta)-\tilde\Y_1(t,\theta))} \tilde\V_1^+\tilde\Y_{1,\eta}(t,\theta) d\theta \Big) d\eta\vert\\ \nn
&\quad \leq \vert \sqtC{1}-\sqtC{2}\vert ^2 + \frac{9}{\sqtC{1}^4\sqtC{2}^2} \int_0^1 (\tilde\U_2-\tilde\U_1)^2\tilde\U_{1,\eta}^2 (t,\eta)\nn\\ \nn
&\qquad\qquad\qquad\qquad\qquad\times \Big( \int_0^\eta e^{-\frac{1}{\sqtC{1}}(\tilde\Y_1(t,\eta) -\tilde\Y_1(t,\theta))} \tilde\V_1^+\tilde\Y_{1,\eta} (t,\theta) d\theta\Big)^2 d\eta\\ \nn
&\quad \leq \vert \sqtC{1}-\sqtC{2}\vert ^2 + \frac{9}{\sqtC{1}\sqtC{2}^2} \int_0^1 (\tilde\U_1-\tilde\U_2)^2 \tilde\Y_{1,\eta} (t,\eta)\nn\\ \nn
&\qquad\qquad\qquad\qquad\qquad\times\Big( \int_0^\eta e^{-\frac{1}{\sqtC{1}}(\tilde\Y_1(t,\eta)-\tilde\Y_1(t,\theta))} \tilde\Y_{1,\eta} (t,\theta) d\theta\Big)\\ \nn
& \qquad \qquad \qquad \qquad \qquad  \times\Big( \int_0^\eta e^{-\frac{1}{\sqtC{1}}(\tilde\Y_1(t,\eta) -\tilde\Y_1(t,\theta))} (\tilde\V_1^+)^2\tilde\Y_{1,\eta} (t,\theta) d\theta\Big) d\eta\\ \nn
& \quad\leq \vert \sqtC{1}-\sqtC{2}\vert ^2+\frac{9}{\sqtC{2}^2}\int_0^1 (\tilde\U_1-\tilde\U_2)^2 \tilde\Y_{1,\eta}(t,\eta) \nn\\ \nn
&\qquad\qquad\qquad\qquad\qquad\times\Big(\int_0^\eta e^{-\frac{1}{\sqtC{1}}(\tilde\Y_1(t,\eta)-\tilde\Y_1(t,\theta))} \tilde\U_1^2\tilde\Y_{1,\eta} (t,\theta) d\theta\Big) d\eta\\ \nn
&\quad \leq \vert \sqtC{1}-\sqtC{2}\vert ^2 + \frac{36}{\A} \int_0^1 (\tilde\U_1-\tilde\U_2)^2 \tilde\P_1\tilde\Y_{1,\eta} (t,\eta) d\eta\\
&\quad \leq \vert \sqtC{1}-\sqtC{2}\vert ^2 + 18\A^4 \norm{\tilde\U_1-\tilde\U_2}^2.  \label{eq:J7}
\end{align}

We now turn to the integral $I_{32}$. Recall equation \eqref{eq:I32_for_U}, namely   
\begin{align*}
I_{32}&=\int_0^{1} (\tilde\U_1- \tilde\U_2)(t, \eta)\Big(\frac{1}{\sqtC{2}^6}\tilde\U_{2,\eta}(t, \eta)\int_0^{1}
 e^{-\frac{1}{\sqtC{2}}\vert \tilde\Y_2(t,\eta)-\tilde\Y_2(t,\theta)\vert }\tilde\P_2\tilde\U_2\tilde\Y_{2,\eta}(t,\theta)d\theta\\
&\qquad\qquad\qquad- \frac{1}{\sqtC{1}^6}\tilde\U_{1,\eta}(t, \eta)\int_0^{1} e^{-\frac{1}{\sqtC{1}}\vert \tilde\Y_1(t,\eta)-\tilde\Y_1(t,\theta)\vert } 
\tilde\P_1\tilde\U_1\tilde\Y_{1,\eta}(t,\theta)d\theta\Big)d\eta.
\end{align*}
By first using the decomposition $\tilde\U_j=\tilde\U_j^+ +\tilde\U_j^-$, and then involving \eqref{eq:triks}, we need to estimate terms of the type
\begin{align*}
\tilde I_{32}&=\int_0^{1} (\tilde\U_1- \tilde\U_2)(t, \eta)\Big(\frac{1}{\sqtC{2}^6}\tilde\U_{2,\eta}(t, \eta)\int_0^{\eta}
 e^{-\frac{1}{\sqtC{2}}( \tilde\Y_2(t,\eta)-\tilde\Y_2(t,\theta)) }\tilde\P_2\tilde\U_2^+\tilde\Y_{2,\eta}(t,\theta)d\theta\notag\\
&\qquad\qquad\qquad- \frac{1}{\sqtC{1}^6}\tilde\U_{1,\eta}(t, \eta)\int_0^{\eta} e^{-\frac{1}{\sqtC{1}}( \tilde\Y_1(t,\eta)-\tilde\Y_1(t,\theta)) } 
\tilde\P_1\tilde\U_1^+\tilde\Y_{1,\eta}(t,\theta)d\theta\Big)d\eta.    %\label{eq:tilde_I32_for_U} 
\end{align*}
We invoke Lemma \ref{lemma:LUR} and find  
\begin{align*}
\tilde I_{32}&=\frac1{\A^6}\int_0^{1} (\tilde\U_1- \tilde\U_2)(t, \eta)\Big(\tilde\U_{2,\eta}(t, \eta)\int_0^{\eta}
 e^{-\frac{1}{\sqtC{2}}( \tilde\Y_2(t,\eta)-\tilde\Y_2(t,\theta)) }\tilde\P_2\tilde\U_2^+\tilde\Y_{2,\eta}(t,\theta)d\theta\\
&\qquad\qquad\qquad- \tilde\U_{1,\eta}(t, \eta)\int_0^{\eta} e^{-\frac{1}{\sqtC{1}}( \tilde\Y_1(t,\eta)-\tilde\Y_1(t,\theta)) } 
\tilde\P_1\tilde\U_1^+\tilde\Y_{1,\eta}(t,\theta)d\theta\Big)d\eta \\
&\quad+\frac{\sqtC{1}^6-\sqtC{2}^6}{\sqtC{1}^6\sqtC{2}^6}\mathbbm{1}_{\sqtC{2}\le \sqtC{1}} \int_0^{1} (\tilde\U_1- \tilde\U_2)\tilde\U_{2,\eta}(t, \eta)\\
&\qquad\qquad\qquad\qquad\qquad\times\Big(\int_0^{\eta}
 e^{-\frac{1}{\sqtC{2}}( \tilde\Y_2(t,\eta)-\tilde\Y_2(t,\theta)) }\tilde\P_2\tilde\U_2^+\tilde\Y_{2,\eta}(t,\theta)d\theta\Big)d\eta\\
&\quad+\frac{\sqtC{1}^6-\sqtC{2}^6}{\sqtC{1}^6\sqtC{2}^6}\mathbbm{1}_{\sqtC{1}< \sqtC{2}} \int_0^{1} (\tilde\U_1- \tilde\U_2)\tilde\U_{1,\eta}(t, \eta)\\
&\qquad\qquad\qquad\qquad\qquad\times\Big(\int_0^{\eta}
 e^{-\frac{1}{\sqtC{1}}( \tilde\Y_1(t,\eta)-\tilde\Y_1(t,\theta)) }\tilde\P_1\tilde\U_1^+\tilde\Y_{1,\eta}(t,\theta)d\theta\Big)d\eta\\
 &= N_1+N_2+N_3.
\end{align*}
We consider first $N_1$ where we get
\begin{align}
N_1&=\frac1{\A^6}\int_0^{1} (\tilde\U_1- \tilde\U_2)(t, \eta)\Big[\tilde\U_{2,\eta}(t, \eta)\int_0^{\eta}
 e^{-\frac{1}{\sqtC{2}}( \tilde\Y_2(t,\eta)-\tilde\Y_2(t,\theta)) }\tilde\P_2\tilde\U_2^+\tilde\Y_{2,\eta}(t,\theta)d\theta 
 \notag\\
&\qquad\qquad\qquad- \tilde\U_{1,\eta}(t, \eta)\int_0^{\eta} e^{-\frac{1}{\sqtC{1}}( \tilde\Y_1(t,\eta)-\tilde\Y_1(t,\theta)) } 
\tilde\P_1\tilde\U_1^+\tilde\Y_{1,\eta}(t,\theta)d\theta\Big]d\eta \notag\\
&=\frac1{\A^6}\int_0^{1} (\tilde\U_1- \tilde\U_2)(t, \eta) \nn \\
&\qquad\times\Big[\tilde\U_{2,\eta}(t, \eta)\int_0^{\eta}
 e^{-\frac{1}{\sqtC{2}}( \tilde\Y_2(t,\eta)-\tilde\Y_2(t,\theta)) }(\tilde\P_2-\tilde\P_1)\tilde\U_2^+\tilde\Y_{2,\eta}
 \mathbbm{1}_{\tilde\P_1\le \tilde\P_2}(t,\theta)d\theta\notag\\
&\qquad\qquad+\tilde\U_{1,\eta}(t, \eta)\int_0^{\eta} e^{-\frac{1}{\sqtC{1}}( \tilde\Y_1(t,\eta)-\tilde\Y_1(t,\theta)) } 
(\tilde\P_2-\tilde\P_1)\tilde\U_1^+\tilde\Y_{1,\eta}\mathbbm{1}_{\tilde\P_2<\tilde\P_1}(t,\theta)d\theta \notag\\
&\qquad +\tilde\U_{2,\eta}(t, \eta)\int_0^{\eta}
 e^{-\frac{1}{\sqtC{2}}( \tilde\Y_2(t,\eta)-\tilde\Y_2(t,\theta)) }\min_j(\tilde\P_j)(\tilde\U_2^+-\tilde\U_1^+)\tilde\Y_{2,\eta}
 \mathbbm{1}_{\tilde\U_1^+\le\tilde\U_2^+}(t,\theta)d\theta\notag\\
&\qquad +\tilde\U_{1,\eta}(t, \eta)\int_0^{\eta}
 e^{-\frac{1}{\sqtC{1}}( \tilde\Y_1(t,\eta)-\tilde\Y_1(t,\theta)) }\min_j(\tilde\P_j)(\tilde\U_2^+-\tilde\U_1^+)\tilde\Y_{1,\eta}
 \mathbbm{1}_{\tilde\U_2^+<\tilde\U_1^+}(t,\theta)d\theta\notag\\
&\qquad+\tilde\U_{2,\eta}(t, \eta)\int_0^{\eta}
 e^{-\frac{1}{\sqtC{2}}( \tilde\Y_2(t,\eta)-\tilde\Y_2(t,\theta)) }\min_j(\tilde\P_j)\min_j(\tilde\U_j^+)\tilde\Y_{2,\eta}(t,\theta)d\theta\notag\\
&\qquad-\tilde\U_{1,\eta}(t, \eta)\int_0^{\eta}
 e^{-\frac{1}{\sqtC{1}}( \tilde\Y_1(t,\eta)-\tilde\Y_1(t,\theta)) }\min_j(\tilde\P_j)\min_j(\tilde\U_j^+)\tilde\Y_{1,\eta}(t,\theta)d\theta
\Big]d\eta \notag\\
&= N_{11}+N_{12}+N_{13}+N_{14}+N_{15}+N_{16}. \label{eq:N1}
\end{align}
The terms $N_{11}$ and $N_{12}$ can be treated similarly. To that end we find
\begin{align*}
\abs{N_{11}}&= \frac1{\A^6}\vert\int_0^{1} (\tilde\U_1- \tilde\U_2)(t, \eta)\\
&\quad\times\Big[\tilde\U_{2,\eta}(t, \eta)\int_0^{\eta}
 e^{-\frac{1}{\sqtC{2}}( \tilde\Y_2(t,\eta)-\tilde\Y_2(t,\theta)) }(\tilde\P_2-\tilde\P_1)\tilde\U_2^+\tilde\Y_{2,\eta}
 \mathbbm{1}_{\tilde\P_1\le\tilde\P_2}(t,\theta)d\theta\Big]d\eta \vert \\
 &\le \norm{\tilde\U_1- \tilde\U_2}^2+\frac1{\A^{12}} \int_0^{1}\tilde\U_{2,\eta}^2(t, \eta)\\
 & \qquad\qquad\times\Big(\int_0^{\eta}
 e^{-\frac{1}{\sqtC{2}}( \tilde\Y_2(t,\eta)-\tilde\Y_2(t,\theta)) }(\tilde\P_2-\tilde\P_1)\tilde\U_2^+\tilde\Y_{2,\eta}
 \mathbbm{1}_{\tilde\P_1\le\tilde\P_2}(t,\theta)d\theta \Big)^2d\eta \\
 &\le \norm{\tilde\U_1- \tilde\U_2}^2 \\
 &\quad\qquad+\frac4{\A^9} \int_0^{1}\tilde\Y_{2,\eta}(t, \eta)\Big(\int_0^{\eta}e^{-\frac{1}{\sqtC{2}}( \tilde\Y_2(t,\eta)-\tilde\Y_2(t,\theta))}
 (\tilde\U_2^+)^2\tilde\Y_{2,\eta}(t,\theta)d\theta \Big)\\
 &\qquad\qquad\times \Big(\int_0^{\eta}e^{-\frac{1}{\sqtC{2}}( \tilde\Y_2(t,\eta)-\tilde\Y_2(t,\theta))}\vert\sqP{2}-\sqP{1}\vert^2 \tilde\P_2\tilde\Y_{2,\eta}(t,\theta)d\theta \Big)d\eta \\
&\le \norm{\tilde\U_1- \tilde\U_2}^2 +\frac{16}{\A^8} \int_0^{1}\tilde\P_2\tilde\Y_{2,\eta}(t, \eta) \\
 &\qquad\qquad\qquad\times
 \Big(\int_0^{\eta}e^{-\frac{1}{\sqtC{2}}( \tilde\Y_2(t,\eta)-\tilde\Y_2(t,\theta))}\vert\sqP{2}-\sqP{1}\vert^2 \tilde\P_2\tilde\Y_{2,\eta}(t,\theta)d\theta \Big)d\eta \\
 &\le \bigO(1)\big(\norm{\tilde\U_1- \tilde\U_2}^2 +\norm{\sqP{1}- \sqP{2}}^2\big),
\end{align*}
where we have used \eqref{eq:all_estimatesE}, \eqref{eq:all_estimatesH}, and \eqref{eq:all_PestimatesC}.

The terms $N_{13}$ and $N_{14}$ follow the same estimates. More precisely,
\begin{align*}
\abs{N_{13}}&=\frac1{\A^6} \vert\int_0^{1} (\tilde\U_1- \tilde\U_2)
\tilde\U_{2,\eta}(t, \eta) \\
&\qquad \times\int_0^{\eta}
 e^{-\frac{1}{\sqtC{2}}( \tilde\Y_2(t,\eta)-\tilde\Y_2(t,\theta)) }\min_j(\tilde\P_j)(\tilde\U_2^+-\tilde\U_1^+)\tilde\Y_{2,\eta}
 \mathbbm{1}_{\tilde\U_1^+\le\tilde\U_2^+}(t,\theta)d\theta d\eta\vert\\
 &\le\frac1{\A^6} \int_0^{1}\vert\tilde\U_1- \tilde\U_2\vert\, 
\vert\tilde\U_{2,\eta}\vert(t, \eta) \\
&\qquad \times\int_0^{\eta}
 e^{-\frac{1}{\sqtC{2}}( \tilde\Y_2(t,\eta)-\tilde\Y_2(t,\theta)) }\min_j(\tilde\P_j)\vert\tilde\U_2^+-\tilde\U_1^+\vert\tilde\Y_{2,\eta}
 \mathbbm{1}_{\tilde\U_1^+\le\tilde\U_2^+}(t,\theta)d\theta d\eta\\
 &\le\norm{\tilde\U_1- \tilde\U_2}^2 +\frac1{\A^{12}} \int_0^{1}
\tilde\U_{2,\eta}^2(t, \eta)\\
 &\qquad\times \Big(\int_0^{\eta}
 e^{-\frac{1}{\sqtC{2}}( \tilde\Y_2(t,\eta)-\tilde\Y_2(t,\theta)) }\min_j(\tilde\P_j)\vert\tilde\U_2^+-\tilde\U_1^+\vert\tilde\Y_{2,\eta}
 \mathbbm{1}_{\tilde\U_1^+\le\tilde\U_2^+}(t,\theta)d\theta\Big)^2 d\eta\\ 
&\le\norm{\tilde\U_1- \tilde\U_2}^2 \\
 &\quad+ \frac1{\A^{12}} \int_0^{1}
\tilde\U_{2,\eta}^2(t, \eta) \big(\int_0^{\eta} e^{-\frac3{2\sqtC{2}}( \tilde\Y_2(t,\eta)-\tilde\Y_2(t,\theta)) }\tilde\P_2\tilde\Y_{2,\eta}
(t,\theta)d\theta\big) \\
 &\qquad\times \big(\int_0^{\eta} e^{-\frac1{2\sqtC{2}}( \tilde\Y_2(t,\eta)-\tilde\Y_2(t,\theta)) }(\tilde\U_2^+-\tilde\U_1^+)^2\tilde\P_2\tilde\Y_{2,\eta}(t,\theta)d\theta\big) d\eta\\   
 &\le\norm{\tilde\U_1- \tilde\U_2}^2 
 + \frac2{\A^{11}} \int_0^{1}\tilde\P_2\tilde\U_{2,\eta}^2(t, \eta)  \\
 &\qquad \times \big(\int_0^{\eta} e^{-\frac1{2\sqtC{2}}( \tilde\Y_2(t,\eta)-\tilde\Y_2(t,\theta)) }(\tilde\U_2-\tilde\U_1)^2\tilde\P_2\tilde\Y_{2,\eta}(t,\theta)d\theta\big) d\eta\\   
&\le \bigO(1) \norm{\tilde\U_2-\tilde\U_1}^2,
\end{align*}
by applying \eqref{eq:343}, \eqref{eq:all_estimatesE}, and \eqref{eq:all_estimatesQ}.

The term  $N_{15}+N_{16}$ can be estimated as follows:
\begin{align*}
N_{15}+N_{16}&=\frac1{\A^6}\int_0^{1} (\tilde\U_1- \tilde\U_2)(t, \eta)\\
&\quad\times\Big[\tilde\U_{2,\eta}(t, \eta)\int_0^{\eta}
 e^{-\frac{1}{\sqtC{2}}( \tilde\Y_2(t,\eta)-\tilde\Y_2(t,\theta)) }\min_j(\tilde\P_j)\min_j(\tilde\U_j^+)\tilde\Y_{2,\eta}(t,\theta)d\theta\\
&\qquad-\tilde\U_{1,\eta}(t, \eta)\int_0^{\eta}
 e^{-\frac{1}{\sqtC{1}}( \tilde\Y_1(t,\eta)-\tilde\Y_1(t,\theta)) }\min_j(\tilde\P_j)\min_j(\tilde\U_j^+)\tilde\Y_{1,\eta}(t,\theta)d\theta
\Big]d\eta \\
& = \frac{1}{\A^6} \mathbbm{1}_{\sqtC{1}\leq\sqtC{2}}\int_0^1 (\tilde \U_1-\tilde \U_2)\tilde \U_{2,\eta}(t,\eta)\\
& \qquad \qquad\qquad  \times\Big(\int_0^\eta (e^{-\frac{1}{\sqtC{2}}(\tilde \Y_2(t,\eta)-\tilde \Y_2(t,\theta))}-e^{-\frac{1}{\sqtC{1}}(\tilde \Y_2(t,\eta)-\tilde \Y_2(t,\theta))})\\
&\qquad\qquad\qquad\qquad \qquad\qquad \qquad\times\min_j(\tilde \P_j)\min_j(\tilde \U_j^+)\tilde \Y_{2,\eta}(t,\theta) d\theta\Big) d\eta\\
& \quad + \frac{1}{\A^6} \mathbbm{1}_{\sqtC{2}<\sqtC{1}}\int_0^1 (\tilde \U_1-\tilde \U_2)\tilde \U_{1,\eta}(t,\eta)\\
& \qquad \qquad \times\Big(\int_0^\eta (e^{-\frac{1}{\sqtC{2}}(\tilde \Y_1(t,\eta)-\tilde \Y_1(t,\theta))}-e^{-\frac{1}{\sqtC{1}}(\tilde \Y_1(t,\eta)-\tilde \Y_1(t,\theta))})\\
&\qquad\qquad\qquad\qquad\qquad\qquad\qquad\times\min_j(\tilde \P_j)\min_j(\tilde \U_j^+)\tilde \Y_{1,\eta}(t,\theta) d\theta\Big) d\eta\\
&\quad +\frac1{\A^6}\int_0^{1} (\tilde\U_1- \tilde\U_2)\tilde\U_{2,\eta}(t, \eta) \\
&\qquad\times\int_0^{\eta}
\big( e^{-\frac{1}{\ma}( \tilde\Y_2(t,\eta)-\tilde\Y_2(t,\theta))}-e^{-\frac{1}{\ma}( \tilde\Y_1(t,\eta)-\tilde\Y_1(t,\theta))}\big) )\\
&\qquad\qquad\qquad\qquad\qquad\times\min_j(\tilde\P_j)\min_j(\tilde\U_j^+)\tilde\Y_{2,\eta}(t,\theta) \mathbbm{1}_{B(\eta)}(t,\theta) d\theta\big)d\eta\\
&\quad+\frac1{\A^6}\int_0^{1} (\tilde\U_1- \tilde\U_2)\tilde\U_{1,\eta}(t, \eta)\\
&\qquad\times\int_0^{\eta}
\big( e^{-\frac{1}{\ma}( \tilde\Y_2(t,\eta)-\tilde\Y_2(t,\theta))}-e^{-\frac{1}{\ma}( \tilde\Y_1(t,\eta)-\tilde\Y_1(t,\theta))}\big) \\
&\qquad\qquad\qquad\qquad\qquad\times\min_j(\tilde\P_j)\min_j(\tilde\U_j^+)\tilde\Y_{1,\eta}(t,\theta) \mathbbm{1}_{B^c(\eta)} (t,\theta)d\theta\big )d\eta\\
&\quad +\frac1{\A^6}\int_0^{1} (\tilde\U_1- \tilde\U_2)(t, \eta)\Big[\tilde\U_{2,\eta}(t, \eta)\int_0^{\eta}
 \min_j(e^{-\frac{1}{\ma}( \tilde\Y_j(t,\eta)-\tilde\Y_j(t,\theta))})\\
&\qquad\qquad\qquad\qquad\qquad\times\min_j(\tilde\P_j)\min_j(\tilde\U_j^+)\tilde\Y_{2,\eta}(t,\theta)d\theta\\
&\qquad\qquad-\tilde\U_{1,\eta}(t, \eta)\int_0^{\eta} \min_j(e^{-\frac{1}{\ma}( \tilde\Y_j(t,\eta)-\tilde\Y_j(t,\theta))})\\
&\qquad\qquad\qquad\qquad\qquad\times
 \min_j(\tilde\P_j)\min_j(\tilde\U_j^+)\tilde\Y_{1,\eta}(t,\theta)d\theta
\Big]d\eta \\
&=N_{151}+N_{152}+N_{153}+N_{154}+N_{155},
\end{align*} 
which, unfortunately, each need a special treatment.

The terms $N_{151}$ and $N_{152}$ can be handled as follows:
\begin{align*}
\vert N_{151}\vert &= \frac{1}{\A^6} \mathbbm{1}_{\sqtC{1}\leq \sqtC{2}} \vert \int_0^1 (\tilde \U_1-\tilde \U_2)\tilde \U_{2,\eta}(t,\eta)\\
& \qquad \quad \times \Big(\int_0^\eta (e^{-\frac{1}{\sqtC{2}}(\tilde \Y_2(t,\eta)-\tilde \Y_2(t,\theta))}-e^{-\frac{1}{\sqtC{1}}(\tilde \Y_2(t,\eta)-\tilde \Y_2(t,\theta))})\\
&\qquad\qquad\qquad\qquad\qquad\qquad\qquad\times\min_j(\tilde \P_j)\min_j(\tilde \V_j^+)\tilde \Y_{2,\eta}(t,\theta) d\theta\Big) d\eta \vert\\
& \leq \frac{4}{\A^6\sqtC{1}e}\int_0^1 \vert \tilde \U_1-\tilde \U_2\vert \vert \tilde \U_{2,\eta}(t,\eta)\\
& \quad  \times \Big(\int_0^\eta e^{-\frac{3}{4\sqtC{2}}(\tilde \Y_2(t,\eta)-\tilde \Y_2(t,\theta))}\min_j(\tilde \P_j)\min_j(\tilde \U_j^+)\tilde \Y_{2,\eta}(t,\theta) d\theta\Big) d\eta \vert \sqtC{1}-\sqtC{2}\vert\\
& \leq \norm{\tilde \U_1-\tilde \U_2}^2 + \frac{16}{\A^{12} \sqtC{1}^2 e^2}\int_0^1 \tilde \U_{2,\eta}^2(t,\eta)\\
& \qquad \qquad \times \Big(\int_0^\eta e^{-\frac{3}{4\sqtC{2}}(\tilde \Y_2(t,\eta)-\tilde \Y_2(t,\theta))}\vert \tilde \U_1\vert \tilde \P_2\tilde \Y_{2,\eta}(t,\theta)d\theta\Big)^2 d\eta \vert \sqtC{1}-\sqtC{2}\vert ^2\\
& \leq \norm{\tilde \U_1-\tilde \U_2}^2  + \frac{8}{\A^4e^2}\int_0^1 \tilde \P_2\tilde \U_{2,\eta}^2(t,\eta) d\eta \vert \sqtC{1}-\sqtC{2}\vert^2\\
& \leq \bigO(1)(\norm{\tilde \U_1-\tilde \U_2}^2+\vert \sqtC{1}-\sqtC{2}\vert^2),
\end{align*}
where we used \eqref{eq:343}.

The terms $N_{153}$ and $N_{154}$ can be handled as follow:
\begin{align*}
\abs{N_{153}}&=\frac1{\A^6}\vert \int_0^{1} (\tilde\U_1- \tilde\U_2)\tilde\U_{2,\eta}(t, \eta) \\
&\quad\times\int_0^{\eta}
\big( e^{-\frac{1}{\ma}( \tilde\Y_2(t,\eta)-\tilde\Y_2(t,\theta))}-e^{-\frac{1}{\ma}( \tilde\Y_1(t,\eta)-\tilde\Y_1(t,\theta))}\big)
 \\
&\qquad\qquad\qquad\qquad\qquad\qquad\times\min_j(\tilde\P_j)\min_j(\tilde\U_j^+)\tilde\Y_{2,\eta}(t,\theta) \mathbbm{1}_{B(\eta)}(t,\theta) d\theta d\eta\vert \\
&\leq\frac1{\A^6\ma}\int_0^{1} \vert\tilde\U_1- \tilde\U_2\vert\vert\tilde\U_{2,\eta}\vert(t, \eta) \\
&\quad\times\int_0^{\eta}
\big( \vert \tilde\Y_2(t,\eta)-\tilde\Y_1(t,\eta) \vert+ \vert \tilde\Y_2(t,\theta)-\tilde\Y_1(t,\theta) \vert\big) e^{-\frac{1}{\ma}( \tilde\Y_2(t,\eta)-\tilde\Y_2(t,\theta))}\\
&\qquad\qquad\times\min_j(\tilde\P_j)\min_j(\tilde\U_j^+)\tilde\Y_{2,\eta}(t,\theta)d\theta  d\eta \\
&=\frac1{\A^6\ma}\int_0^{1} \vert \tilde\U_1- \tilde\U_2\vert \,  \vert \tilde\Y_2-\tilde\Y_1 \vert\vert\tilde\U_{2,\eta}\vert(t, \eta)\\
&\qquad\qquad\qquad\times \int_0^{\eta}e^{-\frac{1}{\ma}( \tilde\Y_2(t,\eta)-\tilde\Y_2(t,\theta))}\min_j(\tilde\P_j)\min_j(\tilde\U_j^+)\tilde\Y_{2,\eta}(t,\theta) d\theta  d\eta \\
&\quad +\frac1{\A^6\ma}\int_0^{1} \vert \tilde\U_1- \tilde\U_2\vert \vert\tilde\U_{2,\eta}\vert(t, \eta)\\
&\qquad\qquad\times \int_0^{\eta}
 e^{-\frac{1}{\ma}( \tilde\Y_2(t,\eta)-\tilde\Y_2(t,\theta))} \vert \tilde\Y_2-\tilde\Y_1 \vert\min_j(\tilde\P_j)\min_j(\tilde\U_j^+)\tilde\Y_{2,\eta}(t,\theta)d\theta  d\eta \\
&\le \frac1{\A^6\ma}\int_0^{1} \vert \tilde\U_1- \tilde\U_2\vert \,  \vert \tilde\Y_2-\tilde\Y_1 \vert\vert\tilde\U_{2,\eta}\vert(t, \eta) \\
&\qquad\qquad\times\Big(\int_0^{\eta}e^{-\frac{1}{\sqtC{2}}( \tilde\Y_2(t,\eta)-\tilde\Y_2(t,\theta))}\tilde\P_2^2\tilde\Y_{2,\eta}(t,\theta)d\theta\Big)^{1/2} \\
&\qquad\qquad\times\Big(\int_0^{\eta}e^{-\frac{1}{\ma}( \tilde\Y_2(t,\eta)-\tilde\Y_2(t,\theta))}\min_j(\tilde\U_j^+)^2\tilde\Y_{2,\eta}(t,\theta) d\theta\Big)^{1/2}  d\eta \\
&\quad+\norm{\tilde\U_1- \tilde\U_2}^2+\frac1{\A^{12}\ma^2} \int_0^{1}\tilde\U_{2,\eta}^2(t, \eta)\\ 
& \qquad\qquad\times 
\Big(\int_0^{\eta}e^{-\frac{1}{\sqtC{2}}( \tilde\Y_2(t,\eta)-\tilde\Y_2(t,\theta))}\vert \tilde\Y_2-\tilde\Y_1 \vert\tilde\P_2\min_j(\tilde\U_j^+)\tilde\Y_{2,\eta}(t,\theta) d\theta\Big)^{2}  d\eta \\
&\le \frac{\sqrt{3}}{2\A^2}\int_0^{1} \vert \tilde\U_1- \tilde\U_2\vert \,  \vert \tilde\Y_2-\tilde\Y_1 \vert\sqP{2}\vert\tilde\U_{2,\eta}\vert(t, \eta)   d\eta +\norm{\tilde\U_1- \tilde\U_2}^2\\
&\quad+\frac{9}{4\A^4} \int_0^1 \tilde \P_2\tilde \U_{2,\eta}^2(t,\eta) d\eta \norm{\tilde \Y_1-\tilde \Y_2}^2 \\
&\le \bigO(1)\big(\norm{\tilde\U_1- \tilde\U_2}^2+\norm{\tilde\Y_1- \tilde\Y_2}^2 \big),
\end{align*} 
where we used \eqref{eq:general}.
We consider now $N_{155}$:
\begin{align*}
N_{155}&=\frac1{\A^6}\int_0^{1} (\tilde\U_1- \tilde\U_2)(t, \eta)\\
&\quad\times\Big[\tilde\U_{2,\eta}(t, \eta)\int_0^{\eta}
 \min_j(e^{-\frac{1}{\ma}( \tilde\Y_j(t,\eta)-\tilde\Y_j(t,\theta))})\min_j(\tilde\P_j)\min_j(\tilde\U_j^+)\tilde\Y_{2,\eta}(t,\theta)d\theta\\
&\qquad-\tilde\U_{1,\eta}(t, \eta)\int_0^{\eta}
 \min_j(e^{-\frac{1}{\ma}( \tilde\Y_j(t,\eta)-\tilde\Y_j(t,\theta))})\min_j(\tilde\P_j)\min_j(\tilde\U_j^+)\tilde\Y_{1,\eta}(t,\theta)d\theta
\Big]d\eta \\
&= \frac1{\A^6}\int_0^{1} (\tilde\U_1- \tilde\U_2) (\tilde\U_2- \tilde\U_1)_\eta (t, \eta)\\
&\qquad\qquad\times\min_k \Big[ \int_0^{\eta}
 \min_j(e^{-\frac{1}{\ma}( \tilde\Y_j(t,\eta)-\tilde\Y_j(t,\theta))})\min_j(\tilde\P_j)\min_j(\tilde\U_j^+)\tilde\Y_{k,\eta}(t,\theta)d\theta\Big]d\eta \\
&\quad  +\frac1{\A^6}\int_0^{1} (\tilde\U_1- \tilde\U_2)\tilde\U_{2,\eta}(t, \eta)\mathbbm{1}_{\tilde E^c}(t,\eta)\\
&\qquad\qquad\times
\big( \int_0^{\eta}\min_j(e^{-\frac{1}{\ma}( \tilde\Y_j(t,\eta)-\tilde\Y_j(t,\theta))})\min_j(\tilde\P_j)\min_j(\tilde\U_j^+)(\tilde\Y_{2}- \tilde\Y_{1})_\eta d\theta\big)d\eta \\
&\quad+  \frac1{\A^6}\int_0^{1} (\tilde\U_1- \tilde\U_2)\tilde\U_{1,\eta}(t, \eta)\mathbbm{1}_{\tilde E}(t,\eta)\\
&\qquad\qquad\times
\big( \int_0^{\eta}\min_j(e^{-\frac{1}{\ma}( \tilde\Y_j(t,\eta)-\tilde\Y_j(t,\theta))})\min_j(\tilde\P_j)\min_j(\tilde\U_j^+)(\tilde\Y_{2}- \tilde\Y_{1})_\eta d\theta\big)d\eta \\   
&= N_{1551}+N_{1552}+N_{1553},
\end{align*}
which, yet again, requires separate treatment.  Here $\tilde E$ is defined in \eqref{eq:tildeE}.

The term $N_{1551}$ can be handled as the term $\tilde L_{31}$, cf.~\eqref{eq:tildeL31},
\begin{align*}
\vert N_{1551}\vert&\le \bigO(1)\norm{\tilde\U_1-\tilde\U_2}^2.
\end{align*}

The terms $N_{1552}$ and $N_{1553}$ can be treated in the same manner:
\begin{align*}
\abs{N_{1552}}&\le\frac1{\A^6}\vert \int_0^{1} (\tilde\U_1- \tilde\U_2)\tilde\U_{2,\eta}(t, \eta)\mathbbm{1}_{\tilde E^c}(t,\eta)\\
&\qquad\times
\big( \int_0^{\eta}\min_j(e^{-\frac{1}{\ma}( \tilde\Y_j(t,\eta)-\tilde\Y_j(t,\theta))})\min_j(\tilde\P_j)\min_j(\tilde\U_j^+)(\tilde\Y_{2}- \tilde\Y_{1})_\eta d\theta\big)d\eta\vert \\
&\leq \frac1{\A^6}\int_0^{1} \vert\tilde\U_1- \tilde\U_2\vert\, \vert \tilde\Y_1- \tilde\Y_2 \vert\min_j(\tilde\P_j)\min_j(\tilde\U_j^+)\vert\tilde\U_{2,\eta}\vert\mathbbm{1}_{\tilde E^c}(t,\eta)d\eta\\
&\quad +  \frac1{\A^6}\vert\int_0^{1}(\tilde\U_1- \tilde\U_2)\tilde\U_{2,\eta}(t, \eta)\mathbbm{1}_{\tilde E^c}(t,\eta)\\
&\quad\times \int_0^{\eta}(\tilde\Y_1- \tilde\Y_2)\Big[\frac{d}{d\theta}\big(\min_j(e^{-\frac{1}{\ma}( \tilde\Y_j(t,\eta)-\tilde\Y_j(t,\theta))})\big)\min_j(\tilde\P_j)\min_j(\tilde\U_j^+) \\
&\qquad\qquad\qquad + \min_j(e^{-\frac{1}{\ma}( \tilde\Y_j(t,\eta)-\tilde\Y_j(t,\theta))})\frac{d}{d\theta}\big(\min_j(\tilde\P_j)\min_j(\tilde\U_j^+) \big)\Big](t,\eta)d\eta\vert \\
&\le \bigO(1) \big(\norm{\tilde\U_1- \tilde\U_2}^2+ \norm{\tilde\Y_1- \tilde\Y_2}^2 \big). 
\end{align*}

The terms $N_2$ (and also $N_3$) can be treated as follows, keeping in mind that $\sqtC{2}\leq \sqtC{1}$.
\begin{align*}
\abs{N_2}&\leq \frac{\sqtC{1}^6-\sqtC{2}^6}{\sqtC{1}^6\sqtC{2}^6}\int_0^1\vert \tilde\U_1-\tilde\U_2\vert \vert \tilde\U_{2,\eta}\vert(t,\eta)\big(\int_0^\eta e^{-\frac{1}{\sqtC{2}}(\tilde\Y_2(t,\eta)-\tilde\Y_2(t,\theta))}\tilde \P_2^2\tilde\Y_{2,\eta}(t,\theta)d\theta\big)^{1/2}\\
& \qquad \qquad \times\big(\int_0^\eta e^{-\frac{1}{\sqtC{2}}(\tilde \Y_2(t,\eta)-\tilde \Y_2(t,\theta))}\tilde \U_2^2\tilde\Y_{2,\eta}(t,\theta) d\theta\big)^{1/2}d\eta\\
& \leq \sqrt{6}\frac{\sqtC{1}^6-\sqtC{2}^6}{\A^6\sqtC{2}^3} \int_0^1 \vert \tilde\U_1-\tilde\U_2\vert \tilde\P_2\vert \tilde\U_{2,\eta}\vert (t,\eta) d\eta\\
& \leq 3\sqrt{3}\A^2 \vert \sqtC{1}-\sqtC{2}\vert \int_0^1 \vert \tilde\U_2-\tilde\U_1\vert (t,\eta)d\eta\\
& \leq \bigO(1) (\norm{\tilde\U_1-\tilde\U_2}^2+\vert \sqtC{1}-\sqtC{2}\vert ^2).
\end{align*}

%----------------------
%

%
%---------------------------
\bigskip 
Finally, we now turn to the integral $I_{33}$. Recall equation \eqref{eq:I33_for_U}, namely
\begin{align*}
I_{33}&=\int_0^{1} (\tilde\U_1-\tilde\U_2)(t, \eta)\Big(\frac{1}{\sqtC{2}^6}\tilde\U_{2,\eta}(t, \eta)\int_0^{1}
 e^{-\frac{1}{\sqtC{2}}\vert \tilde\Y_2(t,\eta)-\tilde\Y_2(t,\theta)\vert }\tilde\Q_2\tilde\U_{2,\eta}(t,\theta)d\theta\\
&\qquad\qquad\qquad- \frac{1}{\sqtC{1}^6}\tilde\U_{1,\eta}(t, \eta)\int_0^{1} e^{-\frac{1}{\sqtC{1}}\vert \tilde\Y_1(t,\eta)-\tilde\Y_1(t,\theta)\vert } 
\tilde\Q_1\tilde\U_{1,\eta}(t,\theta)d\theta\Big)d\eta. 
\end{align*}
By first involving \eqref{eq:triks} and then Lemma \ref{lemma:LUR}, we see that it suffices to estimate
\begin{align*}
\tilde I_{33}&=\int_0^{1} (\tilde\U_1-\tilde\U_2)(t, \eta)\Big(\frac{1}{\sqtC{2}^6}\tilde\U_{2,\eta}(t, \eta)\int_0^{\eta}
 e^{-\frac{1}{\sqtC{2}}( \tilde\Y_2(t,\eta)-\tilde\Y_2(t,\theta)) }\tilde\Q_2\tilde\U_{2,\eta}(t,\theta)d\theta\\
&\qquad\qquad\qquad- \frac{1}{\sqtC{1}^6}\tilde\U_{1,\eta}(t, \eta)\int_0^{\eta} e^{-\frac{1}{\sqtC{1}}( \tilde\Y_1(t,\eta)-\tilde\Y_1(t,\theta)) } 
\tilde\Q_1\tilde\U_{1,\eta}(t,\theta)d\theta\Big)d\eta\\
&= \frac{1}{\A^6}\int_0^{1} (\tilde\U_1-\tilde\U_2)(t, \eta)\Big(\tilde\U_{2,\eta}(t, \eta)\int_0^{\eta}
 e^{-\frac{1}{\sqtC{2}}( \tilde\Y_2(t,\eta)-\tilde\Y_2(t,\theta)) }\tilde\Q_2\tilde\U_{2,\eta}(t,\theta)d\theta \\
& \qquad\qquad\qquad -\tilde\U_{1,\eta}(t, \eta)\int_0^{\eta} e^{-\frac{1}{\sqtC{1}}( \tilde\Y_1(t,\eta)-\tilde\Y_1(t,\theta)) } 
\tilde\Q_1\tilde\U_{1,\eta}(t,\theta)d\theta \Big)d\eta \\
& \quad+\mathbbm{1}_{\sqtC{2}\le \sqtC{1}}\big(\frac{1}{\sqtC{2}^6}-\frac{1}{\sqtC{1}^6} \big)\int_0^{1}(\tilde\U_1-\tilde\U_2)\tilde\U_{2,\eta}(t, \eta) \\
&\qquad\qquad\qquad\qquad\qquad\qquad\times
\int_0^{\eta}e^{-\frac{1}{\sqtC{2}}( \tilde\Y_2(t,\eta)-\tilde\Y_2(t,\theta)) }\tilde\Q_2\tilde\U_{2,\eta}(t,\theta)d\theta d\eta \\
& \quad+\mathbbm{1}_{\sqtC{1}<\sqtC{2}}
\big(\frac{1}{\sqtC{2}^6}-\frac{1}{\sqtC{1}^6} \big)\int_0^{1}(\tilde\U_1-\tilde\U_2)\tilde\U_{1,\eta}(t, \eta)\\
&\qquad\qquad\qquad\qquad\qquad\qquad\times
\int_0^{\eta}e^{-\frac{1}{\sqtC{1}}( \tilde\Y_1(t,\eta)-\tilde\Y_1(t,\theta)) }\tilde\Q_1\tilde\U_{1,\eta}(t,\theta)d\theta d\eta \\
&= M_1+M_2+M_3.
\end{align*}

To our dismay, the estimate for $M_1$ is rather involved. We estimate, using first that $\tilde\Q_i=\sqtC{i}\tilde\P_i-\tilde\D_i$ (cf.~\eqref{eq:DogP}):
\begin{align*}
M_1&=\frac{1}{\A^6}\int_0^{1} (\tilde\U_1-\tilde\U_2)(t, \eta)\Big(\tilde\U_{2,\eta}(t, \eta)\int_0^{\eta}
 e^{-\frac{1}{\sqtC{2}}( \tilde\Y_2(t,\eta)-\tilde\Y_2(t,\theta)) }\tilde\Q_2\tilde\U_{2,\eta}(t,\theta)d\theta \\
& \qquad -\tilde\U_{1,\eta}(t, \eta)\int_0^{\eta} e^{-\frac{1}{\sqtC{1}}( \tilde\Y_1(t,\eta)-\tilde\Y_1(t,\theta)) } 
\tilde\Q_1\tilde\U_{1,\eta}(t,\theta)d\theta \Big)d\eta \\
&= \frac{1}{\A^6}\int_0^{1} (\tilde\U_1-\tilde\U_2)(t, \eta)\Big(\tilde\U_{2,\eta}(t, \eta)\int_0^{\eta}
 e^{-\frac{1}{\sqtC{2}}( \tilde\Y_2(t,\eta)-\tilde\Y_2(t,\theta)) }\sqtC{2}\tilde\P_2\tilde\U_{2,\eta}(t,\theta)d\theta \\
& \qquad -\tilde\U_{1,\eta}(t, \eta)\int_0^{\eta} e^{-\frac{1}{\sqtC{1}}( \tilde\Y_1(t,\eta)-\tilde\Y_1(t,\theta)) } 
\sqtC{1}\tilde\P_1\tilde\U_{1,\eta}(t,\theta)d\theta \Big)d\eta \\
&\quad-\frac{1}{\A^6}\int_0^{1} (\tilde\U_1-\tilde\U_2)(t, \eta)\Big(\tilde\U_{2,\eta}(t, \eta)\int_0^{\eta}
 e^{-\frac{1}{\sqtC{2}}( \tilde\Y_2(t,\eta)-\tilde\Y_2(t,\theta)) }\tilde\D_2\tilde\U_{2,\eta}(t,\theta)d\theta \\
& \qquad -\tilde\U_{1,\eta}(t, \eta)\int_0^{\eta} e^{-\frac{1}{\sqtC{1}}( \tilde\Y_1(t,\eta)-\tilde\Y_1(t,\theta)) } 
\tilde\D_1\tilde\U_{1,\eta}(t,\theta)d\theta \Big)d\eta\\
&=\frac{1}{\A^6}\int_0^{1} (\tilde\U_1-\tilde\U_2)(\sqtC{2}\tilde\P_2 \tilde\U_2\tilde\U_{2,\eta}- \sqtC{1}\tilde\P_1 \tilde\U_1\tilde\U_{1,\eta})(t, \eta)d\eta \\
&\quad-\frac{1}{\A^6}\int_0^{1} (\tilde\U_1-\tilde\U_2)(\tilde\D_2 \tilde\U_2\tilde\U_{2,\eta}- \tilde\D_1 \tilde\U_1\tilde\U_{1,\eta})(t, \eta)d\eta \\
&\quad+\frac{1}{\A^6}\int_0^{1} (\tilde\U_1-\tilde\U_2)(t, \eta)\\
&\qquad\times\Big(\tilde\U_{2,\eta}(t, \eta)\int_0^{\eta} e^{-\frac{1}{\sqtC{2}}( \tilde\Y_2(t,\eta)-\tilde\Y_2(t,\theta))}\frac{1}{\sqtC{2}}\tilde\D_2\tilde\U_{2}\tilde\Y_{2,\eta}(t,\theta)d\theta\\
&\qquad\qquad-\tilde\U_{1,\eta}(t, \eta)\int_0^{\eta} e^{-\frac{1}{\sqtC{1}}( \tilde\Y_1(t,\eta)-\tilde\Y_1(t,\theta))}\frac{1}{\sqtC{1}}\tilde\D_1\tilde\U_{1}\tilde\Y_{1,\eta}(t,\theta)d\theta\Big)d\eta\\
&\quad-\frac{3}{\A^6}\int_0^{1} (\tilde\U_1-\tilde\U_2)(t, \eta)\Big(\tilde\U_{2,\eta}(t, \eta)\int_0^{\eta} e^{-\frac{1}{\sqtC{2}}( \tilde\Y_2(t,\eta)-\tilde\Y_2(t,\theta))}\tilde\P_2\tilde\U_{2}\tilde\Y_{2,\eta}(t,\theta)d\theta\\
&\qquad-\tilde\U_{1,\eta}(t, \eta)\int_0^{\eta} e^{-\frac{1}{\sqtC{1}}( \tilde\Y_1(t,\eta)-\tilde\Y_1(t,\theta))}\tilde\P_1\tilde\U_{1}\tilde\Y_{1,\eta}(t,\theta)d\theta\Big)d\eta\\
&\quad+\frac{1}{\A^6}\int_0^{1} (\tilde\U_1-\tilde\U_2)(t, \eta)\Big(\tilde\U_{2,\eta}(t, \eta)\int_0^{\eta} e^{-\frac{1}{\sqtC{2}}( \tilde\Y_2(t,\eta)-\tilde\Y_2(t,\theta))}\tilde\U_2^3\tilde\Y_{2,\eta}(t,\theta)d\theta\\
&\qquad-\tilde\U_{1,\eta}(t, \eta)\int_0^{\eta} e^{-\frac{1}{\sqtC{1}}( \tilde\Y_1(t,\eta)-\tilde\Y_1(t,\theta))}\tilde\U_1^3\tilde\Y_{1,\eta}(t,\theta)d\theta\Big)d\eta\\
&\quad+\frac{1}{2\A^6}\int_0^{1} (\tilde\U_1-\tilde\U_2)(t, \eta)\Big(\tilde\U_{2,\eta}(t, \eta)\int_0^{\eta} e^{-\frac{1}{\sqtC{2}}( \tilde\Y_2(t,\eta)-\tilde\Y_2(t,\theta))}\sqtC{2}^5\tilde\U_2(t,\theta)d\theta\\
&\qquad-\tilde\U_{1,\eta}(t, \eta)\int_0^{\eta} e^{-\frac{1}{\sqtC{1}}( \tilde\Y_1(t,\eta)-\tilde\Y_1(t,\theta))}\sqtC{1}^5\tilde\U_1(t,\theta)d\theta\Big)d\eta\\
&= W_1+W_2+W_3+W_4+W_5+W_6.
\end{align*}
Here we have used the rewrite employed when manipulating the term $\bar K_1$ from the expression \eqref{eq:barK1_1} to \eqref{eq:barK1_2}.
We start by considering the term $W_1$:
\begin{align*}
W_1&=\frac{1}{\A^6}\int_0^{1} (\tilde\U_1-\tilde\U_2)(\sqtC{2}\tilde\P_2 \tilde\U_2\tilde\U_{2,\eta}- \sqtC{1}\tilde\P_1 \tilde\U_1\tilde\U_{1,\eta})(t, \eta)d\eta \\
&= \frac{1}{\A^6}\mathbbm{1}_{\sqtC{1}\leq\sqtC{2}}(\sqtC{2}-\sqtC{1})\int_0^1 (\tilde \U_1-\tilde \U_2)\tilde \P_2\tilde \U_2\tilde \U_{2,\eta}(t,\eta) d\eta\\
& \quad + \frac{1}{\A^6} \mathbbm{1}_{\sqtC{2}<\sqtC{1}}(\sqtC{2}-\sqtC{1})\int_0^1 (\tilde \U_1-\tilde \U_2)\tilde \P_1\tilde \U_1\tilde \U_{1,\eta}(t,\eta) d\eta\\
&\quad +\frac{\ma}{\A^6}\int_0^{1} (\tilde\U_1-\tilde\U_2)(\tilde\P_2-\tilde\P_1)   \tilde\U_2\tilde\U_{2,\eta}(t, \eta)\mathbbm{1}_{\tilde\P_1\le\tilde\P_2}(t,\eta)d\eta \\
&\quad+\frac{\ma}{\A^6}\int_0^{1} (\tilde\U_1-\tilde\U_2)(\tilde\P_2-\tilde\P_1)   \tilde\U_1\tilde\U_{1,\eta}(t, \eta)\mathbbm{1}_{\tilde\P_2<\tilde\P_1}(t,\eta)d\eta \\
&\quad+\frac{\ma}{\A^6}\int_0^{1} (\tilde\U_1-\tilde\U_2)(\tilde\U_{2,\eta}-\tilde\U_{1,\eta})  \min_j( \tilde\P_j) \tilde\U_2(t, \eta)d\eta \\
&\quad-\frac{\ma}{\A^6}\int_0^{1} (\tilde\U_1-\tilde\U_2)^2 \min_j( \tilde\P_j) \tilde\U_{1,\eta}(t, \eta)d\eta \\
&=W_{11}+W_{12}+W_{13}+W_{14}+W_{15}+W_{16}.
\end{align*}

For $W_{11}$ we have (and similar for $W_{12}$) that
\begin{align*}
\vert W_{11}\vert & \leq \frac{\A^2}{8}\vert \sqtC{1}-\sqtC{2}\vert \int_0^1\vert \tilde \U_1-\tilde \U_2\vert (t,\eta) d\eta\\
& \leq \bigO(1)(\norm{\tilde \U_1-\tilde \U_2}^2+\vert \sqtC{1}-\sqtC{2}\vert^2)
\end{align*}

The terms $W_{13}$ and $W_{14}$ are similar:
\begin{align*}
\abs{W_{11}}&=\frac{\ma}{\A^6}\vert \int_0^{1} (\tilde\U_1-\tilde\U_2)(\tilde\P_2-\tilde\P_1)   \tilde\U_2\tilde\U_{2,\eta}(t, \eta)\mathbbm{1}_{\tilde\P_1\le\tilde\P_2}(t,\eta)d\eta \vert\\
&\le \frac{2\ma}{\A^6}\int_0^{1} \vert\tilde\U_1-\tilde\U_2\vert\, \vert\tilde\P_2^{1/2}-\tilde\P_1^{1/2}\vert  \tilde\P_2^{1/2} \vert \tilde\U_2\tilde\U_{2,\eta}\vert(t, \eta)d\eta \\
&\le \frac{\A}{2}\int_0^{1} \vert\tilde\U_1-\tilde\U_2\vert\, \vert\tilde\P_2^{1/2}-\tilde\P_1^{1/2}\vert d\eta \\
&\le \bigO(1)\big(\norm{\tilde\U_1-\tilde\U_2}^2+ \norm{\tilde\P_2^{1/2}-\tilde\P_1^{1/2}}^2 \big),
\end{align*}
using \eqref{eq:all_estimatesA} and \eqref{eq:all_estimatesF}.

The term $W_{15}$ goes as follows: 
\begin{align*}
\abs{W_{15}}&=\frac{\ma}{\A^6}\vert \int_0^{1} (\tilde\U_1-\tilde\U_2)(\tilde\U_{2,\eta}-\tilde\U_{1,\eta})  \min_j( \tilde\P_j) \tilde\U_2(t, \eta)d\eta \vert\\
&= \vert \frac{\ma}{2\A^6}\big(\tilde\U_1-\tilde\U_2 \big)^2 \min_j( \tilde\P_j)\tilde\U_2(t, \eta)\big\vert_0^1\\
&\qquad- \frac{\ma}{2\A^6}\int_0^{1}(\tilde\U_1-\tilde\U_2 \big)^2\frac{d}{d\eta}\big(\min_j( \tilde\P_j)\tilde\U_2\big)(t, \eta)d\eta\vert \\
&\le \frac{\ma}{2\A^6}\int_0^{1}(\tilde\U_1-\tilde\U_2 \big)^2\vert \frac{d}{d\eta}\big(\min_j( \tilde\P_j)\tilde\U_2\big)\vert(t, \eta)d\eta \\
&\le \bigO(1)\norm{\tilde\U_1-\tilde\U_2}^2,
\end{align*}
see the estimates for $\bar B_{15}$, cf.~\eqref{eq:barB13}.   As for the term $W_{16}$ we get:
\begin{align*}
\abs{W_{16}}&=\frac{\ma}{\A^6}\int_0^{1} (\tilde\U_1-\tilde\U_2)^2 \min_j( \tilde\P_j) \vert\tilde\U_{1,\eta}\vert (t, \eta)d\eta \\
&\le \bigO(1)\norm{\tilde\U_1-\tilde\U_2}^2,
\end{align*}
using \eqref{eq:all_estimatesP}.

As for the term $W_2$ we find
\begin{align*}
-W_2&= \frac{1}{\A^6}\int_0^{1} (\tilde\U_1-\tilde\U_2)(\tilde\D_2 \tilde\U_2\tilde\U_{2,\eta}- \tilde\D_1 \tilde\U_1\tilde\U_{1,\eta})(t, \eta)d\eta \\
&=\frac{1}{\A^6}\int_0^{1} (\tilde\U_1-\tilde\U_2)(\tilde\D_2-\tilde\D_1)   \tilde\U_1\tilde\U_{1,\eta}(t, \eta)\mathbbm{1}_{\tilde\D_2\le\tilde\D_1}(t,\eta)d\eta \\
&\quad+\frac{1}{\A^6}\int_0^{1} (\tilde\U_1-\tilde\U_2)(\tilde\D_2-\tilde\D_1)   \tilde\U_2\tilde\U_{2,\eta}(t, \eta)\mathbbm{1}_{\tilde\D_1<\tilde\D_2}(t,\eta)d\eta \\
&\quad+\frac{1}{\A^6}\int_0^{1} (\tilde\U_1-\tilde\U_2)(\tilde\U_{2,\eta}-\tilde\U_{1,\eta})  \min_j( \tilde\D_j) \tilde\U_2(t, \eta)d\eta \\
&\quad-\frac{1}{\A^6}\int_0^{1} (\tilde\U_1-\tilde\U_2)^2 \min_j( \tilde\D_j) \tilde\U_{1,\eta}(t, \eta)d\eta \\
&=W_{21}+W_{22}+W_{23}+W_{24}.
\end{align*}
The terms $W_{21}$ and $W_{22}$ can be treated similarly.  We need to estimate $\tilde\D_2-\tilde\D_1$. Applying  Lemma \ref{lemma:D}, we have
\begin{align*}
\abs{W_{21}}&= \frac{1}{\A^6}\vert\int_0^{1} (\tilde\U_1-\tilde\U_2)(\tilde\D_2-\tilde\D_1)   \tilde\U_1\tilde\U_{1,\eta}(t, \eta)\mathbbm{1}_{\tilde\D_2\le\tilde\D_1}(t,\eta)d\eta\vert \\
&\leq \frac{2}{\A^{9/2}}\int_0^{1} \vert\tilde\U_1-\tilde\U_2\vert\, \vert \tilde\Y_1-\tilde\Y_2\vert  \tilde\D_1^{1/2}\vert\tilde\U_1\tilde\U_{1,\eta}\vert(t, \eta)d\eta \\
& \quad +\frac{1}{\A^6}\int_0^{1} \vert\tilde\U_1-\tilde\U_2\vert\, \vert \tilde \Y_1-\tilde\Y_2\vert (\tilde \U_1^2+\tilde \P_1)\vert\tilde\U_1\tilde\U_{1,\eta}\vert(t,\eta) d\eta\\
& \quad +  \frac{2\sqrt{2}}{\A^{9/2}}\norm{\tilde\Y_1-\tilde\Y_2}\int_0^{1}\vert\tilde\U_1-\tilde\U_2\vert\tilde \D_1^{1/2}\vert\tilde\U_1\tilde\U_{1,\eta}\vert(t,\eta)d\eta \\
& \quad  + \frac{4}{\A^3}\int_0^{1}\vert\tilde\U_1-\tilde\U_2\vert \\
&\qquad\times\Big(\int_0^\eta e^{-\frac{1}{\A}(\tilde\Y_1(t,\eta)-\tilde\Y_1(t,\theta))}(\tilde\U_1-\tilde\U_2)^2(t,\theta)d\theta\Big)^{1/2}\vert\tilde\U_1\tilde\U_{1,\eta}\vert(t,\eta)d\eta\\
& \quad + \frac{2\sqrt{2}}{\A^3} \int_0^{1}\vert\tilde\U_1-\tilde\U_2\vert\\
&\qquad\times \Big(\int_0^\eta e^{-\frac{1}{\A}(\tilde \Y_1(t,\eta)-\tilde\Y_1(t,\theta))} (\sqP{1}-\sqP{2})^2(t,\theta) d\theta\Big)^{1/2}\vert\tilde\U_1\tilde\U_{1,\eta}\vert(t,\eta)d\eta\\
& \quad + \frac{3}{\sqrt{2}\A^2}\int_0^1\vert \tilde \U_1-\tilde \U_2\vert\\
&\qquad\times\Big(\int_0^\eta e^{-\frac{1}{\ma}(\tilde \Y_1(t,\eta)-\tilde \Y_1(t,\theta))}(\tilde \Y_1-\tilde \Y_2)^2(t,\theta) d\theta\Big)^{1/2}\vert \tilde \U_1\tilde \U_{1,\eta}\vert (t,\eta) d\eta\\
& \quad +\frac{12\sqrt{2}}{\sqrt{3}e\A^2}\int_0^1 \vert \tilde \U_1-\tilde \U_2\vert\\
&\qquad\times \Big(\int_0^\eta e^{-\frac{3}{4\A}(\tilde \Y_1(t,\eta)-\tilde \Y_1(t,\theta))}d\theta\Big)^{1/2}\vert \tilde\U_1\tilde \U_{1,\eta}\vert (t,\eta) d\eta \vert \sqtC{1}-\sqtC{2}\vert\\
& \quad +\frac{3}{2\A^2}\int_0^{1}\vert\tilde\U_1-\tilde\U_2\vert \\
&\qquad\times\Big(\int_0^\eta e^{-\frac{1}{\ma}(\tilde\Y_1(t,\eta)-\tilde\Y_1(t,\theta))} \vert\tilde\Y_1-\tilde\Y_2\vert(t,\theta) d\theta\Big)\vert\tilde\U_1\tilde\U_{1,\eta}\vert(t,\eta)d\eta\\
& \quad + \frac{6}{\A^2}\vert\sqtC{1}-\sqtC{2}\vert\int_0^{1}\vert\tilde\U_1-\tilde\U_2\vert  \Big(\int_0^\eta e^{-\frac{3}{4\A}(\tilde\Y_1(t,\eta)-\tilde\Y_1(t,\theta))} d\theta\Big)\vert\tilde\U_1\tilde\U_{1,\eta}\vert(t,\eta)d\eta \\
&\le \frac{2}{\A^{9/2}} \int_0^{1} \vert\tilde\U_1-\tilde\U_2\vert\, \vert \tilde\Y_1-\tilde\Y_2\vert (t,\eta)\frac{\A^{13/2}}{2\sqrt{2}} d\eta \\
&\quad+ \frac{1}{\A^6}\int_0^{1} \vert\tilde\U_1-\tilde\U_2\vert\, \vert \tilde \Y_1-\tilde\Y_2\vert (t,\eta)\frac{3\A^8}{8}d\eta\\
& \quad +  \frac{2\sqrt{2}}{A^{9/2}}\norm{\tilde\Y_1-\tilde\Y_2}\int_0^{1}\vert\tilde\U_1-\tilde\U_2\vert(t,\eta)\frac{\A^{13/2}}{2\sqrt{2}}d\eta \\
&\quad+\frac{4}{\A^3}\int_0^{1}\vert\tilde\U_1-\tilde\U_2\vert(t,\eta) \Big(\int_0^\eta (\tilde\U_1-\tilde\U_2)^2(t,\theta)d\theta\Big)^{1/2} \frac{\A^4}{2} d\eta\\
&\quad + \frac{2\sqrt{2}}{\A^3} \int_0^{1}\vert\tilde\U_1-\tilde\U_2\vert(t,\eta)\Big(\int_0^\eta  (\sqP{1}-\sqP{2})^2(t,\theta) d\theta\Big)^{1/2}\frac{\A^4}{2} d\eta\\
&\quad + \frac{3}{\sqrt{2}\A^2} \int_0^{1}\vert\tilde\U_1-\tilde\U_2\vert(t,\eta)\Big(\int_0^\eta  (\tilde \Y_1-\tilde \Y_2)^2(t,\theta) d\theta\Big)^{1/2}\frac{\A^4}{2} d\eta\\
& \quad + \frac{12\sqrt{2}}{\sqrt{3}e\A^2}\int_0^1 \vert \tilde \U_1-\tilde \U_2\vert \frac{\A^4}{2}d\eta \vert \sqtC{1}-\sqtC{2}\vert\\
&\quad +\frac{3}{2\A^2}\int_0^{1}\vert\tilde\U_1-\tilde\U_2\vert(t,\eta)\Big(\int_0^\eta \vert\tilde\Y_1-\tilde\Y_2\vert(t,\theta) d\theta\Big)\frac{\A^4}{2} d\eta\\
&\quad +\frac{6}{\A^2}\vert\sqtC{1}-\sqtC{2}\vert\int_0^{1}\vert\tilde\U_1-\tilde\U_2\vert(t,\eta) \frac{\A^4}{2}d\eta \\
&\le \bigO(1)\big(\norm{\tilde\U_1-\tilde\U_2}^2+\norm{\sqP{1}-\sqP{2}}^2+\norm{\tilde\Y_1-\tilde\Y_2}^2+\abs{\sqtC{1}-\sqtC{2}}^2\big),
\end{align*}
where we have used estimates \eqref{eq:all_estimatesA}, \eqref{eq:all_estimatesB},  \eqref{eq:all_estimatesF}, and \eqref{eq:all_estimatesN}.
The term $W_{23}$  goes as follows:
\begin{align*}
\abs{W_{23}}&= \frac{1}{\A^6}\vert \int_0^{1} (\tilde\U_1-\tilde\U_2)(\tilde\U_{2,\eta}-\tilde\U_{1,\eta})  \min_j( \tilde\D_j) \tilde\U_2(t, \eta)d\eta\vert \\
&\le  \vert \frac{1}{2\A^6}(\tilde\U_1-\tilde\U_2)^2 \min_j( \tilde\D_j) \tilde\U_2(t, \eta)\Big\vert_{0}^1 \\
&\qquad-  \frac{1}{2\A^6}\int_0^{1} (\tilde\U_1-\tilde\U_2)^2 \frac{d}{d\theta}
\big(\min_j( \tilde\D_j) \tilde\U_2\big) (t, \eta)d\eta\vert \\
&\leq\frac{1}{2\A^6}\int_0^{1} (\tilde\U_1-\tilde\U_2)^2 \vert\frac{d}{d\theta}
\big(\min_j(\tilde\D_j) \tilde\U_2\big)\vert (t, \eta)d\eta  \\
&\le \bigO(1) \norm{\tilde\U_1-\tilde\U_2}^2,
\end{align*}
by applying  Lemma \ref{lemma:5} (ii).   

The term $W_{24}$  goes as follows:
\begin{align*}
\abs{W_{24}}&\leq \frac{1}{\A^6}\int_0^{1} (\tilde\U_1-\tilde\U_2)^2 \min_j( \tilde\D_j) \vert\tilde\U_{1,\eta}\vert(t, \eta)d\eta \\
&\le \bigO(1)\norm{\tilde\U_1-\tilde\U_2}^2,
\end{align*}
using 
\begin{equation*}
\min_j( \tilde\D_j) \vert\tilde\U_{1,\eta}\vert\le 2\sqtC{1}\tilde\P_1\vert\tilde\U_{1,\eta}\vert\leq\frac{\A^7}{\sqrt{2}},
\end{equation*}
from \eqref{eq:all_estimatesN} and  \eqref{eq:all_estimatesP}.

Next, we turn to the term $W_3$: 
\begin{align*}
W_3&=\frac{1}{\A^6}\int_0^{1} (\tilde\U_1-\tilde\U_2)(t, \eta)\Big(\tilde\U_{2,\eta}(t, \eta)\int_0^{\eta} e^{-\frac{1}{\sqtC{2}}( \tilde\Y_2(t,\eta)-\tilde\Y_2(t,\theta))}\frac{1}{\sqtC{2}}\tilde\D_2\tilde\U_{2}\tilde\Y_{2,\eta}(t,\theta)d\theta\\
&\qquad-\tilde\U_{1,\eta}(t, \eta)\int_0^{\eta} e^{-\frac{1}{\sqtC{1}}( \tilde\Y_1(t,\eta)-\tilde\Y_1(t,\theta))}\frac{1}{\sqtC{1}}\tilde\D_1\tilde\U_{1}\tilde\Y_{1,\eta}(t,\theta)d\theta\Big)d\eta\\
&=\frac{1}{\A^6}\int_0^{1} (\tilde\U_1-\tilde\U_2)(t, \eta)\Big(\tilde\U_{2,\eta}(t, \eta)\int_0^{\eta} e^{-\frac{1}{\sqtC{2}}( \tilde\Y_2(t,\eta)-\tilde\Y_2(t,\theta))}\frac{1}{\sqtC{2}}\tilde\D_2\tilde\U_{2}^+\tilde\Y_{2,\eta}(t,\theta)d\theta\\
&\qquad-\tilde\U_{1,\eta}(t, \eta)\int_0^{\eta} e^{-\frac{1}{\sqtC{1}}( \tilde\Y_1(t,\eta)-\tilde\Y_1(t,\theta))}\frac{1}{\sqtC{1}}\tilde\D_1\tilde\U_{1}^+\tilde\Y_{1,\eta}(t,\theta)d\theta\Big)d\eta\\
&\quad+\frac{1}{\A^6}\int_0^{1} (\tilde\U_1-\tilde\U_2)(t, \eta)\\
&\qquad\qquad\times\Big(\tilde\U_{2,\eta}(t, \eta)\int_0^{\eta} e^{-\frac{1}{\sqtC{2}}( \tilde\Y_2(t,\eta)-\tilde\Y_2(t,\theta))}\frac{1}{\sqtC{2}}\tilde\D_2\tilde\U_{2}^-\tilde\Y_{2,\eta}(t,\theta)d\theta\\
&\qquad\qquad\qquad-\tilde\U_{1,\eta}(t, \eta)\int_0^{\eta} e^{-\frac{1}{\sqtC{1}}( \tilde\Y_1(t,\eta)-\tilde\Y_1(t,\theta))}\frac{1}{\sqtC{1}}\tilde\D_1\tilde\U_{1}^-\tilde\Y_{1,\eta}(t,\theta)d\theta\Big)d\eta\\
&=W_{31}+W_{32};
\end{align*}
these terms can be treated similarly:
\begin{align*}
W_{31}&=\frac{1}{\A^6}\int_0^{1} (\tilde\U_1-\tilde\U_2)(t, \eta)\Big(\tilde\U_{2,\eta}(t, \eta)\int_0^{\eta} e^{-\frac{1}{\sqtC{2}}( \tilde\Y_2(t,\eta)-\tilde\Y_2(t,\theta))}\frac{1}{\sqtC{2}}\tilde\D_2\tilde\U_{2}^+\tilde\Y_{2,\eta}(t,\theta)d\theta\\
&\qquad\qquad\qquad-\tilde\U_{1,\eta}(t, \eta)\int_0^{\eta} e^{-\frac{1}{\sqtC{1}}( \tilde\Y_1(t,\eta)-\tilde\Y_1(t,\theta))}\frac{1}{\sqtC{1}}\tilde\D_1\tilde\U_{1}^+\tilde\Y_{1,\eta}(t,\theta)d\theta\Big)d\eta\\
& = \mathbbm{1}_{\sqtC{1}\leq\sqtC{2}}\frac{1}{\A^6}\Big(\frac{1}{\sqtC{2}}-\frac{1}{\sqtC{1}}\Big)\int_0^1 (\tilde \U_1-\tilde \U_2)\tilde \U_{1,\eta}(t,\eta) \\
& \qquad \qquad \qquad \times\Big(\int_0^\eta e^{-\frac{1}{\sqtC{1}}(\tilde \Y_1(t,\eta)-\tilde \Y_1(t,\theta))} \tilde \D_1\tilde \V_1^+\tilde \Y_{1,\eta}(t,\theta) d\theta\Big) d\eta\\
& \qquad + \mathbbm{1}_{\sqtC{2}<\sqtC{1}}\frac{1}{\A^6}\Big(\frac{1}{\sqtC{2}}-\frac{1}{\sqtC{1}}\Big) \int_0^1 (\tilde \U_1-\tilde \U_2)\tilde \U_{2,\eta}(t,\eta)\\
& \qquad \qquad \qquad \times\Big(\int_0^\eta e^{-\frac{1}{\sqtC{2}}(\tilde \Y_2(t,\eta)-\tilde \Y_2(t,\theta))}\tilde \D_2\tilde \V_2^+\tilde \Y_{2,\eta}(t,\theta) d\theta\Big) d\eta\\
&\quad +\frac{1}{\A^7}\int_0^{1}(\tilde\U_1-\tilde\U_2)\tilde\U_{2,\eta}(t, \eta)\\
&\qquad\times\Big(\int_0^{\eta} e^{-\frac{1}{\sqtC{2}}( \tilde\Y_2(t,\eta)-\tilde\Y_2(t,\theta))}(\tilde\D_2-\tilde\D_1)\tilde\U_{2}^+\tilde\Y_{2,\eta}\mathbbm{1}_{\tilde\D_1\le\tilde\D_2}
(t,\theta)d\theta\Big)d\eta\\
&\quad+\frac{1}{\A^7}\int_0^{1}(\tilde\U_1-\tilde\U_2)\tilde\U_{1,\eta}(t, \eta)\\
&\qquad\times\Big(\int_0^{\eta} e^{-\frac{1}{\sqtC{1}}( \tilde\Y_1(t,\eta)-\tilde\Y_1(t,\theta))}(\tilde\D_2-\tilde\D_1)\tilde\U_{1}^+\tilde\Y_{1,\eta}\mathbbm{1}_{\tilde\D_2<\tilde\D_1}
(t,\theta)d\theta\Big)d\eta\\
&\quad+\frac{1}{\A^7}\int_0^{1}(\tilde\U_1-\tilde\U_2)\tilde\U_{2,\eta}(t, \eta)\\
&\qquad\times\Big(\int_0^{\eta} e^{-\frac{1}{\sqtC{2}}( \tilde\Y_2(t,\eta)-\tilde\Y_2(t,\theta))}(\tilde\U_{2}^+ -\tilde\U_{1}^+)\min_j(\tilde\D_j)\tilde\Y_{2,\eta}\mathbbm{1}_{\tilde\U_1^+\le\tilde\U_2^+}(t,\theta)d\theta\Big)d\eta\\
&\quad+\frac{1}{\A^7}\int_0^{1}(\tilde\U_1-\tilde\U_2)\tilde\U_{1,\eta}(t, \eta)\\
&\qquad\times\Big(\int_0^{\eta} e^{-\frac{1}{\sqtC{1}}( \tilde\Y_1(t,\eta)-\tilde\Y_1(t,\theta))}(\tilde\U_{2}^+ -\tilde\U_{1}^+)\min_j(\tilde\D_j)\tilde\Y_{1,\eta}\mathbbm{1}_{\tilde\U_2^+<\tilde\U_1^+}(t,\theta)d\theta\Big)d\eta\\
& \quad +\frac{1}{\A^7}\mathbbm{1}_{\sqtC{1}\leq \sqtC{2}}\int_0^1 (\tilde \U_1-\tilde \U_2)\tilde \U_{2,\eta}(t,\eta) \\
& \qquad \times\Big(\int_0^\eta (e^{-\frac{1}{\sqtC{2}}(\tilde \Y_2(t,\eta)-\tilde \Y_2(t,\theta))}-e^{-\frac{1}{\sqtC{1}}(\tilde \Y_2(t,\eta)-\tilde \Y_2(t,\theta))})\\
&\qquad\qquad\qquad\qquad\qquad\qquad\times\min_j(\tilde \D_j)\min_j(\tilde \V_j^+)\tilde \Y_{2,\eta} (t,\theta) d\theta\Big) d\eta\\
& \quad +\frac{1}{\A^7} \mathbbm{1}_{\sqtC{2}<\sqtC{1}} \int_0^1 (\tilde \U_1-\tilde \U_2)\tilde \U_{1,\eta}(t,\eta)\\
& \qquad \times \Big(\int_0^\eta (e^{-\frac{1}{\sqtC{2}}(\tilde \Y_1(t,\eta)-\tilde \Y_1(t,\theta))}-e^{-\frac{1}{\sqtC{1}}(\tilde \Y_1(t,\eta)-\tilde \Y_1(t,\theta))})\\
&\qquad\qquad\qquad\qquad\qquad\qquad\times\min_j(\tilde \D_j)\min_j(\tilde \V_j^+)\tilde \Y_{1,\eta}(t,\theta) d\theta\Big)d\eta\\
&\quad+\frac{1}{\A^7}\int_0^{1}(\tilde\U_1-\tilde\U_2)\tilde\U_{2,\eta}(t, \eta)\\
&\qquad\times\Big(\int_0^{\eta} \big(e^{-\frac{1}{\ma}( \tilde\Y_2(t,\eta)-\tilde\Y_2(t,\theta))}-e^{-\frac{1}{\ma}( \tilde\Y_1(t,\eta)-\tilde\Y_1(t,\theta))}\big)\\
&\qquad\qquad\qquad\qquad\qquad\qquad\times\min_j(\tilde\D_j) \min_j(\tilde\U_j^+) \tilde\Y_{2,\eta}\mathbbm{1}_{B(\eta)}(t,\theta)d\theta\Big)d\eta \\
&\quad+\frac{1}{\A^7}\int_0^{1}(\tilde\U_1-\tilde\U_2)\tilde\U_{1,\eta}(t, \eta)\\
&\qquad\times\Big(\int_0^{\eta} \big(e^{-\frac{1}{\ma}( \tilde\Y_2(t,\eta)-\tilde\Y_2(t,\theta))}-e^{-\frac{1}{\ma}( \tilde\Y_1(t,\eta)-\tilde\Y_1(t,\theta))}\big)\\
&\qquad\qquad\qquad\qquad\qquad\qquad\times\min_j(\tilde\D_j) \min_j(\tilde\U_j^+) \tilde\Y_{1,\eta}\mathbbm{1}_{B^c(\eta)}(t,\theta)d\theta\Big)d\eta \\
&\quad-\frac{1}{\A^7}\int_0^{1}(\tilde\U_1-\tilde\U_2)(\tilde\U_1-\tilde\U_2)_\eta(t,\eta) \\
&\qquad\times\min_k\Big(\int_0^{\eta} \min_j\big(e^{-\frac{1}{\ma}( \tilde\Y_j(t,\eta)-\tilde\Y_j(t,\theta))} \big) \min_j(\tilde\D_j)\min_j(\tilde\U_j^+)\tilde\Y_{k,\eta}(t,\theta)  d\theta\Big)d\eta \\
&\quad+\frac{1}{\A^7}\int_0^{1}(\tilde\U_1-\tilde\U_2)\tilde\U_{2,\eta} \mathbbm{1}_{D^c}(t,\eta)\\
&\quad\times\Big(\int_0^{\eta} \min_j\big(e^{-\frac{1}{\ma}( \tilde\Y_j(t,\eta)-\tilde\Y_j(t,\theta))} \big) \min_j(\tilde\D_j)\min_j(\tilde\U_j^+)(\tilde\Y_{2,\eta}-\tilde\Y_{1,\eta})(t,\theta)  d\theta\Big)d\eta \\
&\quad+\frac{1}{\A^7}\int_0^{1}(\tilde\U_1-\tilde\U_2)\tilde\U_{1,\eta} \mathbbm{1}_{D}(t,\eta)\\
&\quad\times\Big(\int_0^{\eta} \min_j\big(e^{-\frac{1}{\ma}( \tilde\Y_j(t,\eta)-\tilde\Y_j(t,\theta))} \big) \min_j(\tilde\D_j)\min_j(\tilde\U_j^+)(\tilde\Y_{2,\eta}-\tilde\Y_{1,\eta}) (t,\theta) d\theta\Big)d\eta \\
&= Z_1+Z_2+Z_3+Z_4+Z_5+Z_6+Z_7+Z_8+Z_9+Z_{10}+Z_{11}+Z_{12}+Z_{13},
\end{align*}
where the set $D$ is defined by \eqref{eq:set_e}.  At this point it cannot come as a surprise that we have to treat these terms separately.  

Let us start with $Z_1$ ($Z_2$ is similar)
\begin{align*}
\vert Z_1\vert &\leq  \frac{2}{\A^7}\vert \sqtC{1}-\sqtC{2}\vert \int_0^1 \vert \tilde \U_1-\tilde \U_2\vert \vert \tilde \U_{1,\eta}\vert (t,\eta)\\
& \qquad \qquad \times \Big(\int_0^\eta e^{-\frac{1}{\sqtC{1}}(\tilde \Y_1(t,\eta)-\tilde \Y_1(t,\theta))}\tilde \U_1^2\tilde \Y_{1,\eta}(t,\theta) d\theta\Big)^{1/2}\\
& \qquad \qquad \qquad \times \Big(\int_0^\eta e^{-\frac{1}{\sqtC{1}}(\tilde \Y_1(t,\eta)-\tilde \Y_1(t,\theta))}\tilde \P_1^2\tilde \Y_{1,\eta}(t,\theta) d\theta\Big)^{1/2}d\eta\\
& \leq \frac{2\sqrt{6}}{\A^4}\vert \sqtC{1}-\sqtC{2}\vert \int_0^1 \vert \tilde \U_1-\tilde \U_2\vert \tilde \P_1\vert \tilde \U_{1,\eta}\vert (t,\eta) d\eta\\
& \leq \bigO(1)(\norm{\tilde \U_1-\tilde \U_2}^2+\vert \sqtC{1}-\sqtC{2}\vert ^2).
\end{align*}

The terms $Z_3$ and $Z_4$ go as follows:
\begin{align*}
\abs{Z_3}&=\frac{1}{\A^7}\vert\int_0^{1}(\tilde\U_1-\tilde\U_2)\tilde\U_{2,\eta}(t, \eta)\\
&\qquad\times\Big(\int_0^{\eta} e^{-\frac{1}{\sqtC{2}}( \tilde\Y_2(t,\eta)-\tilde\Y_2(t,\theta))}(\tilde\D_2-\tilde\D_1)\tilde\U_{2}^+\tilde\Y_{2,\eta}\mathbbm{1}_{\tilde\D_1\le\tilde\D_2}
(t,\theta)d\theta\Big)d\eta\vert\\
&\le \bigO(1)\big( \norm{\tilde\U_1-\tilde\U_2}^2+\norm{\tilde\Y_1-\tilde\Y_2}^2+\norm{\sqP{1}-\sqP{2}}^2+\vert \sqtC{1}-\sqtC{2}\vert^2\big),
\end{align*}
by Lemma~\ref{lemma:D}. The terms $Z_5$ and $Z_6$ can be treated similarly:
\begin{align*}
\abs{Z_5}&=\frac{1}{\A^7}\vert\int_0^{1}(\tilde\U_1-\tilde\U_2)\tilde\U_{2,\eta}(t, \eta)\\
&\qquad\times\Big(\int_0^{\eta} e^{-\frac{1}{\sqtC{2}}( \tilde\Y_2(t,\eta)-\tilde\Y_2(t,\theta))}(\tilde\U_{2}^+ -\tilde\U_{1}^+)\min_j(\tilde\D_j)\tilde\Y_{2,\eta}\mathbbm{1}_{\tilde\U_1^+\le\tilde\U_2^+}(t,\theta)d\theta\Big)d\eta\vert \\
&\le \norm{\tilde\U_1-\tilde\U_2}^2+\frac1{\A^{14}}\int_0^{1}\tilde\U_{2,\eta}^2(t, \eta)\\
&\qquad\times\Big(\int_0^{\eta}  e^{-\frac{1}{\sqtC{2}}( \tilde\Y_2(t,\eta)-\tilde\Y_2(t,\theta))} (\tilde\U_{2}^+ -\tilde\U_{1}^+)\min_j(\tilde\D_j)\tilde\Y_{2,\eta}\mathbbm{1}_{\tilde\U_1^+\le\tilde\U_2^+}(t,\theta)d\theta  \Big)^2d\eta \\
&\le \bigO(1) \norm{\tilde\U_1-\tilde\U_2}^2,
\end{align*}
by applying \eqref{eq:343}, estimating $\tilde \D_2\le 2\sqtC{2}\tilde \P_2$ (cf.~\eqref{eq:all_estimatesN}), and subsequently $ 2\sqrt{2}\tilde \P_2\tilde\U_{2,\eta}^2\le \sqtC{2}^6$ (cf.~\eqref{eq:all_estimatesQ}). 
The terms $Z_7$ and $Z_8$ follow this pattern
\begin{align*}
\abs{Z_8}&\leq \frac{4}{\A^7\ma e}\int_0^1 \vert \tilde \U_1-\tilde \U_2\vert \vert \tilde \U_{1,\eta}\vert (t,\eta)\\
& \qquad \qquad \times \Big(\int_0^\eta e^{-\frac{3}{4\sqtC{1}}(\tilde \Y_1(t,\eta)-\tilde \Y_1(t,\theta))}2\sqtC{1}\tilde \P_1\tilde \Y_{1,\eta}\vert \tilde \U_2\vert (t,\theta) d\theta\Big) d\eta\vert \sqtC{1}-\sqtC{2}\vert\\
&\leq \frac{4\sqrt{2}}{\A^5e}\int_0^1 \vert \tilde \U_1-\tilde \U_2\vert \vert \tilde \U_{1,\eta}\vert (t,\eta) \Big(\int_0^\eta e^{-\frac{1}{\sqtC{2}}(\tilde \Y_1(t,\eta)-\tilde \Y_2(t,\eta))} \tilde \P_1^2\tilde \Y_{1,\eta}(t,\theta) d\theta\Big)^{1/2}\\
& \qquad \qquad \times \Big(\int_0^\eta e^{-\frac{1}{2\sqtC{2}}(\tilde \Y_1(t,\eta)-\tilde \Y_1(t,\theta))}\tilde \Y_{1,\eta}(t,\theta) d\theta \Big)^{1/2} d\eta\vert \sqtC{1}-\sqtC{2}\vert\\
& \leq \frac{4\sqrt{6}}{\A^2e} \int_0^1 \vert \tilde \U_1-\tilde \U_2\vert \sqP{1}\vert \tilde \U_{1,\eta}\vert (t,\eta) d\eta\vert \sqtC{1}-\sqtC{2}\vert\\
& \leq \bigO(1)(\norm{\tilde \U_1-\tilde\U_2}^2+\vert \sqtC{1}-\sqtC{2}\vert^2).
\end{align*}

The terms $Z_9$ and $Z_{10}$ follow this pattern:
\begin{align*}
\abs{Z_9}&=\frac{1}{\A^7}\vert\int_0^{1}(\tilde\U_1-\tilde\U_2)\tilde\U_{2,\eta}(t, \eta)\Big(\int_0^{\eta} \big(e^{-\frac{1}{\ma}( \tilde\Y_2(t,\eta)-\tilde\Y_2(t,\theta))}-e^{-\frac{1}{\ma}( \tilde\Y_1(t,\eta)-\tilde\Y_1(t,\theta))}\big)\\
&\qquad\qquad\qquad\qquad\qquad\qquad\qquad\times\min_j(\tilde\D_j) \min_j(\tilde\U_j^+) \tilde\Y_{2,\eta}\mathbbm{1}_{B(\eta)}(t,\theta)
d\theta\Big)d\eta\vert \\
&\le \frac{\ma}{\sqrt{2}\A^7}\int_0^{1}\vert\tilde\U_1-\tilde\U_2\vert\vert\tilde\U_{2,\eta}\vert(t, \eta)\\
&\qquad\qquad\qquad\qquad\times\Big(\int_0^{\eta} \big(\vert\tilde\Y_2(t,\eta)-\tilde\Y_1(t,\eta)\vert +\vert\tilde\Y_2(t,\theta)-\tilde\Y_1(t,\theta)\vert\big) \\
&\qquad\qquad\qquad\qquad\qquad\qquad\qquad \times e^{-\frac{1}{\ma}( \tilde\Y_2(t,\eta)-\tilde\Y_2(t,\theta))}\tilde\D_2 \tilde\Y_{2,\eta}(t,\theta)d\theta\Big)d\eta  \\
&\le \frac{1}{\sqrt{2}\A^6}\int_0^{1}\vert\tilde\U_1-\tilde\U_2\vert\,\vert\tilde\Y_2-\tilde\Y_1\vert\vert\tilde\U_{2,\eta}\vert(t, \eta)\\
&\qquad\times\Big(\int_0^{\eta} e^{-\frac{1}{\sqtC{2}}( \tilde\Y_2(t,\eta)-\tilde\Y_2(t,\theta))}\tilde\D_2 \tilde\Y_{2,\eta}(t,\theta)d\theta\Big)d\eta \\
&\quad +\frac{1}{\sqrt{2}\A^6}\int_0^{1}\vert\tilde\U_1-\tilde\U_2\vert\vert\tilde\U_{2,\eta}\vert(t, \eta)\\
&\qquad\times\Big(\int_0^{\eta} e^{-\frac{1}{\sqtC{2}}( \tilde\Y_2(t,\eta)-\tilde\Y_2(t,\theta))} \tilde\D_2  \tilde\Y_{2,\eta}
\vert\tilde\Y_2-\tilde\Y_1\vert(t,\theta)d\theta\Big)d\eta \\
&\le \frac{\sqrt{2}}{\A^5}\int_0^{1}\vert\tilde\U_1-\tilde\U_2\vert\,\vert\tilde\Y_2-\tilde\Y_1\vert\vert\tilde\U_{2,\eta}\vert(t, \eta)\\
&\qquad\qquad\qquad\qquad\qquad\times\Big(\int_0^{\eta} e^{-\frac{1}{\sqtC{2}}( \tilde\Y_2(t,\eta)-\tilde\Y_2(t,\theta))}\tilde\P_2^2  \tilde\Y_{2,\eta}(t,\theta)d\theta\Big)^{1/2}   \\
&\qquad\qquad\qquad\qquad\qquad\quad\times\Big(\int_0^{\eta}   e^{-\frac{1}{\sqtC{2}}( \tilde\Y_2(t,\eta)-\tilde\Y_2(t,\theta))}    \tilde\Y_{2,\eta}(t,\theta) d\theta\Big)^{1/2}   d\eta \\
&\quad +\norm{\tilde\U_1-\tilde\U_2}^2+\frac{2}{\A^{10}}\int_0^{1}\tilde\U_{2,\eta}^2(t, \eta)\Big(\int_0^{\eta} e^{-\frac{3}{2\sqtC{2}}( \tilde\Y_2(t,\eta)-\tilde\Y_2(t,\theta))} \tilde \P_2\tilde\Y_{2,\eta}(t,\theta)d\theta\Big)\\
&\qquad\qquad\qquad\qquad \times\Big(\int_0^{\eta} e^{-\frac{1}{2\sqtC{2}}( \tilde\Y_2(t,\eta)-\tilde\Y_2(t,\theta))} \tilde\P_2 \tilde\Y_{2,\eta}\vert\tilde\Y_2-\tilde\Y_1\vert^2(t,\theta)d\theta\Big)d\eta \\
&\le \bigO(1)\big(\norm{\tilde\U_1-\tilde\U_2}^2+\norm{\tilde\Y_1-\tilde\Y_2}^2 \big).
\end{align*} 
The term $Z_{11}$:
\begin{align*}
\abs{Z_{11}}&=\frac{1}{\A^7}\vert\int_0^{1}(\tilde\U_1-\tilde\U_2)(\tilde\U_1-\tilde\U_2)_\eta(t, \eta) \\
&\qquad\times\min_k\Big(\int_0^{\eta} \min_j\big(e^{-\frac{1}{\ma}( \tilde\Y_j(t,\eta)-\tilde\Y_j(t,\theta))} \big) \min_j(\tilde\D_j)\min_j(\tilde\U_j^+)\tilde\Y_{k,\eta}(t, \theta)  d\theta\Big)d\eta\vert \\
&\le \bigO(1) \norm{\tilde\U_1-\tilde\U_2}^2_2,
\end{align*}
following the estimates employed for the term $\bar B_{37}$, see \eqref{eq:barB37}.  

The terms $Z_{12}$ and $Z_{13}$ may be treated as follows:
\begin{align*}
\abs{Z_{12}}&=\frac{1}{\A^7}\vert\int_0^{1}(\tilde\U_1-\tilde\U_2)\tilde\U_{2,\eta} \mathbbm{1}_{D^c}(t,\eta)\Big(\int_0^{\eta} \min_j\big(e^{-\frac{1}{\ma}( \tilde\Y_j(t,\eta)-\tilde\Y_j(t,\theta))} \big)\\
&\qquad\qquad\qquad\qquad\qquad\times \min_j(\tilde\D_j)\min_j(\tilde\U_j^+)(\tilde\Y_{2,\eta}-\tilde\Y_{1,\eta})(t, \theta)  d\theta\Big)d\eta \vert\\
&=\frac{1}{\A^7}\vert\int_0^{1}(\tilde\U_1-\tilde\U_2)\tilde\U_{2,\eta} \mathbbm{1}_{D^c}(t,\eta)\\
&\quad\times\Big[\Big(\min_j\big(e^{-\frac{1}{\ma}( \tilde\Y_j(t,\eta)-\tilde\Y_j(t,\theta))} \big) \min_j(\tilde\D_j)\min_j(\tilde\U_j^+)(\tilde\Y_{2}-\tilde\Y_{1})(t,\theta)\Big)\Big\vert_{\theta=0}^{\eta} \\
&\quad -\int_0^{\eta}(\tilde\Y_{2}-\tilde\Y_{1})\frac{d}{d\theta}\big( \min_j\big(e^{-\frac{1}{\ma}( \tilde\Y_j(t,\eta)-\tilde\Y_j(t,\theta))} \big) \min_j(\tilde\D_j)\min_j(\tilde\U_j^+) \big)(t, \theta)d\theta\Big]d\eta \vert \\
&\le \frac{1}{\A^7}\vert \int_0^{1}(\tilde\U_1-\tilde\U_2)\tilde\U_{2,\eta} \mathbbm{1}_{D^c}
 \min_j(\tilde\D_j)\min_j(\tilde\U_j^+)(\tilde\Y_{2}-\tilde\Y_{1})(t,\eta) d\eta\vert\\
&\quad+\frac{1}{\A^7}\vert\int_0^{1}(\tilde\U_1-\tilde\U_2)\tilde\U_{2,\eta} \mathbbm{1}_{D^c}(t,\eta)\Big(\int_0^{\eta}(\tilde\Y_{2}-\tilde\Y_{1})\frac{d}{d\theta}\big( \min_j\big(e^{-\frac{1}{\ma}( \tilde\Y_j(t,\eta)-\tilde\Y_j(t,\theta))} \big) \\
&\qquad\qquad\qquad\qquad\qquad\qquad\qquad\qquad\times
 \min_j(\tilde\D_j)\min_j(\tilde\U_j^+) \big)(t, \theta)d\theta\Big) d\eta\vert  \\
&\le \frac{\A^2}{4}\int_0^{1}\vert\tilde\U_1-\tilde\U_2\vert\, \vert\tilde\Y_{2}-\tilde\Y_{1} \vert(t,\eta) d\eta\\
&\quad +\frac{1}{\A^7}\int_0^{1}\vert\tilde\U_1-\tilde\U_2\vert\, \vert\tilde\U_{2,\eta} \vert (t,\eta)\Big(\int_0^{\eta}\vert\tilde\Y_1-\tilde\Y_2\vert \vert\frac{d}{d\theta}\big( \min_j\big(e^{-\frac{1}{\ma}( \tilde\Y_j(t,\eta)-\tilde\Y_j(t,\theta))} \big)\\
&\qquad\qquad\qquad\qquad\qquad\qquad\qquad\qquad\times \min_j(\tilde\D_j)\min_j(\tilde\U_j^+) \big)\vert(t, \theta)d\theta\Big)d\eta\\
&= \tilde M_{1}+\tilde M_{2}.
 \end{align*}
Here we find
\begin{equation*}
\tilde M_1\le \bigO(1)\big(\norm{\tilde\U_1-\tilde\U_2}^2+\norm{\tilde\Y_1-\tilde\Y_2}^2 \big),
\end{equation*}
while $\tilde M_2$ requires more care:
\begin{align*}
\tilde M_2&= \frac{1}{\A^7}\int_0^{1}\vert\tilde\U_1-\tilde\U_2\vert\, \vert\tilde\U_{2,\eta} \vert (t,\eta)\Big(\int_0^{\eta}\vert\tilde\Y_1-\tilde\Y_2\vert\vert \frac{d}{d\theta}\big( \min_j\big(e^{-\frac{1}{\ma}( \tilde\Y_j(t,\eta)-\tilde\Y_j(t,\theta))} \big)\\
&\qquad\qquad\qquad\qquad\qquad\qquad\qquad\times \min_j(\tilde\D_j)\min_j(\tilde\U_j^+) \big)(t, \theta)\vert d\theta\Big)d\eta\\
&\le  \frac{1}{\A^7}\int_0^{1}\vert\tilde\U_1-\tilde\U_2\vert\, \vert\tilde\U_{2,\eta} \vert (t,\eta)\Big(\int_0^{\eta}\vert\tilde\Y_1-\tilde\Y_2\vert \\
&\qquad\times\big\vert\frac{d}{d\theta}\big( \min_j\big(e^{-\frac{1}{\ma}( \tilde\Y_j(t,\eta)-\tilde\Y_j(t,\theta))} \big)\big) \min_j(\tilde\D_j)\min_j(\tilde\U_j^+) \\
&\qquad\qquad\qquad+\min_j\big(e^{-\frac{1}{\ma}( \tilde\Y_j(t,\eta)-\tilde\Y_j(t,\theta))} \big) \frac{d}{d\theta}\big(\min_j(\tilde\D_j)\min_j(\tilde\U_j^+) \big) \big)(t, \theta)\vert d\theta\Big)d\eta\\
&\le \frac{1}{\A^7}\int_0^{1}\vert\tilde\U_1-\tilde\U_2\vert\, \vert\tilde\U_{2,\eta} \vert (t,\eta)\Big(\int_0^{\eta}\vert\tilde\Y_1-\tilde\Y_2\vert\vert\frac{d}{d\theta}\big( \min_j\big(e^{-\frac{1}{\ma}( \tilde\Y_j(t,\eta)-\tilde\Y_j(t,\theta))} \big)\big) \\
&\qquad\qquad\qquad\qquad\qquad\qquad\qquad\times\min_j(\tilde\D_j)\min_j(\tilde\U_j^+)(t, \theta)\vert d\theta\Big)d\eta\\
&\quad+\frac{1}{\A^7}\int_0^{1}\vert\tilde\U_1-\tilde\U_2\vert\, \vert\tilde\U_{2,\eta} \vert (t,\eta)\Big(\int_0^{\eta}\vert\tilde\Y_1-\tilde\Y_2\vert\vert\min_j\big(e^{-\frac{1}{\ma}( \tilde\Y_j(t,\eta)-\tilde\Y_j(t,\theta))} \big) \\
&\qquad\qquad\qquad\qquad\qquad\qquad\qquad\qquad\times\frac{d}{d\theta}\big(\min_j(\tilde\D_j)\min_j(\tilde\U_j^+) \big)(t, \theta)\vert d\theta\Big)d\eta\\
&\le \frac{1}{\ma\A^7}\int_0^{1}\vert\tilde\U_1-\tilde\U_2\vert\, \vert\tilde\U_{2,\eta} \vert (t,\eta)
\Big(\int_0^{\eta}\vert\tilde\Y_1-\tilde\Y_2\vert \min_j(e^{-\frac{1}{\ma}(\tilde \Y_j(t,\eta)-\tilde\Y_j(t,\theta))}) \\
&\qquad\qquad\qquad\qquad\qquad\qquad\qquad\qquad\times\max_j(\tilde\Y_{j,\eta})\min_j(\tilde\D_j)\min_j(\tilde\U_j^+)(t, \theta)d\theta\Big)d\eta\\
&\quad+\bigO(1)\frac{1}{\A^{5/2}}\int_0^{1}\vert\tilde\U_1-\tilde\U_2\vert\, \vert\tilde\U_{2,\eta} \vert (t,\eta)\\
&\qquad\times\Big(\int_0^{\eta}\vert\tilde\Y_1-\tilde\Y_2\vert\min_j\big(e^{-\frac{1}{\ma}( \tilde\Y_j(t,\eta)-\tilde\Y_j(t,\theta))} \big) 
 (\min_j(\tilde\D_j)^{1/2}+\vert \tilde\U_2\vert)(t, \theta)d\theta\Big)d\eta\\
 &=\tilde M_{21}+\tilde M_{22},
\end{align*}
where estimates for the derivatives come from Lemma \ref{lemma:1} and Lemma \ref{lemma:5}. We find for the term $\tilde M_{21}$ that
\begin{align*}
\tilde M_{21}&=\frac{1}{\ma\A^7}\int_0^{1}\vert\tilde\U_1-\tilde\U_2\vert\, \vert\tilde\U_{2,\eta} \vert (t,\eta)\Big(\int_0^{\eta}\vert\tilde\Y_1-\tilde\Y_2\vert \min_j(e^{-\frac{1}{\ma}(\tilde \Y_j(t,\eta)-\tilde\Y_j(t,\theta))})\\
&\qquad\qquad\qquad\qquad\qquad\qquad\times \max_j(\tilde\Y_{j,\eta})\min_j(\tilde\D_j)\min_j(\tilde\U_j^+)(t, \theta)d\theta\Big)d\eta\\
&\le \norm{\tilde\U_1-\tilde\U_2}^2+\frac{1}{\ma^2\A^{14}}\int_0^{1}\tilde\U_{2,\eta}^2 (t,\eta) \Big(\int_0^{\eta}\vert\tilde\Y_1-\tilde\Y_2\vert \min_j(e^{-\frac{1}{\ma}(\tilde \Y_j(t,\eta)-\tilde\Y_j(t,\theta))})\\
&\qquad\qquad\qquad\qquad\qquad\qquad\times \max_j(\tilde\Y_{j,\eta})\min_j(\tilde\D_j)\min_j(\tilde\U_j^+)(t, \theta)d\theta\Big)^2 d\eta\\
&\le \norm{\tilde\U_1-\tilde\U_2}^2+\frac12\int_0^{1}\tilde\U_{2,\eta}^2(t,\eta) \\
&\qquad\times\Big(\int_0^{\eta}\vert\tilde\Y_1-\tilde\Y_2\vert \min_j(e^{-\frac{1}{\ma}(\tilde \Y_j(t,\eta)-\tilde\Y_j(t,\theta))}) d\theta\Big)^2 d\eta\\
&\le \norm{\tilde\U_1-\tilde\U_2}^2+\frac{\A^3}{2}\int_0^{1}\tilde\Y_{2,\eta} (t,\eta) \\
&\qquad\times\Big(\int_0^{\eta}(\tilde\Y_1-\tilde\Y_2)^2 e^{-\frac{1}{\ma}(\tilde \Y_2(t,\eta)-\tilde\Y_2(t,\theta))} d\theta\Big) \Big(\int_0^{\eta} e^{-\frac{1}{\ma}(\tilde \Y_2(t,\eta)-\tilde\Y_2(t,\theta))}d\theta\Big) d\eta\\
&\le \norm{\tilde\U_1-\tilde\U_2}^2+\frac{\A^3}{2}\int_0^{1}\tilde\Y_{2,\eta}(t,\eta)e^{-\frac{1}{\ma}\tilde \Y_2(t,\eta)} \Big(\int_0^{\eta} e^{\frac{1}{\ma}\tilde \Y_2(t,\theta)} d\theta\Big) d\eta\norm{\tilde \Y_1-\tilde \Y_2}^2\\
&\le \bigO(1)\big(\norm{\tilde\U_1-\tilde\U_2}^2+\norm{\tilde\Y_1-\tilde\Y_2}^2 \big),
\end{align*}
where we used the estimate 
\begin{equation*}
\max_j(\tilde\Y_{j,\eta})\min_j(\tilde\D_j)\le 2 \max_j(\tilde\Y_{j,\eta})\min_j(\sqtC{j}\tilde\P_j)\le 2A\max_j(\tilde\P_j\tilde\Y_{j,\eta})\leq \A^6,
\end{equation*}
using \eqref{eq:all_estimatesE} and \eqref{eq:all_estimatesN}, as well as \eqref{eq:all_estimatesQ} and \eqref{eq:all_PestimatesA}.  The term $\tilde M_{22}$ reads
\begin{align*}
\tilde M_{22}&=\bigO(1)\frac{1}{\A^{5/2}}\int_0^{1}\vert\tilde\U_1-\tilde\U_2\vert\, \vert\tilde\U_{2,\eta} \vert (t,\eta)\\
&\qquad\times\Big(\int_0^{\eta}\vert\tilde\Y_1-\tilde\Y_2\vert\min_j\big(e^{-\frac{1}{\ma}( \tilde\Y_j(t,\eta)-\tilde\Y_j(t,\theta))} \big) 
 (\min_j(\tilde\D_j)^{1/2}+\vert \tilde\U_2\vert)(t, \theta)d\theta\Big)d\eta\\
&\le \norm{\tilde\U_1-\tilde\U_2}^2 +\bigO(1)\frac{1}{\A^5}\int_0^{1}\tilde\U_{2,\eta}^2(t,\eta)\\
&\qquad\qquad\qquad\qquad\qquad\times
\Big(\int_0^{\eta}e^{-\frac{1}{\ma}( \tilde\Y_2(t,\eta)-\tilde\Y_2(t,\theta))}  \vert\tilde\Y_1-\tilde\Y_2\vert\tilde\P_2^{1/2} (t,\theta)d\theta\Big)^2d\eta \\
&\le \norm{\tilde\U_1-\tilde\U_2}^2 +\bigO(1)\frac{1}{\A^5}\int_0^{1}\tilde\U_{2,\eta}^2 (t,\eta)\Big(\int_0^{\eta}e^{-\frac3{2\sqtC{2}}( \tilde\Y_2(t,\eta)-\tilde\Y_2(t,\theta))} \tilde\P_2(t,\theta)d\theta\Big)\\
&\qquad\qquad\qquad\qquad\qquad \times\Big(\int_0^{\eta} e^{-\frac1{2\sqtC{2}}( \tilde\Y_2(t,\eta)-\tilde\Y_2(t,\theta))} (\tilde\Y_1-\tilde\Y_2)^2(t,\theta)d\theta\Big)d\eta \\
&\le \norm{\tilde\U_1-\tilde\U_2}^2 +\bigO(1)\frac{1}{\A^5}\int_0^{1}\tilde\P_2\tilde\U_{2,\eta}^2(t,\eta)\\
&\qquad\qquad\qquad\qquad\qquad \times \Big(\int_0^{\eta} e^{-\frac1{2\sqtC{2}}( \tilde\Y_2(t,\eta)-\tilde\Y_2(t,\theta))} (\tilde\Y_1-\tilde\Y_2)^2(t,\theta)d\theta\Big)d\eta\\
 &\le \bigO(1)\big(\norm{\tilde\U_1-\tilde\U_2}^2+ \norm{\tilde\Y_1-\tilde\Y_2}^2\big).
\end{align*}
Here we employed 
\begin{align*}
\min_j(\tilde\D_j)^{1/2}+\vert \tilde\U_2\vert&\le \sqrt{2\sqtC{2}}\sqP{2}+\sqrt{2}\sqP{2}\le \sqrt{2}(1+\sqrt{\sqtC{2}})\sqP{2},
\end{align*}
as well as \eqref{eq:32P} and \eqref{eq:all_estimatesQ}.

\bigskip
Next, we turn to the term $W_4$: 
\begin{align*}
-W_4&=\frac{3}{\A^6}\int_0^{1} (\tilde\U_1-\tilde\U_2)(t, \eta)\Big(\tilde\U_{2,\eta}(t, \eta)\int_0^{\eta} e^{-\frac{1}{\sqtC{2}}( \tilde\Y_2(t,\eta)-\tilde\Y_2(t,\theta))}\tilde\P_2\tilde\U_{2}\tilde\Y_{2,\eta}(t,\theta)d\theta\\
&\qquad-\tilde\U_{1,\eta}(t, \eta)\int_0^{\eta} e^{-\frac{1}{\sqtC{1}}( \tilde\Y_1(t,\eta)-\tilde\Y_1(t,\theta))}\tilde\P_1\tilde\U_{1}\tilde\Y_{1,\eta}(t,\theta)d\theta\Big)d\eta\\
&=W_4^+ + W_4^-,
\end{align*}
with
\begin{align*}
W_4^{\pm}&=\frac{3}{\A^6}\int_0^{1} (\tilde\U_1-\tilde\U_2)(t, \eta)\Big(\tilde\U_{2,\eta}(t, \eta)\int_0^{\eta} e^{-\frac{1}{\sqtC{2}}( \tilde\Y_2(t,\eta)-\tilde\Y_2(t,\theta))}\tilde\P_2\tilde\U_{2}^\pm\tilde\Y_{2,\eta}(t,\theta)d\theta\\
&\qquad-\tilde\U_{1,\eta}(t, \eta)\int_0^{\eta} e^{-\frac{1}{\sqtC{1}}( \tilde\Y_1(t,\eta)-\tilde\Y_1(t,\theta))}\tilde\P_1\tilde\U_{1}^\pm\tilde\Y_{1,\eta}(t,\theta)d\theta\Big)d\eta\\
&\le \bigO(1)\big(\norm{\tilde\U_1-\tilde\U_2}^2+\norm{\sqP{1}-\sqP{2}}^2+ \norm{\tilde\Y_1-\tilde\Y_2}^2+\vert \sqtC{1}-\sqtC{2}\vert^2\big),
\end{align*}
using that $W_4^+=3N_1$, see \eqref{eq:N1}, and similar for $W_4^-$. 

\bigskip
Next, we turn to the term $W_5$: 
\begin{align*}
W_5&=\frac{1}{\A^6}\int_0^{1} (\tilde\U_1-\tilde\U_2)(t, \eta)\Big(\tilde\U_{2,\eta}(t, \eta)\int_0^{\eta} e^{-\frac{1}{\sqtC{2}}( \tilde\Y_2(t,\eta)-\tilde\Y_2(t,\theta))}\tilde\U_2^3\tilde\Y_{2,\eta}(t,\theta)d\theta\\
&\qquad-\tilde\U_{1,\eta}(t, \eta)\int_0^{\eta} e^{-\frac{1}{\sqtC{1}}( \tilde\Y_1(t,\eta)-\tilde\Y_1(t,\theta))}\tilde\U_1^3\tilde\Y_{1,\eta}(t,\theta)d\theta\Big)d\eta
\end{align*}
We take positive and negative part of the term $\tilde\U_j^3$, and thus it suffices to study the term
\begin{align*}
W_5^+&=\frac{1}{\A^6}\int_0^{1} (\tilde\U_1-\tilde\U_2)\Big[\tilde\U_{2,\eta}(t, \eta)\int_0^{\eta} e^{-\frac{1}{\sqtC{2}}( \tilde\Y_2(t,\eta)-\tilde\Y_2(t,\theta))}(\tilde\U_2^+)^3\tilde\Y_{2,\eta}(t,\theta)d\theta\\
&\qquad-\tilde\U_{1,\eta}(t, \eta)\int_0^{\eta} e^{-\frac{1}{\sqtC{1}}( \tilde\Y_1(t,\eta)-\tilde\Y_1(t,\theta))}(\tilde\U_1^+)^3\tilde\Y_{1,\eta}(t,\theta)d\theta\Big]d\eta. 
\end{align*}
Having a close look at $W_5^+$ one has 
\begin{equation*}
W_5^+=\int_0^1 (\tilde \U_1-\tilde\U_2)(J_1+J_2+J_3+J_4+J_5+J_6+J_7+J_8)(t,\eta) d\eta,
\end{equation*}
where $J_1,\dots, J_6$ are defined in \eqref{eq:splitU}. Thus we can conclude immediately that 
\begin{align*}
\abs{W_5^+}&\leq \bigO(1) \Big(\norm{\tilde\Y_1-\tilde\Y_2}^2+\norm{\tilde\U_1-\tilde\U_2}^2\\
&\qquad\qquad\qquad+\norm{\sqP{1}-\sqP{2}}^2+\vert \sqtC{1}-\sqtC{2}\vert^2\Big).
\end{align*}

\bigskip
Next, we turn to the term $W_6$: 
\begin{align*}
2W_6&=\frac{1}{\A^6}\int_0^{1} (\tilde\U_1-\tilde\U_2)(t, \eta)\Big(\tilde\U_{2,\eta}(t, \eta)\int_0^{\eta} e^{-\frac{1}{\sqtC{2}}( \tilde\Y_2(t,\eta)-\tilde\Y_2(t,\theta))}\sqtC{2}^5\tilde\U_2(t,\theta)d\theta\\
&\qquad-\tilde\U_{1,\eta}(t, \eta)\int_0^{\eta} e^{-\frac{1}{\sqtC{1}}( \tilde\Y_1(t,\eta)-\tilde\Y_1(t,\theta))}\sqtC{1}^5\tilde\U_1(t,\theta)d\theta\Big)d\eta\\
&= \frac{\sqtC{2}^5-\sqtC{1}^5}{\A^6}\mathbbm{1}_{\sqtC{1}\le\sqtC{2}} \int_0^1(\tilde\U_1-\tilde\U_2)\tilde\U_{2,\eta}(t, \eta)\int_0^{\eta} e^{-\frac{1}{\sqtC{2}}( \tilde\Y_2(t,\eta)-\tilde\Y_2(t,\theta))}\tilde\U_2(t,\theta)d\theta d\eta\\
&\quad +\frac{\sqtC{2}^5-\sqtC{1}^5}{\A^6}\mathbbm{1}_{\sqtC{2}< \sqtC{1}} \int_0^1   (\tilde\U_1-\tilde\U_2) \tilde\U_{1,\eta}(t, \eta)\\
&\qquad\qquad\qquad\qquad\qquad\qquad\times\int_0^{\eta} e^{-\frac{1}{\sqtC{1}}( \tilde\Y_1(t,\eta)-\tilde\Y_1(t,\theta))}\tilde\U_1(t,\theta)d\theta d\eta\\
& \quad +\frac{\ma^5}{\A^6}\mathbbm{1}_{\sqtC{1}\leq \sqtC{2}}\int_0^1 (\tilde \U_1-\tilde \U_2)\tilde \U_{2,\eta}(t,\eta)\\
&\qquad\qquad\times\int_0^\eta (e^{-\frac{1}{\sqtC{2}}(\tilde \Y_2(t,\eta)-\tilde \Y_2(t,\theta))}-e^{-\frac{1}{\sqtC{1}}(\tilde \Y_2(t,\eta)-\tilde \Y_2(t,\theta))})\tilde \U_2(t,\theta) d\theta d\eta\\
& \quad +\frac{\ma^5}{\A^5} \mathbbm{1}_{\sqtC{2}<\sqtC{1}}\int_0^1 (\tilde \U_1-\tilde \U_2)\tilde \U_{1,\eta}(t,\eta)\\
&\qquad\qquad\times\int_0^\eta (e^{-\frac{1}{\sqtC{2}}(\tilde \Y_1(t,\eta)-\tilde \Y_1(t,\theta))}-e^{-\frac{1}{\sqtC{1}}(\tilde \Y_1(t,\eta)-\tilde \Y_1(t,\theta))})\tilde \U_1(t,\theta) d\theta d\eta\\
&\quad+ \frac{\ma^5}{\A^6}\int_0^{1}(\tilde\U_1-\tilde\U_2)\tilde\U_{2,\eta}(t, \eta)\\
&\qquad\qquad\times\int_0^{\eta}\big( e^{-\frac{1}{\ma}( \tilde\Y_2(t,\eta)-\tilde\Y_2(t,\theta))}-e^{-\frac{1}{\ma}( \tilde\Y_1(t,\eta)-\tilde\Y_1(t,\theta))}\big)\tilde\U_2\mathbbm{1}_{B(\eta)}(t,\theta)d\theta d\eta\\
&\quad+ \frac{\ma^5}{\A^6}\int_0^1(\tilde\U_1-\tilde\U_2)\tilde\U_{1,\eta}(t, \eta)\\
&\qquad\qquad\times\int_0^{\eta}\big( e^{-\frac{1}{\ma}( \tilde\Y_2(t,\eta)-\tilde\Y_2(t,\theta))}-e^{-\frac{1}{\ma}( \tilde\Y_1(t,\eta)-\tilde\Y_1(t,\theta))}\big)\tilde\U_1\mathbbm{1}_{B(\eta)^c}(t,\theta)d\theta d\eta\\
&\quad+ \frac{\ma^5}{\A^6}\int_0^1(\tilde\U_1-\tilde\U_2)\tilde\U_{2,\eta}(t, \eta)\\
&\qquad\qquad\times\int_0^{\eta}\min_j(e^{-\frac{1}{\ma}( \tilde\Y_j(t,\eta)-\tilde\Y_j(t,\theta))})(\tilde\U_2^+ -\tilde\U_1^+)
\mathbbm{1}_{\tilde\U_1^+ \le\tilde\U_2^+}(t,\theta)d\theta d\eta\\
&\quad+ \frac{\ma^5}{\A^6}\int_0^1(\tilde\U_1-\tilde\U_2)\tilde\U_{2,\eta}(t, \eta)\\
&\qquad\qquad\times\int_0^{\eta}\min_j(e^{-\frac{1}{\ma}( \tilde\Y_j(t,\eta)-\tilde\Y_j(t,\theta))})(\tilde\U_2^- -\tilde\U_1^-)
\mathbbm{1}_{\tilde\U_2^- \le\tilde\U_1^-}(t,\theta)d\theta d\eta\\
&\quad+ \frac{\ma^5}{\A^6}\int_0^1(\tilde\U_1-\tilde\U_2)\tilde\U_{1,\eta}(t, \eta)\\
&\qquad\qquad\times\int_0^{\eta}\min_j(e^{-\frac{1}{\ma}( \tilde\Y_j(t,\eta)-\tilde\Y_j(t,\theta))})(\tilde\U_2^+ -\tilde\U_1^+)
\mathbbm{1}_{\tilde\U_2^+ <\tilde\U_1^+}(t,\theta)d\theta d\eta\\
&\quad+ \frac{\ma^5}{\A^6}\int_0^1(\tilde\U_1-\tilde\U_2)\tilde\U_{1,\eta}(t, \eta)\\
&\qquad\qquad\times\int_0^{\eta}\min_j(e^{-\frac{1}{\ma}( \tilde\Y_j(t,\eta)-\tilde\Y_j(t,\theta))})(\tilde\U_2^- -\tilde\U_1^-)
\mathbbm{1}_{\tilde\U_1^- <\tilde\U_2^-}(t,\theta)d\theta d\eta\\
&\quad- \frac{\ma^5}{\A^6}\int_0^1(\tilde\U_1-\tilde\U_2)(\tilde\U_1-\tilde\U_2)_\eta(t, \eta)\\
&\qquad\qquad\times\int_0^{\eta}\min_j(e^{-\frac{1}{\ma}( \tilde\Y_j(t,\eta)-\tilde\Y_j(t,\theta))})\min_j(\tilde\U_j^+)(t,\theta)d\theta d\eta\\
&\quad- \frac{\ma^5}{\A^6}\int_0^1(\tilde\U_1-\tilde\U_2)(\tilde\U_1-\tilde\U_2)_\eta(t, \eta)\\
&\qquad\qquad\times\int_0^{\eta}\min_j(e^{-\frac{1}{\ma}( \tilde\Y_j(t,\eta)-\tilde\Y_j(t,\theta))})\max_j(\tilde\U_j^-)(t,\theta)d\theta d\eta\\
&=W_{61}+W_{62}+W_{63}+W_{64}+W_{65}+W_{66}+W_{67}^{\pm} +W_{68}^{\pm} +W_{69}^{\pm}.
\end{align*}
Here we go again!  The terms $W_{61}$ and $W_{62}$:
\begin{align*}
\abs{W_{61}}&=\frac{\sqtC{2}^5-\sqtC{1}^5}{\A^6}\mathbbm{1}_{\sqtC{1}\le \sqtC{2}} \vert\int_0^1(\tilde\U_1-\tilde\U_2)\tilde\U_{2,\eta}(t, \eta)\\
&\qquad\qquad\qquad\qquad\qquad\qquad\times\int_0^{\eta} e^{-\frac{1}{\sqtC{2}}( \tilde\Y_2(t,\eta)-\tilde\Y_2(t,\theta))}\tilde\U_2(t,\theta)d\theta d\eta\vert\\
&\le\frac{\sqtC{2}^5-\sqtC{1}^5}{\A^6}\mathbbm{1}_{\sqtC{1}\le \sqtC{2}} \int_0^1\vert\tilde\U_1-\tilde\U_2\vert\vert\tilde\U_{2,\eta}\vert(t, \eta)\\
&\qquad\qquad\qquad\qquad\qquad\times\Big(\int_0^{\eta} e^{-\frac{1}{\sqtC{2}}( \tilde\Y_2(t,\eta)-\tilde\Y_2(t,\theta))}\tilde\U_2^2(t,\theta)d\theta \Big)^{1/2} \\
& \qquad \qquad \qquad \qquad \qquad \qquad  \times\Big(\int_0^{\eta} e^{-\frac{1}{\sqtC{2}}( \tilde\Y_2(t,\eta)-\tilde\Y_2(t,\theta))}d\theta\Big)^{1/2}d\eta\vert\\
& \leq \frac{\sqrt{6}5}{\A^2}\abs{\sqtC{2}-\sqtC{1}}\int_0^1\vert\tilde\U_1-\tilde\U_2\vert\sqP{2}\vert\tilde\U_{2,\eta}\vert(t, \eta)d\eta\\
& \leq \sqrt{3}5\A^2\vert \sqtC{2}-\sqtC{1}\vert\norm{\tilde\U_1-\tilde\U_2}\\
&\le \bigO(1)\big(\abs{\sqtC{1}-\sqtC{2}}^2+\norm{\tilde\U_1-\tilde\U_2}^2 \big),
\end{align*}
using \eqref{eq:all_PestimatesA} and \eqref{eq:all_estimatesQ}.

The terms $W_{63}$ and $W_{64}$:
\begin{align*}
\abs{W_{63}}& \leq \frac{4\ma^4}{\A^6e}\int_0^1 \vert \tilde \U_1-\tilde \U_2\vert \vert \tilde \U_{2,\eta}\vert (t,\eta) \\&\qquad\qquad\qquad\qquad\times\int_0^\eta e^{-\frac{3}{4\sqtC{2}}(\tilde \Y_2(t,\eta)-\tilde \Y_2(t,\theta))} \vert \tilde \U_2\vert(t,\theta) d\theta d\eta \vert \sqtC{1}-\sqtC{2}\vert\\
& \leq \frac{4}{\A^2e} \int_0^1 \vert \tilde \U_1-\tilde \U_2\vert \vert \tilde \U_{2,\eta}\vert (t,\eta) \Big(\int_0^\eta e^{-\frac{1}{\sqtC{2}}(\tilde \Y_2(t,\eta)-\tilde \Y_2(t,\theta))}\tilde \U_2^2(t,\theta)d\theta\Big)^{1/2}\\
& \qquad \qquad \qquad \times \Big(\int_0^\eta e^{-\frac{1}{2\sqtC{2}}(\tilde \Y_2(t,\eta)-\tilde \Y_2(t,\theta))}d\theta\Big)^{1/2} d\eta \vert \sqtC{1}-\sqtC{2}\vert\\
& \leq \frac{4\sqrt{6}}{\A^2e}\int_0^1 \vert \tilde \U_1-\tilde \U_2\vert \sqP{2}\vert \tilde \U_{2,\eta}\vert(t,\eta) d\eta \vert \sqtC{1}-\sqtC{2}\vert\\
& \leq \bigO(1)(\norm{\tilde \U_1-\tilde \U_2}^2 +\vert \sqtC{1}-\sqtC{2}\vert^2).
\end{align*}

Next comes the terms $W_{65}$ and $W_{66}$:
\begin{align*}
\abs{W_{65}}&=\frac{\ma^5}{\A^6}\vert\int_0^1(\tilde\U_1-\tilde\U_2)\tilde\U_{2,\eta}(t, \eta)\\
&\qquad\qquad\times\int_0^{\eta}\big( e^{-\frac{1}{\ma}( \tilde\Y_2(t,\eta)-\tilde\Y_2(t,\theta))}-e^{-\frac{1}{\ma}( \tilde\Y_1(t,\eta)-\tilde\Y_1(t,\theta))}\big)\tilde\U_2\mathbbm{1}_{B(\eta)}(t,\theta)d\theta d\eta\vert\\
& \le \frac{1}{\A^2}\int_0^1\vert\tilde\U_1-\tilde\U_2\vert\vert\tilde\U_{2,\eta}\vert(t, \eta)\int_0^{\eta}\big(\vert\tilde\Y_2(t,\eta)-\tilde\Y_1(t,\eta)\vert+\vert \tilde\Y_2(t,\theta)-\tilde\Y_1(t,\theta)\vert\big) \\
&\qquad\qquad\times e^{-\frac{1}{\ma}( \tilde\Y_2(t,\eta)-\tilde\Y_2(t,\theta))}\vert\tilde\U_2\vert(t,\theta)d\theta d\eta\\
&\le \frac{1}{\A^2}\int_0^1\vert\tilde\U_1-\tilde\U_2\vert\vert\tilde\Y_2-\tilde\Y_1\vert \vert\tilde\U_{2,\eta}\vert(t, \eta)\int_0^{\eta} e^{-\frac{1}{\ma}( \tilde\Y_2(t,\eta)-\tilde\Y_2(t,\theta))}\vert\tilde\U_2\vert(t,\theta)d\theta d\eta\\
&\quad+\frac{1}{\A^2}\int_0^1\vert\tilde\U_1-\tilde\U_2\vert\vert\tilde\U_{2,\eta}\vert(t, \eta)\int_0^{\eta} e^{-\frac{1}{\ma}( \tilde\Y_2(t,\eta)-\tilde\Y_2(t,\theta))}\vert \tilde\Y_2-\tilde\Y_1\vert\,\vert\tilde\U_2\vert(t,\theta)d\theta d\eta\\
& \leq \frac{1}{\A^2}\int_0^1\vert\tilde\U_1-\tilde\U_2\vert\vert\tilde\Y_2-\tilde\Y_1\vert \vert\tilde\U_{2,\eta}\vert(t, \eta) \\
&\qquad\qquad\qquad\qquad\times\Big(\int_0^{\eta} e^{-\frac{1}{\sqtC{2}}( \tilde\Y_2(t,\eta)-\tilde\Y_2(t,\theta))}\tilde\U_2^2(t,\theta)d\theta \Big)^{1/2}d\eta\\
&\quad+\frac{1}{\A^2}\int_0^1\vert\tilde\U_1-\tilde\U_2\vert\vert\tilde\U_{2,\eta}\vert(t, \eta)\\
&\qquad\qquad\qquad\qquad\times\Big(\int_0^{\eta} e^{-\frac{1}{\sqtC{2}}( \tilde\Y_2(t,\eta)-\tilde\Y_2(t,\theta))}\tilde\U_2^2(t,\theta)d\theta\Big)^{1/2} d\eta \norm{\tilde\Y_1-\tilde\Y_2}\\
& \leq \frac{\sqrt{6}}{\A^2}\int_0^1\vert\tilde\U_1-\tilde\U_2\vert\vert\tilde\Y_2-\tilde\Y_1\vert \sqP{2}\vert\tilde\U_{2,\eta}\vert(t, \eta)d\eta\\
&\quad+\frac{\sqrt{6}}{\A^2}\int_0^1\vert\tilde\U_1-\tilde\U_2\vert\sqP{2}\vert\tilde\U_{2,\eta}\vert(t, \eta) d\eta \norm{\tilde\Y_1-\tilde\Y_2}\\
&\leq \bigO(1)\norm{\tilde\U_1-\tilde\U_2}\norm{\tilde\Y_1-\tilde\Y_2}\\
& \le \bigO(1)\Big(\norm{\tilde\U_1-\tilde\U_2}^2+\norm{\tilde\Y_1-\tilde\Y_2}^2\Big),
\end{align*}
using \eqref{Diff:Exp}, \eqref{eq:all_PestimatesA}, and \eqref{eq:all_estimatesQ}.

The terms $W_{67}^\pm$ and $W_{68}^\pm$ have a similar structure:
\begin{align*}
\abs{W_{67}^+}&= \frac{\ma^5}{\A^6}\vert\int_0^1(\tilde\U_1-\tilde\U_2)\tilde\U_{2,\eta}(t, \eta)\\
&\qquad\qquad\qquad\times\int_0^{\eta}\min_j(e^{-\frac{1}{\ma}( \tilde\Y_j(t,\eta)-\tilde\Y_j(t,\theta))})(\tilde\U_2^+ -\tilde\U_1^+)
\mathbbm{1}_{\tilde\U_1^+ \le\tilde\U_2^+}(t,\theta)d\theta d\eta\vert\\
&\le \norm{\tilde\U_1-\tilde\U_2}^2 +\frac{1}{\A^2}\int_0^1\tilde\U_{2,\eta}^2(t, \eta)\\
&\qquad\qquad\qquad \times\Big(\int_0^{\eta}e^{-\frac{1}{\sqtC{2}}( \tilde\Y_2(t,\eta)-\tilde\Y_2(t,\theta))}(\tilde\U_2^+ -\tilde\U_1^+)
(t,\theta)d\theta\Big)^2 d\eta\\
&\le \norm{\tilde\U_1-\tilde\U_2}^2 +\A\int_0^1\tilde\Y_{2,\eta}(t, \eta)\\
& \qquad \qquad \qquad\times \Big(\int_0^\eta e{-\frac{1}{\sqtC{2}}(\tilde \Y_2(t,\eta)-\tilde \Y_2(t,\theta))}(\tilde \U_2-\tilde \U_1)^2(t,\theta) d\theta\Big)d\eta\\
&\le \norm{\tilde\U_1-\tilde\U_2}^2 \\
&\quad+\A\int_0^1\tilde\Y_{2,\eta}(t, \eta)e^{-\frac{1}{\sqtC{2}}\tilde \Y_2(t,\eta)}\Big(\int_0^\eta e^{\frac{1}{\sqtC{2}}\tilde \Y_2(t,\theta)}(\tilde \U_1-\tilde \U_2)^2(t,\theta) d\theta\Big) d\eta\\
& \leq \bigO(1)\norm{\tilde \U_1-\tilde \U_2}^2,
\end{align*}
using \eqref{eq:all_estimatesH}.

Finally, the terms $W_{69}^\pm$:
\begin{align*}
\abs{W_{69}^+}&=\frac{\ma^5}{\A^6}\vert\int_0^1(\tilde\U_1-\tilde\U_2)(\tilde\U_1-\tilde\U_2)_\eta(t, \eta)\\
&\qquad\qquad\qquad\times\int_0^{\eta}\min_j(e^{-\frac{1}{\ma}( \tilde\Y_j(t,\eta)-\tilde\Y_j(t,\theta))})\min_j(\tilde\U_j^+)(t,\theta)d\theta d\eta\vert\\
&\le \frac{1}{2\A}\vert \Big(  (\tilde\U_1-\tilde\U_2)^2\int_0^{\eta}\min_j(e^{-\frac{1}{\ma}( \tilde\Y_j(t,\eta)-\tilde\Y_j(t,\theta))})\min_j(\tilde\U_j^+)(t,\theta)d\theta \Big)\Big|_{\eta=0}^1\\
&\quad -\int_0^1(\tilde\U_1-\tilde\U_2)^2\frac{d}{d\eta}\int_0^{\eta}\min_j(e^{-\frac{1}{\ma}( \tilde\Y_j(t,\eta)-\tilde\Y_j(t,\theta))})\min_j(\tilde\U_j^+)(t,\theta)d\theta d\eta 
\vert \\
&\le \bigO(1)\norm{\tilde\U_1-\tilde\U_2}^2,
\end{align*}
see estimates for $\bar B_{67}^\pm$, cf.~\eqref{eq:barB67}, and Lemma \ref{lemma:6}.

%-----------------

%----------------------
\medskip
The terms $M_2$ and $M_3$ can be treated similarly. More precisely
\begin{align*}
\abs{M_2}&\le\mathbbm{1}_{\sqtC{2}\le \sqtC{1}} \big(\frac{1}{\sqtC{2}^6}-\frac{1}{\sqtC{1}^6} \big)\vert\int_0^{1}(\tilde\U_1-\tilde\U_2)\tilde\U_{2,\eta}(t, \eta)\\
&\qquad\qquad\qquad\qquad\qquad\times
\int_0^{\eta}e^{-\frac{1}{\sqtC{2}}( \tilde\Y_2(t,\eta)-\tilde\Y_2(t,\theta)) }\tilde\Q_2\tilde\U_{2,\eta}(t,\theta)d\theta d\eta\vert \\
&\le \mathbbm{1}_{\sqtC{2}\le \sqtC{1}} \frac{\vert \sqtC{1}^6-\sqtC{2}^6\vert}{\A^6\sqtC{2}^5}\int_0^{1}\vert\tilde\U_1-\tilde\U_2\vert \, \vert\tilde\U_{2,\eta}\vert(t, \eta)\\
&\qquad\qquad\qquad\qquad\qquad\times
\int_0^{\eta}e^{-\frac{1}{\sqtC{2}}( \tilde\Y_2(t,\eta)-\tilde\Y_2(t,\theta)) }\tilde\P_2\vert\tilde\U_{2,\eta}\vert(t,\theta)d\theta d\eta \\
&\le \mathbbm{1}_{\sqtC{2}\le\sqtC{1}} 6\frac{\vert \sqtC{1}-\sqtC{2}\vert}{\A\sqtC{2}^5}\norm{\tilde\U_1-\tilde\U_2}\\
&\qquad\times\Big(\int_0^{1} \tilde\U_{2,\eta}^2(t, \eta) \Big(
\int_0^{\eta}e^{-\frac{1}{\sqtC{2}}( \tilde\Y_2(t,\eta)-\tilde\Y_2(t,\theta)) }\tilde\P_2\vert\tilde\U_{2,\eta}\vert(t,\theta)d\theta\Big)^2 d\eta\Big)^{1/2} \\
&\le  \mathbbm{1}_{\sqtC{2}\le \sqtC{1}} 6\frac{\vert \sqtC{1}-\sqtC{2}\vert}{\A\sqtC{2}^6}\norm{\tilde\U_1-\tilde\U_2}\\
&\qquad\qquad\qquad\times\Big(\int_0^{1}  \tilde\U_{2,\eta}^2(t, \eta) \Big(\int_0^{\eta}e^{-\frac{1}{\sqtC{2}}( \tilde\Y_2(t,\eta)-\tilde\Y_2(t,\theta)) }\tilde\P_2^2\tilde\Y_{2,\eta}(t,\theta)d\theta\Big) \\
&\qquad\qquad\qquad\qquad\qquad\qquad\times\Big(\int_0^{\eta}e^{-\frac{1}{\sqtC{2}}( \tilde\Y_2(t,\eta)-\tilde\Y_2(t,\theta)) }\tilde\Henergy_{2,\eta}(t,\theta)d\theta\Big)
    d\eta\Big)^{1/2} \\
&\le \mathbbm{1}_{\sqtC{2}\le \sqtC{1}}6\sqrt{6}\frac{\vert\sqtC{1}-\sqtC{2}\vert}{\A\sqtC{2}^3}\norm{\tilde\U_1-\tilde\U_2}\Big(\int_0^{1} \tilde\U_{2,\eta}^2\tilde\P_2^2(t, \eta)
    d\eta\Big)^{1/2}     \\
&\le \bigO(1) \big(\norm{\tilde\U_1-\tilde\U_2}^2+\vert \sqtC{1}-\sqtC{2}\vert^2 \big).    
\end{align*}
Here we used 
\begin{align*}
\tilde\P_2\vert\tilde\U_{2,\eta}\vert& \leq \frac{1}{\sqtC{2}}\tilde\P_2\sqrt{\tilde\Y_{2,\eta}\tilde\Henergy_{2,\eta}},
\end{align*}
cf.~\eqref{eq:all_estimatesM}, as well as \eqref{eq:all_estimatesA} and \eqref{eq:all_estimatesQ}. In addition, we applied  \eqref{eq:all_PestimatesB} and \eqref{eq:all_PestimatesD}.

We have shown the anticipated result.
\begin{lemma}\label{lemlipu}
Let $\tilde\U_i$ be two solutions of \eqref{eq:Ulip}.  Then we have
\begin{align*}
\frac{d}{dt} &\norm{\tilde\U_1-\tilde\U_2}^2\\&\leq \bigO(1)\big(\norm{\tilde\Y_1-\tilde\Y_2}^2+\norm{\tilde\U_1-\tilde\U_2}^2+\norm{\sqP{1}-\sqP{2}}^2+ \vert \sqtC{1}-\sqtC{2}\vert ^2\big),
\end{align*}
where $\bigO(1)$ denotes some constant which only depends on $\A=\max_j(\sqtC{j})$ and which remains bounded as $\A\to 0$.
\end{lemma}

%-------------------------------------------------

\subsection{Lipschitz estimates for $\tilde \P$ (or $\tilde\P^{1/2}$)}\label{subsec:lipp}
From the system of differential equations we recall
\begin{equation}\label{eq:Plip}
\big(\sqP{i}\big)_t+ (\frac23 \frac{1}{\sqtC{i}^5} \tilde \U_i^3+\frac{1}{\sqtC{i}^6}\tilde \So_i)\big(\sqP{i}\big)_\eta=\frac{1}{2\sqtC{i}^2}\frac{\tilde \Q_i\tilde\U_i}{\sqP{i}}+\frac{1}{2\sqtC{i}^3}\frac{\tilde \R_i}{\sqP{i}},
\end{equation}
where
\begin{subequations}
\begin{align}
\tilde \Q_i(t,\eta)
& = -\frac14 \int_0^1 \sign(\eta-\theta)e^{-\frac{1}{\sqtC{i}}\vert \tilde \Y_i(t,\eta)-\tilde \Y_i(t,\theta)\vert } (2(\tilde \U_i^2-\tilde \P_i)\tilde \Y_{i,\eta}(t,\theta)+\sqtC{i}^5)d\theta,\label{eq:Q}
\\
\tilde \So_i(t,\eta)
&  = \int_0^1 e^{-\frac{1}{\sqtC{i}}\vert \tilde \Y_i(t,\eta)-\tilde \Y_i(t,\theta)\vert} (\frac23 \tilde\U_i^3\tilde\Y_{i,\eta}-\tilde \Q_i\tilde \U_{i,\eta}-2\tilde\P_i\tilde \U_i\tilde\Y_{i,\eta})(t,\theta)d\theta, \label{eq:S}
\\[3mm]
\tilde \R_i(t,\eta)& = \frac14 \int_0^1 \sign(\eta-\theta)e^{-\frac{1}{\sqtC{i}}\vert \tilde\Y_i(t,\eta)-\tilde\Y_i(t,\theta)\vert} (\frac23\sqtC{i} \tilde \U_i^3\tilde \Y_{i,\eta}+\sqtC{i}^6\tilde\U_i)(t,\theta)d\theta\nn \\
& \quad -\frac12 \int_0^1 e^{-\frac{1}{\sqtC{i}}\vert \tilde \Y_i(t,\eta)-\tilde \Y_i(t,\theta)\vert} \tilde \U_i\tilde\Q_i\tilde \Y_{i,\eta}(t,\theta) d\theta.\label{eq:Req}
\end{align}
\end{subequations}

Thus we have 
\begin{align*} \nn
\frac{d}{dt}&\int_0^{1}\Big(\sqP{1}-\sqP{2}\Big)^2(t,\eta)d\eta\\  \nn
& = 2\int_0^{1}\Big(\sqP{1}- \sqP{2}\Big)\Big(\big(\sqP{1}\big)_t-\big(\sqP{2}\big)_t\Big)(t,\eta)d\eta\\ \nn
& =  -\frac23\int_0^{1} \Big(\sqP{1}-\sqP{2}\Big)
\Big(
\frac{1}{\sqtC{1}^5} \tilde \U_1^3\big(\sqP{1}\big)_\eta-
\frac{1}{\sqtC{2}^5}  \tilde \U_2^3\big(\sqP{2}\big)_\eta
\Big)(t,\eta)d\eta\\ \nn
& \quad - \int_0^{1} \Big(\sqP{1}-\sqP{2}\Big)
\Big(
\frac{1}{\sqtC{1}^6} \tilde \So_1\big(\sqP{1}\big)_\eta-
\frac{1}{\sqtC{2}^6} \tilde \So_2\big(\sqP{2}\big)_\eta
\Big)(t,\eta)d\eta\\ \nn 
& \quad + \frac12\int_0^{1}\Big(\sqP{1}-\sqP{2}\Big)
\Big(\frac{1}{\sqtC{1}^2}\frac{\tilde \Q_1\tilde\U_1}{\sqP{1}}-\frac{1}{\sqtC{2}^2}\frac{\tilde \Q_2\tilde\U_2}{\sqP{2}}\Big)(t,\eta)d\eta\\ \nn
& \quad +\frac12 \int_0^{1}\Big(\sqP{1}-\sqP{2}\Big)
\Big(\frac{1}{\sqtC{1}^3}\frac{\tilde \R_1}{\sqP{1}}-\frac{1}{\sqtC{2}^3}\frac{\tilde \R_2}{\sqP{2}}\Big)(t,\eta)d\eta\\
& = \frac23 I_1+I_2+\frac12 I_3+\frac12 I_4.
\end{align*}
We will estimate each of these terms, yielding that 
\begin{align*}
\frac{d}{dt} &\norm{\sqP{1}-\sqP{2}}^2\\&\leq \bigO(1)\Big(\norm{\tilde\Y_1-\tilde\Y_2}^2+\norm{\tilde\U_1-\tilde\U_2}^2+\norm{\sqP{1}-\sqP{2}}^2+ \vert\sqtC{1}-\sqtC{2}\vert ^2\Big),
\end{align*}
where $\bigO(1)$ denotes some constant which only depends on $\A=\max_j(\sqtC{j})$ and which remains bounded as $\A\to 0$.

\textit{The term $I_1$:} Here we do as follows:
\begin{align*}
I_1&= \int_0^{1} \Big(\sqP{1}-\sqP{2}\Big)
\Big(
\frac{1}{\sqtC{2}^5} \tilde \U_2^3\big(\sqP{2}\big)_\eta-
\frac{1}{\sqtC{1}^5}  \tilde \U_1^3\big(\sqP{1}\big)_\eta
\Big)(t,\eta)d\eta \\
&=\frac1{\A^5}\int_0^{1} \Big(\sqP{1}-\sqP{2}\Big)
\Big( \tilde \U_2^3\sqDP{2}-\tilde \U_1^3\sqDP{1}\Big)(t,\eta)d\eta \\
&\quad+\frac{\sqtC{1}^5-\sqtC{2}^5}{\sqtC{1}^5\sqtC{2}^5}\int_0^{1}\Big(\sqP{1}-\sqP{2}\Big)\\
&\qquad\qquad\times\Big(\tilde \U_2^3\big(\sqP{2}\big)_\eta
\mathbbm{1}_{\sqtC{2}< \sqtC{1}}+\tilde \U_1^3\big(\sqP{1}\big)_\eta\mathbbm{1}_{\sqtC{1}\le \sqtC{2}} \Big)(t,\eta)d\eta \\
&= \frac1{\A^5}\int_0^{1} \Big(\sqP{1}-\sqP{2}\Big)
\Big( \tilde \U_2-\tilde \U_1\Big)\tilde \U_2^2\big(\sqP{2}\big)_\eta(t,\eta)d\eta \\
&\quad+  \frac1{\A^5}\int_0^{1} \Big(\sqP{1}-\sqP{2}\Big)
\Big( \tilde \U_2^2\big(\sqP{2}\big)_\eta-\tilde \U_1^2\big(\sqP{1}\big)_\eta\Big)\tilde \U_1(t,\eta)d\eta \\
&\quad+\frac{\sqtC{1}^5-\sqtC{2}^5}{\sqtC{1}^5\sqtC{2}^5}\int_0^{1}\Big(\sqP{1}-\sqP{2}\Big)\\
&\qquad\qquad\times\Big(\tilde \U_2^3\big(\sqP{2}\big)_\eta
\mathbbm{1}_{\sqtC{2}<\sqtC{1}}+\tilde \U_1^3\big(\sqP{1}\big)_\eta\mathbbm{1}_{\sqtC{1}\le \sqtC{2}} \Big)(t,\eta)d\eta \\
&= \frac1{\A^5}\int_0^{1} \Big(\sqP{1}-\sqP{2}\Big)
\Big( \tilde \U_2-\tilde \U_1\Big)\tilde \U_2^2\big(\sqP{2}\big)_\eta(t,\eta)d\eta \\
&\quad+  \frac1{\A^5}\int_0^{1} \Big(\sqP{1}-\sqP{2}\Big)
\Big( \tilde \U_2^2-\tilde \U_1^2\Big)\big(\sqP{2}\big)_\eta\tilde \U_1\mathbbm{1}_{\tilde \U_1^2\le \tilde \U_2^2}(t,\eta)d\eta \\
&\quad+  \frac1{\A^5}\int_0^{1} \Big(\sqP{1}-\sqP{2}\Big)
\Big( \tilde \U_2^2-\tilde \U_1^2\Big)\big(\sqP{1}\big)_\eta\tilde \U_1\mathbbm{1}_{\tilde \U_2^2< \tilde \U_1^2}(t,\eta)d\eta \\
&\quad+  \frac1{\A^5}\int_0^{1} \Big(\sqP{1}-\sqP{2}\Big)
\Big(\big(\sqP{2}\big)_\eta-\big(\sqP{1}\big)_\eta\Big)\tilde \U_1\min_j(\tilde \U_j^2)(t,\eta)d\eta \\
&\quad+\frac{\sqtC{1}^5-\sqtC{2}^5}{\sqtC{1}^5\sqtC{2}^5}\int_0^{1}\Big(\sqP{1}-\sqP{2}\Big)\\
&\qquad\qquad\times\Big(\tilde \U_1^3\big(\sqP{1}\big)_\eta
\mathbbm{1}_{\sqtC{1}\leq \sqtC{2}}+\tilde \U_2^3\big(\sqP{2}\big)_\eta\mathbbm{1}_{\sqtC{2}< \sqtC{1}} \Big)(t,\eta)d\eta \\
&=I_{11}+I_{12}+I_{13}+I_{14}+I_{15}.
\end{align*}
We first estimate
\begin{align*}
\abs{\big(\tilde\P_i^{1/2}\big)_\eta}=\frac{\abs{\tilde \P_{i,\eta}}}{2\tilde\P_i^{1/2}}= \frac{\abs{\tilde \Q_i \tilde \Y_{i,\eta}}}{2\sqtC{i}^2\tilde\P_i^{1/2}}
\le \frac1{2\sqtC{i}} \tilde\P_i^{1/2}\tilde \Y_{i,\eta},
\end{align*}
and thus
\begin{align*}
\abs{\tilde\U_i\tilde\P_i^{1/2}\tilde\Y_{i,\eta}}\le \frac12(\tilde\U_i^2\tilde\Y_{i,\eta}+\tilde\P_i\tilde\Y_{i,\eta})\le \frac34 \sqtC{i}^5.
\end{align*}
We start with the term $I_{11}$:
\begin{align*}
\abs{I_{11}}&\le \bigO(1) \int_0^{1} \abs{\sqP{1}-\sqP{2}}\, 
\abs{\tilde \U_2-\tilde \U_1}(t,\eta)d\eta\\
& \le \bigO(1)\Big(\norm{\sqP{1}-\sqP{2}}^2+
\norm{\tilde \U_2-\tilde \U_1}^2\Big). 
\end{align*}
Subsequently (and similar for $I_{13}$)
\begin{align*}
\abs{I_{12}}&\le \frac1{\A^5}\int_0^{1} \abs{\sqP{1}-\sqP{2}}\,
\abs{\tilde \U_2-\tilde \U_1}\, \abs{\tilde \U_2+\tilde \U_1}\,\abs{\big(\sqP{2}\big)_\eta\tilde \U_1}
\mathbbm{1}_{\tilde \U_1^2\le \tilde \U_2^2}(t,\eta)d\eta \\
&\le \frac{2}{\A^5}\int_0^{1} \abs{\sqP{1}-\sqP{2}}\,
\abs{\tilde \U_2-\tilde \U_1}\tilde \U_2^2\big(\sqP{2}\big)_\eta(t,\eta)
d\eta \\
&\le \bigO(1)(\norm{\sqP{1}-\sqP{2}}^2+
\norm{\tilde \U_2-\tilde \U_1}^2).
\end{align*}
The term  $I_{14}$ requires more estimates (see \eqref{eq:470}):
\begin{align*}
\abs{I_{14}}&= \frac1{A^5} \abs{\int_0^{1} \Big(\sqP{1}-\sqP{2}\Big)
\Big(\big(\sqP{2}\big)_\eta-\big(\sqP{1}\big)_\eta\Big)\tilde \U_1\min_j(\tilde \U_j^2)(t,\eta)d\eta} \\
&=  \Big\vert\frac1{2\A^5} \Big(\sqP{1}-\sqP{2}\Big)^2\tilde \U_1\min_j(\tilde \U_j^2)(t,\eta)\Big\vert_{\eta=0}^1\\
&\quad- \frac1{2\A^5} \int_0^{1}\Big(\sqP{1}-\sqP{2}\Big)^2\Big(\frac{d}{d\eta}\tilde \U_1\min_j(\tilde \U_j^2)\Big)(t,\eta) d\eta\Big\vert\\
&=\frac1{2\A^5} \abs{\int_0^{1}\Big(\sqP{1}-\sqP{2}\Big)^2\Big(\frac{d}{d\eta}\tilde \U_1\min_j(\tilde \U_j^2)\Big)(t,\eta) d\eta} \\
&\le \bigO(1) \norm{\sqP{1}-\sqP{2}}^2.
\end{align*}
The last term can be estimated as follows:
\begin{align*}
\abs{I_{15}}& \le
\Big\vert\frac{\sqtC{1}^5-\sqtC{2}^5}{\sqtC{1}^5\sqtC{2}^5}\int_0^{1}\Big(\sqP{1}-\sqP{2}\Big)\\
&\qquad\qquad\qquad\qquad\times\Big(\tilde \U_1^3\big(\sqP{1}\big)_\eta\mathbbm{1}_{\sqtC{1}\le \sqtC{2}}+\tilde \U_2^3\big(\sqP{2}\big)_\eta
\mathbbm{1}_{\sqtC{2}< \sqtC{1}} \Big)(t,\eta)d\eta\Big\vert \\
&\le 5\frac{\abs{\sqtC{1}-\sqtC{2}} \A^4}{\sqtC{1}^5\sqtC{2}^5}\\
&\qquad\times\int_0^{1}\Big(\sqP{1}-\sqP{2}\Big)\Big(\vert \tilde \U_1^3\big(\sqP{1}\big)_\eta\vert
\mathbbm{1}_{\sqtC{1}\leq \sqtC{2}}+\vert\tilde \U_2^3\big(\sqP{2}\big)_\eta\vert\mathbbm{1}_{\sqtC{2}< \sqtC{1}} \Big)(t,\eta)d\eta.
\end{align*}
We consider, in particular, the term
\begin{align*}
& \frac{\A^4}{\sqtC{1}^5\sqtC{2}^5}\abs{\tilde \U_1^3\big(\sqP{1}\big)_\eta}\mathbbm{1}_{\sqtC{1}\leq \sqtC{2}}
\le\frac{\A^4}{\sqtC{1}^5\sqtC{2}^5}\frac{3}{16}\sqtC{1}^8\mathbbm{1}_{\sqtC{1}\leq \sqtC{2}}\le \frac{3}{16}\A^2.
\end{align*}
Thus
\begin{align*}
\abs{I_{15}}&\le \frac{30}{16}\A^2 \abs{\sqtC{1}-\sqtC{2}} \int_0^{1}\abs{\sqP{1}-\sqP{2}}(t,\eta)d\eta\\
&\le \bigO(1)\big( \abs{\sqtC{1}-\sqtC{2}}^2+ \norm{\sqP{1}-\sqP{2}}^2\big).
\end{align*}

\medskip
\textit{The term $I_2$:}  Recall that the term reads as follows:
\begin{align*}
I_{2}& =\int_0^{1} \Big(\sqP{1}-\sqP{2}\Big)
\Big(
\frac{1}{\sqtC{2}^6} \tilde \So_2\big(\sqP{2}\big)_\eta-
\frac{1}{\sqtC{1}^6} \tilde \So_1\big(\sqP{1}\big)_\eta
\Big)(t,\eta)d\eta.
\end{align*}
Here we can follow the estimates of the term $I_3$, cf.~\eqref{eq:I1-3} and \eqref{eq:I3_delt}..  
Here we go:  Define
\begin{align*}
I_{21}&= \int_0^{1} (\sqP{1}-\sqP{2})(t, \eta)\Big(\frac{1}{\sqtC{2}^6}\big(\sqP{2}\big)_\eta(t, \eta)\int_0^{1}
 e^{-\frac{1}{\sqtC{2}}\vert \tilde\Y_2(t,\eta)-\tilde\Y_2(t,\theta)\vert }\tilde \U_2^3\tilde\Y_{2,\eta}(t,\theta)d\theta\\
&\qquad\qquad\qquad -\frac{1}{\sqtC{1}^6}\big(\sqP{1}\big)_\eta(t, \eta)\int_0^{1} e^{-\frac{1}{\sqtC{1}}\vert \tilde\Y_1(t,\eta)-\tilde\Y_1(t,\theta)\vert } \tilde \U_1^3\tilde\Y_{1,\eta}(t,\theta)d\theta\Big)d\eta,
\\[2mm]
I_{22}&=\int_0^{1} (\sqP{1}- \sqP{2})(t, \eta)\\
&\qquad\times\Big(\frac{1}{\sqtC{2}^6}\big(\sqP{2}\big)_\eta(t, \eta)\int_0^{1}
 e^{-\frac{1}{\sqtC{2}}\vert \tilde\Y_2(t,\eta)-\tilde\Y_2(t,\theta)\vert }\tilde\P_2\tilde \U_2\tilde\Y_{2,\eta}(t,\theta)d\theta\\
&\qquad\qquad\qquad- \frac{1}{\sqtC{1}^6}\big(\sqP{1}\big)_\eta(t, \eta)\int_0^{1} e^{-\frac{1}{\sqtC{1}}\vert \tilde\Y_1(t,\eta)-\tilde\Y_1(t,\theta)\vert } 
\tilde\P_1\tilde \U_1\tilde\Y_{1,\eta}(t,\theta)d\theta\Big)d\eta,\\[2mm]
I_{23}&=\int_0^{1} (\sqP{1}-\sqP{2})(t, \eta)\\
&\qquad\times\Big(\frac{1}{\sqtC{2}^6}\big(\sqP{2}\big)_\eta(t, \eta)\int_0^{1}
 e^{-\frac{1}{\sqtC{2}}\vert \tilde\Y_2(t,\eta)-\tilde\Y_2(t,\theta)\vert }\tilde\Q_2\tilde \U_{2,\eta}(t,\theta)d\theta\\
&\qquad\qquad\qquad- \frac{1}{\sqtC{1}^6}\big(\sqP{1}\big)_\eta(t, \eta)\int_0^{1} e^{-\frac{1}{\sqtC{1}}\vert \tilde\Y_1(t,\eta)-\tilde\Y_1(t,\theta)\vert } 
\tilde\Q_1\tilde \U_{1,\eta}(t,\theta)d\theta\Big)d\eta.
\end{align*}
Using the definition \eqref{eq:S} of $\tilde \So$, we can write
\begin{align*}
I_2=\frac23 I_{21}-2I_{22}-I_{23}.
\end{align*}
We commence with the estimates for  $I_{21}$:
\begin{align*}
I_{21}&= \int_0^{1} (\sqP{1}-\sqP{2})(t,\eta)\\ 
& \qquad \times\Big(\frac{1}{\sqtC{2}^6}\big(\sqP{2}\big)_\eta(t,\eta)\int_0^{1} e^{-\frac{1}{\sqtC{2}}\vert \tilde\Y_2(t,\eta)-\tilde\Y_2(t,\theta)\vert}\tilde\U_2^3\tilde\Y_{2,\eta}(t,\theta)d\theta\\
&\qquad\qquad\qquad-\frac{1}{\sqtC{1}^6}\big(\sqP{1}\big)_\eta(t,\eta)\int_0^{1} e^{-\frac{1}{\sqtC{1}}\vert \tilde\Y_1(t,\eta)-\tilde\Y_1(t,\theta)\vert} \tilde\U_1^3\tilde\Y_{1,\eta}(t,\theta) d\theta\Big)d\eta\\
& \quad =  \int_0^{1} (\sqP{1}-\sqP{2})(t,\eta)\\
& \qquad \times\Big(\frac{1}{\sqtC{2}^6}\big(\sqP{2}\big)_\eta(t,\eta)\int_0^{1} e^{-\frac{1}{\sqtC{2}}\vert \tilde\Y_2(t,\eta)-\tilde\Y_2(t,\theta)\vert}(\tilde\V_2^-)^3\tilde\Y_{2,\eta}(t,\theta)d\theta\\
&\qquad\qquad\qquad-\frac{1}{\sqtC{1}^6}\big(\sqP{1}\big)_\eta(t,\eta)\int_0^{1} e^{-\frac{1}{\sqtC{1}}\vert \tilde\Y_1(t,\eta)-\tilde\Y_1(t,\theta)\vert} (\tilde\V_1^-)^3\tilde\Y_{1,\eta}(t,\theta) d\theta\Big)d\eta\\
& \qquad +  \int_0^{1} (\sqP{1}-\sqP{2})(t,\eta)\\ 
& \qquad \times\Big(\frac{1}{\sqtC{2}^6}\big(\sqP{2}\big)_\eta(t,\eta)\int_0^{1} e^{-\frac{1}{\sqtC{2}}\vert \tilde\Y_2(t,\eta)-\tilde\Y_2(t,\theta)\vert}(\tilde\V_2^+)^3\tilde\Y_{2,\eta}(t,\theta)d\theta\\
&\qquad\qquad\qquad-\frac{1}{\sqtC{1}^6}\big(\sqP{1}\big)_\eta(t,\eta)\int_0^{1} e^{-\frac{1}{\sqtC{1}}\vert \tilde\Y_1(t,\eta)-\tilde\Y_1(t,\theta)\vert} (\tilde\V_1^+)^3\tilde\Y_{1,\eta}(t,\theta) d\theta\Big)d\eta.
\end{align*}
The two terms are similar, and we only discuss the second one. Next we use the trick \eqref{eq:triks}. Thus we have to consider the term
\begin{align}\nn
& \frac{1}{\sqtC{2}^6}  \big(\sqP{2}\big)_\eta(t,\eta) \int_0^{\eta} e^{-\frac{1}{\sqtC{2}}\vert \tilde\Y_2(t,\eta)-\tilde\Y_2(t,\theta)\vert } (\tilde\V_2^+)^3\tilde\Y_{2,\eta}(t,\theta) d\theta\\ \nn
&\qquad\qquad\qquad\qquad-\frac{1}{\sqtC{1}^6}\big(\sqP{1}\big)_\eta(t,\eta)\int_0^{\eta} e^{-\frac{1}{\sqtC{1}}\vert \tilde\Y_1(t,\eta)
-\tilde\Y_1(t,\theta)\vert }(\tilde\V_1^+)^3\tilde \Y_{1,\eta}(t,\theta) d\theta\\ \nn
 & = \frac{1}{\A^6}\Big(\big(\sqP{2}\big)_\eta(t,\eta) \int_0^{\eta} e^{-\frac{1}{\sqtC{2}}(\tilde \Y_2(t,\eta)-\tilde\Y_2(t,\theta)) } (\tilde\V_2^+)^3 \tilde\Y_{2,\eta}(t,\theta) d\theta\\ \nn
 &\qquad\qquad\qquad\qquad\qquad- \big(\sqP{1}\big)_\eta(t,\eta)\int_0^{\eta} e^{-\frac{1}{\sqtC{1}}(\tilde \Y_1(t,\eta)-\tilde\Y_1(t,\theta))}(\tilde\V_1^+)^3\tilde\Y_{1,\eta}(t,\theta) d\theta\Big)\\ \nn
 & \quad + \mathbbm{1}_{\sqtC{1}\leq \sqtC{2}} (\frac{1}{\sqtC{2}^6}-\frac{1}{\sqtC{1}^6}) \big(\sqP{1}\big)_\eta(t,\eta)\int_0^\eta e^{-\frac{1}{\sqtC{1}}(\tilde\Y_1(t,\eta)-\tilde\Y_1(t,\theta))} (\tilde\V_1^+)^3\tilde\Y_{1,\eta}(t,\theta)d\theta\\ \nn
 & \quad + \mathbbm{1}_{\sqtC{2}< \sqtC{1}} (\frac{1}{\sqtC{2}^6}-\frac1{\sqtC{1}^6})\big(\sqP{2}\big)_\eta(t,\eta) \int_0^\eta e^{-\frac{1}{\sqtC{2}}(\tilde\Y_2(t,\eta)-\tilde\Y_2(t,\theta))} (\tilde\V_2^+)^3 \tilde\Y_{2,\eta} (t,\theta) d\theta\\ \nn
 &= \frac{1}{\A^6} \big(\sqP{2}\big)_\eta(t,\eta)\\ \nn
& \qquad\qquad\times \int_0^\eta e^{-\frac{1}{\sqtC{2}}(\tilde \Y_2(t,\eta)-\tilde\Y_2(t,\theta))}((\tilde\V_2^+)^3-(\tilde\V_1^+)^3)\tilde\Y_{2,\eta}(t,\theta)\mathbbm{1}_{\tilde\V_1^+\leq \tilde\V_2^+}(t,\theta)d\theta\\ \nn
 & \quad +\frac1{\A^6}\big(\sqP{1}\big)_\eta(t,\eta)\\ \nn
& \qquad\qquad\times\int_0^\eta e^{-\frac{1}{\sqtC{1}}(\tilde \Y_1(t,\eta)-\tilde \Y_1(t,\theta))}((\tilde \V_2^+)^3-(\tilde\V_1^+)^3)\tilde \Y_{1,\eta}(t,\theta) \mathbbm{1}_{\tilde\V_2^+<\tilde\V_1^+}(t,\theta) d\theta\\ \nn
 & \quad +\frac{1}{\A^6} \big(\sqP{2}\big)_\eta(t,\eta)\int_0^\eta e^{-\frac{1}{\sqtC{2}}(\tilde\Y_2(t,\eta) -\tilde \Y_2(t,\theta))}\min_j(\tilde\V_j^+)^3 \tilde \Y_{2,\eta} (t,\theta) d\theta\\ \nn
 & \quad - \frac{1}{\A^6} \big(\sqP{1}\big)_\eta (t,\eta) \int_0^\eta e^{-\frac{1}{\sqtC{1}}(\tilde \Y_1(t,\eta)-\tilde \Y_1(t,\theta))} \min_j(\tilde\V_j^+)^3\tilde\Y_{1,\eta}(t,\theta) d\theta\\ \nn
  & \quad + \mathbbm{1}_{\sqtC{1}\leq \sqtC{2}} (\frac{1}{\sqtC{2}^6}-\frac{1}{\sqtC{1}^6}) \big(\sqP{1}\big)_\eta(t,\eta)\int_0^\eta e^{-\frac{1}{\sqtC{1}}(\tilde\Y_1(t,\eta)-\tilde\Y_1(t,\theta))} (\tilde\V_1^+)^3\tilde\Y_{1,\eta}(t,\theta)d\theta\\ \nn
 & \quad + \mathbbm{1}_{\sqtC{2}< \sqtC{1}} (\frac{1}{\sqtC{2}^6}-\frac1{\sqtC{1}^6})\big(\sqP{2}\big)_\eta(t,\eta) \int_0^\eta e^{-\frac{1}{\sqtC{2}}(\tilde\Y_2(t,\eta)-\tilde\Y_2(t,\theta))} (\tilde\V_2^+)^3 \tilde\Y_{2,\eta} (t,\theta) d\theta\\ \nn
 & = \frac{1}{\A^6} \big(\sqP{2}\big)_\eta(t,\eta)\\ \nn
& \qquad\qquad\times\int_0^\eta e^{-\frac{1}{\sqtC{2}}(\tilde \Y_2(t,\eta)-\tilde\Y_2(t,\theta))}((\tilde\V_2^+)^3-(\tilde\V_1^+)^3)\tilde\Y_{2,\eta}(t,\theta)\mathbbm{1}_{\tilde\V_1^+\leq \tilde\V_2^+}(t,\theta)d\theta\\ \nn
 & \quad +\frac1{\A^6}\big(\sqP{1}\big)_\eta(t,\eta)\\ \nn
& \qquad\qquad\times\int_0^\eta e^{-\frac{1}{\sqtC{1}}(\tilde \Y_1(t,\eta)-\tilde \Y_1(t,\theta))}((\tilde  \V_2^+)^3-(\tilde \V_1^+)^3)\tilde  \Y_{1,\eta}(t,\theta) \mathbbm{1}_{\tilde \V_2^+<\tilde \V_1^+}(t,\theta) d\theta\\ \nn
 & \quad +\mathbbm{1}_{\sqtC{1}\leq \sqtC{2}}\frac{1}{\A^6}\big(\sqP{2}\big)_\eta(t,\eta)\\ \nn
& \qquad\qquad\times\int_0^\eta (e^{-\frac{1}{\sqtC{2}}(\tilde \Y_2(t,\eta)-\tilde\Y_2(t,\theta))}-e^{-\frac{1}{\sqtC{1}}(\tilde \Y_2(t,\eta)-\tilde \Y_2(t,\theta))}\min_j(\tilde\V_j^+)^3 \tilde\Y_{2,\eta} (t,\theta)d\theta d\eta\\ \nn
 & \quad +\mathbbm{1}_{\sqtC{2}< \sqtC{1}}\frac{1}{\A^6}\big(\sqP{1}\big)_\eta(t,\eta)\\ \nn
& \qquad\qquad\times\int_0^\eta (e^{-\frac{1}{\sqtC{2}}(\tilde \Y_1(t,\eta)-\tilde\Y_1(t,\theta))}-e^{-\frac{1}{\sqtC{1}}(\tilde \Y_1(t,\eta)-\tilde \Y_1(t,\theta))}\min_j(\tilde\V_j^+)^3 \tilde\Y_{1,\eta} (t,\theta)d\theta d\eta\\ \nn
 & \quad +\frac{1}{\A^6} \ \big(\sqP{2}\big)_\eta(t,\eta) (\int_0^\eta (e^{-\frac{1}{\ma}(\tilde \Y_2(t,\eta)-\tilde \Y_2(t,\theta))}-e^{-\frac{1}{\ma}(\tilde \Y_1(t,\eta)-\tilde \Y_1(t,\theta))})\\ \nn
 &\qquad\qquad\qquad\qquad\qquad\qquad\times\min_j(\tilde\V_j^+)^3 \tilde\Y_{2,\eta} (t,\theta)\mathbbm{1}_{B(\eta)}(t,\theta) d\theta)\\ \nn
 & \quad -\frac{1}{\A^6}  \big(\sqP{1}\big)_\eta (t,\eta)(\int_0^\eta (e^{-\frac{1}{\ma}(\tilde \Y_1(t,\eta)-\tilde \Y_1(t,\theta))}-e^{-\frac{1}{\ma}(\tilde \Y_2(t,\eta)-\tilde \Y_2(t,\theta))})\\ \nn
 &\qquad\qquad\qquad\qquad\qquad\qquad\times\min_j(\tilde\V_j^+)^3\tilde \Y_{1,\eta} (t,\theta) \mathbbm{1}_{B^c(\eta)}(t,\theta) d\theta)\\ \nn
 &\quad + \frac{1}{\A^6}  \big(\sqP{2}\big)_\eta (t,\eta)(\int_0^\eta \min_j(e^{-\frac{1}{\ma}(\tilde \Y_j(t,\eta)-\tilde\Y_j(t,\theta))})\min_j(\tilde\V_j^+)^3\tilde \Y_{2,\eta} (t,\theta) d\theta)\\ \nn
 & \quad -\frac{1}{\A^6} \big(\sqP{1}\big)_\eta (t,\eta)(\int_0^\eta \min_j(e^{-\frac{1}{\ma}(\tilde \Y_j(t,\eta)-\tilde\Y_j(t,\theta))})\min_j(\tilde \V_j^+)^3\tilde \Y_{1,\eta} (t,\theta) d\theta)\\ \nn
 & \quad + \mathbbm{1}_{\sqtC{1}\leq \sqtC{2}} (\frac{1}{\sqtC{2}^6}-\frac{1}{\sqtC{1}^6}) \big(\sqP{1}\big)_\eta(t,\eta)\int_0^\eta e^{-\frac{1}{\sqtC{1}}(\tilde\Y_1(t,\eta)-\tilde\Y_1(t,\theta))} (\tilde\V_1^+)^3\tilde\Y_{1,\eta}(t,\theta)d\theta\\ \nn
 & \quad + \mathbbm{1}_{\sqtC{2}< \sqtC{1}} (\frac{1}{\sqtC{2}^6}-\frac1{\sqtC{1}^6})\big(\sqP{2}\big)_\eta(t,\eta) \int_0^\eta e^{-\frac{1}{\sqtC{2}}(\tilde\Y_2(t,\eta)-\tilde\Y_2(t,\theta))} (\tilde\V_2^+)^3 \tilde\Y_{2,\eta} (t,\theta) d\theta\\ \label{eq:I21j}
 & = (I_{211}+I_{212}+I_{213}+I_{214}+I_{215}+I_{216}+ I_{217}+I_{218}+I_{219}+I_{220})(t,\eta).
 \end{align}
The same estimate works for $I_{211}$ and $I_{212}$ (cf.~\eqref{eq:J1J2}):
\begin{equation*}
\vert \int_0^1I_{211}(\sqP{1}-\sqP{2})(t,\eta)d\eta\vert \leq \norm{\sqP{1}-\sqP{2}}^2 + \int_0^1 I_{211}^2(t,\eta) d\eta,
\end{equation*}
which shows that it suffices to estimate
\begin{align*}
\int_0^1 I_{211}^2(t,\eta) d\eta&\le  \frac1{\A^{12}} \int_0^1 \big(\sqP{2}\big)_\eta^2(t,\eta) \Big(\int_0^\eta e^{-\frac{1}{\sqtC{2}}(\tilde \Y_2(t,\eta)-\tilde \Y_2(t,\theta))}\\
&\qquad\times((\tilde \V_2^+)^3-(\tilde \V_1^+)^3)\tilde \Y_{2,\eta} (t,\theta) \mathbbm{1}_{\tilde\V_1^+\leq \tilde\V_2^+} (t,\theta) d\theta\Big)^2 d\eta\\
&\le \bigO(1)\norm{\tilde \U_1-\tilde \U_2}^2,
 \end{align*}
where we used that $\big(\sqP{2}\big)_\eta^2\le \frac{1}{8}\sqtC{2}^3 \tilde\Y_{2,\eta}$. 

Regards the term $I_{213}$ (and $I_{214}$), we find
\begin{align*}
& \vert \int_0^1 I_{213}(\sqP{1}-\sqP{2})(t,\eta) d\eta\vert \\
& \quad = \mathbbm{1}_{\sqtC{1}\leq \sqtC{2}}\frac{1}{\A^6} \vert \int_0^1 
\sqP{1}-\sqP{2})\big(\sqP{2}\big)_\eta(t,\eta)\\
&\qquad\times\int_0^\eta (e^{-\frac{1}{\sqtC{2}}(\tilde \Y_2(t,\eta)-\tilde \Y_2(t,\theta))}-e^{-\frac{1}{\sqtC{1}}(\tilde \Y_2(t,\eta)-\tilde \Y_2(t,\theta))})\min_j(\tilde \V_j^+)^3\tilde \Y_{2,\eta}(t,\theta) d\theta d\eta\vert\\
& \quad \leq \frac{2\sqrt{2}\ma}{\A^6 e}  \int_0^1 \vert\sqP{1}-\sqP{2}\vert\vert \big(\sqP{2}\big)_\eta\vert(t,\eta) \\
&\qquad\qquad\qquad\times\int_0^\eta e^{-\frac{3}{4\sqtC{2}}(\tilde \Y_2(t,\eta)-\tilde \Y_2(t,\theta))} \tilde \U_2^2\tilde \Y_{2,\eta}(t,\theta) d\theta d\eta\vert \sqtC{1}-\sqtC{2}\vert\\
& \quad \leq \norm{\sqP{1}-\sqP{2}}^2 \\
& \qquad \qquad + \frac{8}{\A^{10}e^2} \int_0^1 \big(\sqP{2}\big)_\eta^2(t,\eta)\Big(\int_0^\eta e^{-\frac{1}{\sqtC{2}}(\tilde \Y_2(t,\eta)-\tilde \Y_2(t,\theta))}\tilde \U_2^2\tilde \Y_{2,\eta}(t,\theta) d\theta\Big)\\
& \qquad \qquad \qquad \qquad \times \Big(\int_0^\eta e^{-\frac{1}{2\sqtC{2}}(\tilde \Y_2(t,\eta)-\tilde \Y_2(t,\theta))} \tilde \U_2^2\tilde \Y_{2,\eta}(t,\theta) d\theta\Big) d\eta\vert \sqtC{1}-\sqtC{2}\vert^2\\
& \quad \leq \norm{\sqP{1}-\sqP{2}}^2 + \frac{4}{\A e^2} \int_0^1 \tilde \P_2\tilde \Y_{2,\eta}(t,\eta) d\eta \vert \sqtC{1}-\sqtC{2}\vert^2\\
& \leq \bigO(1)(\norm{\sqP{1}-\sqP{2}}^2+\vert \sqtC{1}-\sqtC{2}\vert^2).
\end{align*} 
 
Regarding the term $I_{215}$ (and $I_{216}$), we find
\begin{align*}
&\vert\int_0^{1} I_{215}(\sqP{1}-\sqP{2})(t,\eta)d\eta  \vert\\
&\quad=\frac{1}{\A^6}\vert \int_0^{1}  (\sqP{1}-\sqP{2}) \big(\sqP{2}\big)_\eta(t,\eta)\\
&\qquad\qquad\times \Big(\int_0^\eta (e^{-\frac{1}{\ma}(\tilde\Y_2(t,\eta)-\tilde\Y_2(t,\theta))}-e^{-\frac{1}{\ma}(\tilde\Y_1(t,\eta)-\tilde\Y_1(t,\theta))}) \\
&\qquad\qquad\qquad\qquad\qquad\qquad\qquad\qquad\times\min_j(\tilde\V_j^+)^3\tilde\Y_{2,\eta}(t,\theta)\mathbbm{1}_{B(\eta)}(t,\theta) d\theta\Big)d\eta\vert \\
&\quad \leq \frac{1}{\ma\A^6}\int_0^{1} \vert \sqP{1}-\sqP{2} \vert \vert\big(\sqP{2}\big)_\eta\vert(t,\eta) \\
& \qquad \qquad \times \Big(\int_0^\eta (\vert \tilde\Y_2(t,\eta)-\tilde \Y_1(t,\eta)\vert +\vert \tilde\Y_2(t,\theta)-\tilde\Y_1(t,\theta) \vert) \\
&\qquad\qquad\qquad\qquad\qquad\qquad\qquad\times e^{-\frac{1}{\ma}(\tilde\Y_2(t,\eta)-\tilde\Y_2(t,\theta))} \min_j(\tilde\V_j^+)^3 \tilde\Y_{2,\eta}(t,\theta) d\theta\Big) d\eta\\
&\quad \leq \frac{1}{\ma\A^6}\int_0^{1} \vert \sqP{1}-\sqP{2}\vert \,\vert \tilde \Y_1-\tilde\Y_2\vert \, \vert \big(\sqP{2}\big)_\eta\vert (t,\eta)\\
&\qquad\qquad\qquad\qquad\qquad\qquad\times\int_0^{\eta} e^{-\frac{1}{\ma}(\tilde\Y_2(t,\eta)-\tilde\Y_2(t,\theta))} \min_j(\tilde\V_j^+)^3\tilde\Y_{2,\eta}(t,\theta) d\theta d\eta\\
& \qquad +\norm{\sqP{1}-\sqP{2}}^2
 + \frac{1}{\ma^2\A^{12}}\int_0^{1} \big(\sqP{2}\big)_\eta^2(t,\eta)\\
&\qquad\qquad\times\left( \int_0^\eta \vert \tilde\Y_1(t,\theta)-\tilde\Y_2(t,\theta)\vert e^{-\frac{1}{\ma}(\tilde\Y_2(t,\eta)-\tilde\Y_2(t,\theta))} \min_j(\tilde\V_j^+)^3 \tilde\Y_{2,\eta}(t,\theta) d\theta \right)^2 d\eta\\
& \quad\leq \bigO(1)(\norm{\tilde \Y_1-\tilde \Y_2}^2 +\norm{\sqP{1}-\sqP{2}}^2),
\end{align*}
using the estimates in \eqref{eq:J3} and $\big(\sqP{2}\big)_\eta^2\le \frac{1}8\A^3\tilde \Y_{2,\eta}$.

The next terms are $I_{217}+I_{218}$:
\begin{align*}
I_{217}+I_{218}&= \frac{1}{\A^6}\big(\sqP{2}\big)_\eta(t,\eta) \\
&\quad\times \left(\int_0^\eta \min_j(e^{-\frac{1}{\ma}(\tilde \Y_j(t,\eta)-\tilde\Y_j(t,\theta))})\min_j(\tilde\V_j^+)^3\tilde\Y_{2,\eta}(t,\theta) d\theta\right)\\
& \quad - \frac{1}{\A^6}\big(\sqP{1}\big)_\eta(t,\eta)\\
&\quad\times \left(\int_0^\eta \min_j(e^{-\frac{1}{\ma}(\tilde \Y_j(t,\eta)-\tilde\Y_j(t,\theta))})\min_j(\tilde\V_j^+)^3\tilde \Y_{1,\eta}(t,\theta) d\theta\right)\\
& =  \frac{1}{\A^6}(\big(\sqP{2}\big)_\eta-\big(\sqP{1}\big)_\eta)(t,\eta)\\
&\quad\times \min_k\Big[ \int_0^\eta \min_j(e^{-\frac{1}{\ma}(\tilde \Y_j(t,\eta)-\tilde\Y_j(t,\theta))})\min_j(\tilde\V_j^+)^3\tilde\Y_{k,\eta}(t,\theta) d\theta\Big]\\
& \quad + \frac{1}{\A^6}\big(\sqP{1}\big)_\eta(t,\eta) \int_0^\eta \min_j(e^{-\frac{1}{\ma}(\tilde \Y_j(t,\eta)-\tilde\Y_j(t,\theta))})\\
&\quad\qquad\qquad\qquad\qquad\qquad\times \min_j(\tilde\V_j^+)^3(\tilde\Y_{2,\eta}-\tilde \Y_{1,\eta})(t, \theta) d\theta\mathbbm{1}_E(t,\eta)\\
& \quad +\frac{1}{\A^6}\big(\sqP{2}\big)_\eta(t,\eta) \int_0^\eta \min_j(e^{-\frac{1}{\ma}(\tilde \Y_j(t,\eta)-\tilde\Y_j(t,\theta))})\\
&\quad\qquad\qquad\qquad\qquad\qquad\times \min_j(\tilde\V_j^+)^3(\tilde\Y_{2,\eta}-\tilde \Y_{1,\eta})(t,\theta) d\theta \mathbbm{1}_{E^c}(t,\eta)\\
& = L_{211}+L_{212}+L_{213}.
\end{align*}
The term $L_{211}$ can be estimated as in \eqref{eq:L1}, yielding the result
\begin{equation*}
\vert \int_0^{1} L_{211}(\sqP{1}-\sqP{2})(t,\eta)d\eta\vert  \leq \bigO(1) \norm{\sqP{1}-\sqP{2}}^2.
\end{equation*}
The terms $L_{212}$ and $L_{213}$ can be treated similarly; we list $L_{213}$:
\begin{align*}
\vert \int_0^{1} L_{213}(\sqP{1}-\sqP{2})(t,\eta)d\eta\vert 
 \leq \bigO(1) (\norm{\tilde \Y_1-\tilde \Y_2}^2+\norm{\sqP{1}-\sqP{2}}^2),
\end{align*}
see  \eqref{eq:L3}.
The terms $I_{219}$ and $I_{220}$ share the same behavior.  We find, when $\sqtC{1}\le \sqtC{2}$ (otherwise $I_{219}$ vanishes), see \eqref{eq:J7}, that
\begin{align*}
&\vert \int_0^{1} I_{219}(\sqP{1}-\sqP{2})(t,\eta)d\eta\vert \\
&\quad =(\frac{1}{\sqtC{1}^6}-\frac{1}{\sqtC{2}^6}) \Big\vert \int_0^1 (\sqP{1}-\sqP{2})\big(\sqP{1}\big)_\eta(t,\eta) \\
&\qquad\qquad\qquad\qquad\times\Big(\int_0^\eta    e^{-\frac{1}{\sqtC{1}}(\tilde\Y_1(t,\eta)-\tilde\Y_1(t,\theta))} (\tilde\V_1^+)^3\tilde\Y_{1,\eta}(t,\theta)d\theta\Big) d\eta \Big\vert \\
&\quad\le  \bigO(1)(\vert \sqtC{1}-\sqtC{2}\vert^2 + \norm{\sqP{1}-\sqP{2}}^2).
\end{align*} 

\bigskip
Next, we consider the term $I_{22}$:
\begin{align*}
I_{22}&=\int_0^{1} (\sqP{1}- \sqP{2})(t, \eta)\\
&\qquad\times\Big(\frac{1}{\sqtC{2}^6}\big(\sqP{2}\big)_\eta(t, \eta)\int_0^{1}
 e^{-\frac{1}{\sqtC{2}}\vert \tilde\Y_2(t,\eta)-\tilde\Y_2(t,\theta)\vert }\tilde\P_2\tilde \U_2\tilde\Y_{2,\eta}(t,\theta)d\theta\\
&\qquad\qquad- \frac{1}{\sqtC{1}^6}\big(\sqP{1}\big)_\eta(t, \eta)\int_0^{1} e^{-\frac{1}{\sqtC{1}}\vert \tilde\Y_1(t,\eta)-\tilde\Y_1(t,\theta)\vert } 
\tilde\P_1\tilde \U_1\tilde\Y_{1,\eta}(t,\theta)d\theta\Big)d\eta.
\end{align*}
We start by  replacing the integrals $\int_0^1(\dots)d\theta$ with $\int_0^\eta(\dots)d\theta$, cf.~\eqref{eq:triks}.  
Furthermore, we write $\tilde \U_j=\tilde \U_j^+ +\tilde \U_j^-$, and hence it suffices to consider
\begin{align*}
\tilde I_{22}&=\int_0^{1} (\sqP{1}- \sqP{2})(t, \eta)\\
&\qquad\qquad\times\Big(\frac{1}{\sqtC{2}^6}\big(\sqP{2}\big)_\eta(t, \eta)\int_0^{\eta}
 e^{-\frac{1}{\sqtC{2}}\vert \tilde\Y_2(t,\eta)-\tilde\Y_2(t,\theta)\vert }\tilde\P_2\tilde \U_2^+\tilde\Y_{2,\eta}(t,\theta)d\theta\\
&\qquad\qquad\qquad- \frac{1}{\sqtC{1}^6}\big(\sqP{1}\big)_\eta(t, \eta)\int_0^{\eta} e^{-\frac{1}{\sqtC{1}}\vert \tilde\Y_1(t,\eta)-\tilde\Y_1(t,\theta)\vert } 
\tilde\P_1\tilde \U_1^+\tilde\Y_{1,\eta}(t,\theta)d\theta\Big)d\eta \\
&=\frac{1}{\A^6}\int_0^{1} (\sqP{1}- \sqP{2})(t, \eta)\\
&\qquad\qquad\times\Big(\big(\sqP{2}\big)_\eta(t, \eta)\int_0^{\eta}
 e^{-\frac{1}{\sqtC{2}}( \tilde\Y_2(t,\eta)-\tilde\Y_2(t,\theta))}\tilde\P_2\tilde \U_2^+\tilde\Y_{2,\eta}(t,\theta)d\theta\\
&\qquad\qquad\qquad- \big(\sqP{1}\big)_\eta(t, \eta)\int_0^{\eta} e^{-\frac{1}{\sqtC{1}}(\tilde\Y_1(t,\eta)-\tilde\Y_1(t,\theta))} 
\tilde\P_1\tilde \U_1^+\tilde\Y_{1,\eta}(t,\theta)d\theta\Big)d\eta \\
&\quad+\frac{\sqtC{1}^6-\sqtC{2}^6}{\sqtC{1}^6\sqtC{2}^6}\mathbbm{1}_{\sqtC{2}\le \sqtC{1}} \int_0^{1} (\sqP{1}- \sqP{2})(t, \eta)\sqDP{2}(t, \eta)\\
&\qquad\qquad\qquad\qquad\qquad\times\Big(\int_0^{\eta}
 e^{-\frac{1}{\sqtC{2}}( \tilde\Y_2(t,\eta)-\tilde\Y_2(t,\theta)) }\tilde\P_2\tilde\U_2^+\tilde\Y_{2,\eta}(t,\theta)d\theta\Big)d\eta\\
&\quad+\frac{\sqtC{1}^6-\sqtC{2}^6}{\sqtC{1}^6\sqtC{2}^6}\mathbbm{1}_{\sqtC{1}< \sqtC{2}} \int_0^{1} (\sqP{1}- \sqP{2})(t, \eta)\sqDP{1}(t, \eta)\\
&\qquad\qquad\qquad\qquad\qquad\times\Big(\int_0^{\eta}
 e^{-\frac{1}{\sqtC{1}}( \tilde\Y_1(t,\eta)-\tilde\Y_1(t,\theta)) }\tilde\P_1\tilde\U_1^+\tilde\Y_{1,\eta}(t,\theta)d\theta\Big)d\eta\\
 &= N_1+N_2+N_3.
\end{align*}
We consider first $N_1$ where we get
\begin{align}
N_1&=\frac1{\A^6}\int_0^{1} (\sqP{1}- \sqP{2})(t, \eta) \nn \\
&\qquad\qquad\qquad\times\Big[\sqDP{2}(t, \eta)\int_0^{\eta}
 e^{-\frac{1}{\sqtC{2}}( \tilde\Y_2(t,\eta)-\tilde\Y_2(t,\theta)) }
 \tilde\P_2\tilde\U_2^+\tilde\Y_{2,\eta}(t,\theta)d\theta \notag\\
&\qquad\qquad\qquad\qquad- \sqDP{1}(t, \eta)\int_0^{\eta} e^{-\frac{1}{\sqtC{1}}( \tilde\Y_1(t,\eta)-\tilde\Y_1(t,\theta)) } 
\tilde\P_1\tilde\U_1^+\tilde\Y_{1,\eta}(t,\theta)d\theta\Big]d\eta \notag\\
&=\frac1{\A^6}\int_0^{1} (\sqP{1}- \sqP{2})(t, \eta) \nn \\
&\quad\times\Big[\sqDP{2}(t, \eta)\int_0^{\eta}
 e^{-\frac{1}{\sqtC{2}}( \tilde\Y_2(t,\eta)-\tilde\Y_2(t,\theta)) }(\tilde\P_2-\tilde\P_1)\tilde\U_2^+\tilde\Y_{2,\eta}
 \mathbbm{1}_{\tilde\P_1\le\tilde\P_2}(t,\theta)d\theta\notag\\
&\qquad+ \sqDP{1}(t, \eta)\int_0^{\eta} e^{-\frac{1}{\sqtC{1}}( \tilde\Y_1(t,\eta)-\tilde\Y_1(t,\theta)) } 
(\tilde\P_2-\tilde\P_1)\tilde\U_1^+\tilde\Y_{1,\eta}\mathbbm{1}_{\tilde\P_2<\tilde\P_1}(t,\theta)d\theta \notag\\
&\qquad +\sqDP{2}(t, \eta)\nn \\
&\qquad\qquad\times\int_0^{\eta}
 e^{-\frac{1}{\sqtC{2}}( \tilde\Y_2(t,\eta)-\tilde\Y_2(t,\theta)) }\min_j(\tilde\P_j)(\tilde\U_2^+-\tilde\U_1^+)\tilde\Y_{2,\eta}
 \mathbbm{1}_{\tilde\U_1^+\le\tilde\U_2^+}(t,\theta)d\theta\notag\\
&\qquad +\sqDP{1}(t, \eta) \nn \\
&\qquad\qquad\times\int_0^{\eta}
 e^{-\frac{1}{\sqtC{1}}( \tilde\Y_1(t,\eta)-\tilde\Y_1(t,\theta)) }\min_j(\tilde\P_j)(\tilde\U_2^+-\tilde\U_1^+)\tilde\Y_{1,\eta}
 \mathbbm{1}_{\tilde\U_2^+<\tilde\U_1^+}(t,\theta)d\theta\notag\\
&\qquad+\sqDP{2}(t, \eta)\nn \\
&\qquad\qquad\times\int_0^{\eta}
 e^{-\frac{1}{\sqtC{2}}( \tilde\Y_2(t,\eta)-\tilde\Y_2(t,\theta)) }\min_j(\tilde\P_j)\min_j(\tilde\U_j^+)\tilde\Y_{2,\eta}(t,\theta)d\theta\notag\\
&\qquad-\sqDP{1}(t, \eta)\int_0^{\eta}
 e^{-\frac{1}{\sqtC{1}}( \tilde\Y_1(t,\eta)-\tilde\Y_1(t,\theta)) }\min_j(\tilde\P_j)\min_j(\tilde\U_j^+)\tilde\Y_{1,\eta}(t,\theta)d\theta
\Big]d\eta \notag\\
&= N_{11}+N_{12}+N_{13}+N_{14}+N_{15}+N_{16}. \label{eq:N1P}
\end{align}
The terms $N_{11}$ and $N_{12}$ can be treated similarly. To that end we find
\begin{align*}
\abs{N_{11}}&= \frac1{\A^6}\vert\int_0^{1} (\sqP{1}- \sqP{2})\sqDP{2}(t, \eta)\\
&\qquad\times\Big(\int_0^{\eta}
 e^{-\frac{1}{\sqtC{2}}( \tilde\Y_2(t,\eta)-\tilde\Y_2(t,\theta)) }(\tilde\P_2-\tilde\P_1)\tilde\U_2^+\tilde\Y_{2,\eta}
 \mathbbm{1}_{\tilde\P_1\le\tilde\P_2}(t,\theta)d\theta\Big)d\eta \vert \\
 &\le \norm{\sqP{1}- \sqP{2}}^2+\frac1{A^{12}} \int_0^{1}(\sqDP{2})^2(t, \eta)\\
 &\qquad\qquad\times\Big(\int_0^{\eta}
 e^{-\frac{1}{\sqtC{2}}( \tilde\Y_2(t,\eta)-\tilde\Y_2(t,\theta)) }(\tilde\P_2-\tilde\P_1)\tilde\U_2^+\tilde\Y_{2,\eta}
 \mathbbm{1}_{\tilde\P_1\le\tilde\P_2}(t,\theta)d\theta \Big)^2d\eta \\
 &\le \norm{\sqP{1}- \sqP{2}}^2 \\
 &\quad\qquad+\frac1{2\A^9} \int_0^{1}\tilde\Y_{2,\eta}(t, \eta)\Big(\int_0^{\eta}e^{-\frac{1}{\sqtC{2}}( \tilde\Y_2(t,\eta)-\tilde\Y_2(t,\theta))}
 (\tilde\U_2^+)^2\tilde\Y_{2,\eta}(t,\theta)d\theta \Big)\\
 &\qquad\qquad\qquad\times \Big(\int_0^{\eta}e^{-\frac{1}{\sqtC{2}}( \tilde\Y_2(t,\eta)-\tilde\Y_2(t,\theta))}\vert\sqP{2}-\sqP{1}\vert^2 \tilde\P_2\tilde\Y_{2,\eta}(t,\theta)d\theta \Big)d\eta \\
&\le \norm{\sqP{1}- \sqP{2}}^2 +\frac2{\A^8} \int_0^{1}\tilde\P_2\tilde\Y_{2,\eta}(t, \eta) \\
 &\qquad\qquad\qquad\times
 \Big(\int_0^{\eta}e^{-\frac{1}{\sqtC{2}}( \tilde\Y_2(t,\eta)-\tilde\Y_2(t,\theta))}\vert\sqP{2}-\sqP{1}\vert^2 \tilde\P_2\tilde\Y_{2,\eta}(t,\theta)d\theta \Big)d\eta \\
 &\le \bigO(1)\norm{\sqP{1}- \sqP{2}}^2,
\end{align*}
where we have used \eqref{eq:all_estimatesE} and \eqref{eq:all_Pderiv_estimatesC}.  The terms $N_{13}$ and $N_{14}$ follow the same lines. More precisely,
\begin{align*}
\abs{N_{13}}&=\frac1{\A^6} \vert\int_0^{1} (\sqP{1}- \sqP{2})
\sqDP{2}(t, \eta) \\
&\quad \times\int_0^{\eta}
 e^{-\frac{1}{\sqtC{2}}( \tilde\Y_2(t,\eta)-\tilde\Y_2(t,\theta)) }\min_j(\tilde\P_j)(\tilde\U_2^+-\tilde\U_1^+)\tilde\Y_{2,\eta}
 \mathbbm{1}_{\tilde\U_1^+\le\tilde\U_2^+}(t,\theta)d\theta d\eta\vert\\
 &\le\frac1{\A^6} \int_0^{1}\vert\sqP{1}- \sqP{2}\vert\, 
\vert\sqDP{2}\vert(t, \eta) \\
&\quad \times\int_0^{\eta}
 e^{-\frac{1}{\sqtC{2}}( \tilde\Y_2(t,\eta)-\tilde\Y_2(t,\theta)) }\min_j(\tilde\P_j)\vert\tilde\U_2^+-\tilde\U_1^+\vert\tilde\Y_{2,\eta}
 \mathbbm{1}_{\tilde\U_1^+\le\tilde\U_2^+}(t,\theta)d\theta d\eta\\
 &\le\norm{\sqP{1}- \sqP{2}}^2 +\frac1{\A^{12}} \int_0^{1}
(\sqDP{2})^2(t, \eta)\\
 &\qquad\qquad\times \Big(\int_0^{\eta}
 e^{-\frac{1}{\sqtC{2}}( \tilde\Y_2(t,\eta)-\tilde\Y_2(t,\theta)) }\min_j(\tilde\P_j)\vert\tilde\U_2^+-\tilde\U_1^+\vert\tilde\Y_{2,\eta}
 (t,\theta)d\theta\Big)^2 d\eta\\ 
&\le\norm{\sqP{1}- \sqP{2}}^2 + \frac1{\A^{12}} \int_0^{1}
(\sqDP{2})^2(t, \eta) \\
&\qquad\qquad\qquad\times\big(\int_0^{\eta} e^{-\frac3{2\sqtC{2}}( \tilde\Y_2(t,\eta)-\tilde\Y_2(t,\theta)) }\tilde\P_2\tilde\Y_{2,\eta}
(t,\theta)d\theta\big) \\
 &\qquad\qquad\qquad\times \big(\int_0^{\eta} e^{-\frac1{2\sqtC{2}}( \tilde\Y_2(t,\eta)-\tilde\Y_2(t,\theta)) }(\tilde\U_2^+-\tilde\U_1^+)^2\tilde\P_2\tilde\Y_{2,\eta}(t,\theta)d\theta\big) d\eta\\   
 &\le\norm{\sqP{1}- \sqP{2}}^2 + \frac{2}{\A^{11}} \int_0^{1}
\tilde\P_2(\sqDP{2})^2(t, \eta) \\
 &\qquad\qquad\qquad\qquad\times \big(\int_0^{\eta} e^{-\frac1{2\sqtC{2}}( \tilde\Y_2(t,\eta)-\tilde\Y_2(t,\theta)) }(\tilde\U_2-\tilde\U_1)^2\tilde\P_2\tilde\Y_{2,\eta}(t,\theta)d\theta\big) d\eta\\   
&\le \bigO(1)\big( \norm{\sqP{2}-\sqP{1}}^2+\norm{\tilde\U_2-\tilde\U_1}^2\big),
\end{align*}
by applying \eqref{eq:343},   \eqref{eq:all_Pderiv_estimatesC}, and \eqref{eq:all_estimatesE}. 
The term  $N_{15}+N_{16}$ can be estimated as follows:
\begin{align*}
N_{15}+N_{16}&=\frac1{A^6}\int_0^{1} (\sqP{1}- \sqP{2})(t, \eta)\\
&\quad\times\Big[\sqDP{2}(t, \eta)\int_0^{\eta}
 e^{-\frac{1}{\sqtC{2}}( \tilde\Y_2(t,\eta)-\tilde\Y_2(t,\theta)) }\min_j(\tilde\P_j)\min_j(\tilde\U_j^+)\tilde\Y_{2,\eta}(t,\theta)d\theta\\
&\qquad\qquad\qquad-\sqDP{1}(t, \eta)\int_0^{\eta}
 e^{-\frac{1}{\sqtC{1}}( \tilde\Y_1(t,\eta)-\tilde\Y_1(t,\theta)) }\\
 &\qquad\qquad\qquad\qquad\qquad\times\min_j(\tilde\P_j)\min_j(\tilde\U_j^+)\tilde\Y_{1,\eta}(t,\theta)d\theta
\Big]d\eta \\
&=\mathbbm{1}_{\sqtC{1}\leq \sqtC{2}} \frac{1}{\A^6}\int_0^{1} (\sqP{1}- \sqP{2})(t, \eta)\sqDP{2}(t, \eta)\\
& \qquad \quad \times \Big(\int_0^{\eta}
 (e^{-\frac{1}{\sqtC{2}}( \tilde\Y_2(t,\eta)-\tilde\Y_2(t,\theta)) -}e^{-\frac{1}{\sqtC{1}}(\tilde \Y_2(t,\eta)-\tilde \Y_2(t,\theta))})\\
 &\qquad\qquad\qquad\qquad\qquad\qquad\qquad\times\min_j(\tilde\P_j)\min_j(\tilde\U_j^+)\tilde\Y_{2,\eta}(t,\theta)d\theta\Big) d\eta\\
 &\quad +\mathbbm{1}_{\sqtC{2}< \sqtC{1}} \frac{1}{\A^6}\int_0^{1} (\sqP{1}- \sqP{2})(t, \eta)\sqDP{1}(t, \eta)\\
 & \qquad \qquad \times\Big(\int_0^{\eta}
 (e^{-\frac{1}{\sqtC{2}}( \tilde\Y_1(t,\eta)-\tilde\Y_1(t,\theta)) }-e^{-\frac{1}{\sqtC{1}}(\tilde \Y_1(t,\eta)-\tilde \Y_1(t,\theta))})\\
 &\qquad\qquad\qquad\qquad\qquad\qquad\qquad\times\min_j(\tilde\P_j)\min_j(\tilde\U_j^+)\tilde\Y_{1,\eta}(t,\theta)d\theta\Big) d\eta\\
&+\frac1{\A^6}\int_0^{1} (\sqP{1}- \sqP{2})\sqDP{2}(t, \eta) \\
&\quad\times\int_0^{\eta}
\big( e^{-\frac{1}{\ma}( \tilde\Y_2(t,\eta)-\tilde\Y_2(t,\theta))}-e^{-\frac{1}{\ma}
( \tilde\Y_1(t,\eta)-\tilde\Y_1(t,\theta))}\big) \\
 &\qquad\qquad\qquad\qquad\qquad\times\min_j(\tilde\P_j)\min_j(\tilde\U_j^+)\tilde\Y_{2,\eta}(t,\theta) \mathbbm{1}_{B(\eta)}(t,\theta) d\theta\\
&\quad+\frac1{\A^6}\int_0^{1} (\sqP{1}- \sqP{2})\sqDP{1}(t, \eta)\\
&\quad\times\int_0^{\eta}
\big( e^{-\frac{1}{\ma}( \tilde\Y_2(t,\eta)-\tilde\Y_2(t,\theta))}-e^{-\frac{1}{\ma}
(\tilde\Y_1(t,\eta)-\tilde\Y_1(t,\theta))}\big) \\
 &\qquad\qquad\qquad\qquad\qquad\times\min_j(\tilde\P_j)\min_j(\tilde\U_j^+)\tilde\Y_{1,\eta}(t,\theta) \mathbbm{1}_{B^c(\eta)}(t,\theta) d\theta\\
&\quad +\frac1{\A^6}\int_0^{1} (\sqP{1}- \sqP{2})(t, \eta)\Big[\sqDP{2}(t, \eta)\int_0^{\eta}
 \min_j(e^{-\frac{1}{\ma}( \tilde\Y_j(t,\eta)-\tilde\Y_j(t,\theta))})\\
&\qquad\qquad\qquad\qquad\qquad\qquad\qquad\times\min_j(\tilde\P_j)\min_j(\tilde\U_j^+)\tilde\Y_{2,\eta}(t,\theta)d\theta\\
&\quad-\sqDP{1}(t, \eta)\int_0^{\eta}
 \min_j(e^{-\frac{1}{\ma}( \tilde\Y_j(t,\eta)-\tilde\Y_j(t,\theta))})\\
 &\qquad\qquad\qquad\qquad\qquad\qquad\qquad\times\min_j(\tilde\P_j)\min_j(\tilde\U_j^+)\tilde\Y_{1,\eta}(t,\theta)d\theta
\Big]d\eta \\
&=N_{151}+N_{152}+N_{153}+N_{154}+N_{155},
\end{align*} 
which, unfortunately, all need a special treatment.

The terms $N_{151}$ and $N_{152}$ can be handled as follow:
\begin{align*}
\abs{N_{151}}& \leq \mathbbm{1}_{\sqtC{1}\leq\sqtC{2}} \frac{2\sqrt{2}\ma}{\A^6e}\int_0^{1} \abs{\sqP{1}- \sqP{2}}\abs{\sqDP{2}}(t, \eta) \\
&\qquad\qquad\times\Big(\int_0^{\eta}
 e^{-\frac{3}{4\sqtC{2}}( \tilde\Y_2(t,\eta)-\tilde\Y_2(t,\theta)) }\tilde \P_2\tilde\Y_{2,\eta}(t,\theta)d\theta\Big) d\eta\vert \sqtC{1}-\sqtC{2}\vert\\
 & \leq \frac{2\sqrt{2}\ma}{\A^6e}\int_0^{1} \abs{\sqP{1}- \sqP{2}}\abs{\sqDP{2}}(t, \eta)\\
&\qquad\qquad\qquad\qquad\times\Big(\int_0^\eta e^{-\frac3{2\sqtC{2}}(\tilde \Y_2(t,\eta)-\tilde \Y_2(t,\theta))}\tilde \P_2\tilde \Y_{2,\eta}(t,\theta) d\theta\Big)^{1/2}\\
 & \qquad \qquad\qquad\qquad \times \Big(\int_0^\eta \tilde \P_2\tilde \Y_{2,\eta}(t,\theta) d\theta\Big)^{1/2}d\eta\vert \sqtC{1}-\sqtC{2}\vert\\
 & \leq \frac{2\sqrt{2}\ma}{\A^3e}\int_0^{1} \abs{\sqP{1}- \sqP{2}}\abs{\sqDP{2}}\sqP{2}(t,\eta) d\eta\vert\sqtC{1}-\sqtC{2}\vert\\
 & \leq \bigO(1)(\norm{\sqP{1}-\sqP{2}}^2+\vert \sqtC{1}-\sqtC{2}\vert^2).
 \end{align*}

The terms $N_{153}$ and $N_{154}$ can be handled as follow:
\begin{align*}
\abs{N_{153}}&=\frac1{\A^6}\vert\int_0^{1} (\sqP{1}- \sqP{2})\sqDP{2}(t, \eta) \\
&\quad\times\int_0^{\eta}
\big( e^{-\frac{1}{\ma}( \tilde\Y_2(t,\eta)-\tilde\Y_2(t,\theta))}-e^{-\frac{1}{\ma}
( \tilde\Y_1(t,\eta)-\tilde\Y_1(t,\theta))}\big) \\
&\qquad\qquad\qquad\qquad\times\min_j(\tilde\P_j)\min_j(\tilde\U_j^+)\tilde\Y_{2,\eta}(t,\theta) \mathbbm{1}_{B(\eta)}(t,\theta) d\theta d\eta\vert\\
&\le\frac1{\ma\A^6}\int_0^{1} \vert\sqP{1}- \sqP{2}\vert\, \vert\sqDP{2}\vert(t, \eta) \\
&\quad\times\int_0^{\eta}
\big( \vert \tilde\Y_2(t,\eta)-\tilde\Y_1(t,\eta) \vert+ \vert \tilde\Y_2(t,\theta)-\tilde\Y_1(t,\theta) \vert\big) e^{-\frac{1}{\ma}( \tilde\Y_2(t,\eta)-\tilde\Y_2(t,\theta))}\\
&\qquad\qquad\times\min_j(\tilde\P_j)\min_j(\tilde\U_j^+)\tilde\Y_{2,\eta}(t,\theta) d\theta  d\eta \\
&=\frac1{\ma\A^6}\int_0^{1} \vert \sqP{1}- \sqP{2}\vert \,  \vert \tilde\Y_2-\tilde\Y_1 \vert\,\vert\sqDP{2}\vert(t, \eta) \int_0^{\eta}e^{-\frac{1}{\ma}( \tilde\Y_2(t,\eta)-\tilde\Y_2(t,\theta))}\\
&\qquad\qquad\times\min_j(\tilde\P_j)\min_j(\tilde\U_j^+)\tilde\Y_{2,\eta}(t,\theta)d\theta  d\eta \\
&\quad +\frac1{\ma\A^6}\int_0^{1} \vert \sqP{1}- \sqP{2}\vert\,\vert \sqDP{2}\vert(t, \eta) \int_0^{\eta}
 e^{-\frac{1}{\ma}( \tilde\Y_2(t,\eta)-\tilde\Y_2(t,\theta))} \vert \tilde\Y_2-\tilde\Y_1 \vert\\
&\qquad\qquad\times\min_j(\tilde\P_j)\min_j(\tilde\U_j^+)\tilde\Y_{2,\eta}(t,\theta) d\theta  d\eta \\
&\le \frac1{\sqrt{2}\A^5}\int_0^{1} \vert \sqP{1}- \sqP{2}\vert \,  \vert \tilde\Y_2-\tilde\Y_1 \vert\, 
\vert\sqDP{2}\vert(t, \eta) \\
&\qquad\qquad\qquad\qquad\times\Big(\int_0^{\eta}e^{-\frac{1}{\ma}( \tilde\Y_2(t,\eta)-\tilde\Y_2(t,\theta))}\tilde\P_2^2\tilde\Y_{2,\eta}(t,\theta)d\theta\Big)^{1/2} \\
&\qquad\qquad\qquad\qquad\times\Big(\int_0^{\eta}e^{-\frac{1}{\ma}( \tilde\Y_2(t,\eta)-\tilde\Y_2(t,\theta))}\tilde\Y_{2,\eta}(t,\theta) d\theta\Big)^{1/2}  d\eta \\
&\quad+\norm{\sqP{1}- \sqP{2}}^2+\frac1{2\A^{10}} \int_0^{1}(\sqDP{2})^2(t, \eta) \\
&\qquad\qquad\qquad\times\Big(
\int_0^{\eta}e^{-\frac{1}{\ma}( \tilde\Y_2(t,\eta)-\tilde\Y_2(t,\theta))}\vert \tilde\Y_2-\tilde\Y_1\vert\tilde\P_2\tilde\Y_{2,\eta}(t,\theta) d\theta\Big)^{2}  d\eta \\
&\le \frac{\sqrt{3}}{2\A^{2}}\int_0^{1} \vert \sqP{1}- \sqP{2}\vert \,  \vert \tilde\Y_2-\tilde\Y_1 \vert \tilde\P_2^{1/2}\vert \sqDP{2}\vert(t, \eta)  d\eta +\norm{\sqP{1}- \sqP{2}}^2\\
&\quad+\frac1{2\A^{10}}\int_0^{1} (\sqDP{2})^2(t, \eta) \Big(\int_0^{\eta}e^{-\frac{3}{2\sqtC{2}}( \tilde\Y_2(t,\eta)-\tilde\Y_2(t,\theta))}\tilde\P_2\tilde\Y_{2,\eta}(t,\theta) d\theta\Big) \\
&\qquad\qquad\times \Big(\int_0^{\eta}e^{-\frac{1}{2\ma}(\tilde\Y_2(t,\eta)-\tilde\Y_2(t,\theta))}(\tilde\Y_2-\tilde\Y_1)^2\tilde \P_2\tilde\Y_{2,\eta}(t,\theta) d\theta\Big)d\eta \\
&\le \bigO(1)\big(\norm{\sqP{1}- \sqP{2}}^2+\norm{\tilde\Y_1- \tilde\Y_2}^2 \big),
\end{align*} 
using \eqref{eq:all_PestimatesB}, \eqref{eq:343}, \eqref{eq:all_Pderiv_estimatesA}, and \eqref{eq:all_estimatesE}.
We consider now $N_{155}$:
\begin{align*}
N_{155}&=\frac1{\A^6}\int_0^{1} (\sqP{1}- \sqP{2})(t, \eta)\\
&\quad\times\Big[\sqDP{2}(t, \eta)\int_0^{\eta}
 \min_j(e^{-\frac{1}{\ma}( \tilde\Y_j(t,\eta)-\tilde\Y_j(t,\theta))})\min_j(\tilde\P_j)\min_j(\tilde\U_j^+)\tilde\Y_{2,\eta}(t,\theta)d\theta\\
&\quad-\sqDP{1}(t, \eta)\int_0^{\eta}
 \min_j(e^{-\frac{1}{\ma}( \tilde\Y_j(t,\eta)-\tilde\Y_j(t,\theta))})\min_j(\tilde\P_j)\min_j(\tilde\U_j^+)\tilde\Y_{1,\eta}(t,\theta)d\theta
\Big]d\eta \\
&= -\frac1{\A^6}\int_0^{1} (\sqP{1}- \sqP{2}) (\sqP{1}- \sqP{2})_\eta (t, \eta)\\
&\qquad\quad\times\min_k \Big[ \int_0^{\eta}
 \min_j(e^{-\frac{1}{\ma}( \tilde\Y_j(t,\eta)-\tilde\Y_j(t,\theta))})\min_j(\tilde\P_j)\min_j(\tilde\U_j^+)\tilde\Y_{k,\eta}(t,\theta)d\theta\Big]d\eta \\
&\quad + \frac1{\A^6}\int_0^{1} (\sqP{1}- \sqP{2})\sqDP{2}(t, \eta)\mathbbm{1}_{\tilde E^c}(t,\eta)\\
&\qquad\times\big( \int_0^{\eta}\min_j(e^{-\frac{1}{\ma}( \tilde\Y_j(t,\eta)-\tilde\Y_j(t,\theta))})\min_j(\tilde\P_j)\min_j(\tilde\U_j^+)(\tilde\Y_{2}- \tilde\Y_{1})_\eta (t,\theta)d\theta\big)d\eta \\
&\quad+  \frac1{\A^6}\int_0^{1} (\sqP{1}- \sqP{2})\sqDP{1}(t, \eta)\mathbbm{1}_{\tilde E}(t,\eta)\\
&\qquad\times\big( \int_0^{\eta}\min_j(e^{-\frac{1}{\ma}( \tilde\Y_j(t,\eta)-\tilde\Y_j(t,\theta))})\min_j(\tilde\P_j)\min_j(\tilde\U_j^+)(\tilde\Y_{2}- \tilde\Y_{1})_\eta (t,\theta)d\theta\big)d\eta \\   
&= N_{1551}+N_{1552}+N_{1553},
\end{align*}
which, yet again, requires separate treatment.  Here $\tilde E$ is defined in \eqref{eq:tildeE}.
The term $N_{1551}$ can be handled as the term $\tilde L_{31}$, cf.~\eqref{eq:tildeL31},
\begin{align*}
\vert N_{1551}\vert&\le \bigO(1)\norm{\sqP{1}-\sqP{2}}^2.
\end{align*}
The terms $N_{1552}$ and $N_{1553}$ can be treated in the same manner:
\begin{align*}
\abs{N_{1552}}&\le\frac1{\A^6}\int_0^{1} \vert\sqP{1}- \sqP{2}\vert\,\vert\sqDP{2}\vert(t, \eta)\mathbbm{1}_{\tilde E^c}(t,\eta)\\
&\quad\quad\times\vert\int_0^{\eta}\min_j(e^{-\frac{1}{\ma}( \tilde\Y_j(t,\eta)-\tilde\Y_j(t,\theta))})\min_j(\tilde\P_j)\min_j(\tilde\U_j^+)(\tilde\Y_{2}- \tilde\Y_{1})_\eta (t,\theta) d\theta\vert d\eta \\
&\le \frac1{\A^6}\int_0^{1} \vert\sqP{1}- \sqP{2}\vert\, \vert \tilde\Y_1- \tilde\Y_2 \vert\min_j(\tilde\P_j)\min_j(\tilde\U_j^+)\vert \sqDP{2}\vert(t,\eta)d\eta\\
&\quad +  \frac1{\A^6}\int_0^{1}\vert\sqP{1}- \sqP{2}\vert\vert\sqDP{2}\vert(t, \eta)\\
&\quad\times \vert\int_0^{\eta}(\tilde\Y_1- \tilde\Y_2)\Big[\frac{d}{d\theta}\big(\min_j(e^{-\frac{1}{\ma}( \tilde\Y_j(t,\eta)-\tilde\Y_j(t,\theta))})\big)\min_j(\tilde\P_j)\min_j(\tilde\U_j^+) \\
&\qquad\qquad\qquad + \min_j(e^{-\frac{1}{\ma}( \tilde\Y_j(t,\eta)-\tilde\Y_j(t,\theta))})\frac{d}{d\theta}\big(\min_j(\tilde\P_j)\min_j(\tilde\U_j^+) \big)\Big](t,\theta) d\theta\vert d\eta \\
&\le \bigO(1) \big(\norm{\sqP{1}- \sqP{2}}^2+ \norm{\tilde\Y_1- \tilde\Y_2}^2 \big). 
\end{align*}

The term $N_2$ (and also $N_3$) can be treated as follows.
\begin{align*}
\abs{N_2}&\leq \frac{\sqtC{1}^6-\sqtC{2}^6}{\sqtC{1}^6\sqtC{2}^6}\int_0^1\vert \sqP{1}-\sqP{2}\vert \vert\sqDP{2}\vert(t,\eta)\\
&\qquad\qquad\times\big(\int_0^\eta e^{-\frac{1}{\sqtC{2}}(\tilde\Y_2(t,\eta)-\tilde\Y_2(t,\theta))}\tilde \P_2^2\tilde\Y_{2,\eta}(t,\theta)d\theta\big)^{1/2}\\
& \qquad \qquad \times\big(\int_0^\eta e^{-\frac{1}{\sqtC{2}}(\tilde \Y_2(t,\eta)-\tilde \Y_2(t,\theta))}\tilde \U_2^2\tilde\Y_{2,\eta}(t,\theta) d\theta\big)^{1/2}d\eta\\
& \leq \sqrt{6}\frac{\sqtC{1}^6-\sqtC{2}^6}{\A^6\sqtC{2}^3} \int_0^1 \vert \sqP{1}-\sqP{2}\vert \tilde\P_2\vert \sqDP{2}\vert (t,\eta) d\eta\\
& \leq \bigO(1)\vert \sqtC{1}-\sqtC{2}\vert \int_0^1 \vert \sqP{1}-\sqP{2}\vert (t,\eta)d\eta\\
& \leq \bigO(1) (\norm{\sqP{1}-\sqP{2}}^2+\vert \sqtC{1}-\sqtC{2}\vert ^2).
\end{align*}
This concludes the discussion of the term $I_{22}$.

\bigskip
The term $I_{23}$ is handled as follows:
\begin{align*}
I_{23}&=\int_0^{1} (\sqP{1}-\sqP{2})(t, \eta)\\
&\qquad\qquad\times\Big(\frac{1}{\sqtC{2}^6}\sqDP{2}(t, \eta)\int_0^{1}
 e^{-\frac{1}{\sqtC{2}}\vert \tilde\Y_2(t,\eta)-\tilde\Y_2(t,\theta)\vert }\tilde\Q_2\tilde \U_{2,\eta}(t,\theta)d\theta\\
&\qquad\qquad\qquad- \frac{1}{\sqtC{1}^6}\sqDP{1}(t, \eta)\int_0^{1} e^{-\frac{1}{\sqtC{1}}\vert \tilde\Y_1(t,\eta)-\tilde\Y_1(t,\theta)\vert } 
\tilde\Q_1\tilde \U_{1,\eta}(t,\theta)d\theta\Big)d\eta.
\end{align*}
We start by  replacing the integrals $\int_0^1(\dots)d\theta$ with $\int_0^\eta(\dots)d\theta$, cf.~\eqref{eq:triks}.  Thus
\begin{align*}
\tilde I_{23}&=\int_0^{1} (\sqP{1}-\sqP{2})\Big(\frac{1}{\sqtC{2}^6}\sqDP{2}(t, \eta)\int_0^{\eta}
 e^{-\frac{1}{\sqtC{2}}( \tilde\Y_2(t,\eta)-\tilde\Y_2(t,\theta)) }\tilde\Q_2\tilde\U_{2,\eta}(t,\theta)d\theta\\
&\qquad\qquad\qquad- \frac{1}{\sqtC{1}}\sqDP{1}(t, \eta)\int_0^{\eta} e^{-\frac{1}{\sqtC{1}}( \tilde\Y_1(t,\eta)-\tilde\Y_1(t,\theta)) } 
\tilde\Q_1\tilde\U_{1,\eta}(t,\theta)d\theta\Big)d\eta\\
&= \frac{1}{\A^6}\int_0^{1} (\sqP{1}-\sqP{2})\Big(\sqDP{2}(t, \eta)\int_0^{\eta}
 e^{-\frac{1}{\sqtC{2}}( \tilde\Y_2(t,\eta)-\tilde\Y_2(t,\theta)) }\tilde\Q_2\tilde\U_{2,\eta}(t,\theta)d\theta \\
& \qquad -\sqDP{1}(t, \eta)\int_0^{\eta} e^{-\frac{1}{\sqtC{1}}( \tilde\Y_1(t,\eta)-\tilde\Y_1(t,\theta)) } 
\tilde\Q_1\tilde\U_{1,\eta}(t,\theta)d\theta \Big)d\eta \\
& \quad+\mathbbm{1}_{\sqtC{2}\le \sqtC{1}}\big(\frac{1}{\sqtC{2}^6}-\frac{1}{\sqtC{1}^6} \big)\int_0^{1}(\sqP{1}-\sqP{2})\sqDP{2}(t, \eta) \\
&\qquad\qquad\qquad\qquad\qquad\times\int_0^{\eta}e^{-\frac{1}{\sqtC{2}}( \tilde\Y_2(t,\eta)-\tilde\Y_2(t,\theta)) }\tilde\Q_2\tilde\U_{2,\eta}(t,\theta)d\theta d\eta \\
& \quad+\mathbbm{1}_{\sqtC{1}<\sqtC{2}}\big(\frac{1}{\sqtC{2}^6}-\frac{1}{\sqtC{1}^6} \big)\int_0^{1}(\sqP{1}-\sqP{2})\sqDP{1}(t, \eta)\\
&\qquad\qquad\qquad\qquad\qquad\times
\int_0^{\eta}e^{-\frac{1}{\sqtC{1}}( \tilde\Y_1(t,\eta)-\tilde\Y_1(t,\theta)) }\tilde\Q_1\tilde\U_{1,\eta}(t,\theta)d\theta d\eta \\
&= M_1+M_2+M_3.
\end{align*}
We estimate, using first that $\tilde\Q_i=\sqtC{i}\tilde\P_i-\tilde\D_i$ (cf.~\eqref{eq:DogP}):
\begin{align*}
M_1&=\frac{1}{\A^6}\int_0^{1} (\sqP{1}-\sqP{2})\Big(\sqDP{2}(t, \eta)\int_0^{\eta}
 e^{-\frac{1}{\sqtC{2}}( \tilde\Y_2(t,\eta)-\tilde\Y_2(t,\theta)) }\tilde\Q_2\tilde\U_{2,\eta}(t,\theta)d\theta \\
& \qquad -\sqDP{1}(t, \eta)\int_0^{\eta} e^{-\frac{1}{\sqtC{1}}( \tilde\Y_1(t,\eta)-\tilde\Y_1(t,\theta)) } 
\tilde\Q_1\tilde\U_{1,\eta}(t,\theta)d\theta \Big)d\eta \\
&= \frac{1}{\A^6}\int_0^{1} (\sqP{1}-\sqP{2})\Big(\sqtC{2}\sqDP{2}(t, \eta)\int_0^{\eta}
 e^{-\frac{1}{\sqtC{2}}( \tilde\Y_2(t,\eta)-\tilde\Y_2(t,\theta)) }\tilde\P_2\tilde\U_{2,\eta}(t,\theta)d\theta \\
& \qquad -\sqtC{1}\sqDP{1}(t, \eta)\int_0^{\eta} e^{-\frac{1}{\sqtC{1}}( \tilde\Y_1(t,\eta)-\tilde\Y_1(t,\theta)) } 
\tilde\P_1\tilde\U_{1,\eta}(t,\theta)d\theta \Big)d\eta \\
&\quad-\frac{1}{\A^6}\int_0^{1} (\sqP{1}-\sqP{2})\Big(\sqDP{2}(t, \eta)\int_0^{\eta}
 e^{-\frac{1}{\sqtC{2}}( \tilde\Y_2(t,\eta)-\tilde\Y_2(t,\theta)) }\tilde\D_2\tilde\U_{2,\eta}(t,\theta)d\theta \\
& \qquad -\sqDP{1}(t, \eta)\int_0^{\eta} e^{-\frac{1}{\sqtC{1}}( \tilde\Y_1(t,\eta)-\tilde\Y_1(t,\theta)) } 
\tilde\D_1\tilde\U_{1,\eta}(t,\theta)d\theta \Big)d\eta\\
&=\frac{1}{\A^6}\int_0^{1} (\sqP{1}-\sqP{2})(\sqtC{2}\tilde\P_2 \tilde\U_2\sqDP{2}- \sqtC{1}\tilde\P_1\tilde\U_1\sqDP{1})(t, \eta)d\eta \\
&\quad-\frac{1}{\A^6}\int_0^{1} (\sqP{1}-\sqP{2})(\tilde\D_2 \tilde\U_2\sqDP{2}- \tilde\D_1\tilde\U_1\sqDP{1})(t, \eta)d\eta \\
&\quad+\frac{1}{\A^6}\int_0^{1} (\sqP{1}-\sqP{2})\\
&\qquad\times\Big(\sqDP{2}(t, \eta)\int_0^{\eta} e^{-\frac{1}{\sqtC{2}}( \tilde\Y_2(t,\eta)-\tilde\Y_2(t,\theta))}\frac{1}{\sqtC{2}}\tilde\D_2\tilde\U_{2}\tilde\Y_{2,\eta}(t,\theta)d\theta\\
&\qquad\qquad-\sqDP{1}(t, \eta)\int_0^{\eta} e^{-\frac{1}{\sqtC{1}}( \tilde\Y_1(t,\eta)-\tilde\Y_1(t,\theta))}\frac{1}{\sqtC{1}}\tilde\D_1\tilde\U_{1}\tilde\Y_{1,\eta}(t,\theta)d\theta\Big)d\eta\\
&\quad-\frac{3}{\A^6}\int_0^{1} (\sqP{1}-\sqP{2})\\
&\qquad\times\Big(\sqDP{2}(t, \eta)\int_0^{\eta} e^{-\frac{1}{\sqtC{2}}( \tilde\Y_2(t,\eta)-\tilde\Y_2(t,\theta))}\tilde\P_2\tilde\U_{2}\tilde\Y_{2,\eta}(t,\theta)d\theta\\
&\qquad\qquad-\sqDP{1}(t, \eta)\int_0^{\eta} e^{-\frac{1}{\sqtC{1}}( \tilde\Y_1(t,\eta)-\tilde\Y_1(t,\theta))}\tilde\P_1\tilde\U_{1}\tilde\Y_{1,\eta}(t,\theta)d\theta\Big)d\eta\\
&\quad+\frac{1}{\A^6}\int_0^{1} (\sqP{1}-\sqP{2})\Big(\sqDP{2}(t, \eta)\int_0^{\eta} e^{-\frac{1}{\sqtC{2}}( \tilde\Y_2(t,\eta)-\tilde\Y_2(t,\theta))}\tilde\U_2^3\tilde\Y_{2,\eta}(t,\theta)d\theta\\
&\qquad-\sqDP{1}(t, \eta)\int_0^{\eta} e^{-\frac{1}{\sqtC{1}}( \tilde\Y_1(t,\eta)-\tilde\Y_1(t,\theta))}\tilde\U_1^3\tilde\Y_{1,\eta}(t,\theta)d\theta\Big)d\eta\\
&\quad+\frac{1}{2\A^6}\int_0^{1} (\sqP{1}-\sqP{2})\Big(\sqDP{2}(t, \eta)\int_0^{\eta} e^{-\frac{1}{\sqtC{2}}( \tilde\Y_2(t,\eta)-\tilde\Y_2(t,\theta))}\sqtC{2}^5\tilde\U_2(t,\theta)d\theta\\
&\qquad-\sqDP{1}(t, \eta)\int_0^{\eta} e^{-\frac{1}{\sqtC{1}}( \tilde\Y_1(t,\eta)-\tilde\Y_1(t,\theta))}\sqtC{1}^5\tilde\U_1(t,\theta)d\theta\Big)d\eta\\
&= W_1+W_2+W_3+W_4+W_5+W_6.
\end{align*}
Here we have used the rewrite employed when manipulating the term $\bar K_1$ from the expression \eqref{eq:barK1_1} to \eqref{eq:barK1_2}.
We start by considering the term $W_1$:
\begin{align*}
W_1&=\frac{1}{\A^6}\int_0^{1} (\sqP{1}-\sqP{2})(\sqtC{2}\tilde\P_2 \tilde\U_2\sqDP{2}- \sqtC{1}\tilde\P_1 \tilde\U_1\sqDP{1})(t, \eta)d\eta \\
&= \mathbbm{1}_{\sqtC{1}\leq \sqtC{2}}\frac{\sqtC{2}-\sqtC{1}}{\A^6}\int_0^{1} (\sqP{1}-\sqP{2})\tilde\P_2 \tilde\U_2\sqDP{2}(t, \eta)d\eta \\
&\quad + \mathbbm{1}_{\sqtC{2}< \sqtC{1}}\frac{\sqtC{2}-\sqtC{1}}{\A^6}\int_0^{1} (\sqP{1}-\sqP{2})\tilde\P_1 \tilde\U_1\sqDP{1}(t, \eta)d\eta \\
&\quad +\frac{\ma}{\A^6}\int_0^{1} (\sqP{1}-\sqP{2})(\tilde\P_2-\tilde\P_1)   \tilde\U_2\sqDP{2}(t, \eta)\mathbbm{1}_{\tilde\P_1\le\tilde\P_2}(t,\eta)d\eta \\
&\quad+\frac{\ma}{\A^6}\int_0^{1} (\sqP{1}-\sqP{2})(\tilde\P_2-\tilde\P_1)   \tilde\U_1\sqDP{1}(t, \eta)\mathbbm{1}_{\tilde\P_2<\tilde\P_1}(t,\eta)d\eta \\
&\quad+\frac{\ma}{\A^6}\int_0^{1}(\sqP{1}-\sqP{2})(\sqDP{2}-\sqDP{1})  \min_j( \tilde\P_j) \tilde\U_2(t, \eta)d\eta \\
&\quad+\frac{\ma}{\A^6}\int_0^{1}(\sqP{1}-\sqP{2})(\tilde\U_2-\tilde\U_1) \min_j( \tilde\P_j) \sqDP{1}(t, \eta)d\eta \\
&=W_{11}+W_{12}+W_{13}+W_{14}+W_{15}+W_{16}.
\end{align*}
Using \eqref{eq:all_estimatesA}, \eqref{eq:all_estimatesB}, \eqref{eq:all_estimatesE}, \eqref{eq:all_Pderiv_estimatesA}, and \eqref{eq:all_Pderiv_estimatesB}, we find that
\begin{multline*}
W_{11}+W_{12}+W_{13}+W_{14}+W_{16}\\
\le \bigO(1)\Big(\norm{\tilde\U_1-\tilde\U_2}^2+\norm{\sqP{1}-\sqP{2}}^2+\vert \sqtC{1}-\sqtC{2}\vert^2\Big).
\end{multline*}
Regarding the term $W_{15}$:
\begin{align*}
\abs{W_{15}}&=\frac{\ma}{\A^6}\vert\int_0^{1}(\sqP{1}-\sqP{2})(\sqDP{2}-\sqDP{1})  \min_j( \tilde\P_j) \tilde\U_2(t, \eta)d\eta\vert \\
&=\vert\frac{\ma}{2\A^6}(\sqP{1}-\sqP{2})^2 \min_j( \tilde\P_j) \tilde\U_2(t, \eta)\Big\vert_{0}^1 \\
&\qquad-\frac{\ma}{2\A^6}\int_0^{1}(\sqP{1}-\sqP{2})^2\frac{d}{d\eta}\Big( \min_j( \tilde\P_j) \tilde\U_2\Big)(t, \eta)d\eta\vert\\
&=\frac{\ma}{2\A^6}\vert\int_0^{1}(\sqP{1}-\sqP{2})^2\frac{d}{d\eta}\Big( \min_j( \tilde\P_j) \tilde\U_2\Big)(t, \eta)d\eta\vert\\
&\le \bigO(1)\norm{\sqP{1}-\sqP{2}}^2,
\end{align*}
using the same estimates as for $\bar B_{13}$ (cf.~\eqref{eq:barB13}) and Lemma \ref{lemma:3}(ii).

As for the term $W_2$ we find
\begin{align*}
-W_2&= \frac{1}{\A^6}\int_0^{1} (\sqP{1}-\sqP{2})(\tilde\D_2 \tilde\U_2\sqDP{2}- \tilde\D_1 \tilde\U_1\sqDP{1})(t, \eta)d\eta \\
&=\frac{1}{\A^6}\int_0^{1} (\sqP{1}-\sqP{2})(\tilde\D_2-\tilde\D_1)   \tilde\U_1\sqDP{1}\mathbbm{1}_{\tilde\D_2\le\tilde\D_1}(t,\eta)d\eta \\
&\quad+\frac{1}{\A^6}\int_0^{1} (\sqP{1}-\sqP{2})(\tilde\D_2-\tilde\D_1)   \tilde\U_2\sqDP{2}\mathbbm{1}_{\tilde\D_1<\tilde\D_2}(t,\eta)d\eta \\
&\quad+\frac{1}{\A^6}\int_0^{1} (\sqP{1}-\sqP{2})(\sqDP{2}-\sqDP{1})  \min_j( \tilde\D_j) \tilde\U_2(t, \eta)d\eta \\
&\quad-\frac{1}{\A^6}\int_0^{1} (\sqP{1}-\sqP{2})  (\tilde\U_1-\tilde\U_2) \min_j( \tilde\D_j) \sqDP{1}(t, \eta)d\eta \\
&=W_{21}+W_{22}+W_{23}+W_{24}.
\end{align*}
The terms $W_{21}$ and $W_{22}$ can be treated similarly.  We need to estimate $\tilde\D_2-\tilde\D_1$. Applying  Lemma \ref{lemma:D}, we have
\begin{align*}
\abs{W_{21}}&= \frac{1}{\A^6}\vert\int_0^{1} (\sqP{1}-\sqP{2})(\tilde\D_2-\tilde\D_1)   \tilde\U_1\sqDP{1}(t, \eta)\mathbbm{1}_{\tilde\D_2\le\tilde\D_1}(t,\eta)d\eta\vert \\
&\leq \frac{2}{\A^{9/2}}\int_0^{1} \vert\sqP{1}-\sqP{2}\vert\, \vert \tilde\Y_1-\tilde\Y_2\vert  \tilde\D_1^{1/2}\vert\tilde\U_1\sqDP{1}\vert(t, \eta)d\eta \\
& \quad +\frac{1}{\A^6}\int_0^{1} \vert\sqP{1}-\sqP{2}\vert\, \vert \tilde \Y_1-\tilde\Y_2\vert (\tilde \U_1^2+\tilde \P_1)\vert\tilde\U_1\sqDP{1}\vert(t,\eta) d\eta\\
& \quad +  \frac{2\sqrt{2}}{\A^{9/2}}\norm{\tilde\Y_1-\tilde\Y_2}\int_0^{1}\vert\sqP{1}-\sqP{2}\vert\tilde \D_1^{1/2}\vert\tilde\U_1\sqDP{1}\vert(t,\eta)d\eta \\
& \quad  + \frac{4}{\A^3}\int_0^{1}\vert\sqP{1}-\sqP{2}\vert \\
&\qquad\times\Big(\int_0^\eta e^{-\frac{1}{\A}(\tilde\Y_1(t,\eta)-\tilde\Y_1(t,\theta))}(\tilde\U_1-\tilde\U_2)^2(t,\theta)d\theta\Big)^{1/2}\vert\tilde\U_1\sqDP{1}\vert(t,\eta)d\eta\\
& \quad + \frac{2\sqrt{2}}{\A^3} \int_0^{1}\vert\sqP{1}-\sqP{2}\vert \\
&\qquad\times\Big(\int_0^\eta e^{-\frac{1}{\A}(\tilde \Y_1(t,\eta)-\tilde\Y_1(t,\theta))} (\sqP{1}-\sqP{2})^2(t,\theta) d\theta\Big)^{1/2}\vert\tilde\U_1\sqDP{1}\vert(t,\eta)d\eta\\
& \quad + \frac{3}{\sqrt{2}\A^2} \int_0^{1}\vert\sqP{1}-\sqP{2}\vert \\
&\qquad\times\Big(\int_0^\eta e^{-\frac{1}{\ma}(\tilde \Y_1(t,\eta)-\tilde\Y_1(t,\theta))} (\tilde\Y_1-\tilde\Y_2)^2(t,\theta) d\theta\Big)^{1/2}\vert\tilde\U_1\sqDP{1}\vert(t,\eta)d\eta\\
& \quad +\frac{3}{2\A^2}\int_0^{1}\vert\sqP{1}-\sqP{2}\vert \\
&\qquad\times\Big(\int_0^\eta e^{-\frac{1}{\ma}(\tilde\Y_1(t,\eta)-\tilde\Y_1(t,\theta))} \vert\tilde\Y_1-\tilde\Y_2\vert(t,\theta) d\theta\Big)\vert\tilde\U_1\sqDP{1}\vert(t,\eta)d\eta\\
& \quad + \frac{6}{\A^2}\vert\sqtC{1}-\sqtC{2}\vert\int_0^{1}\vert\sqP{1}-\sqP{2}\vert \\
&\qquad\qquad\qquad\times \Big(\int_0^\eta e^{-\frac{3}{4\A}(\tilde\Y_1(t,\eta)-\tilde\Y_1(t,\theta))} d\theta\Big)\vert\tilde\U_1\sqDP{1}\vert(t,\eta)d\eta \\
& \quad + \frac{12\sqrt{2}}{\sqrt{3}e\A^2}\vert\sqtC{1}-\sqtC{2}\vert\int_0^{1}\vert\sqP{1}-\sqP{2}\vert \\
&\qquad\qquad\qquad\times \Big(\int_0^\eta e^{-\frac{3}{4\A}(\tilde\Y_1(t,\eta)-\tilde\Y_1(t,\theta))} d\theta\Big)^{1/2}\vert\tilde\U_1\sqDP{1}\vert(t,\eta)d\eta \\
&\le \frac{2}{\A^{9/2}} \int_0^{1} \vert\sqP{1}-\sqP{2}\vert\, \vert \tilde\Y_1-\tilde\Y_2\vert(t,\eta)
 \frac{\A^{13/2}}{4} d\eta \\
&\quad+ \frac{1}{\A^6}\int_0^{1} \vert\sqP{1}-\sqP{2}\vert\, \vert \tilde \Y_1-\tilde\Y_2\vert (t,\eta)\frac{3\A^8}{8\sqrt{2}}d\eta\\
& \quad +  \frac{2\sqrt{2}}{\A^{9/2}}\norm{\tilde\Y_1-\tilde\Y_2}\int_0^{1}\vert\sqP{1}-\sqP{2}\vert(t,\eta)\frac{\A^{13/2}}{4}d\eta \\
&\quad+\frac{4}{\A^3}\int_0^{1}\vert\sqP{1}-\sqP{2}\vert(t,\eta) \Big(\int_0^\eta (\tilde\U_1-\tilde\U_2)^2(t,\theta)d\theta\Big)^{1/2} \frac{\A^4}{2\sqrt{2}} d\eta\\
&\quad + \frac{2\sqrt{2}}{\A^3} \int_0^{1}\vert\sqP{1}-\sqP{2}\vert(t,\eta)\Big(\int_0^\eta  (\sqP{1}-\sqP{2})^2(t,\theta) d\theta\Big)^{1/2}\frac{\A^4}{2\sqrt{2}} d\eta\\
&\quad + \frac{3}{\sqrt{2}\A^2} \int_0^{1}\vert\sqP{1}-\sqP{2}\vert(t,\eta)\Big(\int_0^\eta  (\tilde\Y_1-\tilde\Y_2)^2(t,\theta) d\theta\Big)^{1/2}\frac{\A^4}{2\sqrt{2}} d\eta\\
&\quad +\frac{3}{2\A^2}\int_0^{1}\vert\sqP{1}-\sqP{2}\vert(t,\eta)\Big(\int_0^\eta \vert\tilde\Y_1-\tilde\Y_2\vert(t,\theta) d\theta\Big)\frac{\A^4}{2\sqrt{2}} d\eta\\
&\quad +\frac{6}{\A^2}\vert\sqtC{1}-\sqtC{2}\vert\int_0^{1}\vert\sqP{1}-\sqP{2}\vert(t,\eta) \frac{\A^4}{2\sqrt{2}}d\eta \\
&\quad +\frac{12\sqrt{2}}{\sqrt{3}e\A^2}\vert\sqtC{1}-\sqtC{2}\vert\int_0^{1}\vert\sqP{1}-\sqP{2}\vert(t,\eta) \frac{\A^4}{2\sqrt{2}}d\eta \\
&\le \bigO(1)\big(\norm{\tilde\U_1-\tilde\U_2}^2+\norm{\sqP{1}-\sqP{2}}^2+\norm{\tilde\Y_1-\tilde\Y_2}^2+\abs{\sqtC{1}-\sqtC{2}}^2\big),
\end{align*}
where we have used estimates \eqref{eq:all_estimatesA}, \eqref{eq:all_estimatesB},  \eqref{eq:all_Pderiv_estimatesB}, and \eqref{eq:all_estimatesN}.
Furthermore,
\begin{align*}
\abs{W_{23}}&=\frac{1}{\A^6}\vert\int_0^{1} (\sqP{1}-\sqP{2})(\sqDP{2}-\sqDP{1})  \min_j( \tilde\D_j) \tilde\U_2(t, \eta)d\eta\vert \\
&\le \vert \frac{1}{2\A^6}(\sqP{1}-\sqP{2})^2 \min_j( \tilde\D_j) \tilde\U_2(t, \eta)\Big\vert_{0}^1\\
&\qquad-  \frac{1}{2\A^6}\int_0^{1} 
(\sqP{1}-\sqP{2})^2 \frac{d}{d\theta}
\big(\min_j( \tilde\D_j) \tilde\U_2\big) (t, \eta)d\eta\vert\\
&=\frac{1}{2\A^6}\vert\int_0^{1} (\sqP{1}-\sqP{2})^2 \frac{d}{d\theta}
\big(\min_j(\tilde\D_j) \tilde\U_2\big) (t, \eta)d\eta \vert \\
&\le \bigO(1) \norm{\sqP{1}-\sqP{2}}^2,
\end{align*}
by applying  Lemma \ref{lemma:5} (ii). 
The term $W_{24}$  goes as follows:
\begin{align*}
\abs{W_{24}}&= \frac{1}{\A^6}\vert\int_0^{1}(\sqP{1}-\sqP{2}) (\tilde\U_1-\tilde\U_2) \min_j( \tilde\D_j) \sqDP{1}(t, \eta)d\eta\vert \\
&\le \bigO(1)\big(\norm{\tilde\U_1-\tilde\U_2}^2+\norm{\sqP{1}-\sqP{2}}^2\big)
\end{align*}
using 
\begin{align*}
\min_j( \tilde\D_j) \vert\sqDP{1}\vert&\le 2\sqtC{1}\tilde\P_1\vert\sqDP{1}\vert\le \sqP{1} \tilde\P_1\tilde\Y_{1,\eta} \le \frac{\A^7}{4}
\end{align*}
from \eqref{eq:all_estimatesA}, \eqref{eq:all_estimatesE}, \eqref{eq:all_estimatesN}, and \eqref{eq:all_Pderiv_estimatesA}.

Next comes $W_3$, namely
\begin{align*}
W_3&=\frac{1}{\A^6}\int_0^{1} (\sqP{1}-\sqP{2})\\
&\qquad\times\Big(\sqDP{2}(t, \eta)\int_0^{\eta} e^{-\frac{1}{\sqtC{2}}( \tilde\Y_2(t,\eta)-\tilde\Y_2(t,\theta))}\frac{1}{\sqtC{2}}\tilde\D_2\tilde\U_{2}\tilde\Y_{2,\eta}(t,\theta)d\theta\\
&\qquad\qquad-\sqDP{1}(t, \eta)\int_0^{\eta} e^{-\frac{1}{\sqtC{1}}( \tilde\Y_1(t,\eta)-\tilde\Y_1(t,\theta))}\frac{1}{\sqtC{1}}\tilde\D_1\tilde\U_{1}\tilde\Y_{1,\eta}(t,\theta)d\theta\Big)d\eta\\
&=\frac{1}{\A^6}\int_0^{1} (\sqP{1}-\sqP{2})\\
&\qquad\times\Big(\sqDP{2}(t, \eta)\int_0^{\eta} e^{-\frac{1}{\sqtC{2}}( \tilde\Y_2(t,\eta)-\tilde\Y_2(t,\theta))}\frac{1}{\sqtC{2}}\tilde\D_2\tilde\U_{2}^+\tilde\Y_{2,\eta}(t,\theta)d\theta\\
&\qquad\qquad-\sqDP{1}(t, \eta)\int_0^{\eta} e^{-\frac{1}{\sqtC{1}}( \tilde\Y_1(t,\eta)-\tilde\Y_1(t,\theta))}\frac{1}{\sqtC{1}}\tilde\D_1\tilde\U_{1}^+\tilde\Y_{1,\eta}(t,\theta)d\theta\Big)d\eta\\
&\quad+\frac{1}{\A^6}\int_0^{1} (\sqP{1}-\sqP{2})\\
&\qquad\times\Big(\sqDP{2}(t, \eta)\int_0^{\eta} e^{-\frac{1}{\sqtC{2}}( \tilde\Y_2(t,\eta)-\tilde\Y_2(t,\theta))}\frac{1}{\sqtC{2}}\tilde\D_2\tilde\U_{2}^-\tilde\Y_{2,\eta}(t,\theta)d\theta\\
&\qquad\qquad-\sqDP{1}(t, \eta)\int_0^{\eta} e^{-\frac{1}{\sqtC{1}}( \tilde\Y_1(t,\eta)-\tilde\Y_1(t,\theta))}\frac{1}{\sqtC{1}}\tilde\D_1\tilde\U_{1}^-\tilde\Y_{1,\eta}(t,\theta)d\theta\Big)d\eta\\
&=W_{31}+W_{32};
\end{align*}
and the two terms can be treated in the same manner. Thus
\begin{align*}
W_{31}&=\frac{1}{\A^6}\int_0^{1} (\sqP{1}-\sqP{2})\\
&\qquad\times\Big(\sqDP{2}(t, \eta)\int_0^{\eta} e^{-\frac{1}{\sqtC{2}}( \tilde\Y_2(t,\eta)-\tilde\Y_2(t,\theta))}\frac{1}{\sqtC{2}}\tilde\D_2\tilde\U_{2}^+\tilde\Y_{2,\eta}(t,\theta)d\theta\\
&\qquad\qquad-\sqDP{1}(t, \eta)\int_0^{\eta} e^{-\frac{1}{\sqtC{1}}( \tilde\Y_1(t,\eta)-\tilde\Y_1(t,\theta))}\frac{1}{\sqtC{1}}\tilde\D_1\tilde\U_{1}^+\tilde\Y_{1,\eta}(t,\theta)d\theta\Big)d\eta\\
&= \mathbbm{1}_{\sqtC{1}\leq \sqtC{2}}\frac{1}{\A^6}\Big(\frac{1}{\sqtC{2}}-\frac{1}{\sqtC{1}}\Big) \int_0^{1} (\sqP{1}-\sqP{2})\sqDP{1}(t, \eta)\\
&\qquad\qquad\qquad\times\Big(\int_0^{\eta} e^{-\frac{1}{\sqtC{1}}( \tilde\Y_1(t,\eta)-\tilde\Y_1(t,\theta))}\tilde\D_1\tilde\U_{1}^+\tilde\Y_{1,\eta}(t,\theta)d\theta\Big)d\eta\\
& \quad +  \mathbbm{1}_{\sqtC{2}< \sqtC{1}}\frac{1}{\A^6}\Big(\frac{1}{\sqtC{2}}-\frac{1}{\sqtC{1}}\Big) \int_0^{1} (\sqP{1}-\sqP{2})\sqDP{2}(t, \eta)\\
&\qquad\qquad\qquad\times\Big(\int_0^{\eta} e^{-\frac{1}{\sqtC{2}}( \tilde\Y_2(t,\eta)-\tilde\Y_2(t,\theta))}\tilde\D_2\tilde\U_{2}^+\tilde\Y_{2,\eta}(t,\theta)d\theta\Big) d\eta\\
&\quad +\frac{1}{\A^7}\int_0^{1} (\sqP{1}-\sqP{2})\sqDP{2}(t, \eta)\\
&\qquad\times\Big(\int_0^{\eta} e^{-\frac{1}{\sqtC{2}}( \tilde\Y_2(t,\eta)-\tilde\Y_2(t,\theta))}(\tilde\D_2-\tilde\D_1)\tilde\U_{2}^+\tilde\Y_{2,\eta}(t,\theta)
\mathbbm{1}_{\tilde\D_1\le\tilde\D_2}(t,\theta)d\theta\Big)d\eta\\
&\quad+\frac{1}{\A^7}\int_0^{1} (\sqP{1}-\sqP{2})\sqDP{1}(t, \eta)\\
&\qquad\times\Big(\int_0^{\eta} e^{-\frac{1}{\sqtC{1}}( \tilde\Y_1(t,\eta)-\tilde\Y_1(t,\theta))}(\tilde\D_2-\tilde\D_1)\tilde\U_{1}^+\tilde\Y_{1,\eta}(t,\theta)
\mathbbm{1}_{\tilde\D_2<\tilde\D_1}(t,\theta)d\theta\Big)d\eta\\
&\quad+\frac{1}{\A^7}\int_0^{1} (\sqP{1}-\sqP{2})\sqDP{2}(t, \eta)\Big(\int_0^{\eta}e^{-\frac{1}{\sqtC{2}}( \tilde\Y_2(t,\eta)-\tilde\Y_2(t,\theta))}\\
&\qquad\qquad\qquad\qquad\qquad\times\min_j(\tilde\D_j)(\tilde\U_2^+-\tilde\U_1^+)\tilde\Y_{2,\eta}(t,\theta)
\mathbbm{1}_{\tilde\U_1^+\le\tilde\U_2^+}(t,\theta)d\theta\Big)d\eta\\
&\quad+\frac{1}{\A^7}\int_0^{1} (\sqP{1}-\sqP{2})\sqDP{1}(t, \eta)\Big(\int_0^{\eta}e^{-\frac{1}{\sqtC{1}}( \tilde\Y_1(t,\eta)-\tilde\Y_1(t,\theta))}\\
&\qquad\qquad\qquad\qquad\qquad\times\min_j(\tilde\D_j)(\tilde\U_2^+-\tilde\U_1^+)\tilde\Y_{1,\eta}(t,\theta)
\mathbbm{1}_{\tilde\U_2^+<\tilde\U_1^+}(t,\theta)d\theta\Big)d\eta\\
& \quad + \mathbbm{1}_{\sqtC{1}\leq \sqtC{2}}\frac{1}{\A^7}\int_0^{1} (\sqP{1}-\sqP{2})\sqDP{2}(t, \eta)\\
&\qquad\qquad\qquad\qquad\times\Big(\int_0^{\eta}(e^{-\frac{1}{\sqtC{2}}( \tilde\Y_2(t,\eta)-\tilde\Y_2(t,\theta))}-e^{-\frac{1}{\sqtC{1}}( \tilde\Y_2(t,\eta)-\tilde\Y_2(t,\theta))})\\
&\qquad\qquad\qquad\qquad\qquad\times
\min_j(\tilde\D_j)\min_j(\tilde\U_j^+)\tilde\Y_{2,\eta}(t,\theta)d\theta\Big)d\eta\\
& \quad + \mathbbm{1}_{\sqtC{2}< \sqtC{1}}\frac{1}{\A^7}\int_0^{1} (\sqP{1}-\sqP{2})\sqDP{1}(t, \eta)\\
&\qquad\qquad\qquad\qquad\times\Big(\int_0^{\eta}(e^{-\frac{1}{\sqtC{2}}( \tilde\Y_1(t,\eta)-\tilde\Y_1(t,\theta))}-e^{-\frac{1}{\sqtC{1}}( \tilde\Y_1(t,\eta)-\tilde\Y_1(t,\theta))})\\
&\qquad\qquad\qquad\qquad\qquad\times
\min_j(\tilde\D_j)\min_j(\tilde\U_j^+)\tilde\Y_{1,\eta}(t,\theta)d\theta\Big)d\eta\\
&\quad+\frac{1}{\A^7}\int_0^{1} (\sqP{1}-\sqP{2})\sqDP{2}(t, \eta)\\
&\qquad\qquad\times\Big(\int_0^{\eta}(e^{-\frac{1}{\ma}( \tilde\Y_2(t,\eta)-\tilde\Y_2(t,\theta))}-e^{-\frac{1}{\ma}( \tilde\Y_1(t,\eta)-\tilde\Y_1(t,\theta))})\\
&\qquad\qquad\qquad\qquad\qquad\qquad\times
\min_j(\tilde\D_j)\min_j(\tilde\U_j^+)\tilde\Y_{2,\eta}(t,\theta)\mathbbm{1}_{B(\eta)}(t,\theta)d\theta\Big)d\eta\\
&\quad+\frac{1}{\A^7}\int_0^{1} (\sqP{1}-\sqP{2})\sqDP{1}(t, \eta)\\
&\qquad\qquad\qquad\times\Big(\int_0^{\eta}(e^{-\frac{1}{\ma}(\tilde\Y_2(t,\eta)-\tilde\Y_2(t,\theta))}-e^{-\frac{1}{\ma}( \tilde\Y_1(t,\eta)-\tilde\Y_1(t,\theta))})\\
&\qquad\qquad\qquad\qquad\qquad\qquad\times
\min_j(\tilde\D_j)\min_j(\tilde\U_j^+)\tilde\Y_{1,\eta}(t,\theta)\mathbbm{1}_{B(\eta)^c}(t,\theta)d\theta\Big)d\eta\\
&\quad+\frac{1}{\A^7}\int_0^{1} (\sqP{1}-\sqP{2})(\sqDP{2}-\sqDP{1})(t, \eta)\\
&\quad\qquad\times\min_k\Big(\int_0^{\eta}\min_j(e^{-\frac{1}{\ma}( \tilde\Y_j(t,\eta)-\tilde\Y_j(t,\theta))})
\min_j(\tilde\D_j)\min_j(\tilde\U_j^+)\tilde\Y_{k,\eta}(t,\theta)d\theta\Big)d\eta\\
&\quad+\frac{1}{\A^7}\int_0^{1} (\sqP{1}-\sqP{2})\sqDP{2}\mathbbm{1}_{D^c}(t, \eta)\\
&\quad\times\Big(\int_0^{\eta}\min_j(e^{-\frac{1}{\ma}( \tilde\Y_j(t,\eta)-\tilde\Y_j(t,\theta))})
\min_j(\tilde\D_j)\min_j(\tilde\U_j^+)(\tilde\Y_{2,\eta}-\tilde\Y_{1,\eta})(t,\theta)d\theta\Big)d\eta\\
&\quad+\frac{1}{\A^7}\int_0^{1} (\sqP{1}-\sqP{2})\sqDP{1}\mathbbm{1}_{e}(t, \eta)\\
&\quad\times\Big(\int_0^{\eta}\min_j(e^{-\frac{1}{\ma}( \tilde\Y_j(t,\eta)-\tilde\Y_j(t,\theta))})
\min_j(\tilde\D_j)\min_j(\tilde\U_j^+)(\tilde\Y_{2,\eta}-\tilde\Y_{1,\eta})(t,\theta)d\theta\Big)d\eta\\
&=Z_1+Z_2+Z_3+Z_4+Z_5+Z_6+Z_7+Z_8+Z_9+Z_{10}+Z_{11}+Z_{12}+Z_{13}, 
\end{align*}
where the set $e$ is given by \eqref{eq:set_e}.  Here there is no alternative but to treat these terms more or less separately. 
For the terms $Z_{1}$ and $Z_2$ we find
\begin{align*}
\abs{Z_1}&\leq \frac{2}{\A^7} \vert \sqtC{1}-\sqtC{2}\vert \int_0^1 \vert \sqP{1}-\sqP{2}\vert \vert\sqDP{1}\vert\\
&\qquad\times\Big(\int_0^\eta e^{-\frac{1}{\sqtC{1}}(\tilde \Y_1(t,\eta)-\tilde\Y_1(t,\theta))}\tilde \U_1^2\tilde \Y_{1,\eta}(t,\theta) d\theta\Big)^{1/2}\\
& \qquad \qquad \qquad \times \Big(\int_0^\eta e^{-\frac{1}{\sqtC{1}}(\tilde \Y_1(t,\eta)-\tilde \Y_1(t,\theta))} \tilde \P_1^2\tilde \Y_{1,\eta}(t,\theta)d\theta\Big)^{1/2}d\eta\\
& \leq\bigO(1)(\norm{\sqP{1}-\sqP{2}}^2+\vert \sqtC{1}-\sqtC{2}\vert^2).
\end{align*}

For the terms $Z_3$  and $Z_4$ we find
\begin{align*}
\abs{Z_{3}}&=\frac{1}{\A^7}\vert\int_0^{1} (\sqP{1}-\sqP{2})\sqDP{2}(t, \eta)\\
&\quad\times\Big(\int_0^{\eta} e^{-\frac{1}{\sqtC{2}}( \tilde\Y_2(t,\eta)-\tilde\Y_2(t,\theta))}
(\tilde\D_2-\tilde\D_1)\tilde\U_{2}^+\tilde\Y_{2,\eta}(t,\theta)
\mathbbm{1}_{\tilde\D_1\le\tilde\D_2}(t,\theta)d\theta\Big)d\eta\vert\\
&\le \bigO(1)
\big(\norm{\tilde\U_2-\tilde\U_1}^2+ \norm{\tilde\Y_2-\tilde\Y_1}^2+\norm{\sqP{2}-\sqP{1}}^2+\abs{\sqtC{2}-\sqtC{1}}^2\big)
\end{align*}
by Lemma~\ref{lemma:D}.

For the terms $Z_5$  and $Z_6$ we find
\begin{align*}
\abs{Z_{5}}&=\frac{1}{\A^7}\vert\int_0^{1} (\sqP{1}-\sqP{2})\sqDP{2}(t, \eta)\\
&\qquad\times\Big(\int_0^{\eta}e^{-\frac{1}{\sqtC{2}}( \tilde\Y_2(t,\eta)-\tilde\Y_2(t,\theta))}\min_j(\tilde\D_j)(\tilde\U_2^+-\tilde\U_1^+)\tilde\Y_{2,\eta}
\mathbbm{1}_{\tilde\U_1^+\le\tilde\U_2^+}(t,\theta)d\theta\Big)d\eta\vert\\
&\le \norm{\sqP{1}-\sqP{2}}^2+\frac1{\A^{14}}\int_0^{1} (\sqDP{2})^2(t, \eta)\\
&\qquad\times\Big(\int_0^{\eta}e^{-\frac{1}{\sqtC{2}}( \tilde\Y_2(t,\eta)-\tilde\Y_2(t,\theta))}\min_j(\tilde\D_j)(\tilde\U_2^+-\tilde\U_1^+)
\tilde\Y_{2,\eta} \mathbbm{1}_{\tilde\U_1^+\le\tilde\U_2^+}(t,\theta)d\theta\Big)^2d\eta\\
&\le \bigO(1)\big(\norm{\sqP{1}-\sqP{2}}^2+\norm{\tilde\U_2-\tilde\U_1}^2\big),
\end{align*}
by applying \eqref{eq:343}, estimating $\tilde\D_2\leq 2\sqtC{2}\tilde\P_1$ (cf. \eqref{eq:all_estimatesN}), and subsequently \eqref{eq:all_Pderiv_estimatesA}

For the terms $Z_7$ and $Z_8$ we find
\begin{align*}
\abs{Z_7}&=\mathbbm{1}_{\sqtC{1}\leq \sqtC{2}}\frac{1}{\A^7}\vert\int_0^{1} (\sqP{1}-\sqP{2})\sqDP{2}(t, \eta)\\
&\qquad\times\Big(\int_0^{\eta}(e^{-\frac{1}{\sqtC{2}}( \tilde\Y_2(t,\eta)-\tilde\Y_2(t,\theta))}-e^{-\frac{1}{\sqtC{1}}( \tilde\Y_2(t,\eta)-\tilde\Y_2(t,\theta))})\\
&\qquad\qquad\qquad\qquad\qquad\qquad\qquad\qquad\times
\min_j(\tilde\D_j)\min_j(\tilde\U_j^+)\tilde\Y_{2,\eta}(t,\theta)d\theta\Big)d\eta\vert\\
& \leq \frac{4}{\sqtC{1}\A^7e}\vert\int_0^{1} \vert\sqP{1}-\sqP{2}\vert\vert\sqDP{2}\vert(t, \eta)\\
&\qquad\times\Big(\int_0^{\eta}e^{-\frac{3}{4\sqtC{2}}( \tilde\Y_2(t,\eta)-\tilde\Y_2(t,\theta))}\tilde\D_2\vert\tilde\U_1\vert\tilde\Y_{2,\eta}(t,\theta)d\theta\Big)d\eta\vert\vert \sqtC{1}-\sqtC{2}\vert \\
& \leq \frac{4\sqrt{2}\ma}{\A^6e}\int_0^1 \vert \sqP{1}-\sqP{2}\vert \vert\sqDP{2}\vert(t,\eta)\\
&\qquad\qquad\qquad\times \Big(\int_0^\eta e^{-\frac{3}{2\sqtC{2}}(\tilde\Y_2(t,\eta)-\tilde \Y_2(t,\theta))}\tilde \P_2\tilde \Y_{2,\eta}(t,\theta) d\theta\Big)^{1/2}\\
& \qquad \qquad \qquad\qquad\qquad \times\Big(\int_0^\eta \tilde \P_2\tilde \Y_{2,\eta}(t,\theta) d\theta\Big)^{1/2} d\eta\vert \sqtC{1}-\sqtC{2}\vert\\
& \leq \bigO(1)(\norm{\sqP{1}-\sqP{2}}^2 +\vert \sqtC{1}-\sqtC{2}\vert^2).
\end{align*}

For the terms $Z_9$  and $Z_{10}$ we find
\begin{align*}
\abs{Z_{9}}&=\frac{1}{\A^7}\vert\int_0^{1} (\sqP{1}-\sqP{2})\sqDP{2}(t, \eta)\\
&\qquad\times\Big(\int_0^{\eta}(e^{-\frac{1}{\ma}( \tilde\Y_2(t,\eta)-\tilde\Y_2(t,\theta))}-e^{-\frac{1}{\ma}( \tilde\Y_1(t,\eta)-\tilde\Y_1(t,\theta))})\\
&\qquad\qquad\qquad\qquad\qquad\times
\min_j(\tilde\D_j)\min_j(\tilde\U_j^+)\tilde\Y_{2,\eta}(t,\theta)\mathbbm{1}_{B(\eta)}(t,\theta)d\theta\Big)d\eta\vert\\
&\le\frac{1}{\ma\A^7}\int_0^{1} \vert\sqP{1}-\sqP{2}\vert\, \vert\sqDP{2}\vert(t, \eta)\\
&\qquad\times\Big(\int_0^{\eta}\big(\abs{\tilde\Y_2- \tilde\Y_1}(t,\eta)+\abs{\tilde\Y_2- \tilde\Y_1}(t,\theta)\big)\\
&\qquad\qquad\qquad\qquad\times e^{-\frac{1}{\sqtC{2}}( \tilde\Y_2(t,\eta)-\tilde\Y_2(t,\theta))}
\tilde\D_2\min_j(\tilde\U_j^+)\tilde\Y_{2,\eta}(t,\theta)d\theta\Big)d\eta\\
&\le\frac{1}{\sqrt{2}\A^6}\int_0^{1} \abs{\sqP{1}-\sqP{2}} \abs{\tilde\Y_2- \tilde\Y_1} \vert\sqDP{2}\vert(t, \eta)\\
&\qquad\times\Big(\int_0^{\eta}e^{-\frac{1}{\sqtC{2}}( \tilde\Y_2(t,\eta)-\tilde\Y_2(t,\theta))}
\tilde\D_2\tilde\Y_{2,\eta}(t,\theta)d\theta\Big)d\eta\\
&\quad+\frac{1}{\sqrt{2}\A^6}\int_0^{1} \abs{\sqP{1}-\sqP{2}} \vert \sqDP{2}\vert(t, \eta)\\
&\qquad\times\Big(\int_0^{\eta}e^{-\frac{1}{\sqtC{2}}( \tilde\Y_2(t,\eta)-\tilde\Y_2(t,\theta))}\abs{\tilde\Y_2- \tilde\Y_1}
\tilde\D_2\tilde\Y_{2,\eta}(t,\theta)d\theta\Big)d\eta\\
&\le\frac{\sqrt{2}}{\A^5}\int_0^{1} \abs{\sqP{1}-\sqP{2}} \abs{\tilde\Y_2- \tilde\Y_1} \vert\sqDP{2}\vert(t, \eta)\\
&\qquad\times\Big(\int_0^{\eta}e^{-\frac{1}{\sqtC{2}}( \tilde\Y_2(t,\eta)-\tilde\Y_2(t,\theta))}
\tilde\P_2^2\tilde\Y_{2,\eta}(t,\theta)d\theta\Big)^{1/2} \\
&\qquad\times\Big(\int_0^{\eta}e^{-\frac{1}{\sqtC{2}}( \tilde\Y_2(t,\eta)-\tilde\Y_2(t,\theta))}
\tilde\Y_{2,\eta}(t,\theta)d\theta\Big)^{1/2} d\eta\\
&\quad+\norm{\sqP{1}-\sqP{2}}^2 \\
&\quad +\frac{2}{\A^{10}}\int_0^{1}(\sqDP{2})^2(t, \eta)\Big(\int_0^{\eta}e^{-\frac{3}{2\sqtC{2}}( \tilde\Y_2(t,\eta)-\tilde\Y_2(t,\theta))} \tilde\P_2^2 \tilde\Y_{2,\eta}^2(t,\theta)d\theta\Big) \\
&\qquad\qquad\qquad\times\Big(\int_0^{\eta}e^{-\frac{1}{2\sqtC{2}}( \tilde\Y_2(t,\eta)-\tilde\Y_2(t,\theta))} (\tilde\Y_2-\tilde\Y_1)^2 (t,\theta)d\theta\Big)d\eta\\
&\le \bigO(1)\big(\norm{\sqP{1}-\sqP{2}}^2+\norm{\tilde\Y_1-\tilde\Y_2}^2 \big), 
\end{align*}
applying the same set of estimates applied when studying $\bar B_{35}$, see \eqref{eq:barB35}.
For the term $Z_{11}$  we find
\begin{align*}
\abs{Z_{11}}&=\frac{1}{\A^7}\vert\int_0^{1} (\sqP{1}-\sqP{2})(\sqDP{2}-\sqDP{1})(t, \eta)\\
&\qquad\times\min_k\Big(\int_0^{\eta}\min_j(e^{-\frac{1}{\ma}( \tilde\Y_j(t,\eta)-\tilde\Y_j(t,\theta))})
\min_j(\tilde\D_j)\min_j(\tilde\U_j^+)\tilde\Y_{k,\eta}(t,\theta)d\theta\Big)d\eta\vert\\
&\le \bigO(1)\norm{\sqP{1}-\sqP{2}}^2,
\end{align*}
see estimates for $\bar B_{37}$, cf.~\eqref{eq:barB37}.

For the terms $Z_{12}$  and $Z_{13}$ we find
\begin{align*}
\abs{Z_{12}}&=\frac{1}{\A^7}\vert\int_0^{1} (\sqP{1}-\sqP{2})\sqDP{2}\mathbbm{1}_{D^c}(t, \eta)\\
&\quad\times\Big(\int_0^{\eta}\min_j(e^{-\frac{1}{\ma}( \tilde\Y_j(t,\eta)-\tilde\Y_j(t,\theta))})
\min_j(\tilde\D_j)\min_j(\tilde\U_j^+)(\tilde\Y_{2,\eta}-\tilde\Y_{1,\eta})(t,\theta)d\theta\Big)d\eta\vert\\
&=\frac{1}{\A^7}\vert\int_0^{1} (\sqP{1}-\sqP{2})\sqDP{2}\mathbbm{1}_{D^c}(t, \eta)\\
&\qquad\times\Big[\Big(\min_j(e^{-\frac{1}{\ma}( \tilde\Y_j(t,\eta)-\tilde\Y_j(t,\theta))})
\min_j(\tilde\D_j)\min_j(\tilde\U_j^+)(\tilde\Y_{2}-\tilde\Y_{1})(t,\theta)\Big)\Big\vert_{\theta=0}^\eta \\
&\quad -\int_0^{\eta} (\tilde\Y_{2}-\tilde\Y_{1})\frac{d}{d\theta} \Big(\min_j(e^{-\frac{1}{\ma}( \tilde\Y_j(t,\eta)-\tilde\Y_j(t,\theta))})
\min_j(\tilde\D_j)\min_j(\tilde\U_j^+) \Big)(t,\theta)d\theta\Big]d\eta\vert\\
&=\vert\frac{1}{\A^7}\int_0^{1} (\sqP{1}-\sqP{2})\sqDP{2}\mathbbm{1}_{D^c}
\min_j(\tilde\D_j)\min_j(\tilde\U_j^+)(\tilde\Y_{2}-\tilde\Y_{1})(t, \eta)d\eta\\
&\quad-\frac{1}{\A^7}\int_0^{1} (\sqP{1}-\sqP{2})\sqDP{2}\mathbbm{1}_{D^c}(t, \eta)\Big(\int_0^{\eta} (\tilde\Y_{2}-\tilde\Y_{1}) \\
&\quad\times\frac{d}{d\theta} \Big(\min_j(e^{-\frac{1}{\ma}( \tilde\Y_j(t,\eta)-\tilde\Y_j(t,\theta))})
\min_j(\tilde\D_j)\min_j(\tilde\U_j^+) \Big)(t,\theta)d\theta\Big) d\eta\vert \\
&\le \frac{\A^2}{4\sqrt{2}} \int_0^{1} \abs{\sqP{1}-\sqP{2}} \abs{\tilde\Y_{2}-\tilde\Y_{1}}(t,\eta)d\eta \\
&\quad+ \frac{1}{\A^7}\int_0^{1} \abs{\sqP{1}-\sqP{2}} \abs{\sqDP{2}}\mathbbm{1}_{D^c}(t, \eta) \Big(\int_0^{\eta} \abs{\tilde\Y_{2}-\tilde\Y_{1}}\\
&\qquad\times\vert\frac{d}{d\theta} \Big(\min_j(e^{-\frac{1}{\ma}( \tilde\Y_j(t,\eta)-\tilde\Y_j(t,\theta))})
\min_j(\tilde\D_j)\min_j(\tilde\U_j^+) \Big)\vert (t,\theta)d\theta\Big) d\eta \\
&=\tilde M_1+\tilde M_2.
\end{align*}
We find directly
\begin{equation*}
\abs{\tilde M_1}\le \bigO(1)\big(\norm{\sqP{1}-\sqP{2}}^2+\norm{\tilde\Y_1-\tilde\Y_2}^2 \big).
\end{equation*}
For the other term $\tilde M_2$ we proceed as follows:
\begin{align*}
\abs{\tilde M_{2}}&\le\frac{1}{\A^7}\int_0^{1} \abs{\sqP{1}-\sqP{2}} \abs{\sqDP{2}}\mathbbm{1}_{D^c}(t, \eta) \Big(\int_0^{\eta} \abs{\tilde\Y_{2}-\tilde\Y_{1}}\vert\\
&\qquad\times\frac{d}{d\theta} \Big(\min_j(e^{-\frac{1}{\ma}( \tilde\Y_j(t,\eta)-\tilde\Y_j(t,\theta))})
\min_j(\tilde\D_j)\min_j(\tilde\U_j^+) \Big)\vert (t,\theta)d\theta\Big) d\eta \\
&\le \frac{1}{\A^7}\int_0^{1} \abs{\sqP{1}-\sqP{2}} \abs{\sqDP{2}}\mathbbm{1}_{D^c}(t, \eta) \\
&\quad\times\Big(\int_0^{\eta} \abs{\tilde\Y_{2}-\tilde\Y_{1}}\Big[\vert\Big(\frac{d}{d\theta} \min_j(e^{-\frac{1}{\ma}( \tilde\Y_j(t,\eta)-\tilde\Y_j(t,\theta))})\Big)\vert \, \vert
\min_j(\tilde\D_j)\min_j(\tilde\U_j^+)\vert (t,\theta) \\
&\qquad\quad+\vert\min_j(e^{-\frac{1}{\ma}( \tilde\Y_j(t,\eta)-\tilde\Y_j(t,\theta))})\vert\, \vert\frac{d}{d\theta}\Big( 
\min_j(\tilde\D_j)\min_j(\tilde\U_j^+) \Big)\vert (t,\theta)\Big]d\theta\Big) d\eta \\
&\le \frac{1}{\A^7}\int_0^{1} \abs{\sqP{1}-\sqP{2}} \abs{\sqDP{2}}\mathbbm{1}_{D^c}(t, \eta) \Big(\int_0^{\eta} \abs{\tilde\Y_{2}-\tilde\Y_{1}}\\
&\quad\times\Big[\frac{1}{\ma}\min_j(e^{-\frac{1}{\ma}( \tilde\Y_j(t,\eta)-\tilde\Y_j(t,\theta))})
 \max_j(\tilde\Y_{j,\eta})\min_j(\tilde\D_j)\min_j(\tilde\U_j^+)(t,\theta) \\
&\qquad\quad+\bigO(1)\A^{9/2}\min_j(e^{-\frac{1}{\ma}( \tilde\Y_j(t,\eta)-\tilde\Y_j(t,\theta))}) (\min_j(\tilde\D_j)^{1/2}+\vert \tilde\U_2\vert) (t,\theta)\Big] d\theta\Big)d\eta \\
&=\tilde M_{21}+\tilde M_{22}, 
\end{align*}
by using Lemmas \ref{lemma:1} and \ref{lemma:5}.  For $\tilde M_{21}$ we find
\begin{align*}
\tilde M_{21}&=\frac{1}{\ma\A^7}\int_0^{1} \abs{\sqP{1}-\sqP{2}} \abs{\sqDP{2}}\mathbbm{1}_{D^c}(t, \eta) \Big(\int_0^{\eta} \min_j(e^{-\frac{1}{\ma}( \tilde\Y_j(t,\eta)-\tilde\Y_j(t,\theta))})\\
&\qquad\qquad\qquad\qquad\qquad\times\abs{\tilde\Y_{2}-\tilde\Y_{1}}
\vert \max_j(\tilde\Y_{j,\eta})\min_j(\tilde\D_j)\min_j(\tilde\U_j^+)\vert (t,\theta) d\theta\Big)d\eta \\
&\le \norm{\sqP{1}-\sqP{2}}^2+\frac{1}{\ma^2\A^{14}}\int_0^{1}  (\sqDP{2})^2(t, \eta)\Big(\int_0^{\eta} \min_j(e^{-\frac{1}{\ma}( \tilde\Y_j(t,\eta)-\tilde\Y_j(t,\theta))})\\
&\qquad\qquad\qquad\qquad\qquad\times\abs{\tilde\Y_{2}-\tilde\Y_{1}}
\max_j(\tilde\Y_{j,\eta})\min_j(\tilde\D_j)\min_j(\tilde\U_j^+) (t,\theta) d\theta\Big)^2d\eta \\
&\le \norm{\sqP{1}-\sqP{2}}^2\\
&\quad+\frac{1}{2}\int_0^{1}  (\sqDP{2})^2(t, \eta)\Big(\int_0^{\eta} e^{-\frac{1}{\ma}( \tilde\Y_2(t,\eta)-\tilde\Y_2(t,\theta))}\abs{\tilde\Y_{2}-\tilde\Y_{1}}
 (t,\theta) d\theta\Big)^2d\eta \\
&\le \norm{\sqP{1}-\sqP{2}}^2\\
&\quad+\frac{1}{8\A^2}\int_0^{1}  \tilde \P_2\tilde \Y_{2,\eta}^2(t, \eta)\Big(\int_0^{\eta} e^{-\frac{1}{\ma}( \tilde\Y_2(t,\eta)-\tilde\Y_2(t,\theta))}(\tilde\Y_{2}-\tilde\Y_{1})^2d\theta\Big)\\
&\qquad\qquad\qquad\qquad\times
\Big(\int_0^{\eta} e^{-\frac{1}{\ma}( \tilde\Y_2(t,\eta)-\tilde\Y_2(t,\theta))}(t,\theta) d\theta\Big)d\eta \\
&\le \norm{\sqP{1}-\sqP{2}}^2+\bigO(1)\int_0^{1}  \int_0^{\eta}(\tilde\Y_{2}-\tilde\Y_{1})^2d\theta d\eta \\
&\le \bigO(1)\big(\norm{\sqP{1}-\sqP{2}}^2+\norm{\tilde\Y_{2}-\tilde\Y_{1}}^2 \big),
\end{align*}
using
\begin{equation*}
\max_j(\tilde\Y_{j,\eta})\min_j(\tilde\D_j)\le 2\max_j(\tilde\Y_{j,\eta})\min_j(\sqtC{i}\tilde\P_j)
\le 2\max_j(\sqtC{i}\tilde\P_j\tilde\Y_{j,\eta})\leq \A^6,
\end{equation*}
and \eqref{eq:all_PestimatesA}, \eqref{eq:all_Pderiv_estimatesC}, and  \eqref{eq:all_estimatesE}.
For $\tilde M_{22}$ we find
\begin{align*}
\tilde M_{22}&=\frac{\bigO(1)}{\A^{5/2}}\int_0^{1} \abs{\sqP{1}-\sqP{2}} \abs{\sqDP{2}}\mathbbm{1}_{D^c}(t, \eta) \\
&\qquad\times\Big(\int_0^{\eta}\min_j(e^{-\frac{1}{\ma}( \tilde\Y_j(t,\eta)-\tilde\Y_j(t,\theta))}) \abs{\tilde\Y_{2}-\tilde\Y_{1}}
(\min_j(\tilde\D_j)^{1/2}+\vert \tilde\U_2\vert) (t,\theta)d\theta\Big)d\eta \\
&\le\norm{\sqP{1}-\sqP{2}}^2 +\frac{\bigO(1)}{\A^5}\int_0^{1} (\sqDP{2})^2(t, \eta) \\
&\quad\times\Big(\int_0^{\eta}\min_j(e^{-\frac{1}{\ma}( \tilde\Y_j(t,\eta)-\tilde\Y_j(t,\theta))}) \abs{\tilde\Y_{2}-\tilde\Y_{1}}
(\min_j(\tilde\D_j)^{1/2}+\vert \tilde\U_2\vert) (t,\theta)d\theta\Big)^2d\eta \\
&\le\norm{\sqP{1}-\sqP{2}}^2 +\frac{\bigO(1)}{\A^5}\int_0^{1} (\sqDP{2})^2(t, \eta) \\
&\quad\times\Big(\int_0^{\eta}e^{-\frac{1}{\ma}( \tilde\Y_2(t,\eta)-\tilde\Y_2(t,\theta))} \abs{\tilde\Y_{2}-\tilde\Y_{1}}
 \tilde\P_2^{1/2} (t,\theta)d\theta\Big)^2d\eta \\
 &\le\norm{\sqP{1}-\sqP{2}}^2 \\
&\qquad+\frac{\bigO(1)}{\A^7}\int_0^{1} \tilde\P_2\tilde\Y_{2,\eta}^2(t, \eta) \Big(\int_0^{\eta}e^{-\frac{3}{2\sqtC{2}}( \tilde\Y_2(t,\eta)-\tilde\Y_2(t,\theta))} \tilde\P_2(t,\theta)d\theta\Big)\\
&\qquad\times
\Big(\int_0^{\eta}e^{-\frac1{2\sqtC{2}}( \tilde\Y_2(t,\eta)-\tilde\Y_2(t,\theta))} \abs{\tilde\Y_{2}-\tilde\Y_{1}}(t,\theta)d\theta\Big)d\eta \\
&\le\norm{\sqP{1}-\sqP{2}}^2 \\
&\quad+\frac{\bigO(1)}{\A^7}\int_0^{1} \tilde\P_2^2\tilde\Y_{2,\eta}^2(t, \eta) 
\Big(\int_0^{\eta}e^{-\frac1{2\sqtC{2}}( \tilde\Y_2(t,\eta)-\tilde\Y_2(t,\theta))} (\tilde\Y_{2}-\tilde\Y_{1})^2(t,\theta)d\theta\Big)d\eta \\
&\le \bigO(1)\big(\norm{\sqP{1}-\sqP{2}}^2+\norm{\tilde\Y_{2}-\tilde\Y_{1}}^2 \big),
\end{align*}
using that, cf.~\eqref{eq:all_estimatesD} and \eqref{eq:all_estimatesN},
\begin{equation*}
\min_j(\tilde\D_j)^{1/2}+\vert \tilde\U_2\vert\le \sqrt{2\sqtC{2}}\tilde\P_2^{1/2}+\vert \tilde\U_2\vert\le \sqrt{2}(1+\sqrt{\sqtC{2}})\tilde\P_2^{1/2},
\end{equation*}
and \eqref{eq:32P}, \eqref{eq:all_estimatesE}, and \eqref{eq:all_Pderiv_estimatesA}.

\bigskip
We now consider the term $W_4$:
\begin{align*}
W_4&=-\frac{3}{\A^6}\int_0^{1} (\sqP{1}-\sqP{2})\\
&\qquad\times\Big(\sqDP{2}(t, \eta)\int_0^{\eta} e^{-\frac{1}{\sqtC{2}}( \tilde\Y_2(t,\eta)-\tilde\Y_2(t,\theta))}\tilde\P_2\tilde\U_{2}\tilde\Y_{2,\eta}(t,\theta)d\theta\\
&\qquad-\sqDP{1}(t, \eta)\int_0^{\eta} e^{-\frac{1}{\sqtC{1}}( \tilde\Y_1(t,\eta)-\tilde\Y_1(t,\theta))}\tilde\P_1\tilde\U_{1}\tilde\Y_{1,\eta}(t,\theta)d\theta\Big)d\eta\\
&=-\frac{3}{\A^6}\int_0^{1} (\sqP{1}-\sqP{2})\\
&\qquad\times\Big(\sqDP{2}(t, \eta)\int_0^{\eta} e^{-\frac{1}{\sqtC{2}}( \tilde\Y_2(t,\eta)-\tilde\Y_2(t,\theta))}\tilde\P_2\tilde\U_{2}^+\tilde\Y_{2,\eta}(t,\theta)d\theta\\
&\qquad-\sqDP{1}(t, \eta)\int_0^{\eta} e^{-\frac{1}{\sqtC{1}}( \tilde\Y_1(t,\eta)-\tilde\Y_1(t,\theta))}\tilde\P_1\tilde\U_{1}^+\tilde\Y_{1,\eta}(t,\theta)d\theta\Big)d\eta\\
&\quad-\frac{3}{\A^6}\int_0^{1} (\sqP{1}-\sqP{2})\\
&\qquad\times\Big(\sqDP{2}(t, \eta)\int_0^{\eta} e^{-\frac{1}{\sqtC{2}}( \tilde\Y_2(t,\eta)-\tilde\Y_2(t,\theta))}\tilde\P_2\tilde\U_{2}^-\tilde\Y_{2,\eta}(t,\theta)d\theta\\
&\qquad-\sqDP{1}(t, \eta)\int_0^{\eta} e^{-\frac{1}{\sqtC{1}}( \tilde\Y_1(t,\eta)-\tilde\Y_1(t,\theta))}\tilde\P_1\tilde\U_{1}^-\tilde\Y_{1,\eta}(t,\theta)d\theta\Big)d\eta\\
&= W_4^++W_4^-.
\end{align*}
Since $W_4^+=-3N_1$, see \eqref{eq:N1P}, we find that
\begin{equation*}
\abs{W_4}\le \bigO(1)\big(\norm{\sqP{1}-\sqP{2}}^2+\norm{\tilde\U_{1}-\tilde\U_{2}}^2+\norm{\tilde\Y_2-\tilde\Y_1}^2+\vert \sqtC{1}-\sqtC{2}\vert \big).
\end{equation*}

The next term $W_5$ would be laborious:
\begin{align*}
W_5&=\frac{1}{\A^6}\int_0^{1} (\sqP{1}-\sqP{2})\Big(\sqDP{2}(t, \eta)\int_0^{\eta} e^{-\frac{1}{\sqtC{2}}( \tilde\Y_2(t,\eta)-\tilde\Y_2(t,\theta))}\tilde\U_2^3\tilde\Y_{2,\eta}(t,\theta)d\theta\\
&\qquad-\sqDP{1}(t, \eta)\int_0^{\eta} e^{-\frac{1}{\sqtC{1}}( \tilde\Y_1(t,\eta)-\tilde\Y_1(t,\theta))}\tilde\U_1^3\tilde\Y_{1,\eta}(t,\theta)d\theta\Big)d\eta\\
&=\frac{1}{\A^6}\int_0^{1} (\sqP{1}-\sqP{2})\Big(\sqDP{2}(t, \eta)\int_0^{\eta} e^{-\frac{1}{\sqtC{2}}( \tilde\Y_2(t,\eta)-\tilde\Y_2(t,\theta))}(\tilde\U_2^+)^3\tilde\Y_{2,\eta}(t,\theta)d\theta\\
&\qquad-\sqDP{1}(t, \eta)\int_0^{\eta} e^{-\frac{1}{\sqtC{1}}( \tilde\Y_1(t,\eta)-\tilde\Y_1(t,\theta))}(\tilde\U_1^+)^3\tilde\Y_{1,\eta}(t,\theta)d\theta\Big)d\eta\\
&+\frac{1}{\A^6}\int_0^{1} (\sqP{1}-\sqP{2})\Big(\sqDP{2}(t, \eta)\int_0^{\eta} e^{-\frac{1}{\sqtC{2}}( \tilde\Y_2(t,\eta)-\tilde\Y_2(t,\theta))}(\tilde\U_2^-)^3\tilde\Y_{2,\eta}(t,\theta)d\theta\\
&\qquad-\sqDP{1}(t, \eta)\int_0^{\eta} e^{-\frac{1}{\sqtC{1}}( \tilde\Y_1(t,\eta)-\tilde\Y_1(t,\theta))}(\tilde\U_1^-)^3\tilde\Y_{1,\eta}(t,\theta)d\theta\Big)d\eta\\
&=W_5^+ + W_5^-.
\end{align*}

Unfortunately, having a close look at $W_5^+$ one has 
\begin{equation*}
W_5^+=\int_0^1 (\sqP{1}-\sqP{2})(I_{211}+I_{212}+I_{213}+I_{214}+I_{215}+I_{216}+I_{217}+I_{218})(t,\eta)d\eta,
\end{equation*}
where $I_{211}, \dots, I_{218}$ are defined in \eqref{eq:I21j}. Thus we can conclude immediately that 
\begin{equation*}
\abs{W_5^+}\leq \bigO(1) \Big(\norm{\tilde\Y_1-\tilde\Y_2}^2+\norm{\tilde\U_1-\tilde\U_2}^2
+\norm{\sqP{1}-\sqP{2}}^2+\vert \sqtC{1}-\sqtC{2}\vert^2\Big).
\end{equation*} 

\bigskip
For the term $W_6$ we find
\begin{align*}
W_6&=\frac{1}{2\A^6}\int_0^{1} (\sqP{1}-\sqP{2})\Big(\sqDP{2}(t, \eta)\int_0^{\eta} e^{-\frac{1}{\sqtC{2}}( \tilde\Y_2(t,\eta)-\tilde\Y_2(t,\theta))}\sqtC{2}^5\tilde\U_2(t,\theta)d\theta\\
&\qquad\qquad-\sqDP{1}(t, \eta)\int_0^{\eta} e^{-\frac{1}{\sqtC{1}}( \tilde\Y_1(t,\eta)-\tilde\Y_1(t,\theta))}\sqtC{1}^5\tilde\U_1(t,\theta)d\theta\Big)d\eta\\
&=\frac{\sqtC{2}^5-\sqtC{1}^5}{2\A^6}\mathbbm{1}_{\sqtC{1}\le \sqtC{2}}\int_0^1(\sqP{1}-\sqP{2})
\sqDP{2}(t, \eta)\\
&\qquad\qquad\qquad\qquad\qquad\times\int_0^{\eta} e^{-\frac{1}{\sqtC{2}}( \tilde\Y_2(t,\eta)-\tilde\Y_2(t,\theta))}\tilde\U_2(t,\theta)d\theta d\eta\\
&\quad+\frac{\sqtC{2}^5-\sqtC{1}^5}{2\A^6}\mathbbm{1}_{\sqtC{2}< \sqtC{1}}\int_0^1(\sqP{1}-\sqP{2})
\sqDP{1}(t, \eta)\\
&\qquad\qquad\qquad\qquad\qquad\times\int_0^{\eta} e^{-\frac{1}{\sqtC{2}}( \tilde\Y_2(t,\eta)-\tilde\Y_2(t,\theta))}\tilde\U_1(t,\theta)d\theta d\eta\\
&\quad+\mathbbm{1}_{\sqtC{1}\leq \sqtC{2}}\frac{\ma^5}{2\A^6}\int_0^1(\sqP{1}-\sqP{2})\sqDP{2}(t, \eta)\\
&\qquad\qquad\times\Big(\int_0^{\eta} \big(e^{-\frac{1}{\sqtC{2}}( \tilde\Y_2(t,\eta)-\tilde\Y_2(t,\theta))}-e^{-\frac{1}{\sqtC{1}}( \tilde\Y_2(t,\eta)-\tilde\Y_2(t,\theta))}\big)\tilde\U_2(t,\theta)d\theta\Big) d\eta\\
&\quad+\mathbbm{1}_{\sqtC{2}< \sqtC{1}}\frac{\ma^5}{2\A^6}\int_0^1(\sqP{1}-\sqP{2})\sqDP{1}(t, \eta)\\
&\qquad\qquad\times\Big(\int_0^{\eta} \big(e^{-\frac{1}{\sqtC{2}}( \tilde\Y_1(t,\eta)-\tilde\Y_1(t,\theta))}-e^{-\frac{1}{\sqtC{1}}( \tilde\Y_1(t,\eta)-\tilde\Y_1(t,\theta))}\big)\tilde\U_1(t,\theta)d\theta\Big) d\eta\\
&\quad+\frac{\ma^5}{2\A^6}\int_0^1(\sqP{1}-\sqP{2})\sqDP{2}(t, \eta)\\
&\qquad\qquad\times\Big(\int_0^{\eta} \big(e^{-\frac{1}{\ma}( \tilde\Y_2(t,\eta)-\tilde\Y_2(t,\theta))}-e^{-\frac{1}{\ma}( \tilde\Y_1(t,\eta)-\tilde\Y_1(t,\theta))}\big)\tilde\U_2\mathbbm{1}_{B(\eta)}(t,\theta)d\theta\Big) d\eta\\
&\quad+\frac{\ma^5}{2\A^6}\int_0^1(\sqP{1}-\sqP{2})\sqDP{1}(t, \eta)\\
&\qquad\qquad\times\Big(\int_0^{\eta} \big(e^{-\frac{1}{\ma}( \tilde\Y_2(t,\eta)-\tilde\Y_2(t,\theta))}-e^{-\frac{1}{\ma}( \tilde\Y_1(t,\eta)-\tilde\Y_1(t,\theta))}\big)\tilde\U_1\mathbbm{1}_{B(\eta)^c}(t,\theta)d\theta\Big) d\eta\\
&\quad+\frac{\ma^5}{2\A^6}\int_0^1(\sqP{1}-\sqP{2})\sqDP{2}(t, \eta)\\
&\qquad\qquad\times\Big(\int_0^{\eta} \min_j(e^{-\frac{1}{\ma}( \tilde\Y_j(t,\eta)-\tilde\Y_j(t,\theta))})(\tilde\U_2^+ -\tilde\U_1^+)\mathbbm{1}_{\tilde\U_1^+ \le\tilde\U_2^+}(t,\theta)d\theta\Big) d\eta\\
&\quad+\frac{\ma^5}{2\A^6}\int_0^1(\sqP{1}-\sqP{2})\sqDP{2}(t, \eta)\\
&\qquad\qquad\times\Big(\int_0^{\eta} \min_j(e^{-\frac{1}{\ma}( \tilde\Y_j(t,\eta)-\tilde\Y_j(t,\theta))})(\tilde\U_2^- -\tilde\U_1^-)\mathbbm{1}_{\tilde\U_2^- \le\tilde\U_1^-}(t,\theta)d\theta\Big) d\eta\\
&\quad+\frac{\ma^5}{2\A^6}\int_0^1(\sqP{1}-\sqP{2})\sqDP{1}(t, \eta)\\
&\qquad\qquad\times\Big(\int_0^{\eta} \min_j(e^{-\frac{1}{\ma}( \tilde\Y_j(t,\eta)-\tilde\Y_j(t,\theta))})(\tilde\U_2^+ -\tilde\U_1^+)\mathbbm{1}_{\tilde\U_2^+ <\tilde\U_1^+}(t,\theta)d\theta\Big) d\eta\\
&\quad+\frac{\ma^5}{2\A^6}\int_0^1(\sqP{1}-\sqP{2})\sqDP{1}(t, \eta)\\
&\qquad\qquad\times\Big(\int_0^{\eta} \min_j(e^{-\frac{1}{\ma}( \tilde\Y_j(t,\eta)-\tilde\Y_j(t,\theta))})(\tilde\U_2^- -\tilde\U_1^-)\mathbbm{1}_{\tilde\U_1^- <\tilde\U_2^-}(t,\theta)d\theta\Big) d\eta\\
&\quad+\frac{\ma^5}{2\A^6}\int_0^1(\sqP{1}-\sqP{2})(\sqDP{2} -\sqDP{1}(t, \eta)\\
&\qquad\qquad\times\Big(\int_0^{\eta} \min_j(e^{-\frac{1}{\ma}( \tilde\Y_j(t,\eta)-\tilde\Y_j(t,\theta))})\min_j(\tilde\U_j^+)(t,\theta)d\theta\Big) d\eta\\
&\quad+\frac{\ma^5}{2\A^6}\int_0^1(\sqP{1}-\sqP{2})(\sqDP{2} -\sqDP{1})(t, \eta)\\
&\qquad\qquad\times\Big(\int_0^{\eta} \max_j(e^{-\frac{1}{\ma}( \tilde\Y_j(t,\eta)-\tilde\Y_j(t,\theta))})\min_j(\tilde\U_j^-)(t,\theta)d\theta\Big) d\eta\\
&= W_{61}+W_{62}+W_{63}+W_{64}+W_{65}+W_{66}\\
&\quad +W_{67}^+ +W_{67}^- +W_{68}^+ +W_{68}^- +W_{69}^+ +W_{69}^-.
\end{align*}

The terms $W_{61}$ and $W_{62}$:
\begin{align*}
\abs{W_{61}}&=\frac{\sqtC{2}^5-\sqtC{1}^5}{2\A^6}\mathbbm{1}_{\sqtC{1}\le\sqtC{2}} \vert\int_0^1(\sqP{1}-\sqP{2})\sqDP{2}(t, \eta)\\
&\qquad\qquad\qquad\qquad\times\int_0^{\eta} e^{-\frac{1}{\sqtC{2}}( \tilde\Y_2(t,\eta)-\tilde\Y_2(t,\theta))}\tilde\U_2(t,\theta)d\theta d\eta\vert\\
&\le5\frac{\sqtC{2}-\sqtC{1}}{2\A^2} \int_0^1\vert\sqP{1}-\sqP{2}\vert\vert\sqDP{2}\vert(t, \eta)\\
&\qquad\qquad\qquad\qquad\times\Big(\int_0^{\eta} e^{-\frac{1}{\sqtC{2}}( \tilde\Y_2(t,\eta)-\tilde\Y_2(t,\theta))}\tilde\U_2^2(t,\theta)d\theta \Big)^{1/2} \\
& \qquad \qquad \qquad \qquad \qquad \qquad  \times\Big(\int_0^{\eta} e^{-\frac{1}{\sqtC{2}}( \tilde\Y_2(t,\eta)-\tilde\Y_2(t,\theta))}d\theta\Big)^{1/2}d\eta\vert\\
& \leq 5\sqrt{6}\frac{\sqtC{2}-\sqtC{1}}{2\A^2}\int_0^1\vert\sqP{1}-\sqP{2}\vert\sqP{2}\vert\sqDP{2}\vert(t, \eta)d\eta\\
&\le \bigO(1)\big(\abs{\sqtC{1}-\sqtC{2}}^2+\norm{\sqP{1}-\sqP{2}}^2 \big),
\end{align*}
using \eqref{eq:all_estimatesE}, \eqref{eq:all_PestimatesA}, and \eqref{eq:all_Pderiv_estimatesA}.

The terms $W_{63}$ and $W_{64}$:
\begin{align*}
\abs{W_{63}}&=\mathbbm{1}_{\sqtC{1}\leq \sqtC{2}}\frac{\ma^5}{2\A^6}\int_0^1\vert\sqP{1}-\sqP{2}\vert\vert\sqDP{2}\vert(t, \eta)\\
& \qquad \qquad \times\Big\vert\int_0^\eta (e^{-\frac{1}{\sqtC{2}}(\tilde \Y_2(t,\eta)-\tilde \Y_2(t,\theta))}-e^{-\frac{1}{\sqtC{1}}(\tilde \Y_2(t,\eta)-\tilde \Y_2(t,\theta))})\tilde \U_2(t,\theta) d\theta\Big\vert d\eta\\
& \leq \frac{4\ma^5}{2\ma\A^6e}\int_0^1\vert\sqP{1}-\sqP{2}\vert\vert\sqDP{2}\vert(t, \eta)\\
&\qquad\qquad\qquad\qquad\times\Big(\int_0^\eta e^{-\frac{3}{4\sqtC{2}}(\tilde \Y_2(t,\eta)-\tilde \Y_2(t,\theta))}\vert\tilde \U_2\vert(t,\theta) d\theta\Big) d\eta\vert \sqtC{1}-\sqtC{2}\vert\\
& \leq \frac{2}{\A^2e}\int_0^1\vert\sqP{1}-\sqP{2}\vert\vert\sqDP{2}\vert(t, \eta)\\
&\qquad\qquad\qquad\qquad\times\Big(\int_0^\eta e^{-\frac{1}{\sqtC{2}}(\tilde \Y_2(t,\eta)-\tilde \Y_2(t,\theta))}\tilde \U_2^2(t,\theta) d\theta\Big)^{1/2} d\eta\vert \sqtC{1}-\sqtC{2}\vert \\
& \leq \frac{2\sqrt{6}}{\A^2 e}\int_0^1\vert\sqP{1}-\sqP{2}\vert\sqP{2}\vert\sqDP{2}\vert(t, \eta)d\eta\vert\sqtC{1}-\sqtC{2}\vert \\
& \leq \bigO(1)(\norm{\sqP{1}-\sqP{2}}^2+\vert \sqtC{1}-\sqtC{2}\vert^2)
\end{align*}

The terms $W_{65}$ and $W_{66}$:
\begin{align*}
\abs{W_{65}}&=\frac{\ma^5}{2\A^6}\vert\int_0^1(\sqP{1}-\sqP{2})\sqDP{2}(t, \eta)\\
&\qquad\times\int_0^{\eta}\big( e^{-\frac{1}{\ma}( \tilde\Y_2(t,\eta)-\tilde\Y_2(t,\theta))}-e^{-\frac{1}{\ma}( \tilde\Y_1(t,\eta)-\tilde\Y_1(t,\theta))}\big)\tilde\U_2\mathbbm{1}_{B(\eta)}(t,\theta)d\theta d\eta\vert\\
& \le \frac{1}{2\A^2}\int_0^1\vert\sqP{1}-\sqP{2}\vert\vert\sqDP{2}\vert(t, \eta)\\
&\qquad\qquad\qquad\qquad\times\int_0^{\eta}\big(\vert\tilde\Y_2(t,\eta)-\tilde\Y_1(t,\eta)\vert+\vert \tilde\Y_2(t,\theta)-\tilde\Y_1(t,\theta)\vert\big) \\
&\qquad\qquad\qquad\qquad\qquad\times e^{-\frac{1}{\ma}( \tilde\Y_2(t,\eta)-\tilde\Y_2(t,\theta))}\vert\tilde\U_2\vert(t,\theta)d\theta d\eta\\
&\le \frac{1}{2\A^2}\int_0^1\vert\sqP{1}-\sqP{2}\vert\tilde\Y_2-\tilde\Y_1\vert \vert\sqDP{2}\vert(t, \eta)\\
&\qquad\qquad\qquad\qquad\times\int_0^{\eta} e^{-\frac{1}{\ma}( \tilde\Y_2(t,\eta)-\tilde\Y_2(t,\theta))}\vert\tilde\U_2\vert(t,\theta)d\theta d\eta\\
&\quad+\frac{1}{2\A^2}\int_0^1\vert\sqP{1}-\sqP{2}\vert\vert\sqDP{2}\vert(t, \eta)\\
&\qquad\qquad\qquad\qquad\times\int_0^{\eta} e^{-\frac{1}{\ma}( \tilde\Y_2(t,\eta)-\tilde\Y_2(t,\theta))}\vert \tilde\Y_2-\tilde\Y_1\vert\,\vert\tilde\U_2\vert(t,\theta)d\theta d\eta\\
& \leq \frac{1}{2\A^2}\int_0^1\vert\sqP{1}-\sqP{2}\vert\vert\tilde\Y_2-\tilde\Y_1\vert \vert\sqDP{2}\vert(t, \eta)\\
&\qquad\qquad\qquad\qquad\times\Big(\int_0^{\eta} e^{-\frac{1}{\ma}( \tilde\Y_2(t,\eta)-\tilde\Y_2(t,\theta))}\tilde\U_2^2(t,\theta)d\theta \Big)^{1/2}d\eta\\
&\quad+\frac{1}{2\A^2}\int_0^1\vert\sqP{1}-\sqP{2}\vert\vert\sqDP{2}\vert(t, \eta)\\
&\qquad\qquad\qquad\qquad\times\Big(\int_0^{\eta} e^{-\frac{1}{\ma}( \tilde\Y_2(t,\eta)-\tilde\Y_2(t,\theta))}\tilde\U_2^2(t,\theta)d\theta\Big)^{1/2} d\eta \norm{\tilde\Y_1-\tilde\Y_2}\\
& \leq \frac{\sqrt{3}}{\sqrt{2}\A^2}\int_0^1\vert\sqP{1}-\sqP{2}\vert\vert\tilde\Y_2-\tilde\Y_1\vert \sqP{2}\vert\sqDP{2}\vert(t, \eta)d\eta\\
&\quad+\frac{\sqrt{3}}{\sqrt{2}\A^2}\int_0^1\vert\sqP{1}-\sqP{2}\vert\sqP{2}\vert\sqDP{2}\vert(t, \eta) d\eta \norm{\tilde\Y_1-\tilde\Y_2}\\
& \le \bigO(1)\Big(\norm{\sqP{1}-\sqP{2}}^2+\norm{\tilde\Y_1-\tilde\Y_2}^2\Big),
\end{align*}
using \eqref{Diff:Exp}, \eqref{eq:all_PestimatesA}, \eqref{eq:all_estimatesE}, and \eqref{eq:all_Pderiv_estimatesA}.

The terms $W_{67}^\pm$ and $W_{68}^\pm$:
\begin{align*}
\abs{W_{67}^+}&=\frac{\ma^5}{2\A^6}\vert\int_0^1(\sqP{1}-\sqP{2})\sqDP{2}(t, \eta)\\
&\qquad\qquad\times\Big(\int_0^{\eta} \min_j(e^{-\frac{1}{\ma}( \tilde\Y_j(t,\eta)-\tilde\Y_j(t,\theta))})(\tilde\U_2^+ -\tilde\U_1^+)\mathbbm{1}_{\tilde\U_1^+ \le\tilde\U_2^+}(t,\theta)d\theta\Big) d\eta\vert\\
&\le \norm{\sqP{1}-\sqP{2}}^2 +\frac{1}{4\A^2}\int_0^1(\sqDP{2})^2(t, \eta) \\
&\qquad\times\Big( \int_0^{\eta} \min_j(e^{-\frac{1}{\ma}( \tilde\Y_j(t,\eta)-\tilde\Y_j(t,\theta))})(\tilde\U_2^+ -\tilde\U_1^+)\mathbbm{1}_{\tilde\U_1^+ \le\tilde\U_2^+}(t,\theta)d\theta\Big)^2d\eta\\
&\le \norm{\sqP{1}-\sqP{2}}^2 \\
&\quad +\frac{1}{16\A^4}\int_0^1\tilde\P_2\tilde \Y_2^2(t, \eta)\Big( \int_0^{\eta} e^{-\frac{1}{\sqtC{2}}( \tilde\Y_2(t,\eta)-\tilde\Y_2(t,\theta))}(\tilde\U_2^+ -\tilde\U_1)^2(t,\theta)d\theta\Big)d\eta\\
&\le \norm{\sqP{1}-\sqP{2}}^2 \\
&\quad +\frac{\A}{32}\int_0^1\tilde\Y_{2,\eta}(t, \eta)\Big( \int_0^{\eta}e^{-\frac{1}{\sqtC{2}}( \tilde\Y_2(t,\eta)-\tilde\Y_2(t,\theta))}(\tilde\U_2 -\tilde\U_1)^2(t,\theta)d\theta\Big)d\eta\\
&\le \bigO(1)\big(\norm{\sqP{1}-\sqP{2}}^2+\norm{\tilde\U_1-\tilde\U_2}^2 \big)
\end{align*}
following the estimates used for $\bar B_{65}^+$, see \eqref{eq:barB65}.
The terms $W_{69}^\pm$:
\begin{align*}
\abs{W_{69}^+}&=\frac{\ma^5}{2\A^6}\vert\int_0^1(\sqP{1}-\sqP{2})(\sqDP{2} -\sqDP{1}(t, \eta)\\
&\qquad\qquad\times\Big(\int_0^{\eta} \min_j(e^{-\frac{1}{\ma}( \tilde\Y_j(t,\eta)-\tilde\Y_j(t,\theta))})\min_j(\tilde\U_j^+)(t,\theta)d\theta\Big) d\eta\vert\\
&= \frac{\ma^5}{4\A^6}\Big\vert \Big((\sqP{1}-\sqP{2})^2\int_0^{\eta} \min_j(e^{-\frac{1}{\ma}( \tilde\Y_j(t,\eta)-\tilde\Y_j(t,\theta))})\min_j(\tilde\U_j^+)(t,\theta)d\theta\Big)\Big\vert_{\eta=0}^1\\
&\quad -\int_0^1(\sqP{1}-\sqP{2})^2  \frac{d}{d\eta}\Big(\int_0^{\eta} \min_j(e^{-\frac{1}{\ma}( \tilde\Y_j(t,\eta)-\tilde\Y_j(t,\theta))})\min_j(\tilde\U_j^+)(t,\theta) d\theta \Big) d\eta \Big\vert\\
&\le  \frac{\ma^5}{4\A^6}\int_0^1(\sqP{1}-\sqP{2})^2 \\
&\qquad\qquad\qquad\times \vert\frac{d}{d\eta}\Big(\int_0^{\eta} \min_j(e^{-\frac{1}{\ma}( \tilde\Y_j(t,\eta)-\tilde\Y_j(t,\theta))})\min_j(\tilde\U_j^+)\vert(t,\theta)d\theta \Big) d\eta\\
&\le \bigO(1)\norm{\sqP{1}-\sqP{2}}^2,
\end{align*}
see estimates employed for $\bar B_{67}$, cf.~\eqref{eq:barB67}, and Lemma \ref{lemma:6}.

The terms $M_2$ and $M_3$ can be treated similarly. More precisely
\begin{align*}
\abs{M_2}&\le\mathbbm{1}_{\sqtC{2}\le \sqtC{1}} \big(\frac{1}{\sqtC{2}^6}-\frac{1}{\sqtC{1}^6} \big)\vert\int_0^{1}(\sqP{1}-\sqP{2})\sqDP{2}(t, \eta)\\
&\qquad\qquad\qquad\qquad\times\int_0^{\eta}e^{-\frac{1}{\sqtC{2}}( \tilde\Y_2(t,\eta)-\tilde\Y_2(t,\theta)) }\tilde\Q_2\tilde\U_{2,\eta}(t,\theta)d\theta d\eta\vert \\
&\le \mathbbm{1}_{\sqtC{2}\le \sqtC{1}} \frac{\vert \sqtC{1}^6-\sqtC{2}^6\vert}{\A^6\sqtC{2}^5}\int_0^{1}\vert\sqP{1}-\sqP{2}\vert \, \vert\sqDP{2}\vert(t, \eta)\\
&\qquad\qquad\qquad\qquad\times
\int_0^{\eta}e^{-\frac{1}{\sqtC{2}}( \tilde\Y_2(t,\eta)-\tilde\Y_2(t,\theta)) }\tilde\P_2\vert\tilde\U_{2,\eta}\vert(t,\theta)d\theta d\eta \\
&\le \mathbbm{1}_{\sqtC{2}\le\sqtC{1}} 6\frac{\vert \sqtC{1}-\sqtC{2}\vert}{\A\sqtC{2}^5}
\norm{\sqP{1}-\sqP{2}} \\
&\qquad\times\Big(\int_0^{1} (\sqDP{2})^2(t, \eta) \Big(
\int_0^{\eta}e^{-\frac{1}{\sqtC{2}}( \tilde\Y_2(t,\eta)-\tilde\Y_2(t,\theta)) }\tilde\P_2\vert\tilde\U_{2,\eta}\vert(t,\theta)d\theta\Big)^2 d\eta\Big)^{1/2} \\
&\le  \mathbbm{1}_{\sqtC{2}\le \sqtC{1}} 6\frac{\vert \sqtC{1}-\sqtC{2}\vert}{\A\sqtC{2}^6}\norm{\sqP{1}-\sqP{2}}\\
&\qquad\times\Big(\int_0^{1}  (\sqDP{2})^2(t, \eta) \Big(\int_0^{\eta}e^{-\frac{1}{\sqtC{2}}( \tilde\Y_2(t,\eta)-\tilde\Y_2(t,\theta)) }\tilde\P_2^2\tilde\Y_{2,\eta}(t,\theta)d\theta\Big) \\
&\qquad\qquad\qquad\qquad\qquad\qquad\times\Big(\int_0^{\eta}e^{-\frac{1}{\sqtC{2}}( \tilde\Y_2(t,\eta)-\tilde\Y_2(t,\theta)) }\tilde\Henergy_{2,\eta}(t,\theta)d\theta\Big)
    d\eta\Big)^{1/2} \\
&\le \mathbbm{1}_{\sqtC{2}\le \sqtC{1}}6\sqrt{6}\frac{\vert\sqtC{1}-\sqtC{2}\vert}{\A\sqtC{2}^3}\norm{\sqP{1}-\sqP{2}}\Big(\int_0^{1} (\sqDP{2})^2\tilde\P_2^2(t, \eta)
    d\eta\Big)^{1/2}     \\
&\le \bigO(1) \big(\norm{\sqP{1}-\sqP{2}}^2+\vert \sqtC{1}-\sqtC{2}\vert^2 \big).    
\end{align*}
Here we used 
\begin{align*}
\tilde\P_2\vert\tilde\U_{2,\eta}\vert& \leq \frac{1}{\sqtC{2}}\tilde\P_2\sqrt{\tilde\Y_{2,\eta}\tilde\Henergy_{2,\eta}},
\end{align*}
cf.~\eqref{eq:all_estimatesM}, as well as \eqref{eq:all_estimatesC}, \eqref{eq:all_estimatesE}, and \eqref{eq:all_Pderiv_estimatesD}. In addition, we applied  \eqref{eq:all_PestimatesB} and \eqref{eq:all_PestimatesD}.

\textit{The term $I_3$:}  We have the following estimates 
\begin{align*}
\abs{I_3}&\le\int_0^{1}\abs{\sqP{1}-\sqP{2}}
\abs{\frac{1}{\sqtC{1}^2}\frac{\tilde \Q_1\tilde\U_1}{\sqP{1}}-\frac{1}{\sqtC{2}^2}\frac{\tilde \Q_2\tilde\U_2}{\sqP{2}}}(t,\eta)d\eta\\
&\le\norm{\sqP{1}-\sqP{2}}^2
+\norm{\frac{1}{\sqtC{1}^2}\frac{\tilde \Q_1\tilde\U_1}{\sqP{1}}-\frac{1}{\sqtC{2}^2}
\frac{\tilde \Q_2\tilde\U_2}{\sqP{2}}}^2.
\end{align*}
We consider the latter term, and find
\begin{align*}
&\norm{\frac{1}{\sqtC{1}^2}\frac{\tilde \Q_1\tilde\U_1}{\sqP{1}}-\frac{1}{\sqtC{2}^2}\frac{\tilde \Q_2\tilde\U_2}{\sqP{2}}}^2\\
&\qquad \leq3 \mathbbm{1}_{\sqtC{1}\leq \sqtC{2}} \Big\vert \frac{1}{\sqtC{1}^2}-\frac{1}{\sqtC{2}^2}\Big\vert^2 \norm{\frac{\tilde \Q_1\tilde\U_1}{\sqP{1}}}^2
+3\mathbbm{1}_{\sqtC{2}<\sqtC{1}}\Big\vert\frac{1}{\sqtC{1}^2}-\frac{1}{\sqtC{2}}\Big\vert^2 \norm{\frac{\tilde \Q_2\tilde\U_2}{\sqP{2}}}^2\\
&\qquad \quad +3\frac{1}{\A^4}\norm{\frac{\tilde \Q_1\tilde\U_1}{\sqP{1}}-\frac{\tilde \Q_2\tilde\U_2}{\sqP{2}}}^2\\
&\qquad \leq \bigO(1)\vert \sqtC{1}-\sqtC{2}\vert^2 +3\frac{1}{\A^4}\norm{\frac{\tilde \Q_1\tilde\U_1}{\sqP{1}}-\frac{\tilde \Q_2\tilde\U_2}{\sqP{2}}}^2,
\end{align*}
where we have used that 
\begin{equation*}
\vert \frac{\tilde \Q_i\tilde\U_i}{\sqP{i}}\vert \leq \frac{\sqtC{i}\tilde\P_i\tilde\U_i}{\sqP{i}}\vert = \sqtC{i}\sqP{i}\vert\tilde\U_i\vert \leq \bigO(1)\sqtC{i}^5.
\end{equation*}

We consider the latter term and find
\begin{align*}
\frac{1}{\A^4}&\norm{\frac{\tilde \Q_1\tilde\U_1}{\sqP{1}}-\frac{\tilde \Q_2\tilde\U_2}{\sqP{2}}}^2 \\
&= \frac{1}{\A^4}\int_0^1\abs{\frac{\tilde \Q_1\tilde\U_1}{\sqP{1}}-\frac{\tilde \Q_2\tilde\U_2}{\sqP{2}}}^2
\big( \mathbbm{1}_{\tilde \P_1\le \tilde \P_2}+ \mathbbm{1}_{\tilde \P_2< \tilde \P_1} \big)(t,\eta)d\eta  \\
&\le \frac{2}{\A^4}\int_0^1\abs{\frac{1}{\sqP{2}}\big(\tilde \Q_1\tilde\U_1-\tilde \Q_2\tilde\U_2\big)+\big(\frac{1}{\sqP{1}}-\frac{1}{\sqP{2}}\big)\tilde \Q_1\tilde\U_1 }^2
\mathbbm{1}_{\tilde \P_1\le \tilde \P_2} (t,\eta)d\eta \\
&\quad+\frac{2}{\A^4}\int_0^1\abs{\frac{1}{\sqP{1}}\big(\tilde \Q_1\tilde\U_1-\tilde \Q_2\tilde\U_2\big)+\big(\frac{1}{\sqP{1}}-\frac{1}{\sqP{2}}\big)\tilde \Q_2\tilde\U_2 }^2
\mathbbm{1}_{\tilde \P_2< \tilde \P_1} (t,\eta)d\eta \\
&\le\frac{4}{\A^4}\int_0^1\abs{\frac{1}{\max_i(\sqP{i})}\big(\tilde \Q_1(\tilde\U_1-\tilde\U_2)+\tilde\U_2(\tilde \Q_1-\tilde\Q_2)\big)}^2
\mathbbm{1}_{\tilde \P_1\le \tilde \P_2} (t,\eta)d\eta \\
&\quad+\frac{4}{\A^4}\int_0^1\abs{\frac{\tilde \Q_1\tilde\U_1}{\sqP{1}\sqP{2}}\big(\sqP{2}-\sqP{1}\big)}^2
\mathbbm{1}_{\tilde \P_1\le \tilde \P_2} (t,\eta)d\eta \\
&\quad+\frac{4}{\A^4}\int_0^1\abs{\frac{1}{\max_i(\sqP{i})}\big(\tilde \Q_1(\tilde\U_1-\tilde\U_2)+\tilde\U_2(\tilde \Q_1-\tilde\Q_2)\big)}^2
\mathbbm{1}_{\tilde \P_2< \tilde \P_1} (t,\eta)d\eta \\
&\quad+\frac{4}{\A^4}\int_0^1\abs{\frac{\tilde \Q_2\tilde\U_2}{\sqP{1}\sqP{2}}\big(\sqP{2}-\sqP{1}\big)}^2
\mathbbm{1}_{\tilde \P_2< \tilde \P_1} (t,\eta)d\eta\\
&= 4\big(I_{31}+I_{32}+I_{33}+I_{34}\big).
\end{align*}
We treat terms  $I_{31}$ and $I_{32}$; the others are similar.   
\begin{align*}
I_{31}&\le\frac{2}{\A^4} \int_0^1\Big|\Big(\frac{\tilde \Q_1}{\max_i(\sqP{i})}\Big)^2 \abs{\tilde\U_1-\tilde\U_2}^2\\
&\qquad\qquad\qquad\qquad+\Big(\frac{\tilde \U_2}{\max_i(\sqP{i})}\Big)^2 \abs{\tilde\Q_1-\tilde\Q_2}^2\Big| \mathbbm{1}_{\tilde \P_1\le \tilde \P_2}(t,\eta)d\eta\\
 & \leq \bigO(1)\Big(\norm{\tilde\Y_1-\tilde\Y_2}^2+\norm{\tilde\U_1-\tilde\U_2}^2+\norm{\sqP{1}-\sqP{2}}^2+\vert \sqtC{1}-\sqtC{2}\vert ^2\Big),
\end{align*}
where we have used \eqref{LipQ}, that
\begin{align*}
\abs{\frac{\tilde \Q_1}{\max_i(\sqP{i})}}\mathbbm{1}_{\tilde \P_1\le \tilde \P_2}&\le \frac{\sqtC{1}\tilde \P_1}{\max_i(\sqP{i})}\mathbbm{1}_{\tilde \P_1\le \tilde \P_2}\\
&\le \A\frac{\max_i(\tilde\P_i)}{\max_i(\sqP{i})}\mathbbm{1}_{\tilde \P_1\le \tilde \P_2}\le \A\max_i(\sqP{i})\le \frac{\A^3}{2},\\
\intertext{and} \\
\abs{\frac{\tilde \U_2}{\max_i(\sqP{i})}}\mathbbm{1}_{\tilde \P_1\le \tilde \P_2}&\le \frac{\abs{\tilde \U_2}}{\sqP{2}}\mathbbm{1}_{\tilde \P_1\le \tilde \P_2}\le \sqrt{2}.
\end{align*}
Next follows 
\begin{align*}
I_{32}&\le\frac{1}{\A^4}\int_0^1\abs{\frac{\tilde \Q_1\tilde\U_1}{\sqP{1}\sqP{2}}\big(\sqP{2}-\sqP{1}\big)}^2 \mathbbm{1}_{\tilde \P_1\le \tilde \P_2} (t,\eta)d\eta \\
& \le\bigO(1) \norm{\sqP{2}-\sqP{1}}^2,
\end{align*}
using that
\begin{align*}
\abs{\frac{\tilde \Q_1\tilde\U_1}{\sqP{1}\sqP{2}}}\mathbbm{1}_{\tilde \P_1\le \tilde \P_2}\le \abs{\frac{\tilde \Q_1\tilde\U_1}{\tilde\P_1}}
\le \abs{\sqtC{1}\tilde\U_1}\le \frac{\A^3}{\sqrt{2}}.
\end{align*}
This proves that 
\begin{multline*}
\frac{1}{\A^4}\norm{\frac{\tilde \Q_1\tilde\U_1}{\sqP{1}}-\frac{\tilde \Q_2\tilde\U_2}{\sqP{2}}}^2\\
\le \bigO(1)\Big(\norm{\tilde\Y_1-\tilde\Y_2}^2+\norm{\tilde\U_1-\tilde\U_2}^2+\norm{\sqP{1}-\sqP{2}}^2+\vert \sqtC{1}-\sqtC{2}\vert ^2\Big),
\end{multline*}
and thus $I_3$ has the right form.

\bigskip  
%-----------------------------------
\textit{The term $I_4$:}
\begin{align*}
I_4&= \int_0^{1}\Big(\sqP{1}-\sqP{2}\Big)
\Big(\frac{1}{\sqtC{1}^3}\frac{\tilde \R_1}{\sqP{1}}-\frac{1}{\sqtC{2}^3}\frac{\tilde \R_2}{\sqP{2}}\Big)(t,\eta)d\eta\\
&= \mathbbm{1}_{\sqtC{1}\leq \sqtC{2}}\Big(\frac{1}{\sqtC{1}^3}-\frac{1}{\sqtC{2}^3}\Big)\int_0^{1}\Big(\sqP{1}-\sqP{2}\Big)
\frac{\tilde \R_1}{\sqP{1}}(t,\eta)d\eta\\
& \quad +\mathbbm{1}_{\sqtC{2}< \sqtC{1}}\Big(\frac{1}{\sqtC{1}^3}-\frac{1}{\sqtC{2}^3}\Big)\int_0^{1}\Big(\sqP{1}-\sqP{2}\Big)
\frac{\tilde \R_2}{\sqP{2}}(t,\eta)d\eta\\
& \quad +\frac{1}{\A^3} \int_0^{1}\Big(\sqP{1}-\sqP{2}\Big)\Big(\frac1{\max_j(\sqP{j})}(\tilde\R_1-\tilde\R_2)\\
&\qquad\qquad\qquad\qquad-\big(\frac1{\sqP{2}}-\frac1{\sqP{1}}\big)\big(\tilde\R_2\mathbbm{1}_{\tilde \P_2<\tilde \P_1}+\tilde\R_1\mathbbm{1}_{\tilde \P_1\le\tilde \P_2}\big)  \Big)(t,\eta)d\eta\\
&= \mathbbm{1}_{\sqtC{1}\leq \sqtC{2}}\Big(\frac{1}{\sqtC{1}^3}-\frac{1}{\sqtC{2}^3}\Big)\int_0^{1}\Big(\sqP{1}-\sqP{2}\Big)
\frac{\tilde \R_1}{\sqP{1}}(t,\eta)d\eta\\
& \quad +\mathbbm{1}_{\sqtC{2}< \sqtC{1}}\Big(\frac{1}{\sqtC{1}^3}-\frac{1}{\sqtC{2}^3}\Big)\int_0^{1}\Big(\sqP{1}-\sqP{2}\Big)
\frac{\tilde \R_2}{\sqP{2}}(t,\eta)d\eta\\
& \quad +\frac{1}{\A^3}\int_0^{1}\Big[\frac1{\max_j(\sqP{j})}(\sqP{1}-\sqP{2})(\tilde\R_1-\tilde\R_2)-(\sqP{1}\\
&\qquad\qquad\qquad\qquad-\sqP{2})^2\frac1{\sqP{1}\sqP{2}}\big(\tilde\R_2\mathbbm{1}_{\tilde \P_2<\tilde \P_1}+\tilde\R_1\mathbbm{1}_{\tilde \P_1\le\tilde \P_2}\big)  \Big](t,\eta)d\eta\\
&=J_1+J_2+J_3-J_4-J_5,
\end{align*}
where
\begin{align*}
J_1&=\mathbbm{1}_{\sqtC{1}\leq \sqtC{2}}\Big(\frac{1}{\sqtC{1}^3}-\frac{1}{\sqtC{2}^3}\Big)\int_0^{1}\Big(\sqP{1}-\sqP{2}\Big)
\frac{\tilde \R_1}{\sqP{1}}(t,\eta)d\eta,\\[2mm]
J_2&=\mathbbm{1}_{\sqtC{2}< \sqtC{1}}\Big(\frac{1}{\sqtC{1}^3}-\frac{1}{\sqtC{2}^3}\Big)\int_0^{1}\Big(\sqP{1}-\sqP{2}\Big)
\frac{\tilde \R_2}{\sqP{2}}(t,\eta)d\eta,\\[2mm]
J_3&=\frac{1}{\A^3} \int_0^{1}\frac1{\max_j(\sqP{j})}(\sqP{1}-\sqP{2})(\tilde\R_1-\tilde\R_2)(t,\eta)d\eta,\\[2mm]
J_4&=\frac{1}{\A^3}\int_0^{1} (\sqP{1}-\sqP{2})^2\frac1{\sqP{1}\sqP{2}}\tilde\R_2\mathbbm{1}_{\tilde \P_2<\tilde \P_1}(t,\eta)d\eta,\\[2mm]
J_5&= \frac{1}{\A^3}\int_0^{1} (\sqP{1}-\sqP{2})^2\frac1{\sqP{1}\sqP{2}} \tilde\R_1\mathbbm{1}_{\tilde \P_1\le\tilde \P_2}(t,\eta)d\eta.
\end{align*}

For $J_1$ (and similar for $J_2$), we find
\begin{align*}
\abs{J_1}&\leq \frac{\sqtC{2}^3-\sqtC{1}^3}{\ma^3\A^3}\int_0^{1}\Big\vert\sqP{1}-\sqP{2}\Big\vert
(t,\eta)\bigO(1)\ma^5d\eta\\
& \leq \bigO(1)(\norm{\sqP{1}-\sqP{2}}^2+\vert \sqtC{1}-\sqtC{2}\vert^2),
\end{align*}
where we used that 
\begin{equation*}
\vert \frac{\tilde \R_i}{\sqP{i}}\vert\leq \bigO(1)\sqtC{i}^3\sqP{i}\leq \bigO(1)\sqtC{i}^5.
\end{equation*}

For $J_3$ we find: 
\begin{align*}
J_3&=\frac{1}{\A^3}\int_0^{1}\frac1{\max_j(\sqP{j})}(\sqP{1}-\sqP{2})(\tilde\R_1-\tilde\R_2)(t,\eta)d\eta \\
&= \frac{1}{6\A^3}\int_0^{1}\frac1{\max_j(\sqP{j})}(\sqP{1}-\sqP{2})(t,\eta)\Big(\int_0^1 \sign(\eta-\theta)\\
&\quad \times \big(e^{-\frac{1}{\sqtC{1}}\vert \tilde\Y_1(t,\eta)-\tilde\Y_1(t,\theta)\vert} \sqtC{1}\tilde \U_1^3\tilde \Y_{1,\eta}
-e^{-\frac{1}{\sqtC{2}}\vert \tilde\Y_2(t,\eta)-\tilde\Y_2(t,\theta)\vert} \sqtC{2}\tilde \U_2^3\tilde \Y_{2,\eta}\big)(t,\theta)d\theta\Big)d\eta \\
&\quad+\frac1{4\A^3}\int_0^{1}\frac1{\max_j(\sqP{j})}(\sqP{1}-\sqP{2})(t,\eta)\\
&\qquad\qquad\qquad\times\Big(\int_0^1 \sign(\eta-\theta)\big(e^{-\frac{1}{\sqtC{1}}\vert \tilde\Y_1(t,\eta)-\tilde\Y_1(t,\theta)\vert} \sqtC{1}^6\tilde\U_1\\
&\qquad\qquad\qquad\qquad-e^{-\frac{1}{\sqtC{2}}\vert \tilde\Y_2(t,\eta)-\tilde\Y_2(t,\theta)\vert} \sqtC{2}^6\tilde\U_2\big)
(t,\theta)d\theta\Big)d\eta\\
&\quad+\frac1{2\A^3}\int_0^{1}\frac1{\max_j(\sqP{j})}(\sqP{1}-\sqP{2})(t,\eta)\\
&\qquad\qquad\qquad\times\Big(\int_0^1\big(e^{-\frac{1}{\sqtC{2}}\vert \tilde\Y_2(t,\eta)-\tilde\Y_2(t,\theta)\vert} \tilde \U_2\tilde\Q_2\tilde \Y_{2,\eta}\\
&\qquad\qquad\qquad\qquad
-e^{-\frac{1}{\sqtC{1}}\vert \tilde\Y_1(t,\eta)-\tilde\Y_1(t,\theta)\vert} \tilde \U_1\tilde\Q_1\tilde \Y_{1,\eta}\big)(t,\theta)d\theta\Big)d\eta\\
&= J_{11}+J_{12}+J_{13}.
\end{align*}
Next we replace all the inner integrals $\int_0^1\dots d\theta$ by $(\int_0^\eta+\int_\eta^1)\dots d\theta$, cf.~\eqref{eq:triks}, consider only the terms with  $\int_0^\eta\dots d\theta$, 
and call the corresponding quantities $\tilde J_{11},\tilde J_{12},\tilde J_{13}$. Thus
\begin{align}
\tilde J_{11}&=\frac1{6\A^3}\int_0^{1}\frac1{\max_j(\sqP{j})}(\sqP{1}-\sqP{2})(t,\eta) \Big(\int_0^\eta\big(e^{-\frac{1}{\sqtC{1}}(\tilde\Y_1(t,\eta)-\tilde\Y_1(t,\theta))} \sqtC{1}\tilde \U_1^3\tilde \Y_{1,\eta}\notag\\ 
&\qquad \qquad\qquad\qquad\qquad\qquad\qquad
-e^{-\frac{1}{\sqtC{2}}(\tilde\Y_2(t,\eta)-\tilde\Y_2(t,\theta))} \sqtC{2}\tilde \U_2^3\tilde \Y_{2,\eta}\big)(t,\theta)d\theta\Big)d\eta, \label{eq:tildeJ11}\\[2mm]
\tilde J_{12}&=\frac1{4\A^3}\int_0^{1}\frac1{\max_j(\sqP{j})}(\sqP{1}-\sqP{2})(t,\eta)\Big(\int_0^\eta \big(e^{-\frac{1}{\sqtC{1}}(\tilde\Y_1(t,\eta)-\tilde\Y_1(t,\theta))} \sqtC{1}^6\tilde\U_1\notag\\ 
&\qquad \qquad\qquad\qquad\qquad\qquad\qquad-e^{-\frac{1}{\sqtC{2}}(\tilde\Y_2(t,\eta)-\tilde\Y_2(t,\theta))} \sqtC{2}^6\tilde\U_2\big)
(t,\theta)d\theta\Big)d\eta, \label{eq:tildeJ12} \\[2mm]
\tilde J_{13}&=\frac1{2\A^3}\int_0^{1}\frac1{\max_j(\sqP{j})}(\sqP{1}-\sqP{2})(t,\eta)\Big(\int_0^\eta\big(e^{-\frac{1}{\sqtC{2}}(\tilde\Y_2(t,\eta)-\tilde\Y_2(t,\theta))} \tilde \U_2\tilde\Q_2\tilde \Y_{2,\eta}\notag\\ 
&\qquad \qquad\qquad\qquad\qquad\qquad\qquad
-e^{-\frac{1}{\sqtC{1}}(\tilde\Y_1(t,\eta)-\tilde\Y_1(t,\theta))} \tilde \U_1\tilde\Q_1\tilde \Y_{1,\eta}\big)(t,\theta)d\theta\Big)d\eta.  \label{eq:tildeJ13}
\end{align}
For the term $\tilde J_{11}$ we write $ \tilde\U_j= \tilde\U_j^+ + \tilde\U_j^-$, collect terms, and study the positive part and negative part separately.

With a slight abuse of notation we need to consider the term 
\begin{align*}
\tilde J_{11}&=\frac1{6\A^3}\int_0^{1}\frac1{\max_j(\sqP{j})}(\sqP{1}-\sqP{2})(t,\eta) \\
&\qquad\qquad\qquad\times\Big(\int_0^\eta\big(e^{-\frac{1}{\sqtC{1}}(\tilde\Y_1(t,\eta)-\tilde\Y_1(t,\theta))}\sqtC{1} (\tilde \U_1^+)^3\tilde \Y_{1,\eta}\\
&\qquad\qquad\qquad\qquad-e^{-\frac{1}{\sqtC{2}}(\tilde\Y_2(t,\eta)-\tilde\Y_2(t,\theta))}\sqtC{2} (\tilde \U_2^+)^3\tilde \Y_{2,\eta}\big)(t,\theta)d\theta\Big)d\eta \\
&= \frac1{6\A^3}\int_0^{1}\frac1{\max_j(\sqP{j})}(\sqP{1}-\sqP{2})(t,\eta)\\ 
& \quad\times \Big(\mathbbm{1}_{\sqtC{2}<\sqtC{1}}(\sqtC{1}-\sqtC{2})\int_0^\eta e^{-\frac{1}{\sqtC{1}}(\tilde \Y_1(t,\eta)-\tilde \Y_1(t,\theta))}(\tilde \V_1^+)^3\tilde \Y_{1,\eta}(t,\theta) d\theta\\
& \qquad +\mathbbm{1}_{\sqtC{1}\leq \sqtC{2}}(\sqtC{1}-\sqtC{2})\int_0^\eta e^{-\frac{1}{\sqtC{2}}(\tilde \Y_2(t,\eta)-\tilde \Y_2(t,\theta))}(\tilde \V_2^+)^3\tilde \Y_{2,\eta}(t,\theta)d\theta\\
&\qquad + \ma\int_0^\eta e^{-\frac{1}{\sqtC{1}}(\tilde\Y_1(t,\eta)-\tilde\Y_1(t,\theta))} \big((\tilde \U_1^+)^3-(\tilde \U_2^+)^3\big)\mathbbm{1}_{\tilde \U_2^+<\tilde \U_1^+}\tilde \Y_{1,\eta}(t,\theta)d\theta\\
&\qquad +\ma\int_0^\eta e^{-\frac{1}{\sqtC{2}}(\tilde\Y_2(t,\eta)-\tilde\Y_2(t,\theta))} \big((\tilde \U_1^+)^3-(\tilde \U_2^+)^3\big)\mathbbm{1}_{\tilde \U_1^+\le\tilde \U_2^+}\tilde \Y_{2,\eta}(t,\theta)d\theta\\
&\qquad -\ma\mathbbm{1}_{\sqtC{2}<\sqtC{1}}\int_0^\eta\big(e^{-\frac{1}{\sqtC{2}}(\tilde\Y_1(t,\eta)-\tilde\Y_1(t,\theta))} -  e^{-\frac{1}{\sqtC{1}}(\tilde\Y_1(t,\eta)-\tilde\Y_1(t,\theta))}\big)\\
&\qquad\qquad\qquad\qquad\qquad\qquad\qquad\qquad\qquad\times\min_j\big((\tilde \U_j^+)^3 \big)\tilde\Y_{1,\eta}
(t,\theta)d\theta\\
&\qquad -\ma\mathbbm{1}_{\sqtC{1}\leq \sqtC{2}}\int_0^\eta\big(e^{-\frac{1}{\sqtC{2}}(\tilde\Y_2(t,\eta)-\tilde\Y_2(t,\theta))} -  e^{-\frac{1}{\sqtC{1}}(\tilde\Y_2(t,\eta)-\tilde\Y_2(t,\theta))}\big)\\
&\qquad\qquad\qquad\qquad\qquad\qquad\qquad\qquad\qquad\times\min_j\big((\tilde \U_j^+)^3 \big)\tilde \Y_{2,\eta}
(t,\theta)d\theta\\
&\qquad -\ma\int_0^\eta\big(e^{-\frac{1}{\ma}(\tilde\Y_2(t,\eta)-\tilde\Y_2(t,\theta))} -  e^{-\frac{1}{\ma}(\tilde\Y_1(t,\eta)-\tilde\Y_1(t,\theta))}\big)\\
&\qquad\qquad\qquad\qquad\qquad\qquad\qquad\qquad\qquad\times\min_j\big((\tilde \U_j^+)^3 \big)\tilde \Y_{2,\eta}
\mathbbm{1}_{B(\eta)}(t,\theta)d\theta\\
&\qquad -\ma\int_0^\eta\big(e^{-\frac{1}{\ma}(\tilde\Y_2(t,\eta)-\tilde\Y_2(t,\theta))} -  e^{-\frac{1}{\ma}(\tilde\Y_1(t,\eta)-\tilde\Y_1(t,\theta))}\big)\\
&\qquad\qquad\qquad\qquad\qquad\qquad\qquad\qquad\qquad\times\min_j\big((\tilde \U_j^+)^3 \big)\tilde \Y_{1,\eta}
\mathbbm{1}_{B(\eta)^c}(t,\theta)d\theta\\
&\qquad  -\ma\int_0^\eta \min_j\big( e^{-\frac{1}{\ma}(\tilde\Y_j(t,\eta)-\tilde\Y_j(t,\theta))}\big)\min_j\big((\tilde \U_j^+)^3 \big)\tilde \Y_{2,\eta}(t,\theta)d\theta\\
&\qquad  +\ma\int_0^\eta \min_j\big( e^{-\frac{1}{\ma}(\tilde\Y_j(t,\eta)-\tilde\Y_j(t,\theta))}\big)\min_j\big((\tilde \U_j^+)^3 \big)\tilde \Y_{1,\eta}(t,\theta)d\theta\Big)d\eta \\
&=K_1+K_2+K_3+K_4+K_5+K_6+K_7+K_8+K_9+K_{10}.
\end{align*}
It never stops, but we need to study the terms in groups.  

The term $K_1$ can be estimated as follows ($K_2$ is similar):
\begin{align*} 
\abs{K_1}& =\frac1{6\A^3}\int_0^{1}\frac1{\max_j(\sqP{j})}\vert\sqP{1}-\sqP{2}\vert(t,\eta)\\ 
& \quad\times \Big\vert\mathbbm{1}_{\sqtC{2}<\sqtC{1}}(\sqtC{1}-\sqtC{2})\int_0^\eta e^{-\frac{1}{\sqtC{1}}(\tilde \Y_1(t,\eta)-\tilde \Y_1(t,\theta))}(\tilde \V_1^+)^3\tilde \Y_{1,\eta}(t,\theta) d\theta\Big\vert d\eta\\
& \leq \frac{\sqrt{2}}{3}\int_0^1 \frac{\tilde\P_1}{\max_j(\sqP{j})}\vert \sqP{1}-\sqP{2}\vert (t,\eta) d\eta \vert \sqtC{1}-\sqtC{2}\vert \\
& \leq \bigO(1)(\norm{\sqP{1}-\sqP{2}}^2+\vert \sqtC{1}-\sqtC{2}\vert^2).
\end{align*}

The term $K_3$ can be estimated as follows ($K_4$ is similar):
\begin{align*}
\abs{K_3}&=\frac{\ma}{6\A^3}\Big|\int_0^{1}\frac1{\max_j(\sqP{j})}(\sqP{1}-\sqP{2})(t,\eta)\\ 
&\qquad \times \Big(\int_0^\eta e^{-\frac{1}{\sqtC{1}}(\tilde\Y_1(t,\eta)-\tilde\Y_1(t,\theta))}
 \big((\tilde \U_1^+)^3-(\tilde \U_2^+)^3\big)\mathbbm{1}_{\tilde \U_2^+<\tilde \U_1^+}\tilde \Y_{1,\eta}(t,\theta)d\theta\Big)d\eta\Big|\\
 &\le \frac{1}{2\A^2}\int_0^{1}\frac1{\max_j(\sqP{j})}\abs{\sqP{1}-\sqP{2}}(t,\eta)\\ 
&\qquad \times \Big(\int_0^\eta e^{-\frac{1}{\sqtC{1}}(\tilde\Y_1(t,\eta)-\tilde\Y_1(t,\theta))} (\tilde \U_1^+)^2\abs{\tilde \U_1^+ -\tilde \U_2^+}\tilde \Y_{1,\eta}(t,\theta)d\theta\Big)d\eta\\
&\le\frac12\norm{\sqP{1}-\sqP{2}}^2  + \frac1{2\A^4} \int_0^{1} \frac1{\max_j(\sqP{j})^2}\\
&\qquad\qquad\times
\Big(\int_0^\eta e^{-\frac{1}{\sqtC{1}}(\tilde\Y_1(t,\eta)-\tilde\Y_1(t,\theta))} (\tilde \U_1^+)^2\abs{\tilde \U_1^+ -\tilde \U_2^+}\tilde \Y_{1,\eta}(t,\theta)d\theta \Big)^2d\eta \\
&\le \frac12 \norm{\sqP{1}-\sqP{2}}^2 \\
&\quad + \frac1{2\A^4} \int_0^{1} \frac1{\max_j(\sqP{j})^2} \Big(\int_0^\eta e^{-\frac{1}{\sqtC{1}}(\tilde\Y_1(t,\eta)-\tilde\Y_1(t,\theta))}(\tilde \U_1^+)^2 \tilde \Y_{1,\eta}(t,\theta)d\theta \Big) \\
&\qquad\times \Big(\int_0^\eta e^{-\frac{1}{\sqtC{1}}(\tilde\Y_1(t,\eta)-\tilde\Y_1(t,\theta))}\abs{\tilde \U_1^+ -\tilde \U_2^+}^2(\tilde \U_1^+)^2 \tilde \Y_{1,\eta}(t,\theta)d\theta \Big)d\eta\\
&\le \bigO(1)\big(\norm{\sqP{1}-\sqP{2}}^2+\norm{\tilde \U_1 -\tilde \U_2}^2 \big),
\end{align*}
using \eqref{eq:all_PestimatesC}, \eqref{eq:all_estimatesG}, \eqref{eq:pm2}.

As for $K_5$ ($K_6$ is similar) we traverse the following path
\begin{align*}
\abs{K_5}
=& \frac1{6\A^3}\vert\int_0^{1}\frac1{\max_j(\sqP{j})}\vert\sqP{1}-\sqP{2}\vert(t,\eta)\\ 
&\qquad \times\Big(-\ma\mathbbm{1}_{\sqtC{2}<\sqtC{1}}\int_0^\eta\big(e^{-\frac{1}{\sqtC{2}}(\tilde\Y_1(t,\eta)-\tilde\Y_1(t,\theta))} -  e^{-\frac{1}{\sqtC{1}}(\tilde\Y_1(t,\eta)-\tilde\Y_1(t,\theta))}\big)\\
&\qquad\qquad\qquad\qquad\qquad\qquad\qquad\qquad\times\min_j\big((\hat \U_j^+)^3 \big)\tilde\Y_{1,\eta}
(t,\theta)d\theta\Big) d\eta\vert\\
& \leq \frac{2}{3\A^3e}\int_0^1 \frac{1}{\max_j(\sqP{j})}\vert \sqP{1}-\sqP{2}\vert (t,\eta)\\
& \qquad \times \Big(\int_0^\eta e^{-\frac{3}{4\sqtC{1}}(\tilde \Y_1(t,\eta)-\tilde \Y_1(t,\theta))}\min_j((\tilde \V_j^+)^3)\tilde \Y_{1,\eta}(t,\theta) d\theta\Big) d\eta \vert \sqtC{1}-\sqtC{2}\vert\\
& \leq\frac{ 2\sqrt{2}\A^2}{3}\int_0^1 \frac{\sqP{1}}{\max_j(\sqP{j})}\vert \sqP{1}-\sqP{2}\vert (t,\eta) d\eta\vert \sqtC{1}-\sqtC{2}\vert\\
& \leq \bigO(1)(\norm{\sqP{1}-\sqP{2}}^2+\vert \sqtC{1}-\sqtC{2}\vert^2).
\end{align*}

As for $K_7$ ($K_8$ is similar) we traverse the following path
\begin{align*}
\abs{K_7}&= \frac\ma{6\A^3}\Big\vert\int_0^{1}\frac1{\max_j(\sqP{j})}(\sqP{1}-\sqP{2})(t,\eta)\\ 
&\qquad \times\Big(\int_0^\eta\big(e^{-\frac{1}{\ma}(\tilde\Y_2(t,\eta)-\tilde\Y_2(t,\theta))} -  e^{-\frac{1}{\ma}(\tilde\Y_1(t,\eta)-\tilde\Y_1(t,\theta))}\big)\\
&\qquad\qquad\qquad\qquad\qquad\qquad\qquad\qquad\times\min_j\big((\tilde \U_j^+)^3 \big)\tilde \Y_{2,\eta}
\mathbbm{1}_{B(\eta)}(t,\theta)d\theta\Big)d\eta\Big\vert \\
&\le \frac1{6\A^3}\int_0^{1}\frac1{\max_j(\sqP{j})}\abs{\sqP{1}-\sqP{2}}(t,\eta) \Big( \int_0^\eta e^{-\frac{1}{\ma}(\tilde\Y_2(t,\eta)-\tilde\Y_2(t,\theta))}\\
&\qquad\times \big( \abs{\tilde\Y_2(t,\eta)-\tilde\Y_1(t,\eta)}+\abs{\tilde\Y_2(t,\theta)-\tilde\Y_1(t,\theta)}\big)
(\tilde \U_2^+)^3\tilde \Y_{2,\eta}(t,\theta)d\theta\Big)d\eta \\
&\le \frac1{6\A^3}\int_0^{1}\frac1{\max_j(\sqP{j})}\abs{\sqP{1}-\sqP{2}} \abs{\tilde\Y_2-\tilde\Y_1}(t,\eta)\\
&\qquad\times \Big( \int_0^\eta e^{-\frac{1}{\sqtC{2}}(\tilde\Y_2(t,\eta)-\tilde\Y_2(t,\theta))}
(\tilde \U_2^+)^3\tilde \Y_{2,\eta}(t,\theta)d\theta\Big)d\eta  \\
&\quad +\frac1{6\A^3}\int_0^{1}\frac1{\max_j(\sqP{j})}\abs{\sqP{1}-\sqP{2}} (t,\eta) \\
&\qquad\times\Big(\int_0^\eta e^{-\frac{1}{\ma}(\tilde\Y_2(t,\eta)-\tilde\Y_2(t,\theta))}\abs{\tilde\Y_2-\tilde\Y_1}
(\tilde \U_2^+)^3\tilde \Y_{2,\eta}(t,\theta)d\theta\Big)d\eta  \\
&\le \frac{\sqrt{2}}{3}\int_0^{1}\frac1{\max_j(\sqP{j})}\abs{\sqP{1}-\sqP{2}} \abs{\tilde\Y_2-\tilde\Y_1}  \tilde\P_2 (t,\eta)d\eta\\
&\quad +\norm{\sqP{1}-\sqP{2}}^2 \\
&\quad +\frac1{36\A^6}\int_0^{1}\frac1{\max_j(\sqP{j})^2}(t,\eta)\\
&\qquad\qquad\times\Big(\int_0^\eta  e^{-\frac{1}{\ma}(\tilde\Y_2(t,\eta)-\tilde\Y_2(t,\theta))}\abs{\tilde\Y_2-\tilde\Y_1}
(\tilde \U_2^+)^3\tilde \Y_{2,\eta}(t,\theta)d\theta\Big)^2d\eta  \\
&\le \bigO(1)\big(\norm{\sqP{1}-\sqP{2}}^2+ \norm{\tilde\Y_2-\tilde\Y_1}^2 \big) \\
&\quad +\frac{1}{36\A^6}\int_0^{1}\frac1{\max_j(\sqP{j})^2}(t,\eta)\Big(\int_0^\eta  e^{-\frac{1}{\sqtC{2}}(\tilde\Y_2(t,\eta)-\tilde\Y_2(t,\theta))}(\tilde \U_2^+)^2\tilde\Y_{2,\eta}(t,\theta)d\theta\Big) \\
&\qquad \times \Big(\int_0^\eta  e^{-\frac{1}{\sqtC{2}}(\tilde\Y_2(t,\eta)-\tilde\Y_2(t,\theta))} \abs{\tilde\Y_2-\tilde\Y_1}^2(\tilde \U_2^+)^4\tilde \Y_{2,\eta}(t,\theta)d\theta\Big)d\eta  \\
&\le \bigO(1)\big(\norm{\sqP{1}-\sqP{2}}^2+ \norm{\tilde\Y_2-\tilde\Y_1}^2 \big) \\
&\quad +\frac{1}{9\A^5}\int_0^{1}\frac1{\max_j(\sqP{j})^2}\tilde\P_2(t,\eta) \\
&\qquad \times \Big(\int_0^\eta  e^{-\frac{1}{\sqtC{2}}(\tilde\Y_2(t,\eta)-\tilde\Y_2(t,\theta))} \abs{\tilde\Y_2-\tilde\Y_1}^2(\tilde \U_2^+)^4\tilde \Y_{2,\eta}(t,\theta)d\theta\Big)d\eta  \\
&\le \bigO(1)\big(\norm{\sqP{1}-\sqP{2}}^2+ \norm{\tilde\Y_2-\tilde\Y_1}^2 \big).
\end{align*}
Here we have applied \eqref{Diff:Exp}, \eqref{eq:all_PestimatesC}, \eqref{eq:all_estimatesG}, and \eqref{eq:all_estimatesB}. 
Lo and behold, we can do $K_9$ and $K_{10}$ in one sweep: 
\begin{align*}
K_9+K_{10}&= -\frac{\ma}{6\A^3}\int_0^{1}\frac1{\max_j(\sqP{j})}(\sqP{1}-\sqP{2})(t,\eta)\\ 
&\quad  \times\Big(\int_0^\eta \min_j\big( e^{-\frac{1}{\ma}(\tilde\Y_j(t,\eta)-\tilde\Y_j(t,\theta))}\big)\min_j\big((\tilde \U_j^+)^3 \big)\tilde \Y_{2,\eta}(t,\theta)d\theta\\
&\qquad  -\int_0^\eta \min_j\big( e^{-\frac{1}{\ma}(\tilde\Y_j(t,\eta)-\tilde\Y_j(t,\theta))}\big)\min_j\big((\tilde \U_j^+)^3 \big)\tilde \Y_{1,\eta}(t,\theta)d\theta\Big)d\eta \\
&= -\frac{\ma}{6\A^3}\int_0^{1}\frac1{\max_j(\sqP{j})}(\sqP{1}-\sqP{2})(t,\eta)\\
&\quad  \times\Big(\int_0^\eta \min_j\big( e^{-\frac{1}{\ma}(\tilde\Y_j(t,\eta)-\tilde\Y_j(t,\theta))}\big)\min_j\big((\tilde \U_j^+)^3 \big)(\tilde \Y_{2,\eta}-\tilde \Y_{1,\eta})(t,\theta)d\theta\Big)d\eta \\
&= -\frac{\ma}{6\A^3}\int_0^{1}\frac1{\max_j(\sqP{j})}(\sqP{1}-\sqP{2})(t,\eta)\\
&\quad  \times\Big[ \Big(\min_j\big( e^{-\frac{1}{\ma}(\tilde\Y_j(t,\eta)-\tilde\Y_j(t,\theta))}\big)\min_j\big((\tilde \U_j^+)^3 \big)(\tilde \Y_{2}-\tilde \Y_{1})\Big)(t,\theta)\Big|_{\theta=0}^\eta\\
&\quad -\int_0^\eta (\tilde \Y_{2}-\tilde \Y_{1})\frac{d}{d\theta} \Big(\min_j\big( e^{-\frac{1}{\ma}(\tilde\Y_j(t,\eta)-\tilde\Y_j(t,\theta))}\big)\min_j\big((\tilde \U_j^+)^3 \big)\Big)(t,\theta)d\theta\Big]d\eta \\
&=- \frac{\ma}{6\A^3}\int_0^{1}\frac1{\max_j(\sqP{j})}(\sqP{1}-\sqP{2})(\tilde \Y_{2}-\tilde \Y_{1})\min_j\big((\tilde \U_j^+)^3 \big)(t,\eta)d\eta\\
&\qquad+  \frac\ma{6\A^3}\int_0^{1}\frac1{\max_j(\sqP{j})}(\sqP{1}-\sqP{2})(t,\eta) \\
&\quad \times\Big(\int_0^\eta (\tilde \Y_{2}-\tilde \Y_{1}) \min_j\big((\tilde \U_j^+)^3 \big)\frac{d}{d\theta}\min_j\big( e^{-\frac{1}{\ma}(\tilde\Y_j(t,\eta)-\tilde\Y_j(t,\theta))}\big)(t,\theta)d\theta\Big)d\eta\\
&\quad + \frac{\ma}{6\A^3}\int_0^{1}\frac1{\max_j(\sqP{j})}(\sqP{1}-\sqP{2})(t,\eta) \\
&\quad \times\Big(\int_0^\eta \min_j\big( e^{-\frac{1}{\ma}(\tilde\Y_j(t,\eta)-\tilde\Y_j(t,\theta))}\big)(\tilde \Y_{2}-\tilde \Y_{1}) \frac{d}{d\theta}\min_j\big((\tilde \U_j^+)^3 \big)(t,\theta)d\theta\Big)d\eta\\
&= L_1+L_2+L_3,
\end{align*}
where
\begin{align*}
L_1&=  -\frac{\ma}{6\A^3}\int_0^{1}\frac1{\max_j(\sqP{j})}(\sqP{1}-\sqP{2})(\tilde \Y_{2}-\tilde \Y_{1})\min_j\big((\tilde \U_j^+)^3 \big)(t,\eta)d\eta,  \\[2mm]
L_2&=  \frac{\ma}{6\A^3}\int_0^{1}\frac1{\max_j(\sqP{j})}(\sqP{1}-\sqP{2})(t,\eta) \\
&\qquad \times\Big(\int_0^\eta (\tilde \Y_{2}-\tilde \Y_{1}) \min_j\big((\tilde \U_j^+)^3 \big)\frac{d}{d\theta}\min_j\big( e^{-\frac{1}{\ma}(\tilde\Y_j(t,\eta)-\tilde\Y_j(t,\theta))}\big)(t,\theta)d\theta\Big)d\eta,\\[2mm]
L_3&= \frac{\ma}{6\A^3}\int_0^{1}\frac1{\max_j(\sqP{j})}(\sqP{1}-\sqP{2})(t,\eta) \\
&\qquad \times\Big(\int_0^\eta \min_j\big( e^{-\frac{1}{\ma}(\tilde\Y_j(t,\eta)-\tilde\Y_j(t,\theta))}\big)(\tilde \Y_{2}-\tilde \Y_{1}) \frac{d}{d\theta}\min_j\big((\tilde \U_j^+)^3 \big)(t,\theta)d\theta\Big)d\eta.  
\end{align*}

We easily find
\begin{align*}
\abs{L_1}&\le  \frac{\ma}{6\A^3}\int_0^{1}\frac1{\max_j(\sqP{j})}\abs{\sqP{1}-\sqP{2}}\abs{\tilde \Y_{2}-\tilde \Y_{1}}\min_j\big((\tilde \U_j^+)^3 \big)(t,\eta)d\eta  \\
&\le \bigO(1)\big( \norm{\sqP{1}-\sqP{2}}^2+\norm{\tilde \Y_{2}-\tilde \Y_{1}}^2\big),  
\end{align*}
using \eqref{eq:all_estimatesD} and  \eqref{eq:all_estimatesB}.  Furthermore, applying Lemma \ref{lemma:1}
\begin{align*}
\abs{L_2}&\le  \frac1{6\A^3}\int_0^{1}\frac1{\max_j(\sqP{j})}\abs{\sqP{1}-\sqP{2}}(t,\eta) \\
&\qquad \times\Big(\int_0^\eta  \min_j(e^{-\frac{1}{\ma}(\tilde \Y_j(t,\eta)-\tilde\Y_j(t,\theta))})\abs{\tilde \Y_{2}-\tilde \Y_{1}} \\
&\qquad\qquad\qquad\times\min_j\big((\tilde \U_j^+)^3 \big)
 \max_j(\tilde\Y_{j,\eta})(t,\theta)d\theta\Big)d\eta\\
 & \leq \frac1{6\A^3} \int_0^1 \frac{1}{\max_j(\sqP{j})}\abs{\sqP{1}-\sqP{2}}(t,\eta)\\
 & \qquad \times \Big(\int_0^\eta \min_j\big(e^{-\frac{1}{\sqtC{j}}(\tilde\Y_j(t,\eta)-\tilde\Y_j(t,\theta))}\tilde\U_j^2 \big)(t,\theta) d\theta \Big)^\frac12\\
 & \qquad \times \Big(\int_0^\eta \vert \tilde\Y_2-\tilde\Y_1\vert^2 \max_j(\tilde\U_j^4\tilde\Y_{j,\eta})^2(t,\theta) d\theta\Big)^\frac12 d\eta\\
 &\le \bigO(1) \int_0^{1}\abs{\sqP{1}-\sqP{2}}(t,\eta)d\eta\norm{\tilde\Y_1-\tilde\Y_2}\\
 &\le \bigO(1) \big(\norm{\sqP{1}-\sqP{2}}^2+ \norm{\tilde \Y_{2}-\tilde \Y_{1}}^2 \big),
\end{align*}
using \eqref{eq:all_PestimatesA}, \eqref{eq:all_estimatesG}, and
\begin{equation*}
\min_j\big((\tilde \U_j^+)^3 \big) \max_j(\tilde\Y_{j,\eta})\le \min_j\big(\tilde \U_j^+ \big)\max_j(\tilde\U_j^2\tilde\Y_{j,\eta})\le\frac{A^7}{\sqrt{2}}.
\end{equation*}
Another application of  Lemma \ref{lemma:1} yields 
\begin{align*}
\abs{L_3}&\le \frac{\ma}{6\A^3}\int_0^{1}\frac1{\max_j(\sqP{j})}\abs{\sqP{1}-\sqP{2}}(t,\eta) \\
&\qquad \times\Big(\int_0^\eta \min_j\big( e^{-\frac{1}{\ma}(\tilde\Y_j(t,\eta)-\tilde\Y_j(t,\theta))}\big)\abs{\tilde \Y_{2}-\tilde \Y_{1}} \vert\frac{d}{d\theta}\min_j\big((\tilde \U_j^+)^3 \big)\vert(t,\theta)d\theta\Big)d\eta\\
 &  \quad \le \frac{\A^2}{3}\int_0^{1}\frac1{\max_j(\sqP{j})}\abs{\sqP{1}-\sqP{2}}(t,\eta) \\ 
&\qquad \times\Big(\int_0^\eta \min_j\big( e^{-\frac{1}{\ma}(\tilde\Y_j(t,\eta)-\tilde\Y_j(t,\theta))}\big)\abs{\tilde \Y_{2}-\tilde \Y_{1}} \min_j(\tilde\V_j^+)(t,\theta) d\theta\Big)d\eta\\
& \quad \leq \frac{\A^2}{3} \int_0^1 \frac{1}{\max_j(\sqP{j})}\abs{\sqP{1}-\sqP{2}}(t,\eta)\\
& \qquad \times \Big(\int_0^\eta e^{-\frac{1}{\sqtC{2}}(\tilde\Y_2(t,\eta)-\tilde\Y_2(t,\theta))}\tilde\U_2^2(t,\theta)d\theta\Big)^{1/2}\\
& \qquad \times \Big(\int_0^\eta \min_j(e^{-\frac{1}{\ma}(\tilde\Y_j(t,\eta)-\tilde\Y_j(t,\theta))})\vert \tilde\Y_2-\tilde\Y_1\vert ^2 (t,\theta) d\theta\Big)^{1/2}d\eta\\
&\le \bigO(1) \big(\norm{\sqP{1}-\sqP{2}}^2+ \norm{\tilde \Y_{2}-\tilde \Y_{1}}^2 \big),
\end{align*}
using \eqref{eq:all_PestimatesA}.
We find
\begin{align*}
\abs{K_9+K_{10}}&\le \abs{L_1}+\abs{L_2}+\abs{L_3}\\
&\le \bigO(1)\big(\norm{\sqP{1}-\sqP{2}}^2+ \norm{\tilde \Y_{2}-\tilde \Y_{1}}^2 \big). 
\end{align*}

Thus we can sum up the estimates for $\tilde J_{11}$, and we find
\begin{align*}
\abs{\tilde J_{11}}&\le \bigO(1)\big(\norm{\sqP{1}-\sqP{2}}^2+\norm{\tilde \U_1 -\tilde \U_2}^2+ \norm{\tilde \Y_{2}-\tilde \Y_{1}}^2+\vert \sqtC{1}-\sqtC{2}\vert^2 \big). 
\end{align*}

We now proceed to the next term from \eqref{eq:tildeJ12}
\begin{align*}
\tilde J_{12}&=\frac1{4\A^3}\int_0^{1}\frac1{\max_j(\sqP{j})}(\sqP{1}-\sqP{2})(t,\eta)\\
&\quad\times\Big(\int_0^\eta \big(e^{-\frac{1}{\sqtC{1}}(\tilde\Y_1(t,\eta)-\tilde\Y_1(t,\theta))}\sqtC{1}^6\tilde\U_1-e^{-\frac{1}{\sqtC{2}}(\tilde\Y_2(t,\eta)-\tilde\Y_2(t,\theta))} \sqtC{2}^6\tilde\U_2\big)
(t,\theta)d\theta\Big)d\eta\\
&=\frac1{4\A^3}\int_0^{1}\frac1{\max_j(\sqP{j})}(\sqP{1}-\sqP{2})(t,\eta)\\
&\quad\times\Big((\sqtC{1}^6-\sqtC{2}^6)\mathbbm{1}_{\sqtC{1}<\sqtC{2}}\int_0^\eta e^{-\frac{1}{\sqtC{2}}(\tilde\Y_2(t,\eta)-\tilde\Y_2(t,\theta))} \tilde\U_2(t,\theta)d\theta\\
&\qquad +(\sqtC{1}^6-\sqtC{2}^2)\mathbbm{1}_{\sqtC{2}\le \sqtC{1}}\int_0^\eta e^{-\frac{1}{\sqtC{1}}(\tilde\Y_1(t,\eta)-\tilde\Y_1(t,\theta))} \tilde\U_1(t,\theta)d\theta\\
&\qquad +\ma^6\int_0^\eta\big( e^{-\frac{1}{\sqtC{1}}(\tilde\Y_1(t,\eta)-\tilde\Y_1(t,\theta))}\tilde\U_1-e^{-\frac{1}{\sqtC{2}}(\tilde\Y_2(t,\eta)-\tilde\Y_2(t,\theta))}\tilde\U_2\big)(t,\theta)d\theta\Big)d\eta\\
&=M_1+M_2+M_3.
\end{align*}
The term $M_1$  ($M_2$ is similar)  can be estimated as follows
\begin{align*}
\abs{M_1}&\le\frac1{4\A^3}\abs{\sqtC{1}^6-\sqtC{2}^6}\mathbbm{1}_{\sqtC{1}<\sqtC{2}}\int_0^{1}\frac1{\max_j(\sqP{j})}\abs{\sqP{1}-\sqP{2}}(t,\eta)\\
&\quad\times\Big(\int_0^\eta e^{-\frac{1}{\sqtC{2}}(\tilde\Y_2(t,\eta)-\tilde\Y_2(t,\theta))} \abs{\tilde\U_2}(t,\theta)d\theta\Big)d\eta\\
&\le \frac{3\A^2}{2}\abs{\sqtC{1}-\sqtC{2}}\int_0^{1}\frac1{\max_j(\sqP{j})}\abs{\sqP{1}-\sqP{2}}(t,\eta)\\
&\quad\times\Big(\int_0^\eta e^{-\frac{1}{\sqtC{2}}(\tilde\Y_2(t,\eta)-\tilde\Y_2(t,\theta))} \tilde\U_2^2(t,\theta)d\theta\Big)^{1/2}  \Big(\int_0^\eta e^{-\frac{1}{\sqtC{2}}(\tilde\Y_2(t,\eta)-\tilde\Y_2(t,\theta))} d\theta\Big)^{1/2}d\eta\\
&\le \frac{3\sqrt{3}\A^2}{\sqrt{2}}\abs{\sqtC{1}-\sqtC{2}}\int_0^{1}\abs{\sqP{1}-\sqP{2}}(t,\eta)d\eta\\
&\le \bigO(1) \big(\norm{\sqP{1}-\sqP{2}}^2+\abs{\sqtC{1}-\sqtC{2}}^2\big),
\end{align*}
using \eqref{eq:all_PestimatesA}. As for the next term $M_3$ we first  write $ \tilde\U_j= \tilde\U_j^+ + \tilde\U_j^-$, collect terms, and study the positive part and negative part separately. 
In the interest of the reader we do not change the notation. Thus
\begin{align*}
M_3&=\frac{\ma^6}{4\A^3}\int_0^{1}\frac1{\max_j(\sqP{j})}(\sqP{1}-\sqP{2})(t,\eta)\\
&\qquad \times\int_0^\eta\big( e^{-\frac{1}{\sqtC{1}}(\tilde\Y_1(t,\eta)-\tilde\Y_1(t,\theta))}\tilde\U_1^+ -e^{-\frac{1}{\sqtC{2}}(\tilde\Y_2(t,\eta)-\tilde\Y_2(t,\theta))}\tilde\U_2^+\big)(t,\theta)d\theta d\eta\\
&=\frac{\ma^6}{4\A^3}\int_0^{1}\frac1{\max_j(\sqP{j})}(\sqP{1}-\sqP{2})(t,\eta)\\
&\qquad \times\Big(\int_0^\eta e^{-\frac{1}{\sqtC{2}}(\tilde\Y_2(t,\eta)-\tilde\Y_2(t,\theta))} \big(\mathbbm{1}_{\tilde\U_1^+<\tilde\U_2^+} (\tilde\U_1^+ -\tilde\U_2^+) \big)(t,\theta)d\theta  \\
&\qquad+\int_0^\eta e^{-\frac{1}{\sqtC{1}}(\tilde\Y_1(t,\eta)-\tilde\Y_1(t,\theta))} \big(\mathbbm{1}_{\tilde\U_2^+\le\tilde\U_1^+} (\tilde\U_1^+ -\tilde\U_2^+) \big)(t,\theta)d\theta \\
&\qquad+\int_0^\eta \big(e^{-\frac{1}{\sqtC{1}}(\tilde\Y_1(t,\eta)-\tilde\Y_1(t,\theta))} -e^{-\frac{1}{\sqtC{2}}(\tilde\Y_2(t,\eta)-\tilde\Y_2(t,\theta))}\big)\min_j\big(\tilde\U_j^+\big)(t,\theta)d\theta  \Big)d\eta\\
&=M_{31}+M_{32}+M_{33}.  
\end{align*}
The terms $M_{31}$ and $M_{32}$ can be treated similarly. Thus
\begin{align*}
\abs{M_{31}}&=\frac{\ma^6}{4\A^3}\Big\vert\int_0^{1}\frac1{\max_j(\sqP{j})}(\sqP{1}-\sqP{2})(t,\eta)\\
&\qquad \times\Big(\int_0^\eta e^{-\frac{1}{\sqtC{2}}(\tilde\Y_2(t,\eta)-\tilde\Y_2(t,\theta))} \big(\mathbbm{1}_{\tilde\U_1^+<\tilde\U_2^+} (\tilde\U_1^+ -\tilde\U_2^+) \big)(t,\theta)d\theta \Big)d\eta\Big\vert \\
&\le \frac{1}{4\A^2}   \int_0^{1}\frac1{\max_j(\sqP{j})}  \abs{\sqP{1}-\sqP{2}}(t,\eta) \\
&\qquad\times\Big(\int_0^\eta e^{-\frac{1}{\sqtC{2}}(\tilde\Y_2(t,\eta)-\tilde\Y_2(t,\theta))}\sqtC{2}^5 \abs{\tilde\U_1^+ -\tilde\U_2^+} (t,\theta)d\theta\Big)d\eta \\
&\le \bigO(1) \int_0^{1}\frac1{\max_j(\sqP{j})}  \abs{\sqP{1}-\sqP{2}}\sqP{2}(t,\eta) \norm{\tilde\U_1 -\tilde\U_2}d\eta \\
&\le \bigO(1)\int_0^{1} \abs{\sqP{1}-\sqP{2}}(t,\eta) \norm{\tilde\U_1 -\tilde\U_2}d\eta \\
&\le \bigO(1) \big(\norm{\sqP{1}-\sqP{2}}^2+\norm{\tilde\U_1 -\tilde\U_2}^2\big),
\end{align*}
using
\begin{align*}
\ma^5&\int_0^\eta e^{-\frac{1}{\sqtC{2}}(\tilde\Y_2(t,\eta)-\tilde\Y_2(t,\theta))}\abs{\tilde\U_1^+ -\tilde\U_2^+} (t,\theta)d\theta\\
& \le \int_0^\eta e^{-\frac{1}{\sqtC{2}}(\tilde\Y_2(t,\eta)-\tilde\Y_2(t,\theta))}\sqtC{2}^5\abs{\tilde\U_1 -\tilde\U_2}(t,\theta) d\theta\\
&= \int_0^\eta e^{-\frac{1}{\sqtC{2}}(\tilde\Y_2(t,\eta)-\tilde\Y_2(t,\theta))}\big((2\tilde\P_2-\tilde\U_2^2) \tilde\Y_{2,\eta}+\tilde\Henergy_{2,\eta}\big)\abs{\tilde\U_1 -\tilde\U_2}(t,\theta) d\theta \\
&= 2 \int_0^\eta e^{-\frac{1}{\sqtC{2}}(\tilde\Y_2(t,\eta)-\tilde\Y_2(t,\theta))} \tilde\P_2 \tilde\Y_{2,\eta}\abs{\tilde\U_1 -\tilde\U_2}(t,\theta) d\theta  \\
&\qquad +\int_0^\eta e^{-\frac{1}{\sqtC{2}}(\tilde\Y_2(t,\eta)-\tilde\Y_2(t,\theta))}\tilde\Henergy_{2,\eta}\abs{\tilde\U_1 -\tilde\U_2}(t,\theta)d\theta\\
&\qquad -\int_0^\eta e^{-\frac{1}{\sqtC{2}}(\tilde\Y_2(t,\eta)-\tilde\Y_2(t,\theta))} \tilde\U_2^2 \tilde\Y_{2,\eta}\abs{\tilde\U_1 -\tilde\U_2}(t,\theta)d\theta\\
&\le  2\int_0^\eta e^{-\frac{1}{\sqtC{2}}(\tilde\Y_2(t,\eta)-\tilde\Y_2(t,\theta))} \tilde\P_2 \tilde\Y_{2,\eta}\abs{\tilde\U_1 -\tilde\U_2}(t,\theta) d\theta  \\
&\qquad +\int_0^\eta e^{-\frac{1}{\sqtC{2}}(\tilde\Y_2(t,\eta)-\tilde\Y_2(t,\theta))}\tilde\Henergy_{2,\eta}\abs{\tilde\U_1 -\tilde\U_2}(t,\theta)d\theta\\ 
&\le  2\Big(\int_0^\eta e^{-\frac3{2\sqtC{2}}(\tilde\Y_2(t,\eta)-\tilde\Y_2(t,\theta))} \tilde\P_2 \tilde\Y_{2,\eta}(t,\theta) d\theta\Big)^{1/2}  \\
&\qquad\times\Big(\int_0^\eta e^{-\frac1{2\sqtC{2}}(\tilde\Y_2(t,\eta)-\tilde\Y_2(t,\theta))}(\tilde\U_1 -\tilde\U_2)^2 \tilde\P_2\tilde\Y_{2,\eta}(t,\theta) d\theta\Big)^{1/2} \\
&\quad+\Big(\int_0^\eta e^{-\frac{1}{\sqtC{2}}(\tilde\Y_2(t,\eta)-\tilde\Y_2(t,\theta))} \tilde\Henergy_{2,\eta}(t,\theta) d\theta\Big)^{1/2} \\
&\qquad  \times\Big(\int_0^\eta e^{-\frac{1}{\sqtC{2}}(\tilde\Y_2(t,\eta)-\tilde\Y_2(t,\theta))}(\tilde\U_1 -\tilde\U_2)^2 \tilde\Henergy_{2,\eta}(t,\theta) d\theta\Big)^{1/2} \\
&\le \bigO(1)\A^3 \sqP{2}(t,\eta) \norm{\tilde\U_1 -\tilde\U_2},
\end{align*}
using \eqref{eq:PUYH_scale}, \eqref{eq:343}, \eqref{eq:all_PestimatesD},  \eqref{eq:all_estimatesE}, and \eqref{eq:all_estimatesK}. The term $M_{33}$ goes as follows 
\begin{align*}
M_{33}&=\frac{\ma^6}{4\A^3}\int_0^{1}\frac1{\max_j(\sqP{j})}(\sqP{1}-\sqP{2})(t,\eta)\\
&\qquad \times\Big(\int_0^\eta \big(e^{-\frac{1}{\sqtC{1}}(\tilde\Y_1(t,\eta)-\tilde\Y_1(t,\theta))} -e^{-\frac{1}{\sqtC{2}}(\tilde\Y_2(t,\eta)-\tilde\Y_2(t,\theta))}\big)\\
&\qquad\qquad\qquad\qquad\qquad\qquad\qquad\qquad\qquad\qquad\qquad\times\min_j\big(\tilde\U_j^+\big)(t,\theta)d\theta  \Big)d\eta\\
&= \frac{\ma^6}{4\A^3}\int_0^{1}\frac1{\max_j(\sqP{j})}(\sqP{1}-\sqP{2})(t,\eta)\\
& \qquad \times \mathbbm{1}_{\sqtC{1}\leq \sqtC{2}}\Big( \int_0^\eta (e^{-\frac{1}{\sqtC{1}}(\tilde \Y_2(t,\eta)-\tilde \Y_2(t,\theta))}-e^{-\frac{1}{\sqtC{2}}(\tilde \Y_2(t,\eta)-\tilde \Y_2(t,\theta))})\\
&\qquad\qquad\qquad\qquad\qquad\qquad\qquad\qquad\qquad\qquad\qquad\times\min_j(\tilde \V_j^+)(t,\theta)d\theta\Big)d\eta\\
& \quad + \frac{\ma^6}{4\A^3}\int_0^{1}\frac1{\max_j(\sqP{j})}(\sqP{1}-\sqP{2})(t,\eta)\\
& \qquad \times \mathbbm{1}_{\sqtC{2}< \sqtC{1}}\Big(\int_0^\eta (e^{-\frac{1}{\sqtC{1}}(\tilde \Y_1(t,\eta)-\tilde \Y_1(t,\theta))}-e^{-\frac{1}{\sqtC{2}}(\tilde \Y_1(t,\eta)-\tilde \Y_1(t,\theta))})\\
&\qquad\qquad\qquad\qquad\qquad\qquad\qquad\qquad\qquad\qquad\qquad\times\min_j(\tilde \V_j^+)(t,\theta)d\theta\Big)d\eta\\
&\quad +\frac{\ma^6}{4\A^3}\int_0^{1}\frac1{\max_j(\sqP{j})}(\sqP{1}-\sqP{2})(t,\eta)\\
&\qquad \times\Big(\int_0^\eta \big(e^{-\frac{1}{\ma}(\tilde\Y_1(t,\eta)-\tilde\Y_1(t,\theta))} -e^{-\frac{1}{\ma}(\tilde\Y_2(t,\eta)-\tilde\Y_2(t,\theta))}\big)\mathbbm{1}_{B(\eta)}\\
&\qquad\qquad\qquad\qquad\qquad\qquad\qquad\qquad\qquad\qquad\qquad\times\min_j\big(\tilde\U_j^+\big)(t,\theta)d\theta  \Big)d\eta\\
&\quad+ \frac{\ma^6}{4\A^3}\int_0^{1}\frac1{\max_j(\sqP{j})}(\sqP{1}-\sqP{2})(t,\eta)\\
&\qquad \times\Big(\int_0^\eta \big(e^{-\frac{1}{\ma}(\tilde\Y_1(t,\eta)-\tilde\Y_1(t,\theta))} -e^{-\frac{1}{\ma}(\tilde\Y_2(t,\eta)-\tilde\Y_2(t,\theta))}\big)\mathbbm{1}_{B(\eta)^c}\\
&\qquad\qquad\qquad\qquad\qquad\qquad\qquad\qquad\qquad\qquad\qquad\times\min_j\big(\tilde\U_j^+\big)(t,\theta)d\theta  \Big)d\eta\\
&=M_{331}+M_{332}+M_{333}+M_{334},
\end{align*}
where $B(\eta)$ is defined by \eqref{Def:Bn}. The terms $M_{331}$ and $M_{332}$  can be treated in the same manner. 
More specifically
\begin{align*}
\abs{M_{331}}& \leq \frac{\A^2}{e}\int_0^{1}\frac1{\max_j(\sqP{j})}\vert\sqP{1}-\sqP{2}\vert (t,\eta)\\
&\qquad\qquad\qquad\qquad\times\Big( \int_0^\eta e^{-\frac{3}{4\sqtC{2}}(\tilde \Y_2(t,\eta)-\tilde \Y_2(t,\theta))}\tilde \U_2^+(t,\theta)d\theta\Big)d\eta\vert \sqtC{1}-\sqtC{2}\vert\\
& \leq \frac{\A^2}{e}\int_0^1 \frac1{\max_j(\sqP{j})}\vert\sqP{1}-\sqP{2}\vert (t,\eta)\\
&\qquad\qquad\qquad\qquad\times\Big(\int_0^\eta e^{-\frac{1}{\sqtC{2}}(\tilde \Y_2(t,\eta)-\tilde \Y_2(t,\theta))}\tilde \U_2^2(t,\theta) d\theta\Big)^{1/2}d\eta\vert \sqtC{1}-\sqtC{2}\vert\\
& \leq \bigO(1)(\norm{\sqP{1}-\sqP{2}}^2+\vert \sqtC{1}-\sqtC{2}\vert^2).
\end{align*}

The terms $M_{333}$ and $M_{334}$ can be treated in a similar manner.
More specifically
\begin{align*}
\abs{M_{333}}&\le\frac{\ma^6}{4\A^3}\int_0^{1}\frac1{\max_j(\sqP{j})}\abs{\sqP{1}-\sqP{2}}(t,\eta)\\
&\qquad \times\Big(\int_0^\eta \abs{e^{-\frac{1}{\ma}(\tilde\Y_1(t,\eta)-\tilde\Y_1(t,\theta))} -e^{-\frac{1}{\ma}(\tilde\Y_2(t,\eta)-\tilde\Y_2(t,\theta))}}\mathbbm{1}_{B(\eta)}\\
&\qquad\qquad\qquad\qquad\qquad\qquad\qquad\qquad\qquad\qquad\times\min_j\big(\tilde\U_j^+\big)(t,\theta)d\theta  \Big)d\eta\\
&\le\frac{\A^2}{4}\int_0^{1}\frac1{\max_j(\sqP{j})}\abs{\sqP{1}-\sqP{2}}(t,\eta)\Big(\int_0^\eta e^{-\frac{1}{\sqtC{2}}(\tilde\Y_2(t,\eta)-\tilde\Y_2(t,\theta))}\\
&\quad\times\big(\abs{\tilde\Y_2(t,\eta)-\tilde\Y_1(t,\eta)}+\abs{\tilde\Y_2(t,\theta)-\tilde\Y_1(t,\theta)} \big) \min_j\big(\tilde\U_j^+\big)(t,\theta)d\theta  \Big)d\eta\\
&\le\frac{\A^2}{4}\int_0^{1}\frac1{\max_j(\sqP{j})}\abs{\sqP{1}-\sqP{2}}\,\abs{\tilde\Y_2-\tilde\Y_1}(t,\eta)\\
&\qquad \times\Big(\int_0^\eta e^{-\frac{1}{\sqtC{2}}(\tilde\Y_2(t,\eta)-\tilde\Y_2(t,\theta))}\mathbbm{1}_{B(\eta)}\min_j\big(\tilde\U_j^+\big)(t,\theta)d\theta  \Big)d\eta\\
&\quad+\frac{\A^2}{4}\int_0^{1}\frac1{\max_j(\sqP{j})}\abs{\sqP{1}-\sqP{2}}(t,\eta)\\
&\qquad \times\Big(\int_0^\eta e^{-\frac{1}{\sqtC{2}}(\tilde\Y_2(t,\eta)-\tilde\Y_2(t,\theta))}\abs{\tilde\Y_2-\tilde\Y_1}\min_j\big(\tilde\U_j^+\big)(t,\theta)d\theta  \Big)d\eta\\
&\le\frac{\A^2}{4}\int_0^{1}\frac1{\max_j(\sqP{j})}\abs{\sqP{1}-\sqP{2}}\,\abs{\tilde\Y_2-\tilde\Y_1}(t,\eta)\\
&\qquad \times\Big(\int_0^\eta e^{-\frac{1}{\sqtC{2}}(\tilde\Y_2(t,\eta)-\tilde\Y_2(t,\theta))}\tilde\U_2^2(t,\theta)d\theta  \Big)^{1/2} \\
&\qquad \times\Big(\int_0^\eta e^{-\frac{1}{\sqtC{2}}(\tilde\Y_2(t,\eta)-\tilde\Y_2(t,\theta))}d\theta\Big)^{1/2}d\eta\\
&\quad+\frac{A^2}{4}\int_0^{1}\frac1{\max_j(\sqP{j})}\abs{\sqP{1}-\sqP{2}}(t,\eta)\\
&\qquad \times\Big(\int_0^\eta e^{-\frac{1}{\sqtC{2}}(\tilde\Y_2(t,\eta)-\tilde\Y_2(t,\theta))}\tilde\U_2^2 (t,\theta)d\theta\Big)^{1/2}\\
&\qquad \times
\Big(\int_0^\eta e^{-\frac{1}{\sqtC{2}}(\tilde\Y_2(t,\eta)-\tilde\Y_2(t,\theta))}(\tilde\Y_2-\tilde\Y_1)^2(t,\theta)\Big)^{1/2}d\eta\\
&\le \bigO(1)\int_0^{1}\frac1{\max_j(\sqP{j})}\abs{\sqP{1}-\sqP{2}}\,\abs{\tilde\Y_2-\tilde\Y_1}\sqP{2}(t,\eta)d\eta\\
&\quad + \bigO(1)\int_0^{1}\frac1{\max_j(\sqP{j})}\abs{\sqP{1}-\sqP{2}}\,\norm{\tilde\Y_2-\tilde\Y_1}\sqP{2}(t,\eta)d\eta\\
&\le \bigO(1) \big(\norm{\sqP{1}-\sqP{2}}^2+ \norm{\tilde\Y_2-\tilde\Y_1}^2\big),
\end{align*}
using \eqref{Diff:Exp} and \eqref{eq:all_PestimatesA}.

\bigskip
And the final term from $J_3$, cf.~\eqref{eq:tildeJ13}, can be estimated as follows:
\begin{align*}
\tilde J_{13}&=\frac1{2\A^3}\int_0^{1}\frac1{\max_j(\sqP{j})}(\sqP{1}-\sqP{2})(t,\eta)\Big(\int_0^\eta\big(e^{-\frac{1}{\sqtC{2}}(\tilde\Y_2(t,\eta)-\tilde\Y_2(t,\theta))} \tilde \U_2\tilde\Q_2\tilde \Y_{2,\eta}\\
&\qquad\qquad\qquad\qquad
-e^{-\frac{1}{\sqtC{1}}(\tilde\Y_1(t,\eta)-\tilde\Y_1(t,\theta))} \tilde \U_1\tilde\Q_1\tilde \Y_{1,\eta}\big)(t,\theta)d\theta\Big)d\eta.
\end{align*}
We introduce the positive and negative part of $\tilde\U_j$, that is, $\tilde\U_j=\tilde\U_j^+ +\tilde\U_j^-$ (see \eqref{eq:pm0}) and introduce $\tilde\Q_j=\sqtC{j}\tilde\P_j-\tilde\D_j$. 
We study the term with $\tilde\P_j$ first
\begin{align*}
\tilde J_{131}&=\frac1{2\A^3}\int_0^{1}\frac1{\max_j(\sqP{j})}(\sqP{1}-\sqP{2})(t,\eta)\\
&\qquad\qquad\times\Big(\int_0^\eta\big(e^{-\frac{1}{\sqtC{2}}(\tilde\Y_2(t,\eta)-\tilde\Y_2(t,\theta))} \sqtC{2}\tilde \U_2^+\tilde\P_2\tilde \Y_{2,\eta}\\
&\qquad\qquad\qquad\qquad\qquad
-e^{-\frac{1}{\sqtC{1}}(\tilde\Y_1(t,\eta)-\tilde\Y_1(t,\theta))} \sqtC{1}\tilde \U_1^+\tilde\P_1\tilde \Y_{1,\eta}\big)(t,\theta)d\theta\Big)d\eta\\
&=\frac1{2\A^3}\int_0^{1}\frac1{\max_j(\sqP{j})}(\sqP{1}-\sqP{2})(t,\eta)\\
& \qquad \times\Big[\mathbbm{1}_{\sqtC{1}\leq \sqtC{2}}(\sqtC{2}-\sqtC{1})\int_0^\eta e^{-\frac{1}{\sqtC{2}}
(\tilde \Y_2(t,\eta)-\tilde \Y_2(t,\theta))}\tilde \V_2^+\tilde \P_2\tilde \Y_{2,\eta}(t,\theta)d\theta\\
& \qquad +\mathbbm{1}_{\sqtC{2}< \sqtC{1}}(\sqtC{2}-\sqtC{1})\int_0^\eta e^{-\frac{1}{\sqtC{1}}(\tilde \Y_1(t,\eta)-\tilde \Y_1(t,\theta))}\tilde \V_1^+\tilde \P_1\tilde \Y_{1,\eta}(t,\theta)d\theta\\
&\qquad+\ma\int_0^\eta e^{-\frac{1}{\sqtC{2}}(\tilde\Y_2(t,\eta)-\tilde\Y_2(t,\theta))}\big(\tilde\P_2-\tilde\P_1 \big)\mathbbm{1}_{\tilde\P_1<\tilde\P_2} \tilde \U_2^+\tilde \Y_{2,\eta}(t,\theta)d\theta \\
&\qquad+\ma\int_0^\eta e^{-\frac{1}{\sqtC{1}}(\tilde\Y_1(t,\eta)-\tilde\Y_1(t,\theta))}\big(\tilde\P_2-\tilde\P_1 \big)\mathbbm{1}_{\tilde\P_2\le\tilde\P_1} \tilde \U_1^+\tilde \Y_{1,\eta}(t,\theta)d\theta \\
&\qquad+\ma\int_0^\eta e^{-\frac{1}{\sqtC{2}}(\tilde\Y_2(t,\eta)-\tilde\Y_2(t,\theta))}\min_j\big(\tilde\P_j \big)\big(\tilde\U_2^+-\tilde\U_1^+ \big)\mathbbm{1}_{\tilde\U_1^+<\tilde\U_2^+} \tilde \Y_{2,\eta}(t,\theta)d\theta \\
&\qquad+\ma\int_0^\eta e^{-\frac{1}{\sqtC{1}}(\tilde\Y_1(t,\eta)-\tilde\Y_1(t,\theta))}\min_j\big(\tilde\P_j \big)\big(\tilde\U_2^+-\tilde\U_1^+ \big)\mathbbm{1}_{\tilde\U_2^+\le\tilde\U_1^+} \tilde \Y_{1,\eta}(t,\theta)d\theta \\
& \qquad +\ma\mathbbm{1}_{\sqtC{1}\leq \sqtC{2}}\int_0^\eta (e^{-\frac{1}{\sqtC{2}}(\tilde \Y_2(t,\eta)-\tilde \Y_2(t,\theta))}-e^{-\frac{1}{\sqtC{1}}(\tilde \Y_2(t,\eta)-\tilde \Y_2(t,\theta))})\\
&\qquad\qquad\qquad\qquad\qquad\times\min_j(\tilde \P_j)\min_j(\tilde \V_j^+)\tilde \Y_{2,\eta}(t,\theta) d\theta\\
& \qquad +\ma\mathbbm{1}_{\sqtC{2}< \sqtC{1}}\int_0^\eta (e^{-\frac{1}{\sqtC{2}}(\tilde \Y_1(t,\eta)-\tilde \Y_1(t,\theta))}-e^{-\frac{1}{\sqtC{1}}(\tilde \Y_1(t,\eta)-\tilde \Y_1(t,\theta))})\\
&\qquad\qquad\qquad\qquad\qquad\times\min_j(\tilde \P_j)\min_j(\tilde \V_j^+)\tilde \Y_{1,\eta}(t,\theta) d\theta\\
&\qquad+\ma\int_0^\eta \big(e^{-\frac{1}{\ma}(\tilde\Y_2(t,\eta)-\tilde\Y_2(t,\theta))}-e^{-\frac{1}{\ma}(\tilde\Y_1(t,\eta)-\tilde\Y_1(t,\theta))}\big) 
\mathbbm{1}_{B(\eta)}\\
&\qquad\qquad\qquad\qquad\qquad\times \min_j\big(\tilde\P_j \big)\min_j\big(\tilde\U_j^+\big)\tilde \Y_{2,\eta}(t,\theta)d\theta \\
&\qquad+\ma\int_0^\eta \big(e^{-\frac{1}{\ma}(\tilde\Y_2(t,\eta)-\tilde\Y_2(t,\theta))}-e^{-\frac{1}{\ma}(\tilde\Y_1(t,\eta)-\tilde\Y_1(t,\theta))}\big)
\mathbbm{1}_{B(\eta)^c} \\
&\qquad\qquad\qquad\qquad\qquad\times\min_j\big(\tilde\P_j \big)\min_j\big(\tilde\U_j^+\big)\tilde \Y_{1,\eta}(t,\theta)d\theta \\
&\qquad+\ma\int_0^\eta \min_j\big(e^{-\frac{1}{\ma}(\tilde\Y_j(t,\eta)-\tilde\Y_j(t,\theta))}\big) \\
&\qquad\qquad\qquad\qquad\qquad\times\min_j\big(\tilde\P_j \big)\min_j\big(\tilde\U_j^+\big)\big(\tilde \Y_{2,\eta}-\tilde \Y_{1,\eta}\big)(t,\theta)d\theta \Big]d\eta\\
&= K_1+K_2+K_3+K_4+K_5+K_6+K_7+K_8+K_9+K_{10}+K_{11}.
\end{align*}

The terms $K_1$ and $K_2$ can be treated similarly, thus
\begin{align*}
\abs{K_1}&=\frac1{2\A^3}\vert\int_0^{1}\frac1{\max_j(\sqP{j})}(\sqP{1}-\sqP{2})(t,\eta)\\
& \qquad \times\Big(\mathbbm{1}_{\sqtC{1}\leq \sqtC{2}}(\sqtC{2}-\sqtC{1})\int_0^\eta e^{-\frac{1}{\sqtC{2}}(\tilde \Y_2(t,\eta)-\tilde \Y_2(t,\theta))}\tilde \V_2^+\tilde \P_2\tilde \Y_{2,\eta}(t,\theta)d\theta\Big)d\eta\vert\\
& \leq\frac1{2\A^3}\int_0^1 \frac{1}{\max_j(\sqP{j})} \vert \sqP{1}-\sqP{2}\vert (t,\eta)\\
& \qquad \times \Big(\int_0^\eta e^{-\frac{1}{\sqtC{2}}(\tilde \Y_2(t,\eta)-\tilde \Y_2(t,\theta))}\tilde \P_2^2\tilde \Y_{2,\eta}(t,\theta) d\theta\Big)^{1/2}\\
& \qquad \times \Big(\int_0^\eta e^{-\frac{1}{\sqtC{2}}(\tilde \Y_2(t,\eta)-\tilde \Y_2(t,\theta))} \tilde \U_2^2\tilde \Y_{2,\eta}(t,\theta) d\theta\Big)^{1/2}d\eta\vert \sqtC{1}-\sqtC{2}\vert\\
& \leq \bigO(1)\big(\norm{\sqP{1}-\sqP{2}}^2+\vert \sqtC{1}-\sqtC{2}\vert^2).
\end{align*}

The terms $K_3$ and $K_4$ can be treated similarly, thus
\begin{align*}
\abs{K_3}&=\frac{\ma}{2\A^3}\Big\vert\int_0^{1}\frac1{\max_j(\sqP{j})}(\sqP{1}-\sqP{2})(t,\eta)\\
&\qquad\times\Big(\int_0^\eta e^{-\frac{1}{\sqtC{2}}(\tilde\Y_2(t,\eta)-\tilde\Y_2(t,\theta))}\big(\tilde\P_1-\tilde\P_2 \big)\mathbbm{1}_{\tilde\P_1<\tilde\P_2} \tilde \U_2^+\tilde \Y_{2,\eta}(t,\theta)d\theta \Big)d\eta\Big\vert\\
&\le\frac{1}{\A^2}\Big\vert\int_0^{1}\frac1{\max_j(\sqP{j})}\abs{\sqP{1}-\sqP{2}}(t,\eta)\\
&\qquad\times\Big(\int_0^\eta e^{-\frac{1}{\sqtC{2}}(\tilde\Y_2(t,\eta)-\tilde\Y_2(t,\theta))}\abs{\sqP{1}-\sqP{2}} \sqP{2} \tilde \U_2^+\tilde \Y_{2,\eta}(t,\theta)d\theta \Big)d\eta\Big\vert\\
&\le\frac{1}{\A^2}\Big\vert\int_0^{1}\frac1{\max_j(\sqP{j})}\abs{\sqP{1}-\sqP{2}}(t,\eta)\\
&\qquad\times\Big(\int_0^\eta e^{-\frac{1}{\sqtC{2}}(\tilde\Y_2(t,\eta)-\tilde\Y_2(t,\theta))}(\sqP{1}-\sqP{2})^2 (t,\theta)d\theta \Big)^{1/2} \\
&\qquad\times
\Big(\int_0^\eta  e^{-\frac{1}{\sqtC{2}}(\tilde\Y_2(t,\eta)-\tilde\Y_2(t,\theta))} \tilde\P_2 \tilde \U_2^2\tilde \Y_{2,\eta}^2(t,\theta)d\theta \Big)^{1/2} d\eta\\
&\le \bigO(1)\Big\vert\int_0^{1}\frac1{\max_j(\sqP{j})}\abs{\sqP{1}-\sqP{2}}\sqP{2}(t,\eta)\\
&\qquad\times\Big(\int_0^\eta e^{-\frac{1}{\sqtC{2}}(\tilde\Y_2(t,\eta)-\tilde\Y_2(t,\theta))}(\sqP{1}-\sqP{2})^2 (t,\theta)d\theta \Big)^{1/2} d\eta\\
&\le \bigO(1)\norm{\sqP{1}-\sqP{2}}^2, 
\end{align*}
using \eqref{eq:all_estimatesE} and \eqref{eq:all_PestimatesC}.  As for $K_5$ and $K_6$ we find
\begin{align*}
\abs{K_5}&=\frac{\ma}{2\A^3}\Big\vert\int_0^{1}\frac1{\max_j(\sqP{j})}(\sqP{1}-\sqP{2})(t,\eta)\\
&\qquad+\Big(\int_0^\eta e^{-\frac{1}{\sqtC{2}}(\tilde\Y_2(t,\eta)-\tilde\Y_2(t,\theta))}\min_j\big(\tilde\P_j \big)\big(\tilde\U_1^+-\tilde\U_2^+ \big)
\mathbbm{1}_{\tilde\U_1^+<\tilde\U_2^+} \tilde \Y_{2,\eta}(t,\theta)d\theta\Big)d\eta\Big\vert \\
&\le \frac1{2\A^2}\int_0^{1}\frac1{\max_j(\sqP{j})}\abs{\sqP{1}-\sqP{2}}(t,\eta)\\
&\qquad+\Big(\int_0^\eta e^{-\frac{1}{\sqtC{2}}(\tilde\Y_2(t,\eta)-\tilde\Y_2(t,\theta))}\min_j\big(\tilde\P_j \big)\abs{\tilde\U_1^+-\tilde\U_2^+}
\mathbbm{1}_{\tilde\U_1^+<\tilde\U_2^+} \tilde \Y_{2,\eta}(t,\theta)d\theta\Big)d\eta \\
&\le \norm{\sqP{1}-\sqP{2}}^2   +\frac1{4\A^4} \int_0^{1}\frac1{(\max_j(\sqP{j}))^2}(t,\eta)\\
&\qquad\times
\Big(\int_0^\eta e^{-\frac{1}{\sqtC{2}}(\tilde\Y_2(t,\eta)-\tilde\Y_2(t,\theta))}\min_j\big(\tilde\P_j \big)\abs{\tilde\U_1^+-\tilde\U_2^+}
\mathbbm{1}_{\tilde\U_1^+<\tilde\U_2^+} \tilde \Y_{2,\eta}(t,\theta)d\theta\Big)^2d\eta \\
&\le \norm{\sqP{1}-\sqP{2}}^2  \\
&\quad +\frac1{4\A^4} \int_0^{1}\frac1{(\max_j(\sqP{j}))^2}(t,\eta)
\Big(\int_0^\eta e^{-\frac3{2\sqtC{2}}(\tilde\Y_2(t,\eta)-\tilde\Y_2(t,\theta))}\tilde\P_2\tilde \Y_{2,\eta}(t,\theta)d\theta\Big) \\
&\qquad\times\Big(\int_0^\eta e^{-\frac1{2\sqtC{2}}(\tilde\Y_2(t,\eta)-\tilde\Y_2(t,\theta))}\tilde\P_2\tilde \Y_{2,\eta}(\tilde\U_1^+-\tilde\U_2^+)^2 (t,\theta)d\theta\Big)d\eta \\
&\le \norm{\sqP{1}-\sqP{2}}^2  +\frac1{2\A^3}\int_0^{1}\frac1{(\max_j(\sqP{j}))^2}\tilde\P_2(t,\eta)\\
&\qquad\qquad\times
\Big(\int_0^\eta e^{-\frac1{2\sqtC{2}}(\tilde\Y_2(t,\eta)-\tilde\Y_2(t,\theta))}\tilde\P_2\tilde \Y_{2,\eta}(\tilde\U_1^+-\tilde\U_2^+)^2 (t,\theta)d\theta\Big)d\eta \\
&\le \bigO(1)\big(\norm{\sqP{1}-\sqP{2}}^2+\norm{\tilde\U_1-\tilde\U_2}^2 \big),
\end{align*}
using \eqref{eq:343} and \eqref{eq:all_estimatesE}.  

The terms $K_7$ and $K_8$ allow for the following estimates
\begin{align*}
\abs{K_7}
&= \frac{\ma}{2\A^3}\Big\vert\int_0^{1}\frac1{\max_j(\sqP{j})}(\sqP{1}-\sqP{2})(t,\eta)\\
& \qquad \times\mathbbm{1}_{\sqtC{1}\leq \sqtC{2}}\int_0^\eta (e^{-\frac{1}{\sqtC{2}}(\tilde \Y_2(t,\eta)-\tilde \Y_2(t,\theta))}-e^{-\frac{1}{\sqtC{1}}(\tilde \Y_2(t,\eta)-\tilde \Y_2(t,\theta))}\\
&\qquad\qquad\times\min_j(\tilde \P_j)\min_j(\tilde \V_j^+)\tilde \Y_{2,\eta}(t,\theta) d\theta d\eta\vert\\
& \leq \frac{2}{\A^3 e}\int_0^1 \frac1{\max_j(\sqP{j})}\vert \sqP{1}-\sqP{2}\vert (t,\eta)\\
&\qquad\qquad\times\int_0^\eta e^{-\frac{3}{4\sqtC{2}}(\tilde \Y_2(t,\eta)-\tilde \Y_2(t,\theta))}\tilde \P_2\vert \tilde \U_2\vert \tilde \Y_{2,\eta}(t,\theta) d\theta d\eta \vert \sqtC{1}-\sqtC{2}\vert\\
& \leq \frac{2}{\A^3 e}\int_0^1 \frac1{\max_j(\sqP{j})}\vert \sqP{1}-\sqP{2}\vert (t,\eta)\\
& \qquad \times \Big(\int_0^1 e^{-\frac{1}{\sqtC{2}}(\tilde \Y_2(t,\eta)-\tilde \Y_2(t,\theta))}\tilde \U_2^2\tilde \Y_{2,\eta}(t,\theta) d\theta\Big)^{1/2}\\
& \qquad \times \Big(\int_0^1 e^{-\frac{1}{2\sqtC{2}}(\tilde \Y_2(t,\eta)-\tilde \Y_2(t,\theta))}\tilde \P_2^2\tilde \Y_{2,\eta}(t,\theta) d\theta\Big)^{1/2}d\eta \vert \sqtC{1}-\sqtC{2}\vert \\
& \leq \bigO(1)\big(\norm{\sqP{1}-\sqP{2}}^2+\vert \sqtC{1}-\sqtC{2}\vert^2\big).
\end{align*}

The terms $K_9$ and $K_{10}$ allow for the following estimates
\begin{align*}
\abs{K_9}&=\frac{\ma}{2\A^3}\Big\vert\int_0^{1}\frac1{\max_j(\sqP{j})}(\sqP{1}-\sqP{2})(t,\eta)\\
&\qquad+\int_0^\eta \big(e^{-\frac{1}{\ma}(\tilde\Y_1(t,\eta)-\tilde\Y_1(t,\theta))}-e^{-\frac{1}{\ma}(\tilde\Y_2(t,\eta)-\tilde\Y_2(t,\theta))}\big) 
\mathbbm{1}_{B(\eta)} \\
&\qquad\qquad\qquad\qquad\times\min_j\big(\tilde\P_j \big)\min_j\big(\tilde\U_j^+\big)\tilde \Y_{2,\eta}(t,\theta)d\theta d\eta\Big\vert\\
&\le\frac1{2\A^3}\Big\vert\int_0^{1}\frac1{\max_j(\sqP{j})}\abs{\sqP{1}-\sqP{2}}(t,\eta)\\
&\qquad\times\int_0^\eta e^{-\frac{1}{\ma}(\tilde\Y_2(t,\eta)-\tilde\Y_2(t,\theta))}\big(\abs{\tilde\Y_2(t,\eta)-\tilde\Y_1(t,\eta)}+\abs{\tilde\Y_2(t,\theta)-\tilde\Y_1(t,\theta)}\big) 
\\
&\qquad\qquad\qquad\qquad\times\min_j\big(\tilde\P_j \big)\min_j\big(\tilde\U_j^+\big)\tilde \Y_{2,\eta}(t,\theta)d\theta d\eta\Big\vert\\
&\le\frac1{2\A^3}\int_0^{1}\frac1{\max_j(\sqP{j})}\abs{\sqP{1}-\sqP{2}}\abs{\tilde\Y_2-\tilde\Y_1}(t,\eta)\\
&\qquad\times\Big(\int_0^\eta e^{-\frac{1}{\ma}(\tilde\Y_2(t,\eta)-\tilde\Y_2(t,\theta))}
\min_j\big(\tilde\P_j \big)\min_j\big(\tilde\U_j^+\big)\tilde \Y_{2,\eta}(t,\theta)d\theta\Big) d\eta \\
&\quad +\frac1{2\A^3}\int_0^{1}\frac1{\max_j(\sqP{j})}\abs{\sqP{1}-\sqP{2}}(t,\eta)\\
&\qquad\times\Big(\int_0^\eta e^{-\frac{1}{\ma}(\tilde\Y_2(t,\eta)-\tilde\Y_2(t,\theta))}\abs{\tilde\Y_2-\tilde\Y_1}
\min_j\big(\tilde\P_j \big)\min_j\big(\tilde\U_j^+\big)\tilde \Y_{2,\eta}(t,\theta)d\theta\Big) d\eta \\
&\le\frac1{2\A^3}\int_0^{1}\frac1{\max_j(\sqP{j})}\abs{\sqP{1}-\sqP{2}}\abs{\tilde\Y_2-\tilde\Y_1}(t,\eta)\\
&\qquad\times\Big(\int_0^\eta e^{-\frac{1}{\sqtC{2}}(\tilde\Y_2(t,\eta)-\tilde\Y_2(t,\theta))}\tilde\P_2^2\tilde \Y_{2,\eta}(t,\theta)d\theta\Big)^{1/2}\\
&\qquad\qquad\times
\Big(\int_0^\eta e^{-\frac{1}{\sqtC{2}}(\tilde\Y_2(t,\eta)-\tilde\Y_2(t,\theta))} \tilde\U_2^2\tilde \Y_{2,\eta}(t,\theta)d\theta\Big)^{1/2} d\eta \\
&\quad +\norm{\sqP{1}-\sqP{2}}^2 +\frac1{4\A^6}\int_0^{1}\frac1{(\max_j(\sqP{j}))^2}(t,\eta)\\
&\qquad\times\Big(\int_0^\eta e^{-\frac{1}{\sqtC{2}}(\tilde\Y_2(t,\eta)-\tilde\Y_2(t,\theta))}\abs{\tilde\Y_2-\tilde\Y_1}
\min_j\big(\tilde\P_j \big)\min_j\big(\tilde\U_j^+\big)\tilde \Y_{2,\eta}(t,\theta)d\theta\Big)^2 d\eta \\
&\le \bigO(1)\big(\norm{\sqP{1}-\sqP{2}}^2+\norm{\tilde\Y_2-\tilde\Y_1}^2 \big) +\frac1{4\A^6}\int_0^{1}\frac1{(\max_j(\sqP{j}))^2}(t,\eta)\\
&\qquad\qquad\times\Big(\int_0^\eta e^{-\frac{1}{\sqtC{2}}(\tilde\Y_2(t,\eta)-\tilde\Y_2(t,\theta))}\tilde\U_2^2\tilde \Y_{2,\eta}(t,\theta)d\theta\Big)\\
&\qquad\qquad\times\Big(\int_0^\eta e^{-\frac{1}{\sqtC{2}}(\tilde\Y_2(t,\eta)-\tilde\Y_2(t,\theta))}(\tilde\Y_2-\tilde\Y_1)^2
\tilde\P_2^2\tilde \Y_{2,\eta}(t,\theta)d\theta\Big) d\eta \\
&\le \bigO(1)\big(\norm{\sqP{1}-\sqP{2}}^2+\norm{\tilde\Y_2-\tilde\Y_1}^2 \big),
\end{align*}
using \eqref{eq:all_estimatesA}, \eqref{eq:all_estimatesE}, and \eqref{eq:all_PestimatesC}. The last term $K_{11}$ receives special treatment
\begin{align*}
\abs{K_{11}}&=\frac{\ma}{2\A^3}\Big\vert\int_0^{1}\frac1{\max_j(\sqP{j})}(\sqP{1}-\sqP{2})(t,\eta)\Big(\int_0^\eta \min_j\big(e^{-\frac{1}{\ma}(\tilde\Y_j(t,\eta)-\tilde\Y_j(t,\theta))}\big) \\
&\qquad\qquad\qquad\times\min_j\big(\tilde\P_j \big)\min_j\big(\tilde\U_j^+\big)\big(\tilde \Y_{1,\eta}-\tilde \Y_{2,\eta}\big)(t,\theta)d\theta \Big)d\eta\Big\vert\\
&= \frac{\ma}{2\A^3}\Big\vert\int_0^{1}\frac1{\max_j(\sqP{j})}(\sqP{1}-\sqP{2})(t,\eta)\\
&\qquad\times\Big[ \Big(\min_j\big(e^{-\frac{1}{\ma}(\tilde\Y_j(t,\eta)-\tilde\Y_j(t,\theta))}\big) \min_j\big(\tilde\P_j \big)\min_j\big(\tilde\U_j^+\big)\big(\tilde \Y_{1}-\tilde \Y_{2}\big)(t,\theta)\Big)\Big\vert_{\theta=0}^1\\
&\quad -\int_0^\eta \big(\tilde \Y_{1}-\tilde \Y_{2}\big) \\
&\qquad\quad\times
\frac{d}{d\theta}\Big(\min_j\big(e^{-\frac{1}{\ma}(\tilde\Y_j(t,\eta)-\tilde\Y_j(t,\theta))}\big) \min_j\big(\tilde\P_j \big)\min_j\big(\tilde\U_j^+\big)\Big) (t,\theta)d\theta \Big]d\eta\Big\vert\\
&\le \bigO(1)\big(\norm{\sqP{1}-\sqP{2}}^2+\norm{\tilde\Y_2-\tilde\Y_1}^2 \big) \\
&\qquad+ \frac{\ma}{2\A^3}\int_0^{1}\frac1{\max_j(\sqP{j})} \abs{\sqP{1}-\sqP{2}}(t,\eta)\\
&\qquad\times \int_0^\eta \abs{\tilde \Y_{1}-\tilde \Y_{2}}\Big\vert\big(\frac{d}{d\theta}\min_j\big(e^{-\frac{1}{\ma}(\tilde\Y_j(t,\eta)-\tilde\Y_j(t,\theta))}\big)\big)\min_j\big(\tilde\P_j \big)\min_j\big(\tilde\U_j^+
\big)\\
&\qquad\quad+\min_j\big(e^{-\frac{1}{\ma}(\tilde\Y_j(t,\eta)-\tilde\Y_j(t,\theta))}\big)\big(\frac{d}{d\theta}\min_j\big(\tilde\P_j \big)\min_j\big(\tilde\U_j^+
\big) \big)\Big\vert (t,\theta)d\theta \Big]d\eta\\
&\le \bigO(1)\big(\norm{\sqP{1}-\sqP{2}}^2+\norm{\tilde\Y_1-\tilde\Y_2}^2 \big) \\
&\qquad+ \frac1{2\A^3}\int_0^{1}\frac1{\max_j(\sqP{j})} \abs{\sqP{1}-\sqP{2}}(t,\eta)\\
&\qquad\times \int_0^\eta \abs{\tilde \Y_{1}-\tilde \Y_{2}}\Big(\min_j(e^{-\frac{1}{\ma}(\tilde \Y_j(t,\eta)-\tilde\Y_j(t,\theta))}) \max_j(\tilde\Y_{j,\eta})\min_j\big(\tilde\P_j \big)\min_j\big(\tilde\U_j^+
\big)\\
&\qquad\quad+2\ma\A^4\min_j\big(e^{-\frac{1}{\ma}(\tilde\Y_j(t,\eta)-\tilde\Y_j(t,\theta))}\big)(\min_j(\sqP{j})+\vert \tilde\U_2\vert)\Big) (t,\theta)d\theta \Big]d\eta\\
&\le \bigO(1)\big(\norm{\sqP{1}-\sqP{2}}^2+\norm{\tilde\Y_1-\tilde\Y_2}^2 \big) \\
&\qquad+ \frac1{2\A^3}\int_0^{1}\frac1{\max_j(\sqP{j})} \abs{\sqP{1}-\sqP{2}}(t,\eta)\\
&\qquad\quad\times\Big( \int_0^\eta e^{-\frac{1}{\sqtC{1}}(\tilde \Y_1(t,\eta)-\tilde\Y_1(t,\theta))} \abs{\tilde \Y_{1}-\tilde \Y_{2}}\tilde\Y_{1,\eta}\tilde\P_1 \tilde\U_1^+(t,\theta)d\theta \Big)d\eta\\
&\qquad+ \frac1{2\A^3}\int_0^{1}\frac1{\max_j(\sqP{j})} \abs{\sqP{1}-\sqP{2}}(t,\eta)\\
&\qquad\quad\times\Big( \int_0^\eta e^{-\frac{1}{\sqtC{2}}(\tilde \Y_2(t,\eta)-\tilde\Y_2(t,\theta))} \abs{\tilde \Y_{1}-\tilde \Y_{2}}\tilde\Y_{2,\eta}\tilde\P_2 \tilde\U_2^+(t,\theta)d\theta \Big)d\eta\\
&\qquad+ \A^2\int_0^{1}\frac1{\max_j(\sqP{j})} \abs{\sqP{1}-\sqP{2}}(t,\eta)\\
&\qquad\times\Big( \int_0^\eta\abs{\tilde \Y_{1}-\tilde \Y_{2}}\min_j\big(e^{-\frac{1}{\sqtC{2}}(\tilde\Y_2(t,\eta)-\tilde\Y_2(t,\theta))}\big)(\sqP{2}+\vert \tilde\U_2\vert) (t,\theta)d\theta \Big)d\eta\\
&=\bigO(1)\big(\norm{\sqP{1}-\sqP{2}}^2+\norm{\tilde\Y_1-\tilde\Y_2}^2 \big) + K_{71}+K_{72}+K_{73},
\end{align*}
using Lemmas \ref{lemma:1} and \ref{lemma:3}, and
\begin{align*}
&\min_j(e^{-\frac{1}{\ma}(\tilde \Y_j(t,\eta)-\tilde\Y_j(t,\theta))}) \max_j(\tilde\Y_{j,\eta})\min_j\big(\tilde\P_j \big)\min_j\big(\tilde\U_j^+
\big)(t,\theta)\\
&\qquad \le e^{-\frac{1}{\sqtC{1}}(\tilde \Y_1(t,\eta)-\tilde\Y_1(t,\theta))} \tilde\Y_{1,\eta}\tilde\P_1 \tilde\U_1^+(t,\theta)
+e^{-\frac{1}{\sqtC{2}}(\tilde \Y_2(t,\eta)-\tilde\Y_2(t,\theta))} \tilde\Y_{2,\eta}\tilde\P_2 \tilde\U_2^+(t,\theta) .
\end{align*}
Here $K_{71}$ and $K_{72}$ allow for the same treatment
\begin{align*}
2K_{71}&=\frac{1}{\A^3}\int_0^{1}\frac1{\max_j(\sqP{j})} \abs{\sqP{1}-\sqP{2}}(t,\eta)\\
&\qquad\times\Big( \int_0^\eta \abs{\tilde \Y_{1}-\tilde \Y_{2}} e^{-\frac{1}{\sqtC{1}}(\tilde \Y_1(t,\eta)-\tilde\Y_1(t,\theta))} \tilde\Y_{1,\eta}\tilde\P_1 \tilde\U_1^+ 
(t,\theta)d\theta \Big)d\eta\\
&\le \frac{1}{\A^3}\int_0^{1}\frac1{\max_j(\sqP{j})} \abs{\sqP{1}-\sqP{2}}(t,\eta)\\
&\qquad\times\Big( \int_0^\eta e^{-\frac{1}{\sqtC{1}}(\tilde \Y_1(t,\eta)-\tilde\Y_1(t,\theta))}  (\tilde \Y_{1}-\tilde \Y_{2})^2(t,\theta)d\theta \Big)^{1/2}\\
&\qquad\times
\Big(\int_0^\eta e^{-\frac{1}{\sqtC{1}}(\tilde \Y_1(t,\eta)-\tilde\Y_1(t,\theta))}  \tilde\Y_{1,\eta}^2\tilde\P_1^2 (\tilde\U_1^+)^2 (t,\theta)d\theta \Big)^{1/2}d\eta\\
&\le \A^2{\sqrt{2}}\int_0^{1}\frac1{\max_j(\sqP{j})} \abs{\sqP{1}-\sqP{2}}\sqP{1}(t,\eta)\norm{\tilde \Y_{1}-\tilde \Y_{2}}d\eta\\
&\le \bigO(1)\big(\norm{\sqP{1}-\sqP{2}}^2+ \norm{\tilde \Y_{1}-\tilde \Y_{2}}^2\big),
\end{align*}
using \eqref{eq:all_estimatesG} and \eqref{eq:all_PestimatesB}.   Furthermore,
\begin{align*}
K_{73}&= \A^2\int_0^{1}\frac1{\max_j(\sqP{j})} \abs{\sqP{1}-\sqP{2}}(t,\eta)\\
&\qquad\times\Big( \int_0^\eta\abs{\tilde \Y_{1}-\tilde \Y_{2}}\min_j\big(e^{-\frac{1}{\sqtC{2}}(\tilde\Y_2(t,\eta)-\tilde\Y_2(t,\theta))}\big)(\sqP{2}+\vert \tilde\U_2\vert) (t,\theta)d\theta \Big)d\eta\\
&\le \bigO(1)\int_0^{1}\frac1{\max_j(\sqP{j})} \abs{\sqP{1}-\sqP{2}}(t,\eta)\\
&\qquad\times\Big( \int_0^\eta\abs{\tilde \Y_{1}-\tilde \Y_{2}}\min_j\big(e^{-\frac{1}{\sqtC{2}}(\tilde\Y_2(t,\eta)-\tilde\Y_2(t,\theta))}\big)\sqP{2} (t,\theta)d\theta \Big)d\eta\\
&\le \bigO(1)\norm{\sqP{1}-\sqP{2}}^2 +\int_0^{1}\frac1{(\max_j(\sqP{j}))^2}(t,\eta)\\
&\qquad\times\Big( \int_0^\eta\abs{\tilde \Y_{1}-\tilde \Y_{2}}\min_j\big(e^{-\frac{1}{\sqtC{2}}(\tilde\Y_2(t,\eta)-\tilde\Y_2(t,\theta))}\big)\sqP{2} (t,\theta)d\theta \Big)^2d\eta\\
&\le \bigO(1)\norm{\sqP{1}-\sqP{2}}^2\\
&\quad +\int_0^{1}\frac1{(\max_j(\sqP{j}))^2}(t,\eta)
\Big( \int_0^\eta e^{-\frac1{2\sqtC{2}}(\tilde\Y_2(t,\eta)-\tilde\Y_2(t,\theta))}(\tilde \Y_{1}-\tilde \Y_{2})^2 (t,\theta)d\theta \Big)\\
&\qquad\times
\Big( \int_0^\eta e^{-\frac3{2\sqtC{2}}(\tilde\Y_2(t,\eta)-\tilde\Y_2(t,\theta))}\tilde\P_2 (t,\theta)d\theta \Big)d\eta\\
&\le \bigO(1)\norm{\sqP{1}-\sqP{2}}^2 +\bigO(1)\int_0^{1}\frac1{(\max_j(\sqP{j}))^2}\tilde\P_2 (t,\eta)\\
&\qquad\times
\Big( \int_0^\eta e^{-\frac1{2\sqtC{2}}(\tilde\Y_2(t,\eta)-\tilde\Y_2(t,\theta))}(\tilde \Y_{1}-\tilde \Y_{2})^2 (t,\theta)d\theta \Big)d\eta\\
&\le \bigO(1)\big(\norm{\sqP{1}-\sqP{2}}^2+\norm{\tilde \Y_{1}-\tilde \Y_{2}}^2\big),
\end{align*}
using  \eqref{eq:all_estimatesD} and \eqref{eq:32P}.
This proves that
\begin{align*}
\tilde J_{131}&\le \bigO(1)\big(\norm{\sqP{1}-\sqP{2}}^2+\norm{\tilde \U_{1}-\tilde \U_{2}}^2+\norm{\tilde \Y_{1}-\tilde \Y_{2}}^2+\vert \sqtC{1}-\sqtC{2}\vert^2\big).
\end{align*}

It remains to consider $\tilde J_{132}$:  
\begin{align*}
-\tilde J_{132}&=\frac1{2\A^3}\int_0^{1}\frac1{\max_j(\sqP{j})}(\sqP{1}-\sqP{2})(t,\eta)\\
&\qquad\qquad\times\Big(\int_0^\eta\big(e^{-\frac{1}{\sqtC{2}}(\tilde\Y_2(t,\eta)-\tilde\Y_2(t,\theta))} \tilde \U_2^+\tilde\D_2\tilde \Y_{2,\eta}\\
&\qquad\qquad\qquad\qquad
-e^{-\frac{1}{\sqtC{1}}(\tilde\Y_1(t,\eta)-\tilde\Y_1(t,\theta))} \tilde \U_1^+\tilde\D_1\tilde \Y_{1,\eta}\big)(t,\theta)d\theta\Big)d\eta\\
&= \frac1{2\A^3}\int_0^{1}\frac1{\max_j(\sqP{j})}(\sqP{1}-\sqP{2})(t,\eta)\\
&\qquad\times\Big[\int_0^\eta e^{-\frac{1}{\sqtC{2}}(\tilde\Y_2(t,\eta)-\tilde\Y_2(t,\theta))} \tilde \U_2^+\tilde \Y_{2,\eta} \big(\tilde\D_2-\tilde\D_1 \big)\mathbbm{1}_{\tilde\D_1<\tilde\D_2}(t,\theta)d\theta \\
&\qquad\quad +\int_0^\eta e^{-\frac{1}{\sqtC{1}}(\tilde\Y_1(t,\eta)-\tilde\Y_1(t,\theta))} \tilde \U_1^+\tilde \Y_{1,\eta} \big(\tilde\D_2-\tilde\D_1 \big)\mathbbm{1}_{\tilde\D_2\le\tilde\D_1}(t,\theta)d\theta \\
&\qquad\quad +\int_0^\eta \big(e^{-\frac{1}{\sqtC{2}}(\tilde\Y_2(t,\eta)-\tilde\Y_2(t,\theta))} \tilde \U_2^+\tilde \Y_{2,\eta} -e^{-(\tilde\Y_1(t,\eta)-\tilde\Y_1(t,\theta))} \tilde \U_1^+\tilde \Y_{1,\eta} \big)\\
&\qquad\qquad\qquad\qquad\qquad\qquad\qquad\qquad\qquad\times  
\min_j\big(\tilde\D_j\big)(t,\theta)d\theta \Big]d\eta \\
&= A_1+A_2+A_3.
\end{align*}
The terms $A_1$ and $A_2$ allow for the same treatment. Specifically (see Lemma \ref{lemma:D})
\begin{align*}
\abs{A_1}&= \frac1{2\A^3}\Big\vert\int_0^{1}\frac1{\max_j(\sqP{j})}(\sqP{1}-\sqP{2})(t,\eta)\\
&\qquad\times\Big[\int_0^\eta e^{-\frac{1}{\sqtC{2}}(\tilde\Y_2(t,\eta)-\tilde\Y_2(t,\theta))} \tilde \U_2^+\tilde \Y_{2,\eta} \big(\tilde\D_2-\tilde\D_1 \big)\mathbbm{1}_{\tilde\D_1<\tilde\D_2}(t,\theta)d\theta \Big]d\eta\Big\vert \\
&\le \frac1{2\A^3}\int_0^{1}\frac1{\max_j(\sqP{j})}\abs{\sqP{1}-\sqP{2}}(t,\eta)\\
&\qquad\times\Big[\int_0^\eta e^{-\frac{1}{\sqtC{2}}(\tilde\Y_2(t,\eta)-\tilde\Y_2(t,\theta))} \tilde \U_2^+\tilde \Y_{2,\eta} \big(\bar d_{11}+\bar d_{12} \big)\mathbbm{1}_{\tilde\D_1<\tilde\D_2}(t,\theta)d\theta \\
&\qquad\quad +\int_0^\eta e^{-\frac{1}{\sqtC{2}}(\tilde\Y_2(t,\eta)-\tilde\Y_2(t,\theta))} \tilde \U_2^+\tilde \Y_{2,\eta} T_{1}\mathbbm{1}_{\tilde\D_1<\tilde\D_2}(t,\theta)d\theta  \\
&\qquad\quad +\int_0^\eta e^{-\frac{1}{\sqtC{2}}(\tilde\Y_2(t,\eta)-\tilde\Y_2(t,\theta))} \tilde \U_2^+\tilde \Y_{2,\eta} T_{2}\mathbbm{1}_{\tilde\D_1<\tilde\D_2}(t,\theta)d\theta  \\
&\qquad\quad +\int_0^\eta e^{-\frac{1}{\sqtC{2}}(\tilde\Y_2(t,\eta)-\tilde\Y_2(t,\theta))} \tilde \U_2^+\tilde \Y_{2,\eta} T_{3}\mathbbm{1}_{\tilde\D_1<\tilde\D_2}(t,\theta)d\theta   \Big]d\eta \\
&= \frac12(A_{11}+A_{12}+A_{13}+A_{14}).
\end{align*}
Again we are forced to consider individual terms.
\begin{align*}
A_{11}&=\frac{1}{\A^3}\int_0^{1}\frac1{\max_j(\sqP{j})}\abs{\sqP{1}-\sqP{2}}(t,\eta)\\
&\qquad\times\Big[\int_0^\eta e^{-\frac{1}{\sqtC{2}}(\tilde\Y_2(t,\eta)-\tilde\Y_2(t,\theta))} \tilde \U_2^+\tilde \Y_{2,\eta} \big(\bar d_{11}+\bar d_{12} \big)\mathbbm{1}_{\tilde\D_1<\tilde\D_2}(t,\theta)d\theta\Big]d\eta \\
&\le \frac{1}{\A^3}\int_0^{1}\frac1{\max_j(\sqP{j})}\abs{\sqP{1}-\sqP{2}}(t,\eta)\\
&\qquad\times\Big[\int_0^\eta e^{-\frac{1}{\sqtC{2}}(\tilde\Y_2(t,\eta)-\tilde\Y_2(t,\theta))} \tilde \U_2^+\tilde \Y_{2,\eta} 
\big(2\A^{3/2}\tilde\D_2^{1/2} \abs{\tilde \Y_{2}-\tilde \Y_{1}} \\
&\qquad\qquad\qquad\qquad\qquad\qquad+2\sqrt{2}\A^{3/2} \tilde\D_2^{1/2} \norm{\tilde \Y_{2}-\tilde \Y_{1}}\big)(t,\theta)d\theta\Big]d\eta \\
&\le  \norm{\sqP{1}-\sqP{2}}^2+\frac{4}{\A^3}\int_0^{1}\frac1{(\max_j(\sqP{j}))^2}\\
&\qquad\qquad\times
\Big(\int_0^\eta e^{-\frac{1}{\sqtC{2}}(\tilde\Y_2(t,\eta)-\tilde\Y_2(t,\theta))} \tilde \U_2^+\tilde \Y_{2,\eta} \tilde\D_2^{1/2} \abs{\tilde \Y_{2}-\tilde \Y_{1}} (t,\theta)d\theta\Big)^2d\eta\\
&\quad+\frac{2\sqrt{2}}{\A^{3/2}} \norm{\tilde \Y_{2}-\tilde \Y_{1}} \norm{\sqP{1}-\sqP{2}}\Big(\int_0^{1}\frac1{(\max_j(\sqP{j}))^2}\\
&\qquad\times\Big(\int_0^\eta e^{-\frac{1}{\sqtC{2}}(\tilde\Y_2(t,\eta)-\tilde\Y_2(t,\theta))}
\tilde \U_2^+\tilde \Y_{2,\eta}\tilde\D_2^{1/2}(t,\theta)d\theta\Big)^{2}d\eta\Big)^{1/2} \\
&\le  \norm{\sqP{1}-\sqP{2}}^2\\
&\quad+\frac{4}{\A^3}\int_0^{1}\frac1{(\max_j(\sqP{j}))^2}
\Big(\int_0^\eta e^{-\frac{1}{\sqtC{2}}(\tilde\Y_2(t,\eta)-\tilde\Y_2(t,\theta))} \tilde \U_2^2\tilde \Y_{2,\eta}(t,\theta)d\theta\Big) \\
&\qquad\times\Big(\int_0^\eta e^{-\frac{1}{\sqtC{2}}(\tilde\Y_2(t,\eta)-\tilde\Y_2(t,\theta))}\tilde\D_2 \tilde \Y_{2,\eta}\abs{\tilde \Y_{2}-\tilde \Y_{1}} ^2(t,\theta)d\theta\Big)d\eta\\
&\quad+\frac{2\sqrt{2}}{\A^3}\norm{\tilde \Y_{2}-\tilde \Y_{1}} \norm{\sqP{1}-\sqP{2}}\Big[\int_0^{1}\frac1{(\max_j(\sqP{j}))^2}\\
&\qquad\times\Big(\int_0^\eta e^{-\frac{1}{\sqtC{2}}(\tilde\Y_2(t,\eta)-\tilde\Y_2(t,\theta))}\tilde \U_2^2\tilde \Y_{2,\eta}(t,\theta)d\theta\Big) \\
&\qquad\times\Big(\int_0^\eta e^{-\frac{1}{\sqtC{2}}(\tilde\Y_2(t,\eta)-\tilde\Y_2(t,\theta))}2\A\tilde\P_2\tilde \Y_{2,\eta}(t,\theta)d\theta\Big) d\eta\Big]^{1/2} \\
&\le \bigO(1)\big( \norm{\sqP{1}-\sqP{2}}^2+\norm{\tilde \Y_{2}-\tilde \Y_{1}}^2 \big),
\end{align*}
using \eqref{eq:all_estimatesE}, \eqref{eq:all_PestimatesC}, and \eqref{eq:all_estimatesN}. Next we find (see \eqref{eq:barT1})
\begin{align*}
A_{12}&=\frac{1}{\A^3}\int_0^{1}\frac1{\max_j(\sqP{j})}\abs{\sqP{1}-\sqP{2}}(t,\eta)\\
&\qquad\quad \times\Big(\int_0^\eta e^{-\frac{1}{\sqtC{2}}(\tilde\Y_2(t,\eta)-\tilde\Y_2(t,\theta))} \tilde \U_2^+\tilde \Y_{2,\eta} T_{1}\mathbbm{1}_{\tilde\D_1<\tilde\D_2}(t,\theta)d\theta \Big)d\eta \\
&\le\frac{1}{\A^3}\int_0^{1}\frac1{\max_j(\sqP{j})}\abs{\sqP{1}-\sqP{2}}(t,\eta)\\
&\qquad\quad \times\Big(\int_0^\eta e^{-\frac{1}{\sqtC{2}}(\tilde\Y_2(t,\eta)-\tilde\Y_2(t,\theta))} \abs{\tilde \U_2}^3\tilde \Y_{2,\eta} \abs{\tilde \Y_{2}-\tilde \Y_{1}}(t,\theta)d\theta \Big)d\eta \\
&\quad + 4 \int_0^{1} \frac1{\max_j(\sqP{j})}\abs{\sqP{1}-\sqP{2}}(t,\eta)  \Big( \int_0^\eta e^{-\frac{1}{\sqtC{2}}(\tilde\Y_2(t,\eta)-\tilde\Y_2(t,\theta))} \\
&\qquad\qquad\times\Big( \int_0^\theta e^{-\frac{1}{\A}(\tilde\Y_2(t,\theta)-\tilde\Y_2(t,l))} \big(\tilde \U_2-\tilde \U_1\big)^2 (t,l) dl\Big)^{1/2}  
\tilde \U_2^+\tilde \Y_{2,\eta}(t,\theta)d\theta  \Big)d\eta \\
&\quad + \sqrt{2}\A \int_0^{1} \frac1{\max_j(\sqP{j})}\abs{\sqP{1}-\sqP{2}}(t,\eta) \Big( \int_0^\eta e^{-\frac{1}{\sqtC{2}}(\tilde\Y_2(t,\eta)-\tilde\Y_2(t,\theta))} \\
&\qquad\qquad\times \Big( \int_0^\theta e^{-\frac{1}{\ma}(\tilde\Y_2(t,\theta)-\tilde\Y_2(t,l))} \big(\tilde \Y_2-\tilde \Y_1\big)^2 (t,l) dl\Big)^{1/2}  
\tilde \U_2^+\tilde \Y_{2,\eta}(t,\theta)d\theta  \Big)d\eta \\
&\quad +\A \int_0^{1}\frac1{\max_j(\sqP{j})}\abs{\sqP{1}-\sqP{2}}(t,\eta)\Big( \int_0^\eta e^{-\frac{1}{\sqtC{2}}(\tilde\Y_2(t,\eta)-\tilde\Y_2(t,\theta))} \\
&\qquad\qquad\times \Big( \int_0^\theta e^{-\frac{1}{\ma}(\tilde\Y_2(t,\theta)-\tilde\Y_2(t,l))} \abs{\tilde \Y_2-\tilde \Y_1} (t,l) dl\Big)
\tilde \U_2^+\tilde \Y_{2,\eta}(t,\theta)d\theta  \Big)d\eta \\
&\quad + \frac{8\sqrt{2}\A}{\sqrt{3}e} \int_0^{1} \frac1{\max_j(\sqP{j})}\abs{\sqP{1}-\sqP{2}}(t,\eta) \Big( \int_0^\eta e^{-\frac{1}{\sqtC{2}}(\tilde\Y_2(t,\eta)-\tilde\Y_2(t,\theta))}  \\
&\qquad\qquad\times\Big( \int_0^\theta e^{-\frac{3}{4\A}(\tilde\Y_2(t,\theta)-\tilde\Y_2(t,l))}  dl\Big)^{1/2}  \vert \sqtC{1}-\sqtC{2}\vert
\tilde \U_2^+\tilde \Y_{2,\eta}(t,\theta)d\theta  \Big)d\eta \\
&\le 5\norm{\sqP{1}-\sqP{2}}^2 + \frac{1}{A^6}\int_0^{1}\frac1{(\max_j(\sqP{j}))^2}(t,\eta)\\
&\qquad\times\Big(\int_0^\eta e^{-\frac{1}{\sqtC{2}}(\tilde\Y_2(t,\eta)-\tilde\Y_2(t,\theta))} \vert\tilde \U_2\vert ^3\tilde \Y_{2,\eta} \abs{\tilde \Y_{2}-\tilde \Y_{1}}(t,\theta)d\theta \Big)^2d\eta\\
&\quad + 16\int_0^{1}\frac1{(\max_j(\sqP{j}))^2}(t,\eta)\Big( \int_0^\eta e^{-\frac{1}{\sqtC{2}}(\tilde\Y_2(t,\eta)-\tilde\Y_2(t,\theta))} \\
&\qquad\times\Big( \int_0^\theta e^{-\frac{1}{\A}(\tilde\Y_2(t,\theta)-\tilde\Y_2(t,l))} \big(\tilde \U_2-\tilde \U_1\big)^2 (t,l) dl\Big)^{1/2}  
\tilde \U_2^+\tilde \Y_{2,\eta}(t,\theta)d\theta \Big)^2d\eta\\
&\quad + 2\A^2\int_0^{1}\frac1{(\max_j(\sqP{j}))^2}(t,\eta)\Big( \int_0^\eta e^{-\frac{1}{\sqtC{2}}(\tilde\Y_2(t,\eta)-\tilde\Y_2(t,\theta))}\\
&\qquad\times \Big( \int_0^\theta e^{-\frac{1}{\ma}(\tilde\Y_2(t,\theta)-\tilde\Y_2(t,l))} \big(\tilde \Y_2-\tilde \Y_1\big)^2 (t,l) dl\Big)^{1/2}  
\tilde \U_2^+\tilde \Y_{2,\eta}(t,\theta)d\theta \Big)^2d\eta\\
&\quad +\A^2 \int_0^{1}\frac1{(\max_j(\sqP{j}))^2}(t,\eta)\Big(\int_0^\eta e^{-\frac{1}{\sqtC{2}}(\tilde\Y_2(t,\eta)-\tilde\Y_2(t,\theta))}\\
&\qquad\times \Big( \int_0^\theta e^{-\frac{1}{\ma}(\tilde\Y_2(t,\theta)-\tilde\Y_2(t,l))} \abs{\tilde \Y_2-\tilde \Y_1} (t,l) dl\Big)\tilde \U_2^+\tilde \Y_{2,\eta}(t,\theta)d\theta  \Big)^2d\eta \\
&\quad + \frac{128\A^2}{3e^2} \int_0^{1} \frac1{(\max_j(\sqP{j}))^2}\Big( \int_0^\eta e^{-\frac{1}{\sqtC{2}}(\tilde\Y_2(t,\eta)-\tilde\Y_2(t,\theta))} \\
&\qquad\qquad\times\Big( \int_0^\theta e^{-\frac{3}{4\A}(\tilde\Y_2(t,\theta)-\tilde\Y_2(t,l))}  dl\Big)^{1/2} 
\tilde \U_2^+\tilde \Y_{2,\eta}(t,\theta)d\theta  \Big)^2d\eta\vert \sqtC{1}-\sqtC{2}\vert^2 \\
&\le 5\norm{\sqP{1}-\sqP{2}}^2\\
&\quad + \frac{1}{\A^6}\int_0^{1}\frac1{(\max_j(\sqP{j}))^2}(t,\eta)\Big(\int_0^\eta e^{-\frac{1}{\sqtC{2}}(\tilde\Y_2(t,\eta)-\tilde\Y_2(t,\theta))} \tilde \U_2^2\tilde \Y_{2,\eta}(t,\theta)d\theta \Big)\\
&\qquad\times\Big(\int_0^\eta e^{-\frac{1}{\sqtC{2}}(\tilde\Y_2(t,\eta)-\tilde\Y_2(t,\theta))} \tilde \U_2^4\tilde \Y_{2,\eta} (\tilde \Y_{2}-\tilde \Y_{1})^2(t,\theta)d\theta \Big)d\eta\\
&\quad + 16\int_0^{1}\frac1{(\max_j(\sqP{j}))^2}(t,\eta)\Big(\int_0^\eta e^{-\frac{1}{\sqtC{2}}(\tilde\Y_2(t,\eta)-\tilde\Y_2(t,\theta))}\tilde\U_2^2\tilde \Y_{2,\eta}d\theta \Big) \\
&\qquad\times\Big( \int_0^\eta e^{-\frac{1}{\sqtC{2}}(\tilde\Y_2(t,\eta)-\tilde\Y_2(t,\theta))}\\
&\qquad\qquad\times\Big( \int_0^\theta e^{-\frac{1}{\A}(\tilde\Y_2(t,\theta)-\tilde\Y_2(t,l))} \big(\tilde \U_2-\tilde \U_1\big)^2 (t,l) dl\Big) 
\tilde \Y_{2,\eta}(t,\theta)d\theta \Big)d\eta\\
&\quad + 2\A^2\int_0^{1}\frac1{(\max_j(\sqP{j}))^2}(t,\eta)\Big(\int_0^\eta e^{-\frac{1}{\sqtC{2}}(\tilde\Y_2(t,\eta)-\tilde\Y_2(t,\theta))}\tilde\U_2^2\tilde \Y_{2,\eta}d\theta \Big) \\
&\qquad\times\Big( \int_0^\eta e^{-\frac{1}{\sqtC{2}}(\tilde\Y_2(t,\eta)-\tilde\Y_2(t,\theta))}\\
&\qquad\qquad\times\Big( \int_0^\theta e^{-\frac{1}{\ma}(\tilde\Y_2(t,\theta)-\tilde\Y_2(t,l))} \big(\tilde \Y_2-\tilde \Y_1\big)^2 (t,l) dl\Big) 
\tilde \Y_{2,\eta}(t,\theta)d\theta \Big)d\eta\\
&\quad +\A^2 \int_0^{1}\frac1{(\max_j(\sqP{j}))^2}(t,\eta)\Big(\int_0^\eta e^{-\frac{1}{\sqtC{2}}(\tilde\Y_2(t,\eta)-\tilde\Y_2(t,\theta))} \tilde \U_2^2\tilde \Y_{2,\eta}(t,\theta)d\theta  \Big)\\
&\qquad\times\Big(\int_0^\eta e^{-\frac{1}{\sqtC{2}}(\tilde\Y_2(t,\eta)-\tilde\Y_2(t,\theta))} \\
&\qquad\times \Big( \int_0^\theta e^{-\frac{1}{\ma}(\tilde\Y_2(t,\theta)-\tilde\Y_2(t,l))} \abs{\tilde \Y_2-\tilde \Y_1} (t,l) dl\Big)^2
\tilde \Y_{2,\eta}(t,\theta)d\theta  \Big)d\eta \\
&\quad + \frac{128\A^2}{3e^2} \int_0^{1} \frac1{(\max_j(\sqP{j}))^2}(t,\eta)\Big(\int_0^\eta e^{-\frac{1}{\sqtC{2}}(\tilde \Y_2(t,\eta)-\tilde \Y_2(t,\theta))}\tilde \U_2^2\tilde \Y_{2,\eta}(t,\theta) d\theta\Big)\\
&\quad\times\Big( \int_0^\eta e^{-\frac{1}{\sqtC{2}}(\tilde\Y_2(t,\eta)-\tilde\Y_2(t,\theta))}\\
&\qquad\times \Big( \int_0^\theta e^{-\frac{3}{4\A}(\tilde\Y_2(t,\theta)-\tilde\Y_2(t,l))}  dl\Big) 
\tilde \Y_{2,\eta}(t,\theta)d\theta  \Big)d\eta\vert \sqtC{1}-\sqtC{2}\vert^2 \\
&\le \bigO(1) \big(\norm{\sqP{1}-\sqP{2}}^2+\norm{\tilde \Y_2-\tilde \Y_1}^2 \big)  +\bigO(1) \int_0^{1}\Big( \int_0^\eta e^{-\frac{1}{\sqtC{2}}(\tilde\Y_2(t,\eta)-\tilde\Y_2(t,\theta))}\\
&\qquad\qquad\times\Big( \int_0^\theta e^{-\frac{1}{\A}(\tilde\Y_2(t,\theta)-\tilde\Y_2(t,l))} \big(\tilde \U_2-\tilde \U_1\big)^2 (t,l) dl\Big) 
\tilde \Y_{2,\eta}(t,\theta)d\theta \Big)d\eta\\
&\quad +\bigO(1) \int_0^{1}\Big( \int_0^\eta e^{-\frac{1}{\sqtC{2}}(\tilde\Y_2(t,\eta)-\tilde\Y_2(t,\theta))}\\
&\qquad\qquad\times\Big( \int_0^\theta e^{-\frac{1}{\ma}(\tilde\Y_2(t,\theta)-\tilde\Y_2(t,l))} \big(\tilde \Y_2-\tilde \Y_1\big)^2 (t,l) dl\Big) 
\tilde \Y_{2,\eta}(t,\theta)d\theta \Big)d\eta\\
&\quad +\bigO(1)\int_0^1\Big(\int_0^\eta e^{-\frac{1}{\sqtC{2}}(\tilde \Y_2(t,\eta)-\tilde \Y_2(t,\theta))}\\
&\qquad\qquad\times\Big(\int_0^\theta e^{-\frac{3}{4\A}(\tilde \Y_2(t,\theta)-\tilde \Y_2(t,l))}dl\Big) \tilde \Y_{2,\eta}(t,\theta)d\theta\Big) d\eta\vert \sqtC{1}-\sqtC{2}\vert^2\\
&\le \bigO(1) \big(\norm{\sqP{1}-\sqP{2}}^2+\norm{\tilde \Y_2-\tilde \Y_1}^2 \big)+B_1+B_2+B_3,
\end{align*}
using \eqref{eq:all_estimatesB}, \eqref{eq:all_estimatesG}, and  \eqref{eq:all_PestimatesC}. As for $B_1$ we find
\begin{align*}
B_1&=\bigO(1) \int_0^{1}\Big( \int_0^\eta e^{-\frac{1}{\sqtC{2}}(\tilde\Y_2(t,\eta)-\tilde\Y_2(t,\theta))}\\
&\qquad\times\Big( \int_0^\theta e^{-\frac{1}{\A}(\tilde\Y_2(t,\theta)-\tilde\Y_2(t,l))} \big(\tilde \U_2-\tilde \U_1\big)^2 (t,l) dl\Big) 
\tilde \Y_{2,\eta}(t,\theta)d\theta \Big)d\eta\\
&\le \bigO(1) \int_0^{1}\Big( \int_0^\eta e^{-\frac1{2\A}(\tilde\Y_2(t,\eta)-\tilde\Y_2(t,\theta))}\\
&\qquad\times\Big( \int_0^\theta e^{-\frac{1}{\A}(\tilde\Y_2(t,\theta)-\tilde\Y_2(t,l))} \big(\tilde \U_2-\tilde \U_1\big)^2 (t,l) dl\Big) 
\tilde \Y_{2,\eta}(t,\theta)d\theta \Big)d\eta\\
&= \bigO(1) \int_0^{1}\Big( \int_0^\eta e^{-\frac1{2\A}\tilde\Y_2(t,\theta)}\tilde \Y_{2,\eta}(t,\theta)\\
&\qquad\times\Big( \int_0^\theta e^{-(\frac1{2\A}\tilde\Y_2(t,\eta)-\frac{1}{\A}\tilde\Y_2(t,l))} \big(\tilde \U_2-\tilde \U_1\big)^2 (t,l) dl\Big) 
d\theta \Big)d\eta\\
&= \bigO(1) \int_0^{1}\Big[ \Big( -2\A e^{-\frac1{2\A}\tilde\Y_2(t,\theta)}\\
&\qquad\qquad\times\Big( \int_0^\theta e^{-(\frac1{2\A}\tilde\Y_2(t,\eta)-\frac1{\A}\tilde\Y_2(t,l))} \big(\tilde \U_2-\tilde \U_1\big)^2 (t,l) dl\Big)\Big)\Big\vert_{\theta=0}^\eta \\
&\quad+2\A\int_0^\eta e^{-\frac1{2\A}\tilde\Y_2(t,\theta)} e^{-(\frac1{2\A}\tilde\Y_2(t,\eta)-\frac{1}{\A}\tilde\Y_2(t,\theta))}(\tilde \U_2-\tilde\U_1)^2(t,\theta)d\theta\Big]d\eta\\
&= -2\A\bigO(1) \int_0^{1}\int_0^\eta e^{-\frac{1}{\A}(\tilde\Y_2(t,\eta)-\tilde\Y_2(t,\theta))} \big(\tilde \U_2-\tilde \U_1\big)^2 (t,\theta) d\theta d\eta \\
&\quad+2\A\bigO(1)\int_0^{1}\int_0^\eta  e^{-\frac1{2\A}(\tilde\Y_2(t,\eta)-\tilde\Y_2(t,\theta))}\big(\tilde \U_2-\tilde \U_1\big)^2 (t,\theta)d\theta d\eta\\
&\le \bigO(1) \norm{\tilde \U_2-\tilde \U_1}^2,
\end{align*}
while for $B_2$ we estimate as follows
\begin{align*}
B_2&=\bigO(1)\int_0^{1}
\Big(\int_0^\eta e^{-\frac{1}{\sqtC{2}}(\tilde\Y_2(t,\eta)-\tilde\Y_2(t,\theta))}  \\
&\qquad\times\Big( \int_0^\theta e^{-\frac{1}{\ma}(\tilde\Y_2(t,\theta)-\tilde\Y_2(t,l))} (\tilde \Y_2-\tilde \Y_1)^2 (t,l) dl\Big)
\tilde \Y_{2,\eta}(t,\theta)d\theta  \Big)d\eta \\
&\le \bigO(1)\int_0^{1}
\Big(\int_0^\eta e^{-\frac1{2\sqtC{2}}(\tilde\Y_2(t,\eta)-\tilde\Y_2(t,\theta))} \\
&\qquad\times \Big( \int_0^\theta e^{-\frac{1}{\sqtC{2}}(\tilde\Y_2(t,\theta)-\tilde\Y_2(t,l))^2} (\tilde \Y_2-\tilde \Y_1)?2 (t,l) dl\Big)
\tilde \Y_{2,\eta}(t,\theta)d\theta  \Big)d\eta \\
&= \bigO(1) \int_0^{1}\Big( \int_0^\eta e^{-\frac1{2\sqtC{2}}\tilde\Y_2(t,\theta)}\tilde \Y_{2,\eta}(t,\theta)\\
&\qquad\times\Big( \int_0^\theta e^{-(\frac1{2\sqtC{2}}\tilde\Y_2(t,\eta)-\frac{1}{\sqtC{2}}\tilde\Y_2(t,l))} \big(\tilde \Y_2-\tilde \Y_1\big)^2 (t,l) dl\Big) 
d\theta \Big)d\eta\\
&= \bigO(1) \int_0^{1}\Big[ \Big( -2\sqtC{2} e^{-\frac1{2\sqtC{2}}\tilde\Y_2(t,\theta)}\\
&\qquad\times\Big( \int_0^\theta e^{-(\frac1{2\sqtC{2}}\tilde\Y_2(t,\eta)-\frac1{\sqtC{2}}\tilde\Y_2(t,l))} \big(\tilde \Y_2-\tilde \Y_1\big)^2 (t,l) dl\Big)\Big)\Big\vert_{\theta=0}^\eta \\
&\quad+2\sqtC{2}\int_0^\eta e^{-\frac1{2\sqtC{2}}\tilde\Y_2(t,\theta)} e^{-(\frac1{2\sqtC{2}}\tilde\Y_2(t,\eta)-\frac{1}{\sqtC{2}}\tilde\Y_2(t,\theta))}(\tilde\Y_2-\tilde\Y_1)2(t,\theta)d\theta\Big]d\eta\\
&= -2\sqtC{2}\bigO(1) \int_0^{1}\int_0^\eta e^{-\frac{1}{\sqtC{2}}(\tilde\Y_2(t,\eta)-\tilde\Y_2(t,\theta))} \big(\tilde \Y_2-\tilde \Y_1\big)^2 (t,\theta) d\theta d\eta \\
&\quad+2\sqtC{2}\bigO(1)\int_0^{1}\int_0^\eta  e^{-\frac1{2\sqtC{2}}(\tilde\Y_2(t,\eta)-\tilde\Y_2(t,\theta))}\big(\tilde \Y_2-\tilde \Y_1\big)^2 (t,\theta)d\theta d\eta\\
&\le \bigO(1) \norm{\tilde \Y_2-\tilde \Y_1}^2,
\end{align*}
while for $B_3$ we estimate as follows
\begin{align*}
B_3&=\bigO(1)\int_0^1\Big(\int_0^\eta e^{-\frac{1}{\sqtC{2}}(\tilde \Y_2(t,\eta)-\tilde \Y_2(t,\theta))}\\
&\qquad\times\Big(\int_0^\theta e^{-\frac{3}{4\A}(\tilde \Y_2(t,\theta)-\tilde \Y_2(t,l))}dl\Big) \tilde \Y_{2,\eta}(t,\theta)d\theta\Big) d\eta\vert \sqtC{1}-\sqtC{2}\vert^2\\
& \leq\bigO(1)\int_0^{1}
\Big(\int_0^\eta e^{-\frac3{8\A}(\tilde\Y_2(t,\eta)-\tilde\Y_2(t,\theta))} \\
&\qquad\times \Big( \int_0^\theta e^{-\frac{3}{4\A}(\tilde\Y_2(t,\theta)-\tilde\Y_2(t,l))} dl\Big)
\tilde \Y_{2,\eta}(t,\theta)d\theta  \Big)d\eta \vert \sqtC{1}-\sqtC{2}\vert^2\\
&= \bigO(1) \int_0^{1}\Big( \int_0^\eta e^{-\frac3{8\A}\tilde\Y_2(t,\theta)}\tilde \Y_{2,\eta}(t,\theta)\\
&\qquad\times\Big( \int_0^\theta e^{-(\frac3{8\A}\tilde\Y_2(t,\eta)-\frac{3}{4\A}\tilde\Y_2(t,l))}  dl\Big) 
d\theta \Big)d\eta\vert \sqtC{1}-\sqtC{2}\vert^2\\
&= \bigO(1) \int_0^{1}\Big[ \Big( -\frac83\A e^{-\frac3{8\A}\tilde\Y_2(t,\theta)}\\
&\qquad\times\Big( \int_0^\theta e^{-(\frac3{8\A}\tilde\Y_2(t,\eta)-\frac3{4\A}\tilde\Y_2(t,l))} dl\Big)\Big)\Big\vert_{\theta=0}^\eta \\
&\quad+\frac83\A\int_0^\eta e^{-\frac3{8\A}\tilde\Y_2(t,\theta)} e^{-(\frac3{8\A}\tilde\Y_2(t,\eta)-\frac{3}{4\A}\tilde\Y_2(t,\theta))}d\theta\Big]d\eta\vert \sqtC{1}-\sqtC{2}\vert^2\\
&= -\frac83\A\bigO(1) \int_0^{1}\int_0^\eta e^{-\frac{3}{4\A}(\tilde\Y_2(t,\eta)-\tilde\Y_2(t,\theta))}  d\theta d\eta\vert \sqtC{1}-\sqtC{2}\vert^2 \\
&\quad+\frac83\A\bigO(1)\int_0^{1}\int_0^\eta  e^{-\frac3{8\A}(\tilde\Y_2(t,\eta)-\tilde\Y_2(t,\theta))}d\theta d\eta\vert \sqtC{1}-\sqtC{2}\vert^2\\
&\le \bigO(1) \vert \sqtC{1}-\sqtC{2}\vert^2.
\end{align*}

Thus we find that
\begin{equation*}
A_{12}\le \bigO(1) \big(\norm{\sqP{1}-\sqP{2}}^2+\norm{\tilde \Y_2-\tilde \Y_1}^2+\norm{\tilde \U_2-\tilde \U_1}^2 +\vert \sqtC{1}-\sqtC{2}\vert^2\big).
\end{equation*}

Next we find (see \eqref{eq:barT2})
\begin{align*}
A_{13}&=\frac{1}{\A^3}\int_0^{1}\frac1{\max_j(\sqP{j})}\abs{\sqP{1}-\sqP{2}}(t,\eta)\\
&\qquad\quad \times\Big(\int_0^\eta e^{-\frac{1}{\sqtC{2}}(\tilde\Y_2(t,\eta)-\tilde\Y_2(t,\theta))} \tilde \U_2^+\tilde \Y_{2,\eta} T_{2}\mathbbm{1}_{\tilde\D_1<\tilde\D_2}(t,\theta)d\theta \Big)d\eta \\
&\le \frac{1}{\A^3}\int_0^{1}\frac1{\max_j(\sqP{j})}\abs{\sqP{1}-\sqP{2}}(t,\eta)\\
&\qquad\quad \times\Big(\int_0^\eta e^{-\frac{1}{\sqtC{2}}(\tilde\Y_2(t,\eta)-\tilde\Y_2(t,\theta))} \tilde \U_2^+\tilde \Y_{2,\eta}(t,\theta)\Big(\tilde\P_j \abs{\tilde\Y_2-\tilde\Y_1}(t,\theta)\\
&\qquad\quad+2\sqrt{2}\A^3\Big( \int_0^\theta e^{-\frac{1}{\A}(\tilde\Y_2(t,\theta)-\tilde\Y_2(t,l))} (\sqP{1}-\sqP{2})^2 (t,l)dl\Big)^{1/2} \\
& \qquad \quad +\frac{\A^4}{\sqrt{2}}\Big(\int_0^\theta e^{-\frac{1}{\ma}(\tilde\Y_2(t,\theta)-\tilde \Y_2(t,l))}(\tilde \Y_1-\tilde \Y_2)^2(t,l) dl \Big)^{1/2}\\
&\qquad\quad+\frac{\A^4}2\Big( \int_0^\theta e^{-\frac{1}{\ma}(\tilde\Y_2(t,\theta)-\tilde\Y_2(t,l))} \abs{\tilde\Y_2-\tilde\Y_1} (t,l)dl\Big)\\
& \qquad \quad+\frac{4\sqrt{2}\A^4}{\sqrt{3}e}\Big(\int_0^\eta e^{-\frac{3}{4\A}(\tilde \Y_2(t,\theta)-\tilde \Y_2(t,l))}dl\Big)^{1/2}\vert \sqtC{1}-\sqtC{2}\vert
 \Big)d\theta \Big)d\eta \\
&\le 5 \norm{\sqP{1}-\sqP{2}}^2  +\frac{1}{\A^6}\int_0^{1}\frac1{(\max_j(\sqP{j}))^2}(t,\eta)\\
&\qquad\qquad\times\Big(\int_0^\eta e^{-\frac{1}{\sqtC{2}}(\tilde\Y_2(t,\eta)-\tilde\Y_2(t,\theta))} \tilde \U_2^+\tilde \Y_{2,\eta}\tilde\P_j \abs{\tilde\Y_2-\tilde\Y_1}(t,\theta)d\theta \Big)^2d\eta\\
&\quad  +8\int_0^{1}\frac1{(\max_j(\sqP{j}))^2}(t,\eta)\Big(\int_0^\eta e^{-\frac{1}{\sqtC{2}}(\tilde\Y_2(t,\eta)-\tilde\Y_2(t,\theta))} \tilde \U_2^+\tilde \Y_{2,\eta}(t,\theta)\\
&\qquad\qquad\times\Big(\int_0^\theta e^{-\frac{1}{\A}(\tilde\Y_2(t,\theta)-\tilde\Y_2(t,l))} (\sqP{1}-\sqP{2})^2 (t,l)dl\Big)^{1/2}d\theta \Big)^2d\eta \\
&\quad  +\frac{\A^2}{2}\int_0^{1}\frac1{(\max_j(\sqP{j}))^2}(t,\eta)\Big(\int_0^\eta e^{-\frac{1}{\sqtC{2}}(\tilde\Y_2(t,\eta)-\tilde\Y_2(t,\theta))} \tilde \U_2^+\tilde \Y_{2,\eta}(t,\theta)\\
&\qquad\times\Big(\int_0^\theta e^{-\frac{1}{\ma}(\tilde\Y_2(t,\theta)-\tilde\Y_2(t,l))} (\tilde \Y_1-\tilde\Y_2)^2 (t,l)dl\Big)^{1/2}d\theta \Big)^2d\eta \\
&\quad  +\frac{\A^2}{4}\int_0^{1}\frac1{(\max_j(\sqP{j}))^2}(t,\eta)\Big(\int_0^\eta e^{-\frac{1}{\sqtC{2}}(\tilde\Y_2(t,\eta)-\tilde\Y_2(t,\theta))} \tilde \U_2^+\tilde \Y_{2,\eta}(t,\theta)\\
&\qquad\times\Big( \int_0^\theta e^{-\frac{1}{\ma}(\tilde\Y_2(t,\theta)-\tilde\Y_2(t,l))} \abs{\tilde\Y_2-\tilde\Y_1} (t,l)dl
 \Big)d\theta \Big)^2d\eta \\
 & \quad +\frac{32\A^2}{3e^2}\int_0^{1}\frac1{(\max_j(\sqP{j}))^2}(t,\eta)\Big(\int_0^\eta e^{-\frac{1}{\sqtC{2}}(\tilde\Y_2(t,\eta)-\tilde\Y_2(t,\theta))} \tilde \U_2^+\tilde \Y_{2,\eta}(t,\theta)\\
&\qquad\times\Big( \int_0^\theta e^{-\frac{3}{4\A}(\tilde\Y_2(t,\theta)-\tilde\Y_2(t,l))}dl
 \Big)^{1/2}d\theta \Big)^2d\eta\vert \sqtC{1}-\sqtC{2}\vert^2 \\
&\le 5 \norm{\sqP{1}-\sqP{2}}^2\\
&\quad  +\frac{1}{\A^6}\int_0^{1}\frac1{(\max_j(\sqP{j}))^2}(t,\eta)\Big(\int_0^\eta e^{-\frac{1}{\sqtC{2}}(\tilde\Y_2(t,\eta)-\tilde\Y_2(t,\theta))} \tilde\P_2^2\tilde \Y_{2,\eta}(t,\theta) d\theta \Big) \\
&\qquad\times \Big(\int_0^\eta e^{-\frac{1}{\sqtC{2}}(\tilde\Y_2(t,\eta)-\tilde\Y_2(t,\theta))} \tilde \U_2^2\tilde \Y_{2,\eta}(\tilde\Y_2-\tilde\Y_1)^2(t,\theta)d\theta \Big)d\eta\\ 
&\quad  +8\int_0^{1}\frac1{(\max_j(\sqP{j}))^2}(t,\eta)\Big(\int_0^\eta e^{-\frac{1}{\sqtC{2}}(\tilde\Y_2(t,\eta)-\tilde\Y_2(t,\theta))} \tilde\U_2^2\tilde \Y_{2,\eta}(t,\theta) d\theta \Big) \\
&\qquad\times \Big(\int_0^\eta e^{-\frac{1}{\sqtC{2}}(\tilde\Y_2(t,\eta)-\tilde\Y_2(t,\theta))} \tilde \Y_{2,\eta}(t,\theta)\\
&\qquad\times
\Big(\int_0^\theta e^{-\frac{1}{\A}(\tilde\Y_2(t,\theta)-\tilde\Y_2(t,l))} (\sqP{1}-\sqP{2})^2 (t,l)dl\Big) d\theta\Big)d\eta\\
&\quad  +\frac{\A^2}{2}\int_0^{1}\frac1{(\max_j(\sqP{j}))^2}(t,\eta)\Big(\int_0^\eta e^{-\frac{1}{\sqtC{2}}(\tilde\Y_2(t,\eta)-\tilde\Y_2(t,\theta))} \tilde\U_2^2\tilde \Y_{2,\eta}(t,\theta) d\theta \Big) \\
&\qquad\times \Big(\int_0^\eta e^{-\frac{1}{\sqtC{2}}(\tilde\Y_2(t,\eta)-\tilde\Y_2(t,\theta))} \tilde \Y_{2,\eta}(t,\theta)\\
&\qquad\times
\Big(\int_0^\theta e^{-\frac{1}{\ma}(\tilde\Y_2(t,\theta)-\tilde\Y_2(t,l))} (\tilde\Y_1-\tilde\Y_2)^2 (t,l)dl\Big) d\theta\Big)d\eta\\
&\quad  +\frac{\A^2}{4}\int_0^{1}\frac1{(\max_j(\sqP{j}))^2}(t,\eta)\Big(\int_0^\eta e^{-\frac{1}{\sqtC{2}}(\tilde\Y_2(t,\eta)-\tilde\Y_2(t,\theta))} \tilde\U_2^2\tilde \Y_{2,\eta}(t,\theta) d\theta \Big) \\
&\qquad\times \Big(\int_0^\eta e^{-\frac{1}{\sqtC{2}}(\tilde\Y_2(t,\eta)-\tilde\Y_2(t,\theta))} \tilde \Y_{2,\eta}(t,\theta)\\
&\qquad\times\Big(\int_0^\theta e^{-\frac{1}{\ma}(\tilde\Y_2(t,\theta)-\tilde\Y_2(t,l))}\abs{\tilde\Y_2-\tilde\Y_1} (t,l)dl\Big)^2 d\theta \Big)d\eta \\
& \quad +\frac{32\A^2}{3e^2}\int_0^{1}\frac1{(\max_j(\sqP{j}))^2}(t,\eta)\Big(\int_0^\eta e^{-\frac{1}{\sqtC{2}}(\tilde \Y_2(t,\eta)-\tilde \Y_2(t,\theta))}\tilde \U_2^2\tilde \Y_{2,\eta}(t,\theta)d\theta\Big)\\
&\qquad\times\Big(\int_0^\eta e^{-\frac{1}{\sqtC{2}}(\tilde\Y_2(t,\eta)-\tilde\Y_2(t,\theta))} \tilde \U_2^+\tilde \Y_{2,\eta}(t,\theta)\\
&\qquad\times\Big( \int_0^\theta e^{-\frac{3}{4\A}(\tilde\Y_2(t,\theta)-\tilde\Y_2(t,l))}dl
 \Big)d\theta \Big)d\eta\vert \sqtC{1}-\sqtC{2}\vert^2 \\
&\le \bigO(1)\big(\norm{\sqP{1}-\sqP{2}}^2+\norm{\tilde \Y_2-\tilde \Y_1}^2 \big) \\
& \quad+ \bigO(1)\int_0^{1}\Big(\int_0^\eta e^{-\frac{1}{\sqtC{2}}(\tilde\Y_2(t,\eta)-\tilde\Y_2(t,\theta))} \\
&\qquad\times
\Big(\int_0^\theta e^{-\frac{1}{\A}(\tilde\Y_2(t,\theta)-\tilde\Y_2(t,l))} (\sqP{1}-\sqP{2})^2 (t,l)dl\Big) \tilde \Y_{2,\eta}(t,\theta)d\theta\Big)d\eta\\
&\quad  +\bigO(1)\int_0^{1}\Big(\int_0^\eta e^{-\frac{1}{\sqtC{2}}(\tilde\Y_2(t,\eta)-\tilde\Y_2(t,\theta))}\\
&\qquad\times \Big(\int_0^\theta e^{-\frac{1}{\ma}(\tilde\Y_2(t,\theta)-\tilde\Y_2(t,l))}(\tilde\Y_2-\tilde\Y_1)^2 (t,l)dl\Big)\tilde \Y_{2,\eta}(t,\theta)
d\theta \Big)d\eta \\
& \quad +\bigO(1)\Big(\int_0^\eta e^{-\frac{1}{\sqtC{2}}(\tilde\Y_2(t,\eta)-\tilde\Y_2(t,\theta))}\\
&\qquad\times \Big(\int_0^\theta e^{-\frac{3}{4\A}(\tilde\Y_2(t,\theta)-\tilde\Y_2(t,l))}dl\Big)
 \tilde \Y_{2,\eta}(t,\theta)d\theta \Big)d\eta \vert \sqtC{1}-\sqtC{2}\vert^2\\
& \leq  \bigO(1)\big(\norm{\sqP{1}-\sqP{2}}^2+\norm{\tilde \Y_2-\tilde \Y_1}^2 \big) + B_4+B_2+B_3\\
 &  \leq \bigO(1)\big(\norm{\sqP{1}-\sqP{2}}^2+\norm{\tilde \Y_2-\tilde \Y_1}^2+\abs{\sqtC{1}-\sqtC{2}}^2 \big),
 \end{align*}
following the approach used for $A_{12}$. As for $B_4$ we find
\begin{align*}
B_4&=\bigO(1) \int_0^{1}\Big( \int_0^\eta e^{-\frac{1}{\sqtC{2}}(\tilde\Y_2(t,\eta)-\tilde\Y_2(t,\theta))} \\
&\qquad\times\Big( \int_0^\theta e^{-\frac{1}{\A}(\tilde\Y_2(t,\theta)-\tilde\Y_2(t,l))} \big(\sqP{1}-\sqP{2}\big)^2 (t,l) dl\Big) 
\tilde \Y_{2,\eta}(t,\theta)d\theta \Big)d\eta\\
&\le \bigO(1) \int_0^{1}\Big( \int_0^\eta e^{-\frac1{2\A}(\tilde\Y_2(t,\eta)-\tilde\Y_2(t,\theta))} \\
&\qquad\times\Big( \int_0^\theta e^{-\frac{1}{\A}(\tilde\Y_2(t,\theta)-\tilde\Y_2(t,l))} \big(\sqP{1}-\sqP{2}\big)^2 (t,l) dl\Big) 
\tilde \Y_{2,\eta}(t,\theta)d\theta \Big)d\eta\\
&= \bigO(1) \int_0^{1}\Big( \int_0^\eta e^{-\frac1{2\A}\tilde\Y_2(t,\theta)}\tilde \Y_{2,\eta}(t,\theta) \\
&\qquad\times\Big( \int_0^\theta e^{-(\frac1{2\A}\tilde\Y_2(t,\eta)-\frac{1}{\A}\tilde\Y_2(t,l))} \big(\sqP{1}-\sqP{2}\big)^2 (t,l) dl\Big) 
d\theta \Big)d\eta\\
&= \bigO(1) \int_0^{1}\Big[ \Big( -2\A e^{-\frac1{2\A}\tilde\Y_2(t,\theta)} \\
&\qquad\times\Big( \int_0^\theta e^{-(\frac1{2\A}\tilde\Y_2(t,\eta)-\frac1{\A}\tilde\Y_2(t,l))} \big(\sqP{1}-\sqP{2}\big)^2 (t,l) dl\Big)\Big)\Big\vert_{\theta=0}^\eta \\
&\quad+2\A\int_0^\eta e^{-\frac1{2\A}\tilde\Y_2(t,\theta)} e^{-(\frac1{2\A}\tilde\Y_2(t,\eta)-\frac{1}{\A}\tilde\Y_2(t,\theta))}(\sqP{1}-\sqP{2})^2(t,\theta)d\theta\Big]d\eta\\
&= -2\A\bigO(1) \int_0^{1}\int_0^\eta e^{-\frac{1}{\A}(\tilde\Y_2(t,\eta)-\tilde\Y_2(t,\theta))} \big(\sqP{1}-\sqP{2}\big)^2 (t,\theta) d\theta d\eta \\
&\quad+2\A\bigO(1)\int_0^{1}\int_0^\eta  e^{-\frac1{2\A}(\tilde\Y_2(t,\eta)-\tilde\Y_2(t,\theta))}\big(\sqP{1}-\sqP{2}\big)^2 (t,\theta)d\theta d\eta\\
&\le \bigO(1) \norm{\sqP{1}-\sqP{2}}^2.
\end{align*}
 
Thus we find that 
 \begin{equation*}
A_{12}\le \bigO(1) \big(\norm{\sqP{1}-\sqP{2}}^2+\norm{\tilde \Y_2-\tilde \Y_1}^2+\norm{\tilde \U_2-\tilde \U_1}^2 +\vert \sqtC{1}-\sqtC{2}\vert^2\big).
\end{equation*}

Next we find (see \eqref{eq:barT3})
\begin{align*}
A_{14}&=\frac{1}{\A^3}\int_0^{1}\frac1{\max_j(\sqP{j})}\abs{\sqP{1}-\sqP{2}}(t,\eta)\\
&\qquad\quad \times\Big(\int_0^\eta e^{-\frac{1}{\sqtC{2}}(\tilde\Y_2(t,\eta)-\tilde\Y_2(t,\theta))} \tilde \U_2^+\tilde \Y_{2,\eta} T_{3}\mathbbm{1}_{\tilde\D_1<\tilde\D_2}(t,\theta)d\theta \Big)d\eta \\
&\le 12\A\abs{\sqtC{1}-\sqtC{2}}\int_0^{1}\frac1{\max_j(\sqP{j})}\abs{\sqP{1}-\sqP{2}}(t,\eta)\\
&\quad \times\Big(\int_0^\eta e^{-\frac{1}{\sqtC{2}}(\tilde\Y_2(t,\eta)-\tilde\Y_2(t,\theta))} \tilde \U_2^+\tilde \Y_{2,\eta} (t,\theta)
\Big(\int_0^\theta e^{-\frac{3}{4\A}(\tilde\Y_2(t,\theta)-\tilde\Y_2(t,l))} dl \Big)d\theta \Big)d\eta \\
&\le 12\A\abs{\sqtC{1}-\sqtC{2}}\norm{\sqP{1}-\sqP{2}}\Big(\int_0^{1}\frac1{(\max_j(\sqP{j}))^2}(t,\eta)\\
&\quad \times\Big(\int_0^\eta e^{-\frac{1}{\sqtC{2}}(\tilde\Y_2(t,\eta)-\tilde\Y_2(t,\theta))} \tilde \U_2^+\tilde \Y_{2,\eta}(t,\theta) \\
&\qquad\times
\Big(\int_0^\theta  e^{-\frac{3}{4\A}(\tilde\Y_2(t,\theta)-\tilde\Y_2(t,l))} dl \Big)d\theta \Big)^2d\eta\Big)^{1/2} \\
&\le 12\A\abs{\sqtC{1}-\sqtC{2}}\norm{\sqP{1}-\sqP{2}}\\
&\quad\times\Big(\int_0^{1}\frac1{(\max_j(\sqP{j}))^2}(t,\eta)
\Big(\int_0^\eta e^{-\frac{1}{\sqtC{2}}(\tilde\Y_2(t,\eta)-\tilde\Y_2(t,\theta))} \tilde \U_2^2\tilde \Y_{2,\eta}(t,\theta)d\theta \Big)  \\
&\quad \times\Big(\int_0^\eta e^{-\frac{1}{\sqtC{2}}(\tilde\Y_2(t,\eta)-\tilde\Y_2(t,\theta))} \tilde \Y_{2,\eta} (t,\theta)
\Big(\int_0^\theta  e^{-\frac{3}{4\A}(\tilde\Y_2(t,\theta)-\tilde\Y_2(t,l))} dl \Big)^2d\theta \Big) d\eta\Big)^{1/2} \\
&\le  \bigO(1)\abs{\sqtC{1}-\sqtC{2}}\norm{\sqP{1}-\sqP{2}} \Big(\int_0^{1}  \Big(\int_0^\eta e^{-\frac{1}{\sqtC{2}}(\tilde\Y_2(t,\eta)-\tilde\Y_2(t,\theta))} \tilde \Y_{2,\eta} (t,\theta)\\
&\qquad \times
\Big(\int_0^\theta  e^{-\frac{3}{4\A}(\tilde\Y_2(t,\theta)-\tilde\Y_2(t,l))} dl \Big)^2d\theta \Big) d\eta\Big)^{1/2}   \\
&\le  \bigO(1)\abs{\sqtC{1}-\sqtC{2}}\norm{\sqP{1}-\sqP{2}} \\
&\qquad \times\Big(\int_0^{1}  \Big(\int_0^\eta e^{-\frac3{8\A}(\tilde\Y_2(t,\eta)-\tilde\Y_2(t,\theta))} \tilde \Y_{2,\eta} (t,\theta)\\
&\qquad \times
\Big(\int_0^\theta  e^{-\frac{3}{4\A}(\tilde\Y_2(t,\theta)-\tilde\Y_2(t,l))} dl \Big)d\theta \Big) d\eta\Big)^{1/2}   \\
&\le  \bigO(1)\abs{\sqtC{1}-\sqtC{2}}\norm{\sqP{1}-\sqP{2}} \\
&\qquad \times\Big(\int_0^{1}  \Big(\int_0^\eta e^{-\frac3{8\A}\tilde\Y_2(t,\theta)} \tilde \Y_{2,\eta} (t,\theta)\\
&\qquad \qquad\times
\Big(\int_0^\theta  e^{-(\frac3{8\A}\tilde\Y_2(t,\eta)-\frac{3}{4\A}\tilde\Y_2(t,l))} dl \Big)d\theta \Big) d\eta\Big)^{1/2}   \\
&\le  \bigO(1)\abs{\sqtC{1}-\sqtC{2}}\norm{\sqP{1}-\sqP{2}} \\
&\qquad\qquad \times\Big(\int_0^{1}  \Big[  \Big( -\frac83\A e^{-\frac3{8\A}\tilde\Y_2(t,\theta)}  
\int_0^\theta  e^{-(\frac3{8\A}\tilde\Y_2(t,\eta)-\frac{3}{4\A}\tilde\Y_2(t,l))} dl \Big)\Big\vert_{\theta=0}^\eta\\
&\qquad\qquad\qquad\quad +\frac83\A \int_0^\eta e^{-\frac3{8\A}\tilde\Y_2(t,\theta)} e^{-(\frac3{8\A}\tilde\Y_2(t,\eta)-\frac{3}{4\A}\tilde\Y_2(t,\theta))} d\theta  \Big] d\eta\Big)^{1/2}   \\
&\le  \bigO(1)\abs{\sqtC{1}-\sqtC{2}}\norm{\sqP{1}-\sqP{2}} \\
&\qquad\qquad\qquad \times\Big(\int_0^{1}  \Big[ -\frac83 \A \int_0^\eta 
 e^{-\frac{3}{4\A}(\tilde\Y_2(t,\eta)-\tilde\Y_2(t,\theta))} d\theta \\
&\qquad\qquad\qquad\quad +\frac83 \A \int_0^\eta e^{-\frac3{8\A}(\tilde\Y_2(t,\eta)-\tilde\Y_2(t,\theta))} d\theta  \Big] d\eta\Big)^{1/2}   \\
&\le \bigO(1)\big(\abs{\sqtC{1}-\sqtC{2}}^2+\norm{\sqP{1}-\sqP{2}}^2 \big).
\end{align*}

Thus we have shown that
\begin{equation*}
A_{1}+A_{2}\le \bigO(1) \big(\norm{\sqP{1}-\sqP{2}}^2+\norm{\tilde \Y_2-\tilde \Y_1}^2+\norm{\tilde \U_2-\tilde \U_1}^2+\abs{\sqtC{1}-\sqtC{2}}^2 \big).
\end{equation*}

We next consider the term $A_3$: 
\begin{align}
A_3&= \frac1{2\A^3}\int_0^{1}\frac1{\max_j(\sqP{j})}(\sqP{1}-\sqP{2})(t,\eta)\int_0^\eta \big(e^{-\frac{1}{\sqtC{2}}(\tilde\Y_2(t,\eta)-\tilde\Y_2(t,\theta))} \tilde \U_2^+\tilde \Y_{2,\eta} \notag\\
&\qquad \qquad\qquad\qquad-e^{-\frac{1}{\sqtC{1}}(\tilde\Y_1(t,\eta)-\tilde\Y_1(t,\theta))} \tilde \U_1^+\tilde \Y_{1,\eta} \big)  
\min_j\big(\tilde\D_j\big)(t,\theta)d\theta d\eta \notag\\
&=\frac1{2\A^3}\int_0^{1}\frac1{\max_j(\sqP{j})}(\sqP{1}-\sqP{2})(t,\eta)\notag\\
&\qquad \times\Big[\int_0^\eta e^{-\frac{1}{\sqtC{2}}(\tilde\Y_2(t,\eta)-\tilde\Y_2(t,\theta))}\min_j\big(\tilde\D_j\big)\tilde \Y_{2,\eta}
\big(\tilde \U_2^+-\tilde \U_1^+ \big)\mathbbm{1}_{\tilde \U_1^+ <\tilde \U_2^+ }(t,\theta)d\theta  \notag\\
&\qquad +\int_0^\eta e^{-\frac{1}{\sqtC{1}}(\tilde\Y_1(t,\eta)-\tilde\Y_1(t,\theta))}\min_j\big(\tilde\D_j\big)\tilde \Y_{1,\eta}
\big(\tilde \U_2^+-\tilde \U_1^+ \big)\mathbbm{1}_{\tilde \U_2^+ \le\tilde \U_1^+ }(t,\theta)d\theta  \notag\\
& \qquad +\mathbbm{1}_{\sqtC{1}\leq \sqtC{2}}\int_0^\eta (e^{-\frac{1}{\sqtC{2}}(\tilde \Y_2(t,\eta)-\tilde \Y_2(t,\theta))}-e^{-\frac{1}{\sqtC{1}}(\tilde \Y_2(t,\eta)-\tilde \Y_2(t,\theta))})\notag \\
&\qquad\qquad\qquad\qquad\qquad\times\min_j(\tilde \D_j)\min_j(\tilde \V_j^+)\tilde \Y_{2,\eta}(t,\theta) d\theta \notag\\
& \qquad +\mathbbm{1}_{\sqtC{2}<\sqtC{1}}\int_0^\eta (e^{-\frac{1}{\sqtC{2}}(\tilde \Y_1(t,\eta)-\tilde\Y_1(t,\theta))}-e^{-\frac{1}{\sqtC{1}}(\tilde\Y_1(t,\eta)-\tilde \Y_1(t,\theta))}\notag \\
&\qquad\qquad\qquad\qquad\qquad\times\min_j(\tilde \D_j)\min_j(\tilde \V_j^+)\tilde \Y_{1,\eta}(t,\theta) d\theta \notag \\
&\qquad +\int_0^\eta \big(e^{-\frac{1}{\ma}(\tilde\Y_2(t,\eta)-\tilde\Y_2(t,\theta))}-e^{-\frac{1}{\ma}(\tilde\Y_1(t,\eta)-\tilde\Y_1(t,\theta))} \big)\notag \\
&\qquad\qquad\qquad\qquad\qquad\times\min_j\big(\tilde\D_j\big)\min_j\big(\tilde\U^+_j\big)
\tilde \Y_{2,\eta}\mathbbm{1}_{B(\eta)}(t,\theta)d\theta  \notag\\
&\qquad +\int_0^\eta \big(e^{-\frac{1}{\ma}(\tilde\Y_2(t,\eta)-\tilde\Y_2(t,\theta))}-e^{-\frac{1}{\ma}(\tilde\Y_1(t,\eta)-\tilde\Y_1(t,\theta))} \big)\notag \\
&\qquad\qquad\qquad\qquad\qquad\times\min_j\big(\tilde\D_j\big)\min_j\big(\tilde\U^+_j\big)
\tilde \Y_{1,\eta}\mathbbm{1}_{B(\eta)^c}(t,\theta)d\theta  \notag\\
&\quad +\int_0^\eta \min_j\big(e^{-\frac{1}{\ma}(\tilde\Y_j(t,\eta)-\tilde\Y_j(t,\theta))}\big)\min_j\big(\tilde\D_j\big)\min_j\big(\tilde\U^+_j\big)\big(\tilde \Y_{2,\eta}-\tilde \Y_{1,\eta} \big)(t,\theta)d\theta
\Big]d\eta \notag\\
&= A_{31}+A_{32}+A_{33}+A_{34}+A_{35}+A_{36}+A_{37}. \label{eq:A3}
\end{align}
Terms $A_{31}$ and $A_{32}$ allow for the same treatment:
\begin{align*}
\abs{A_{31}}&=\frac1{2\A^3}\Big\vert\int_0^{1}\frac1{\max_j(\sqP{j})}(\sqP{1}-\sqP{2})(t,\eta)\\
&\quad \times\Big(\int_0^\eta e^{-\frac{1}{\sqtC{2}}(\tilde\Y_2(t,\eta)-\tilde\Y_2(t,\theta))}\min_j\big(\tilde\D_j\big)\tilde \Y_{2,\eta}
\big(\tilde \U_2^+-\tilde \U_1^+ \big)\mathbbm{1}_{\tilde \U_1^+ <\tilde \U_2^+ }(t,\theta)d\theta  \Big)d\eta\Big\vert\\
&\le \frac{1}{\A^2}\Big\vert\int_0^{1}\frac1{\max_j(\sqP{j})}\abs{\sqP{1}-\sqP{2}}(t,\eta)\\
&\qquad \times\Big(\int_0^\eta e^{-\frac{1}{\sqtC{2}}(\tilde\Y_2(t,\eta)-\tilde\Y_2(t,\theta))}\tilde\P_2\tilde \Y_{2,\eta}
\abs{\tilde \U_2^+-\tilde \U_1^+}(t,\theta)d\theta  \Big)d\eta\\
&\le \norm{\sqP{1}-\sqP{2}}^2  +\frac{1}{\A^4}\int_0^{1}\frac1{(\max_j(\sqP{j}))^2}(t,\eta)\\
&\qquad\times
\Big(\int_0^\eta e^{-\frac{1}{\sqtC{2}}(\tilde\Y_2(t,\eta)-\tilde\Y_2(t,\theta))}\tilde\P_2\tilde \Y_{2,\eta}
\abs{\tilde \U_2^+-\tilde \U_1^+}(t,\theta)d\theta  \Big)^2d\eta\\
&\le \norm{\sqP{1}-\sqP{2}}^2 +\frac{1}{\A^4}\int_0^{1}\frac1{(\max_j(\sqP{j}))^2}(t,\eta)\\
&\qquad \times
\Big(\int_0^\eta e^{-\frac3{2\sqtC{2}}(\tilde\Y_2(t,\eta)-\tilde\Y_2(t,\theta))}\tilde\P_2\tilde \Y_{2,\eta} (t,\theta)d\theta  \Big) \\
&\qquad \times\Big(\int_0^\eta e^{-\frac1{2\sqtC{2}}(\tilde\Y_2(t,\eta)-\tilde\Y_2(t,\theta))}\tilde\P_2\tilde \Y_{2,\eta}\big(\tilde \U_2-\tilde \U_1\big)^2(t,\theta)d\theta  \Big)d\eta\\
&\le \norm{\sqP{1}-\sqP{2}}^2 \\
&\qquad +\frac{2}{\A^3}\int_0^{1}
\Big(\int_0^\eta e^{-\frac1{2\sqtC{2}}(\tilde\Y_2(t,\eta)-\tilde\Y_2(t,\theta))}\tilde\P_2\tilde \Y_{2,\eta}\big(\tilde \U_2-\tilde \U_1\big)^2(t,\theta)d\theta  \Big)d\eta\\
&\le \bigO(1) \big( \norm{\sqP{1}-\sqP{2}}^2+ \norm{\tilde \U_2-\tilde \U_1}^2 \big),
\end{align*}
using \eqref{eq:all_estimatesE}, \eqref{eq:all_estimatesN}, and \eqref{eq:343}. 

The terms $A_{33}$ and $A_{34}$ can be treated as follows
\begin{align*}
\abs{A_{33}}&=\frac{1}{2\A^3}\vert \int_0^1 \frac{1}{\max_j(\sqP{j})}(\sqP{1}-\sqP{2})\\
& \qquad \times\mathbbm{1}_{\sqtC{1}\leq \sqtC{2}}\int_0^\eta (e^{-\frac{1}{\sqtC{2}}(\tilde \Y_2(t,\eta)-\tilde \Y_2(t,\theta)}-e^{-\frac{1}{\sqtC{1}}(\tilde \Y_2(t,\eta)-\tilde \Y_2(t,\theta))})\\
&\qquad\qquad\qquad\qquad\times\min_j(\tilde \D_j)\min_j(\tilde \U_j^+)\tilde \Y_{2,\eta}(t,\theta) d\theta d\eta\\
& \leq \frac{\sqrt{2}\ma}{\A^3e}\int_0^1 \frac{1}{\max_j(\sqP{j})}\abs{\sqP{1}-\sqP{2}}(t,\eta)\\
& \qquad \times \Big(\int_0^\eta e^{-\frac{3}{4\sqtC{2}}(\tilde \Y_2(t,\eta)-\tilde \Y_2(t,\theta))}\tilde \D_2\tilde \Y_{2,\eta}(t,\theta) d\theta\Big) d\eta \vert \sqtC{1}-\sqtC{2}\vert\\
& \leq \frac{2\sqrt{2}}{\A^2 e}\int_0^1 \frac{1}{\max_j(\sqP{j})}\abs{\sqP{1}-\sqP{2}}(t,\eta)\\
&\qquad\qquad\times\Big(\int_0^\eta e^{-\frac{3}{2\sqtC{2}}(\tilde \Y_2(t,\eta)-\tilde \Y_2(t,\theta))}\tilde \P_2\tilde \Y_{2,\eta} (t,\theta) d\theta\Big)^{1/2}\\
& \qquad \times \Big(\int_0^\eta \tilde \P_2\tilde \Y_{2,\eta}(t,\theta) d\theta\Big)^{1/2}d\eta \vert \sqtC{1}-\sqtC{2}\vert\\
& \leq \bigO(1)(\norm{\sqP{1}-\sqP{2}}^2+\abs{\sqtC{1}-\sqtC{2}}^2).
\end{align*}

 Furthermore, the terms $A_{35}$ and $A_{36}$ can be estimated like this
\begin{align*}
\abs{A_{35}}&=\frac1{2\A^3}\Big\vert\int_0^{1}\frac1{\max_j(\sqP{j})}(\sqP{1}-\sqP{2})(t,\eta)\\
&\qquad \times\Big(\int_0^\eta \big(e^{-\frac{1}{\ma}(\tilde\Y_2(t,\eta)-\tilde\Y_2(t,\theta))}-e^{-\frac{1}{\ma}(\tilde\Y_1(t,\eta)-\tilde\Y_1(t,\theta))} \big)\\
&\qquad\qquad\times\min_j\big(\tilde\D_j\big)\min_j\big(\tilde\U^+_j\big)
\tilde \Y_{2,\eta}\mathbbm{1}_{B(\eta)}(t,\theta)d\theta \Big)d\eta \Big\vert \\
&\le\frac1{2\ma\A^3}\int_0^{1}\frac1{\max_j(\sqP{j})}\abs{\sqP{1}-\sqP{2}}(t,\eta)\\
&\qquad \times\Big(\int_0^\eta e^{-\frac{1}{\ma}(\tilde\Y_2(t,\eta)-\tilde\Y_2(t,\theta))}\big(\abs{\tilde\Y_2(t,\eta)-\tilde\Y_1(t,\eta)}+\abs{\tilde\Y_2(t,\theta)-\tilde\Y_1(t,\theta)} \big)\\
&\qquad\times
\min_j\big(\tilde\D_j\big)\min_j\big(\tilde\U^+_j\big)
\tilde \Y_{2,\eta}\mathbbm{1}_{B(\eta)}(t,\theta)d\theta \Big)d\eta \\
&\le\frac{\ma}{2\sqrt{2}\A^3}\int_0^{1}\frac1{\max_j(\sqP{j})}\abs{\sqP{1}-\sqP{2}}\abs{\tilde\Y_2-\tilde\Y_1}(t,\eta)\\
&\qquad \times\Big(\int_0^\eta e^{-\frac{1}{\ma}(\tilde\Y_2(t,\eta)-\tilde\Y_2(t,\theta))}
\min_j\big(\tilde\D_j\big)
\tilde \Y_{2,\eta}\mathbbm{1}_{B(\eta)}(t,\theta)d\theta \Big)d\eta \\
&\quad +\frac{\ma}{2\sqrt{2}\A^3}\int_0^{1}\frac1{\max_j(\sqP{j})}\abs{\sqP{1}-\sqP{2}}(t,\eta)\\
&\qquad \times\Big(\int_0^\eta e^{-\frac{1}{\ma}(\tilde\Y_2(t,\eta)-\tilde\Y_2(t,\theta))}\abs{\tilde\Y_2-\tilde\Y_1}
\min_j\big(\tilde\D_j\big)
\tilde \Y_{2,\eta}\mathbbm{1}_{B(\eta)}(t,\theta)d\theta \Big)d\eta \\
&\le\frac{1}{\sqrt{2}\A}\int_0^{1}\frac1{\max_j(\sqP{j})}\abs{\sqP{1}-\sqP{2}}\abs{\tilde\Y_2-\tilde\Y_1}(t,\eta) \\
&\qquad\qquad\times
\Big(\int_0^\eta e^{-\frac{3}{2\sqtC{2}}(\tilde\Y_2(t,\eta)-\tilde\Y_2(t,\theta))} \tilde\P_2 \tilde \Y_{2,\eta} (t,\theta)d\theta \Big)^{1/2}d\eta\\
&\qquad\qquad \times\Big(\int_0^\eta e^{-\frac{1}{2\sqtC{2}}(\tilde\Y_2(t,\eta)-\tilde\Y_2(t,\theta))}\tilde\P\tilde \Y_{2,\eta}(t,\theta)d\theta \Big)^{1/2}d\eta \\
&\quad +\frac{1}{\sqrt{2}\A}\int_0^{1}\frac1{\max_j(\sqP{j})}\abs{\sqP{1}-\sqP{2}}(t,\eta) \\
&\qquad\qquad \times\Big(\int_0^\eta e^{-\frac{3}{2\sqtC{2}}(\tilde\Y_2(t,\eta)-\tilde\Y_2(t,\theta))}  \tilde\P_2 \tilde \Y_{2,\eta}  (t,\theta)d\theta \Big)^{1/2} \\
&\qquad \times\Big(\int_0^\eta e^{-\frac{1}{2\sqtC{2}}(\tilde\Y_2(t,\eta)-\tilde\Y_2(t,\theta))}\tilde\P\tilde \Y_{2,\eta}(\tilde\Y_2-\tilde\Y_1)^2(t,\theta)d\theta \Big)^{1/2}d\eta \\
&\le \bigO(1)\big( \norm{\sqP{1}-\sqP{2}}^2+ \norm{\tilde \Y_2-\tilde \Y_1}^2 \big),
\end{align*}
using \eqref{eq:343}, \eqref{eq:all_estimatesE}, and \eqref{eq:all_estimatesN}.  As for $A_{37}$ we follow this path: 
\begin{align*}
\abs{A_{37}}&=\frac1{2\A^3}\Big\vert\int_0^{1}\frac1{\max_j(\sqP{j})}(\sqP{1}-\sqP{2})(t,\eta)\Big(\int_0^\eta \min_j\big(e^{-\frac{1}{\ma}(\tilde\Y_j(t,\eta)-\tilde\Y_j(t,\theta))}\big)\\
&\qquad \qquad\times\min_j\big(\tilde\D_j\big)\min_j\big(\tilde\U^+_j\big)\big(\tilde \Y_{2,\eta}-\tilde \Y_{1,\eta} \big)(t,\theta)d\theta
\Big)d\eta\Big\vert\\
&=\frac1{2\A^3}\Big\vert\int_0^{1}\frac1{\max_j(\sqP{j})}(\sqP{1}-\sqP{2})(t,\eta)\\
&\quad \times\Big[\Big(\min_j\big(e^{-\frac{1}{\ma}(\tilde\Y_j(t,\eta)-\tilde\Y_j(t,\theta))}\big)\min_j\big(\tilde\D_j\big)\min_j\big(\tilde\U^+_j\big)\big(\tilde \Y_{2}-\tilde \Y_{1} \big)(t,\theta)\Big)\Big\vert_{\theta=0}^\eta \\
&\quad -\int_0^\eta \big(\tilde \Y_{2}-\tilde \Y_{1} \big)\\
&\qquad\qquad\times\frac{d}{d\theta}\Big(\min_j\big(e^{-\frac{1}{\ma}(\tilde\Y_j(t,\eta)-\tilde\Y_j(t,\theta))}\big)\min_j\big(\tilde\D_j\big)\min_j\big(\tilde\U^+_j\big) \Big)(t,\theta)d\theta
\Big]d\eta\Big\vert\\
&\le\frac1{2\A^3}\Big\vert\int_0^{1}\frac1{\max_j(\sqP{j})}(\sqP{1}-\sqP{2})\min_j\big(\tilde\D_j\big)\min_j\big(\tilde\U^+_j\big)\big(\tilde \Y_{2}-\tilde \Y_{1} \big)(t,\eta)d\eta \Big\vert\\
&\quad +\frac1{2\A^3}\Big\vert\int_0^{1}\frac1{\max_j(\sqP{j})}(\sqP{1}-\sqP{2})(t,\eta)\Big(\int_0^\eta \big(\tilde \Y_{2}-\tilde \Y_{1} \big)\\
&\qquad\times\frac{d}{d\theta}\Big(\min_j\big(e^{-\frac{1}{\ma}(\tilde\Y_j(t,\eta)-\tilde\Y_j(t,\theta))}\big)\min_j\big(\tilde\D_j\big)\min_j\big(\tilde\U^+_j\big) \Big)(t,\theta)d\theta
\Big)d\eta\Big\vert\\
&\le\frac{1}{\A^2}\int_0^{1}\frac1{\max_j(\sqP{j})}\abs{\sqP{1}-\sqP{2}}\tilde\P_2\abs{\tilde\U_2}\abs{\tilde \Y_{2}-\tilde \Y_{1}}(t,\eta)d\eta \\
&\quad +\frac1{2\A^3}\Big\vert\int_0^{1}\frac1{\max_j(\sqP{j})}(\sqP{1}-\sqP{2})(t,\eta)\Big(\int_0^\eta \big(\tilde \Y_{2}-\tilde \Y_{1} \big) \\
&\qquad\qquad\times \Big(\big(\frac{d}{d\theta}\min_j\big(e^{-\frac{1}{\ma}(\tilde\Y_j(t,\eta)-\tilde\Y_j(t,\theta))}\big)\big)\min_j\big(\tilde\D_j\big)\min_j\big(\tilde\U^+_j\big) \\
&\qquad+\min_j\big(e^{-\frac{1}{\ma}(\tilde\Y_j(t,\eta)-\tilde\Y_j(t,\theta))}\big)\big(\frac{d}{d\theta}\min_j\big(\tilde\D_j\big)\min_j\big(\tilde\U^+_j\big) \big)\Big)(t,\theta)d\theta
\Big)d\eta\Big\vert\\
&\le \bigO(1)\big( \norm{\sqP{1}-\sqP{2}}^2+ \norm{\tilde \Y_1-\tilde \Y_2}^2 \big) \\
&\quad +\frac1{2\A^3}\int_0^{1}\frac1{\max_j(\sqP{j})}\abs{\sqP{1}-\sqP{2}}(t,\eta)\\
&\quad\times\Big(\int_0^\eta \abs{\tilde \Y_{2}-\tilde \Y_{1}} \frac{1}{\ma}\Big(\min_j\big(e^{-\frac{1}{\ma}(\tilde\Y_j(t,\eta)-\tilde\Y_j(t,\theta))}\big)\\
&\qquad\qquad\qquad\qquad\times\max_j\big(\tilde\Y_{j,\eta}\big)\min_j\big(\tilde\D_j\big)\min_j\big(\tilde\U^+_j\big) \\
&\qquad+\bigO(1)\A^5\min_j\big(e^{-\frac{1}{\ma}(\tilde\Y_j(t,\eta)-\tilde\Y_j(t,\theta))}\big)\big(\min_j\big(\tilde\D_j^{1/2}\big)+\abs{\tilde\U_2} \big)\Big)(t,\theta)d\theta
\Big)d\eta\\
&\le \bigO(1)\big( \norm{\sqP{1}-\sqP{2}}^2+ \norm{\tilde \Y_1-\tilde \Y_2}^2 \big) +\frac12(A_{371}+A_{372}),
\end{align*}
using \eqref{eq:all_estimatesA},  \eqref{eq:all_estimatesB}, \eqref{eq:all_estimatesN}, and Lemmas \ref{lemma:1} and \ref{lemma:5}. Here
\begin{align*}
A_{371}&=\frac{1}{\ma\A^3}\int_0^{1}\frac1{\max_j(\sqP{j})}\abs{\sqP{1}-\sqP{2}}(t,\eta)\\
&\quad\times\Big(\int_0^\eta \abs{\tilde \Y_{2}-\tilde \Y_{1}} \Big(\min_j\big(e^{-\frac{1}{\ma}(\tilde\Y_j(t,\eta)-\tilde\Y_j(t,\theta))}\big)\\
&\qquad\qquad\qquad\qquad\times\max_j\big(\tilde\Y_{j,\eta}\big)\min_j\big(\tilde\D_j\big)\min_j\big(\tilde\U^+_j\big)(t,\theta)d\theta\Big)d\eta \\
&\le \frac{\sqrt{2}\ma}{\A^2}\int_0^{1}\frac1{\max_j(\sqP{j})}\abs{\sqP{1}-\sqP{2}}(t,\eta)\Big(\int_0^\eta \abs{\tilde \Y_{2}-\tilde \Y_{1}} \\
&\qquad\times\Big(e^{-\frac{1}{\sqtC{1}}(\tilde\Y_1(t,\eta)-\tilde\Y_1(t,\theta))}\tilde\P_1\tilde\Y_{1,\eta}
+e^{-\frac{1}{\sqtC{2}}(\tilde\Y_2(t,\eta)-\tilde\Y_2(t,\theta))}\tilde\P_2\tilde\Y_{2,\eta}\Big)(t,\theta)d\theta\Big)d\eta \\
&\le \frac{\sqrt{2}}{\A}\int_0^{1}\frac1{\max_j(\sqP{j})}\abs{\sqP{1}-\sqP{2}}(t,\eta)\\
&\quad\times\Big[\Big(\int_0^\eta e^{-\frac{3}{2\sqtC{1}}(\tilde\Y_1(t,\eta)-\tilde\Y_1(t,\theta))}\tilde\P_1\tilde\Y_{1,\eta}(t,\theta)d\theta\Big)^{1/2} \\
&\qquad\times\Big(\int_0^\eta e^{-\frac{1}{2\sqtC{1}}(\tilde\Y_1(t,\eta)-\tilde\Y_1(t,\theta))}\tilde\P_1\tilde\Y_{1,\eta} (\tilde \Y_{2}-\tilde \Y_{1})^2(t,\theta)d\theta\Big)^{1/2} \\
&\qquad+ \Big(\int_0^\eta e^{-\frac{3}{2\sqtC{2}}(\tilde\Y_2(t,\eta)-\tilde\Y_2(t,\theta))}\tilde\P_2\tilde\Y_{2,\eta}(t,\theta)d\theta\Big)^{1/2} \\
&\qquad\qquad\times\Big(\int_0^\eta e^{-\frac{1}{2\sqtC{2}}(\tilde\Y_2(t,\eta)-\tilde\Y_2(t,\theta))}\tilde\P_2\tilde\Y_{2,\eta} (\tilde \Y_{2}-\tilde \Y_{1})^2(t,\theta)d\theta\Big) ^{1/2}\Big]d\eta \\
&\le \bigO(1)\big( \norm{\sqP{1}-\sqP{2}}^2+ \norm{\tilde \Y_1-\tilde \Y_2}^2 \big),
\end{align*}
using $\max_j(a_j)\le a_1+a_2$ and $\min_j(b_j)\le b_k$.  Then
\begin{align*}
A_{372}&=\bigO(1)\int_0^{1}\frac1{\max_j(\sqP{j})}\abs{\sqP{1}-\sqP{2}}(t,\eta)\\
&\quad\times\Big(\int_0^\eta \abs{\tilde \Y_{2}-\tilde \Y_{1}}\min_j\big(e^{-\frac{1}{\ma}(\tilde\Y_j(t,\eta)-\tilde\Y_j(t,\theta))}\big)\big(\min_j\big(\tilde\D_j^{1/2}\big)+\abs{\tilde\U_2} \big)(t,\theta)d\theta
\Big)d\eta\\
&\le \bigO(1)\int_0^{1}\frac1{\max_j(\sqP{j})}\abs{\sqP{1}-\sqP{2}}(t,\eta)\\
&\qquad\times\Big(\int_0^\eta \abs{\tilde \Y_{2}-\tilde \Y_{1}}e^{-\frac{1}{\sqtC{2}}(\tilde\Y_2(t,\eta)-\tilde\Y_2(t,\theta))}\tilde\P_2^{1/2}(t,\theta)d\theta
\Big)d\eta\\
&\le \bigO(1)\int_0^{1}\frac1{\max_j(\sqP{j})}\abs{\sqP{1}-\sqP{2}}(t,\eta)\\
&\qquad\qquad\times\Big(\int_0^\eta e^{-\frac3{2\sqtC{2}}(\tilde\Y_2(t,\eta)-\tilde\Y_2(t,\theta))}\tilde\P_2 (t,\theta)d\theta\Big)^{1/2}\\
&\qquad\qquad\times\Big(\int_0^\eta e^{-\frac1{2\sqtC{2}}(\tilde\Y_2(t,\eta)-\tilde\Y_2(t,\theta))}\big(\tilde \Y_{2}-\tilde \Y_{1})^2(t,\theta)d\theta \Big)^{1/2}d\eta\\
&\le \bigO(1)\big( \norm{\sqP{1}-\sqP{2}}^2+ \norm{\tilde \Y_1-\tilde \Y_2}^2 \big),
\end{align*}
using \eqref{eq:all_estimatesD}, \eqref{eq:all_estimatesN}, and \eqref{eq:32P}. This completes the estimate for $A_3$ (see \eqref{eq:A3}):
\begin{equation*}
A_3\le \bigO(1)\big( \norm{\sqP{1}-\sqP{2}}^2+ \norm{\tilde \U_1-\tilde \U_2}^2+ \norm{\tilde \Y_1-\tilde \Y_2}^2+\abs{\sqtC{1}-\sqtC{2}}^2 \big).
\end{equation*}

\bigskip
For $J_4$ (and similarly for $J_5$) the estimates read
\begin{align*}
\abs{J_4}&\le\frac{1}{\A^3}\int_0^{1} \abs{\sqP{1}-\sqP{2}}^2\frac{\abs{\tilde\R_2}}{\tilde\P_2}(t,\eta)d\eta\\
&\le \bigO(1) \norm{\sqP{1}-\sqP{2}}^2,
\end{align*}
using \eqref{eq:all_estimatesJ}.

We have shown the anticipated result.
\begin{lemma} \label{lemma:sqP}
Let $\sqP{i}$ be two solutions of \eqref{eq:Plip} for $i=1,2$.  Then we have
\begin{align*}
\frac{d}{dt} &\norm{\sqP{1}-\sqP{2}}^2\\&\leq \bigO(1)\big(\norm{\tilde\Y_1-\tilde\Y_2}^2+\norm{\tilde\U_1-\tilde\U_2}^2+\norm{\sqP{1}-\sqP{2}}^2+ \abs{\sqtC{1}-\sqtC{2}} ^2\big),
\end{align*}
where $\bigO(1)$ denotes some constant which only depends on $\A=\max_j(\sqtC{j})$ and which remains bounded as $\A\to 0$.
\end{lemma}

%------------------- appendix -------------------
\appendix

\section{Lipschitz continuity and uniformly bounded Lipschitz constants}
We need to establish that a number of complicated functions are Lipschitz continuous with uniformly bounded Lipschitz constants. In a desperate attempt to ease the readability of the main estimates, we collect these results here together with some other estimates which are essential in Section~\ref{sec:Lip}.

%----------------------------

\begin{lemma}\label{lemma:enkel}
(i) We have
\begin{align}\label{eq:enkel}
\vert e^{a_1-a_2}-e^{b_1-b_2} \vert&\le
\max\big(e^{a_1-a_2}, e^{b_1-b_2} \big) \big( \vert b_1-a_1 \vert +\vert b_2-a_2 \vert \big)  \\
&\le \vert b_1-a_1 \vert +\vert b_2-a_2 \vert , \label{eq:enkel1} \\
&\qquad\qquad\qquad\qquad a_1<a_2, \quad  b_1<b_2. \notag
\end{align}
(ii) Let $0<a\le A_j\le A$, $j=1,2$. Then we have
\begin{equation}
\vert e^{-\frac{1}{\sqtC{2}}x}-e^{-\frac{1}{\sqtC{1}}x}\vert\le 
\frac{4}{ae}e^{-\frac{3}{4A} x}\vert \sqtC{2}-\sqtC{1}\vert, \quad x\in [0,\infty).
\end{equation}
\end{lemma}
\begin{proof}
(i) The results follows from the elementary inequality
\begin{equation*}
\vert e^a-e^b\vert =\vert \int_b^a e^x \,dx \vert\le 
e^{b}\vert \int_b^a\,dx \vert\le e^{b} \vert b-a \vert, \quad a<b.
\end{equation*}
(ii)  For $x\in [0,\infty)$ one can write
\begin{align*}
\vert e^{-\frac{1}{\sqtC{2}}x}-e^{-\frac{1}{\sqtC{1}}x}\vert & = \vert \int_{\sqtC{1}}^{\sqtC{2}} \frac{1}{s^2}x e^{-\frac{1}{s}x}ds\vert \\
& \leq e^{-\frac{3}{4A}x} \vert \int_{\sqtC{1}}^{\sqtC{2}}\frac{1}{s^2} xe^{-\frac{1}{4s}x} ds\vert \\
& \leq \frac{4}{ae}e^{-\frac{3}{4A} x}\vert \sqtC{2}-\sqtC{1}\vert.
\end{align*}
Here we used in the last step that for $s\in [\ma, \A]$, the function $f\colon[0,\infty)\to [0,\infty)$ with $f(x)=\frac{1}{s^2} xe^{-\frac{1}{4s} x}$ attains its maximum at $x=4s$ and 
\begin{equation*}
0\leq f(4s)=\frac{4}{se}\leq \frac4{\ma e}.
\end{equation*}
\end{proof}

%----------------------------
\begin{lemma}  \label{lemma:1}
(i): The function $\theta\mapsto \min_j(e^{-\frac{1}{\ma}(\tilde \Y_j(t,\eta)-\tilde\Y_j(t,\theta))})$ is non-decreasing for almost every $\eta$ and thus differentiable almost everywhere. We have that
\begin{equation} \label{est:L3a_lemma}
\vert \frac{d}{d\theta}  \min_j(e^{-\frac{1}{\ma}(\tilde \Y_j(t,\eta)-\tilde\Y_j(t,\theta))})\vert
 \leq
 \frac{1}{\ma}\min_j(e^{-\frac{1}{\ma}(\tilde \Y_j(t,\eta)-\tilde\Y_j(t,\theta))}) \max_j(\tilde\Y_{j,\eta}(t,\theta)).
\end{equation}
(ii)  The function $\theta\mapsto \min_j(\tilde \U_j^2)(t,\theta)$ is  differentiable almost everywhere with  
\begin{equation}\label{est:L2c_lemma}
\vert \frac{d}{d\eta}\min_j(\tilde\U_j^2(t,\theta))\vert \leq A^4.
\end{equation}
(iii) The function  $\theta\mapsto\min_j(\tilde\V_j^+)^3(t,\theta)$ is differentiable almost everywhere with
\begin{equation}\label{est:L3c_lemma}
\big\vert \frac{d}{d\theta} \min_j(\tilde\V_j^+)^3(t,\theta)\big\vert \leq 2\A^4\min_j(\tilde\V_j^+)(t,\theta).
\end{equation}
\end{lemma}
\begin{proof}
(i):  First of all note that the function $\theta\mapsto \min_j(e^{-\frac{1}{\ma}(\tilde \Y_j(t,\eta)-\tilde\Y_j(t,\theta))})$ is non-decreasing and hence differentiable almost everywhere. Consider in the following $\theta<\eta$.
Assume that for fixed $\theta$, such that the given function is differentiable, we have 
\begin{equation} \label{eq:fix_theta1}
\min_j(e^{-\frac{1}{\ma}(\tilde \Y_j(t,\eta)-\tilde\Y_j(t,\theta))})= e^{-\frac{1}{\ma}(\tilde\Y_1(t,\eta)-\tilde\Y_1(t,\theta))} 
\end{equation}
and that there exists a sequence $\theta_n\uparrow \theta$ such that 
\begin{equation*}
\min_j(e^{-\frac{1}{\ma}(\tilde\Y_j(t,\eta)-\tilde\Y_j(t,\theta_n))})= e^{-\frac{1}{\ma}(\tilde\Y_1(t, \eta)-\tilde\Y_1(t,\theta_n))}\quad \text{ for all } n.
\end{equation*}
Then we have 
\begin{align*}
\vert \min_j(e^{-\frac{1}{\ma}(\tilde\Y_j(t,\eta)-\tilde\Y_j(t,\theta))})&-\min_j(e^{-\frac{1}{\ma}(\tilde\Y_j(t, \eta)-\tilde\Y_j(t,\theta_n))})\vert\\
& \qquad= e^{-\frac{1}{\ma}(\tilde\Y_1(t,\eta)-\tilde\Y_1(t,\theta))}- e^{-\frac{1}{\ma}(\tilde\Y_1(t, \eta)-\tilde \Y_1(t,\theta_n))}\\
&\qquad \leq \frac{1}{\ma}e^{-\frac{1}{\ma}(\tilde\Y_1(t,\eta)-\tilde\Y_1(t,\theta))} (\tilde\Y_1(t,\theta)-\tilde\Y_1(t,\theta_n)),
\end{align*}
since $\tilde \Y_i(t,\dott)$ is non-decreasing.
\begin{align*}
\vert \frac{d}{d\theta}  \min_j(e^{-\frac{1}{\ma}(\tilde\Y_j(t,\eta)-\tilde\Y_j(t,\theta))})\vert 
& \leq \lim_{\theta_n\uparrow \theta } \frac{1}{\ma}e^{-\frac{1}{\ma}(\tilde\Y_1(t,\eta)-\tilde\Y_1(t,\theta))} \frac{\tilde\Y_1(t,\theta)-\tilde\Y_1(t,\theta_n)} {\theta-\theta_n}\\
& = \frac{1}{\ma}e^{-\frac{1}{\ma}(\tilde\Y_1(t,\eta)-\tilde\Y_1(t,\theta))}\tilde\Y_{1,\eta}(t,\theta)\\
& =\frac{1}{\ma} \min_j(e^{-\frac{1}{\ma}(\tilde \Y_j(t,\eta)-\tilde\Y_j(t,\theta))}) \tilde\Y_{1,\eta}(t,\theta).
\end{align*}
Assume, on the other hand, that for fixed $\theta$ we have
\begin{equation*}
\min_j(e^{-\frac{1}{\ma}(\tilde \Y_j(t,\eta)-\tilde\Y_j(t,\theta))})= e^{-\frac{1}{\ma}(\tilde\Y_1(t,\eta)-\tilde\Y_1(t,\theta))} 
\end{equation*}
and that there exists a sequence $\theta_n\downarrow \theta$ such that 
\begin{equation*}
\min_j(e^{-\frac{1}{\ma}(\tilde\Y_j(t,\eta)-\tilde\Y_j(t,\theta_n))})= e^{-\frac{1}{\ma}(\tilde\Y_1(t, \eta)-\tilde\Y_1(t,\theta_n))}\quad \text{ for all } n.
\end{equation*}
Then we have 
\begin{align*}
\vert \min_j(e^{-\frac{1}{\ma}(\tilde\Y_j(t,\eta)-\tilde\Y_j(t,\theta))})&-\min_j(e^{-\frac{1}{\ma}(\tilde\Y_j(t, \eta)-\tilde\Y_j(t,\theta_n))})\vert\\
& \qquad= e^{-\frac{1}{\ma}(\tilde\Y_1(t,\eta)-\tilde\Y_1(t,\theta_n))}- e^{-\frac{1}{\ma}(\tilde\Y_1(t, \eta)-\tilde\Y_1(t,\theta))}\\
& \qquad\leq \frac{1}{\ma}e^{-\frac{1}{\ma}(\tilde\Y_1(t,\eta)-\tilde\Y_1(t,\theta_n))} (\tilde\Y_1(t,\theta_n)-\tilde\Y_1(t,\theta)),
\end{align*}
since $\tilde\Y_i(t,\dott)$ is non-decreasing.
Thus  
\begin{align*}
\vert \frac{d}{d\theta}  \min_j(e^{-\frac{1}{\ma}(\tilde\Y_j(t,\eta)-\tilde\Y_j(t,\theta))})\vert 
& \leq \lim_{\theta_n\downarrow \theta } \frac{1}{\ma}e^{-\frac{1}{\ma}(\tilde \Y_1(t,\eta)-\tilde\Y_1(t,\theta_n))} \frac{\tilde\Y_1(t,\theta_n)-\tilde\Y_1(t,\theta)} {\theta_n-\theta}\\
& = \frac{1}{\ma}e^{-\frac{1}{\ma}(\tilde\Y_1(t,\eta)-\tilde\Y_1(t,\theta))}\tilde\Y_{1,\eta}(t,\theta)\\
& = \min_j(e^{-\frac{1}{\ma}(\tilde \Y_j(t,\eta)-\tilde\Y_j(t,\theta))}) \tilde\Y_{1,\eta}(t,\theta).
\end{align*}
Thus in the case  \eqref{eq:fix_theta1} we find
\begin{equation*}
\vert \frac{d}{d\theta}  \min_j(e^{-\frac{1}{\ma}(\tilde \Y_j(t,\eta)-\tilde\Y_j(t,\theta))})\vert  \leq
\frac{1}{\ma} \min_j(e^{-\frac{1}{\ma}(\tilde \Y_j(t,\eta)-\tilde\Y_j(t,\theta))}) \tilde\Y_{1,\eta}(t,\theta).
\end{equation*}
If we instead of \eqref{eq:fix_theta1} assume
\begin{equation*} %\label{eq:fix_theta12}
\min_j(e^{-\frac{1}{\ma}(\tilde \Y_j(t,\eta)-\tilde\Y_j(t,\theta))})= e^{-\frac{1}{\ma}(\tilde\Y_2(t,\eta)-\tilde\Y_2(t,\theta))}, 
\end{equation*}
a similar argument yields
\begin{equation*}
\vert \frac{d}{d\theta}  \min_j(e^{-\frac{1}{\ma}(\tilde \Y_j(t,\eta)-\tilde\Y_j(t,\theta))})\vert   \leq
 \min_j(e^{-\frac{1}{\ma}(\tilde \Y_j(t,\eta)-\tilde\Y_j(t,\theta))}) \tilde\Y_{2,\eta}(t,\theta).
\end{equation*}
Thus we conclude that in general we have
\begin{equation*} %\label{est:L3a}
\vert \frac{d}{d\theta}  \min_j(e^{-\frac{1}{\ma}(\tilde \Y_j(t,\eta)-\tilde\Y_j(t,\theta))})\vert
 \leq
\frac{1}{\ma} \min_j(e^{-\frac{1}{\ma}(\tilde \Y_j(t,\eta)-\tilde\Y_j(t,\theta))}) \max_j(\tilde\Y_{j,\eta})(t,\theta).
\end{equation*}

\medskip
(ii):  We have (cf.~Lemma \ref{lemma:LUR}) that
\begin{align}
\vert   \min_j(\tilde\U_j^2(t,\eta))-\min_j(\tilde\U_j^2(t,\theta))\vert&\le 
\max\big(\vert\tilde\U_1^2(t,\eta)-\tilde\U_1^2(t,\theta)\vert, \vert \tilde\U_2^2(t,\eta)-\tilde\U_2^2(t,\eta)\vert \big)\nn \\
&\le 2 \max_j\norm{\tilde\U_j\tilde\U_{j\eta}}_\infty \vert\eta-\theta\vert \le A^4\vert\eta-\theta\vert, \label{eq:L2estim}
\end{align}
using \eqref{eq:all_estimatesF}.

\medskip
(iii): Note that one has that for any positive function $m(x)$, 
\begin{align*}
\vert m^3(x)-m^3(y)\vert & = (m^2(x)+m(x)m(y)+m^2(y))\vert m(x)-m(y)\vert \notag\\
& \leq (m^2(x)+2m(x)m(y)+m^2(y))\vert m(x)-m(y)\vert \notag\\
& = (m(x)+m(y))^2 \vert m(x)-m(y)\vert\notag \\
& =(m(x)+m(y))\vert m^2(x)-m^2(y)\vert. %\label{eq:m3}
\end{align*}
If we replace $m(x)$ by $\min_j(\tilde\V_j^+)(t,\eta)$ we have 
\begin{align*}
&\vert \min_j(\tilde\V_j^+)^3(t,\eta)  - \min_j(\tilde\V_j^+)^3 (t, \theta)\vert \\
& \quad \leq (\min_j(\tilde\V_j^+)(t,\eta)+\min_j(\tilde\V_j^+)(t,\theta) )\vert \min_j(\tilde\V_j^+)^2(t,\eta)-\min_j(\tilde\V_j^+)^2(t,\theta)\vert \\
&\quad \leq (\min_j(\tilde\V_j^+)(t,\eta)+\min_j(\tilde\V_j^+)(t,\theta) )\\
&\qquad\qquad\qquad\qquad\qquad\times\max (\vert (\tilde\V_1^+)^2(t,\eta)-(\tilde\V_1^+)^2(t,\theta)\vert , \vert (\tilde \V_2^+)^2(t,\eta)-(\tilde \V_2^+)^2(t,\theta)\vert )\\
& \quad\leq A^4(\min_j(\tilde\V_j^+)(t,\eta)+\min_j(\tilde\V_j^+)(t,\theta) ) \vert \eta-\theta\vert, 
\end{align*}
see \eqref{eq:L2estim}, and hence $\min_j(\tilde \V_j^+)^3(t,\eta)$ is differentiable almost everywhere with
\begin{equation*}%\label{est:L3c}
\big\vert \frac{d}{d\theta} \min_j(\tilde\V_j^+)^3(t,\theta)\big\vert \leq 2A^4\min_j(\tilde\V_j^+)(t,\theta).
\end{equation*}
\end{proof}
%-----------------------------------------
%----------------------
\begin{lemma}\label{lemma:3}
(i) The function $\eta\mapsto\min_j(\tilde\P_j)\min_j(\tilde \V_j^+)(t,\eta)$ is 
Lipschitz continuous with a uniformly bounded Lipschitz constant and thus differentiable almost everywhere with
\begin{equation*}
\vert \frac{d}{d\eta} (\min_j(\tilde \P_j)\min_j(\tilde\V_j^+))(t,\eta)\vert\leq 2\A^4 (\min_j(\tilde \P_j)^{1/2}+\vert \tilde\U_k\vert)(t,\eta), \quad k=1,2.
\end{equation*}
(ii) The function $\eta\mapsto\min_j(\tilde\P_j)\tilde\U_k(t,\eta)$, $k=1,2$ is Lipschitz continuous with a uniformly bounded Lipschitz constant and thus differentiable almost everywhere with 
\begin{equation*}
\vert \frac{d}{d\eta} (\min_j(\tilde \P_j)\tilde\U_k)(t,\eta)\vert\leq \frac{1}{\sqrt{2}}\A^6.
\end{equation*}
\end{lemma}

\begin{proof}
(i) We only present the proof for the case $k=2$, since the case $k=1$ is similar. 

Given $0\leq \eta_1<\eta_2\leq 1$, we assume without loss of generality that
\begin{equation*}
0\leq \min_j(\tilde \P_j)\min_j(\tilde\V_j^+)(t,\eta_2)-\min_j(\tilde\P_j)\min_j(\tilde\V_j^+)(t,\eta_1).
\end{equation*}
To ease the notation we introduce the function
\begin{equation*}
d(t,\eta)=\min_j(\tilde \P_j)\min_j(\tilde\V_j^+)(t,\eta).
\end{equation*}
We will distinguish several cases:

(1): If $d(t,\eta_2)=0$, then $d(t,\eta_1)=0$ and one has 
\begin{align*}
0&\leq \min_j(\tilde\V_j^+)(t,\eta_2)\vert \min_j(\tilde\P_j)(t,\eta_2)-\min_j(\tilde\P_j)(t,\eta_1)\vert )\\
 & \leq \min_j(\tilde\V_j^+)(t,\eta_2)\vert (\vert \tilde\P_1(t,\eta_2)-\tilde\P_1(t,\eta_1)\vert +\vert \tilde\P_2(t,\eta_2)-\tilde\P_2(t,\eta_1)\vert )\\
 & \leq \A^4\min_j(\tilde\V_j^+)(t,\eta_2)\vert \eta_2-\eta_1\vert.
 \end{align*}
 
(2): If $d(t,\eta_2)>0$, then $\min_j(\tilde\P_j)(t,\eta_2)>0$ and $\min_j(\tilde\V_j^+)(t,\eta_2)>0$ or equivalently
\begin{equation*}
\tilde\P_1(t,\eta_2)>0, \quad \tilde\P_2(t,\eta_2)>0, \quad \tilde \U_1(t,\eta_2)>0, \quad \text{and}\quad \tilde\U_2(t,\eta_2)>0.
\end{equation*}

(2a): Assume that $d(t,\eta_1)=0$ and $\min_j(\tilde\P_j)(t,\eta_1)=0$, then 
\begin{align*}
0& \leq d(t,\eta_2)-d(t,\eta_1)\\
& \leq (\min_j(\tilde\P_j)(t,\eta_2)-\min_j(\tilde\P_j)(t,\eta_1))\min_j(\tilde\V_j^+)(t,\eta_2)\\
& \leq \min_j(\tilde\V_j^+)(t,\eta_2)(\vert \tilde \P_1(t,\eta_2)-\tilde\P_1(t,\eta_1)\vert +\vert \tilde\P_2(t,\eta_2)-\tilde \P_2(t,\eta_1)\vert )\\
&\leq \A^4 \min_j(\tilde\V_j^+)(t,\eta_2)\vert \eta_2-\eta_1\vert.
\end{align*} 

(2b): Assume that $d(t,\eta_1)=0$ and $\min_j(\tilde\V_j^+)(t,\eta_1)=0$, then one either has that 

(i): $\tilde \U_2(t,\eta_1)\leq \tilde \V_2^+(t,\eta_1)=\min_j(\tilde \V_j^+)(t,\eta_1)$, and we can write 
\begin{align*}
0&\leq d(t,\eta_2)-d(t,\eta_1)\\
& \leq \tilde \U_2\min_j(\tilde\P_j)(t,\eta_2)-\tilde\U_2\min_j(\tilde\P_j)(t,\eta_1)\\
& \leq \int_{\eta_1}^{\eta_2} \Big(\tilde \U_{2,\eta} \min_j(\tilde\P_j)+ \tilde \U_2\frac{d}{d\theta} \min_j(\tilde \P_j)\Big) (t,s)ds\\
& \leq \int_{\eta_1}^{\eta_2} \Big(\vert \tilde\U_{2,\eta} \vert \min_j(\tilde \P_j) + \vert \tilde \U_2\vert (\frac{1}{\sqtC{1}}\tilde \P_1\tilde \Y_{1,\eta}+\frac{1}{\sqtC{2}}\tilde \P_2\tilde\Y_{2,\eta} )\Big)(t,s) ds\\
& \leq \frac{\A^4}{\sqrt{2}}\int_{\eta_1}^{\eta_2} \min_j(\tilde\P_j)^{1/2} (t,s) ds+ \A^4\int_{\eta_1}^{\eta_2} \vert \tilde\U_2\vert (t,s) ds\\
& \leq \frac{3}{2\sqrt{2}}\A^6 \vert \eta_2-\eta_1\vert,
\end{align*}
or 

(ii): $\min_j(\tilde\V_j^+)(t,\eta_1)=\tilde \V_1^+(t,\eta_1)$. Then there exists a maximal interval $[\eta_1, a]$ such that $\tilde \V_1^+(t,s)<\tilde \V_2^+(t,s)$ for all $s\in [\eta_1, a)$ and $\tilde \V_1^+(t,a)=\tilde \V_2^+(t,a)$. Moreover, there exists a maximal interval $[b,a]\subset [\eta_1, a]$ such that $\tilde \V_1^+(t,s)>0$ for all $s\in (b,a]$ and $\tilde \V_1^+(t,b)=0$.
Hence we can write
\begin{align*}
0&\leq d(t,\eta_2)-d(t,\eta_1)\\
& \leq d(t,\eta_2)-d(t,a)+d(t,a)-d(t,b)\\
& \leq \tilde \U_2\min_j(\tilde\P_j)(t,\eta_2)-\tilde\U_2\min_j(\tilde \P_j)(t,a)\\
& \qquad +\tilde \U_1\min_j(\tilde\P_j)(t,a)-\tilde \U_1\min_j(\tilde \P_j) (t,b) \\
& \leq \int_{a}^{\eta_2}\Big( \tilde\U_{2,\eta} \min_j(\tilde \P_j) + \tilde\U_2 \frac{d}{d\eta} \min_j(\tilde\P_j)\Big)(t,s) ds\\
& \quad +\int_b^a\Big( \tilde\U_{1,\eta} \min_j(\tilde \P_j) + \tilde\U_1 \frac{d}{d\eta} \min_j(\tilde\P_j)\Big)(t,s) ds\\
& \leq \int_a^{\eta_2} \Big(\frac{\A^4}{\sqrt{2}} \min_j(\tilde \P_j)^{1/2}  + \A^4 \vert \tilde\U_2\vert\Big) (t,s)ds\\
& \qquad + \int_b^a \Big(\frac{\A^4}{\sqrt{2}} \min_j(\tilde\P_j)^{1/2} +\A^4 \vert \tilde \U_1\vert\Big) (t,s) ds\\
& \leq \frac{\A^4}{\sqrt{2}}\int_b^{\eta_2} \min_j(\tilde \P_j)^{1/2}(t,s)ds + \A^4\int_b^{\eta_2} \vert \tilde\U_2\vert (t,s) ds\\
& \leq \frac{\A^4}{\sqrt{2}}\int_{\eta_1}^{\eta_2} \min_j(\tilde \P_j)^{1/2}(t,s)ds + \A^4\int_{\eta_1}^{\eta_2} \vert \tilde\U_2\vert (t,s) ds\\
& \leq\frac{3}{2\sqrt{2}}\A^6\vert \eta_2-\eta_1\vert.
\end{align*}

Note, in the case that $\tilde \U_1^+(t,s)<\tilde \U_2^+(t,s)$ for all $s\in [\eta_1, \eta_2]$, the estimate starts with
\begin{equation*}
0\leq d(t,\eta_2)-d(t,\eta_1)\leq \tilde \U_1\min_j(\tilde\P_j)(t,\eta_2)-\tilde \U_1\min_j(\tilde \P_j) (t,b),
\end{equation*}
where $(b,\eta_2]$ denotes the maximal interval such that $\tilde \U_1^+(t,s)>0$ for all $s\in (b, \eta_2]$.
(2c): Assume that $d(t,\eta_1)>0$, then $\min_j(\tilde \P_j)(t,\eta_1)>0$ and $\min_j(\tilde\V_j^+)(t,\eta_1)=\min_j(\tilde \U_j)(t,\eta_1)>0$. Then one either has that

(i): $\min_j(\tilde \U_j)(t,\eta_1)=\tilde\U_2(t,\eta_1)$, and we have (as before)
\begin{align*}
0&\leq d(t,\eta_2)-d(t,\eta_1)\\
& \leq \tilde \U_2\min_j(\tilde\P_j)(t,\eta_2)-\tilde\U_2\min_j(\tilde\P_j)(t,\eta_1)\\
& \leq\frac{\A^4}{\sqrt{2}} \int_{\eta_1}^{\eta_2}  \min_j(\tilde\P_j)^{1/2}(t,s)ds +\A^4\int_{\eta_1}^{\eta_2}\vert \tilde\U_2\vert (t,s) ds\\
& \leq \frac{3}{2\sqrt{2}}\A^6 \vert \eta_2-\eta_1\vert,
\end{align*}
 or 
 
 (ii): $\min_j(\tilde\U_j)(t,\eta_1)=\tilde\U_1(t,\eta_1)$. Then there exists a maximal interval $[\eta_1, a]$ such that $\tilde \V_1^+(t,s)<\tilde \V_2^+(t,s)$ for all $s\in[\eta_1, a)$ and $\tilde \V_1^+(t,a)=\tilde\V_2^+(t,a)$. Moreover, there exists a maximal interval $[b,a]\subset [\eta_1, a]$ such that $\tilde \V_1^+(t,s)>0$ for all $s\in (b,a]$ and $\tilde \V_1^+(t,b)=0$. Hence we can write
 \begin{align*}
 0&\leq d(t,\eta_2)-d(t,\eta_1)\\
 & \leq d(t,\eta_2)-d(t,a)+d(t,a)-d(t,b)\\
 & \leq \tilde \U_2\min_j(\tilde\P_j)(t,\eta_2)-\tilde\U_2\min_j(\tilde \P_j)(t,a)\\
& \qquad +\tilde \U_1\min_j(\tilde\P_j)(t,a)-\tilde \U_1\min_j(\tilde \P_j) (t,b) \\
& \leq \int_{a}^{\eta_2}\Big( \tilde\U_{2,\eta} \min_j(\tilde \P_j) + \tilde\U_2 \frac{d}{d\eta} \min_j(\tilde\P_j)\Big)(t,s) ds\\
& \quad +\int_b^a \Big(\tilde\U_{1,\eta} \min_j(\tilde \P_j) + \tilde\U_1 \frac{d}{d\eta} \min_j(\tilde\P_j)\Big)(t,s) ds\\
& \leq \int_a^{\eta_2} \Big(\frac{\A^4}{\sqrt{2}} \min_j(\tilde \P_j)^{1/2}  + \A^4 \vert \tilde\U_2\vert\Big)(t,s)ds\\
& \qquad + \int_b^a \Big(\frac{\A^4}{\sqrt{2}} \min_j(\tilde\P_j)^{1/2} +\A^4 \vert \tilde\U_1\vert\Big) (t,s) ds\\
& \leq \frac{\A^4}{\sqrt{2}}\int_b^{\eta_2} \min_j(\tilde \P_j)^{1/2}(t,s)ds + \A^4\int_b^{\eta_2} \vert \tilde\U_2\vert (t,s) ds\\
& \leq \frac{\A^4}{\sqrt{2}}\int_{\eta_1}^{\eta_2} \min_j(\tilde \P_j)^{1/2}(t,s)ds + \A^4\int_{\eta_1}^{\eta_2} \vert \tilde\U_2\vert (t,s) ds\\
& \leq \frac{3}{2\sqrt{2}}\A^6 \vert \eta_2-\eta_1\vert.
\end{align*}

Note, in the case that $\tilde \U_1^+(t,s)<\tilde \U_2^+(t,s)$ for all $s\in [\eta_1, \eta_2]$, the estimate starts with
\begin{equation*}
0\leq d(t,\eta_2)-d(t,\eta_1)\leq \tilde \U_1\min_j(\tilde\P_j)(t,\eta_2)-\tilde \U_1\min_j(\tilde \P_j) (t,b),
\end{equation*}
where $(b,\eta_2]$ denotes the maximal interval such that $\tilde \U_1^+(t,s)>0$ for all $s\in (b, \eta_2]$.

Thus we showed that 
\begin{equation*}
\vert d(t,\eta_2)-d(t,\eta_1)\vert \leq\frac{3}{2\sqrt{2}}\A^6\vert \eta_2-\eta_1\vert,
\end{equation*}
or, in other words,
$d(t,\dott)= \min_j(\tilde \P_j)\min_j(\tilde\V_j^+)(t,\dott)$ is Lipschitz continuous with Lipschitz constant $\frac{3}{2\sqrt{2}}\A^6$, which is independent of time, and thus differentiable almost everywhere. 
Moreover a closer look reveals that 
\begin{align*}
\vert d(t,\eta_2)-d(t,\eta_1)\vert&  \leq \A^4\min_j(\tilde\V_j^+)(t,\eta_2)\vert \eta_2-\eta_1\vert \\ 
& \qquad+ \frac{\A^4}{\sqrt{2}}\vert \int_{\eta_1}^{\eta_2} \min_j(\tilde \P_j)^{1/2}(t,s) ds\vert + \A^4 \vert \int_{\eta_1}^{\eta_2} \vert \tilde \U_2\vert (t,s) ds\vert. 
\end{align*}
Since both $\vert \tilde\U_2\vert (t,\dott)$ and $\min_j(\tilde\P_j)^{1/2}(t,\dott)$ are continuous, the fundamental theorem of calculus implies that 
\begin{align*}
\left\vert \frac{d(t,\eta_2)-d(t,\eta_1)}{\eta_2-\eta_1}\right\vert& \leq\A^4\min_j(\tilde\V_j^+)(t,\eta_2)\\
& \qquad + \frac{\A^4}{\sqrt{2}}\min_j(\tilde\P_j)^{1/2} (t,\tilde\eta)+ \A^4\vert \tilde \U_2\vert (t,\tilde\eta)
\end{align*} 
for some $\tilde \eta$ between $\eta_1$ and $\eta_2$. Letting $\eta_2\to \eta_1$ we thus obtain for almost every $\eta$ that 
\begin{align*}
\left\vert \frac{d}{d\eta} \min_j(\tilde \P_j) \min_j(\tilde\V_j^+)(t,\eta)\right\vert &=\vert \frac{d}{d\eta} d(t,\eta) \vert \\
& \leq 2\A^4( \min_j(\tilde \P_j)^{1/2} +\vert \tilde\U_2\vert)(t,\eta).
\end{align*}

(ii) We have that
\begin{align*}\nonumber
\vert \min_j(\tilde\P_j)(t,\eta)& -\min_j(\tilde\P_j)(t,\tilde \eta)\vert \\ \nn
& \leq \max \big(\vert\tilde \P_1(t,\eta)-\tilde\P_1(t,\tilde\eta)\vert , \vert \tilde\P_2(t,\eta)- \tilde\P_2(t,\tilde\eta)\vert \big)\\% \label{der:estminP}
& \leq \frac{\A^4}{2}\vert \eta-\tilde\eta\vert.
\end{align*}
Thus, for almost every $\eta$,
\begin{align*}
\vert\frac{d}{d\eta}(\min_j(\tilde\P_j)\tilde\U_k)(t,\eta)\vert&= \vert\big(\frac{d}{d\eta}\min_j(\tilde\P_j)\big)\tilde\U_k(t,\eta)+\min_j(\tilde\P_j)\tilde\U_{k,\eta}(t,\eta)\vert\\
&\le \frac{\A^4}{2}\norm{\tilde\U_k}_\infty+\norm{\tilde\P_k\tilde\U_{k,\eta}}_\infty\\
&\le \frac{1}{\sqrt{2}}\A^6.
\end{align*}
\end{proof}
%-------------------
%----------- 
\begin{lemma}\label{lemma:5}
(i) The function $\eta\mapsto\min_j(\tilde\D_j)\min_j(\tilde \V_j^+)(t,\eta)$ is Lipschitz continuous with a uniformly bounded Lipschitz constant and thus differentiable almost everywhere with
\begin{equation*}
\vert \frac{d}{d\eta} (\min_j(\tilde\D_j)\min_j(\tilde\V_j^+))(t,\eta)\vert\leq \bigO(1)\sqrt{\A}\A^4 (\min_j(\tilde\D_j)^{1/2}+\vert \tilde\U_k\vert)(t,\eta), \quad k=1,2.
\end{equation*}
(ii) The function $\eta\mapsto\min_j(\tilde\D_j)\tilde\U_k(t,\eta)$, $k=1,2$ Lipschitz continuous with a uniformly bounded Lipschitz constant and thus differentiable almost everywhere with 
\begin{equation*}
\vert \frac{d}{d\eta} (\min_j(\tilde \D_j)\tilde\U_k)(t,\eta)\vert\leq \bigO(1)\A^7.
\end{equation*}
\end{lemma}

\begin{proof}
(i) We only present the proof for the case $k=2$, since the case $k=1$ is similar. 

Given $0\leq \eta_1<\eta_2\leq 1$, we assume without loss of generality that
\begin{equation*}
0\leq \min_j(\tilde\D_j)\min_j(\tilde\V_j^+)(t,\eta_2)-\min_j(\tilde\D_j)\min_j(\tilde\V_j^+)(t,\eta_1).
\end{equation*}
To ease the notation we introduce the function
\begin{equation*}
\bar d(t,\eta)=\min_j(\tilde\D_j)\min_j(\tilde\V_j^+)(t,\eta).
\end{equation*}
We will distinguish several cases:

(1): If $\bar d(t,\eta_2)=0$, then $\bar d(t,\eta_1)=0$ and one has 
\begin{align*}
0&\leq \min_j(\tilde\V_j^+)(t,\eta_2)\vert \min_j(\tilde\D_j)(t,\eta_2)-\min_j(\tilde \D_j)(t,\eta_1)\vert )\\
 & \leq \min_j(\tilde\V_j^+)(t,\eta_2)\vert (\vert \tilde\D_1(t,\eta_2)-\tilde\D_1(t,\eta_1)\vert +\vert \tilde\D_2(t,\eta_2)-\tilde\D_2(t,\eta_1)\vert )\\
 & \leq \bigO(1)\A^5\min_j(\tilde\V_j^+)(t,\eta_2)\vert \eta_2-\eta_1\vert
 \end{align*}

(2): If $\bar d(t,\eta_2)>0$, then $\min_j(\tilde\D_j)(t,\eta_2)>0$ and $\min_j(\tilde\V_j^+)(t,\eta_2)>0$ or equivalently
\begin{equation*}
\tilde\D_1(t,\eta_2)>0, \quad \tilde\D_2(t,\eta_2)>0, \quad \tilde \U_1(t,\eta_2)>0, \quad \text{and}\quad \tilde\U_2(t,\eta_2)>0.
\end{equation*}

(2a): Assume that $\bar d(t,\eta_1)=0$ and $\min_j(\tilde \D_j)(t,\eta_1)=0$, then 
\begin{align*}
0& \leq \bar d(t,\eta_2)-\bar d(t,\eta_1)\\
& \leq (\min_j(\tilde\D_j)(t,\eta_2)-\min_j(\tilde\D_j)(t,\eta_1))\min_j(\tilde\V_j^+)(t,\eta_2)\\
& \leq \min_j(\tilde\V_j^+)(t,\eta_2)(\vert \tilde \D_1(t,\eta_2)-\tilde\D_1(t,\eta_1)\vert +\vert \tilde\D_2(t,\eta_2)-\tilde \D_2(t,\eta_1)\vert )\\
&\leq \bigO(1)\A^5 \min_j(\tilde\V_j^+)(t,\eta_2)\vert \eta_2-\eta_1\vert.
\end{align*} 

(2b): Assume that $\bar d(t,\eta_1)=0$ and $\min_j(\tilde\V_j^+)(t,\eta_1)=0$, then one either has that 

(i): $\tilde \U_2(t,\eta_1)\leq \tilde \V_2^+(t,\eta_1)=\min_j(\tilde \V_j^+)(t,\eta_1)$ and we can write 
\begin{align*}
0&\leq d(t,\eta_2)-d(t,\eta_1)\\
& \leq \tilde \U_2\min_j(\tilde\D_j)(t,\eta_2)-\tilde\U_2\min_j(\tilde\D_j)(t,\eta_1)\\
& \leq \int_{\eta_1}^{\eta_2} \tilde \U_{2,\eta} \min_j(\tilde\D_j)(t,s)+ \tilde \U_2\frac{d}{d\theta} \min_j(\tilde\D_j)(t,s) ds\\
& \leq \int_{\eta_1}^{\eta_2} \vert \tilde \U_{2,\eta} \vert \min_j(\tilde\D_j)(t,s) ds\\
& \quad +\int_{\eta_1}^{\eta_2} \vert \tilde \U_2\vert \Big(\vert (\tilde\U_1^2-\tilde\P_1)\tilde\Y_{1,\eta}-\frac{1}{\sqtC{1}}\tilde\D_1\tilde\Y_{1,\eta}+\frac12 \sqtC{1}^5\vert \\
&\qquad\qquad\qquad\qquad+\vert (\tilde\U_2^2-\tilde\P_2)\tilde\Y_{2,\eta}-\frac{1}{\sqtC{2}}\tilde\D_2\tilde\Y_{2,\eta}+\frac12 \sqtC{2}^5\vert \Big)(t,s) ds\\
& \leq \sqrt{\A}\A^4\int_{\eta_1}^{\eta_2} \min_j(\tilde\D_j)^{1/2} (t,s) ds+ \bigO(1)\A^5 \int_{\eta_1}^{\eta_2} \vert \tilde\U_2\vert (t,s) ds\\
& \leq \bigO(1)\A^7 \vert \eta_2-\eta_1\vert
\end{align*}
or 

(ii): $\min_j(\tilde\V_j^+)(t,\eta_1)=\tilde \V_1^+(t,\eta_1)$. Then there exists a maximal interval $[\eta_1, a]$ such that $\tilde \V_1^+(t,s)<\tilde \V_2^+(t,s)$ for all $s\in [\eta_1, a)$ and $\tilde \V_1^+(t,a)=\tilde \V_2^+(t,a)$. Moreover, there exists a maximal interval $[b,a]\subset [\eta_1, a]$ such that $\tilde \V_1^+(t,s)>0$ for all $s\in (b,a]$ and $\tilde \V_1^+(t,b)=0$.
Hence we can write
\begin{align*}
0&\leq \bar d(t,\eta_2)-\bar d(t,\eta_1)\\
& \leq \bar d(t,\eta_2)-\bar d(t,a)+\bar d(t,a)-\bar d(t,b)\\
& \leq \tilde \U_2\min_j(\tilde\D_j)(t,\eta_2)-\tilde\U_2\min_j(\tilde\D_j)(t,a)\\
& \qquad +\tilde \U_1\min_j(\tilde\D_j)(t,a)-\tilde \U_1\min_j(\tilde\D_j) (t,b) \\
& \leq \int_{a}^{\eta_2}\Big( \tilde\U_{2,\eta} \min_j(\tilde\D_j) + \tilde\U_2 \frac{d}{d\eta} \min_j(\tilde\D_j)\Big)(t,s) ds\\
& \quad +\int_b^a \Big(\tilde\U_{1,\eta} \min_j(\tilde\D_j) + \tilde\U_1 \frac{d}{d\eta} \min_j(\tilde\D_j)\Big)(t,s) ds\\
& \leq \int_a^{\eta_2} \Big(\sqrt{\A}\A^4 \min_j(\tilde\D_j)^{1/2}  + \bigO(1)\A^5 \vert \tilde\U_2\vert\Big) (t,s)ds\\
& \qquad + \int_b^a \Big(\sqrt{\A} \A^4 \min_j(\tilde \D_j)^{1/2}+\bigO(1)\A^5 \vert \tilde \U_1\vert\Big) (t,s) ds\\
& \leq \sqrt{\A}\A^4\int_b^{\eta_2} \min_j(\tilde\D_j)^{1/2}(t,s)ds + \bigO(1)\A^5\int_b^{\eta_2} \vert \tilde\U_2\vert (t,s) ds\\
& \leq \sqrt{\A}\A^4\int_{\eta_1}^{\eta_2} \min_j(\tilde\D_j)^{1/2}(t,s)ds + \bigO(1)\A^5\int_{\eta_1}^{\eta_2} \vert \tilde\U_2\vert (t,s) ds\\
& \leq \bigO(1)\A^7 \vert \eta_2-\eta_1\vert.
\end{align*}

Note, in the case that $\tilde \U_1^+(t,s)<\tilde \U_2^+(t,s)$ for all $s\in [\eta_1, \eta_2]$, the estimate starts with
\begin{equation*}
0\leq d(t,\eta_2)-d(t,\eta_1)\leq \tilde \U_1\min_j(\tilde\D_j)(t,\eta_2)-\tilde \U_1\min_j(\tilde \D_j) (t,b),
\end{equation*}
where $(b,\eta_2]$ denotes the maximal interval such that $\tilde \U_1^+(t,s)>0$ for all $s\in (b, \eta_2]$.

(2c): Assume that $\bar d(t,\eta_1)>0$, then $\min_j(\tilde\D_j)(t,\eta_1)>0$ and $\min_j(\tilde\V_j^+)(t,\eta_1)=\min_j(\tilde \U_j)(t,\eta_1)>0$. Then one either has that

(i): $\min_j(\tilde \U_j)(t,\eta_1)=\tilde\U_2(t,\eta_1)$, and we have (as before)
\begin{align*}
0&\leq \bar d(t,\eta_2)-\bar d(t,\eta_1)\\
& \leq \tilde \U_2\min_j(\tilde \D_j)(t,\eta_2)-\tilde\U_2\min_j(\tilde\D_j)(t,\eta_1)\\
& \leq \int_{\eta_1}^{\eta_2}\Big( \frac12 \sqrt{\A}\A^6 \min_j(\tilde \D_j)^{1/2} 
+ \bigO(1)\A^5 \vert \tilde\U_2\vert \Big)(t,s) ds\\
& \leq \bigO(1)\A^7 \vert \eta_2-\eta_1\vert,
\end{align*}
 or 
 
 (ii): $\min_j(\tilde\U_j)(t,\eta_1)=\tilde\U_1(t,\eta_1)$. Then there exists a maximal interval $[\eta_1, a]$ such that $\tilde \V_1^+(t,s)<\tilde \V_2^+(t,s)$ for all $s\in[\eta_1, a)$ and $\tilde \V_1^+(t,a)=\tilde\V_2^+(t,a)$. Moreover, there exists a maximal interval $[b,a]\subset [\eta_1, a]$ such that $\tilde \V_1^+(t,s)>0$ for all $s\in (b,a]$ and $\tilde \V_1^+(t,b)=0$. Hence we can write
 \begin{align*}
 0&\leq \bar d(t,\eta_2)-\bar d(t,\eta_1)\\
 & \leq \bar d(t,\eta_2)-\bar d(t,a)+\bar d(t,a)-\bar d(t,b)\\
 & \leq \tilde \U_2\min_j(\tilde\D_j)(t,\eta_2)-\tilde\U_2\min_j(\tilde\D_j)(t,a)\\
& \qquad +\tilde \U_1\min_j(\tilde\D_j)(t,a)-\tilde \U_1\min_j(\tilde\D_j) (t,b) \\
& \leq \int_{a}^{\eta_2}\Big( \tilde\U_{2,\eta} \min_j(\tilde\D_j) + \tilde\U_2 \frac{d}{d\eta} \min_j(\tilde\D_j)\Big)(t,s) ds\\
& \quad +\int_b^a \Big(\tilde\U_{1,\eta} \min_j(\tilde\D_j) + \tilde\U_1 \frac{d}{d\eta} \min_j(\tilde\D_j)\Big)(t,s) ds\\
& \leq \int_a^{\eta_2} \Big(\frac12 \sqrt{\A}\A^6 \min_j(\tilde\D_j)^{1/2} + \bigO(1)\A^5 \vert \tilde\U_2\vert\Big) (t,s)ds\\
& \qquad + \int_b^a \Big(\frac12 \sqrt{\A}\A^6 \min_j(\tilde \D_j)^{1/2}+\bigO(1)\A^5 \vert \tilde \U_1\vert\Big) (t,s) ds\\
& \leq \sqrt{\A}\A^4\int_b^{\eta_2} \min_j(\tilde\D_j)^{1/2}(t,s)ds + \bigO(1)\A^5\int_b^{\eta_2} \vert \tilde\U_2\vert (t,s) ds\\
& \leq  \sqrt{\A} \A^4\int_{\eta_1}^{\eta_2} \min_j(\tilde\D_j)^{1/2}(t,s)ds + \bigO(1)\A^5\int_{\eta_1}^{\eta_2} \vert \tilde\U_2\vert (t,s) ds\\
& \leq \bigO(1)\A^7 \vert \eta_2-\eta_1\vert.
\end{align*}

Note, in the case that $\tilde \U_1^+(t,s)<\tilde \U_2^+(t,s)$ for all $s\in [\eta_1, \eta_2]$, the estimate starts with
\begin{equation*}
0\leq d(t,\eta_2)-d(t,\eta_1)\leq \tilde \U_1\min_j(\tilde\D_j)(t,\eta_2)-\tilde \U_1\min_j(\tilde \D_j) (t,b),
\end{equation*}
where $(b,\eta_2]$ denotes the maximal interval such that $\tilde \U_1^+(t,s)>0$ for all $s\in (b, \eta_2]$.

Thus we showed that 
\begin{equation*}
\vert \bar d(t,\eta_2)-\bar d(t,\eta_1)\vert \leq \bigO(1)\A^7\vert \eta_2-\eta_1\vert,
\end{equation*}
or in other words
$\bar d(t,\dott)= \min_j(\tilde\D_j)\min_j(\tilde\V_j^+)(t,\dott)$ is Lipschitz continuous with Lipschitz constant $\bigO(1)\A^7$, which is independent of time, and thus differentiable almost everywhere. 
Moreover a close look reveals that 
\begin{align*}
\vert \bar d(t,\eta_2)-\bar d(t,\eta_1)\vert&  \leq \bigO(1)\A^5\min_j(\tilde\V_j^+)(t,\eta_2)\vert \eta_2-\eta_1\vert \\ 
& \quad+ \sqrt{\A}\A^4\vert \int_{\eta_1}^{\eta_2} \min_j(\tilde\D_j)^{1/2}(t,s) ds\vert + \bigO(1)\A^5 \vert \int_{\eta_1}^{\eta_2} \vert \tilde \U_2\vert (t,s) ds\vert 
\end{align*}
Since both $\vert \tilde\U_2\vert (t,\dott)$ and $\min_j(\tilde\D_j)^{1/2}(t,\dott)$ are continuous, the fundamental theorem of calculus implies that 
\begin{align*}
\left\vert \frac{\bar d(t,\eta_2)-\bar d(t,\eta_1)}{\eta_2-\eta_1}\right\vert& \leq \bigO(1)\A^5\min_j(\tilde\V_j^+)(t,\eta_2)\\
& \quad + \sqrt{\A}\A^4\min_j(\tilde\D_j)^{1/2} (t,\tilde\eta)+ \bigO(1)\A^5\vert \tilde \U_2\vert (t,\tilde\eta)
\end{align*} 
for some $\tilde \eta$ between $\eta_1$ and $\eta_2$. Letting $\eta_2\to \eta_1$ we thus obtain for almost every $\eta$ that 
\begin{align*}
\left\vert \frac{d}{d\eta} \min_j(\tilde\D_j) \min_j(\tilde\V_j^+)(t,\eta)\right\vert &=\vert \frac{d}{d\eta} \bar d(t,\eta) \vert \\
& \leq \bigO(1)\sqrt{\A}\A^4( \min_j(\tilde\D_j)^{1/2} +\vert \tilde\U_2\vert )(t,\eta).
\end{align*}

\bigskip
(ii) We have that 
\begin{align*}
\vert \min_j(\tilde\D_j)(t,\eta)&-\min_j(\tilde \D_j)(t,\tilde\eta)\vert \\
&\qquad \leq \max(\vert \tilde \D_1(t,\eta)-\tilde\D_1(t,\tilde\eta)\vert +\vert \tilde \D_2(t,\eta)-\tilde\D_2(t,\tilde\eta)\vert )\\
&\qquad \leq \bigO(1)\A^5 \vert \eta-\tilde\eta\vert, 
\end{align*}
and hence, for almost every $\eta$,
\begin{equation*} *\label{eq:DUderiv}
\vert \frac{d}{d\eta} \min_j(\tilde\D_j)(t,\eta)\vert \leq \bigO(1)\A^5.
\end{equation*}
This implies, for almost every $\eta$,
\begin{align*}
\vert \frac{d}{d\eta} (\min_j(\tilde \D_j)\tilde \U_k)(t,\eta)\vert &= \vert (\frac{d}{d\eta} \min_j(\tilde\D_j))\tilde\U_k(t,\eta)+ \min_j(\tilde\D_j)\tilde\U_{k,\eta}(t,\eta) \vert \\
& \leq \bigO(1)\A^7 + 2\A\norm{ \tilde\P_k \vert \tilde\U_{k,\eta}\vert} \\
& \leq \bigO(1)\A^7.
\end{align*}
\end{proof}
%--------------------------------
%-------------------------
\begin{lemma} \label{lemma:6}
The function  $\eta\mapsto\ma\int_0^\eta \min_j(e^{-\frac{1}{\ma}(\tilde  \Y_j(t,\eta)-\tilde\Y_j(t,\theta))})\min_j(\tilde\V_j^+)(t,\theta) d\theta$ is Lipschitz continuous with a uniformly bounded Lipschitz constant and thus differentiable almost everywhere with 
\begin{equation*}
\vert \frac{d}{d\eta}\Big( \ma\int_0^\eta \min_j(e^{-\frac{1}{\ma}(\tilde  \Y_j(t,\eta)-\tilde\Y_j(t,\theta))})\min_j(\tilde\V_j^+)(t,\theta) d\theta\Big)\vert \leq \bigO(1)\A^2.
\end{equation*}
\end{lemma}

\begin{proof}
To prove the existence and boundedness of the derivative, we will prove Lipschitz continuity. Let $0\leq \eta_1<\eta_2\leq1$, then 
\begin{align*}
\ma\vert \int_0^{\eta_1} &\min_j(e^{-\frac{1}{\ma}(\tilde\Y_j(t,\eta_1)-\tilde\Y_j(t,\theta))})\min_j(\tilde\V_j^+)(t,\theta) d\theta\\
& \qquad \qquad- \int_0^{\eta_2} \min_j(e^{-\frac{1}{\ma}(\tilde\Y_j(t,\eta_2)-\tilde\Y_j(t,\theta))})\min_j(\tilde\V_j^+)(t,\theta)d\theta\vert \\
& \leq \ma\int_{\eta_1}^{\eta_2} \min_j(e^{-\frac{1}{\ma}(\tilde\Y_j(t,\eta_2)-\tilde\Y_j(t,\theta))})\min_j(\tilde\V_j^+)(t,\theta) d\theta\\
& \quad + \ma\int_0^{\eta_1} \Big(\min_j(e^{-\frac{1}{\ma}(\tilde\Y_j(t,\eta_1)-\tilde\Y_j(t,\theta))})\\
&\qquad\qquad\qquad-\min_j(e^{-\frac{1}{\ma}(\tilde\Y_j(t,\eta_2)-\tilde\Y_j(t,\theta))})\Big)\min_j(\tilde\V_j^+)(t,\theta) d\theta\\
& \leq \frac{\ma\A^2}{\sqrt{2}}\vert \eta_2-\eta_1\vert\\
& \quad +\ma\int_0^{\eta_1} (e^{-\frac{1}{\ma}(\tilde\Y_1(t,\eta_1)-\tilde\Y_1(t,\theta))}-e^{-\frac{1}{\ma}(\tilde\Y_1(t,\eta_2)-\tilde\Y_1(t,\theta))}) \min_j(\tilde\V_j^+)(t,\theta) d\theta\\
& \quad + \ma\int_0^{\eta_1} (e^{-\frac{1}{\ma}(\tilde \Y_2(t,\eta_1)-\tilde\Y_2(t,\theta))}-e^{-\frac{1}{\ma}(\tilde\Y_2(t,\eta_2)-\tilde\Y_2(t,\theta))} \min_j(\tilde\V_j^+)(t,\theta) d\theta\\
& \leq \frac{\ma\A^2}{\sqrt{2}} \vert \eta_2-\eta_1\vert \\
& \quad +\int_0^{\eta_1} \int_{\eta_1}^{\eta_2} e^{-\frac{1}{\ma}(\tilde\Y_1(t,s)-\tilde\Y_1(t,\theta))} \tilde\Y_{1,\eta}(t,s) ds\min_j(\tilde\V_j^+)(t,\theta) d\theta\\
& \quad +\int_0^{\eta_1}\int_{\eta_1}^{\eta_2} e^{-\frac{1}{\ma}(\tilde\Y_2(t,s)-\tilde\Y_2(t,\theta))}\tilde\Y_{2,\eta}(t,s) ds\min_j(\tilde\V_j^+)(t,\theta) d\theta\\
& \leq \frac{\ma\A^2}{\sqrt{2}}\vert \eta_2-\eta_1\vert \\
& \quad +\int_{\eta_1}^{\eta_2} \Big(\int_0^s e^{-\frac{1}{\sqtC{1}}(\tilde\Y_1(t,s)-\tilde\Y_1(t,\theta))} \tilde\V_1^+(t,\theta)d\theta\Big) \tilde\Y_{1,\eta}(t,s)ds\\
& \quad +\int_{\eta_1}^{\eta_2} \Big(\int_0^s e^{-\frac{1}{\sqtC{2}}(\tilde \Y_2(t,s)-\tilde\Y_2(t,\theta))} \tilde\V_2^+(t,\theta) d\theta\Big)\tilde\Y_{2,\eta}(t,s) ds\\
& \leq \frac{\ma\A^2}{\sqrt{2}} \vert \eta_2-\eta_1\vert  +\frac{1}{\sqtC{1}^5}\int_{\eta_1}^{\eta_2} \Big( \int_0^s e^{-\frac{1}{\sqtC{1}}(\tilde\Y_1(t,s)-\tilde\Y_1(t,\theta))} \\
&\qquad\times\big(\frac{1}{\sqtC{1}}\tilde\P_1^2\tilde\Y_{1,\eta}+\sqtC{1}\tilde\U_1^2\tilde\Y_{1,\eta}+\tilde\V_1^+\tilde\Henergy_{1,\eta}\big)(t,\theta) d\theta\Big)
\tilde\Y_{1,\eta}(t,s)ds\\
& \quad +\frac{1}{\sqtC{2}^5} \int_{\eta_1}^{\eta_2} \Big(\int_0^s e^{-\frac{1}{\sqtC{2}}(\tilde\Y_2(t,s)-\tilde\Y_2(t,\theta)}\\
&\qquad\times\big(\frac{1}{\sqtC{2}}\tilde \P_2^2\tilde\Y_{2,\eta}+\sqtC{2}\tilde\U_2^2\tilde\Y_{2,\eta}+\tilde\V_2^+\tilde\Henergy_{2,\eta}\big)(t,\theta) d\theta\Big) 
\tilde\Y_{2,\eta}(t,s)ds\\
& \leq \frac{\ma\A^2}{\sqrt{2}} \vert \eta_2-\eta_1\vert +\bigO(1)\frac{1}{\sqtC{1}^3} \int_{\eta_1}^{\eta_2} \tilde \P_1\tilde\Y_{1,\eta}(t,s)ds 
 +\bigO(1)\frac{1}{\sqtC{2}^3}\int_{\eta_1}^{\eta_2}\tilde\P_2\tilde\Y_{2,\eta}(t,s) ds\\
& \leq \bigO(1)\A^2 \vert \eta_2-\eta_1\vert.
\end{align*}
\end{proof}
%---------------------------

%-------------------
\begin{lemma}\label{lemma:2}
The function 
\[
\eta\mapsto\min_k\Big( \int_0^\eta \min_j(e^{-\frac{1}{\ma}(\tilde \Y_j(t,\eta)-\tilde\Y_j(t,\theta))})\min_j(\tilde\P_j)
\min_j(\tilde\V_j^+)\tilde \Y_{k,\eta}(t,\theta) d\theta\Big)
\]
is Lipschitz continuous with a uniformly bounded Lipschitz constant and thus differentiable almost everywhere. The derivative satisfies,
\begin{align}\label{eq:MinMax2}
&\vert\frac{d}{d\eta} \min_k\Big( \int_0^\eta \min_j(e^{-\frac{1}{\ma}(\tilde \Y_j(t,\eta)-\tilde\Y_j(t,\theta))})\min_j(\tilde\P_j)
\min_j(\tilde\V_j^+)\tilde \Y_{k,\eta}(t,\theta) d\theta\Big)\vert \leq \bigO(1)\A^7.
\end{align} 
\end{lemma}
\begin{proof}
Introduce   
\begin{align*}
\tilde a(t,\eta)& = \min_k\Big( \int_0^\eta \min_j(e^{-\frac{1}{\ma}(\tilde \Y_j(t,\eta)-\tilde\Y_j(t,\theta))})\min_j(\tilde\P_j)
\min_j(\tilde\V_j^+)\tilde \Y_{k,\eta}(t,\theta) d\theta\Big) \\ 
\tilde b(t,\eta)& = \int_0^\eta \min_j(e^{-\frac{1}{\ma}(\tilde \Y_j(t,\eta)-\tilde\Y_j(t,\theta))})\min_j(\tilde\P_j)
\min_j(\tilde\V_j^+)\tilde \Y_{1,\eta} (t,\theta) d\theta\\
\tilde c(t,\eta)& = \int_0^\eta \min_j(e^{-\frac{1}{\ma}(\tilde \Y_j(t,\eta)-\tilde\Y_j(t,\theta))})\min_j(\tilde\P_j)
\min_j(\tilde\V_j^+)\tilde \Y_{2,\eta} (t,\theta) d\theta.
\end{align*}
Thus 
\begin{equation*} 
\tilde a(t,\eta)=\min(\tilde b, \tilde c)(t,\eta).
\end{equation*}
Then we have to show that $\tilde a(t, \dott)$ is Lipschitz continuous with a Lipschitz constant, which only depends on $\A$. Clearly we have that 
\begin{equation*}
\vert \tilde a(t,\eta_1)-\tilde a(t,\eta_2)\vert \leq \max(\vert \tilde b(t,\eta_1)-\tilde b(t,\eta_2)\vert, \vert \tilde c(t,\eta_1)-\tilde c(t,\eta_2)\vert ),
\end{equation*}
and it suffices to show that both $\tilde b(t,\dott)$ and $\tilde c(t,\dott)$ are Lipschitz continuous with a Lipschitz constant, which only depends on $\A$. We are only going to establish the Lipschitz continuity for $\tilde b$, since the argument for $\tilde c$ follows the same lines. 

Let $0\leq \eta_1<\eta_2\leq 1$. Then we have to consider two cases
\begin{equation*}
0\leq \tilde b(t,\eta_1)-\tilde b(t,\eta_2)\quad \text{ and }\quad 0\leq \tilde b(t,\eta_2)-\tilde b(t,\eta_1).
\end{equation*}

(i): $0\leq \tilde b(t,\eta_2)-\tilde b(t,\eta_1)$: By defintion we have 
\begin{align*}
\tilde b(t,\eta_2)-\tilde b(t,\eta_1)& = \int_0^{\eta_2} \min_j(e^{-\frac{1}{\ma}(\tilde \Y_j(t,\eta_2)-\tilde \Y_j(t,\theta))})\min_j(\tilde\P_j)\\ 
& \qquad \qquad \qquad \qquad \qquad \qquad\qquad\times\min_j(\tilde \V_j^+)\tilde \Y_{1,\eta} (t,\theta) d\theta\\
& \qquad - \int_0^{\eta_1} \min_j(e^{-\frac{1}{\ma}(\tilde \Y_j(t,\eta_1)-\tilde \Y_j(t,\theta))})\min_j(\tilde\P_j) \\ 
& \qquad \qquad \qquad \qquad \qquad \qquad\qquad\times\min_j(\tilde\V_j^+)\tilde \Y_{1,\eta} (t,\theta) d\theta\\
& = \int_{\eta_1}^{\eta_2} \min_j(e^{-\frac{1}{\ma}(\tilde \Y_j(t,\eta_2)-\tilde \Y_j(t,\theta))})\min_j(\tilde\P_j)\\ 
& \qquad \qquad \qquad \qquad \qquad \qquad\qquad\times\min_j(\tilde \V_j^+)\tilde \Y_{1,\eta} (t,\theta) d\theta\\
& \qquad + \int_0^{\eta_1} (\min_j(e^{-\frac{1}{\ma}(\tilde \Y_j(t,\eta_2)-\tilde \Y_j(t,\theta))})\\ 
& \qquad \qquad \qquad \qquad  -\min_j(e^{-\frac{1}{\ma}(\tilde \Y_j(t,\eta_1)-\tilde \Y_j(t,\theta))}))\\ 
& \qquad \qquad \qquad \qquad\qquad \times\min_j(\tilde\P_j)\min_j(\tilde \V_j^+)\tilde \Y_{1,\eta} (t,\theta) d\theta\\
& \leq  \int_{\eta_1}^{\eta_2} \min_j(e^{-\frac{1}{\ma}(\tilde \Y_j(t,\eta_2)-\tilde \Y_j(t,\theta))})\min_j(\tilde\P_j)\\ 
& \qquad \qquad \qquad \qquad \qquad \qquad \times\min_j(\tilde \V_j^+)\tilde \Y_{1,\eta} (t,\theta) d\theta,
\end{align*}
where we used in the last step that $0\leq \min_j(\tilde\P_j) \min_j(\tilde \V_j^+)\tilde \Y_{1,\eta} (t,\eta)$ and that $\tilde \Y_1(t,\eta)$ is increasing, which implies 
\begin{equation*}
\min_j(e^{-\frac{1}{\ma}(\tilde \Y_j(t,\eta_1)-\tilde \Y_j(t,\theta))})\geq \min_j(e^{-\frac{1}{\ma}(\tilde \Y_j(t,\eta_2)-\tilde \Y_j(t,\theta))}). 
\end{equation*}
Moreover, note that $\min_j(e^{-\frac{1}{\ma}(\tilde \Y_j(t,\eta_2)-\tilde \Y_j(t,\theta))})\leq 1$, for $0\leq \eta_1\leq \theta\leq\eta_2$ and that 
\begin{equation*}
0\leq 2\sqrt{2}\min_j(\tilde\P_j)\min_j(\tilde \V_j^+)\tilde \Y_{1,\eta} (t, \theta)\leq \A_1^7\leq \A^7, 
\end{equation*}
by \eqref{eq:all_estimatesB} and \eqref{eq:all_estimatesE}.
Thus 
\begin{align*}
0& \leq \tilde b(t,\eta_2)-\tilde b(t,\eta_1)\\
& \leq \int_{\eta_1}^{\eta_2} \min_j(e^{-\frac{1}{\ma}(\tilde \Y_j(t,\eta_2)-\tilde \Y_j(t,\theta))}) \min_j(\tilde\P_j) \min_j(\tilde \V_j^+)\tilde \Y_{1,\eta} (t,\theta) d\theta\\
& \leq \frac{\A^7}{2\sqrt{2}} \vert \eta_2-\eta_1\vert.
\end{align*}

(ii): $0\leq \tilde b(t,\eta_1)-\tilde b(t,\eta_2)$: By definition we have 
\begin{align*}
0& \leq \tilde b(t,\eta_1)-\tilde b(t,\eta_2)\\
& = \int_0^{\eta_1} \min_j(e^{-\frac{1}{\ma}(\tilde \Y_j(t,\eta_1)-\tilde \Y_j(t,\theta))}) \min_j(\tilde\P_j) \min_j(\tilde \V_j^+) \tilde \Y_{1,\eta} (t,\theta) d\theta\\
& \quad -\int_0^{\eta_2} \min_j(e^{-\frac{1}{\ma}(\tilde \Y_j(t,\eta_2)-\tilde \Y_j(t,\theta))})\min_j(\tilde\P_j) \min_j(\tilde \V_j^+)\tilde \Y_{1,\eta} (t,\theta) d\theta\\
& = \int_0^{\eta_1} (\min_j(e^{-\frac{1}{\ma}(\tilde \Y_j(t,\eta_1)-\tilde \Y_j(t,\theta))})-\min_j(e^{-\frac{1}{\ma}(\tilde \Y_j(t,\eta_2)-\tilde \Y_j(t,\theta))}))\\
& \quad \qquad \qquad \qquad \times\min_j(\tilde\P_j)\min_j(\tilde \V_j^+)\tilde \Y_{1,\eta} (t,\theta) d\theta\\
& \quad - \int_{\eta_1}^{\eta_2} \min_j(e^{-\frac{1}{\ma}(\tilde \Y_j(t,\eta_2)-\tilde \Y_j(t,\theta))})\min_j(\tilde\P_j)\min_j(\tilde \V_j^+) \tilde \Y_{1,\eta}(t,\theta) d\theta\\
& \leq  \int_0^{\eta_1} (\min_j(e^{-\frac{1}{\ma}(\tilde \Y_j(t,\eta_1)-\tilde \Y_j(t,\theta))})-\min_j(e^{-\frac{1}{\ma}(\tilde \Y_j(t,\eta_2)-\tilde \Y_j(t,\theta))}))\\
& \qquad \qquad \qquad \qquad \times\min_j(\tilde\P_j)\min_j(\tilde \V_j^+)\tilde \Y_{1,\eta}(t,\theta) d\theta.
\end{align*}
Now we have to be much more careful then before. Namely, we have (as before)
\begin{align*}
0&\leq \min_j(e^{-\frac{1}{\ma}(\tilde \Y_j(t,\eta_1)-\tilde \Y_j(t,\theta))})-\min_j(e^{-\frac{1}{\ma}(\tilde \Y_j(t,\eta_2)-\tilde \Y_j(t,\theta))})\\
& \leq \frac{1}{\ma}\int_{\eta_1}^{\eta_2} (e^{\frac{1}{\ma}(\tilde \Y_1(t,\theta)-\tilde \Y_1(t,s))}\tilde \Y_{1, \eta} (t,s)+e^{\frac{1}{\ma}(\tilde \Y_2(t,\theta)-\tilde \Y_2(t,s))} \tilde \Y_{2, \eta} (t,s)) ds.
\end{align*}
Hence 
\begin{align*}
&\tilde b(t,\eta_1)-\tilde b(t,\eta_2) \\
&\quad \leq \int_0^{\eta_1} (\min_j(e^{-\frac{1}{\ma}(\tilde \Y_j(t,\eta_1)-\tilde \Y_j(t,\theta))})-\min_j(e^{-\frac{1}{\ma}(\tilde \Y_j(t,\eta_2)-\tilde \Y_j(t,\theta))}))\\
& \qquad \qquad \qquad \qquad \times\min_j(\tilde\P_j)\min_j(\tilde \V_j^+)\tilde \Y_{1,\eta}(t,\theta) d\theta\\
& \quad\leq \frac{1}{\ma}\int_0^{\eta_1} \int_{\eta_1}^{\eta_2} e^{-\frac{1}{\ma}(\tilde \Y_1(t,s)-\tilde \Y_1(t,\theta))} \tilde \Y_{1,\eta} (t,s) ds \min_j(\tilde\P_j) \min_j(\tilde \V_j^+)\tilde \Y_{1, \eta} (t,\theta) d\theta\\
& \qquad + \frac{1}{\ma}\int_0^{\eta_1} \int_{\eta_1}^{\eta_2} e^{-\frac{1}{\ma}(\tilde \Y_2(t,s)-\tilde \Y_2(t,\theta))}\tilde \Y_{2,\eta} (t,s)ds \min_j(\tilde\P_j) \min_j(\tilde \V_j^+)\tilde \Y_{1,\eta} (t,\theta) d\theta\\
&\quad = \frac{1}{\ma}\int_{\eta_1}^{\eta_2} \tilde \Y_{1,\eta} (t,s) \Big(\int_0^{\eta_1} e^{-\frac{1}{\ma}(\tilde \Y_1(t,s) -\tilde \Y_1(t,\theta))} \min_j(\tilde\P_j) \min_j(\tilde \V_j^+)\tilde \Y_{1,\eta} (t,\theta) d\theta\Big) ds\\
& \qquad + \frac{1}{\ma}\int_{\eta_1}^{\eta_2} \tilde \Y_{2,\eta} (t,s) \Big(\int_0^{\eta_1} e^{-\frac{1}{\ma}(\tilde \Y_2(t,s)-\tilde \Y_2(t,\theta))} \min_j(\tilde\P_j) \min_j(\tilde \V_j^+)\tilde \Y_{1,\eta} (t,\theta) d\theta\Big) ds\\
&\quad \leq  \frac{1}{\ma}\int_{\eta_1}^{\eta_2} \tilde \Y_{1,\eta} (t,s) \Big(\int_0^{s} e^{-\frac{1}{\ma}(\tilde \Y_1(t,s) -\tilde \Y_1(t,\theta))} \min_j(\tilde\P_j) \min_j(\tilde \V_j^+)\tilde \Y_{1,\eta} (t,\theta) d\theta\Big) ds\\
& \qquad + \frac{1}{\ma}\int_{\eta_1}^{\eta_2} \tilde \Y_{2,\eta} (t,s) \Big(\int_0^{s} e^{-\frac{1}{\ma}(\tilde \Y_2(t,s)-\tilde \Y_2(t,\theta))} \min_j(\tilde\P_j) \min_j(\tilde \V_j^+)\tilde \Y_{1,\eta} (t,\theta) d\theta\Big) ds\\
& \quad=\int_{\eta_1}^{\eta_2} \tilde B_1(t,s)ds+\int_{\eta_1}^{\eta_2} \tilde B_2(t,s) ds.
\end{align*}

As far as $\tilde B_1(t,s)$ is concerned, we have 
\begin{align*}
\int_{\eta_1}^{\eta_2}& \tilde B_1(t,s)ds \\ 
& = \frac{1}{\ma}\int_{\eta_1}^{\eta_2} \Big(\int_0^s e^{-\frac{1}{\ma}(\tilde \Y_1(t,s)-\tilde \Y_1(t,\theta))} \min_j(\tilde \P_j) \min_j(\tilde \V_j^+) \tilde \Y_{1,\eta} (t,\theta) d\theta\Big)\tilde \Y_{1,\eta} (t,s) ds\\
& \leq \frac{1}{\sqrt{2}}\int_{\eta_1}^{\eta_2} \Big(\int_0^s e^{-\frac{1}{\sqtC{1}}(\tilde \Y_1(t,s)-\tilde \Y_1(t,\theta))} \tilde \P_1^{3/4}\tilde \V_1^+\tilde \Y_{1,\eta}(t,\theta) d\theta\Big) \tilde \Y_{1, \eta}(t,s) ds\\
& \leq \frac{1}{\sqrt{2}}\int_{\eta_1}^{\eta_2} \Big(\int_0^s e^{-\frac{1}{\sqtC{1}}(\tilde \Y_1(t,s)-\tilde \Y_1(t,\theta))}\tilde \P_1^{3/2}\tilde \Y_{1,\eta} (t,\theta) d\theta\Big)^{1/2}\\
&\qquad\qquad\qquad\qquad\times \Big(\int_0^s e^{-\frac{1}{\sqtC{1}}(\tilde \Y_1(t,s)-\tilde \Y_1(t,\theta))} \tilde \U_1^2\tilde \Y_{1,\eta} (t,\theta) d\theta\Big)^{1/2} \tilde \Y_{1,\eta} (t,s) ds\\
& \leq 3\A^2 \int_{\eta_1}^{\eta_2} \tilde \P_1\tilde \Y_{1,\eta} (t,s) ds\\
& \leq \frac32 \A^7\vert \eta_2-\eta_1\vert,
\end{align*}
using \eqref{eq:general}.

As far as $\tilde B_2(t,s)$ is concerned, we have to be more careful. Therefore recall that we have, 
cf.~\eqref{eq:PUYH_scale}, that
\begin{equation*}
\sqtC{2}^5=2\tilde \P_2\tilde \Y_{2,\eta}(t,\eta) -\tilde \U_2^2\tilde \Y_{2,\eta}(t,\eta)+\tilde \Henergy_{2,\eta} (t,\eta)\leq 2\tilde \P_2\tilde \Y_{2,\eta} (t,\eta)+\tilde \Henergy_{2,\eta} (t,\eta).
\end{equation*}
Therefore we can write
\begin{align*}
\frac{1}{\ma}\int_0^s&  e^{-\frac{1}{\ma}(\tilde \Y_2(t,s)-\tilde \Y_2(t,\theta))} \min_j(\tilde\P_j)\min_j(\tilde \V_j^+)\tilde \Y_{1,\eta} (t,\theta) d\theta\\
&\le \frac{\sqrt{2}}{\ma}\int_0^s  e^{-\frac{1}{\sqtC{2}}(\tilde \Y_2(t,s)-\tilde \Y_2(t,\theta))} \min_j(\tilde\P_j)^{3/2}\tilde \Y_{1,\eta} (t,\theta) d\theta\\
& \leq \int_0^s e^{-\frac{1}{\sqtC{2}}(\tilde \Y_2(t,s) -\tilde \Y_2(t,\theta))} \tilde \P_2^{1/4} \tilde \P_1\tilde \Y_{1,\eta} (t,\theta) d\theta\\
& \leq \frac{\sqtC{1}^5}{2\sqtC{2}^5}\int_0^s e^{-\frac{1}{\sqtC{2}}(\tilde \Y_2(t,s)-\tilde \Y_2(t,\theta))} \tilde \P_2^{1/4} (2\tilde \P_2\tilde \Y_{2,\eta} + \tilde \Henergy_{2,\eta} )(t, \theta) d\theta\\
& \leq \frac{\sqtC{1}^5}{2\sqtC{2}^5} \int_0^s e^{-\frac{1}{\sqtC{2}}(\tilde \Y_2(t,s)-\tilde \Y_2(t,\theta))} (2\tilde \P_2^{5/4}  \tilde \Y_{2,\eta} + \frac{\sqtC{2}}{\sqrt{2}}\tilde \Henergy_{2,\eta} ) (t,\theta) d\theta\\
& \leq \frac{\sqtC{1}^5}{2\sqtC{2}^5}\Big(2\sqrt{2} \sqtC{2}^2\tilde\P_2(t,s) +2 \int_0^s e^{-\frac{1}{\sqtC{2}}(\tilde \Y_2(t,s)-\tilde \Y_2(t,\theta))}\tilde \P_2^{5/4}\tilde \Y_{2,\eta} (t, \theta) d\theta\Big).
\end{align*}

Note that the integral term can be bounded by $\bigO(1) \tilde \P_2(t,s)$ since 
\begin{equation*}
\int_0^s e^{-\frac{1}{\sqtC{2}}( \tilde \Y_2(t,s)-\tilde \Y_2(t,\theta)\vert } \tilde \P_2^{5/4}\tilde \Y_{2,\eta} (t,\theta) d\theta\leq \frac{15}{\sqrt{2}}\sqtC{2}^2\tilde \P_2(t,s)
\end{equation*}
by \eqref{eq:general}. We end up with 
\begin{align*}
\int_{\eta_1}^{\eta_2} \tilde B_2(t,s)ds & =\frac{1}{\ma} \int_{\eta_1}^{\eta_2} \tilde \Y_{2,\eta} (t,s)\\
&\quad\times\Big(\int_0^s e^{-\frac{1}{\ma}(\tilde \Y_2(t,s)-\tilde \Y_2(t,\theta))}\min_j(\tilde\P_j) \min_j(\tilde\V_j^+)\tilde \Y_{1,\eta} (t,\theta)d\theta\Big) ds\\
& \leq \frac{17}{\sqrt{2}}\frac{\sqtC{1}^5}{\sqtC{2}^3} \int_{\eta_1}^{\eta_2}  \tilde \P_2\tilde \Y_{2,\eta} (t,s)  ds\\
& \leq \frac{17}{2\sqrt{2}}\A^7 \vert \eta_2-\eta_1\vert .
\end{align*}
Moreover, 
\begin{align*}
\tilde b(t,\eta_1)-\tilde b(t,\eta_2) & \leq \int_{\eta_1}^{\eta_2} \tilde B_1(t,s) ds+ \int_{\eta_1}^{\eta_2} \tilde B_2(t,s) ds\\
& \leq \bigO(1)\A^7\vert \eta_2-\eta_1\vert,
\end{align*}
where $\bigO(1)$ denotes some constant, which only depends on $\A$ and which remains bounded as $\A\to 0$.

Finally combining both cases yields that there exists a constant $\bigO(1)$, which only depends on $A$ and which remains bounded as $\A\to 0$, such that 
\begin{equation*}
\vert \tilde b(t,\eta_2)-\tilde b(t,\eta_1)\vert \leq \bigO(1)\A^7\vert \eta_2-\eta_1\vert 
\end{equation*}
and subsequently
\begin{equation*}
\vert \tilde a(t,\eta_2)-\tilde a(t,\eta_1)\vert \leq \bigO(1)\A^7\vert \eta_2-\eta_1\vert.
\end{equation*}
This proves that the derivative exists for almost every $\eta$ and is bounded by \eqref{eq:MinMax2}.
\end{proof}
%------------------
%-----------------------------

\begin{lemma}\label{lemma:4}
The function 
\[
\eta\mapsto\min_k\Big(\int_0^\eta \min_j(e^{-\frac{1}{\ma}(\tilde \Y_j(t,\eta)-\tilde\Y_j(t,\theta))})
\min_j(\tilde\D_j)\min_j(\tilde\V_j^+)\tilde\Y_{k,\eta}(t,\theta)d\theta\Big)
\] 
is Lipschitz continuous with uniformly bounded Lipschitz constant and thus differentiable almost everywhere. The derivative satisfies,
\begin{align}
\vert\frac{d}{d\eta}\min_k\Big(\int_0^\eta \min_j(e^{-\frac{1}{\ma}(\tilde \Y_j(t,\eta)-\tilde\Y_j(t,\theta))})
\min_j(\tilde\D_j)\min_j(\tilde\V_j^+)\tilde\Y_{k,\eta}(t,\theta)d\theta\Big)\vert \leq \bigO(1)\A^8.  \label{eq:MinMax3}
\end{align}
\end{lemma}
\begin{proof}
We present the following argument. Introduce
\begin{align*}
\bar a(t,\eta)& =\min_k\Big(\int_0^\eta \min_j(e^{-\frac{1}{\ma}(\tilde \Y_j(t,\eta)-\tilde\Y_j(t,\theta))})
\min_j(\tilde\D_j)\min_j(\tilde\V_j^+)\tilde\Y_{k,\eta}(t,\theta)d\theta\Big),\\
\bar b(t,\eta)&= \int_0^\eta \min_j(e^{-\frac{1}{\ma}(\tilde \Y_j(t,\eta)-\tilde\Y_j(t,\theta))}) 
\min_j(\tilde\D_j)\min_j(\tilde\V_j^+)\tilde\Y_{1,\eta}(t,\theta) d\theta,\\
\bar c(t,\eta)&= \int_0^\eta \min_j(e^{-\frac{1}{\ma}(\tilde \Y_j(t,\eta)-\tilde\Y_j(t,\theta))})
\min_j(\tilde\D_j)\min_j(\tilde\V_j^+)\tilde\Y_{2,\eta}(t,\theta) d\theta.
\end{align*}
Thus 
\begin{equation*}
\bar a(t,\eta)=\min(\bar b, \bar c)(t,\eta).
\end{equation*}
We have to show that $\bar a(t,\dott)$ is Lipschitz continuous with a Lipschitz constant, which only depends on $\A$. Clearly, we have that 
\begin{equation*}
\vert \bar a(t,\eta_1)-\bar a(t,\eta_2)\vert \leq \max(\vert \bar b(t,\eta_1)-\bar b(t,\eta_2)\vert, \vert \bar c(t,\eta_1)-\bar c(t,\eta_2)\vert ),
\end{equation*}
and it suffices to show that both $\bar b(t,\dott)$ and $\bar c(t,\dott)$ are Lipschitz continuous with a Lipschitz constant, which only depends on $\A$. We are only going to establish the Lipschitz continuity for $\bar b$ since the argument for $\bar c$ follows the same lines. Let $0\leq \eta_1<\eta_2\leq 1$. Then we have to consider two cases
\begin{equation*}
0\leq \bar b(t,\eta_1)-\bar b(t,\eta_2)\qquad \text{ and }\qquad 0\leq \bar b(t,\eta_2)-\bar b(t,\eta_1).
\end{equation*}

(i): $0\leq\bar b(t,\eta_2)-\bar b(t,\eta_1)$: By definition we have
\begin{align*}
\bar b(t,\eta_2)&-\bar b(t,\eta_1)\\
&= \int_0^{\eta_2} \min_j(e^{-\frac{1}{\ma}(\tilde\Y_j(t,\eta_2)-\tilde\Y_j(t,\theta))})
\min_j(\tilde\D_j)\min_j(\tilde\V_j^+)\tilde\Y_{1,\eta}(t,\theta)d\theta\\
& \quad -\int_0^{\eta_1}\min_j(e^{-\frac{1}{\ma}(\tilde\Y_j(t,\eta_1)-\tilde\Y_j(t,\theta))})\min_j(\tilde\D_j)\min_j(\tilde\V_j^+)\tilde\Y_{1,\eta}(t,\theta)d\theta\\
& = \int_{\eta_1}^{\eta_2} \min_j(e^{-\frac{1}{\ma}(\tilde\Y_j(t,\eta_2)-\tilde\Y_j(t,\theta))})\\
&\qquad\qquad\qquad\qquad\qquad\qquad\qquad\times\min_j(\tilde\D_j)\min_j(\tilde\V_j^+)\tilde\Y_{1,\eta}(t,\theta) d\theta\\
& \quad + \int_0^{\eta_1} (\min_j(e^{-\frac{1}{\ma}(\tilde\Y_j(t,\eta_2)-\tilde\Y_j(t,\theta))})\\
&\qquad\qquad\qquad\qquad\qquad-\min_j(e^{-\frac{1}{\ma}(\tilde\Y_j(t,\eta_1)-\tilde\Y_j(t,\theta))}))\\
& \qquad \qquad \qquad\qquad\qquad\qquad\qquad \times\min_j(\tilde\D_j)\min_j(\tilde\V_j^+)\tilde\Y_{1,\eta}(t,\theta)d\theta\\
& \leq \int_{\eta_1}^{\eta_2} \min_j(e^{-\frac{1}{\ma}(\tilde\Y_j(t,\eta_2)-\tilde\Y_j(t,\theta))}) \\
&\qquad\qquad\qquad\qquad\qquad\qquad\qquad\times\min_j(\tilde\D_j)\min_j(\tilde\V_j^+)\tilde\Y_{1,\eta}(t,\theta) d\theta,
\end{align*}
where we used in the last step that $0\leq \min_j(\tilde\D_j)\min_j(\tilde\V_j^+)\tilde\Y_{1,\eta}(t,\eta)$ and that $\tilde \Y_i(t,\eta)$ is increasing, which implies 
\begin{equation*}
\min_j(e^{-\frac{1}{\ma}(\tilde\Y_j(t,\eta_1)-\tilde\Y_j(t,\theta))})\geq \min_j(e^{-\frac{1}{\ma}(\tilde\Y_j(t,\eta_2)-\tilde\Y_j(t,\theta))}).
\end{equation*}
Moreover, note that $\min_j(e^{-\frac{1}{\ma}(\tilde\Y_j(t,\eta_2)-\tilde\Y_j(t,\theta))})\leq 1$ for $0\leq \eta_1\leq \theta\leq \eta_2$ and that 
\begin{equation*}
0\leq \min_j(\tilde \D_j)\min_j(\tilde\V_j^+)\tilde\Y_{1,\eta}(t,\theta) \leq \frac{1}{\sqrt{2}}\sqtC{1}^8\leq \frac{1}{\sqrt{2}}\A^8
\end{equation*}
by \eqref{eq:all_estimatesB}, \eqref{eq:all_estimatesE}, and \eqref{eq:all_estimatesN}.
Thus 
\begin{equation*}
0\leq \bar b(t,\eta_2)-\bar b(t,\eta_1)\leq \frac{1}{\sqrt{2}}\A^8\vert \eta_2-\eta_1\vert.
\end{equation*}

(ii) $0\leq \bar b(t,\eta_1)-\bar b(t,\eta_2)$: By definition, we have
\begin{align*}
0& \leq \bar b(t,\eta_1)-\bar b(t,\eta_2)\\
& = \int_0^{\eta_1} \min_j(e^{-\frac{1}{\ma}(\tilde \Y_j(t,\eta_1)-\tilde \Y_j(t,\theta))}) \min_j(\tilde\D_j) \min_j(\tilde \V_j^+) \tilde \Y_{1,\eta} (t,\theta) d\theta\\
& \quad -\int_0^{\eta_2} \min_j(e^{-\frac{1}{\ma}(\tilde \Y_j(t,\eta_2)-\tilde \Y_j(t,\theta))})\min_j(\tilde\D_j) \min_j(\tilde \V_j^+)\tilde \Y_{1,\eta} (t,\theta) d\theta\\
& = \int_0^{\eta_1} (\min_j(e^{-\frac{1}{\ma}(\tilde \Y_j(t,\eta_1)-\tilde \Y_j(t,\theta))})-\min_j(e^{-\frac{1}{\ma}(\tilde \Y_j(t,\eta_2)-\tilde \Y_j(t,\theta))}))\\
& \qquad \qquad \qquad \qquad\qquad\qquad\qquad \times\min_j(\tilde\D_j)\min_j(\tilde \V_j^+)\tilde \Y_{1,\eta} (t,\theta) d\theta\\
& \quad - \int_{\eta_1}^{\eta_2} \min_j(e^{-\frac{1}{\ma}(\tilde \Y_j(t,\eta_2)-\tilde \Y_j(t,\theta))})\min_j(\tilde\D_j)\min_j(\tilde \V_j^+) \tilde \Y_{1,\eta}(t,\theta) d\theta\\
& \leq  \int_0^{\eta_1} (\min_j(e^{-\frac{1}{\ma}(\tilde \Y_j(t,\eta_1)-\tilde \Y_j(t,\theta))})-\min_j(e^{-\frac{1}{\ma}(\tilde \Y_j(t,\eta_2)-\tilde \Y_j(t,\theta))}))\\
& \qquad \qquad \qquad \qquad\qquad\qquad\qquad \times\min_j(\tilde\D_j)\min_j(\tilde \V_j^+)\tilde \Y_{1,\eta}(t,\theta) d\theta.
\end{align*}
Now we have to be much more careful then before. Namely, we have (as before)
\begin{align*}
0&\leq \min_j(e^{-\frac{1}{\ma}(\tilde \Y_j(t,\eta_1)-\tilde \Y_j(t,\theta))})-\min_j(e^{-\frac{1}{\ma}(\tilde \Y_j(t,\eta_2)-\tilde \Y_j(t,\theta))})\\
& \leq \frac{1}{\ma}\int_{\eta_1}^{\eta_2} (e^{-\frac{1}{\ma}(\tilde \Y_1(t,s)-\tilde \Y_1(t,\theta))}\tilde \Y_{1, \eta} (t,s)+e^{-\frac{1}{\ma}(\tilde \Y_2(t,s)-\tilde \Y_2(t,\theta))} \tilde \Y_{2, \eta} (t,s)) ds.
\end{align*}
Hence 
\begin{align*}
\bar b(t,\eta_1)&-\bar b(t,\eta_2) \\
& \leq \int_0^{\eta_1} (\min_j(e^{-\frac{1}{\ma}(\tilde \Y_j(t,\eta_1)-\tilde \Y_j(t,\theta))})-\min_j(e^{-\frac{1}{\ma}(\tilde \Y_j(t,\eta_2)-\tilde \Y_j(t,\theta))}))\\
& \qquad \qquad \qquad \qquad\qquad\qquad \times\min_j(\tilde\D_j)\min_j(\tilde \V_j^+)\tilde \Y_{1,\eta}(t,\theta) d\theta\\
& \leq\frac{1}{\ma} \int_0^{\eta_1} \int_{\eta_1}^{\eta_2} e^{-\frac{1}{\ma}(\tilde \Y_1(t,s)-\tilde \Y_1(t,\theta))} \tilde \Y_{1,\eta} (t,s) ds \min_j(\tilde\D_j) \min_j(\tilde \V_j^+)\tilde \Y_{1, \eta} (t,\theta) d\theta\\
& \quad + \frac{1}{\ma}\int_0^{\eta_1} \int_{\eta_1}^{\eta_2} e^{-\frac{1}{\ma}(\tilde \Y_2(t,s)-\tilde \Y_2(t,\theta))}\tilde \Y_{2,\eta} (t,s)ds\\
&\qquad\qquad\qquad\qquad\qquad\qquad\qquad\times \min_j(\tilde\D_j) \min_j(\tilde \V_j^+)\tilde \Y_{1,\eta} (t,\theta) d\theta\\
& = \frac{1}{\ma}\int_{\eta_1}^{\eta_2} \tilde \Y_{1,\eta} (t,s)\\
&\qquad\qquad\times \Big(\int_0^{\eta_1} e^{-\frac{1}{\ma}(\tilde \Y_1(t,s) -\tilde \Y_1(t,\theta))} \min_j(\tilde\D_j) \min_j(\tilde \V_j^+)\tilde \Y_{1,\eta} (t,\theta) d\theta\Big) ds\\
& \qquad + \frac{1}{\ma}\int_{\eta_1}^{\eta_2} \tilde \Y_{2,\eta} (t,s) \\
&\qquad\qquad\times\Big(\int_0^{\eta_1} e^{-\frac{1}{\ma}(\tilde \Y_2(t,s)-\tilde \Y_2(t,\theta))} \min_j(\tilde\D_j) \min_j(\tilde \V_j^+)\tilde \Y_{1,\eta} (t,\theta) d\theta\Big) ds\\
& \leq \frac{1}{\ma} \int_{\eta_1}^{\eta_2} \tilde \Y_{1,\eta} (t,s) \\
&\qquad\qquad\times\Big(\int_0^{s} e^{-\frac{1}{\ma}(\tilde \Y_1(t,s) -\tilde \Y_1(t,\theta))} \min_j(\tilde\D_j) \min_j(\tilde \V_j^+)\tilde \Y_{1,\eta} (t,\theta) d\theta\Big) ds\\
& \qquad + \frac{1}{\ma}\int_{\eta_1}^{\eta_2} \tilde \Y_{2,\eta} (t,s)\\
&\qquad\qquad\times \Big(\int_0^{s} e^{-\frac{1}{\ma}(\tilde \Y_2(t,s)-\tilde \Y_2(t,\theta))} \min_j(\tilde\D_j) \min_j(\tilde \V_j^+)\tilde \Y_{1,\eta} (t,\theta) d\theta\Big) ds\\
& =\int_{\eta_1}^{\eta_2} \bar B_1(t,s)ds+\int_{\eta_1}^{\eta_2} \bar B_2(t,s) ds.
\end{align*}

As far as $\bar B_1(t,s)$ is concerned, we have 
\begin{align*}
\int_{\eta_1}^{\eta_2}& \bar B_1(t,s)ds \\
& = \frac{1}{\ma}\int_{\eta_1}^{\eta_2} \Big(\int_0^s e^{-\frac{1}{\ma}(\tilde \Y_1(t,s)-\tilde \Y_1(t,\theta))} \min_j(\tilde\D_j) \min_j(\tilde \V_j^+) \tilde \Y_{1,\eta} (t,\theta) d\theta\Big)\tilde \Y_{1,\eta} (t,s) ds\\
& \leq \sqrt{2}\A\int_{\eta_1}^{\eta_2} \Big(\int_0^s e^{-\frac{1}{\sqtC{1}}(\tilde \Y_1(t,s)-\tilde \Y_1(t,\theta))} \tilde \P_1^{3/4}\tilde \V_1^+\tilde \Y_{1,\eta} (t,\theta) d\theta\Big) \tilde \Y_{1, \eta}(t,s) ds\\
& \leq \sqrt{2}\A\int_{\eta_1}^{\eta_2} \Big(\int_0^s e^{-\frac{1}{\sqtC{1}}(\tilde \Y_1(t,s)-\tilde \Y_1(t,\theta))}\tilde \P_1^{3/2}\tilde \Y_{1,\eta} (t,\theta) d\theta\Big)^{1/2}\\
&\qquad\qquad\qquad\qquad\times \Big(\int_0^s e^{-\frac{1}{\sqtC{1}}(\tilde \Y_1(t,s)-\tilde \Y_1(t,\theta))} \tilde \U_1^2\tilde \Y_{1,\eta} (t,\theta) d\theta\Big)^{1/2} \tilde \Y_{1,\eta} (t,s) ds\\
& \leq 6\A^3 \int_{\eta_1}^{\eta_2} \tilde \P_1\tilde \Y_{1,\eta} (t,s) ds\\
& \leq 3\A^8\vert \eta_2-\eta_1\vert,
\end{align*}
where we used \eqref{eq:all_estimatesN}.

As far as $\bar B_2(t,s)$ is concerned, we have to be more careful. Therefore recall that we have 
cf.~\eqref{eq:PUYH_scale}, that
\begin{equation*}
\sqtC{2}^5=2\tilde \P_2\tilde \Y_{2,\eta}(t,\eta) -\tilde \U_2^2\tilde \Y_{2,\eta}(t,\eta)+\tilde \Henergy_{2,\eta} (t,\eta)\leq 2\tilde \P_2\tilde \Y_{2,\eta} (t,\eta)+\tilde \Henergy_{2,\eta} (t,\eta).
\end{equation*}
Therefore we can write
\begin{align*}
\frac{1}{\ma}\int_0^s&  \e^{-\frac{1}{\ma}(\tilde \Y_2(t,s)-\tilde \Y_2(t,\theta))} \min_j(\tilde\D_j)\min_j(\tilde \V_j^+)\tilde \Y_{1,\eta} (t,\theta) d\theta\\
& \leq \frac{2\sqrt{2}\A}{\ma}\int_0^s e^{-\frac{1}{\sqtC{2}}(\tilde \Y_2(t,s) -\tilde \Y_2(t,\theta))} \min_j(\tilde \P_j)^{3/2}\tilde \Y_{1,\eta} (t,\theta) d\theta\\
& \leq 2\A\int_0^s e^{-\frac{1}{\sqtC{2}}(\tilde \Y_2(t,s)-\tilde \Y_2(t,\theta))} \tilde \P_2^{1/4}\tilde \P_1\tilde \Y_{1,\eta}(t,\theta) d\theta\\
& \leq A\frac{\sqtC{1}^5}{\sqtC{2}^5}\int_0^s e^{-\frac{1}{\sqtC{2}}(\tilde \Y_2(t,s)-\tilde \Y_2(t,\theta))} \tilde \P_2^{1/4}(2\tilde \P_2\tilde \Y_{2,\eta}+\tilde \Henergy_{2,\eta})(t,\theta) d\theta\\
& \leq A\frac{\sqtC{1}^5}{\sqtC{2}^5}\int_0^s e^{-\frac{1}{\sqtC{2}}(\tilde \Y_2(t,s)-\tilde \Y_2(t,\theta))}(2\tilde \P_2^{5/4}\tilde \Y_{2,\eta}+\frac{\sqtC{2}}{\sqrt{2}}\tilde \Henergy_{2,\eta})(t,\theta) d\theta\\
& \leq \A \frac{\sqtC{1}^5}{\sqtC{2}^5}\Big(2\sqrt{2}\sqtC{2}^2\tilde \P_2(t,s)+2\int_0^s e^{-\frac{1}{\sqtC{2}}(\tilde \Y_2(t,s)-\tilde \Y_2(t,\theta))}\tilde \P_2^{5/4}\tilde \Y_{2,\eta}(t,\theta)d\theta\Big).
\end{align*}

Note that the integral term can be bounded by $\bigO(1) \tilde \P_2(t,s)$ since 
\begin{equation*}
\int_0^s e^{-\frac{1}{\ma}(\hat \Y_2(t,s)-\tilde \Y_2(t,\theta)) } \tilde \P_2^{5/4}\tilde \Y_{2,\eta} (t,\theta) d\theta\leq \frac{15}{\sqrt{2}}\sqtC{2}^2\tilde \P_2(t,s) 
\end{equation*}
by \eqref{eq:general}. We end up with 
\begin{align*}
\int_{\eta_1}^{\eta_2}& \bar B_2(t,s)ds \\
& = \frac{1}{\ma}\int_{\eta_1}^{\eta_2} \tilde \Y_{2,\eta} (t,s)\Big(\int_0^s e^{-\frac{1}{\ma}(\tilde \Y_2(t,s)-\tilde \Y_2(t,\theta))}\min_j(\tilde\D_j) \min_j(\tilde\V_j^+)\tilde \Y_{1,\eta} (t,\theta)d\theta\Big) ds\\
& \leq \frac{17\sqrt{2}\A^6}{\sqtC{2}^3} \int_{\eta_1}^{\eta_2} \tilde \P_2\tilde \Y_{2,\eta} (t,s)  ds\\
& \leq \frac{17\A^8}{\sqrt{2}}\vert \eta_2-\eta_1\vert .
\end{align*}
Moreover, 
\begin{align*}
\bar b(t,\eta_1)-\bar b(t,\eta_2) & \leq \int_{\eta_1}^{\eta_2} \bar B_1(t,s) ds+ \int_{\eta_1}^{\eta_2} \bar B_2(t,s) ds\\
& \leq \bigO(1)\A^8\vert \eta_2-\eta_1\vert,
\end{align*}
where $\bigO(1)$ denotes some constant, which only depends on $\A$ and which remains bounded as $\A\to 0$.

Finally, combining both cases yields that there exists a constant $\bigO(1)$, which only depends on $\A$ and which remains bounded as $\A\to 0$, such that 
\begin{equation*}
\vert \bar b(t,\eta_2)-\bar b(t,\eta_1)\vert \leq \bigO(1)\A^8\vert \eta_2-\eta_1\vert 
\end{equation*}
and subsequently
\begin{equation*}
\vert \bar a(t,\eta_2)-\bar a(t,\eta_1)\vert \leq \bigO(1)\A^8\vert \eta_2-\eta_1\vert.
\end{equation*}
This proves that the derivative exists for almost every $\eta$ and is bounded by \eqref{eq:MinMax3}.\end{proof}
%-----------------------
%---------------------------
%----------------------------

\begin{lemma}\label{lemma:7}
The function
\[
\eta \mapsto \min_k \Big[ \int_0^\eta \min_j(e^{-\frac{1}{\ma}(\tilde \Y_j(t,\eta)-\tilde\Y_j(t,\theta))})\min_j(\tilde\V_j^+)^3\tilde\Y_{k,\eta}(t,\theta) d\theta\Big]
\]
is Lipschitz continuous with a uniformly bounded  Lipschitz constant and thus differentiable almost everywhere. The derivative satisfies
\begin{equation} 
\vert \frac{d}{d\eta} \min_k \Big[ \int_0^\eta \min_j(e^{-(\tilde \Y_j(t,\eta)-\tilde\Y_j(t,\theta))})\min_j(\tilde\V_j^+)^3\tilde\Y_{k,\eta}(t,\theta) d\theta\Big]\vert \leq \bigO(1)\A^{6},\label{eq:MinMax}
\end{equation}
where $\bigO(1)$ denotes a constant, which only depends on $\A$ and which remains bounded as $\A\to 0$.
\end{lemma}
\begin{proof}
To that end, we present the following argument.
Introduce 
\begin{align*}
a(t,\eta)& =\min_k[\int_0^\eta \min_j(e^{-\frac{1}{\ma}(\tilde \Y_j(t,\eta)-\tilde\Y_j(t,\theta))})\min_j(\tilde\V_j^+)^3\tilde\Y_{k,\eta}(t,\theta) d\theta],\\[2mm]
b(t,\eta)&= \int_0^\eta \min_j(e^{-\frac{1}{\ma}(\tilde \Y_j(t,\eta)-\tilde\Y_j(t,\theta))})\min_j(\tilde\V_j^+)^3\tilde\Y_{1,\eta}(t,\theta) d\theta,\\[2mm]
c(t,\eta)&=\int_0^\eta \min_j(e^{-\frac{1}{\ma}(\tilde \Y_j(t,\eta)-\tilde\Y_j(t,\theta))})\min_j(\tilde\V_j^+)^3\tilde\Y_{2,\eta}(t,\theta) d\theta.
\end{align*}
Thus
\begin{equation*}
a(t,\eta)=\min(b,c) (t,\eta).
\end{equation*} 
Then we have to show that $a(t,\dott)$ is Lipschitz continuous with a Lipschitz constant, which only depends on $\A$ and hence on $C$. Clearly we have that 
\begin{equation*}
\vert a(t,\eta_1)-a(t,\eta_2)\vert \leq \max(\vert b(t,\eta_1)-b(t,\eta_2)\vert,\vert c(t,\eta_1)-c(t,\eta_2)\vert ),
\end{equation*} 
and it suffices to show that both $b(t,\dott)$ and $c(t,\dott)$ are Lipschitz continuous with a Lipschitz constant, which only depends on $\A$. We are only going to establish the Lipschitz continuity for $b$, since the argument for $c$ follows the same lines. Let $0\leq \eta_1<\eta_2\leq 1$. Then we have to consider two cases
\begin{equation*}
0\leq b(t,\eta_1)-b(t,\eta_2)\qquad \text{ and } \qquad 0\leq b(t,\eta_2)-b(t,\eta_1).
\end{equation*}

(i): $0\leq b(t,\eta_2)-b(t,\eta_1)$: By definition we have 
\begin{align*}
b(t,\eta_2)&-b(t,\eta_1)\\
&= \int_0^{\eta_2} \min_j(e^{-\frac{1}{\ma}(\tilde\Y_j(t,\eta_2)-\tilde\Y_j(t,\theta))})\min_j(\tilde\V_j^+)^3\tilde \Y_{1,\eta} (t,\theta) d\theta\\ 
& \quad - \int_0^{\eta_1} \min_j(e^{-\frac{1}{\ma}(\tilde\Y_j(t,\eta_1)-\tilde\Y_j(t,\theta))})\min_j(\tilde\V_j^+)^3\tilde \Y_{1,\eta} (t,\theta) d\theta\\
& = \int_{\eta_1}^{\eta_2} \min_j(e^{-\frac{1}{\ma}(\tilde\Y_j(t,\eta_2)-\tilde\Y_j(t,\theta))})\min_j(\tilde\V_j^+)^3\tilde\Y_{1,\eta}(t,\theta) d\theta\\
& \quad +\int_0^{\eta_1}\big(\min_j(e^{-\frac{1}{\ma}(\tilde\Y_j(t,\eta_2)-\tilde\Y_j(t,\theta))})\\
&\qquad\qquad\qquad\qquad\qquad\qquad-\min_j(e^{-\frac{1}{\ma}(\tilde \Y_j(t,\eta_1)-\tilde\Y_j(t,\theta))})\big)\\
& \qquad \qquad\qquad \qquad\qquad\qquad\qquad\qquad  \times\min_j(\tilde\V_j^+)^3\tilde\Y_{1,\eta}(t,\theta) d\theta\\
& \leq  \int_{\eta_1}^{\eta_2} \min_j(e^{-\frac{1}{\ma}(\tilde\Y_j(t,\eta_2)-\tilde\Y_j(t,\theta))})\min_j(\tilde\V_j^+)^3\tilde\Y_{1,\eta}(t,\theta) d\theta,
\end{align*} 
where we used in the last step that $0\leq \min_j(\tilde\V_j^+)^3\tilde\Y_{1,\eta}(t,\eta)$ and that $\tilde\Y_1(t,\dott)$ is increasing, which implies 
\begin{align*}
&\min_j(e^{-\frac{1}{\ma}(\tilde \Y_j(t,\eta_1)-\tilde\Y_j(t,\theta))})\\
& \qquad = \min_j(e^{-\frac{1}{\ma}(\tilde\Y_j(t,\eta_1)-\tilde \Y_j(t,\eta_2)+\tilde\Y_j(t,\eta_2)-\tilde \Y_j(t,\theta))})\\
& \qquad \geq \min_j(e^{-\frac{1}{\ma}(\tilde \Y_j(t,\eta_2)-\tilde \Y_j(t,\theta))})\min_j (e^{-\frac{1}{\ma}(\tilde \Y_j(t,\eta_1)-\tilde \Y_j(t,\eta_2))})\\
& \qquad \geq \min_j(e^{-\frac{1}{\ma}(\tilde \Y_j(t,\eta_2)-\tilde \Y_j(t,\theta))}).
\end{align*}
Moreover, note that $\min_j(e^{-\frac{1}{\ma}(\tilde \Y_j(t,\eta_2)-\tilde \Y_j(t,\theta))})\leq 1$ since $ \eta_1\le\theta\le\eta_2$,  and that 
\begin{equation*}
0\leq \min_j(\tilde\V_j^+)^3 \tilde \Y_{1, \eta}(t,\eta)\leq \vert \tilde \U_1^3\vert \tilde \Y_{1, \eta}(t,\eta) \leq \frac{1}{\sqrt{2}}\sqtC{1}^7\leq \frac{1}{\sqrt{2}}\A^7
\end{equation*}
by \eqref{eq:all_estimatesB} and \eqref{eq:all_estimatesG}. Thus 
\begin{align*}
0& \leq b(t,\eta_2)-b(t,\eta_1)\\
& \leq \int_{\eta_1}^{\eta_2} \min_j(e^{-\frac{1}{\ma}(\tilde \Y_j(t,\eta_2)-\tilde\Y_j(t,\theta))})\min_j(\tilde\V_j^+)^3\tilde \Y_{1,\eta} (t,\theta) d\theta\\
& \leq \frac{1}{\sqrt{2}}A^7 \vert \eta_2-\eta_1\vert.
\end{align*}

(ii): $0\leq b(t,\eta_1)-b(t,\eta_2)$: By definition we have 
\begin{align*}
b(t,\eta_1)&-b(t,\eta_2)\\
&= \int_0^{\eta_1} \min_j(e^{-\frac{1}{\ma}(\tilde \Y_j(t,\eta_1)-\tilde \Y_j(t,\theta))}) \min_j(\tilde \V_j^+)^3 \tilde \Y_{1, \eta}(t,\theta) d\theta\\
& \quad -\int_0^{\eta_2} \min_j(e^{-\frac{1}{\ma}(\tilde \Y_j(t,\eta_2)-\tilde \Y_j(t,\theta))})\min_j(\tilde \V_j^+)^3 \tilde \Y_{1, \eta} (t,\theta) d\theta\\
& = \int_0^{\eta_1} (\min_j(e^{-\frac{1}{\ma}(\tilde \Y_j(t,\eta_1)-\tilde \Y_j(t,\theta))})\\
&\qquad\qquad\qquad\qquad\qquad\qquad-\min_j(e^{-\frac{1}{\ma}(\tilde \Y_j(t,\eta_2)-\tilde \Y_j(t,\theta))}))\\
& \qquad \qquad \qquad \qquad\qquad\qquad\qquad\qquad\times \min_j(\tilde \V_j^+)^3\tilde \Y_{1, \eta}(t, \theta) d\theta\\
& \quad -\int_{\eta_1}^{\eta_2} \min_j(e^{-\frac{1}{\ma}(\tilde\Y_j(t,\eta_2)-\tilde \Y_j(t,\theta))}\min_j(\tilde\V_j^+)^3\tilde \Y_{1, \eta}(t,\theta) d\theta\\
& \leq \int_0^{\eta_1} (\min_j(e^{-\frac{1}{\ma}(\tilde \Y_j(t,\eta_1)-\tilde \Y_j(t,\theta))})\\
&\qquad\qquad\qquad\qquad\qquad\qquad-\min_j(e^{-\frac{1}{\ma}(\tilde \Y_j(t,\eta_2)-\tilde \Y_j(t,\theta))}))\\
&  \qquad \qquad \qquad \qquad\qquad\qquad\qquad\qquad\times \min_j(\tilde \V_j^+)^3\tilde \Y_{1, \eta}(t, \theta) d\theta.
\end{align*}
Now we have to be much more careful than before. Namely, if $\min_j(\tilde \Y_j(t,\theta)-\tilde \Y_j(t,\eta_2))=\tilde \Y_i(t,\theta)-\tilde \Y_i(t,\eta_2)$, we have
\begin{align*}
0& \leq \min_j(e^{-\frac{1}{\ma}(\tilde \Y_j(t,\eta_1)-\tilde \Y_j(t,\theta))})-\min_j(e^{-\frac{1}{\ma}(\tilde \Y_j(t,\eta_2)-\tilde \Y_j(t,\theta))})\\
& = \min_j(e^{-\frac{1}{\ma}(\tilde \Y_j(t,\eta_1)-\tilde \Y_j(t,\theta))}) -e^{\frac{1}{\ma}(\tilde \Y_i(t, \theta)-\tilde \Y_i(t,\eta_2))}\\
& \leq e^{\frac{1}{\ma}(\tilde \Y_i(t,\theta)-\tilde \Y_i(t,\eta_1))}-e^{\frac{1}{\ma}(\tilde \Y_i(t,\theta) -\tilde \Y_i(t,\eta_2))} \\
& = \frac{1}{\ma}\int_{\eta_1}^{\eta_2} e^{\frac{1}{\ma}(\tilde \Y_i(t,\theta) -\tilde \Y_i(t,s))} \tilde\Y_{i,\eta} (t,s) ds\\
& \leq \frac{1}{\ma}\int_{\eta_1}^{\eta_2} \big( e^{\frac{1}{\ma}(\tilde \Y_1(t,\theta)-\tilde \Y_1(t,s))}\tilde \Y_{1, \eta}(t,s)+e^{\frac{1}{\ma}(\tilde \Y_2(t,\theta)-\tilde \Y_2(t,s))} \tilde \Y_{2,\eta} (t,s)\big)ds.
\end{align*}
Hence 
\begin{align*}
b(t,\eta_1)&-b(t,\eta_2)\\
& \leq \int_0^{\eta_1} \big(\min_j(e^{-\frac{1}{\ma}(\tilde \Y_j(t,\eta_1)-\tilde \Y_j(t,\theta))})\\
&\qquad\qquad\qquad\qquad\qquad\qquad-\min_j(e^{-\frac{1}{\ma}(\tilde \Y_j(t,\eta_2)-\tilde \Y_j(t,\theta))}\big)\\
& \qquad \qquad \qquad \qquad\qquad\qquad\qquad\qquad\times \min_j(\tilde \V_j^+)^3\tilde \Y_{1, \eta}(t, \theta) d\theta\\
& \leq \frac{1}{\ma}\int_0^{\eta_1} \int_{\eta_1}^{\eta_2} e^{-\frac{1}{\ma}(\tilde \Y_1(t,s) -\tilde \Y_1(t,\theta))}\tilde \Y_{1,\eta} (t,s) ds \min_j(\tilde \V_j^+)^3\tilde \Y_{1,\eta} (t,\theta) d\theta\\
& \quad + \frac{1}{\ma}\int_0^{\eta_1} \int_{\eta_1}^{\eta_2} e^{-\frac{1}{\ma}(\tilde \Y_2(t,s) -\tilde \Y_2(t,\theta))} \tilde \Y_{2,\eta} (t,s) ds \min_j(\tilde \V_j^+)^3\tilde \Y_{1,\eta} (t,\theta) d\theta\\
& = \frac{1}{\ma}\int_{\eta_1}^{\eta_2}\tilde \Y_{1,\eta} (t,s) \left( \int_0^{\eta_1} e^{-\frac{1}{\ma}(\tilde \Y_1(t,s)-\tilde \Y_1(t,\theta)} \min_j(\tilde \V_j^+)^3 \tilde \Y_{1, \eta} (t,\theta) d\theta\right) ds\\ 
& \quad + \frac{1}{\ma}\int_{\eta_1}^{\eta_2}\tilde \Y_{2,\eta} (t,s) \left( \int_0^{\eta_1} e^{-\frac{1}{\ma}(\tilde \Y_2(t,s)-\tilde \Y_2(t,\theta))} \min_j(\tilde \V_j^+)^3 \tilde \Y_{1, \eta} (t,\theta) d\theta\right) ds\\ 
& \leq \frac{1}{\ma}\int_{\eta_1}^{\eta_2}\tilde \Y_{1,\eta} (t,s) \left( \int_0^{s} e^{-\frac{1}{\ma}(\tilde \Y_1(t,s)-\tilde \Y_1(t,\theta))} \min_j(\tilde \V_j^+)^3 \tilde \Y_{1, \eta} (t,\theta) d\theta\right) ds\\ 
& \quad + \frac{1}{\ma}\int_{\eta_1}^{\eta_2}\tilde \Y_{2,\eta} (t,s) \left( \int_0^{s} e^{-\frac{1}{\ma}(\tilde \Y_2(t,s)-\tilde \Y_2(t,\theta))} \min_j(\tilde \V_j^+)^3 \tilde \Y_{1, \eta} (t,\theta) d\theta\right) ds\\
& = \int_{\eta_1}^{\eta_2}B_1(t,s)ds+\int_{\eta_1}^{\eta_2}B_2(t,s)ds.
\end{align*}
As far as $B_1(t,s)$ is concerned, recall that
\begin{align*}
\tilde \P_i(t,s)&=\frac{1}{4\sqtC{i}} \int_0^1 e^{-\frac{1}{\sqtC{i}}\vert\tilde \Y_i(t,s)-\tilde \Y_i(t,\theta)\vert } (2(\tilde \U_i^2-\tilde \P_i)\tilde \Y_{i,\eta} (t,\theta) + \sqtC{i}^5) d\theta\\
& = \frac{1}{4\sqtC{i}} \int_0^1 e^{-\frac{1}{\sqtC{i}}\vert\tilde \Y_i(t,s) -\tilde \Y_i(t,\theta)\vert } (\tilde \U_i^2\tilde \Y_{i, \eta} +\tilde \Henergy_{i,\eta})(t,\theta) d\theta,
\end{align*}
which implies that 
\begin{align*}
\int_{\eta_1}^{\eta_2}& B_1(t,s)ds \\
& = \frac{1}{\ma}\int_{\eta_1}^{\eta_2}\tilde \Y_{1,\eta} (t,s) \left( \int_0^{s} e^{-\frac{1}{\ma}(\tilde \Y_1(t,s)-\tilde \Y_1(t,\theta))} \min_j(\tilde \V_j^+)^3 \tilde \Y_{1, \eta} (t,\theta) d\theta\right) ds\\
& \leq \frac{1}{\ma}\int_{\eta_1}^{\eta_2} \tilde \Y_{1,\eta}(t,s) \left(\int_0^s e^{-\frac{1}{\ma}(\tilde \Y_1(t,s)-\tilde \Y_1(t,\theta))} \min_j(\tilde \V_j^+)\tilde\U_1^2 \tilde \Y_{1,\eta} (t,\theta) d\theta\right) ds\\
& \leq \frac{\ma}{\sqrt{2}}\int_{\eta_1}^{\eta_2} \tilde \Y_{1,\eta} (t,s) \left (\int_0^s e^{-\frac{1}{\ma}(\tilde \Y_1(t,s) -\tilde \Y_1(t,\theta))} \tilde \U_1^2 \tilde \Y_{1,\eta}(t,\theta) d\theta\right) ds\\
& \leq \frac{4}{\sqrt{2}}\A^2 \int_{\eta_1}^{\eta_2} \tilde \P_1\tilde \Y_{1,\eta} (t,s)ds\\
& \leq \sqrt{2}A^7\vert \eta_2-\eta_1\vert.
\end{align*}
As far as $B_2(t,s)$ is concerned, we have to be more careful. Therefore recall that we have 
\begin{equation*}
\sqtC{2}^5= 2\tilde \P_2\tilde \Y_{2,\eta}(t,\eta) -\tilde \U_2^2\tilde \Y_{2,\eta}(t,\eta) +\tilde \Henergy_{2,\eta}(t,\eta)\leq 2\tilde \P_2\tilde \Y_{2,\eta} (t,\eta) + \tilde \Henergy_{2,\eta}(t,\eta).
\end{equation*}
Therefore we can write
\begin{align*}
&\frac{1}{\ma}\int_0^s e^{-\frac{1}{\ma}(\tilde \Y_2(t,s) -\tilde \Y_2(t,\theta))}\min_j(\tilde\V_j^+)^3 \tilde \Y_{1,\eta} (t,\theta) d\theta\\
& \qquad \leq \frac{1}{\ma}\int_0^s e^{-\frac{1}{\ma}(\tilde \Y_2(t,s)-\tilde \Y_2(t,\theta))} \vert \tilde\U_2\vert \tilde \U_1^2\tilde \Y_{1,\eta} (t,\theta) d\theta\\
& \qquad \leq\frac{\sqtC{1}^5}{\ma\sqtC{2}^5} \int_0^s  e^{-\frac{1}{\ma}(\tilde \Y_2(t,s)-\tilde \Y_2(t,\theta))} \vert \tilde \U_2\vert (2\tilde \P_2\tilde \Y_{2,\eta} +\tilde \Henergy_{2,\eta}) (t,\theta) d\theta\\
& \qquad \leq \frac{\sqtC{1}^5}{\ma\sqtC{2}^5} \int_0^s e^{-\frac{1}{\ma}(\tilde\Y_2(t,s) -\tilde \Y_2(t,\theta))} \big(\sqtC{2}\tilde \U_2^2\tilde \Y_{2,\eta} + \frac{1}{\sqtC{2}}\tilde \P_2^2\tilde \Y_{2,\eta} + \sqtC{2}^2 \tilde \Henergy_{2,\eta}\big) (t,\theta) d\theta\\
& \qquad \leq \frac{\sqtC{1}^5}{\ma\sqtC{2}^5} \int_0^1 e^{-\frac{1}{\sqtC{2}}\vert \tilde \Y_2(t,s)-\tilde \Y_2(t,\theta)\vert } \big(\sqtC{2}\tilde \U_2^2\tilde \Y_{2,\eta} + \sqtC{2}^2 \tilde \Henergy_{2,\eta} +\frac{1}{\sqtC{2}} \tilde \P_2^2\tilde \Y_{2,\eta} \big)(t,\theta) d\theta\\
& \qquad \leq \frac{\sqtC{1}^5}{\ma\sqtC{2}^5} \Big(4\sqtC{2}^2\tilde \P_2(t,s) + 4\sqtC{2}^3 \tilde \P_2(t,s) \\
&\qquad\qquad\qquad\qquad\qquad+\frac{1}{\sqtC{2}}\int_0^1 e^{-\frac{1}{\sqtC{2}}\vert \tilde \Y_2(t,s)-\tilde \Y_2(t,\theta)\vert }\tilde \P_2^2\tilde \Y_{2,\eta} (t,\theta) d\theta\Big).
\end{align*}
Note that the integral term can is bounded by $\bigO(1)\tilde \P_2(t,s)$ since 
\begin{align*}
\int_0^1 e^{-\frac{1}{\sqtC{2}}\vert \tilde \Y_2(t,s)-\tilde \Y_2(t,\theta)\vert}& \tilde \P_2^2\tilde \Y_{2,\eta}(t,\theta) d\theta \\
& = \sqtC{2}^7\int_0^1 e^{-\vert \Y_2(t,\sqtC{2}^2s)-\Y_2(t,\sqtC{2}^2\theta)\vert }\P_2^2 \Y_{2,\eta} (t, \sqtC{2}^2\theta) d\theta\\
& =\sqtC{2}^5 \int_0^{\sqtC{2}^2} e^{-\vert \Y_2(t, \sqtC{2}^2s) - \Y_2(t, \theta)\vert }\P_2^2\Y_{2,\eta} (t, \theta) d\theta\\
& \leq \bigO(1)\sqtC{2}^5 \P_2(t, \sqtC{2}^2s)= \bigO(1)\sqtC{2}^3\tilde \P_2(t,s)
\end{align*}
by \eqref{eq:109}.  We end up with
\begin{align*}
\int_{\eta_1}^{\eta_2} &B_2(t,s)ds\\
& =\frac{1}{\ma} \int_{\eta_1}^{\eta_2} \tilde \Y_{2,\eta} (t,s) \left(\int_0^s e^{-\frac{1}{\ma}(\tilde \Y_2(t,s)-\tilde \Y_2(t,\theta))} \min_j(\tilde \V_j^+)^3\tilde \Y_{1,\eta} (t,\theta) d\theta\right ) ds \\
& \leq \frac{\sqtC{1}^5}{a\sqtC{2}^5}\int_{\eta_1}^{\eta_2}\tilde \Y_{2,\eta} (t,s)\\
&\qquad\times\left (\sqtC{2}^2\bigO(\A) \tilde \P_2(t,s) +\frac{1}{\sqtC{2}} \int_0^1e^{-\frac{1}{\sqtC{2}}\vert \tilde \Y_2(t,s)-\tilde \Y_2(t,\theta)\vert } \tilde \P_2^2\tilde \Y_{2,\eta} (t,\theta) d\theta\right)ds\\
& \leq \frac{\bigO(1)\sqtC{1}^5\sqtC{2}^2}{a\sqtC{2}^5} \int_{\eta_1}^{\eta_2} \tilde \P_2\tilde \Y_{2,\eta} (t,s) ds\\
& \leq \bigO(1)\A^6\vert \eta_2-\eta_1\vert .
\end{align*}
Moreover,
\begin{align*}
b(t,\eta_1)-b(t,\eta_2)& \leq \int_{\eta_1}^{\eta_2} B_1(t,s) ds + \int_{\eta_1}^{\eta_2} B_2(t,s)ds\\
& \leq \bigO(1)\A^6 \vert \eta_2-\eta_1\vert,
\end{align*}
where $\bigO(1)$ only depends on $\A$  and hence on $C=\frac{\A^2}{2}$, and remains bounded as $\A\to 0$. 

Finally, combining both cases yields that 
\begin{equation*}
\vert b(t,\eta_2)-b(t,\eta_1)\vert \leq \bigO(1)\A^6 \vert \eta_2-\eta_1\vert 
\end{equation*} 
and subsequently
\begin{equation*}%\label{subeq:MinMax}
\vert a(t,\eta_2)-a(t,\eta_1)\vert \leq \bigO(1)\A^6 \vert \eta_2-\eta_1\vert ,
\end{equation*}
where $\bigO(1)$ remains bounded as $\A\to 0$.
This proves that the derivative exists for almost every $\eta$ and is bounded by \eqref{eq:MinMax}.
\end{proof}

%-------------------------

We need to estimate the pointwise difference between two functions $\tilde\D_j$, $j=1,2$. This is the content of the following lemma.
\begin{lemma}\label{lemma:D}
We have that
\begin{align*}
\vert \tilde\D_1(t,\eta)-\tilde\D_2(t,\eta)\vert & \leq 2\A^{3/2}\max_j(\tilde\D_j^{1/2})(t,\eta)\vert \tilde \Y_1(t,\eta)-\tilde \Y_2(t,\eta)\vert \\
& \quad + 2\sqrt{2}\A^{3/2}\max_j(\tilde \D_j^{1/2})(t,\eta) \norm{\tilde \Y_1-\tilde \Y_2}\\
& \quad + 4\A^3\Big(\int_0^\eta e^{-\frac{1}{\A}(\tilde \Y_j(t,\eta)-\tilde \Y_j(t,\theta))}(\tilde \U_1-\tilde \U_2)^2(t,\theta) d\theta\Big)^{1/2}\\
& \quad +2\sqrt{2} \A^3 \Big(\int_0^\eta e^{-\frac{1}{\A}(\tilde \Y_j(t,\eta)-\tilde \Y_j(t,\theta))}(\sqP{1}-\sqP{2})^2(t,\theta) d\theta\Big)^{1/2} \\
& \quad + \frac{3\A^4}{\sqrt{2}} \Big(\int_0^\eta e^{-\frac{1}{\ma}(\tilde \Y_j(t,\eta)-\tilde \Y_j(t,\theta))}(\tilde \Y_1-\tilde \Y_2)^2(t,\theta) d\theta\Big)^{1/2}\\
& \quad + \frac{12\sqrt{2}\A^4}{\sqrt{3}e}\Big(\int_0^\eta e^{-\frac{3}{4\A}(\tilde \Y_j(t,\eta)-\tilde \Y_j(t,\theta))}d\theta\Big)^{1/2} \vert \sqtC{1}-\sqtC{2}\vert\\
& \quad + \tilde \U_j^2\vert \tilde \Y_1-\tilde \Y_2\vert (t,\eta)\\
& \quad + \tilde \P_j\vert \tilde\Y_1-\tilde \Y_2\vert (t,\eta)\\
& \quad +\frac{3\A^4}{2} \int_0^\eta e^{-\frac{1}{\ma} (\tilde \Y_j(t,\eta)-\tilde \Y_j(t,\theta))}\vert \tilde \Y_1-\tilde \Y_2\vert (t,\theta) d\theta\\
& \quad + 6\A^4 \int_0^\eta e^{-\frac{3}{4\A} (\tilde \Y_j(t,\eta)-\tilde \Y_j(t,\theta))}d\theta\vert \sqtC{1}-\sqtC{2}\vert,
\end{align*}
for any value of $j=1,2$.
\end{lemma}
\begin{proof}
Direct calculations yield 
\begin{align*}
\vert \tilde\D_1(t,\eta)-&\tilde\D_2(t,\eta)\vert \\
& = | \int_0^\eta e^{-\frac{1}{\sqtC{1}}(\tilde \Y_1(t,\eta)-\tilde \Y_1(t,\theta))}((\tilde \U_1^2-\tilde \P_1)\tilde\Y_{1,\eta} (t,\theta)+\frac12 \sqtC{1}^5)d\theta\\
& \qquad -\int_0^\eta e^{-\frac{1}{\sqtC{2}}(\tilde \Y_2(t,\eta)-\tilde\Y_2(t,\theta))}((\tilde \U_2^2-\tilde\P_2)\tilde\Y_{2,\eta}(t,\theta)+\frac12 \sqtC{2}^5) d\theta |\\
& \leq \vert \int_0^\eta (e^{-\frac{1}{\sqtC{1}}(\tilde \Y_1(t,\eta)-\tilde\Y_1(t,\theta))}-e^{-\frac{1}{\sqtC{1}}(\tilde\Y_2(t,\eta)-\tilde\Y_2(t,\theta))})\\
&\qquad\qquad\qquad\qquad\times((\tilde \U_1^2-\tilde \P_1)\tilde\Y_{1,\eta} (t,\theta)+\frac12 \sqtC{1}^5) \mathbbm{1}_{B(\eta)^c}(t,\theta)d\theta\vert \\
& \quad +\vert \int_0^\eta (e^{-\frac{1}{\sqtC{2}}(\tilde \Y_1(t,\eta)-\tilde\Y_1(t,\theta))}-e^{-\frac{1}{\sqtC{2}}(\tilde \Y_2(t,\eta)-\tilde\Y_2(t,\theta))})\\
&\qquad\qquad\qquad\qquad\times((\tilde \U_2^2-\tilde \P_2)\tilde\Y_{2,\eta} (t,\theta)+\frac{1}{2}\sqtC{2}^5)\mathbbm{1}_{B(\eta)}(t,\theta)d\theta\vert \\
& \quad + \vert \int_0^\eta \min_j(e^{-\frac{1}{\sqtC{1}}(\tilde \Y_j(t,\eta)-\tilde\Y_j(t,\theta))})((\tilde\U_1^2-\tilde\P_1)\tilde\Y_{1,\eta} (t,\theta)+\frac12\sqtC{1}^5)d\theta\\
& \quad -\int_0^\eta\min_j(e^{-\frac{1}{\sqtC{2}}(\tilde \Y_j(t,\eta)-\tilde \Y_j(t,\theta))})(\tilde \U_2^2-\tilde\P_2)\tilde\Y_{2,\eta} (t,\theta) +\frac12 \sqtC{2}^5) d\theta\vert\\
&= \bar d_{11}(t,\eta)+\bar d_{12}(t,\eta)+\bar d_{13}(t,\eta),
\end{align*}
where $B(\eta)$ is defined in \eqref{Def:Bn}.

For $\bar d_{11}(t,\eta)$ we immediately obtain 
\begin{align}\nn
 \bar d_{11}(t,\eta)& \leq \frac{1}{\sqtC{1}}\int_0^\eta (\vert \tilde \Y_1(t,\eta)-\tilde \Y_2(t,\eta)\vert +\vert \tilde \Y_1(t,\theta)-\tilde \Y_2(t,\theta)\vert)\\ \nn
 & \qquad \qquad \qquad \qquad \times e^{-\frac{1}{\sqtC{1}}(\tilde \Y_1(t,\eta)-\tilde \Y_1(t,\theta))}((\tilde \U_1^2-\tilde \P_1)\tilde \Y_{1,\eta}(t,\theta) +\frac12 \sqtC{1}^5)d\theta\\ \nn
 & \leq \frac{\tilde  \D_1(t,\eta)}{\sqtC{1}} \vert \tilde \Y_1(t,\eta)-\tilde \Y_2(t,\eta)\vert + \sqrt{2}\sqtC{1}^{3/2}\tilde \D_1^{1/2}(t,\eta)\norm{\tilde \Y_1-\tilde \Y_2}\\ \label{est:d11}
 & \leq \sqtC{1}^{3/2}\tilde \D_1^{1/2}(t,\eta) \vert \tilde\Y_1(t,\eta)-\tilde \Y_2(t,\eta)\vert +\sqrt{2}\sqtC{1}^{3/2}\tilde \D_1^{1/2}(t,\eta) \norm{\tilde\Y_1-\tilde\Y_2}.
 \end{align}
Following the same lines we end up with  
\begin{align}\nn
\bar d_{12}(t,\eta) &\leq \frac{1}{\sqtC{2}}\int_0^\eta (\vert \tilde \Y_1(t,\eta)-\tilde\Y_2(t,\eta)\vert +\vert \tilde\Y_1(t,\theta)-\tilde\Y_2(t,\theta)\vert)  \\ \nn
 &\qquad\qquad\qquad\qquad\times e^{-\frac{1}{\sqtC{2}}(\tilde \Y_2(t,\eta)-\tilde \Y_2(t,\theta))} \vert(\tilde\U_2^2-\tilde\P_2)\tilde\Y_{2,\eta} (t,\theta) +\frac12\sqtC{2}^5\vert d\theta \\ \label{est:d12}
& \leq \sqtC{2}^{3/2}\tilde \D_2^{1/2}(t,\eta)\vert \tilde\Y_1(t,\eta)-\tilde\Y_2(t,\eta)\vert +\sqrt{2}\sqtC{2}^{3/2}\tilde \D_2^{1/2} (t,\eta)\norm{\tilde \Y_1-\tilde\Y_2}.
\end{align}The last term $\bar d_{13}(t,\eta)$ needs a bit more work. Indeed, 
\begin{align*}
\bar d_{13}(t,\eta)&=\vert \int_0^\eta \min_j(e^{-\frac{1}{\sqtC{1}}(\tilde \Y_j(t,\eta)-\tilde\Y_j(t,\theta))})((\tilde\U_1^2-\tilde\P_1)\tilde\Y_{1,\eta} (t,\theta)+\frac12\sqtC{1}^5)d\theta\\
& \qquad -\int_0^\eta \min_j(e^{-\frac{1}{\sqtC{2}}(\tilde\Y_j(t,\eta)-\tilde \Y_j(t,\theta))})((\tilde \U_2^2-\tilde\P_2)\tilde\Y_{2,\eta} (t,\theta) +\frac12\sqtC{2}^5) d\theta\vert\\
&\leq\vert \int_0^\eta (\min_j(e^{-\frac{1}{\sqtC{1}}(\tilde \Y_j(t,\eta)-\tilde\Y_j(t,\theta))})\tilde\U_1^2\tilde\Y_{1,\eta}(t,\theta)\\
&\qquad\qquad\qquad\qquad\qquad -\min_j(e^{-\frac{1}{\sqtC{2}}(\tilde \Y_j(t,\eta)-\tilde \Y_j(t,\theta))})\tilde\U_2^2\tilde\Y_{2,\eta})(t,\theta))d\theta\vert\\
& \quad + \vert \int_0^\eta (\min_j(e^{-\frac{1}{\sqtC{1}}(\tilde \Y_j(t,\eta)-\tilde\Y_j(t,\theta))}) \tilde\P_1\tilde\Y_{1,\eta}(t,\theta)\\
&\qquad\qquad\qquad\qquad\qquad-\min_j(e^{-\frac{1}{\sqtC{2}}(\tilde \Y_j(t,\eta)-\tilde \Y_j(t,\theta))})\tilde\P_2\tilde\Y_{2,\eta}) (t,\theta) )d\theta\vert \\
& \quad +\vert  \int_0^\eta (\min_j(e^{-\frac{1}{\sqtC{1}}(\tilde \Y_j(t,\eta)-\tilde\Y_j(t,\theta))}) \frac12\sqtC{1}^5\\
&\qquad\qquad\qquad\qquad\qquad-\min_j(e^{-\frac{1}{\sqtC{2}}(\tilde \Y_j(t,\eta)-\tilde \Y_j(t,\theta))})\frac12\sqtC{2}^5) d\theta\vert\\
& = \bar T_1(t,\eta)+\bar T_2(t,\eta)+\bar T_3(t,\eta).
\end{align*}

To estimate $\bar T_1(t,\eta)$, recall \eqref{decay:impl}, \eqref{est:L3a_lemma}, \eqref{est:L2c_lemma}, and Lemma~\ref{lemma:enkel} (ii), which imply
\begin{align}\nn
\bar T_1(t,\eta)&= \vert \int_0^\eta (\min_j(\e^{-\frac{1}{\sqtC{1}}(\tilde \Y_j(t,\eta)-\tilde \Y_j(t,\theta))})\tilde \U_1^2\tilde\Y_{1,\eta}(t,\theta)\\ \nn
&\qquad\qquad\qquad\qquad\qquad-\min_j(e^{-\frac{1}{\sqtC{2}}(\tilde \Y_j(t,\eta)-\tilde \Y_j(t,\theta))})\tilde \U_2^2\tilde \Y_{2,\eta}(t,\theta) ) d\theta\vert \\ \nn
& \leq \vert \int_0^\eta \min_j(e^{-\frac{1}{\sqtC{1}}(\tilde\Y_j(t,\eta)-\tilde \Y_j(t,\theta))})(\tilde \U_1^2-\tilde\U_2^2)\tilde \Y_{1,\eta}(t,\theta) \mathbbm{1}_{\tilde \U_2^2\leq \tilde \U_1^2}(t,\theta)d\theta\vert\\ \nn
& \quad +\vert \int_0^\eta \min_j(e^{-\frac{1}{\sqtC{2}}(\tilde \Y_j(t,\eta)-\tilde \Y_j(t,\theta))})(\tilde \U_1^2-\tilde \U_2^2)\tilde \Y_{2,\eta}(t,\theta) \mathbbm{1}_{\tilde \U_1^2<\tilde \U_2^2}(t,\theta) d\theta\vert\\ \nn
& \quad + \mathbbm{1}_{\sqtC{1}\leq \sqtC{2}}\vert\int_0^\eta (\min_j(e^{-\frac{1}{\sqtC{1}}(\tilde\Y_j(t,\eta)-\tilde \Y_j(t,\theta))})\\ \nn
&\qquad\qquad\qquad\qquad\qquad-\min_j(e^{-\frac{1}{\sqtC{2}}(\tilde\Y_j(t,\eta)-\tilde \Y_j(t,\theta))}))\min_j(\tilde \U_j^2)\tilde \Y_{2,\eta}(t,\theta) d\theta\vert \\ \nn
& \quad +\mathbbm{1}_{\sqtC{2}<\sqtC{1}}\vert\int_0^\eta (\min_j(e^{-\frac{1}{\sqtC{1}}(\tilde \Y_j(t,\eta)-\tilde \Y_j(t,\theta))})\\ \nn
&\qquad\qquad\qquad\qquad\qquad-\min_j(e^{-\frac{1}{\sqtC{2}}(\tilde \Y_j(t,\eta)-\tilde \Y_j(t,\theta))}))\min_j(\tilde \U_j^2)\tilde \Y_{1,\eta}(t,\theta) d\theta\vert \\ \nn
& \quad + \vert \int_0^\eta \min_j(e^{-\frac1{\ma}(\tilde \Y_j(t,\eta)-\tilde \Y_j(t,\theta))})\min_j(\tilde \U_j^2)(\tilde \Y_{1,\eta}-\tilde \Y_{2,\eta})(t,\theta) d\theta\vert\\ \nn
& \leq 2\int_0^\eta \min_j(e^{-\frac{1}{\sqtC{1}}(\tilde \Y_j(t,\eta)-\tilde \Y_j(t,\theta))})\vert \tilde \U_1\vert \tilde \Y_{1,\eta}\vert \tilde \U_1-\tilde \U_2\vert (t,\theta) d\theta\vert\\ \nn
& \quad +2 \int_0^\eta \min_j(e^{-\frac{1}{\sqtC{2}}(\tilde\Y_j(t,\eta)-\tilde \Y_j(t,\theta))}) \vert \tilde \U_2\vert \tilde \Y_{2,\eta}\vert \tilde \U_1-\tilde \U_2\vert (t,\theta) d\theta\vert \\ \nn
& \quad +\mathbbm{1}_{\sqtC{1}\leq \sqtC{2}}\frac{4}{\ma e} \vert \int_0^\eta \min_j(e^{-\frac3{4\A} (\tilde \Y_j(t,\eta)-\tilde \Y_j(t,\theta))})\min_j(\tilde \U_j^2)\tilde \Y_{2,\eta}(t,\theta) d\theta\vert\vert \sqtC{1}-\sqtC{2}\vert \\ \nn
& \quad +\mathbbm{1}_{\sqtC{2}<\sqtC{1}}\frac{4}{\ma e} \vert \int_0^\eta \min_j(e^{-\frac{3}{4\A} (\tilde \Y_j(t,\eta)-\tilde \Y_j(t,\theta))})\min_j(\tilde \U_j^2)\tilde \Y_{1,\eta}(t,\theta) d\theta\vert \vert \sqtC{1}-\sqtC{2}\vert\\ \nn
& \quad +\vert \min_j(e^{-\frac{1}{\ma}(\tilde \Y_j(t,\eta)-\tilde \Y_j(t,\theta))})\min_j(\tilde \U_j^2)(\tilde \Y_1-\tilde \Y_2)\vert_{\theta=0}^{\eta}\\ \nn
& \qquad \qquad - \int_0^\eta \frac{d}{d\theta} (\min_j(e^{-\frac{1}{\ma}(\tilde \Y_j(t,\eta)-\tilde \Y_j(t,\theta))})\min_j(\tilde \U_j^2))(\tilde \Y_1-\tilde \Y_2)(t,\theta) d\theta\vert\\ \nn
& \leq 2\Big( \int_0^\eta e^{-\frac{1}{\sqtC{1}}(\tilde \Y_1(t,\eta)-\tilde \Y_1(t,\theta))}\tilde \U_1^2\tilde \Y_{1,\eta}^2(t,\theta) d\theta\Big)^{1/2}\\ \nn
& \qquad \qquad \qquad \qquad \qquad \quad \times \Big(\int_0^\eta e^{-\frac{1}{\sqtC{1}}(\tilde \Y_j(t,\eta)-\tilde \Y_j(t,\theta))}(\tilde \U_1-\tilde \U_2)^2(t,\theta) d\theta\Big)^{1/2}\\ \nn
& \quad +2 \Big(\int_0^\eta e^{-\frac{1}{\sqtC{2}}(\tilde \Y_2(t,\eta)-\tilde \Y_2(t,\theta))}\tilde \U_2^2\tilde \Y_{2,\eta}^2(t,\theta) d\theta\Big)^{1/2}\\ \nn
& \qquad \qquad \qquad \qquad \qquad \quad \times\Big(\int_0^\eta e^{-\frac{1}{\sqtC{2}}(\tilde \Y_j(t,\eta)-\tilde \Y_j(t,\theta))} (\tilde \U_1-\tilde \U_2)^2(t,\theta) d\theta\Big)^{1/2}\\ \nn
& \quad + \frac{2\sqrt{2}\A}{e}\Big(\int_0^\eta e^{-\frac{3}{4\sqtC{2}}(\tilde \Y_2(t,\eta)-\tilde \Y_2(t,\theta))}\tilde \U_2^2\tilde \Y_{2,\eta}^2(t,\theta)d\theta\Big)^{1/2}\\ \nn
& \qquad \qquad \qquad \qquad \qquad  \quad \times \Big(\int_0^\eta e^{-\frac{3}{4\A}(\tilde \Y_j(t,\eta)-\tilde \Y_j(t,\theta))}d\theta\Big)^{1/2} \vert \sqtC{1}-\sqtC{2}\vert\\ \nn
& \quad + \frac{2\sqrt{2}\A}{e}\Big(\int_0^\eta e^{-\frac{3}{4\sqtC{1}}(\tilde \Y_1(t,\eta)-\tilde \Y_1(t,\theta))}\tilde \U_1^2\tilde \Y_{1,\eta}^2(t,\theta) d\theta\Big)^{1/2}\\ \nn
& \qquad \qquad \qquad \qquad \qquad \quad \times \Big(\int_0^\eta e^{-\frac{3}{4\A} (\tilde \Y_j(t,\eta)-\tilde \Y_j(t,\theta))}d\theta\Big)^{1/2} \vert \sqtC{1}-\sqtC{2}\vert\\ \nn
& \quad + \tilde\U_j^2 \vert \tilde\Y_1-\tilde \Y_2\vert (t,\eta)\\ \nn
& \quad + \frac{1}{\ma}\int_0^\eta e^{-\frac{1}{\ma} (\tilde \Y_j(t,\eta)-\tilde \Y_j(t,\theta))}\min_j(\tilde \U_j^2) (\tilde \Y_{1,\eta}+\tilde \Y_{2,\eta})\vert \tilde \Y_1-\tilde \Y_2\vert (t,\theta)d\theta\\ \nn
& \quad + \A^4 \int_0^\eta e^{-\frac{1}{\ma}(\tilde \Y_j(t,\eta)-\tilde \Y_j(t,\theta))} \vert \tilde \Y_1-\tilde \Y_2\vert (t,\theta) d\theta\\ \nn
& \leq 4\A^3 \Big(\int_0^\eta e^{-\frac{1}{\A}(\tilde \Y_j(t,\eta)-\tilde \Y_j(t,\theta))}(\tilde \U_1-\tilde\U_2)^2(t,\theta)d\theta\Big)^{1/2}\\ \nn
& \quad + \frac{8\sqrt{2}\A^4}{\sqrt{3}e}\Big(\int_0^\eta e^{-\frac{3}{4\A} (\tilde \Y_j(t,\eta)-\tilde \Y_j(t,\theta))} d\theta\Big)^{1/2} \vert \sqtC{1}-\sqtC{2}\vert\\ \nn
& \quad + \tilde \U_j^2\vert \tilde \Y_1-\tilde \Y_2\vert (t,\eta)\\ \nn
& \quad + \sqrt{2}\A^4 \Big(\int_0^\eta e^{-\frac{1}{\ma} (\tilde \Y_j(t,\eta)-\tilde \Y_j(t,\theta))} (\tilde \Y_1-\tilde \Y_2)^2(t,\theta) d\theta\Big)^{1/2}\\ \label{eq:barT1}
& \quad + \A^4 \int_0^\eta e^{-\frac{1}{\ma}(\tilde \Y_j(t,\eta)-\tilde \Y_j(t,\theta))} \vert \tilde \Y_1-\tilde \Y_2\vert (t,\theta) d\theta.
\end{align}
Following the same lines one has 
\begin{align} \nn
\bar T_2(t,\eta)&= \vert \int_0^\eta (\min_j(\e^{-\frac{1}{\sqtC{1}}(\tilde \Y_j(t,\eta)-\tilde \Y_j(t,\theta))})\tilde \P_1\tilde\Y_{1,\eta}(t,\theta)\\ \nn
&\qquad\qquad\qquad\qquad\qquad-\min_j(e^{-\frac{1}{\sqtC{2}}(\tilde \Y_j(t,\eta)-\tilde \Y_j(t,\theta))})\tilde \P_2\tilde \Y_{2,\eta}(t,\theta) ) d\theta\vert \\ \nn
& \leq \vert \int_0^\eta \min_j(e^{-\frac{1}{\sqtC{1}}(\tilde\Y_j(t,\eta)-\tilde \Y_j(t,\theta))})(\tilde \P_1-\tilde\P_2)\tilde \Y_{1,\eta}(t,\theta) \mathbbm{1}_{\tilde \P_2\leq \tilde \P_1}(t,\theta)d\theta\vert\\ \nn
& \quad +\vert \int_0^\eta \min_j(e^{-\frac{1}{\sqtC{2}}(\tilde \Y_j(t,\eta)-\tilde \Y_j(t,\theta))})(\tilde \P_1-\tilde \P_2)\tilde \Y_{2,\eta}(t,\theta) \mathbbm{1}_{\tilde \P_1<\tilde \P_2}(t,\theta) d\theta\vert\\ \nn
& \quad + \mathbbm{1}_{\sqtC{1}\leq \sqtC{2}}\vert\int_0^\eta (\min_j(e^{-\frac{1}{\sqtC{1}}(\tilde\Y_j(t,\eta)-\tilde \Y_j(t,\theta))})\\ \nn
&\qquad\qquad\qquad\qquad\qquad-\min_j(e^{-\frac{1}{\sqtC{2}}(\tilde\Y_j(t,\eta)-\tilde \Y_j(t,\theta))}))\min_j(\tilde \P_j)\tilde \Y_{2,\eta}(t,\theta) d\theta\vert \\ \nn
& \quad +\mathbbm{1}_{\sqtC{2}<\sqtC{1}}\vert\int_0^\eta (\min_j(e^{-\frac{1}{\sqtC{1}}(\tilde \Y_j(t,\eta)-\tilde \Y_j(t,\theta))})\\ \nn
&\qquad\qquad\qquad\qquad\qquad-\min_j(e^{-\frac{1}{\sqtC{2}}(\tilde \Y_j(t,\eta)-\tilde \Y_j(t,\theta))}))\min_j(\tilde \P_j)\tilde \Y_{1,\eta}(t,\theta) d\theta\vert \\ \nn
& \quad + \vert \int_0^\eta \min_j(e^{-\frac1{\ma}(\tilde \Y_j(t,\eta)-\tilde \Y_j(t,\theta))})\min_j(\tilde \P_j )(\tilde \Y_{1,\eta}-\tilde \Y_{2,\eta})(t,\theta) d\theta\vert\\ \nn
& \leq 2\int_0^\eta \min_j(e^{-\frac{1}{\sqtC{1}}(\tilde \Y_j(t,\eta)-\tilde \Y_j(t,\theta))})\vert \sqP{1}\vert \tilde \Y_{1,\eta}\vert \sqP{1}-\sqP{2}\vert (t,\theta) d\theta\vert\\ \nn
& \quad +2 \int_0^\eta \min_j(e^{-\frac{1}{\sqtC{2}}(\tilde\Y_j(t,\eta)-\tilde \Y_j(t,\theta))}) \vert \sqP{2}\vert \tilde \Y_{2,\eta}\vert \sqP{1}-\sqP{2}\vert (t,\theta) d\theta\vert \\ \nn
& \quad +\mathbbm{1}_{\sqtC{1}\leq \sqtC{2}}\frac{4}{\ma e} \vert \int_0^\eta \min_j(e^{-\frac3{4\A} (\tilde \Y_j(t,\eta)-\tilde \Y_j(t,\theta))})\min_j(\tilde \P_j)\tilde \Y_{2,\eta}(t,\theta) d\theta\vert\vert \sqtC{1}-\sqtC{2}\vert \\ \nn
& \quad +\mathbbm{1}_{\sqtC{2}<\sqtC{1}}\frac{4}{\ma e} \vert \int_0^\eta \min_j(e^{-\frac{3}{4\A} (\tilde \Y_j(t,\eta)-\tilde \Y_j(t,\theta))})\min_j(\tilde \P_j)\tilde \Y_{1,\eta}(t,\theta) d\theta\vert \vert \sqtC{1}-\sqtC{2}\vert\\ \nn
& \quad +\vert \min_j(e^{-\frac{1}{\ma}(\tilde \Y_j(t,\eta)-\tilde \Y_j(t,\theta))})\min_j(\tilde \P_j)(\tilde \Y_1-\tilde \Y_2)\vert_{\theta=0}^{\eta}\\ \nn
& \qquad \qquad - \int_0^\eta \frac{d}{d\theta} (\min_j(e^{-\frac{1}{\ma}(\tilde \Y_j(t,\eta)-\tilde \Y_j(t,\theta))})\min_j(\tilde \P_j))(\tilde \Y_1-\tilde \Y_2)(t,\theta) d\theta\vert\\ \nn
& \leq 2\Big( \int_0^\eta e^{-\frac{1}{\sqtC{1}}(\tilde \Y_1(t,\eta)-\tilde \Y_1(t,\theta))}\tilde \P_1\tilde \Y_{1,\eta}^2(t,\theta) d\theta\Big)^{1/2}\\ \nn
& \qquad \qquad \qquad \qquad   \quad \times \Big(\int_0^\eta e^{-\frac{1}{\sqtC{1}}(\tilde \Y_j(t,\eta)-\tilde \Y_j(t,\theta))}(\sqP{1}-\sqP{2})^2(t,\theta) d\theta\Big)^{1/2}\\ \nn
& \quad +2 \Big(\int_0^\eta e^{-\frac{1}{\sqtC{2}}(\tilde \Y_2(t,\eta)-\tilde \Y_2(t,\theta))}\tilde \P_2\tilde \Y_{2,\eta}^2(t,\theta) d\theta\Big)^{1/2}\\ \nn
& \qquad \qquad \qquad \qquad   \quad \times\Big(\int_0^\eta e^{-\frac{1}{\sqtC{2}}(\tilde \Y_j(t,\eta)-\tilde \Y_j(t,\theta))} (\sqP{1}-\sqP{2})^2(t,\theta) d\theta\Big)^{1/2}\\ \nn
& \quad + \frac{2\A}{e}\Big(\int_0^\eta e^{-\frac{3}{4\sqtC{2}}(\tilde \Y_2(t,\eta)-\tilde \Y_2(t,\theta))}\tilde \P_2\tilde \Y_{2,\eta}^2(t,\theta)d\theta\Big)^{1/2}\\ \nn
& \qquad \qquad \qquad \qquad \qquad  \quad \times \Big(\int_0^\eta e^{-\frac{3}{4\A}(\tilde \Y_j(t,\eta)-\tilde \Y_j(t,\theta))}d\theta\Big)^{1/2} \vert \sqtC{1}-\sqtC{2}\vert\\ \nn
& \quad + \frac{2\A}{e}\Big(\int_0^\eta e^{-\frac{3}{4\sqtC{1}}(\tilde \Y_1(t,\eta)-\tilde \Y_1(t,\theta))}\tilde \P_1\tilde \Y_{1,\eta}^2(t,\theta) d\theta\Big)^2\\ \nn
& \qquad \qquad \qquad \qquad \qquad  \quad \times \Big(\int_0^\eta e^{-\frac{3}{4\A} (\tilde \Y_j(t,\eta)-\tilde \Y_j(t,\theta))}d\theta\Big)^{1/2} \vert \sqtC{1}-\sqtC{2}\vert\\ \nn
& \quad + \tilde\P_j \vert \tilde\Y_1-\tilde \Y_2\vert (t,\eta)\\ \nn
& \quad + \frac{1}{\ma}\int_0^\eta e^{-\frac{1}{\ma} (\tilde \Y_j(t,\eta)-\tilde \Y_j(t,\theta))}\min_j(\tilde \P_j) (\tilde \Y_{1,\eta}+\tilde \Y_{2,\eta})\vert \tilde \Y_1-\tilde \Y_2\vert (t,\theta)d\theta\\ \nn
& \quad + \frac{\A^4}{2} \int_0^\eta e^{-\frac{1}{\ma}(\tilde \Y_j(t,\eta)-\tilde \Y_j(t,\theta))} \vert \tilde \Y_1-\tilde \Y_2\vert (t,\theta) d\theta\\ \nn
& \leq 2\sqrt{2}\A^3 \Big(\int_0^\eta e^{-\frac{1}{\A}(\tilde \Y_j(t,\eta)-\tilde \Y_j(t,\theta))}(\sqP{1}-\sqP{2})^2(t,\theta)d\theta\Big)^{1/2}\\ \nn
& \quad + \frac{4\sqrt{2}\A^4}{\sqrt{3}e}\Big(\int_0^\eta e^{-\frac{3}{4\A} (\tilde \Y_j(t,\eta)-\tilde \Y_j(t,\theta))} d\theta\Big)^{1/2} \vert \sqtC{1}-\sqtC{2}\vert\\ \nn
& \quad + \tilde \P_j\vert \tilde \Y_1-\tilde \Y_2\vert (t,\eta)\\ \nn
& \quad + \frac{\A^4}{\sqrt{2}} \Big(\int_0^\eta e^{-\frac{1}{\ma} (\tilde \Y_j(t,\eta)-\tilde \Y_j(t,\theta))} (\tilde \Y_1-\tilde \Y_2)^2(t,\theta) d\theta\Big)^{1/2}\\ \label{eq:barT2}
& \quad + \frac{\A^4}{2} \int_0^\eta e^{-\frac{1}{\ma}(\tilde \Y_j(t,\eta)-\tilde \Y_j(t,\theta))} \vert \tilde \Y_1-\tilde \Y_2\vert (t,\theta) d\theta.
\end{align}

For the last term $\bar T_3(t,\eta)$, we have 
\begin{align} \nn
2\bar T_3(t,\eta)& =\vert \int_0^\eta (\min_j(e^{-\frac{1}{\sqtC{1}}(\tilde \Y_j(t,\eta)-\tilde \Y_j(t,\theta))})\sqtC{1}^5-(\min_j(e^{-\frac{1}{\sqtC{2}}(\tilde \Y_j(t,\eta)-\tilde \Y_j(t,\theta))})\sqtC{2}^5 )d\theta\vert\\ \nn
& \leq \mathbbm{1}_{\sqtC{1}\leq \sqtC{2}}\vert \int_0^\eta(e^{-\frac{1}{\sqtC{2}}(\tilde \Y_j(t,\eta)-\tilde \Y_j(t,\theta))})(\sqtC{1}^5-\sqtC{2}^5) d\theta \vert \\ \nn
& \quad + \mathbbm{1}_{\sqtC{2}<\sqtC{1}} \vert \int_0^\eta (e^{-\frac{1}{\sqtC{1}}(\tilde\Y_j(t,\eta)-\tilde \Y_j(t,\theta))})(\sqtC{1}^5-\sqtC{2}^5) d\theta\vert\\ \nn
& \quad + \ma^5 \int_0^\eta \vert \min_j(e^{-\frac{1}{\sqtC{1}}(\tilde \Y_j(t,\eta)-\tilde \Y_j(t,\theta))})-\min_j(e^{-\frac{1}{\sqtC{2}}(\tilde \Y_j(t,\eta)-\tilde \Y_j(t,\theta))})\vert d\theta\\ \nn
& \leq 10\A^4\int_0^\eta e^{-\frac{1}{\A}(\tilde \Y_j(t,\eta)-\tilde \Y_j(t,\theta))}d\theta \vert \sqtC{1}-\sqtC{2}\vert\\ \nn
& \quad + \frac{4\ma^4}{e}\int_0^\eta e^{-\frac{3}{4\A}(\tilde \Y_j(t,\eta)-\tilde \Y_j(t,\theta))}d\theta\vert \sqtC{1}-\sqtC{2}\vert\\
& \leq 12\A^4 \int_0^\eta e^{-\frac{3}{4\A} (\tilde \Y_j(t,\eta)-\tilde \Y_j(t,\theta))}d\theta\vert \sqtC{1}-\sqtC{2}\vert.\label{eq:barT3}
\end{align}

Thus we end up with 
\begin{align*}
\vert \tilde\D_1(t,\eta)-\tilde\D_2(t,\eta)\vert & \leq 2\A^{3/2}\max_j(\tilde\D_j^{1/2})(t,\eta)\vert \tilde \Y_1(t,\eta)-\tilde \Y_2(t,\eta)\vert \\
& \quad + 2\sqrt{2}\A^{3/2}\max_j(\tilde \D_j^{1/2})(t,\eta) \norm{\tilde \Y_1-\tilde \Y_2}\\
& \quad + 4\A^3\Big(\int_0^\eta e^{-\frac{1}{\A}(\tilde \Y_j(t,\eta)-\tilde \Y_j(t,\theta))}(\tilde \U_1-\tilde \U_2)^2(t,\theta) d\theta\Big)^{1/2}\\
& \quad +2\sqrt{2} \A^3 \Big(\int_0^\eta e^{-\frac{1}{\A}(\tilde \Y_j(t,\eta)-\tilde \Y_j(t,\theta))}(\sqP{1}-\sqP{2})^2(t,\theta) d\theta\Big)^{1/2} \\
& \quad + \frac{3\A^4}{\sqrt{2}} \Big(\int_0^\eta e^{-\frac{1}{\ma}(\tilde \Y_j(t,\eta)-\tilde \Y_j(t,\theta))}(\tilde \Y_1-\tilde \Y_2)^2(t,\theta) d\theta\Big)^{1/2}\\
& \quad + \frac{12\sqrt{2}\A^4}{\sqrt{3}e}\Big(\int_0^\eta e^{-\frac{3}{4\A}(\tilde \Y_j(t,\eta)-\tilde \Y_j(t,\theta))}d\theta\Big)^{1/2} \vert \sqtC{1}-\sqtC{2}\vert\\
& \quad + \tilde \U_j^2\vert \tilde \Y_1-\tilde \Y_2\vert (t,\eta)\\
& \quad + \tilde \P_j\vert \tilde\Y_1-\tilde \Y_2\vert (t,\eta)\\
& \quad +\frac{3\A^4}{2} \int_0^\eta e^{-\frac{1}{\ma} (\tilde \Y_j(t,\eta)-\tilde \Y_j(t,\theta))}\vert \tilde \Y_1-\tilde \Y_2\vert (t,\theta) d\theta\\
& \quad + 6\A^4 \int_0^\eta e^{-\frac{3}{4\A} (\tilde \Y_j(t,\eta)-\tilde \Y_j(t,\theta))}d\theta\vert \sqtC{1}-\sqtC{2}\vert,
\end{align*}
which proves the lemma.
\end{proof}
%-----------------------------------

\begin{lemma}\label{lemma:E}
We have the following estimates
\begin{align}
\int_0^\eta e^{-\frac3{2A} (\tilde\Y(t,\eta)-\tilde\Y(t,\theta))}\tilde\P\tilde\Y_{\eta} (t,\theta) d\theta&\leq 2A\tilde\P(t,\eta), 
\label{eq:343XX}  \\
\int_0^\eta e^{-\frac{3}{2A}(\tilde \Y(t,\eta)-\tilde \Y(t,\theta))} \tilde \Henergy_{\eta}(t,\theta) d\theta 
& \le 4A \tilde\P(t,\eta), \label{eq:Henergy32XX}  \\
\int_0^\eta e^{-\frac{1}{A}(\tilde\Y(t,\eta)-\tilde\Y(t,\theta))} \tilde\U^2(t,\theta)d \theta & \leq 6 \tilde\P(t,\eta), \label{est:L3bXX} \\
\int_0^\eta e^{-\frac{5}{4A}(\tilde \Y(t,\eta)-\tilde \Y(t,\theta))} \tilde \P\tilde \Y_{\eta}(t,\theta) d\theta&\leq 4 A\tilde\P(t,\eta),
   \label{eq:all_PestimatesEXX}  \\
   \int_0^\eta e^{-\frac3{2A} (\tilde\Y(t,\eta)-\tilde\Y(t,\theta))}\tilde \P(t,\theta) d\theta
& \leq 7\tilde\P(t,\eta), \label{eq:32PXX}  \\
\int_0^\eta e^{-\frac1{2A}(\tilde\Y(t,\eta)-\tilde\Y(t,\theta))} \tilde\P^2\tilde\Y_{\eta} (t,\theta) d\theta&\leq A\bigO(1)\P^{1/2}(t,\eta), 
\label{eq:p2yXX}  \\
\int_0^\eta e^{-\frac{1}{A}(\tilde \Y(t,\eta)-\tilde \Y(t,\theta))} \tilde \P^{1+\beta}\tilde \Y_{\eta}(t,\theta) d\theta
&\leq 3\frac{1+\beta}{\beta}\frac{A^{1+4\beta}}{4^{\beta}}\tilde \P(t,\eta), \quad\beta>0.  \label{eq:generalXX}
\end{align}
\end{lemma}
\begin{proof}  The proof of \eqref{eq:343XX} goes as follows
\begin{align*}
\int_0^\eta e^{-\frac3{2A} (\tilde \Y(t,\eta)-\tilde \Y(t,\theta))}&\tilde\P\tilde\Y_{\eta} (t,\theta) d\theta \\
& = \frac23A \tilde \P(t,\eta)-\frac2{3A}\int_0^\eta e^{-\frac3{2A}(\tilde\Y(t,\eta)-\tilde\Y(t,\theta))}\tilde\Q\tilde\Y_{\eta} (t,\theta) d\theta\\
& \leq \frac23 A \tilde \P(t,\eta) +\frac23 \int_0^\eta e^{-\frac3{2A}(\tilde\Y(t,\eta)-\tilde\Y(t,\theta))}\tilde\P\tilde\Y_{\eta}(t,\theta)d\theta
\end{align*}
which implies 
\begin{equation*}
\int_0^\eta e^{-\frac3{2A} (\tilde\Y(t,\eta)-\tilde\Y(t,\theta))}\tilde\P\tilde\Y_{\eta} (t,\theta) d\theta\leq 2A\tilde\P(t,\eta).
\end{equation*}

Next, we use that
\begin{align*}
\int_0^\eta e^{-\frac{3}{2A}(\tilde \Y(t,\eta)-\tilde \Y(t,\theta))} \tilde \Henergy_{\eta}(t,\theta) d\theta 
& \le \int_0^\eta e^{-\frac{1}{A}(\tilde \Y(t,\eta)-\tilde \Y(t,\theta))} \tilde \Henergy_{\eta}(t,\theta) d\theta  \\ 
& \le 4A \tilde\P(t,\eta), 
\end{align*}
see \eqref{eq:all_PestimatesD}, showing \eqref{eq:Henergy32XX}.

For \eqref{est:L3bXX} we find
\begin{align*}
\int_0^\eta &e^{-\frac{1}{A}(\tilde\Y(t,\eta)-\tilde\Y(t,\theta))} \tilde\U^2(t,\theta)d \theta\\ 
&\quad = \theta e^{-\frac{1}{A}(\tilde\Y(t,\eta)-\tilde\Y(t,\theta))} \tilde\U^2(t,\theta)\Big|_{\theta=0}^{\eta} \\ 
&\qquad\qquad\qquad -\int_0^{\eta} \theta e^{-\frac{1}{A}(\tilde\Y(t,\eta)-\tilde\Y(t,\theta))} (\frac{1}{A}\tilde\V^2\tilde\Y_{\eta}+2\tilde\U\tilde\U_{\eta}) (t,\theta)d\theta\\ 
&\quad = \eta \tilde\U^2(t,\eta)-\int_0^{\eta} \theta e^{-\frac{1}{A}(\tilde\Y(t,\eta)-\tilde\Y(t,\theta))}(\frac{1}{A}\tilde\U^2\tilde\Y_{\eta}+2\tilde\U\tilde\U_{\eta}) (t,\theta) d\theta\\ 
&\quad \leq  \tilde\U^2(t,\eta) + \int_0^\eta e^{-\frac{1}{A}(\tilde\Y(t,\eta)-\tilde\Y(t,\theta))} (\frac{1}{A}\tilde\U^2\tilde\Y_{\eta}+ 2\vert \tilde\U\tilde\U_{\eta}\vert)(t,\theta)d\theta\\ 
&\quad \leq \tilde\U^2 (t,\eta) + 4 \tilde\P(t,\eta) \\ 
&\quad \leq 6 \tilde\P(t,\eta).
\end{align*}

The proof of   \eqref{eq:all_PestimatesEXX}
\begin{equation*}
\int_0^\eta e^{-\frac{5}{4A}(\tilde \Y(t,\eta)-\tilde \Y(t,\theta))} \tilde \P\tilde \Y_{\eta}(t,\theta) d\theta\leq 4A\tilde\P(t,\eta)
\end{equation*}
follows in the same manner as \eqref{eq:343XX}.

Furthermore, for  \eqref{eq:32PXX} we find
\begin{align*}
\int_0^\eta e^{-\frac3{2A} (\tilde\Y(t,\eta)-\tilde\Y(t,\theta))}&\tilde \P(t,\theta) d\theta\\
& = \theta e^{-\frac3{2A} (\tilde\Y(t,\eta)-\tilde\Y(t,\theta))} \tilde \P(t,\theta) \Big\vert_{\theta=0}^{\eta}\\
& \quad-\int_0^\eta \theta e^{-\frac3{2A}(\tilde \Y(t,\eta)-\tilde \Y(t,\theta))}(\frac3{2A}\tilde \P\tilde\Y_{\eta}+\frac{1}{A^2}\tilde\Q\tilde\Y_{\eta})(t,\theta) d\theta\\
& \leq\eta\tilde \P(t,\eta)+\frac{3}{A}\int_0^\eta e^{-\frac3{2A}(\tilde\Y(t,\eta)-\tilde\Y(t,\theta))}\tilde\P\tilde\Y_{\eta} (t,\theta) d\theta\\
& \leq 7\tilde\P(t,\eta),
\end{align*}
which follows from \eqref{eq:343XX}.

In order to prove  \eqref{eq:p2yXX} we do as follows:
\begin{align*}
\int_0^\eta &e^{-\frac1{2A} (\tilde\Y(t,\eta)-\tilde\Y(t,\theta))} \tilde \P^2\tilde\Y_{\eta} (t,\theta) d\theta \\
&\qquad = 2A\tilde \P^2(t,\eta)-8\tilde\P\tilde\Q(t,\eta) \\
& \qquad\quad + 8\int_0^\eta e^{-\frac1{2A}(\tilde\Y(t,\eta)-\tilde\Y(t,\theta))}(\tilde \P^2+\frac1{A^2}\tilde\Q^2)\tilde\Y_{\eta} (t,\theta) d\theta\\
& \qquad\quad + 8 \int_0^\eta e^{-\frac1{2A} (\tilde\Y(t,\eta)-\tilde\Y(t,\theta))} ((\tilde\P-\tilde\U^2)\tilde\Y_{\eta}(t,\theta) -\frac12A^5)\tilde\P(t,\theta) d\theta,
\end{align*} 
and hence 
\begin{align*}
\int_0^\eta &e^{-\frac1{2A}(\tilde\Y(t,\eta)-\tilde\Y(t,\theta))} \tilde\P^2\tilde\Y_{\eta} (t,\theta) d\theta\\
& \quad\leq 8\tilde \P\tilde\Q(t,\eta) + 8 \int_0^\eta e^{-\frac1{2A} (\tilde \Y(t,\eta)-\tilde\Y(t,\theta))} ((\tilde\U^2-\tilde\P)\tilde\Y_{\eta}(t,\theta)+\frac12A^5)\tilde\P(t,\theta) d\theta\\
& \quad\leq 8A\tilde\P^2(t,\eta) +8\Big(\int_0^\eta e^{-\frac{1}{A}(\tilde\Y(t,\eta)-\tilde\Y(t,\theta))}((\tilde\U^2-\tilde\P)\tilde\Y_{\eta}(t,\theta) +\frac12A^5) d\theta\Big)^{1/2}\\ 
& \quad \qquad \qquad \qquad \qquad\times \Big(\int_0^\eta ((\tilde \U^2-\tilde\P)\tilde\Y_{\eta}(t,\theta)+\frac{1}{2}A^5 )\tilde \P^2(t,\theta) d\theta\Big)^{1/2}\\
& \quad\leq A\bigO(1)\tilde\P^{1/2}(t,\eta).
\end{align*}

As for the proof of  \eqref{eq:generalXX}, we proceed as follows. 
\begin{align}
\int_0^\eta & e^{-\frac{1}{A}(\tilde \Y(t,\eta)-\tilde \Y(t,\theta))}  \tilde \P^{1+\beta}\tilde \Y_{\eta}(t,\theta) d\theta \nn \\
&\qquad= A e^{-\frac{1}{A}(\tilde \Y(t,\eta)-\tilde \Y(t,\theta))} \tilde \P^{1+\beta}\big|_{\theta=0}^\eta \nn \\
&\qquad\quad -A(1+\beta) \int_0^\eta e^{-\frac{1}{A}(\tilde \Y(t,\eta)-\tilde \Y(t,\theta))} \tilde \P^{\beta}\tilde \P_\eta (t,\theta) d\theta\nn\\
&\qquad=A  \tilde \P^{1+\beta}(t,\eta)
-\frac{1}{A}(1+\beta) \int_0^\eta e^{-\frac{1}{A}(\tilde \Y(t,\eta)-\tilde \Y(t,\theta))} \tilde \P^{\beta}\tilde \Q \tilde \Y_{\eta}(t,\theta) d\theta\nn\\
&\qquad=A  \tilde \P^{1+\beta}(t,\eta) -(1+\beta) e^{-\frac{1}{A}(\tilde \Y(t,\eta)-\tilde \Y(t,\theta))} \tilde \P^{\beta}\tilde \Q(t,\theta)\big|_{\theta=0}^\eta \nn\\
&\qquad\quad+(1+\beta) \int_0^\eta e^{-\frac{1}{A}(\tilde \Y(t,\eta)-\tilde \Y(t,\theta))}\big(\beta \tilde \P^{\beta-1}\tilde \P_\eta \tilde \Q+ \tilde \P^{\beta}\tilde \Q_\eta \big)(t,\theta)d\theta\nn\\
&\qquad=A  \tilde \P^{1+\beta}(t,\eta) -(1+\beta)  \tilde \P^{\beta}\tilde \Q(t,\eta)\nn \\
&\qquad\quad+(1+\beta) \int_0^\eta e^{-\frac{1}{A}(\tilde \Y(t,\eta)-\tilde \Y(t,\theta))}\big(\beta \tilde \P^{\beta-1}\tilde \Q^2\tilde \Y_\eta\frac{1}{A^2} 
+ \tilde \P^{\beta}\tilde \Q_\eta \big)(t,\theta)d\theta. \label{eq:beta}
\end{align}
Recall \eqref{eq:DogE}, \eqref{eq:D_deriv}, and \eqref{eq:E_deriv}, which together imply
\begin{equation*}
\tilde \Q_\eta(t,\eta)= \frac12\big(\tilde\E_\eta -\tilde\D_\eta \big)(t,\eta)  
= -\big(\tilde\U^2-\tilde\P \big) \tilde\Y_\eta(t,\eta)-\frac12 A^5+\tilde\P\tilde \Y_\eta(t,\eta).
\end{equation*}
Thus
\begin{multline*}
\beta \tilde \P^{\beta-1}\tilde \Q^2\tilde \Y_\eta\frac{1}{A^2} 
+ \tilde \P^{\beta}\tilde \Q_\eta \\= 
\frac{\beta}{A^2} \tilde \Q^2\tilde \P^{\beta-1}\tilde\Y_\eta+\tilde \P^{\beta+1}\tilde\Y_\eta
+\tilde \P^{\beta}\big( -(\tilde\U^2-\tilde\P)\tilde\Y_\eta-\frac12 A^5\big).
\end{multline*}
Inserting this expression in \eqref{eq:beta} and re-ordering the terms, we find
\begin{align*}
\beta\int_0^\eta & e^{-\frac{1}{A}(\tilde \Y(t,\eta)-\tilde \Y(t,\theta))}  \tilde \P^{1+\beta}\tilde \Y_{\eta}(t,\theta) d\theta \\
 & \qquad\qquad\qquad\qquad+ \frac{1+\beta}{A^2} \int_0^\eta  e^{-\frac{1}{A}(\tilde \Y(t,\eta)-\tilde \Y(t,\theta))}  \tilde \P^{\beta-1}\tilde \Q^2\tilde \Y_{\eta}(t,\theta) d\theta \\
&=(1+\beta)  \tilde \P^{\beta}\tilde \Q(t,\eta) -A  \tilde \P^{1+\beta}(t,\eta) \\
&\qquad +(1+\beta) \int_0^\eta  e^{-\frac{1}{A}(\tilde \Y(t,\eta)-\tilde \Y(t,\theta))}  \tilde \P^{\beta}\big( (\tilde\U^2-\tilde\P)\tilde\Y_\eta+\frac12 A^5\big)(t,\theta) d\theta.
\end{align*}
Estimating this, using \eqref{eq:Ddef}, we find
\begin{align*}
\beta\int_0^\eta & e^{-\frac{1}{A}(\tilde \Y(t,\eta)-\tilde \Y(t,\theta))}  \tilde \P^{1+\beta}\tilde \Y_{\eta}(t,\theta) d\theta \\
&\le (1+\beta)  \tilde \P^{\beta}\vert\tilde \Q\vert (t,\eta) \\
&\quad +(1+\beta) \int_0^\eta  e^{-\frac{1}{A}(\tilde \Y(t,\eta)-\tilde \Y(t,\theta))}  \tilde \P^{\beta}\big( (\tilde\U^2-\tilde\P)\tilde\Y_\eta+\frac12 A^5\big)(t,\theta) d\theta\\
&\le (1+\beta)\norm{ \tilde \P}^\beta_\infty A\tilde \P (t,\eta)\\
&\qquad+(1+\beta)\norm{ \tilde \P}^\beta_\infty
\int_0^\eta  e^{-\frac{1}{A}(\tilde \Y(t,\eta)-\tilde \Y(t,\theta))}  \big( (\tilde\U^2-\tilde\P)\tilde\Y_\eta+\frac12 A^5\big)(t,\theta) d\theta\\
&\le (1+\beta)\norm{ \tilde \P}^\beta_\infty A\tilde \P (t,\eta)+(1+\beta)\norm{ \tilde \P}^\beta_\infty \tilde \D(t,\eta)\\
&\le 3(1+\beta)\Big(\frac{A^4}4\Big)^\beta A \tilde \P (t,\eta),
\end{align*}
where we used  \eqref{eq:all_estimatesC}, \eqref{eq:DPEP}, and \eqref{eq:all_estimatesA}.

\end{proof}
%------------------- an explicit example
\section{The antisymmetric peakon-antipeakon example}  \label{app:peak}

Fortunately, one can compute explicitly the quantities described in this paper in the important case of 
an antisymmetric peakon-antipeakon solution. 

More precisely, consider the function  \cite{GH}
\begin{equation*}
u(t,x)= \begin{cases}
\beta(t)\sinh(x), & \quad \vert x\vert \le\gamma(t),\\
\sign(x)\alpha(t)e^{-\vert x\vert } , & \quad \vert x\vert \ge\gamma(t),
\end{cases}
= \begin{cases} -\alpha(t)e^{x}, & \quad x\le-\gamma(t),\\
\beta(t)\sinh(x), & \quad -\gamma(t)\le x\le \gamma(t),\\
\alpha(t)e^{-x}, & \quad \gamma(t)\le x,
\end{cases}
\end{equation*}
where 
\begin{equation*}
\alpha(t)= \frac{E}{2}\sinh(\frac{E}{2}(t-t_0)), \quad \beta(t)=E\frac{1}{\sinh(\frac{E}{2}(t-t_0))}, \quad \gamma(t)=\ln(\cosh(\frac{E}{2}(t-t_0))).
\end{equation*}
Here $E=\norm{u(t)}_{H^1}$, $t\neq t_0$, denotes the total energy.
The corresponding energy density is given by 
\begin{equation*}
\mu(t,x)=(u^2+u_x^2)(t,x)=\begin{cases} 2\alpha^2(t)e^{2x}, & \quad x\leq-\gamma(t),\\
\beta^2(t)\cosh(2x), & \quad -\gamma(t)\leq x\leq\gamma(t),\\
2\alpha^2(t)e^{-2x}, & \quad \gamma(t)\leq x,
\end{cases} \quad t\neq t_0, 
\end{equation*}
with $\mu(t_0,x)=E^2\delta_0(x)$ for $t=t_0$. Hence   $C=\mu(t,\Real)=E^2$,  and 
\begin{equation*}
F(t,x)= 
\begin{cases}
\alpha(t)^2 e^{2x}, & \quad x<-\gamma(t),\\
\frac{E^2}{4} \tanh^2(\frac{E}{2}(t-t_0)), & \quad x=-\gamma(t),\\
\frac12 E^2+\frac12 \beta(t)^2 \sinh(2x),& \quad -\gamma(t)< x< \gamma(t),\\
E^2-\frac{E^2}{4} \tanh^2(\frac{E}{2}(t-t_0)), & \quad x=\gamma(t),\\
E^2-\alpha(t)^2e^{-2x}, & \quad \gamma(t)< x. 
\end{cases} 
\end{equation*}
In particular, this solution experiences wave breaking at time $t=t_0$, that is, $u_x(t,0)$ tends to $-\infty$ as $t\to t_0-$ and 
\begin{equation*}
F(t_0-,x)=\begin{cases}
0, & \quad x\leq 0,\\
E^2, & \quad 0<x.
\end{cases}
\end{equation*}
The corresponding function $p_x(t,x)$, which can be computed using $p_x(t,x)=-u_t(t,x)-uu_x(t,x)$, is given by 
\begin{equation*}
p_x(t,x)= \begin{cases}
\alpha'(t)e^x-\alpha(t)^2e^{2x}, & \quad x< -\gamma(t),\\
\frac{E^2}{4}-\frac{E^2}{4}\tanh^2(\frac{E}{2}(t-t_0)), & \quad x=-\gamma(t),\\
-\beta'(t)\sinh(x)-\frac12 \beta(t)^2 \sinh(2x), & \quad -\gamma(t)< x< \gamma(t),\\
-\frac{E^2}{4}+\frac{E^2}{4}\tanh^2(\frac{E}{2}(t-t_0)), & \quad \gamma(t)=x,\\
-\alpha'(t)e^{-x}+\alpha(t)^2e^{-2x}, & \quad \gamma(t)< x.
\end{cases}
\end{equation*}
In particular, one obtains at wave breaking time $t=t_0$ that
\begin{equation*}
p_x(t_0-,x)=
\begin{cases}
\frac{E^2}{4} e^x, & \quad x< 0,\\
-\frac{E^2}{4}e^{-x}, & \quad 0<x.
\end{cases}
\end{equation*}
Here it is important to note that $p_x(t_0-,x)$ has a (negative) jump of height $-\frac{E^2}{2}$ at $x=0$ at time $t=t_0$. For all other points $x\in\Real$, the function $p_x(t_0-,x)$ is continuous.

Thus the function $G(t,x)=2p_x(t,x)+2F(t,x)$ is given by 
\begin{equation*}
G(t,x)= \begin{cases}
2\alpha'(t)e^x, & \quad x<-\gamma(t),\\
\frac{E^2}{2}, & \quad x=-\gamma(t),\\
E^2-2\beta'(t)\sinh(x), &\quad -\gamma(t)< x <\gamma(t),\\
\frac{3E^2}{2}, & \quad x=\gamma(t),\\
2E^2-2\alpha'(t)e^{-x}, & \quad \gamma(t)< x.\\
\end{cases}
\end{equation*}
In particular, one observes that $G_x(t,x)>0$ for all $x\in \Real$,
\begin{equation*}
\lim_{x\to-\infty} G(t,x)= 0\quad \text{ and } \quad \lim_{x\to\infty}G(t,x)=2E^2=2C.
\end{equation*}
Thus the limits at $\pm\infty$ are independent of time. Moreover, one has
\begin{equation*}
G(t_0-,x)=
\begin{cases}
\frac{E^2}{2}e^x,& \quad x\leq0,\\
2E^2-\frac{E^2}{2} e^{-x}, & \quad 0<x.
\end{cases}
\end{equation*}
Here it is important to note that $G(t_0-,x)$ has a jump of size $E^2$ due to the wave breaking at $t=t_0$, i.e, $\mu(t_0,x)=E^2\delta_0(x)$.

Direct computations using $p=p_{xx}+\frac12 u^2+\frac12 d\mu$ yield
\begin{equation*}
p(t,x)=\begin{cases}
\alpha'(t)e^x-\frac12 \alpha(t)^2e^{2x}, & \quad x\leq -\gamma(t),\\
-\beta'(t)\cosh(x)-\frac12 \beta(t)^2\cosh^2(x), & \quad -\gamma(t)\leq x\leq\gamma(t),\\
\alpha'(t)e^{-x}-\frac12 \alpha(t)^2e^{-2x}, & \quad \gamma(t)\leq x.
\end{cases}
\end{equation*}

\begin{figure}\centering
\includegraphics[width=13cm]{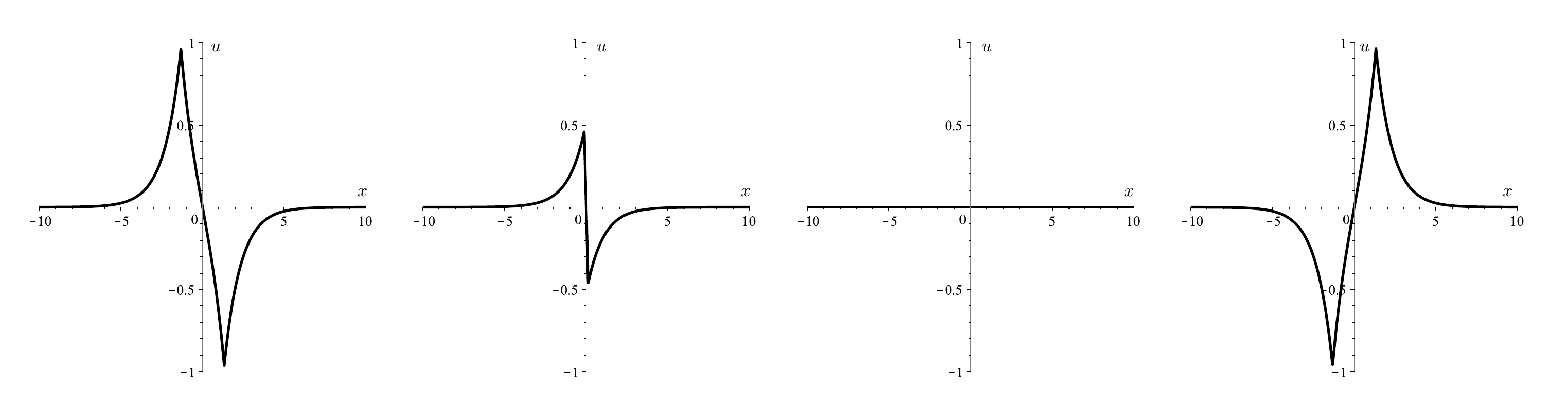}
\caption{Time evolution of $u(t,x)$ with $C=E^2=4$ and $t_0=2$ at $t=0,1.5,2,4$.} 
\label{fig:1}
\end{figure}

\begin{figure}\centering
\includegraphics[width=13cm]{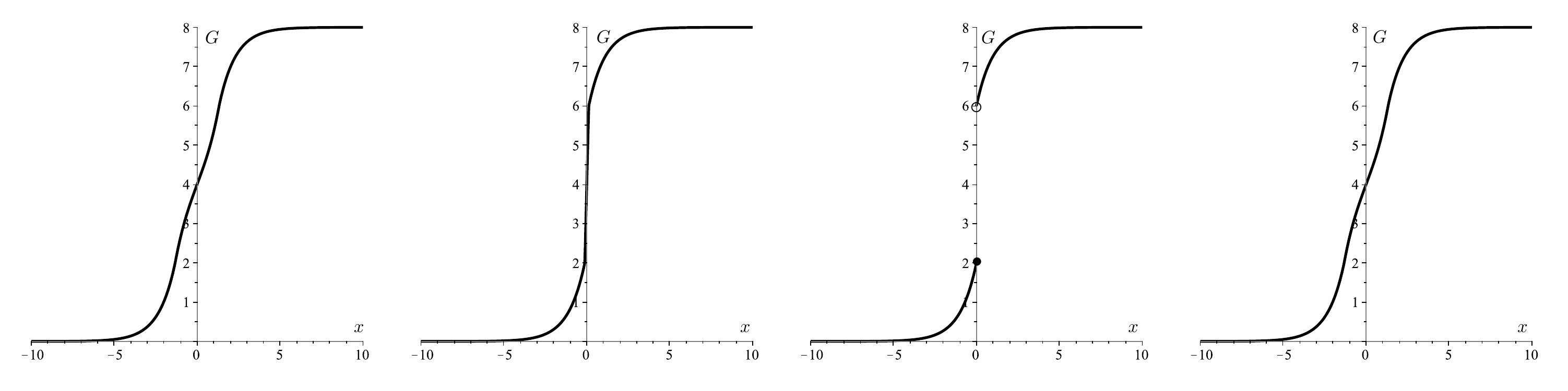}
\caption{Time evolution of $G(t,x)$ with $C=E^2=4$ and $t_0=2$ at $t=0,1.5,2,4$.}
\end{figure}

\begin{figure}\centering
\includegraphics[width=13cm]{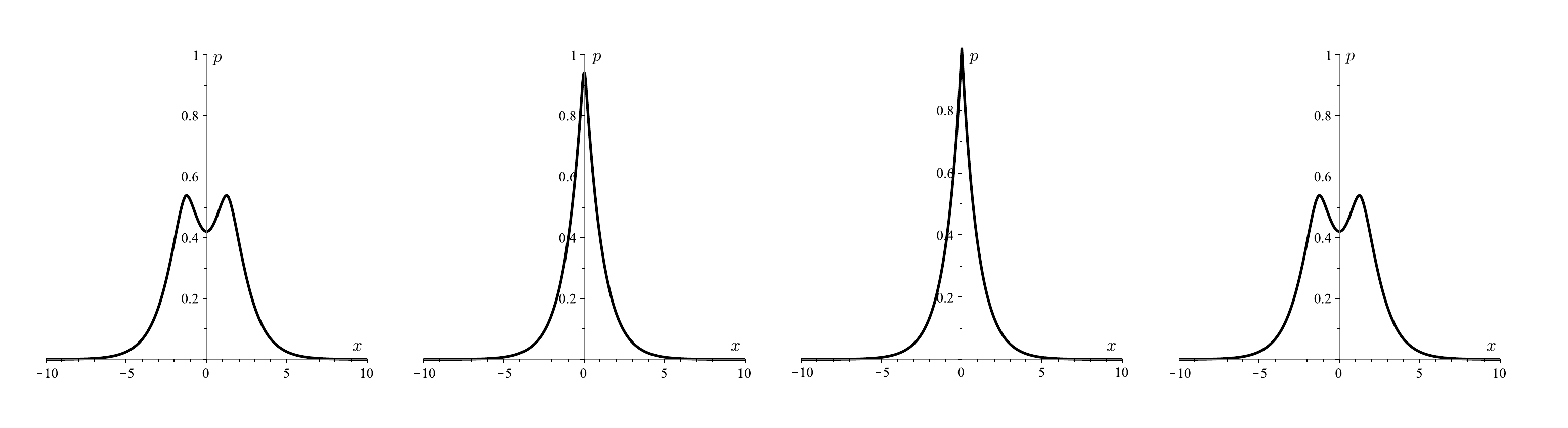}
\caption{Time evolution of $p(t,x)$ with $C=E^2=4$ and $t_0=2$ at $t=0,1.5,2,4$.}
\end{figure}

\begin{figure}\centering
\includegraphics[width=13cm]{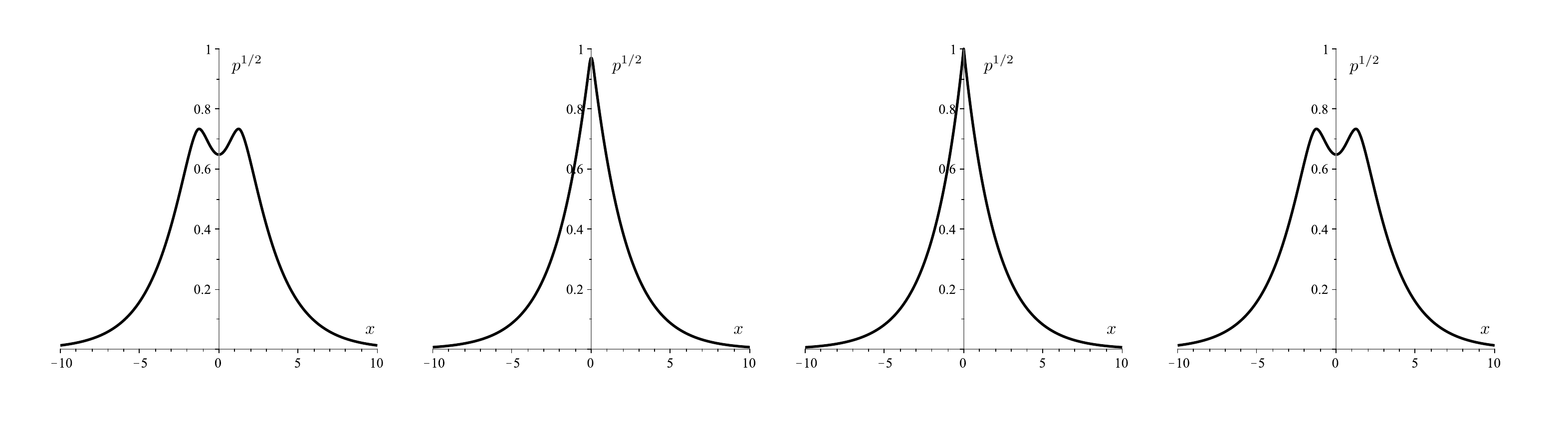}
\caption{Time evolution of $p^{1/2}(t,x)$ with $C=E^2=4$ and $t_0=2$ at $t=0,1.5,2,4$.}
\end{figure}

In the new coordinates the solution reads, using $G(t,\Y(t,\eta))=\eta$, 
\begin{equation*}
\Y(t,\eta)= \begin{cases}
\ln\left(\frac{\eta}{2\alpha'(t)}\right), &\quad 0<\eta\leq \frac{E^2}{2},\\
\sinh^{-1}(\frac{E^2-\eta}{2\beta'(t)}), & \quad \frac{E^2}{2}\leq \eta\leq \frac{3E^2}{2},\\
\ln\left(\frac{2\alpha'(t)}{2E^2-\eta}\right), & \quad  \frac{3E^2}{2}\leq \eta< 2E^2 ,
\end{cases} 
\end{equation*}
and, applying  $\U(t,\eta)=u(t,\Y(t,\eta))$,
\begin{equation*}
\U(t,\eta)  = \begin{cases}
-\frac{\alpha(t)}{2\alpha'(t)} \eta, & \quad 0<\eta\leq \frac{E^2}{2},\\
\frac{\beta(t)}{2\beta'(t)} (E^2-\eta), & \quad \frac{E^2}{2}\leq \eta\leq \frac{3E^2}{2},\\
\frac{\alpha(t)}{2\alpha'(t)} (2E^2-\eta), & \quad \frac{3E^2}{2}\leq \eta< 2E^2,
\end{cases}
\end{equation*}
and, by  invoking $\P(t,\eta)=p(t,\Y(t,\eta))$,
\begin{equation*}
\P(t,\eta) = \begin{cases}
\frac12 \eta-\frac{\alpha(t)^2}{8\alpha'(t)^2}\eta^2, & \quad 0<\eta\leq \frac{E^2}{2},\\
-\beta'(t)\sqrt{1+\frac{(E^2-\eta)^2}{4\beta'(t)^2}}-\frac12 \beta(t)^2(1+\frac{(E^2-\eta)^2}{4\beta'(t)^2}), & \quad \frac{E^2}{2}\leq \eta\leq \frac{3E^2}{2},\\
\frac12 (2E^2-\eta)-\frac{\alpha(t)^2}{8\alpha'(t)^2}(2E^2-\eta)^2, &\quad \frac{3E^2}{2}\leq\eta< 2E^2,
\end{cases}
\end{equation*}
where we used the convention that $\sinh^{-1}(x)$ denotes the inverse of $\sinh(x)$.

\begin{figure}\centering
\includegraphics[width=13cm]{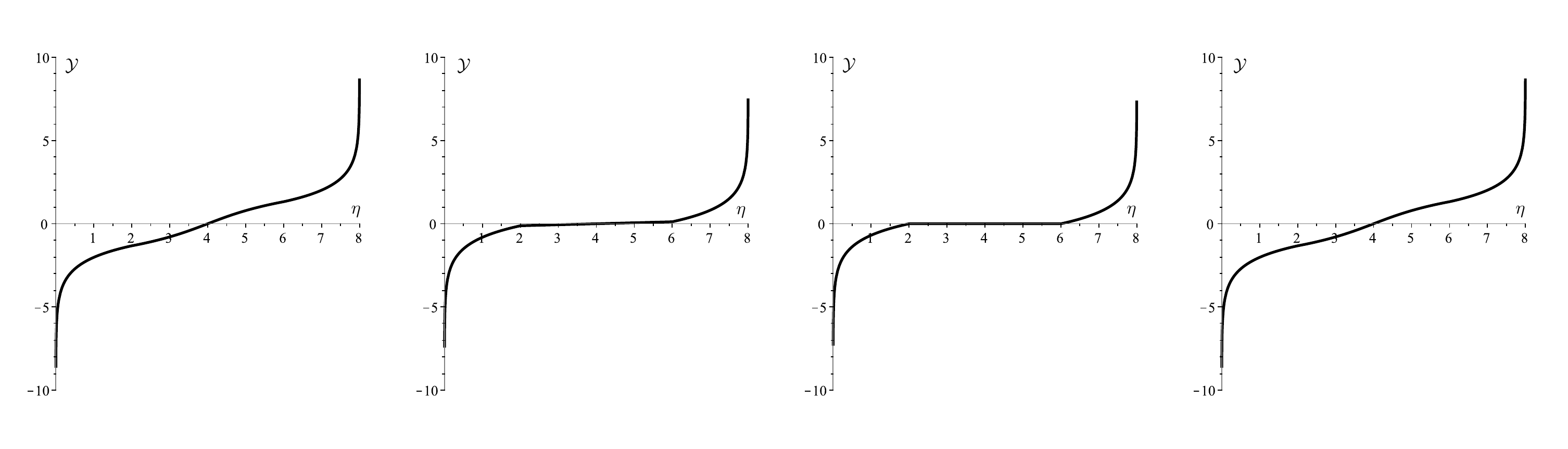}
\caption{Time evolution of $\Y(t,\eta)$ with $C=E^2=4$ and $t_0=2$ at $t=0,1.5,2,4$.}
\end{figure}

\begin{figure}\centering
\includegraphics[width=13cm]{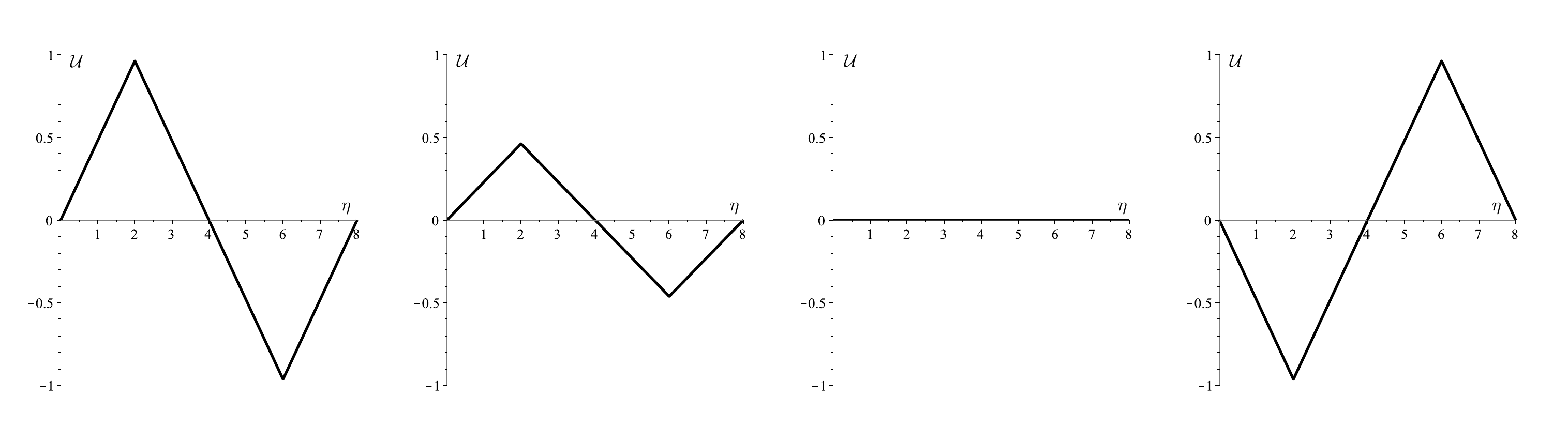}
\caption{Time evolution of $\U(t,\eta)$ with $C=E^2=4$ and $t_0=2$ at $t=0,1.5,2, 4$.}
\end{figure}

\begin{figure}\centering
\includegraphics[width=13cm]{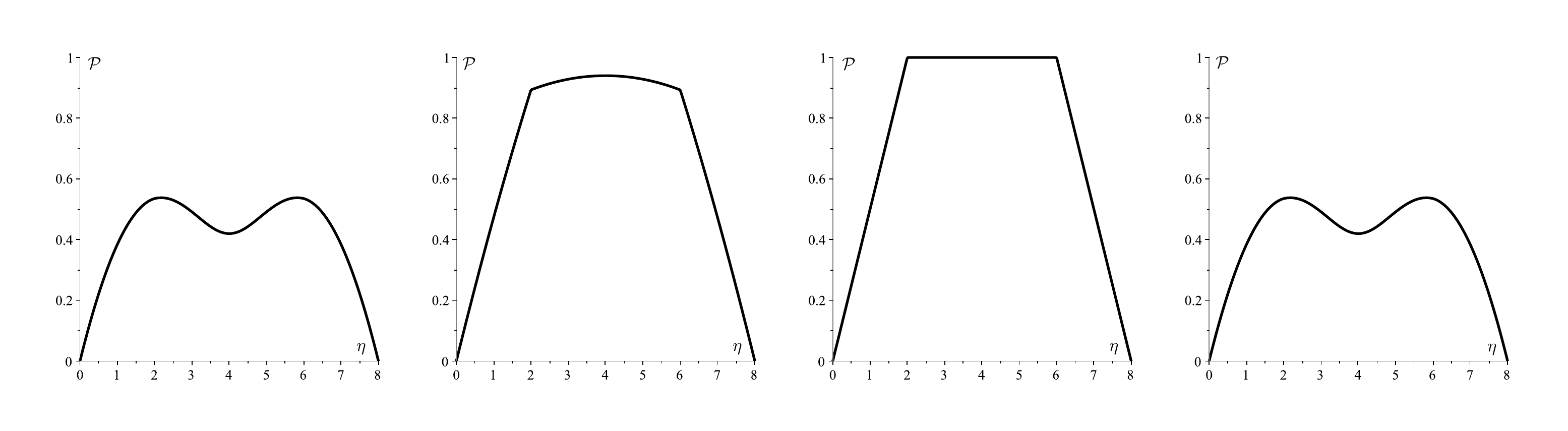}
\caption{Time evolution of $\P(t,\eta)$ with $C=E^2=4$ and $t_0=2$ at $t=0,1.5,2,4$.}
\end{figure}

\begin{figure}\centering
\includegraphics[width=13cm]{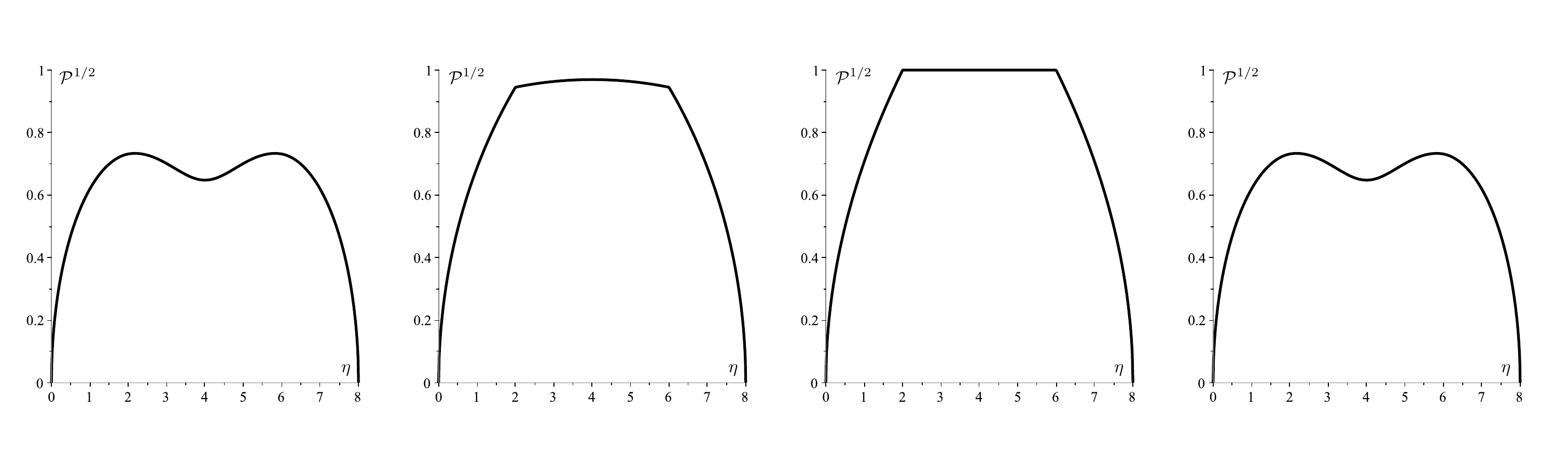}
\caption{Time evolution of $\P^{1/2}(t,\eta)$ with $C=E^2=4$ and $t_0=2$ at $t=0,1.5,2,4$.}
\end{figure}

For the scaled quantities we find, when we introduce $A=\sqrt{2C}=\sqrt{2}E$, that
\begin{align*}
\tilde\Y(t,\eta)&=A\Y(t,A^2\eta)= \sqrt{2C}\begin{cases}
\ln\left(\frac{C\eta}{\alpha'(t)}\right), &\quad 0<\eta\leq \frac{1}{4},\\
\sinh^{-1}(\frac{C(1-2\eta)}{2\beta'(t)}), & \quad \frac{1}{4}\leq \eta\leq \frac{3}{4},\\
\ln\left(\frac{\alpha'(t)}{C(1-\eta)}\right), & \quad \frac{3}{4}\leq \eta<1,
\end{cases} \\[3mm]
\tilde\U(t,\eta)  &=A\U(t,A^2\eta)= \sqrt{2C}\begin{cases}
-\frac{\alpha(t)}{\alpha'(t)} C\eta, & \quad 0<\eta\leq \frac{1}{4},\\
\frac{\beta(t)}{2\beta'(t)} C(1-2\eta), & \quad \frac{1}{4}\leq \eta\leq \frac{3}{4},\\
\frac{\alpha(t)}{\alpha'(t)} C(1-\eta), & \quad \frac{3}{4}\leq \eta<1,
\end{cases}\\[3mm]
\tilde\P(t,\eta) &=A^2 \P(t,A^2\eta)\\
&=2C\begin{cases}
C\eta-\frac{\alpha(t)^2}{2\alpha'(t)^2}C^2\eta^2, & \quad 0<\eta\leq \frac{1}{4},\\
-\beta'(t)\sqrt{1+\frac{C^2(1-2\eta)^2}{4\beta'(t)^2}}-\frac12 \beta(t)^2(1+\frac{C^2(1-2\eta)^2}{4\beta'(t)^2}), & \quad \frac{1}{4}\leq \eta\leq \frac{3}{4},\\
 C(1-\eta)-\frac{\alpha(t)^2}{2\alpha'(t)^2}C^2(1-\eta)^2, &\quad \frac{3}{4}\leq\eta<1.
\end{cases}
\end{align*}

\begin{figure}\centering
\includegraphics[width=13cm]{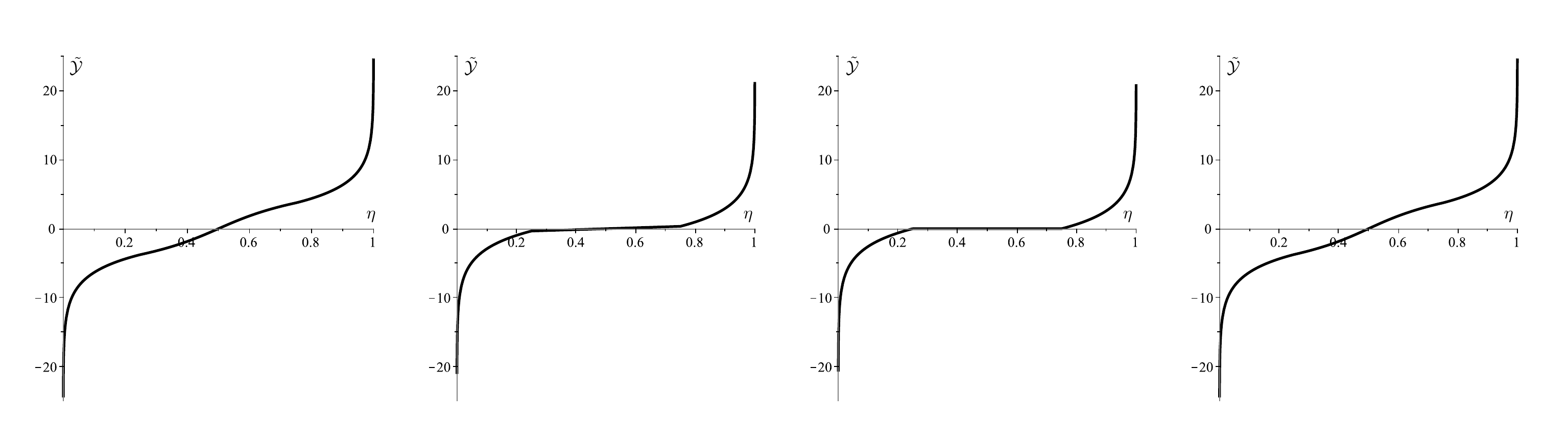}
\caption{Time evolution of $\tilde\Y(t,\eta)$ with $C=E^2=4$ and $t_0=2$ at $t=0,1.5,2,4$.}
\end{figure}

\begin{figure}\centering
\includegraphics[width=13cm]{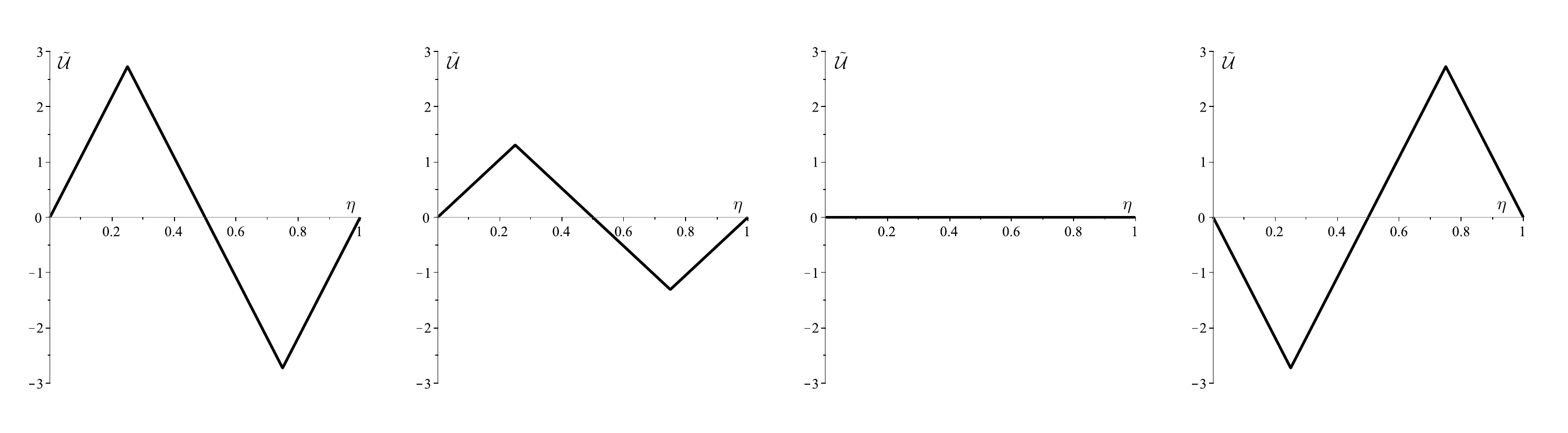}
\caption{Time evolution of $\tilde\U(t,\eta)$ with $C=E^2=4$ and $t_0=2$ at $t=0,1.5,2, 4$.}
\end{figure}

\begin{figure}\centering
\includegraphics[width=13cm]{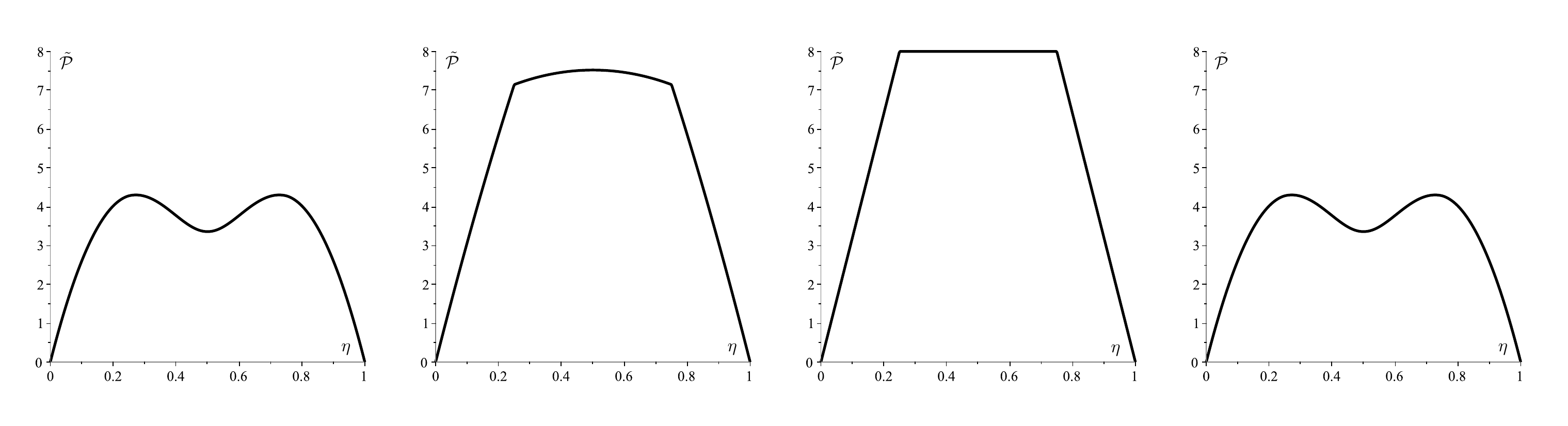}
\caption{Time evolution of $\tilde\P(t,\eta)$ with $C=E^2=4$ and $t_0=2$ at $t=0,1.5,2,4$.}
\end{figure}

\begin{figure}\centering
\includegraphics[width=13cm]{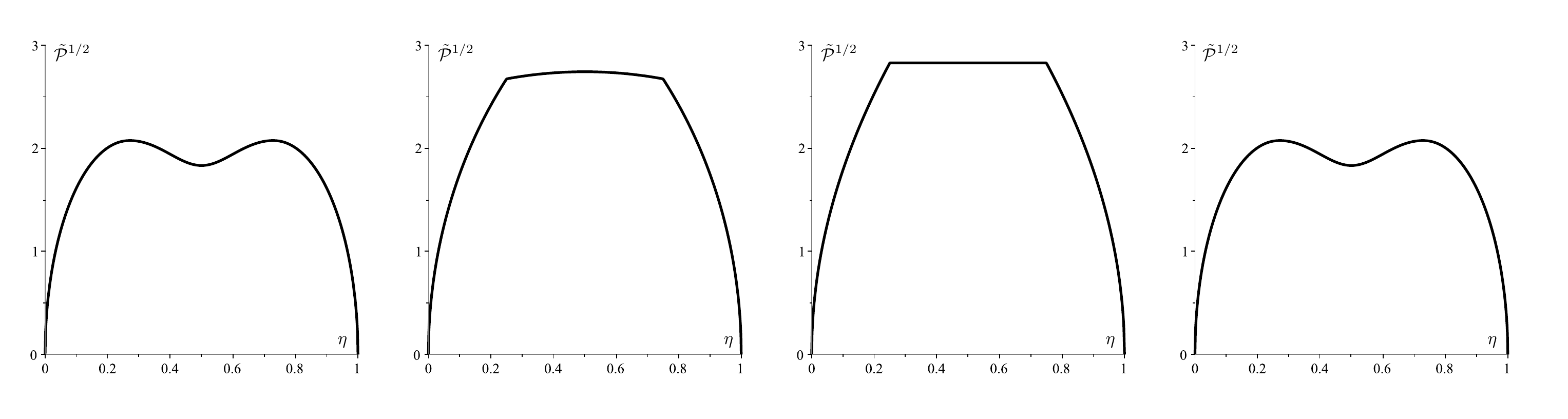}
\caption{Time evolution of $\tilde\P^{1/2}(t,\eta)$ with $C=E^2=4$ and $t_0=2$ at $t=0,1.5,2,4$.}
\end{figure}

%-----------------------
%
%  Additional formulas added after \end{document}
%
%-----------------------
%---------- references ------------------

\end{document}